\documentclass[10pt]{amsart}
\usepackage{amsthm}
\usepackage{tikz}
\usepackage{xcolor}
\usepackage{epsf}
\usepackage{amsfonts,amssymb}
\usepackage{amsmath}
\usepackage{fancyhdr}
\usepackage[all,cmtip]{xy}
\usepackage[margin=1in]{geometry}

\newtheorem{theorem}{Theorem}
\newtheorem*{theorem*}{Theorem}
\newtheorem{lemma}[theorem]{Lemma}
\newtheorem{corollary}[theorem]{Corollary}

\newtheorem*{corollary*}{Corollary}
\newtheorem{proposition}[theorem]{Proposition}
\newtheorem*{proposition*}{Proposition}

\newtheorem*{claim*}{Claim}

\newtheorem{definition}{Definition}[section]

\theoremstyle{definition}
\newtheorem*{example}{Example}

\theoremstyle{remark}
\newtheorem{remark}[theorem]{Remark}
\newtheorem*{remark*}{Remark}

\theoremstyle{notation}

\theoremstyle{plain}

\DeclareMathOperator*{\llimsup}{limsup}
\renewcommand{\limsup}{\llimsup}

\newcommand{\p}{\partial}
\newcommand{\Z}{{\mathbb Z}}
\newcommand{\T}{{\mathbb T}}
\newcommand{\C}{{\mathbb C}}

\newcommand{\R}{{\mathbb R}}

\newcommand{\N}{{\mathbb N}}
\newcommand{\Q}{{\mathbb Q}}

\newtheorem{Notation}{Notation}[section]

\begin{document}

\title{The precise regularity of the Lyapunov exponent for  Cos-type quasiperiodic Schr\"odinger cocycles with large couplings} \maketitle
\begin{center}
\author{
    Jiahao Xu\footnote{Department of Mathematics,
Nanjing University, Nanjing, China,   178881559@qq.com},\ \  Lingrui Ge\footnote{Beijing International Center for Mathematical Research, Peking University, Beijing, China and Department of Mathematics, Nanjing University, Nanjing, China,  gelingrui@bicmr.pku.edu.cn},\ \  Yiqian Wang\footnote{The corresponding author.}\ \footnote{Department of Mathematics,
Nanjing University, Nanjing, China, yiqianw@nju.edu.cn}
  }
\end{center}

\fancyhead{}

\fancyhead[CO]{\small Lyapunov Exponent for a class of $C^2$ quasiperiodic Schr\"odinger cocycles}

\fancyhead[CE]{Jiahao Xu, Lingrui Ge and Yiqian Wang}
\tableofcontents
\begin{abstract} In this paper, we study the regularity of the
Lyapunov exponent for quasiperiodic Schr\"odinger cocycles with $C^2$ cos-type potentials,
large coupling constants, and a fixed Diophantine frequency. We obtain the absolute continuity of the Lyapunov exponent. Moreover, we prove  the
Lyapunov exponent is $\frac{1}{2}$-H\"older continuous.  Furthermore, for any given
$r\in (\frac12, 1)$, we can find some energy in the spectrum where
the  local regularity of the Lyapunov exponent is between $(r-\epsilon)$-H\"older continuity and $(r+\epsilon)$-H\"older continuity.
\end{abstract}
\section{Introduction}
In this paper, we consider the following one-dimensional quasiperiodic Schr\"odinger operators on $\ell^2(\Z)$,
\begin{equation}\label{sch} (H_{\lambda v,\alpha,x}u)_n=u_{n+1}+u_{n-1}+\lambda v(x+n\alpha)u_n.\end{equation}
Here $v\in C^r(\R/\Z,\R),$ $r\in\N\cup\{\infty,\omega\}$ is the potential, $\lambda\in\R$ is the coupling constant, $x\in\T=\R/\Z$ is the phase, and $\alpha\in\R\backslash\Q$ is the frequency.
The spectral properties of the operator \eqref{sch} are closely related to the {\it Schr\"odinger cocycle} $(\alpha, A^{(E-\lambda v)})\in \R/\Z\times C^r(\R/\Z, SL(2,\R))$ with

$$A^{(E-\lambda v)}(x)=
\begin{pmatrix}
E-\lambda v(x)& -1\\
1& 0
\end{pmatrix},$$
which defines a family of dynamical systems on $\R/\Z \times \R^2$  given by

$$
(x,w) \rightarrow (x+\alpha,A^{(E-\lambda v)}(x)w).
$$
The $n$th iteration of the cocycle is denoted by
$$
(\alpha,A^{(E-\lambda v)})^n=(n\alpha,A_n^{(E-\lambda v)}),
$$
where
\begin{align*}
\begin{split}
A_n^{(E-\lambda v)}(x):=A_n(x)=\left\{
\begin{array}{ll}
A^{(E-\lambda v)}(x+(n-1)\alpha)\cdots A^{(E-\lambda v)}(x), & n\geq 1\\
I_2, & n= 0\\
A_{-n}^{(E-\lambda v)}(x+n\alpha)^{-1}, & n\leq -1.
\end{array}
\right.
\end{split}
\end{align*}

The Lyapunov exponent $L(E,\lambda)$ of the cocycle is defined as $$\lim\limits_{n\rightarrow\infty}\frac{1}{n}\int_{\R/\Z}\ln\|A_n^{(E-\lambda v)}(x)\|dx
=\inf_{n\geq 1}\frac{1}{n}\int_{\R/\Z}\ln\|A_n^{(E-\lambda v)}(x)\|dx\geq 0.$$ The limit
exists and is equal to the infimum since $\{\ln\|A_n^{(E-\lambda v)}(x)\|\}_{n\geq1}$ is a subadditive sequence. Moreover, by Kingman's subadditive ergodic theorem, we also
have
$$L(E, \lambda)=\lim\limits_{n\rightarrow\infty}\frac{1}{n}\ln\|A_n^{(E-\lambda v)}(x)\|$$
for Lebesgue almost every $x\in\R/\Z$.


In the past twenty years, a large number of papers have been dedicated to studying the regularity of the Lyapunov exponent for quasiperiodic Schr\"odinger cocycles. Both the results and their proofs depend heavily on the vanishing/nonvanishing of the Lyapunov exponent (LE). In the positive LE regime, the main method is the Large Deviation Theorem (LDT) and the Avalanche Principal (AP), based on which various H\"older continuity results of the LE were obtained \cite{bourgain1,goldsteinschlag,goldsteinschlag2,Han-Zhang,LWY,schlag,wz1,YZ}. In the zero LE regime, the method is the Quantitative Almost Reducibility Theorem (QART), see \cite{amor,aj,ccyz,gyzh,LYZZ,Puig06}. Later, it was known that both the LDT and QART depend sensitively on the arithmetic property of frequency and the regularity of the potential.

LDT was {\it first} established by Bourgain and Goldstein \cite{bourgaingoldstein} and further developed by Goldstein and Schlag \cite{goldsteinschlag} in 2000 for analytic quasiperiodic Schr\"odinger cocycles with positive LEs and strong Diophantine frequencies, to obtain Anderson localization and H\"older continuity of the associated LE, respectively. It remains a challenge  for a few years how to use LDT and AP to obtain optimal estimates of the H\"older exponent. The {\it first} breakthrough belongs to Bourgain \cite{bourgain1} who proved for the almost Mathieu operator (AMO) with $v(x)=2\cos (2\pi x)$, a Diophantine frequency \footnote{Fix two constants $\tau>2,\gamma>0.$ We say $\alpha$ satisfies a $Diophantine$ condition $DC_{\tau,\gamma}$ if $$\vert\alpha-\frac{p}{q}\vert\geq\frac{\gamma}{|q|^{\tau}}$$
for all $p,q\in\Z$ with $q\neq0.$

It is a standard result that for any $\tau>2,$
$$DC_{\tau}:=\bigcup\limits_{\gamma>0} DC_{\tau,\gamma}$$
is of full Lebesgue measure. We fix  $\tau>2$ and  $\alpha\in DC_{\tau}.$
} and $|\lambda|\gg 1$, the LE is $(\frac 12-\epsilon)$-H\"older continuous.
It was generalized by Goldstein and Schlag \cite{goldsteinschlag2} to the result that for arbitrary analytic potential near a trigonometric polynomial of degree $k$, the LE is $(\frac{1}{2k}-\epsilon)$-H\"older continuous provided the frequency is Diophantine and the LE is positive. Roughly speaking, the results in \cite{goldsteinschlag2} are based on refinements of the LDT and AP for the entries of the transfer matrix, which is different from \cite{goldsteinschlag}.  They could obtain $(\frac{1}{2k}-\epsilon)$-H\"older continuity of integrated density state (IDS) by controlling the distribution of zeros for the characteristic determinants.

For $0<|\lambda| \ll 1$,
Puig \cite{Puig06} proved that with a Diophantine frequency, the LE of AMO is locally $\frac 12$-H\"older continuous at endpoints of spectral gaps ({\bf EP}) and cannot be better. Later, it was proved by Amor \cite{amor} that in the perturbative regime, if the frequency is Diophantine, then the LE is $\frac 12$-H\"older continuous. Amor's result was extended by Avila and Jitomirskaya \cite{aj} to the non-perturbative regime and they also proved $\frac 12$-H\"older continuity for $\lambda\not=0, 1$ and all Diophantine frequencies. Recently Avila--Jitomirskaya's result was further extended by Leguil-You-Zhao-Zhou \cite{LYZZ} to general subcritical potentials and Diophantine frequencies. In contrast, there is no $\frac 12$-H\"older continuous result on a general analytic potential in the positive LE regime.

\subsection{Main results}

In this paper, we focus on a class of $C^{2}(\mathbb{S}^1)$  cos-type potentials which was first considered by Sinai \cite{sinai} and satisfies the following conditions:
\begin{itemize}
\item $\frac{dv}{dx}=0$ at exactly two points, one is the minimal and the other is the maximal, which is denoted by $x_1$ and $x_2$.
\item These two extremals are non-degenerate, that is, $\frac{d^2v}{dx^2}(x_j)\neq0$ for $j=1,2.$
\end{itemize}

\begin{definition}

We say a function $f: \mathbb{R}\rightarrow \mathbb{R}$ is locally $\kappa$-H\"older continuous at  $t_0\in \R$ with a H\"older exponent $0<\kappa\leq 1$, if there exists $\epsilon=\epsilon(t_0)>0$ and  $C=C(t_0, \epsilon)>0$ such that
$$|f(t_0)-f(t)|<C|t-t_0|^{\kappa},\ \ ~t~\in~(t_0-\epsilon,t_0+\epsilon).$$
If in addition there also exist $c=c(t_0)>0$ and a sequence $~t_n\rightarrow t_0$ such that $$|f(t_n)-f(t_0)|>c|t_n-t_0|^{\kappa},$$ we say  $f$ is exactly locally $\kappa$-H\"older continuous at $t_0$.
\end{definition}

\begin{definition}For a compact interval $I\subset \R$ and $0<\kappa\leq 1,$ we say a function $f$  is \textbf{absolutely $\kappa$-H\"older continuous} on $I,$ if there exists a uniform constant $C=C(I)>0$ such that for any finitely many  pair-wise disjoint intervals $(a_i,b_i)\subset I, i=1,2,\cdots,N$, it holds that
$$\sum\limits_{i=1}^N|f(a_i)-f(b_i)|\leq C\left(\sum\limits_{i=1}^N|a_i-b_i|\right)^{\kappa}.$$

\end{definition}
\begin{remark} It is easy to see that any absolutely $\kappa$-H\"older continuous function is indeed both (globally) $\kappa$-H\"older continuous and absolutely continuous on $I$. On the other hand, the following function on $[-1,1]$ is both $\frac{1}{2}$-H\"older continuous and absolutely continuous but not absolutely $\frac{1}{2}$-H\"older continuous:
$$f(x)=\left\{\begin{matrix} &0 & x\in [-1,0);\\& \sqrt{x-a_{n-1}}+\sum\limits_{j=1}^{n-1}e^{-100^j} &\qquad\quad x\in [a_{n-1},a_n),\ n\in \N;\\ &\sum\limits_{j=1}^{+\infty} e^{-100^j} & x\in [a_{\infty},1],\end{matrix}\right.$$
where $a_0=0$ and $a_n:=\sum\limits_{j=1}^n e^{-2\cdot (100)^j}, n\geq 1.$
\end{remark}
The main result of this paper is as follows.

\begin{theorem}\label{Th1} Let $\alpha$ be Diophantine and $v$ be of $C^{2}$ cos-type. Consider the quasiperiodic Schr\"odinger cocycle $(\alpha, A^{(E-\lambda v)})$ and let $L(E)=L(E,\lambda)$ be its LE. Then there exists some $\lambda_1=\lambda_1(\alpha,v)>0 $ such that for any fixed $\lambda>\lambda_1$, the followings hold true:
\begin{enumerate}
\item $L(E)$ is absolutely $\frac{1}{2}$-H\"older continuous on any compact interval $\rm I$. In particular,

{\rm (i)} $L(E)$ is $\frac{1}{2}$-H\"older continuous on any compact interval $\rm I$;

{\rm (ii)} $L(E)$ is absolutely continuous on any compact interval $\rm I$.

\item The set of all endpoints of spectral gaps is dense in the spectrum and the regularity of  $L(E)$  at each endpoint cannot be better than locally $\frac{1}{2}$-H\"older continuous.
\item $L(E)$ is  differentiable almost everywhere on $\R.$ Moreover, there exists $C=C(\alpha,v)>0$ such that
$$Leb\{E\in~\Sigma^{\lambda} ~\big\vert~|L'(E)|>M\}\leq C\cdot M^{-2}~for~M>0,$$
 where $\Sigma^{\lambda}$ is  the  spectrum of $H_{\lambda v,\alpha,x}$.
\item For any  $\frac{1}{2}<\beta<1$,
  there exists some point $E'\in \Sigma^{\lambda} $ such that $\liminf\limits_{E\rightarrow E'}\frac{\log |L(E)-L(E')|}{\log |E-E'|}= \beta.$
\end{enumerate}

\end{theorem}

\noindent\textbf{Remark(a):}
 The first part of Conclusion (2) was proved in \cite{wz2}, which shows that the set of {\bf EP} is small in the sense of measure but is large in the sense of topology.
 Conclusions (1) and (2) together imply exact (locally) $\frac{1}{2}$-H\"older continuity of $L(E)$  at each {\bf EP}.

 \noindent\textbf{Remark(b):} The H\"older continuity of the LE for Schr\"odinger cocycles is also expected to play important roles in studying Cantor spectrum, typical localization length,
phase transition, etc., for quasiperiodic Schr\"odinger operators.

\noindent\textbf{Remark(c):} An arithmatic condition (e.g. the Diophantine condition) on frequency is also necessary.
 For rational frequencies and generic irrational frequencies (e.g. extremely
Liouvillean ones), the LE is not H\"older continuous \cite{aj} (in a recent work by \cite{alsz}, it was proved that IDS is not H\"older continuous for AMO at extremely Liouvillean frequencies).

\noindent\textbf{Remark(d):} Part (2) of Theorem \ref{Th1} was also proved by Figueras and Timoudas \cite{FT} via a different method.
A related result of Theorem \ref{Th1} on analytic cases was recently obtained by \cite{KXZ}.


It  follows from the definition of absolute H\"older continuity and Theorem \ref{Th1} that
\begin{corollary} Let $\alpha,v,\lambda>\lambda_1$ be defined in Theorem \ref{Th1}, for any Borel measurable set $\Omega\in [\inf\Sigma^{\lambda} ,\sup\Sigma^{\lambda} ],$ it holds that
$$\int_{\Omega} |L'(E)| dE \leq C(\alpha, v)|\Omega|^{\frac{1}{2}}$$ with some constant $C(\alpha, v)>0,$ where $|\Omega|$ denotes the Lebesgue measure of $\Omega$.
\end{corollary}
\begin{remark} Here the constant $\frac{1}{2}$ is also optimal by the results in Theorem \ref{Th1}.

\end{remark}
\subsection{{Remarks on the regularity of the Lyapunov exponents}}

Much work has been devoted to the regularity properties of the LE and IDS as well. Here we focus on other results on the regularity of the LE of cocycles not mentioned above.

For lower regularity cases, Klein \cite{klein} and Cheng-Ge-You-Zhou \cite{cgyz} proved the  weak H\"older continuity of the LE for a class of Gevrey potentials.
\cite{FX} gave a new proof of the continuity of the LE  under the setting in \cite{wz1} without the use of LDT and AP. For other related results, one can refer to \cite{avilakrikorian,bje,Jiro-Mavi1,Jiro-Mavi2}.

There are also many negative results on the positivity and continuity of the LE for non-analytic
cases. It is well known that in $C^0$-topology, discontinuity of the LE holds true at every
non-uniformly hyperbolic cocycle, see \cite{furman,furstenberg}. Moreover, motivated by Ma\~ne \cite{mane1,mane2},
Bochi \cite{boc} proved that with an ergodic base system, any non-uniformly hyperbolic
$SL(2, \R)$-cocycle can be approximated by cocycles with zero LE in the $C^0$ topology.
Wang-You \cite{wy1} constructed examples to show
that the LE can be discontinuous even in the space of $C^{\infty}$ Schr\"odinger cocycles. Recently,
Wang-You \cite{wy2}  improved the result in \cite{wy1} by showing that in $C^r $topology, $1\leq r\leq +\infty,$ there exists Schr\"odinger cocycles with a positive LE that can be approximated by
ones with zero LE. Ge-Wang-You-Zhao \cite{gwyz} recently found the transition space for the continuity of the LE.
Jitomirskaya-Marx \cite{jitomirskaya2} constructed examples showing that the LE of $M(2, \C)$ cocycles is
discontinuous in $C^{\infty}$ topology.

 Weak H\"older continuity of the Lyapunov exponents for multi-frequency $GL(m, \C)$-cocycles, $m \geq 2,$ was obtained by Schlag \cite{schlag} and an important recent progress is that
Duarte-Klein \cite{duarteklein} proved weak H\"older continuity of the Lyapunov exponents for multi-frequency $M(m, \C)$-cocycles. All these results were obtained via LDT and AP.  Without the use of LDT or Avalanche principle, continuity of the LE for $M(d, \C)$ cocycles was given by Avila-Jitomirskaya-Sadel \cite{avilajitomirskayasadel}.

As for the continuity of the LE on the frequency, an arithmetic version of large deviation was developed by Bourgain and Jitomirskaya  in \cite{{bourgainjitomirskaya}} allowing them to obtain joint continuity of the LE for $\mathrm{SL}(2,\mathbb{R})$ cocycles on the frequency and the energy, at any irrational frequencies. This result has been crucial in many important developments, such as the proof of the Ten Martini problem \cite{avilajitomirskaya}, Avila's global theory of one-frequency cocycles \cite{avila1}. It was extended to the multi-frequency case by Bourgain \cite{bourgain2} and to general $\mathrm{M}(2,\mathbb{C})$ case by Jitomirskaya and Marx  \cite{jitomirskaya2} and Powell \cite{mathe}.       Jitomirskaya-Koslover-Schulteis \cite{jitomirskaya} obtained the continuity of the LE with respect to potentials for a class of
analytic quasiperiodic $M(2, \C)$ cocycles which is applicable to general quasiperiodic
Jacobi matrices or orthogonal polynomials on the unit circle in various parameters.
Jitomirskaya-Marx \cite{jitomirskaya1} later extended it to all (including singular) $M(2, \C)$ cocycles.

Other types of base dynamics on which regularity of the LE of analytic or differential Schr\"odinger operators holds true include a shift or skew-shift of a higher
dimensional torus by Bourgain-Goldstein-Schlag \cite{bourgaingoldsteinschlag}, doubling map  and Anosov diffeomorphism by Bourgain-Schlag \cite{bourgainschlag}.

\subsection{The idea for the proof}

Note that the LDT established in \cite{bourgaingoldstein} depends heavily on the analyticity of the potential because they used the subharmonic estimation techniques. If the potential is not analytic, Wang-Zhang \cite{wz1} developed a new iteration scheme to investigate the LDT for smooth Schr\"odinger cocycles. Based on it, they proved that for $C^2$ cos-type (Morse) potential with a large coupling, the LE is weak-H\"older continuous. Subsequently, it was improved to be $r$-H\"older continuous with $r>0$ small  and independent of the coupling by Liang, Wang and You \cite{LWY}. One natural question is, can we combine the method in \cite{goldsteinschlag2} and \cite{LWY} to obtain an optimal estimate on the regularity of the Lyapunov exponent? It might not work since the methods in \cite{goldsteinschlag2} depend on Jensen's formula which only holds for analytic functions. For this purpose, we will add some new ingredients to the classical LDT and AP, instead of proving a refined LDT.

\

Thus the main purpose of this paper is to explain how to use a relatively weak LDT and AP to {\it directly} obtain optimal regularity of the LE \footnote{The methods in \cite{bourgain1,goldsteinschlag2} work with IDS, instead of LE.}, which is completely different from those in \cite{bourgain1,goldsteinschlag2}. The key  is to find an iteration scheme to give a nice control on the derivative of the finite Lyapunov exponent (FLE). More precisely, we denote the FLE by $L_{N}(E)=\frac1{N}\int\log\|A_{N}(x,E)\|dx$. From the facts that $\|A_{N}(x,E)\|\ge 1$ and $|\partial_E\|A_{N}(x,E)\||\le C^{N}$ with some $C>1$ depending only on $A(x)$, we have
$$
|L'_{N}(E)|=\left|\frac{1}{N}\int\frac{\partial_E\|A_{N}(x,E)\|}{\|A_{N}(x,E)\|}dx\right|\le C^{N}.
$$
 The traditional way in \cite{goldsteinschlag} and \cite{wz1} is as follows. With the help of the LDT and AP, for each pair $(E_0, E)$ and (large) $N$ satisfying $|E-E_0|\thickapprox C^{-2N}$ such that
$$
|L(E)-L(E_0)|\le |L_{N}(E)-L_{N}(E_0)|+|L_{2N}(E)-L_{2N}(E_0)|+Ce^{-N^{\sigma}},$$ where $0<\sigma\leq 1$ comes from LDT.
It follows that
$$
|L(E)-L(E_0)|\le 2C^{N}|E-E_0|+Ce^{-N^{\sigma}}.
$$

If $0<\sigma<1,$ then it does not even guarantee the H\"older continuity, see \cite{wz1}. If $\sigma=1,$  then it can only lead  to $r$-H\"older continuity with $0<r<\frac12$ small, see \cite{goldsteinschlag} and a recent work in \cite{LWY}.

Note that these results are far from being optimal. The reason lies in the fact that the estimate on the upper bound for $L'_{N}(E)$ is far from optimal. Clearly, the upper bound for $L'_{N}(E)$ only depends on $\left\|\frac{\partial_E\|A_{N}(x,E)\|}{\|A_{N}(x,E)\|}\right\|_{\mathcal{L}^1(\R/\Z)}$, which
should be much smaller than $\|\frac{\partial_E\|A_{N}(x,E)\|}{\|A_{N}(x,E)\|}\|_{\mathcal{C}^0(\R/\Z)}$. Indeed, the set of bad $x$ satisfying $\left|\frac{\partial_E\|A_{N}(x,E)\|}{\|A_{N}(x,E)\|}\right|\approx C^{N}$ is of a small measure depending on $N$. We will start from this observation in this paper.

In fact, we will show that the resonance is responsible (see part II in (\ref{qheshi})) for the appearance of the set of bad $x$ and then leads to regularity of the LE not better than $\frac12$-H\"older continuous (recall that it is pointed out  in \cite{sinai} and \cite{wz2} the resonance is also responsible for the appearance of the gaps at least for large coupling cases). Then we will assign  each gap a label related to the resonance, see Theorem \ref{15}. The label can  provide a very precise information on $\frac{\partial_E \|A_{N}\|}{\|A_{N}\|}$ for almost every $x$ which is not available in \cite{wz1}, see Lemma \ref{lemma18}.  Thus we will be able to obtain a sharp upper bound  on $L'_{N}(E)$ with only a  {\it weak} LDT as in \cite{wz1}, which is sufficient to prove $\frac{1}{2}$-H\"older continuity and absolute continuity of the LE.

More precisely, in {\bf EP} case, due to the resonance, the measure of `bad' $x$ is much larger than other cases, which leads to a bad regularity. Indeed, (1) and (2) of Lemma \ref{lemma18} show that for each gap $(E_-, E_+)$ the local shape of the LE at $E=E_-$ on $(E_-, \frac{E_-+E_+}{2})$ is like that of $\sqrt{E}$ at $E=0$ on $E\ge 0$, which is of $\frac12$-H\"older continuity. Due to the Cantor spectrum result in \cite{wz2}, the set of {\bf EP} is dense in the spectrum. For other spectral points $E$, the regularity of the LE is determined by
the speed of approximation to $E$ by the set of {\bf EP} according to labels of {\bf EP}.
Due to Theorem \ref{15}, the measure of the union of the gaps is much smaller than that of the spectrum. Thus for almost every spectral point, the speed of approximation is slow enough, which leads to
differentiability. There remain some $E$ of measure zero such that each of them is approximated by the set of {\bf EP} with
a fast speed, which corresponds to a regularity between $\frac12$-H\"older continuity and differentiability.

The remaining part of the paper is as follows.  Section 2 provides some basic preparations which include some technical lemmas from \cite{wz1}. In Section 3, we present the structure of the spectrum. Then we prove the main Theorem in Section 4 based on a key lemma. In the remaining sections, we focus on the proof of the key lemma. The letters $C,\ C^*$ and $c$ will denote universal constants satisfying $0<c<1<C,\  C^*$ depending only on the potential $v$ and the frequency $\alpha$. Moreover, the letter $\hat{\epsilon}$ denotes a universal constant satisfying $0<\hat{\epsilon}\ll 1$ and in particular $\hat{\epsilon}\ll c<1$. We emphasize that \( C_{\alpha}>0 \) depends only on \( \alpha \).

\subsection{Definitions and Notations}
In this subsection, we provide a comprehensive list of definitions and notations used throughout the paper. This will help the reader to quickly reference the meaning of terms and concepts used in the subsequent sections.

\begin{Notation}[Rotation matrix]
For $\theta \in \R,$ let
$$
R_{\theta}=\begin{pmatrix}
\cos{\theta}&-\sin{\theta}\\
\sin{\theta}&\cos{\theta}
\end{pmatrix}
\in SO(2,\R).
$$
\end{Notation}

\begin{Notation}[Diagonal $SL(2,\R)$ matrix]
For $x\not=0$, denote $\Lambda(x)=\begin{pmatrix} x& 0 \\ 0 &x^{-1}
\end{pmatrix}.$
\end{Notation}

\begin{Notation}[Derivative]
For \(m = (m_1, \ldots, m_s)\) with \(m_i \geq 0\) and \(|m| := \sum_{i=1}^s m_i \in \mathbb{N}\) and  a function \(F = F(x_1, \ldots, x_s)\), let
$$
D^m F = \frac{\partial^{|m|} F}{\partial x_1^{m_1} \cdots \partial x_s^{m_s}}.
$$
Moreover, for any $F(x,y)\in C^2(I)$ with $I\subset \R^2$ a compact domain and $1\leq j\leq 2,$ we denote $$(F(x,y))^{(j)}=\sum\limits_{a,b\ge 0;~a+b=j}\left\vert\frac{\partial^j F}{\partial x^a \partial y^b}\right\vert;$$ $$(F_x(x,y))^{(j)}=\left\vert\frac{\partial^j F}{\partial x^j}\right\vert;~(F_y(x,y))^{(j)}=\left\vert\frac{\partial^j F}{\partial y^j}\right\vert.$$

\end{Notation}

\begin{definition}[Equivalent relation]\label{dominator}
Consider \(0 < \epsilon \ll 1\), \(l, s \in \mathbb{N}\), intervals \(I_i \subset \mathbb{R}\) or \(\R/\Z\), \(i = 1, 2, \ldots, s\), and two functions \(F_1, F \in C^l(I)\) where \(I = I_1 \times I_2 \times \cdots \times I_s\). We say that
$$
F \sim_{l, \epsilon} F_1 \text{ on } I
$$
if for any \(m = (m_1, \ldots, m_s)\) with \(m_i \geq 0\) and \(|m| \leq l\), the following conditions hold:

\begin{enumerate}
\item[i:] \(\{x \in I \mid D^m F_1 = 0\} = \{x \in I \mid D^m F = 0\}\);

\item[ii:]  For each $x\in I$,
$$
 \left|  {D^m F(x)} - {D^m F_1(x)}\right| \leq  \epsilon\cdot (|{D^m F_1(x)}|+|{D^m F(x)}|).
$$

\end{enumerate}

Here we say \(I\) is the domain of the relation \(\sim_{l, \epsilon}\). In the following, if not otherwise specified, the domain of the relation \(\sim_{l, \epsilon}\) refers to the domain of the independent variables.
\end{definition}

\begin{definition}[Projection operator]
For a rectangle  $D=X\times Y\subset \mathbb{R}^2$, we denote $\Pi_1 D:=X$ and $\Pi_2 D:=Y$.
\end{definition}
\begin{Notation}For an interval $I=(b-a,b+a)$ with $0<a<1$, we denote $(b-a^C,b+a^C)$ by $I^C$.
For two intervals $I$ and $J,$ $I^C\times J^C$ is denoted by $(I\times J)^C$.
\end{Notation}
\begin{Notation}
$B(t,\lambda)=(t-\lambda,\ t+\lambda)$.
\end{Notation}
\begin{Notation} For $x\in \R$ and $A,\ B\subset \R$, dist$(x, A)$ denotes\ the \ distance\ between\ $x$ \ and\ $A$, and
dist$(A, B)$ denotes\ the \ distance\ between\ $A$ \ and\ $B$.
\end{Notation}

\begin{Notation}
Let $\{\frac{p_n}{q_n}\}_{n\geq1}$ be the continued fraction approximants of an irrational number $\alpha.$
 For $k\in\Z,$ we denote $s(|k|)\in \Z_+$ satisfying \begin{equation}\label{sk_df}q^2_{N+s(|k|)-1}\leq |k|<q^2_{N+s(|k|)}.\end{equation}
\end{Notation}

\begin{Notation}
For  a Lebesgue measurable set $A\subset \R$, the  measure of $A$ (if is nonzero) is denoted by $|A|$.
For a finite set $S^0$, the number of the elements in $S$ is denoted by $|S|$.
\end{Notation}

\begin{Notation}[Abbreviate notations]
{\bf LE}=Lyapunov exponent,\quad {\bf FLE}=finite Lyapunov exponent,\quad {\bf LDT}=Large deviation theorem,\quad
{\bf AP}=Avalanche Principle, \quad {\bf EP}=endpoints of spectral gaps,\quad{$\mathcal{UH}$}=uniform hyperbolicity,\quad{$\mathcal{FR}$}=finite resonance,\quad{ $\mathcal{IR}$}=infinite resonance,\ {\bf IDS}=integrated density states,\ {\bf AMO}=Almost Mathieu operators.
\end{Notation}

\section{Preliminaries}\label{preli}


\begin{lemma}[\cite{z1}]\label{lemma4}
Assume the potential $v$ is a $C^2$ $\cos$-type function and $\alpha$ is an irrational number. Then
the Schr\"odinger cocycle $(\alpha,A^{(E-\lambda v)})$ is conjugate to the cocycle $(\alpha, A)$ with
$$A(x,t, \lambda)= \Lambda(\|A(x,t,\lambda)\|)\cdot R_{\frac\pi2-\phi(x,t,\lambda)},
$$
where \begin{equation}\label{oc} t=E/\lambda,\quad c \lambda \leq \|A(x, t, \lambda)\| \leq C \lambda, \quad \left|\partial_x^j \|A(x, t, \lambda)\|\right| \leq C \lambda, \quad j = 1, 2,\end{equation}
and $\tan \phi(x,t,\lambda)\rightarrow t-v(x-\alpha)$ in $C^2$-topology as $\lambda\rightarrow \infty$. Thus \(\phi\) is also a \(\cos\)-type function in \(x\) for large \(\lambda\).

\end{lemma}

Let $\lambda$ be fixed and define
\begin{equation}\label{oc1}A(x,t)=\Lambda(\|A(x+\alpha,t,\lambda)\|)\cdot R_{\frac\pi2-\phi(x, t, \lambda)}.\end{equation} To consider the H\"older regularity of the LE $L(t)$ of $A(x,t)$, we may restrict $t$ to the following interval:
$$t\in [\inf v-\frac{2}{\lambda},\sup v+\frac{2}{\lambda}].$$
It is due to the fact that if $t_0\notin[\inf v-\frac{2}{\lambda},\sup v+\frac{2}{\lambda}]$, then $(\alpha,A(\cdot,t_0))$ is uniformly hyperbolic, which implies that $L(t)$ is differential at $t_0$ if $v$ does, see \cite{z1} for details.

\subsection {Some basic lemmas for the sharp estimate on the derivative of finite LE}
In this subsection, we will provide some lemmas based on which we will prove the sharp estimate on the derivative of finite LE later.




\begin{lemma}\label{lemma8}Let $\lambda_1,\lambda_2\gg 1~and~\theta\neq \frac{\pi}{2}.$ Consider the matrix
$$A(\lambda_1, \lambda_2,\theta):=
\begin{pmatrix}
\lambda_1&0\\
0&{\lambda_1}^{-1}
\end{pmatrix}\begin{pmatrix}
\cos{\theta}&-\sin{\theta}\\
\sin{\theta}&\cos{\theta}
\end{pmatrix} \begin{pmatrix}
\lambda_2&0\\
0&{\lambda_2}^{-1}
\end{pmatrix}.$$
Then, we have
$$\begin{array}{ll}
(a)&\frac{\partial\|A\|}{\partial\theta}\cdot\frac{1}{\|A\|}\sim_{0,6\sqrt{\lambda_1^{-4}+\lambda_2^{-4}}}W(\lambda_1,\lambda_2,\theta),\\
(b)& \left\vert\frac{\partial\|A\|}{\partial\lambda_i}\cdot\frac{1}{\|A\|}\right\vert\leq \frac{1}{\lambda_i},\quad i=1, \ 2,
\end{array}$$

where $W(\lambda_1,\lambda_2,\theta)=\frac{sgn(\cot\theta)}{\sqrt{\cot^2\theta+\cot^{-2}\theta(\frac{1}{\lambda_1^4}-\frac{1}{\lambda_2^4})^2+(\frac{2}{\lambda_1^4}+\frac{2}{\lambda_2^4})}}$ if $\theta\not=\pi$ {\rm(let\ } $W=0$ {\rm\ if\ } $\theta=\pi${\rm).}

Moreover,~if~$\lambda_1,~\lambda_2,~\theta$~are~$C^1$-functions~of~$X=(x_1,\cdots,x_n)\in \R^n$~and~$A(X)=A(\lambda_1(X), ~\lambda_2(X),\theta(X))$~,~then
\begin{equation}\label{lemma4-3}\left|{\partial_{x_i} \log \|A(x_1,\cdots,x_n)\|}\right|\leq \left\vert\lambda_1^{-1}\cdot{\partial_{x_i} \lambda_1}\right\vert+\left\vert\lambda_2^{-1}\cdot{\partial_{x_i} \lambda_2}\right\vert+\left\vert\|A\|^{-1}\cdot \partial_{\theta}\|A\|\cdot \partial_{x_i}\theta\right\vert,\ ~1\le i\le n.\end{equation}

\end{lemma}

\begin{proof}
It follows from a direct calculation that
\begin{equation}\label{lemma8,1}\|A\|^2+\|A\|^{-2}=trace(A^TA)=(\lambda_1^2\lambda_2^2+\lambda_1^{-2}\lambda_2^{-2})\cos^2\theta+(\lambda_1^2\lambda_2^{-2}+\lambda_1^{-2}\lambda_2^{2})\sin^2\theta.\end{equation}

Note that
\eqref{lemma8,1} implies that $\|A\|=1$ if and only if $\lambda_1=\lambda_2$ and $\theta=\frac{\pi}{2}.$ Therefore, for the case $\theta\neq \frac{\pi}{2},$ we always have $\|A\|>1.$ By taking derivatives with respect to $\theta$ on both sides of the above equation, we have
$$2\|A\|\frac{\partial\|A\|}{\partial\theta}-2\frac{\partial\|A\|}{\partial\theta}\|A\|^{-3}=(\lambda_1^2\lambda_2^{-2}+\lambda_1^{-2}\lambda_2^{2}-(\lambda_1^2\lambda_2^2+\lambda_1^{-2}\lambda_2^{-2}))\sin2\theta.$$
Note $0\neq \|A\|^2-\|A\|^{-2}=\sqrt{(\|A\|^2+\|A\|^{-2})^2-4}$, hence
$$\frac{\partial\|A\|}{\partial\theta}\cdot\frac{1}{\|A\|}=\frac{(\lambda_1^2\lambda_2^{-2}+\lambda_1^{-2}\lambda_2^{2}-(\lambda_1^2\lambda_2^2+\lambda_1^{-2}\lambda_2^{-2}))\sin2\theta}{2\sqrt{((\lambda_1^2\lambda_2^2+\lambda_1^{-2}\lambda_2^{-2})\cos^2\theta+(\lambda_1^2\lambda_2^{-2}+\lambda_1^{-2}\lambda_2^{2})\sin^2\theta)^2-4}}.$$
 Equivalently, we have$$ \frac{\partial\|A\|}{\partial\theta}\cdot\frac{1}{\|A\|}=\frac{sgn(\sin(2\theta))(1-\lambda_2^{-4})(1-\lambda_1^{-4})}
 {\sqrt{(1-\frac{1}{\lambda_1^4\lambda_2^4})^2\cot^2\theta+(\frac{1}{\lambda_1^4}-\frac{1}
 {\lambda_2^4})^2\tan^2\theta+2(1+\frac{1}{\lambda_1^4\lambda_2^4})(\frac{1}{\lambda_1^4}+
 \frac{1}{\lambda_2^4})-\frac{8}{\lambda_1^4\lambda_2^4}}}.$$
On one hand, it implies \begin{equation}\label{7}\begin{array}{ll}
\left\vert\frac{\partial\|A\|}{\partial\theta}\cdot\frac{1}{\|A\|}\right\vert &\geq\frac{(1-\frac{1}{\lambda_2^4})(1-\frac{1}{\lambda_1^4})}{\sqrt{1+\frac{1}{\lambda_1^4\lambda_2^4}}}
\cdot(\cot^2\theta+(\frac{1}{\lambda_1^4}-\frac{1}{\lambda_2^4})^2\tan^2\theta+\frac{2}
{\lambda_1^4}+\frac{2}{\lambda_2^4})^{-\frac12}\\
&\geq (1-\frac{6}{\lambda_2^4}-\frac{6}{\lambda_1^4})\cdot(\cot^2\theta+(\frac{1}{\lambda_1^4}-\frac{1}{\lambda_2^4})^2\tan^2\theta+\frac{2}
{\lambda_1^4}+\frac{2}{\lambda_2^4})^{-\frac12}.
\end{array}\end{equation}

On the other hand,\begin{equation}\label{8}\begin{array}{ll}
\left\vert\frac{\partial\|A\|}{\partial\theta}\cdot\frac{1}{\|A\|}\right\vert &\leq  \frac{1}{\sqrt{(1-\frac{1}{\lambda_1^4\lambda_2^4})^2\cot^2\theta+(\frac{1}{\lambda_1^4}-\frac{1}{\lambda_2^4})^2\tan^2\theta+2(1+\frac{1}{\lambda_1^4\lambda_2^4})(\frac{1}{\lambda_1^4}+\frac{1}{\lambda_2^4})-\frac{8}{\lambda_1^4\lambda_2^4}}}\\
&\leq (1+\frac{6}{\lambda_1^4}+\frac{6}{\lambda_2^4})\frac{1}{\sqrt{\cot^2\theta+(\frac{1}{\lambda_1^4}-\frac{1}{\lambda_2^4})^2\tan^2\theta+\frac{2}{\lambda_1^4}+\frac{2}{\lambda_2^4}}}.
\end{array}\end{equation}

Note $W(\lambda_1,\lambda_2,\theta)\neq 0$ for $\theta\in (0,\frac{\pi}{2})\bigcup (\frac{\pi}{2},\pi)$ and $sgn(\cot \theta)=sgn(\sin 2\theta)$ for $\theta\neq \frac{\pi}{2}.$
Hence \eqref{7} and \eqref{8} yield $$\left\vert\frac{\frac{\partial\|A\|}{\partial\theta}\cdot\frac{1}{\|A\|}}{W(\lambda_1,\lambda_2,\theta)}-1\right\vert=\left\vert \left\vert\frac{\frac{\partial\|A\|}{\partial\theta}\cdot\frac{1}{\|A\|}}{W(\lambda_1,\lambda_2,\theta)}\right\vert-1\right\vert\leq 6\cdot \left(\lambda_1^{-4}+\lambda_2^{-4}\right),~\theta \in (0,\frac{\pi}{2})\bigcup (\frac{\pi}{2},\pi).$$ Note that on $(0,\pi]$, $W(\lambda_1,\lambda_2,\theta)$ and $\frac{\partial\|A\|}{\partial\theta}\cdot\frac{1}{\|A\|}$ have the same zero $\theta=\pi.$ And
$$\lim\limits_{\theta\rightarrow \pi}\frac{\frac{\partial\|A\|}{\partial\theta}\cdot\frac{1}{\|A\|}}{W(\lambda_1,\lambda_2,\theta)}=\frac{(1-\lambda_2^{-4})(1-\lambda_1^{-4})}{1-\frac{1}{\lambda_1^4\lambda_2^4}}.$$

Therefore on $(0,\frac{\pi}{2})\bigcup (\frac{\pi}{2},\pi],$ $$\frac{\partial\|A\|}{\partial\theta}\cdot\frac{1}{\|A\|}\sim_{0,6\sqrt{\lambda_1^{-4}+\lambda_2^{-4}}}W(\lambda_1,\lambda_2,\theta).$$
\

(b) and (\ref{lemma4-3}) can be obtained similarly.
\hfill\qed\end{proof}

\begin{proposition}\label{PROP8}
Given an open rectangle \(I \times J \subset \mathbb{T} \times \mathbb{R}\) and functions \(f(x, y), g(x, y) \in C^1(I \times J)\) with \(0 < \gamma \ll 1\), if \(f \sim_{1, \gamma} g\), then
$$
\arctan f \ (\text{mod} \ \pi) \sim_{1, 4\gamma} \arctan g \ (\text{mod} \ \pi).
$$
\end{proposition}

\begin{proof}
Since \(f \sim_{1, \gamma} g\), we have
$$
\frac{|f - g|}{|g|} \leq \gamma \quad \text{and} \quad \frac{|f' - g'|}{|g'|} \leq \gamma,
$$
where \(f' = \partial_X f\) and \(g' = \partial_X g\) with \(X\in \{x, y\}\).

Assuming without loss of generality that \(\arctan f\) and \(\arctan g\) lie in \((- \frac{\pi}{2}, \frac{\pi}{2}]\), we have
$$
\left| \frac{\arctan f - \arctan g}{\arctan g} \right| \leq \frac{1}{|\arctan g|} |f - g|.
$$

If \(|g| > 1\), then \(|\arctan g| > \frac{\pi}{4}\). Thus
$$
\left| \frac{\arctan f - \arctan g}{\arctan g} \right| \leq \frac{1}{|\arctan g|} |f - g| \leq \frac{4}{\pi} |f - g| \leq \frac{2}{\pi} \gamma.
$$

If \(|g| \leq 1\), then
$
|\arctan g| \geq \frac{\pi}{4} |g|.
$
Hence
$$
\left| \frac{\arctan f - \arctan g}{\arctan g} \right| \leq \frac{1}{|\arctan g|} |f - g| \leq \frac{4}{\pi} \frac{|f - g|}{|g|} \leq \frac{4}{\pi} \gamma.
$$

Therefore
$$
\left| \frac{\arctan f - \arctan g}{\arctan g} \right| \leq \frac{4}{\pi} \gamma.
$$

Similarly,
$$
\left| \frac{\arctan f - \arctan g}{\arctan f} \right| \leq \frac{4}{\pi} \gamma.
$$

Thus for \(x \in I\) and \(y \in J\),
$$
\arctan f \sim_{0, \frac{4}{\pi} \gamma} \arctan g.
$$

On the other hand,
$$
\left| \frac{(\arctan f)'}{(\arctan g)'} - 1 \right| = \left| \frac{(1 + g^2) f'}{(1 + f^2) g'} - 1 \right| \leq \left| \frac{f'}{g'} - 1 \right| + \frac{|f + g|}{|f|} \left| \frac{f - g}{f} \right| \left| \frac{f'}{g'} \right| \leq \gamma + 2 \gamma (1 + \gamma) \leq 4 \gamma.
$$
This completes the proof.
\hfill\qed\end{proof}

\begin{lemma}\label{lemma8*}
Given an open rectangle \(I \times J \subset \mathbb{T} \times \mathbb{R}\), let \(\lambda_1(x, y), \lambda_2(x, y), \theta(x, y) \in C^2(I \times J)\) and define \(A(x,y)=A(\lambda_1(x, y), \lambda_2(x, y), \theta(x, y))\) with
$$
A(\lambda_1, \lambda_2, \theta) :=
\begin{pmatrix}
\lambda_1 & 0 \\
0 & \lambda_1^{-1}
\end{pmatrix}
\begin{pmatrix}
\cos \theta & -\sin \theta \\
\sin \theta & \cos \theta
\end{pmatrix}
\begin{pmatrix}
\lambda_2 & 0 \\
0 & \lambda_2^{-1}
\end{pmatrix}.
$$

Given numbers \(\mu_1, \mu_2 \ge \lambda_0 \gg 1\) and \(0 < \hat{\epsilon} \ll 1\), assume that for \(i = 1, 2\), and \(X, Y\in\{ x, y\}\),
\begin{equation}\label{assumm}
\begin{array}{ll}
&\lambda_i(x, y) \sim_{0, e^{-(\log \lambda_0)^{2\hat{\epsilon}}}} \mu_i,\qquad\qquad
~
|\partial_X \lambda_i| + |\partial^2_{XY} \lambda_i| \leq e^{(\log \lambda_0)^{\hat{\epsilon}}}\cdot \mu_i,\\
~
&|\partial_X \theta| + |\partial^2_{XY} \theta| \leq e^{(\log \lambda_0)^{\hat{\epsilon}}},\qquad\qquad
~
|\partial_X \theta| \geq e^{(\log \lambda_0)^{-\hat{\epsilon}}}.
\end{array}
\end{equation}

Then the following three statements hold true.

\begin{enumerate}
\item[i:] If \(|\tan \theta| < e^{(\log \lambda_0)^{2\hat{\epsilon}}}\), then
\begin{equation}\label{lem8-i}
\|\frac{\pi}{2} - s(A)\|_{C^2}, \|u(A)\|_{C^2} \leq e^{(\log \lambda_0)^{3\hat{\epsilon}}} \lambda_0^{-2}.
\end{equation}

\item[ii:] If \(|\tan \theta| \geq e^{(\log \lambda_0)^{2\hat{\epsilon}}}\) and \(\mu_1 \geq \mu_2^{16}\), then

\begin{enumerate}
\item[2a:]
$$
\|u(A)\|_{C^2} \leq \mu_1^{-\frac{3}{2}}.
$$

\item[2b:] If \(|\tan \theta| \leq \mu_2^2 e^{-(\log \lambda_0)^{2\hat{\epsilon}}}\), then
$$
s(A) \sim_{2, \mu_1^{-1}} \frac{\pi}{2} - \arctan [\mu_2^{-2} \tan \theta] \ (\text{mod} \ \pi).
$$

\item[2c:] If \(|\tan \theta| \geq \mu_2^2 e^{-(\log \lambda_0)^{2\hat{\epsilon}}}\), then
$$
|\frac{\pi}{2} - s(A)| \geq e^{-(\log \lambda_0)^{4\hat{\epsilon}}} \quad \text{and} \quad |\partial_X s(A)| \geq \mu_2.
$$
\end{enumerate}

\item[iii:] If \(|\tan \theta| \geq e^{-(\log \lambda_0)^{2\hat{\epsilon}}}\) and \(\mu_2 \geq \mu_1^{16}\), then

\begin{enumerate}
\item[3a:]
$$
\|\frac{\pi}{2} - s(A)\|_{C^2} \leq \mu_2^{-\frac{3}{2}}.
$$

\item[3b:] If \(|\tan \theta| \leq \mu_1^2 e^{-(\log \lambda_0)^{2\hat{\epsilon}}}\), then
$$
u(A) \  \sim_{2, \mu_2^{-1}} \arctan [\mu_1^{-2} \tan \theta] \ (\text{mod} \ \pi).
$$

\item[3c:] If \(|\tan \theta| \geq \mu_1^2 e^{-(\log \lambda_0)^{2\hat{\epsilon}}}\), then
$$
|u(A)| \geq e^{-(\log \lambda_0)^{4\hat{\epsilon}}} \quad \text{and} \quad |\partial_X u(A)| \geq \mu_1.
$$
\end{enumerate}
\end{enumerate}
\end{lemma}

\begin{proof}
Consider
\[
F(\lambda_1, \lambda_2, \theta, \phi) = A(\lambda_1, \lambda_2, \theta) (\cos \phi, \sin \phi)^T
\]
with
$
A(\lambda_1, \lambda_2, \theta)
$ defined as above,
and
\[
\tilde{F}(\lambda_1, \lambda_2, \theta, \phi) = A^{-1}(\lambda_1, \lambda_2, \theta) (\cos \phi, \sin \phi)^T
\]
with
\[
A^{-1}(\lambda_1, \lambda_2, \theta) =
\begin{pmatrix}
\lambda_2^{-1} & 0 \\
0 & \lambda_2
\end{pmatrix}
\begin{pmatrix}
\cos{\theta} & \sin{\theta} \\
-\sin{\theta} & \cos{\theta}
\end{pmatrix}
\begin{pmatrix}
\lambda_1^{-1} & 0 \\
0 & \lambda_1
\end{pmatrix}
.
\]

A direct calculation yields
\[
\begin{aligned}
\|F\|^2 &= \left(\lambda_1^2\lambda_2^2\cos^2\theta + \lambda_2^2\lambda_1^{-2}\sin^2\theta \right)\cos^2 \phi + \left(\lambda_1^2\lambda_2^{-2}\sin^2\theta + \lambda_1^{-2}\lambda_2^{-2}\cos^2\theta \right)\sin^2 \phi \\
&\quad - (\lambda_1^2 - \lambda_1^{-2})2\cos\theta\sin\theta\cos\phi\sin\phi,
\end{aligned}.
\]
\[
\begin{aligned}
\|\tilde{F}\|^2 &= \left(\lambda_1^{-2}\lambda_2^{-2}\cos^2\theta + \lambda_2^2\lambda_1^{-2}\sin^2\theta \right)\cos^2 \phi + \left(\lambda_1^2\lambda_2^{-2}\sin^2\theta + \lambda_1^{2}\lambda_2^{2}\cos^2\theta \right)\sin^2 \phi \\
&\quad - (\lambda_2^2 - \lambda_2^{-2})2\cos\theta\sin\theta\cos\phi\sin\phi
\end{aligned}.
\]

Then
\begin{align*}
\frac{\partial \|F\|^2}{\partial \phi}
&= \sin(2\phi) \left( \lambda_1^2\lambda_2^{-2}\sin^2\theta + \lambda_1^{-2}\lambda_2^{-2}\cos^2\theta - \lambda_1^2\lambda_2^2\cos^2\theta - \lambda_2^2\lambda_1^{-2}\sin^2\theta \right) \\
&\quad - 2\sin\theta\cos\theta (\lambda_1^2 - \lambda_1^{-2}) \cos(2\phi). \nonumber
\end{align*}
\begin{align*}
\frac{\partial \|\tilde{F}\|^2}{\partial \phi}
&= \sin(2\phi) \left( \lambda_1^2\lambda_2^{-2}\sin^2\theta + \lambda_1^{2}\lambda_2^{2}\cos^2\theta - \lambda_1^{-2}\lambda_2^{-2}\cos^2\theta - \lambda_2^2\lambda_1^{-2}\sin^2\theta \right) \\
&\quad - 2\sin\theta\cos\theta (\lambda_2^2 - \lambda_2^{-2}) \cos(2\phi). \nonumber
\end{align*}

Note that
\begin{equation}\label{1dd0}
\frac{\partial \|F\|^2}{\partial \phi}\left|_{\phi=s(A)}\right.= 2\|F\| \frac{\partial \|F\|}{\partial \phi}\left|_{\phi=s(A)}\right.= 0,
\end{equation}
and similarly
\[
\frac{\partial \|\tilde{F}\|^2}{\partial \phi}\left|_{\phi=u(A)}\right.= 2\|\tilde{F}\| \frac{\partial \|\tilde{F}\|}{\partial \phi}\left|_{\phi=u(A)}\right.= 0.
\]

Here we only show (ii-2b), (ii-2c), (iii-3a) and (i), the remaining cases can be considered similarly (replace \( s(A) \) with \( u(A) \) and \( F \) with \( \tilde{F} \)).

\

Next we consider the case \begin{equation}\label{new18}e^{(\log \lambda_0)^{2\hat{\epsilon}}}\leq |\tan\theta|.\end{equation}
By \eqref{1dd0}, we obtain
\begin{align*}
&\sin(2s(A))(\lambda_1^2 \lambda_2^{-2} \sin^2\theta + \lambda_1^{-2} \lambda_2^{-2} \cos^2\theta - \lambda_1^2 \lambda_2^2 \cos^2\theta - \lambda_2^2 \lambda_1^{-2} \sin^2\theta)\\
&= 2 \sin\theta \cos\theta (\lambda_1^2 - \lambda_1^{-2}) \cos(2s(A)).\end{align*}

A direct calculation yields
\begin{equation} \label{saaa}
\begin{aligned}
\tan(2s(A)) &= -\frac{2 \tan \theta (1-\lambda_1^{-4})}{(\lambda_2^2 \lambda_1^{-4} - \lambda_2^{-2}) \tan^2 \theta + \lambda_2^2 - \lambda_1^{-4} \lambda_2^{-2}}= -\frac{2 \lambda_2^{-2} \tan \theta \frac{(1 - \lambda_1^{-4})}{1 - \lambda_1^{-4} \lambda_2^{-4}}}{\frac{(\lambda_2^2 \lambda_1^{-4} - \lambda_2^{-2})}{\lambda_2^2 - \lambda_1^{-4} \lambda_2^{-2}} \tan^2\theta + 1} \\
&= -\frac{2 \left[\frac{\sqrt{(\lambda_2^{-4} - \lambda_1^{-4})}}{\sqrt{1 - \lambda_1^{-4} \lambda_2^{-4}}} \tan \theta \right]}{1 - \left[\frac{(\lambda_2^{-4} - \lambda_1^{-4})}{1 - \lambda_1^{-4} \lambda_2^{-4}} \tan^2\theta\right]} \left[\frac{(1 - \lambda_1^{-4})}{\sqrt{1 - \lambda_1^{-4} \lambda_2^4} \sqrt{1 - \lambda_1^{-4} \lambda_2^{-4}}}\right] \\
&= -\tan(2(\arctan [\frac{\sqrt{(\lambda_2^{-4} - \lambda_1^{-4})}}{\sqrt{1 - \lambda_1^{-4} \lambda_2^{-4}}} \tan\theta])) \left[\frac{(1 - \lambda_1^{-4})}{\sqrt{1 - \lambda_1^{-4} \lambda_2^4} \sqrt{1 - \lambda_1^{-4} \lambda_2^{-4}}}\right]
\\&=-\tan(2(\arctan [(1+O(\lambda_2^4\lambda_1^{-4}))\lambda_2^{-2}\tan\theta]))[1+O(\lambda_2^4\lambda_1^{-4})].
\end{aligned}
\end{equation}

Note
$$
\begin{aligned}
& \partial_X(\lambda_2^{-2} \tan \theta) = -2 \lambda_2^{-3} (\partial_X \lambda_2) \tan \theta + \lambda_2^{-2} (1 + \tan^2 \theta) \partial_X \theta \\
& = \lambda_2^{-2} (1 + \tan^2 \theta) \partial_X \theta \left(1 - \frac{2 (\partial_X \lambda_2) \tan \theta}{\lambda_2 (1 + \tan^2 \theta) \partial_X \theta}\right).
\end{aligned}
$$

By \eqref{assumm} and \eqref{new18}, we have
\[
\left\vert \frac{2 \partial_X \lambda_2 \tan \theta}{\lambda_2 (1 + \tan^2 \theta) \partial_X \theta} \right\vert \leq \left\vert \frac{2 \partial_X \lambda_2}{\lambda_2 \tan \theta \partial_X \theta} \right\vert \leq \left\vert \frac{2 e^{(\log \lambda_0)^{\hat{\epsilon}}}}{e^{(\log \lambda_0)^{2\hat{\epsilon}} - (\log \lambda_0)^{\hat{\epsilon}}}} \right\vert \leq e^{-(\log \lambda_0)^{\hat{\epsilon}}},
\]
which implies
\begin{equation} \label{1d}
\partial_X(\lambda_2^{-2} \tan \theta) \sim_{0, e^{-(\log \lambda_0)^{\hat{\epsilon}}}} \lambda_2^{-2} (1 + \tan^2 \theta) \partial_X \theta.
\end{equation}

Similarly,
$$
\begin{aligned}
& \frac{\partial^2(\lambda_2^{-2} \tan \theta)}{\partial X \partial Y} = 6 \lambda_2^{-4} (\partial_X \lambda_2)(\partial_Y \lambda_2) \tan \theta - 2 \lambda_2^{-3} \frac{\partial^2 \lambda_2}{\partial X \partial Y} \tan \theta \\
& - 2 \lambda_2^{-3} (\partial_X \lambda_2)(1 + \tan^2 \theta) \partial_Y \theta - 2 \lambda_2^{-3} (\partial_Y \lambda_2)(1 + \tan^2 \theta) \partial_X \theta \\
& + 2 (\lambda_2^{-2} (1 + \tan^2 \theta) \tan \theta \partial_X \theta \partial_Y \theta) +  (\lambda_2^{-2} (1 + \tan^2 \theta) \frac{\partial^2 \theta}{\partial X \partial Y}) \\
& = 2 (\lambda_2^{-2} (1 + \tan^2 \theta) \tan \theta \partial_X \theta \partial_Y \theta) \cdot (1 + R_{XY}),
\end{aligned}
$$
with
\begin{align*}
R_{XY} &= \frac{
    3 \lambda_2^{-4} (\partial_X \lambda_2)(\partial_Y \lambda_2)
    - \lambda_2^{-3} \frac{\partial^2 \lambda_2}{\partial X \partial Y} \tan \theta
    - \lambda_2^{-3} (1 + \tan^2 \theta) \left( (\partial_X \lambda_2) \partial_Y \theta + (\partial_Y \lambda_2) \partial_X \theta \right)
}{2 \lambda_2^{-2} (1 + \tan^2 \theta) \tan \theta \partial_X \theta \partial_Y \theta} \\
&\quad + \frac{
    \lambda_2^{-2} (1 + \tan^2 \theta) \frac{\partial^2 \theta}{\partial X \partial Y}
}{2 \lambda_2^{-2} (1 + \tan^2 \theta) \tan \theta \partial_X \theta \partial_Y \theta}.
\end{align*}

A direct computation with \eqref{assumm} and \eqref{new18} yields
$
\left\vert R_{XY} \right\vert \leq e^{-(\log \lambda_0)^{\hat{\epsilon}}}.
$
Therefore for \(X, Y\in  \{x, y\}\),
\begin{equation} \label{2d}
\frac{\partial^2 (\lambda_2^{-2} \tan \theta)}{\partial X \partial Y} \sim_{0, e^{-(\log \lambda_0)^{\hat{\epsilon}}}} 2 \lambda_2^{-2} (1 + \tan^2 \theta) \tan \theta \partial_X \theta \partial_Y \theta.
\end{equation}

From \eqref{1d} and \eqref{2d}, we obtain
\begin{equation} \label{1212d}
\frac{\frac{\partial^2 (\lambda_2^{-2} \tan \theta)}{\partial X \partial Y}}{\partial_X (\lambda_2^{-2} \tan \theta)} \sim_{0, e^{-(\log \lambda_0)^{\hat{\epsilon}}}} 2 \tan \theta \partial_Y \theta, \quad \frac{\partial_X (\lambda_2^{-2} \tan \theta)}{\lambda_2^{-2} \tan \theta} \sim_{0, e^{-(\log \lambda_0)^{\hat{\epsilon}}}} (\tan \theta + \cot \theta) \partial_X \theta\sim_{0, e^{-(\log \lambda_0)^{\hat{\epsilon}}}}  (\tan \theta) \partial_X \theta.
\end{equation}
\textbf{Proof of (ii-2b):}

In this case, we have
\begin{equation}\label{assumiib}
 e^{(\log \lambda_0)^{2\hat{\epsilon}}}\leq |\tan\theta| \leq \mu_2^{2} e^{-(\log \lambda_0)^{2\hat{\epsilon}}}.
\end{equation}

By \eqref{assumiib} and $\mu_1\geq \mu_2^{16}$, one notes that
\begin{equation} \label{lambda1212}
\left\vert \partial_X(\frac{\lambda_2^4}{\lambda_1^4})\right\vert +\left\vert \frac{\partial^2(\frac{\lambda_2^4}{\lambda_1^4})}{\partial X \partial Y}\right\vert \leq \mu_1^{-3} \ll \mu_1^{-2}e^{(\log \lambda_0)^{2\hat{\epsilon}}} \leq \lambda_2^{-2}|\tan \theta|,
\end{equation}
which, together with \eqref{1212d}, yields that
\begin{equation} \label{daO}
(1 + O(\lambda_2^4 \lambda_1^{-4})) \lambda_2^{-2} \tan \theta \sim_{2, \mu_1^{-1}} \lambda_2^{-2} \tan \theta.
\end{equation}

Next, we will show
\begin{equation} \label{arctan-eq}
\arctan [(1 + O(\lambda_2^4 \lambda_1^{-4})) \lambda_2^{-2} \tan \theta] \sim_{2,  \mu_1^{-1}} \arctan[\lambda_2^{-2} \tan \theta].
\end{equation}

\textbf{Proof of \eqref{arctan-eq}}:
By Proposition \ref{PROP8}, \eqref{daO} implies
\[
\arctan \left[(1 + O(\lambda_2^4 \lambda_1^{-4})) \lambda_2^{-2} \tan\theta\right] \sim_{1, \mu_1^{-1}} \arctan[\lambda_2^{-2} \tan\theta].
\]
Therefore, we only need to show that for \( X, Y \in\{x, y\} \) we have
\begin{equation} \label{part-XY}
\frac{\partial^2 \arctan \left[(1 + O(\lambda_2^4 \lambda_1^{-4})) \lambda_2^{-2} \tan\theta\right]}{\partial X \partial Y} \sim_{0, \mu_1^{-1}} \frac{\partial^2 \arctan[\lambda_2^{-2} \tan\theta]}{\partial X \partial Y}.
\end{equation}

Note that by \eqref{2d} and the fact $|\lambda_2^{-4} \tan^2 \theta|\leq e^{-(\log \lambda_0)^{\hat{\epsilon}}}$, we have
\[
\begin{aligned}
& (1 + \lambda_2^{-4} \tan^2 \theta) \frac{\partial^2 (\lambda_2^{-2} \tan\theta)}{\partial X \partial Y} \sim_{0, e^{-(\log \lambda_0)^{\hat{\epsilon}}}} 2 \lambda_2^{-2} (1 + \tan^2 \theta) \tan \theta \partial_X \theta \partial_Y \theta.
\end{aligned}
\]

By \eqref{1d}, we also have
\[
2 (\partial_X (\lambda_2^{-2} \tan \theta)) (\partial_Y (\lambda_2^{-2} \tan \theta)) (\lambda_2^{-2} \tan \theta) \sim_{0, e^{-(\log \lambda_0)^{\hat{\epsilon}}}} 2 \lambda_2^{-6} (1 + \tan^2 \theta)^2 \tan \theta \partial_X \theta \partial_Y \theta.
\]

Thus we can conclude
\[
\left\vert \frac{2 (\partial_X (\lambda_2^{-2} \tan \theta)) (\partial_Y (\lambda_2^{-2} \tan \theta)) (\lambda_2^{-2} \tan \theta)}{(1 + \lambda_2^{-4} \tan^2 \theta) \frac{\partial^2 (\lambda_2^{-2} \tan \theta)}{\partial X \partial Y}} \right\vert \leq C \lambda_2^{-4} (1 + \tan^2 \theta) \leq e^{-(\log \lambda_0)^{\hat{\epsilon}}}.
\]

Therefore
\begin{equation} \label{jinsihaofan}
\begin{aligned}
&(1 + \lambda_2^{-4} \tan^2 \theta) \frac{\partial^2 (\lambda_2^{-2} \tan \theta)}{\partial X \partial Y} - 2 (\partial_X (\lambda_2^{-2} \tan \theta)) (\partial_Y (\lambda_2^{-2} \tan \theta)) (\lambda_2^{-2} \tan \theta)\\& \sim_{0, e^{-(\log \lambda_0)^{\hat{\epsilon}}}} \frac{\partial^2 (\lambda_2^{-2} \tan \theta)}{\partial X \partial Y}.
\end{aligned}
\end{equation}

On the other hand, \eqref{1212d}, \eqref{assumiib} and \( \mu_1 > \mu_2^8, \) we have
\begin{equation} \label{jqxj}
\left\vert \frac{\partial^2 (\lambda_2^{-2} \tan \theta)}{\partial X \partial Y} \right\vert, \quad \left\vert \frac{\partial (\lambda_2^{-2} \tan \theta)}{\partial X} \right\vert, \quad \left\vert \partial_{\theta} (\lambda_2^{-2} \tan \theta) \right\vert \geq \lambda_2^{-1} \gg \lambda_1^{-1}.
\end{equation}

A direct calculation yields that for \( X \in\{ x, y\} \)
\begin{equation} \label{lamalam2}
\begin{aligned}
& \left\vert \partial_X \left[1 + \frac{\lambda_2^4}{\lambda_1^4}\right] \right\vert + \left\vert \partial^2_{XY} \left[1 + \frac{\lambda_2^4}{\lambda_1^4}\right] \right\vert \cdot \left( |\lambda_2^{-2} \tan \theta| + |\partial_x (\lambda_2^{-2} \tan \theta)| + |\partial_y (\lambda_2^{-2} \tan \theta)| + |\frac{\partial (\lambda_2^{-2} \tan \theta)}{\partial x \partial y}|\right) \\
& \ll \lambda_1^{-1}.
\end{aligned}
\end{equation}

Then, using the fact
\[
(\arctan f)'' = \frac{f''(1 + f^2) - 2 (f')^2 f}{(1 + f^2)^2}
\]
and applying \eqref{jinsihaofan}, \eqref{jqxj}, and \eqref{lamalam2}, we obtain \eqref{part-XY}, which implies \eqref{arctan-eq}.

\

Since
\[
(\tan f)' = (1 + \tan^2 f) f', \quad (\tan f)'' = (1 + \tan^2 f) f'' + 2 \tan f (1 + \tan^2 f) (f')^2,
\]
by a similar computation as above (note that when \( |f| < e^{-(\log \lambda_0)^{\hat{\epsilon}}}, \) we have \( \tan f \sim_{0, e^{-(\log \lambda_0)^{\hat{\epsilon}}}} f \)), we obtain
\[
(1 + O(\lambda_2^4 \lambda_1^{-4})) \tan \left\{\arctan \left[ 2 (1 + O(\lambda_2^4 \lambda_1^{-4})) \lambda_2^{-2} \tan \theta \right] \right\} \sim_{2, \mu_1^{-1}} \tan \left\{ 2 \arctan [\lambda_2^{-2} \tan \theta] \right\}.
\]

By the above estimate and \eqref{saaa}, we have
\begin{equation}\label{saaa*}
\tan(-2s(A)) \sim_{2, \mu_1^{-1}} \tan \left\{ 2 \arctan[\lambda_2^{-2} \tan \theta] \right\}.
\end{equation}

Next we will show \[
s(A) \ (\text{mod} \ \pi) \sim_{2, \mu_1^{-1}} \frac{\pi}{2} - \arctan[\lambda_2^{-2} \tan \theta] \ (\text{mod} \ \pi).
\]

By the help of Lemma \ref{PROP8}, \eqref{saaa*} implies for $j=0,1,$ \begin{equation}\label{saaa*'}
-2s(A) \ (\text{mod} \ \pi) \sim_{j, \mu_1^{-1}} \arctan[\lambda_2^{-2} \tan \theta] \ (\text{mod} \ \pi).
\end{equation}

Note that \eqref{1d} and \eqref{2d} imply for $X,Y\in~\{x,y\}$, it holds that $$
\frac{\partial^2 (\lambda_2^{-2} \tan \theta)}{\partial X \partial Y} \sim_{0, e^{-(\log \lambda_0)^{\hat{\epsilon}}}} 2 \lambda_2^{-2} \tan^3 \theta\partial_X \theta \partial_Y \theta,\quad
\frac{\partial (\lambda_2^{-2} \tan \theta)}{\partial X} \sim_{0, e^{-(\log \lambda_0)^{\hat{\epsilon}}}} \lambda_2^{-2} \tan^2 \theta\partial_X \theta.
$$

Then setting $Q:=\lambda_2^{-2} \tan \theta$ (note $|Q|<e^{-(\log \lambda_0)^{\hat{\epsilon}}}$) and taking $X=Y=x$ yield that
$$\frac{2(Q')^2Q}{(1+Q^2)Q''}\leq 5\frac{2 \lambda_2^{-6} \tan^5 \theta (\partial_x \theta)^2}{2\lambda_2^{-2} \tan^3 \theta (\partial_x \theta)^2}\leq 5(\lambda^{-2}\tan\theta)^2\leq 5Q^2\leq 5e^{-2(\log \lambda_0)^{\hat{\epsilon}}}.$$
Then
\begin{equation}\label{arctanqqq}\left\vert \frac{(\arctan Q)''(1+Q^2)}{Q''} -1\right\vert=\left\vert \frac{2(Q')^2Q}{(1+Q^2)Q''} \right\vert\leq 5e^{-2(\log \lambda_0)^{\hat{\epsilon}}}.\end{equation}

Hence it holds that
$$(\arctan Q)''\sim_{0,e^{-2(\log \lambda_0)^{\hat{\epsilon}}}} \frac{Q''}{1+Q^2}\sim_{0,e^{-2(\log \lambda_0)^{\hat{\epsilon}}}} Q'',$$
$$(\arctan Q)'= \frac{Q'}{1+Q^2}\sim_{0,e^{-2(\log \lambda_0)^{\hat{\epsilon}}}} Q'$$
and
$$(\arctan Q)\sim_{0,e^{-2(\log \lambda_0)^{\hat{\epsilon}}}} Q~(~\text{by}~|Q|<e^{-(\log \lambda_0)^{\hat{\epsilon}}}).$$

Thus
$$(\arctan Q)\sim_{2,e^{-(\log \lambda_0)^{\hat{\epsilon}}}} Q.$$
Similarly, using $(\tan f)'' = (1 + \tan^2 f) f'' + 2 \tan f (1 + \tan^2 f) (f')^2$ with a direct calculation and setting $U=\tan(2\arctan Q)$ imply that
$$U''\sim_{0,e^{-2(\log \lambda_0)^{\hat{\epsilon}}}} 2Q'';~U'\sim_{0,e^{-2(\log \lambda_0)^{\hat{\epsilon}}}} 2Q';~U\sim_{0,e^{-2(\log \lambda_0)^{\hat{\epsilon}}}} 2Q,$$
which implies
$U\sim_{2,e^{-2(\log \lambda_0)^{\hat{\epsilon}}}} 2Q.$
Thus by \eqref{saaa*} and setting $W:=[\tan(-2s(A))]$, we have
\begin{equation}\label{wsimu}W\sim_{2,\mu_1^{-1}} U \sim_{2,e^{-2(\log \lambda_0)^{\hat{\epsilon}}}}2Q.\end{equation}


Note \eqref{arctanqqq} and \eqref{wsimu} implies
$$\begin{array}{ll}&\left\vert \frac{W''(1+W^2)-(W')^2W}{U''(1+U^2)-(U')^2U}-1\right\vert=\left\vert \frac{W''(1+W^2)-(W')^2W-(U''(1+U^2)-(U')^2U)}{U''(1+U^2)-(U')^2U}\right\vert\\
\\&\leq \left\vert \frac{|(\mu_1^{-1})(U''(1+U^2)-(U')^2U)|+|(\mu_1^{-1})(U')^2U|}{U''(1+U^2)-(U')^2U}\right\vert\leq \mu_1^{-1}+\frac{|(\mu_1^{-1})(U')^2 U|}{|U''(1+U^2)-(U')^2U|}\leq 2\mu_1^{-1}(~\text{by}~\eqref{arctanqqq}).\end{array}$$

This implies
$$(1+W^2)^2(\arctan W)''\sim_{0,\mu_1^{-1}} (1+U^2)^2(\arctan U)''.$$

Since \eqref{wsimu} implies $(1+W^2)\sim_{0,\mu_1^{-1}}(1+U^2),$ we immediately get
$(\arctan W)''\sim_{0,\mu_1^{-1}} (\arctan U)''.$ Combining this with \eqref{saaa*'} we obtain \eqref{saaa*} as desired.

Finally, it is easy to see that when \( \theta = 0, \) we have \( s(A) = \frac{\pi}{2}. \) Therefore
\[
s(A) \ (\text{mod} \ \pi) \sim_{2, \mu_1^{-1}} \frac{\pi}{2} - \arctan[\lambda_2^{-2} \tan \theta] \ (\text{mod} \ \pi) = \arctan[\lambda_2^{2} \cot \theta] \ (\text{mod} \ \pi).
\]

\textbf{Proof of (ii-2c):}

Since $|\tan \theta|\geq \lambda_2^2e^{-(\log \lambda_0)^{\hat{\epsilon}}}>e^{-(\log \lambda_0)^{\hat{\epsilon}}},$
\eqref{daO} also holds true. Therefore, by \eqref{1d}, we have $$\begin{array}{ll}&\left\vert\partial_X [(1+O(\lambda_2^4\lambda_1^{-4}))\lambda_2^{-2}\tan\theta]\right\vert\sim_{0,e^{-(\log \lambda_0)^{\hat{\epsilon}}}}\left\vert\partial_X (\lambda_2^{-2}\tan\theta)\right\vert\sim_{0,e^{-(\log \lambda_0)^{\hat{\epsilon}}}} \left\vert\lambda_2^{-2}(1+\tan^2\theta)\partial_X\theta\right\vert
\\
\\&>\left\vert\lambda_2^{-2}(1+\lambda^4e^{-2(\log \lambda_0)^{\hat{\epsilon}}})\partial_X\theta\right\vert
>\lambda_2^{-2}(1+\lambda_2^4e^{-2(\log \lambda_0)^{\hat{\epsilon}}})e^{-(\log \lambda_0)^{\hat{\epsilon}}}(~\text{by}~\eqref{assumm})\geq \lambda_2^2e^{-4(\log \lambda_0)^{\hat{\epsilon}}}\geq \mu^{\frac{3}{2}}_2.\end{array}$$

Therefore, by denoting $\delta:=O(\lambda_2^4\lambda_1^{-4})$, we obtain $$\begin{array}{ll}
&\left\vert \partial_X s(A) \right\vert = \left\vert \partial_X \arctan \left\{ (1 + \delta) \tan \left[ 2 (1 + \delta) \lambda_2^{-2} \tan \theta \right] \right\} \right\vert
= \left\vert \frac{\partial_X \left\{ (1 + \delta) \tan \left[ 2 (1 + \delta) \lambda_2^{-2} \tan \theta \right] \right\}}{1 + \left\{ (1 + \delta) \tan \left[ 2 (1 + \delta) \lambda_2^{-2} \tan \theta \right] \right\}^2} \right\vert
\\
\\
&= \left\vert \frac{\partial_X (1 + \delta) \tan \left[ 2 (1 + \delta) \lambda_2^{-2} \tan \theta  \right] + (1 + \delta) \partial_X \left\{ \tan \left[ 2 (1 + \delta) \lambda_2^{-2} \tan \theta \right] \right\}}{1 + \left\{ (1 + \delta) \tan \left[ 2 (1 + \delta) \lambda_2^{-2} \tan \theta \right] \right\}^2} \right\vert
\\
\\
&= \left\vert \frac{\partial_X (1 + \delta) \tan \left[ 2 (1 + \delta) \lambda_2^{-2} \tan \theta \right]  + (1 + \delta) \left( 1 + \tan^2 \left[ 2 (1 + \delta) \lambda_2^{-2} \tan \theta \right] \right) \partial_X \left[ 2 (1 + \delta) \lambda_2^{-2} \tan \theta \right]}{1 + \left\{ (1 + \delta) \tan \left[ 2 (1 + \delta) \lambda_2^{-2} \tan \theta \right] \right\}^2} \right\vert
\\
\\
&\geq \left\vert \frac{(1 + \delta) \left( 1 + \tan^2 \left[ 2 (1 + \delta) \lambda_2^{-2} \tan \theta \right] \right) \partial_X \left[ 2 (1 + \delta) \lambda_2^{-2} \tan \theta \right]}{1 + \left\{ (1 + \delta) \tan \left[ 2 (1 + \delta) \lambda_2^{-2} \tan \theta \right] \right\}^2} \right\vert- \left\vert \frac{[\partial_X (1 + \delta)] \left\{ \tan \left[ 2 (1 + \delta) \lambda_2^{-2} \tan \theta \right] \right\}}{ 2\left\{ (1 + \delta) \tan \left[ 2 (1 + \delta) \lambda_2^{-2} \tan \theta \right] \right\}} \right\vert
\geq \mu^{\frac{3}{2}}_2 - \mu_1^{-3} (~\text{by}~\eqref{lambda1212})> \mu_2.
\end{array}$$

\textbf{Proof of (iii-3a):}

In this case, we have \(\mu_2 \geq \mu_1^{16}\). Note that
\begin{equation}\label{saaaa}
\tan(2s(A))= -\frac{2\tan \theta (1 - \lambda_1^{-4})}{(\lambda_2^2 \lambda_1^{-4} - \lambda_2^{-2}) \tan^2 \theta + \lambda_2^2 - \lambda_1^{-4} \lambda_2^{-2}}= -\frac{2 \tan \theta \frac{1 - \lambda_1^{-4}}{\lambda_2^2 - \lambda_1^{-4} \lambda_2^{-2}}}{\frac{(\lambda_2^2 \lambda_1^{-4} - \lambda_2^{-2})}{\lambda_2^2 - \lambda_1^{-4} \lambda_2^{-2}} \tan^2 \theta + 1}= -\frac{2 (1 + O(\lambda_1^{-4})) \lambda_2^{-2} \tan \theta}{\lambda_1^{-4} (1 + O(\lambda_1^{-4})) \tan^2 \theta + 1}.
\end{equation}
Therefore,
$$
\cot(2s(A)) = -\frac{1}{2} \left[ (1 + O(\lambda_1^{-4})) \lambda_1^{-4} \lambda_2^2 \tan \theta + (1 + O(\lambda_1^{-4})) \lambda_2^2 \cot \theta \right].
$$
Clearly, from
\[
\left\vert \frac{1}{2} \left[ (1 + O(\lambda_1^{-4})) \lambda_1^{-4} \lambda_2^2 \tan \theta + (1 + O(\lambda_1^{-4})) \lambda_2^2 \cot \theta \right] \right\vert \geq  \lambda_2^2 \lambda_1^{-2} \geq \mu_2^{\frac{3}{2}},
\]
we have
\begin{equation}\label{sapi2}
\left| \frac{\pi}{2} - s(A) \ (\text{mod}\ \pi) \right| \leq \mu_2^{-\frac{3}{2}}.
\end{equation}
Then \eqref{saaaa} yields
\begin{equation}\label{saaaaaaa}
\begin{array}{ll}
& \left| \partial_X (s(A)) \right| \leq \left| \partial_X [\tan (2s(A))] \right|= 2 (1 + \tan^2(2s(A))) \left| \partial_X (s(A)) \right| \\
\\
& = \left| -\frac{2 \partial_X [(1 + O(\lambda_1^{-4})) \lambda_2^{-2}] \tan \theta + 2 (1 + O(\lambda_1^{-4})) \lambda_2^{-2} (1 + \tan^2 \theta) \partial_X \theta}{\lambda_1^{-4} (1 + O(\lambda_1^{-4})) \tan^2 \theta + 1} \right.  + \left. \frac{2 (1 + O(\lambda_1^{-4})) \lambda_2^{-2} \tan \theta \partial_X [ \lambda_1^{-4} (1 + O(\lambda_1^{-4})) \tan^2 \theta]}{(\lambda_1^{-4} (1 + O(\lambda_1^{-4})) \tan^2 \theta + 1)^2} \right| \\
\\
& \leq \left| \frac{2 \partial_X [(1 + O(\lambda_1^{-4})) \lambda_2^{-2}] \tan \theta}{2 \lambda_1^{-2} (1 + O(\lambda_1^{-4})) \tan \theta} \right| + \left| \frac{2 (1 + O(\lambda_1^{-4})) \lambda_2^{-2} (1 + \tan^2 \theta) \partial_X \theta}{\lambda_1^{-4} (O(1 + \lambda_1^{-4})) (1 + \tan^2 \theta)} \right| + \left| \frac{2 (1 + O(\lambda_1^{-4})) \lambda_2^{-2} \tan \theta \partial_X [ \lambda_1^{-4} (1 + O(\lambda_1^{-4})) \tan^2 \theta]}{(\lambda_1^{-4} (1 + O(\lambda_1^{-4})) \tan^2 \theta + 1)^2} \right| \\
\\
& \leq \lambda_1^2 \lambda_2^{-2} e^{(\log \lambda_0)^{\hat{\epsilon}}} + \lambda_1^4 \lambda_2^{-2} e^{(\log \lambda_0)^{\hat{\epsilon}}}+ \left| \frac{2 \lambda_2^{-2} \lambda_1^{-4} \tan^3 \theta e^{(\log \lambda_0)^{\hat{\epsilon}}} + 2 \lambda_2^{-2} \lambda_1^{-4} (2 \tan^2 \theta (1 + \tan^2 \theta)) e^{(\log \lambda_0)^{\hat{\epsilon}}}}{1 + 2 \lambda_1^{-4} (1 + O(\lambda_1^{-4})) \tan^2 \theta + [\lambda_1^{-4} (1 + O(\lambda_1^{-4})) \tan^2 \theta]^2} \right| \\
\\
& \leq \lambda_1^2 \lambda_2^{-2} e^{(\log \lambda_0)^{\hat{\epsilon}}} + \lambda_1^4 \lambda_2^{-2} e^{(\log \lambda_0)^{\hat{\epsilon}}} + \frac{4 \lambda_2^{-2} \lambda_1^{-4} \tan^2 \theta e^{(\log \lambda_0)^{\hat{\epsilon}}}}{2 \lambda_1^{-4} (1 + O(\lambda_1^{-4})) \tan^2 \theta} + \left| \frac{2 \lambda_2^{-2} \lambda_1^{-4} (2 \tan^4 \theta) e^{(\log \lambda_0)^{\hat{\epsilon}}}}{[\lambda_1^{-4} (1 + O(\lambda_1^{-4})) \tan^2 \theta]^2} \right|  \\
\\
&\quad+\left| \frac{2 \lambda_2^{-2} \lambda_1^{-4} ( |\tan^3 \theta|) e^{(\log \lambda_0)^{\hat{\epsilon}}}}{2 \lambda_1^{-4} (1 + O(\lambda_1^{-4})) \tan^2 \theta + [\lambda_1^{-4} (1 + O(\lambda_1^{-4})) \tan^2 \theta]^2} \right| \\
\\
& \leq \lambda_1^2 \lambda_2^{-2} e^{(\log \lambda_0)^{\hat{\epsilon}}} + \lambda_1^4 \lambda_2^{-2} e^{(\log \lambda_0)^{\hat{\epsilon}}} + \lambda_2^{-2} e^{(\log \lambda_0)^{\hat{\epsilon}}}  + \lambda_2^{-2} \lambda_1^{4} e^{(\log \lambda_0)^{\hat{\epsilon}}} + \lambda_2^{-2} \lambda_1^2 e^{(\log \lambda_0)^{\hat{\epsilon}}} \leq \mu_2^{-\frac{3}{2}}.
\end{array}
\end{equation}

By \eqref{sapi2} and \eqref{saaaaaaa}, \begin{equation}\label{2saxsa}|\tan (2s(A))|+|\partial_X s(A)|<\mu_2^{-\frac{3}{2}},~X\in\{x,y\}.\end{equation}

Note $\left| \frac{\partial^2 \tan (2s(A))}{\partial X \partial Y} \right|=\left| 2 (1 + \tan^2(2s(A))) \frac{\partial^2 (s(A))}{\partial X \partial Y} + 8 \tan (2s(A)) (1 + \tan^2(2s(A))) (\partial_X s(A) \partial_Y s(A)) \right|.$ Then \eqref{2saxsa}
implies
$$
\begin{array}{ll}
& 2\left| \frac{\partial^2 s(A)}{\partial X \partial Y} \right| - 1000\mu_2^{-3} \leq \frac{1}{(1 + \tan^2(2s(A)))}\left| \frac{\partial^2 \tan (2s(A))}{\partial X \partial Y} \right| \leq \left\vert \frac{\partial^2 \left(-\frac{2 (1 + O(\lambda_1^{-4})) \lambda_2^{-2} \tan \theta}{\lambda_1^{-4} (1 + O(\lambda_1^{-4})) \tan^2 \theta + 1}\right)}{\partial X \partial Y} \right\vert\leq \sum_{i=1}^{16} \Upsilon_i,
\end{array}
$$
where $\Upsilon_i$, with the notation $\Upsilon_0={\lambda_1^{-4} (1 + O(\lambda_1^{-4})) \tan^2 \theta + 1}$, satisfies

\[
\begin{array}{ll}
\Upsilon_0\Upsilon_1 &= \left| {2 \partial^2_{XY} [(1 + O(\lambda_
1^{-4})) \lambda_2^{-2}] \tan \theta} \right| + \left| {2 \partial_X [(1 + O(\lambda_1^{-4})) \lambda_2^{-2}] (1 + \tan^2 \theta) \partial_Y \theta} \right|, \\
\Upsilon_0\Upsilon_2 &= \left| {2 \partial_Y [(1 + O(\lambda_1^{-4})) \lambda_2^{-2}] (1 + \tan^2 \theta) \partial_X \theta} \right| + \left|{2 [(1 + O(\lambda_1^{-4})) \lambda_2^{-2}] [2 \tan \theta (1 + \tan^2 \theta)] \partial_Y \theta \partial_X \theta} \right|\cdot \Upsilon_0^{-1}, \\
\Upsilon_0^{2}\Upsilon_3 &= \left| {2 [\lambda_1^{-4} (1 + O(\lambda_1^{-4})) \tan^2 \theta] [(1 + O(\lambda_1^{-4})) \lambda_2^{-2}] 2 \tan \theta \partial_Y \theta \partial_X \theta} \right|, \\
\Upsilon_0\Upsilon_4 &= \left| {2 [(1 + O(\lambda_1^{-4})) \lambda_2^{-2}] (1 + \tan^2 \theta) \partial^2_{XY} \theta} \right|, \\
\Upsilon_0^{2}\Upsilon_5 &= \left|{2 \partial_X [(1 + O(\lambda_1^{-4})) \lambda_2^{-2}] \tan \theta \partial_Y [ \lambda_1^{-4} (1 + O(\lambda_1^{-4})) \tan^2 \theta]} \right|, \\
\Upsilon_0^{2}\Upsilon_6 &= \left| {2 [(1 + O(\lambda_1^{-4})) \lambda_2^{-2}] (1 + \tan^2 \theta) (\partial_X \theta) \partial_Y [ \lambda_1^{-4} (1 + O(\lambda_1^{-4}))] \tan^2 \theta}\right|, \\
\Upsilon_0^{3}\Upsilon_7 &= \left| {2 [(1 + O(\lambda_1^{-4})) \lambda_2^{-2}] (1 + \tan^2 \theta) (\partial_X \theta) \lambda_1^{-4} (1 + O(\lambda_1^{-4})) [2 \tan \theta (1 + \tan^2 \theta) \partial_Y \theta]} \right|, \\
\Upsilon_0^{3}\Upsilon_8 &= \left| {2 [\lambda_1^{-4} (1 + O(\lambda_1^{-4})) \tan^2 \theta] [(1 + O(\lambda_1^{-4})) \lambda_2^{-2}] (1 + \tan^2 \theta) (\partial_X \theta) \lambda_1^{-4} (1 + O(\lambda_1^{-4})) [2 \tan \theta \partial_Y \theta]}\right|, \\
\Upsilon_0^{2}\Upsilon_9 &= \left| {2 (1 + O(\lambda_1^{-4})) \lambda_2^{-2} \tan \theta \partial_X [ \lambda_1^{-4} (1 + O(\lambda_1^{-4}))] [2 \tan \theta (1 + \tan^2 \theta) \partial_Y \theta]} \right|, \\
\Upsilon_0^{2}\Upsilon_{10} &= \left| {2 (1 + O(\lambda_1^{-4})) \lambda_2^{-2} \tan \theta \partial_Y [ \lambda_1^{-4} (1 + O(\lambda_1^{-4}))] 2 \tan \theta [1 + \tan^2 \theta] (\partial_X \theta)}\right|, \\
\Upsilon_0^{3}\Upsilon_{11} &= \left| {2 (1 + O(\lambda_1^{-4})) \lambda_2^{-2} \tan \theta [\lambda_1^{-4} (1 + O(\lambda_1^{-4}))] [\tan \theta (1 + \tan^2 \theta) \partial^2_{XY} \theta + 2 [1 + 4 \tan^2 \theta + 3 \tan^3 \theta] (\partial_X \theta) (\partial_Y \theta)]} \right|, \\
\Upsilon_0^{3}\Upsilon_{12} &= \left| 2 [\lambda_1^{-4} (1 + O(\lambda_1^{-4})) \tan^2 \theta] (1 + O(\lambda_1^{-4})) \lambda_2^{-2} \tan \theta [\lambda_1^{-4} (1 + O(\lambda_1^{-4}))] \right.\\
&\left.\quad\cdot[\tan \theta (1 + \tan^2 \theta) \partial^2_{XY} \theta + 2 [1 + 4 \tan^2 \theta] (\partial_X \theta) (\partial_Y \theta)]\right|, \\
\Upsilon_0^{2}\Upsilon_{13} &= \left| {2 (1 + O(\lambda_1^{-4})) \lambda_2^{-2} \tan \theta \partial^2_{XY} [ \lambda_1^{-4} (1 + O(\lambda_1^{-4}))] \tan^2 \theta} \right|, \\
\Upsilon_0^{3}\Upsilon_{14} &= \left| {4 (1 + O(\lambda_1^{-4})) \lambda_2^{-2} \tan \theta \partial_X [ \lambda_1^{-4} (1 + O(\lambda_1^{-4})) \tan^2 \theta] \partial_Y [ \lambda_1^{-4} (1 + O(\lambda_1^{-4}))] \tan^2 \theta} \right|, \\
\Upsilon_0^{3}\Upsilon_{15} &= \left| {4 (1 + O(\lambda_1^{-4})) \lambda_2^{-2} \tan \theta \partial_X [ \lambda_1^{-4} (1 + O(\lambda_1^{-4}))] \tan^2 \theta \lambda_1^{-4} (1 + O(\lambda_1^{-4})) [2 \tan \theta (1 + \tan^2 \theta) \partial_Y \theta]}\right|, \\
\Upsilon_0^{3}\Upsilon_{16} &= \left| {4 (1 + O(\lambda_1^{-4})) \lambda_2^{-2} \tan \theta [\lambda_1^{-4} (1 + O(\lambda_1^{-4}))]^2 [4 \tan^2 \theta (1 + 2 \tan^2 \theta) \partial_X \theta \partial_Y \theta]} \right|.
\end{array}
\]

By a direct computation, we have
\begin{equation}\label{sa222}\left| \frac{\partial^2 (s(A))}{\partial X \partial Y} \right| \leq \mu_2^{-\frac{3}{2}}.\end{equation}

By \eqref{sapi2}, \eqref{saaaaaaa} and \eqref{sa222}, we obtain
$$\left\| \frac{\pi}{2} - s(A) \right\|_{C^2} < \mu_2^{-\frac{3}{2}}.$$

\textbf{Proof of (i):}

Note in this case, for $i=1,2,$ $\|\lambda_i^{-2}\tan\theta\|_{C^2}<e^{(\log \lambda_0)^{2\hat{\epsilon}}}\lambda_i^{-2}\ll 1.$ Therefore 
$$\begin{array}{ll}&\tan(2s(A))=-\frac{2\tan \theta(1-\lambda_1^{-4})}{(\lambda_2^2\lambda_1^{-4}-\lambda_2^{-2})\tan^2\theta+\lambda_2^2-\lambda_1^{-4}\lambda_2^{-2}}
=-\frac{2(1+O(\lambda_1^{-4}))\lambda_2^{-2}\tan \theta }{(\lambda_1^{-4}-\lambda_2^{-4})\tan^2\theta+(1+O(\lambda_1^{-4}))}
\\
\\&=2(1+O(\lambda_1^{-4}))\lambda_2^{-2}\tan \theta[1+O[(\lambda_1^{-4}-\lambda_2^{-4})\tan^2\theta]],\end{array}$$

which implies that
$\|\tan(2s(A))\|_{C^2}\leq e^{(\log \lambda_0)^{3\hat{\epsilon}}}\mu_2^{-2}.$

Then \begin{align*}&|2 s(A)\ (mod\  \pi)|+|\partial_X 2 s(A)|<|\tan(2s(A))|+|2(1+\tan^2 (2s(A)))\partial_X s(A)|\\&\leq \|\tan(2s(A))\|_{C^1}\leq e^{(\log \lambda_0)^{3\hat{\epsilon}}}\lambda_2^{-2}\end{align*}
and
$$\begin{array}{ll}&e^{(\log \lambda_0)^{2\hat{\epsilon}}}\lambda_2^{-2}> \|\tan(2s(A))\|_{C^2}
 \geq |\frac{\partial^2[\tan (2s(A))]}{\partial X \partial Y}|
 \\&\geq |2(1+\tan^2(2s(A)))\frac{\partial^2 (s(A))}{\partial X\partial Y}|-|8\tan(2s(A))(1+\tan^2(2s(A)))(\partial_X s(A)\partial_Y s(A))|
\\&\geq |2(1+\tan^2(2s(A)))\frac{\partial^2 (s(A))}{\partial X\partial Y}|-e^{(\log \lambda_0)^{3\hat{\epsilon}}}\lambda_2^{-2}.\end{array}$$

Hence $|\frac{\partial^2 (s(A))}{\partial X\partial Y}|<e^{(\log \lambda_0)^{3\hat{\epsilon}}}\lambda_2^{-2}.$
Therefore, $\|\frac{\pi}{2}-s(A)\|_{C^2}<e^{(\log \lambda_0)^{3\hat{\epsilon}}}\lambda_2^{-2}.$
\hfill\qed\end{proof}


\begin{lemma}\label{lemma9*}Let $\lambda_0,\ \mu_1,\ \mu_2\in \R$, $\lambda_1(x,y),\lambda_2(x,y),\theta(x,y)\in C^2(I\times J)$ and $A(x,y)$ be as in Lemma \ref{lemma8*}.

Then the following hold true.

\begin{enumerate}

\item[i:] If $|\tan\theta|<e^{(\log \lambda_0)^{2\hat{\epsilon}}},$ then \begin{equation}\label{lem10-i1}\|A(x,y)\|\sim_{0,\lambda_0^{-1}}\mu_1\mu_2|\cos\theta|;\end{equation}
\begin{equation}\label{lem10-i2}\frac{1}{\|A\|}(\partial_X\|A\|)=\frac{\partial_X (\lambda_1\lambda_2\cos\theta)}{\lambda_1\lambda_2\cos\theta}+o(\lambda_0^{-\frac{3}{2}});\end{equation}
\begin{equation}\label{lem10-i3}\frac{1}{\|A\|}\partial^2_{XY}\|A\|=\frac{\partial^2_{XY} (\lambda_1\lambda_2\cos\theta)}{\lambda_1\lambda_2\cos\theta}+o(\lambda_0^{-\frac{3}{2}}).\end{equation}

\item[ii:] If $\mu_1\geq \mu_2^8$, then

 $$\|A(x,y)\|\geq (1-\mu_2^{-1})\mu_1\mu^{-1}_2;$$
$$\left\vert\frac{1}{\|A\|}(\partial_X\|A\|)\right\vert+\left\vert\frac{1}{\|A\|}\frac{\partial^2\|A\|}{\partial_X\partial_Y}\right\vert\vert\leq \lambda_2^4e^{2(\log \lambda_0)^{\hat{\epsilon}}}.$$

\item[iii:] If $\mu_2\geq \mu_1^8$, then

 $$\|A(x,y)\|\geq (1-\mu_2^{-1})\mu_2\mu^{-1}_1;$$
$$\left\vert\frac{1}{\|A\|}\partial_X\|A\|\right\vert+\left\vert\frac{1}{\|A\|}\frac{\partial^2\|A\|}{\partial_X\partial_Y}\right\vert\vert\leq \lambda_1^4e^{2(\log \lambda_0)^{\hat{\epsilon}}}.$$

\end{enumerate}
\end{lemma}

\begin{proof}

\textbf{The proof of (i):}
We start from (\ref{lemma8,1}), that is,
\begin{equation}\label{6'}
\|A\|^2 + \|A\|^{-2} = (\lambda_1^2 \lambda_2^2 + \lambda_1^{-2} \lambda_2^{-2}) \cos^2 \theta + (\lambda_1^2 \lambda_2^{-2} + \lambda_1^{-2} \lambda_2^2) \sin^2 \theta.
\end{equation}

Hence
\[
\|A\|^2 = \frac{\lambda_1^2 \lambda_2^2 \cos^2 \theta \left( 1 + \lambda_1^{-4} \lambda_2^{-4} + \lambda_2^{-4} \tan^2 \theta + \lambda_1^{-4} \tan^2 \theta \right)}{1 + \|A\|^{-4}}.
\]

Moreover, in the case \( |\tan \theta| < e^{(\log \lambda_0)^{2\hat{\epsilon}}} \),
\[
\begin{array}{ll}
&2\|A\|^2 > (\lambda_1^2 \lambda_2^2 + \lambda_1^{-2} \lambda_2^{-2}) \cos^2 \theta + (\lambda_1^2 \lambda_2^{-2} + \lambda_1^{-2} \lambda_2^2) \sin^2 \theta \\
& = \lambda_1^2 \lambda_2^2 \cos^2 \theta \left( 1 + \lambda_1^{-4} \lambda_2^{-4} + \lambda_2^{-4} \tan^2 \theta + \lambda_1^{-4} \tan^2 \theta \right)\geq \frac{1}{2}\lambda_1^2 \lambda_2^2 \cos^2 \theta.
\end{array}
\]

Hence $\|A\|\geq \frac{1}{4}\lambda_1 \lambda_2 |\cos \theta|.$ Therefore \eqref{6'} implies
$$\|A\|^2(1+\|A\|^{-4})= \lambda_1^2 \lambda_2^2 \cos^2 \theta \left( 1 + \lambda_1^{-4} \lambda_2^{-4} + \lambda_2^{-4} \tan^2 \theta + \lambda_1^{-4} \tan^2 \theta \right).$$
Then
\[
\left\vert \frac{\|A\|}{\lambda_1 \lambda_2 |\cos \theta|} - 1 \right\vert \leq \lambda_0^{-1}.
\]
Hence \begin{equation}\label{Asiml1l2}\|A(x, y)\| \sim_{0, \lambda_0^{-1}} \lambda_1 \lambda_2 |\cos \theta|.\end{equation}

On the other hand,
\begin{equation}\label{28-}
\begin{array}{ll}
2 \|A\| (\partial_X \|A\|) - 2 \|A\|^{-3} (\partial_X \|A\|) & = \partial_X \left[ (\lambda_1^2 \lambda_2^2 + \lambda_1^{-2} \lambda_2^{-2}) \cos^2 \theta + (\lambda_1^2 \lambda_2^{-2} + \lambda_1^{-2} \lambda_2^2) \sin^2 \theta \right] \\
& = \partial_X \left[ \lambda_1^2 \lambda_2^2 \cos^2 \theta \left( 1 + \lambda_1^{-4} \lambda_2^{-4} + \lambda_2^{-4} \tan^2 \theta + \lambda_1^{-4} \tan^2 \theta \right) \right].
\end{array}
\end{equation}

Therefore by \eqref{Asiml1l2},
\vskip -0.4cm\begin{equation}\label{yijiedd}
\begin{array}{ll}
&\frac{1}{\|A\|} (\partial_X \|A\|) = \left( \frac{1}{1 - \|A\|^{-4}} \right) \frac{1}{2 \|A\|^2} \left[ \partial_X \left[ \lambda_1^2 \lambda_2^2 \cos^2 \theta \right] \left( 1 + \lambda_1^{-4} \lambda_2^{-4} + \lambda_2^{-4} \tan^2 \theta + \lambda_1^{-4} \tan^2 \theta \right) \right] \\
& \quad + \left( \frac{1}{1 - \|A\|^{-4}} \right) \frac{1}{2 \|A\|^2} \left[ \left[ \lambda_1^2 \lambda_2^2 \cos^2 \theta \right] \partial_X \left( 1 + \lambda_1^{-4} \lambda_2^{-4} + \lambda_2^{-4} \tan^2 \theta + \lambda_1^{-4} \tan^2 \theta \right) \right] \\
& = (1 + o(\lambda_0^{-1})) \left\vert \frac{1}{2 \lambda_1^2 \lambda_2^2 \cos^2 \theta} \left[ \partial_X \left[ \lambda_1^2 \lambda_2^2 \cos^2 \theta \right] \left( 1 + \lambda_1^{-4} \lambda_2^{-4} + \lambda_2^{-4} \tan^2 \theta + \lambda_1^{-4} \tan^2 \theta \right) \right] \right\vert \\
& \quad + (1 + o(\lambda_0^{-1})) \left\vert \frac{1}{2 \lambda_1^2 \lambda_2^2 \cos^2 \theta} \left[ \left[ \lambda_1^2 \lambda_2^2 \cos^2 \theta \right] \partial_X \left( 1 + \lambda_1^{-4} \lambda_2^{-4} + \lambda_2^{-4} \tan^2 \theta + \lambda_1^{-4} \tan^2 \theta \right) \right] \right\vert \\
& = \left( \frac{1}{\lambda_1} \partial_X \lambda_1 + \frac{1}{\lambda_2} \partial_X \lambda_2 - \tan \theta \partial_X \theta \right) + o(\lambda_0^{-\frac{3}{2}}) = \frac{\partial_X (\lambda_1 \lambda_2 \cos \theta)}{\lambda_1 \lambda_2 \cos \theta} + o(\lambda_0^{-\frac{3}{2}}).
\end{array}
\end{equation}
\vskip -0.3cm
\vskip 0.05cm
Furthermore
\begin{equation}\label{a2a-2}
\begin{array}{ll}
(1 - o(\lambda_0^{-6})) \|A\|^{-1} \partial^2_{XY} \|A\| + (1 + o(\lambda_0^{-6})) \left[ \frac{1}{\|A\|} (\partial_Y \|A\|) \right] \left[ \frac{1}{\|A\|} (\partial_X \|A\|) \right] \\
= (1 - \|A\|^{-4}) \|A\|^{-1} \partial^2_{XY} \|A\| + (1 + 3 \|A\|^{-4}) \left[ \frac{1}{\|A\|} (\partial_Y \|A\|) \right] \left[ \frac{1}{\|A\|} (\partial_X \|A\|) \right] \\
= \|A\|^{-2} (\partial_Y \|A\|) (\partial_X \|A\|) + 3 \|A\|^{-6} (\partial_X \|A\|) (\partial_Y \|A\|) + \|A\|^{-1} \partial^2_{XY} \|A\| - \|A\|^{-5} \partial^2_{XY} \|A\| \\
= \frac{1}{2} \|A\|^{-2} \partial^2_{XY} \left[ \|A\|^2 + \|A\|^{-2} \right]= \frac{1}{2} \|A\|^{-2} \partial^2_{XY} \left[ \lambda_1^2 \lambda_2^2 \cos^2 \theta \left( 1 + \lambda_1^{-4} \lambda_2^{-4} + \lambda_2^{-4} \tan^2 \theta + \lambda_1^{-4} \tan^2 \theta \right) \right] \\
= \frac{1}{\lambda_1^2} (\partial_X \lambda_1) (\partial_Y \lambda_1) + \frac{1}{\lambda_2^2} (\partial_X \lambda_2) (\partial_Y \lambda_2) + \frac{2}{\lambda_1 \lambda_2} (\partial_X \lambda_1) (\partial_Y \lambda_2)
 + \frac{2}{\lambda_1 \lambda_2} (\partial_X \lambda_2) (\partial_Y \lambda_1)\\
 \quad - \tan \theta \partial^2_{XY}

 \theta - \partial_X \theta \partial_Y \theta
+ \frac{1}{\lambda_1} (\partial^2_{XY} \lambda_1) + \frac{1}{\lambda_2} (\partial^2_{XY} \lambda_2)
 + \frac{-2 \tan \theta}{\lambda_1} (\partial_X \lambda_1) (\partial_Y \theta) + \frac{-2 \tan \theta}{\lambda_2} (\partial_X \lambda_2) (\partial_Y \theta)\\
\quad+ \frac{-2 \tan \theta}{\lambda_1} (\partial_Y \lambda_1) (\partial_X \theta) + \frac{-2 \tan \theta}{\lambda_2} (\partial_Y \lambda_2) (\partial_X \theta)
+ o(\lambda_0^{-\frac{3}{2}})
:= \mathcal{P}_1.
\end{array}
\end{equation}

\eqref{yijiedd} and a direct calculation yield
\begin{equation}\label{aaa}
\begin{array}{ll}
\frac{1}{\|A\|} (\partial_X \|A\|) \frac{1}{\|A\|} (\partial_Y \|A\|)\\
= \left[ \left( \frac{1}{\lambda_1} \partial_X \lambda_1 + \frac{1}{\lambda_2} \partial_X \lambda_2 - \tan \theta \partial_X \theta \right) + o(\lambda_0^{-\frac{3}{2}}) \right]  \times \left[ \left( \frac{1}{\lambda_1} \partial_Y \lambda_1 + \frac{1}{\lambda_2} \partial_Y \lambda_2 - \tan \theta \partial_Y \theta \right) + o(\lambda_0^{-\frac{3}{2}}) \right] \\
= \left( \frac{1}{\lambda_1^2} (\partial_X \lambda_1) (\partial_Y \lambda_1) + \frac{1}{\lambda_2^2} (\partial_X \lambda_2) (\partial_Y \lambda_2) \right.
\left. + \frac{1}{\lambda_1 \lambda_2} (\partial_X \lambda_1) (\partial_Y \lambda_2) + \frac{1}{\lambda_1 \lambda_2} (\partial_X \lambda_2) (\partial_Y \lambda_1) \right.
+ \frac{-\tan \theta}{\lambda_1} (\partial_X \lambda_1) (\partial_Y \theta)\\+ \frac{-\tan \theta}{\lambda_2} (\partial_X \lambda_2) (\partial_Y \theta)
\left. + \frac{-\tan \theta}{\lambda_1} (\partial_Y \lambda_1) (\partial_X \theta) + \frac{-\tan \theta}{\lambda_2} (\partial_Y \lambda_2) (\partial_X \theta) \right)
+ \tan^2 \theta \partial_X \theta \partial_Y \theta + o(\lambda_0^{-\frac{3}{2}})
:= \mathcal{P}_2.
\end{array}
\end{equation}

From the conditions of Lemma \ref{lemma8*} on $\lambda_i$ and $\theta$, we have
\begin{equation}\label{p1p2}
|\mathcal{P}_1| + |\mathcal{P}_2| \leq \lambda_0^2.
\end{equation}

Finally, \eqref{a2a-2}, \eqref{aaa}, and \eqref{p1p2} imply
\[
\begin{array}{ll}
\frac{1}{\|A\|} \partial^2_{XY} \|A\|
= (1 + o(\lambda_0^{-6})) (\mathcal{P}_1 - (1 + o(\lambda_0^{-6})) \mathcal{P}_2)\\
= \mathcal{P}_1 - \mathcal{P}_2 - o(\lambda_0^{-6}) \mathcal{P}_2 + o(\lambda_0^{-6}) (\mathcal{P}_1 - (1 + o(\lambda_0^{-6})) \mathcal{P}_2)
= \mathcal{P}_1 - \mathcal{P}_2 + o(\lambda_0^{-4}) \\
=\frac{1}{\lambda_1^2} (\partial_X \lambda_1) (\partial_Y \lambda_1) + \frac{1}{\lambda_2^2} (\partial_X \lambda_2) (\partial_Y \lambda_2)
+ \frac{2}{\lambda_1 \lambda_2} (\partial_X \lambda_1) (\partial_Y \lambda_2) + \frac{2}{\lambda_1 \lambda_2} (\partial_X \lambda_2) (\partial_Y \lambda_1)
 - \tan \theta \partial^2_{XY} \theta \\- (1 - \tan^2 \theta) \partial_X \theta \partial_Y \theta
 + \frac{1}{\lambda_1} (\partial^2_{XY} \lambda_1) + \frac{1}{\lambda_2} (\partial^2_{XY} \lambda_2)
\left. + \frac{-2 \tan \theta}{\lambda_1} (\partial_X \lambda_1) (\partial_Y \theta) + \frac{-2 \tan \theta}{\lambda_2} (\partial_X \lambda_2) (\partial_Y \theta) \right. \\
 + \frac{-2 \tan \theta}{\lambda_1} (\partial_Y \lambda_1) (\partial_X \theta) + \frac{-2 \tan \theta}{\lambda_2} (\partial_Y \lambda_2) (\partial_X \theta)
+ o(\lambda_0^{-\frac{3}{2}})
= \frac{\partial^2_{XY} (\lambda_1 \lambda_2 \cos \theta)}{\lambda_1 \lambda_2 \cos \theta} + o(\lambda_0^{-\frac{3}{2}}).
\end{array}
\]

\

\textbf{Proof of (ii):}
Note from (\ref{lemma8,1}), we have $\begin{array}{ll}
& \|A\|^2 + \|A\|^{-2}
  \geq (\lambda_1^2 \lambda_2^{-2} + \lambda_1^{-2} \lambda_2^2)
 > \lambda_1^2 \lambda_2^{-2}
\geq \mu_1^{\frac{7}{4}}.
\end{array}
$

Hence
\[
\|A\|^2 \geq \frac{1}{1 + \|A\|^{-4}} (\lambda_1^2 \lambda_2^{-2}) > (1 - \mu_1^{-2}) \mu_1^2 \mu_2^{-2}.
\]

Therefore
\begin{equation}\label{A-norm}
\|A\| > (1 - \mu_1^{-2}) (\mu_1 \mu_2^{-1}).
\end{equation}

On the other hand, from the conditions on $\lambda_i,\theta$ in Lemma \ref{lemma8*},
\[
\begin{array}{ll}
& \left\vert \partial_X (\lambda_1^m \lambda_2^n \cos \theta) \right\vert \leq \left\vert (\partial_X \lambda_1) (\lambda_1^{m-1} \lambda_2^n \cos \theta) \right\vert + \left\vert (\partial_X \lambda_2) (\lambda_1^m \lambda_2^{n-1} \cos \theta) \right\vert + \left\vert (\partial_X \theta) (-\lambda_1^m \lambda_2^n \sin \theta) \right\vert \\
& \leq \mu_1^m \mu_2^n e^{(\log \lambda_0)^{\hat{\epsilon}}}
\end{array}
\]

and

\[
\begin{array}{ll}
&\quad \left\vert \partial^2_{XY} (\lambda_1^m \lambda_2^n \cos \theta) \right\vert = \left\vert (\partial_X \lambda_1) (\partial_Y \lambda_1) (\lambda_1^{m-2} \lambda_2^n \cos \theta) \right\vert
+ \left\vert (\partial_X \lambda_2) (\partial_Y \lambda_2) (\lambda_1^m \lambda_2^{n-2} \cos \theta) \right\vert \\
&\quad + \left\vert (\partial^2_{XY} \theta) (-\lambda_1^m \lambda_2^n \sin \theta) \right\vert
+ \left\vert (\partial^2_{XY} \lambda_1) (\lambda_1^{m-1} \lambda_2^n \cos \theta) \right\vert
+ \left\vert (\partial^2_{XY} \lambda_2) (\lambda_1^m \lambda_2^{n-1} \cos \theta) \right\vert \\
&\quad + \left\vert (\partial_X \lambda_1) (\partial_Y \theta) (-\lambda_1^{m-1} \lambda_2^n \sin \theta) \right\vert
+ \left\vert (\partial_X \lambda_2) (\partial_Y \theta) (-\lambda_1^m \lambda_2^{n-1} \sin \theta) \right\vert \\
&\quad + \left\vert (\partial_Y \lambda_1) (\partial_X \theta) (-\lambda_1^{m-1} \lambda_2^n \sin \theta) \right\vert
+ \left\vert (\partial_Y \lambda_2) (\partial_X \theta) (-\lambda_1^m \lambda_2^{n-1} \sin \theta) \right\vert \\
&\quad + \left\vert (\partial_X \theta) (\partial_Y \theta) (-\lambda_1^m \lambda_2^n \cos \theta) \right\vert
+ \left\vert (\partial_X \lambda_1) (\partial_Y \lambda_1) (\lambda_1^{m-1} \lambda_2^{n-1} \cos \theta) \right\vert \\
&\quad + \left\vert (\partial_Y \lambda_1) (\partial_X \lambda_1) (\lambda_1^{m-1} \lambda_2^{n-1} \cos \theta) \right\vert\leq \mu_1^m \mu_2^n e^{2 (\log \lambda_0)^{\hat{\epsilon}}}.
\end{array}
\]

It together with (\ref{28-}) implies
\[
\begin{array}{ll}
& \left\vert \|A\| |\partial_X \|A\|| (2 - 2 \|A\|^{-4}) \right\vert = \left\vert 2 \|A\| (\partial_X \|A\|) - 2 \|A\|^{-3} (\partial_X \|A\|) \right\vert  \leq C \mu_1^2 \mu_2^2 e^{(\log \lambda_0)^{\hat{\epsilon}}}.
\end{array}
\]
Thus with the help of (\ref{A-norm}), we have
\[\begin{array}{ll}
\left\vert \|A\|^{-1} (\partial_X \|A\|) \right\vert &\leq \left\vert C \|A\|^{-2} \lambda_1^2 \lambda_2^2 e^{(\log \lambda_0)^{\hat{\epsilon}}} \right\vert  \leq [(1 - \lambda_1^{-2}) (\lambda_1 \lambda_2^{-1})]^{-2} \lambda_1^2 \lambda_2^2 e^{(\log \lambda_0)^{\hat{\epsilon}}} \leq C \mu_2^4 e^{(\log \lambda_0)^{\hat{\epsilon}}}.
\end{array}\]

It together with (\ref{lemma8,1}) shows that
\[
\begin{array}{ll}
& \left\vert \|A\|^{-1} \partial^2_{XY} \|A\| (1 - \|A\|^{-4}) \right\vert - C \mu_2^4 e^{(\log \lambda_0)^{\hat{\epsilon}}}
\\
 &\leq \left\vert \|A\|^{-2} (\partial_Y \|A\|) (\partial_X \|A\|) \right\vert  + \left\vert 3 \|A\|^{-6} (\partial_X \|A\|) (\partial_Y \|A\|) \right\vert  + \left\vert \|A\|^{-1} \partial^2_{XY} \|A\| - \|A\|^{-5} \partial^2_{XY} \|A\| \right\vert \leq C e^{2 (\log \lambda_0)^{\hat{\epsilon}}},
\end{array}
\]
which implies
$
\left\vert \|A\|^{-1} (\partial^2_{XY} \|A\|) \right\vert \leq C \mu_2^4 e^{2 (\log \lambda_0)^{\hat{\epsilon}}}.
$

The proof of (iii) is the same as (ii) by replacing \(\lambda_1\) by \(\lambda_2\). Then we complete the proof.
\hfill\qed\end{proof}

By the help of Lemmas \ref{lemma8*} and \ref{lemma9*}, we obtain the following  lemma.
\begin{lemma}\label{jishu1}
Let \( I \subset \mathbb{R}^2 \) be an open rectangle, \( 0 < \hat{\epsilon} \ll 1 \), \( \{l_i\}_{i=0}^n \subset \R_+\), \( l_0 \gg n \), and \( \lambda_i(x,y), f_i(x,y) \in C^2(I),\ 1\le i\le n \). Moreover, for any \((x,y) \in I\) and \( X, Y \in \{x,y\} \), we have
\begin{equation}\label{jfact1}
\lambda_i(x,y) \sim_{0,e^{-(\log l_0)^{\hat{\epsilon}}}} l_i \geq \min_{i=1,2,\ldots,n} \inf_{I} \lambda_i(x,y) = l_0 \text{ for } i=1,2,\ldots,n,
\end{equation}
\begin{equation}\label{jfact2}
\min_{i=1,2,\ldots,n-1} \inf_{I} |f_i(x,y) - \frac{\pi}{2}| \geq e^{-(\log l_0)^{2\hat{\epsilon}}},
\end{equation}
\begin{equation}\label{jfact3}
\sup_{I} \left|\frac{1}{\lambda_k} \frac{\partial \lambda_k}{\partial X}\right|, \quad \sup_{I} \left|\frac{1}{\lambda_k} \frac{\partial^2 \lambda_k}{\partial X \partial Y}\right| \leq e^{(\log l_0)^{\hat{\epsilon}}}, \quad k=1,2,\ldots,n,
\end{equation}
\begin{equation}\label{jfact4}
\sup_{I} \left|\frac{\partial^2 f_k}{\partial X \partial Y}\right|, \quad \sup_{I} \left|\frac{\partial f_k}{\partial X}\right| \leq e^{(\log l_0)^{\hat{\epsilon}}}, \quad k=1,2,\ldots,n-1.
\end{equation}

Consider
$$
A(x,y) := \Lambda_n \cdot \prod_{k=n-1}^{1} \left(R_{f_k} \cdot \Lambda_k\right),
$$
where \( \Lambda_i(x,y) = \Lambda(\lambda_i(x,y))\) for \( i=1,2,\ldots,n \) and
$$
R_{f_i} = \begin{pmatrix}
\cos f_i & -\sin f_i \\
\sin f_i & \cos f_i
\end{pmatrix}.
$$

Then  the following hold true.
\begin{equation}\label{re1}
\|A\| \sim_{0,nl_0^{-1}} \left(\prod_{i=1}^n l_i\right) \left(\prod_{i=1}^{n-1} |\cos f_i|\right) \geq l_0^{(1 - (\log l_0)^{-\frac{1}{2}})n},
\end{equation}
\begin{equation}\label{re2}
\left| \partial_X \|A\| \right| \leq n \cdot \|A\| \cdot e^{(\log l_0)^{5\hat{\epsilon}}}, \quad \left| \frac{\partial^2 \|A\|}{\partial X \partial Y} \right| \leq n^2 \cdot \|A\| \cdot e^{(\log l_0)^{5\hat{\epsilon}}},
\end{equation}
\begin{equation}\label{re3}
\|s(A) - \frac{\pi}{2}\|_{C^2(I)} \leq l_1^{-2} e^{(\log l_0)^{5\hat{\epsilon}}}, \quad \|s(A^{-1})\|_{C^2(I)} \leq l_n^{-2} e^{(\log l_0)^{5\hat{\epsilon}}}.
\end{equation}

Moreover, for a rectangle \(\tilde{I} \subset I\) with \(\tilde{I}_x =\Pi_1\tilde{I}\), \(\tilde{I}_y = \Pi_2\tilde{I}\) and \(\max\limits_{X\in\{x,y\}} |\tilde{I}_X| \leq n^{-5} e^{-(\log l_0)^{50\hat{\epsilon}}}\), we have
\begin{equation}\label{axyxy}
\|A(x,y)\| \sim_{0, \max\limits_{X\in\{x,y\}} |\tilde{I}_X|^{\frac{1}{2}}} \|A(\tilde{x},\tilde{y})\|,\quad (\tilde{x},\tilde{y}) \in \tilde{I}.
\end{equation}
\end{lemma}
\begin{remark}\label{remark10} It is not difficult to see from the proof that if the condition is changed to $
\lambda_i(x,y) \geq l_0 \text{ for } i=1,2,\ldots,n,
$
$
\min\limits_{i=1,2,\ldots,n-1} \inf\limits_{I} |f_i(x,y) - \frac{\pi}{2}| \geq \eta,
$
$
\sup\limits_{I} \left|\frac{1}{\lambda_k} \frac{\partial \lambda_k}{\partial X}\right|, $ $\quad \sup\limits_{I} \left|\frac{1}{\lambda_k} \frac{\partial^2 \lambda_k}{\partial X \partial Y}\right| \leq \eta^{-1},
$
$
\sup_{I} \left|\frac{\partial^2 f_k}{\partial X \partial Y}\right|, \quad \sup_{I} \left|\frac{\partial f_k}{\partial X}\right| \leq \eta^{-1}.
$ with some $\eta\ll e^{(\log n)^C}.$ Then \eqref{re1}-\eqref{re3} can be changed to
$
\|A\| \geq l_0^{(1 - |\log \eta|^C)n},
$
$
\left| \partial_X \|A\| \right| \leq n \cdot \|A\| \cdot \eta^{-C}, \quad \left| \frac{\partial^2 \|A\|}{\partial X \partial Y} \right| \leq n^2 \cdot \|A\| \cdot \eta^{-C},
$
$
\|s(A) - \frac{\pi}{2}\|_{C^2(I)} \leq l_1^{-2} \eta^{-C}, \quad \|s(A^{-1})\|_{C^2(I)} \leq l_n^{-2} \eta^{-C}.
$
\end{remark}
\begin{proof}

By \eqref{jfact2}, using \eqref{lem10-i1} of Lemma \ref{lemma9*}, we have
$$\|\Lambda_2 R_{f_1} \Lambda_1\| \sim_{0,l_0^{-1}} l_1l_2|\cos f_1| > l_0,\quad~\|\Lambda_3 R_{f_2} \Lambda_2\| \sim_{0,l_0^{-1}} l_3l_2|\cos f_2| > l_0.$$

Using \eqref{lem8-i}, we obtain
$$|s((\Lambda_2 R_{f_1} \Lambda_1)^{-1})| \leq Ce^{(\log l_0)^{3\hat{\epsilon}}} l_2^{-2} \ll e^{-(\log l_0)^{2\hat{\epsilon}}} \leq \inf_{I} |f_1(x,y) - \frac{\pi}{2}|,$$
$$|\frac{\pi}{2}-s(\Lambda_3 R_{f_2} \Lambda_2)| \leq Ce^{(\log l_0)^{3\hat{\epsilon}}} l_2^{-2} \ll e^{-(\log l_0)^{2\hat{\epsilon}}} \leq \inf_{I} |f_2(x,y) - \frac{\pi}{2}|.$$

Therefore
$$\left\vert\frac{\sin s((\Lambda_2 R_{f_1} \Lambda_1)^{-1})}{\sin (f_1 - \frac{\pi}{2})}\right\vert+\left\vert\frac{\sin (\frac{\pi}{2}-s(\Lambda_3 R_{f_2} \Lambda_2))}{\sin (f_2 - \frac{\pi}{2})}\right\vert\leq l_0^{-\frac{3}{2}}.$$

Consequently
$$\begin{array}{ll}
&\left\vert\frac{\cos (f_2 + s((\Lambda_2 R_{f_1} \Lambda_1)^{-1}))-\cos f_2}{\cos f_2}\right\vert
= \left\vert\frac{2\sin(f_2+\frac{s((\Lambda_2 R_{f_1} \Lambda_1)^{-1})}{2})\sin(\frac{s((\Lambda_2 R_{f_1} \Lambda_1)^{-1})}{2})}{\sin (f_2-\frac{\pi}{2})}\right\vert
\\&\leq \left\vert\frac{\sin(\frac{s((\Lambda_2 R_{f_1} \Lambda_1)^{-1})}{2})}{\sin (f_2-\frac{\pi}{2})}\right\vert\leq \left\vert\frac{\sin(s((\Lambda_2 R_{f_1} \Lambda_1)^{-1}))}{\sin (f_2-\frac{\pi}{2})}\right\vert\leq l_0^{-\frac{3}{2}}.\end{array}
$$
Hence
\begin{equation}\label{fas}
|\cos (f_2 + s(\Lambda_2 R_{f_1} \Lambda_1)^{-1}))| \sim_{0,l_0^{-1}} |\cos f_2|.
\end{equation}

Similarly, we have
\begin{equation}\label{fas*}
|\cos (f_1 + \frac{\pi}{2}-s(\Lambda_3 R_{f_2} \Lambda_2))| \sim_{0,l_0^{-1}} |\cos f_1|.
\end{equation}

Then, performing polarization decomposition on $\Lambda_2 R_{f_1} \Lambda_1$ yields
$$\Lambda_3 R_{f_2} \Lambda_2 R_{f_1} \Lambda_1 = \Lambda_3 R_{f_2 + s((\Lambda_2 R_{f_1} \Lambda_1)^{-1})} \text{diag}\{\|\Lambda_2 R_{f_1} \Lambda_1\|, \|\Lambda_2 R_{f_1} \Lambda_1\|^{-1}\} R_{s(\Lambda_2 R_{f_1} \Lambda_1)}.$$

Hence
\[
\|\Lambda_3 R_{f_2} \Lambda_2 R_{f_1} \Lambda_1\| = \|\Lambda_3 R_{f_2 + s((\Lambda_2 R_{f_1} \Lambda_1)^{-1})} \text{diag}\{\|\Lambda_2 R_{f_1} \Lambda_1\|, \|\Lambda_2 R_{f_1} \Lambda_1\|^{-1}\}\|
\]
\[
\sim_{0,l_0^{-1}} l_1 l_2 l_3 |\cos (f_2 + s((\Lambda_2 R_{f_1} \Lambda_1)^{-1})) \cos f_1| \sim_{0,l_0^{-1}} l_1 l_2 l_3 |\cos f_1 \cos f_2|.
\]

On the other hand, given that \(\|\Lambda_3\|\) and \(\|\Lambda_2 R_{f_1} \Lambda_1\|\) are both greater than \(l_0\), we apply \eqref{fas} to obtain
$$|s((\Lambda_3 R_{f_2} \Lambda_2 R_{f_1} \Lambda_1)^{-1})| \leq Ce^{(\log l_0)^{3\hat{\epsilon}}} l_3^{-2} \ll e^{-(\log l_0)^{2\hat{\epsilon}}} \leq \min_{i=1,2,\ldots,n-1} \inf_{I} |f_i(x,y) - \frac{\pi}{2}|,$$
and apply \eqref{fas*} to obtain
$$|\frac{\pi}{2}-s(\Lambda_3 R_{f_2} \Lambda_2 R_{f_1} \Lambda_1)| \leq Ce^{(\log l_0)^{3\hat{\epsilon}}} l_1^{-2} \ll e^{-(\log l_0)^{2\hat{\epsilon}}} \leq \min_{i=1,2,\ldots,n-1} \inf_{I} |f_i(x,y) - \frac{\pi}{2}|.$$

Note \eqref{lem10-i2} and \eqref{lem10-i3} of Lemma \eqref{lemma9*} imply
\begin{equation}\label{80e123}
\frac{\partial_X \|\Lambda_2 R_{f_1} \Lambda_1\|}{\|\Lambda_2 R_{f_1} \Lambda_1\|} = \frac{\partial_X (\Lambda_1 \Lambda_2 \cos f_1)}{\Lambda_1 \Lambda_2 \cos f_1} + o(l_0^{-1}),
\end{equation}
and
$$
\frac{\partial^2_{XY} \|\Lambda_2 R_{f_1} \Lambda_1\|}{\|\Lambda_2 R_{f_1} \Lambda_1\|} = \frac{\partial^2_{XY} (\Lambda_1 \Lambda_2 \cos f_1)}{\Lambda_1 \Lambda_2 \cos f_1} + o(l_0^{-1}).
$$
Let \(\|\Lambda_2 R_{f_1} \Lambda_1\| = \tilde{\Lambda}_2\), \(f_2 + s((\Lambda_2 R_{f_1} \Lambda_1)^{-1}) = \tilde{f}_2\), and \(f_1 = \tilde{f}_1\). By \eqref{lem10-i2} and \eqref{80e123},

we have

\begin{align*}
&\frac{\partial_X \|\Lambda_3 R_{f_2} \Lambda_2 R_{\tilde{f}_2} \Lambda_1\|}{\|\Lambda_3 R_{f_2} \Lambda_2 R_{\tilde{f}_2} \Lambda_1\|}= \frac{\partial_X (\Lambda_3 \tilde{\Lambda}_2 |\cos \tilde{f}_2|)}{\Lambda_3 \tilde{\Lambda}_2 |\cos \tilde{f}_2|} + o(l_0^{-1}) \\
&= \frac{\partial_X (\Lambda_3) (\tilde{\Lambda}_2 |\cos \tilde{f}_2|)
+ \partial_X (\tilde{\Lambda}_2) (\Lambda_3 |\cos \tilde{f}_2|)
+ \partial_X (|\cos \tilde{f}_2|) (\tilde{\Lambda}_2 \Lambda_3)}{\Lambda_3 \tilde{\Lambda}_2 |\cos \tilde{f}_2|} + o(l_0^{-1}) \\
&= \frac{\partial_X (\Lambda_3) (\tilde{\Lambda}_2 |\cos \tilde{f}_2|)}{\Lambda_3 \tilde{\Lambda}_2 |\cos \tilde{f}_2|}
+ \frac{(\tilde{\Lambda}_2) (\frac{\partial_X (\Lambda_1 \Lambda_2 |\cos f_1|)}{\Lambda_1 \Lambda_2 |\cos f_1|}+o(l_0^{-1})) (\Lambda_3 |\cos \tilde{f}_2|))}{\Lambda_3 \tilde{\Lambda}_2 |\cos \tilde{f}_2|}
+ \frac{\partial_X (|\cos \tilde{f}_2|) (\tilde{\Lambda}_2 \Lambda_3)}{\Lambda_3 \tilde{\Lambda}_2 |\cos \tilde{f}_2|} + o(l_0^{-1}).
\\&= (1+o(l_0^{-1}))\left(\frac{\partial_X (\Lambda_3) (\Lambda_1\Lambda_2|\cos f_1| |\cos \tilde{f}_2|)}{\Lambda_1\Lambda_2\Lambda_3 |\cos f_1||\cos \tilde{f}_2|}
+ \frac{(\partial_X (\Lambda_1 \Lambda_2 |\cos f_1|)) (\Lambda_3 |\cos \tilde{f}_2|)}{\Lambda_1\Lambda_2\Lambda_3 |\cos f_1||\cos \tilde{f}_2|}
+ \frac{\partial_X (|\cos \tilde{f}_2|) (\Lambda_1 \Lambda_2 \Lambda_3|\cos f_1|)}{\Lambda_1\Lambda_2\Lambda_3 |\cos f_1||\cos \tilde{f}_2|}\right) + o(l_0^{-1}).
\\
&= (1 + o(l_0^{-1}))
\frac{\partial_X (\Lambda_3 \Lambda_1 \Lambda_2 |\cos {f}_1| |\cos \tilde{f}_2|)}{\Lambda_3 \Lambda_1 \Lambda_2 |\cos {f}_1| |\cos \tilde{f}_2|} + o(l_0^{-1}) = (1 + 2 o(l_0^{-1}))
\frac{\partial_X (\Lambda_1 \Lambda_2 \Lambda_3 |\cos f_1| |\cos f_2|)}{\Lambda_1 \Lambda_2 \Lambda_3 |\cos f_1| |\cos f_2|} + o(l_0^{-1}).
\end{align*}

Similarly,
$$\frac{\partial^2_{XY} \|\Lambda_3 R_{f_2} \Lambda_2 R_{\tilde{f}_2} \Lambda_1\|}{\|\Lambda_3 R_{f_2} \Lambda_2 R_{\tilde{f}_2} \Lambda_1\|} = (1 + 2 o(l_0^{-1})) \frac{\partial^2_{XY} (\Lambda_1 \Lambda_2 \Lambda_3 |\cos f_1| |\cos f_2|)}{\Lambda_1 \Lambda_2 \Lambda_3 |\cos f_1| |\cos f_2|} + o(l_0^{-1}).
$$

By induction, we obtain
\begin{equation}\label{new1}
\|A\| \sim_{0, nl_0^{-1}} \left(\prod_{i=1}^n l_i\right) \left(\prod_{i=1}^{n-1} |\cos f_i|\right) > l_0^n e^{-n (\log l_0)^{2\hat{\epsilon}}} = l_0^{(1 - (\log l_0)^{2\hat{\epsilon}-1})n} > l_0^{(1 - (\log l_0)^{-\frac{1}{2}})n},
\end{equation}
which implies \eqref{re1}. Similarly, we have
\begin{equation}\label{wehave}
\left|\frac{\partial_X \|A\|}{\|A\|}\right| = \left|(1 + 2n [o(l_0^{-1})]) \frac{\partial_X [\left(\prod_{i=1}^n \Lambda_i\right) \left(\prod_{i=1}^{n-1} \cos f_i\right)]}{\left(\prod_{i=1}^n \Lambda_i\right) \left(\prod_{i=1}^{n-1} \cos f_i\right)} + o(l_0^{-1})\right|
\end{equation}
\[
\leq (1 + 2n [o(l_0^{-1})]) (2n - 1) e^{(\log l_0)^{\hat{\epsilon}}} \leq n e^{(\log l_0)^{5\hat{\epsilon}}}
\]
and
\[
\left|\frac{\partial^2 \|A\|}{\partial X \partial Y}\right| = \left|(1 + 2n [o(l_0^{-1})]) \frac{\partial^2_{XY} [\left(\prod_{i=1}^n \Lambda_i\right) \left(\prod_{i=1}^{n-1} \cos f_i\right)]}{\left(\prod_{i=1}^n \Lambda_i\right) \left(\prod_{i=1}^{n-1} \cos f_i\right)} + o(l_0^{-1})\right|
\]
\[
\leq (1 + 2n [o(l_0^{-1})]) (2n - 1)^2 e^{2 (\log l_0)^{\hat{\epsilon}}} \leq n^2 e^{(\log l_0)^{5\hat{\epsilon}}},
\]
which imply \eqref{re2}. Additionally, \eqref{re3} can be obtained from (i) of Lemma \ref{lemma8*}.

It remains to prove \eqref{axyxy}. By (\ref{new1}), we have
\begin{equation}\label{momomo}
\|A\| \sim_{0, nl_0^{-1}} \left(\prod_{i=1}^n l_i\right) \left(\prod_{i=1}^{n-1} |\cos f_i|\right).
\end{equation}
Moreover, since
$$\inf_{I} |f_i(x, y) - \frac{\pi}{2}| \geq e^{-(\log l_0)^{2\hat{\epsilon}}} \gg |\tilde{I}|,$$
$$\sup_{I} \left|\frac{\partial^2 f_i}{\partial X \partial Y}\right|,\quad \sup_{I} \left|\frac{\partial f_i}{\partial X}\right| \leq e^{(\log l_0)^{\hat{\epsilon}}} \ll |\tilde{I}|^{-1}, \quad i = 1, 2, \ldots, n-1$$
for any fixed \((\tilde{x}, \tilde{y}) \in \tilde{I}\), we have
$$\cos f_i(x, y) \sim_{0, \max\limits_{X=x, y} |\tilde{I}_X|} \cos f_i(\tilde{x}, \tilde{y}) \text{ on } \tilde{I}.$$

Then \eqref{momomo} implies
$$\|A(x, y)\| \sim_{0, nl_0^{-1} + n \max\limits_{X\in\{x,y\}} |\tilde{I}_X|^{\frac{1}{2}}} \|A(\tilde{x}, \tilde{y})\|~\text{ on } \tilde{I}.$$

Since \(l_0^{-1} \ll n^{-5} e^{-(\log l_0)^{50\hat{\epsilon}}}\) and \(\max_{X\in\{x,y\}} |\tilde{I}_X| \leq n^{-5} e^{-(\log l_0)^{50\hat{\epsilon}}}\), we have
$$\|A(x, y)\| \sim_{0, n^{-4} e^{-(\log l_0)^{50\hat{\epsilon}}}} \|A(\tilde{x}, \tilde{y})\|~\text{ on } \tilde{I}.$$

Therefore
$$(1 - n^{-4} e^{-(\log l_0)^{50\hat{\epsilon}}}) \|A(\tilde{x}, \tilde{y})\| \leq \|A(x, y)\| \leq (1 + n^{-4} e^{-(\log l_0)^{50\hat{\epsilon}}}) \|A(\tilde{x}, \tilde{y})\|.$$

Finally, by using \eqref{wehave} we obtain
\[
\frac{\left|\|A(x, y)\| - \|A(\tilde{x}, \tilde{y})\|\right|}{\|A(x, y)\|} \leq \frac{(1 + n^{-4} e^{-(\log l_0)^{50\hat{\epsilon}}}) \|A(\tilde{x}, \tilde{y})\| \max_{X\in\{x,y\}} |\tilde{I}_X|}{(1 - n^{-4} e^{-(\log l_0)^{50\hat{\epsilon}}}) \|A(\tilde{x}, \tilde{y})\|} \leq 2 \max_{X\in\{x,y\}} |\tilde{I}_X|<\max_{X\in\{x,y\}} |\tilde{I}_X|^{\frac{1}{2}},
\]
which implies
$$\|A(x, y)\| \sim_{0, \max_{X\in\{x,y\}} |\tilde{I}_X|^{\frac{1}{2}}} \|A(\tilde{x}, \tilde{y})\|~\text{ on } \tilde{I}$$ as desired.\hfill\qed\end{proof}
Iterative techniques for non-resonant case (characterized by (\ref{jfact2}))  have been provided in the above lemma.

For the resonant case, which is much more complex than the non-resonant case, we have to consider the function with the form $$\arctan [l^2 \tan c_1x]-\frac{\pi}{2}+c_2x~with~c_1\cdot c_2<0.$$
The following lemma accurately  described the geometry of the angle function at resonance, which is the crucial part in the induction Theorem stated later.

\begin{lemma}\label{thetat1fenjie}
Fix $0 < \hat{\epsilon} \ll 1$ and $k \gg 1$, and let
$$\Gamma = (\log k)^{\hat{\epsilon}^{-1}}, \quad \epsilon =e^{-\Gamma}= e^{-(\log k)^{\hat{\epsilon}^{-1}}}, \quad l = e^k.$$
Assume that for $i = 1, 2$ and on $D$, we have
$$\partial_x h_i(x,y) \sim_{0,\epsilon} a_i, \quad \partial_y h_i(x,y) \sim_{0,\epsilon} b_i,$$
where $h_i \in C^2(D)$, with $D$ being a rectangle in $\mathbb{R}^2$, and $a_i, b_i$ are constants satisfying
$$
|a_i|, |a_i|^{-1}, |b_i|, |b_i|^{-1} \leq \Gamma, \quad -\text{sgn}(a_1) = \text{sgn}(a_2) = \text{sgn}(b_1) = \text{sgn}(b_2) = 1.
$$
Let $L(x,y) \in C^2(D)$. Suppose that for any fixed $y \in \Pi_2 D$, we have
\begin{equation}\label{epepep}
[-\epsilon, \epsilon] \subseteq \text{Ran}(\tan h_i(\cdot, y)) \subseteq [-\epsilon^{\frac{2}{3}}, \epsilon^{\frac{2}{3}}],
\end{equation}
and
\begin{equation}\label{lem48-tj}
\sum_{X,Y\in\{x,y\}} \left| \frac{\partial^2 h_i(x,y)}{\partial X \partial Y} \right| \leq \Gamma, \quad L(x,y) \sim_{0,\epsilon} l \text{ on } D,
\end{equation}
\begin{equation}\label{lem48-tj111}
\sum_{X,Y\in\{x,y\}} \left| \frac{\partial^2 (\log L(x,y))}{\partial X \partial Y} \right|, \quad \sum_{X\in\{x,y\}} \left| \partial_X (\log L(x,y)) \right| \leq e^{|\log k|^{C}}.
\end{equation}

Set
$$F(x,y) = \tan^{-1}\left( L^2(x,y) \tan h_2(x,y) \right) - \frac{\pi}{2} + h_1(x,y).$$
Then, for any fixed $y \in \Pi_2 D$, the following hold.

\begin{itemize}
    \item[{\rm i.}]
    $$\left| \{ x \in \Pi_1 D \mid F(x,y) = 0 \pmod{\pi} \} \right| \leq 2.$$

    \item[{\rm ii.}]
    \begin{equation}\label{2jiedaos}
    \max_{(x,y) \in D} \left( |F(x,y)| + \left| \frac{\partial F(x,y)}{\partial x} \right| + \left| \frac{\partial F(x,y)}{\partial y} \right| + \sum_{X,Y\in\{x,y\}} \left| \frac{\partial^2 F(x,y)}{\partial X \partial Y} \right| \right) \leq Cl^8,
    \end{equation}
    \begin{equation}\label{fthx2}
    \min_{x \in \Pi_1 D} \left( \left| \frac{\partial F}{\partial x} \right| + \left| \frac{\partial^2 F}{\partial x^2} \right| \right) \geq \Gamma^{-2}.
    \end{equation}
    There exist two functions $A(x,y), B(x,y) \in C^1(D)$ such that
  $$
    \partial_y F(x,y) = A(x,y) \cdot \partial_x F(x,y) + B(x,y), \quad (x,y) \in D,
  $$
    with
    \begin{equation}\label{ABAB}
    A(x,y) \sim_{0, \epsilon^{\frac{1}{2}}} \frac{b_2}{a_2}, \quad B(x,y) \sim_{0, \epsilon^{\frac{1}{2}}} a_1 \left( \frac{b_2}{a_2} - \frac{b_1}{a_1} \right) \text{ on } D.
    \end{equation}

    \item[{\rm iii.}]

    \

    \begin{enumerate}
        \item[{\rm iii-a.}] The set $\{ x \in \Pi_1 D \mid \partial_x F(x,y) = 0 \pmod{\pi} \} = \{ x^*_1(y), x^*_2(y) \}$. Moreover, $F(x,y)$ is strictly decreasing on $\Pi_1 D - (x_1^*(y), x_2^*(y))$ and strictly increasing on $(x_1^*(y), x_2^*(y))$. Additionally,
        $$\{ x \in \Pi_1 D \mid |\partial_x F(x,y)| \leq \epsilon^{\frac{3}{4}} \} = J_1(y) \cup J_2(y),$$
        where
        $$J_i(y) = \left( x^*_i - \epsilon^*_{i,-} l^{-1}, x^*_i + \epsilon^*_{i,+} l^{-1} \right),$$
        and
        $$\epsilon^*_{i,X} \sim_{0, \epsilon^{\frac{1}{2}}} \frac{1}{2} |a_2|^{-\frac{1}{2}} |a_1|^{-\frac{3}{2}} \epsilon^{\frac{3}{4}}, \quad i=1,2, \ X\in\{+,-\}.$$

        \item[{\rm iii-b.}] For any $x < x^*_1(y)$ or $x > x^*_2(y)$, we have
        \begin{equation}\label{iii-2}
        b_1 \leq |\partial_y F(x,y)| \leq b_1 + b_2 \cdot \frac{|a_1|}{a_2},
        \end{equation}
        and
        \begin{equation}\label{iii-2'}
        |\partial_x F(x,y)| \leq |a_1|.
        \end{equation}

        \item[{\rm iii-c.}] Assume that $F(x,y)$ has two zeros, denoted by $\tilde{z}_1(y)$ and $\tilde{z}_2(y)$. Then for $i = 1, 2$,
        either
        \begin{equation}\label{iii-3*}
        x^*_1(y) \leq \tilde{z}_1(y) \leq x^*_2(y) \leq \tilde{z}_2(y) \quad \text{and} \quad |\partial_x F(\tilde{z}_i(y),y)| \geq \epsilon |x^*_2(y) - \tilde{z}_i(y)|,
        \end{equation}
        or
        $$\tilde{z}_1(y) \leq x^*_1(y) \leq \tilde{z}_2(y) < x^*_2(y) \quad \text{and} \quad |\partial_x F(\tilde{z}_i(y),y)| \geq \epsilon |x^*_1(y) - \tilde{z}_i(y)|.$$

        \item[{\rm iii-d.}] Under the assumptions in iii-c, we have
        \begin{equation}\label{iii-4}
        \eta \Gamma^{-1} \leq |\tilde{z}_2(y) - \tilde{z}_1(y)| \leq \min \{ 4 \Gamma^2 \epsilon^{\frac{2}{3}}, 2 \eta^{\frac{1}{2}} \epsilon^{-\frac{1}{2}} \},
        \end{equation}
        where
        $$\eta = \min \left\{ |F(x^*_1(y),y) \pmod{\pi}|, |F(x^*_2(y),y) \pmod{\pi}| \right\}.$$

        \item[{\rm iii-e.}] Set $\tilde{I}(y) := \left\{ x \in \Pi_1 D \mid |F(x,y)| < \epsilon^3 l^{-16} e^{-|\log |x^*_2(y) -\tilde{z}_2(y)| |^C} \right\}$. Then
        \begin{equation}\label{iii-3}
        |\partial_x F(x,y)| \sim_{0, l^9 |\tilde{I}|^{\frac{1}{2}}} |\partial_x F(\tilde{z}_2(y),y)| \text{ on } \tilde{I}(y),
        \end{equation}
        \begin{equation}\label{iii-3'}
        |\partial_y F(x,y)| \sim_{0, l^9 |\tilde{I}|^{\frac{1}{2}}} |\partial_y F(\tilde{z}_2(y),y)| \text{ on } \tilde{I}(y).
        \end{equation}
    \end{enumerate}

    \item[{\rm iv.}]
    $$\frac{\partial^2 F(x,y)}{\partial x^2} \sim_{0, \epsilon^{\frac{1}{2}}} 2 |a_2|^{\frac{1}{2}} |a_1|^{\frac{3}{2}} l \text{ on } J_1(y),$$
    $$\frac{\partial^2 F(x,y)}{\partial x^2} \sim_{0, \epsilon^{\frac{1}{2}}} -2 |a_2|^{\frac{1}{2}} |a_1|^{\frac{3}{2}} l \text{ on } J_2(y).$$

    \item[{\rm v.}]
    $$|\{F(x,y) \mid x \in \Pi_1 D \}| \sim_{0, \epsilon^{\frac{1}{2}}} \pi - 4 \sqrt{\frac{|a_1|}{|a_2|}} l^{-1}.$$

    \item[{\rm vi.}] Let
    $$\tilde{F} = \tan^{-1}(L^2(x,y) \tan h_1(x,y)) - \frac{\pi}{2} + h_2(x,y)$$
    (swap $h_1$ and $h_2$ in the definition of $F$). Then we have
    $$\{y \mid \min_x |F(x,y)| > l^{-100C} \} \subseteq \{y \mid \min_x |\tilde{F}(x,y)| > l^{-200C}\},$$
    $$\{y \mid \min_x |\tilde{F}(x,y)| > l^{-100C} \} \subseteq \{y \mid \min_x |F(x,y)| > l^{-200C}\}.$$
\end{itemize}

\end{lemma}

The proof can be found in the appendix.

\subsection{The induction Theorem for $C^2$ Cosine Type}

Let $\{p_n/q_n\}$ be the fraction approximant of $\alpha$. Note that there exists $C_{\alpha}>0$ by the Diophantine condition such that
$q_{s+1}<C_{\alpha}q_s^{\tau-1},~s\in \Z_+.$
Suppose that $N\gg \|v\|_{\mathrm{C}^2}$ and sufficiently large such that
$$\sum_{n\geq N}q_{N+n-1}^{-\frac1{100}} \leq\frac 1{100}.$$

 Denote $$s_1(x,t)=s[A_1(x,t)],~u_1(x,t)=s[A_{-1}(x,t)].$$
 From Proposition \ref{lemma4}, for the initial angle function $g_1\in\mathrm{C}^2(\mathbb \T,\R)$ defined as  $g_1(x,t):=s_1(x,t)-u_1(x,t)$, we have
\begin{equation}\label{g_N}
g_1(x,t)=\phi(x, t, \lambda)+o(\lambda^{-1})=\arctan[t-v(x-\alpha)]+o(\lambda^{-1}),\quad \lambda\rightarrow \infty.
\end{equation}


  Throughout the paper, we fix a  large $N=N(\alpha, v)$ and  let $0<\hat{\epsilon}\ll 1$, let $\lambda\gg N$ such that $$(log\lambda)^{\hat{\epsilon}}\gg e^{q_N}.$$ Denote \begin{equation}\label{HN}\mathcal{N}_i=\mathcal{N}_i(\hat{\epsilon})=\left\{\begin{matrix}e^{(\log \lambda)^{\hat{\epsilon}}} & i=1\\ [{\lambda}^{q_{N+i-1}^{\hat{\epsilon}}}]& i\geq 2. \end{matrix}\right.\end{equation}

First we need the following definition.
\begin{definition} Let $f\in C^1(I_l\times J)$ with $I_l\subset \R/\Z,~J\subset \R$ with $l=1,2,\cdots,L$. Given $1>\epsilon>0,$ we say $f$ satisfies $\eta-$\textbf{nonresonant} condition if \\
{\rm (a)}\ for  each $1\leq l\leq L,$ the set $\{x\in I_l \vert f(x,y)=0\}$ consists of one single element, denoted by $x_l(y);$ \\
\rm{(b)}\
for any $y$ satisfying
$(y-\eta,y+\eta)\in J$ and $(x_l(y)-\eta,x_l(y)+\eta)\in I_l$ for all $1\leq l\leq L,$ it holds that
$${\eta^{\hat{\epsilon}^{-1}}}\le|\partial^i_x\partial_y^j f(x_l(y),y)| \leq {\eta^{-\hat{\epsilon}^{-1}}},\quad 1\le i+j\le 2$$

 and $$f(x,y')\sim_{1,\eta^{\frac{1}{2}}}\partial_x f(x_l(y),y)(x-x_l(y))+\partial_y f(x_l(y),y)(y'-y)$$ on $(x_l(y)-\eta,x_l(y)+\eta)\times (y-\eta,y+\eta)$.

\end{definition}

\textbf{Step 1:}

We define the following concepts.
\begin{enumerate}

\item The $\textbf{critical~points}$ for the first step:
Let $c_{1,j}(t),\ j=1, 2, \cdots, J,$ be all points on $\mathbb \T$ minimizing $\{|g_{1,j}(x,t)|(\text{mod}~\pi)\}.$ From (\ref{g_N}) and the cos-type condition on $v$, we have
$J=1$ or $2$ and if $J=2$, $c_{1,j}(t),\ j=1,\ 2$ is roughly equal to zeroes of $t-v(x-\alpha)$.
For simplicity, we regard the case $J=1$ as a special case of $J=2$ by assuming $c_{1,1}(t)=c_{1,2}(t)$ and denote
$$
C^{(1)}(t)=\{c_{1,1}(t),c_{1,2}(t)\}.
$$
\item  The $\textbf{critical~intervals}$ for the first step:
$$I_{1,j}(t)=\{x:|x-c_{1,j}(t)|\leq {\mathcal{N}_1}^{-{\hat{\epsilon}^{-1}}}\},\ j=1, 2~and~I_1(t)=I_{1,1}(t)\cup I_{1,2}(t).$$
Here we have to consider the following two cases:
\begin{enumerate}
\item  $I_{1,1}(t)\bigcap I_{1,2}(t)\not=\emptyset.$ We denote this step by Type $\textbf{II}_{1}$. Note in this case we have
$\min\limits_{x\in I_1}|\partial_x g_1|=0.$ Furthermore \eqref{g_N} and the cos-type condition on $v$ imply for any $x\in I_1(t),$ it holds that \begin{equation}\label{2jfttt}\frac{1}{2}c<\left\vert\frac{\partial^2 g_1(x,t)}{\partial x^2}(x,t)\right\vert<C.\end{equation}
\item \ $I_{1,1}(t)\bigcap I_{1,2}(t)= \emptyset.$ We denote this step by Type $\textbf{I}_{1}.$ Note it holds from \eqref{g_N} and the cos-type condition on $v$ that in this case for $j=1,2$ we have
\begin{equation}\label{1jdfthh0}\mathcal{N}_1^{2{\hat{\epsilon}^{-1}}}>|\partial_x g_1(c_{1,j}(t),t)|,\ |\partial_t g_1(c_{1,j}(t),t)|>c\mathcal{N}_1^{-{\hat{\epsilon}^{-1}}}>\mathcal{N}_1^{-2{\hat{\epsilon}^{-1}}}.\end{equation} Note \eqref{2jfttt} still holds true. Thus for any $t\in [\inf v-\frac{2}{\lambda},\sup v+\frac{2}{\lambda}] $ and any $0<\eta \leq \mathcal{N}_1^{-\hat{\epsilon}^{-2}}$, it holds that
\begin{equation}\label{1jdfthh}g_1(x,t')\sim_{1,\eta^{\frac{1}{2}}} \partial_x g_1(c_{1,j}(t),t)(x-c_{1,j}(t))+\partial_t g_1(c_{1,j}(t),t)(t'-t)\end{equation} on  $(c_{1,j}(t)-\eta,c_{1,j}(t)+\eta)\times (t-\eta,t+\eta)$.

Therefore  $g_1(x,t)$ satisfies $\mathcal{N}_1^{-\hat{\epsilon}^{-2}}-$ \textbf{non-resonant} condition on $I_0\times [\inf v-\frac{2}{\lambda},\sup v+\frac{2}{\lambda}].$
\end{enumerate}
\item The $\textbf{returning~time}$ for the second step:

Let
$$ r_1^{+}(x,t)(resp.~r_1^{-}(x,t))\ge q_N^2:I_1(t)\rightarrow \Z^+$$
be the first forward (resp. backward) returning time of $x\in I_1(t)$ back to $I_1(t)$ {\it after $q_N^2-1.$}

  Let $r_1(t)=\min\{r_1^+(t),r_1^-(t)\}$ with $r_1^{\pm}(t)=\min_{x\in I_1(t)}r_1^{\pm}(x,t)$. Denote $$m_1^{\pm}(t):=\min\{n \vert (I_{1,1}+n\alpha)\bigcap I_{1,2}\neq \emptyset,~\pm n\geq 1\}.$$
\end{enumerate}
Note that if $r_1>q_N^2$, by the Diophantine condition, then
    $$r_1(t)=\min\limits_{X=+,-}\min\limits_{x\in I_1(t)}\{r^{X}_1(x,t)\}>\mathcal{N}_1^c.$$


 Now we assume that for $i\geq 1,$  the following are well defined (will be proved in Appendix).
\begin{enumerate}
\item  The $critical~points$ for the $i$th step:
$$C_i(t)=\{c_{i,1}(t),c_{i,2}(t)\}$$
with $c_{i,j}(t)\in I_{i-1,j}(t)$ minimizing $\{|g_i(x,t)||~x\in I_{i-1,j}(t)\}$
(it is possible that $c_{i,1}(t)=c_{i,2}(t)$).


\item  The $critical~intervals$ for the $i$th step:
$$I_{i,j}(t)=\{x:|x-c_{i,j}(t)|\leq \mathcal{N}^{-1}_i\}~and~I_i(t)=I_{i,1}(t)\cup I_{i,2}(t).$$

\item  The $angle~function$ for the $(i+1)$-th step:
$$g_{i+1}(x,t)=s_{r_i(t)}(x,t)-u_{r_i(t)}(x,t):D_i\rightarrow \R\mathbb{P}^1,$$
where we define
$$D_i(t'):=\{(x,t):x\in I_i(t),t\in (t'-\lambda^{-q_{N+i-1}},t'+\lambda^{-q_{N+i-1}})\}.$$

\item The $returning~time$ for the $(i+1)$th step: (denote $r_0:=0$)
$$ r_{i}^{\pm}(x,t)\ge \max\{q^2_{N+i-1},r_{i-1}\}: I_{i}(t)\rightarrow \Z^+,$$
that is, the first forward (backward) returning time (back to $I_i(t)$) {\it after} $\max\{q^2_{N+i-1},r_{i-1}\}-1.$  Let $r_i(t)=\min\{r_i^+(t),r_i^-(t)\}$ with $r_i^{\pm}(t)=\min_{x\in I_i(t)} r_i^{\pm}(x,t)$ (for convenience, we assume $r_0(t)\equiv 0$).
\end{enumerate}

Denote $\tilde{I}_{n,j}=\{x\in I_{n,j}||g_{n+1}\ {\rm mod\ }\pi|\le \lambda^{-r^{\frac{1}{700}}_{n}}\} ,~j=1,2.$

\begin{definition}[Types~of~step~i]

\

\textbf{(Non-resonant case)} If  $T^k{I_{i,1}} \bigcap I_{i,2}=\emptyset$ for each
$0 \le |k| <  q^2_{N+i-1}$, then we say step $i$ belongs to Type $\textbf{I}_{i}$ (or $\textbf{I}$).

\textbf{(Resonant case)} If there exists  $0 \le |k| <  q^2_{N+i-1}$ such that $T^k{I_{i,1}} \bigcap I_{i,2}\not=\emptyset$,  then we say step $i$ belongs to Type $\textbf{II}^k_{i}$ (or $\textbf{II}_{i}$,\ $\textbf{II}$).

\

We write $X_i\rightarrow Y_{i+1}$ if the induction goes from step $i$ of type $X_i$ to step $i+1$ of type $Y_{i+1},$ where
$X,Y\in\{\bf{I},\bf{II}\}$.
\end{definition}

\begin{theorem}\label{theorem12} Let $t\in [\inf v-\frac{2}{\lambda},\sup v+\frac{2}{\lambda}]$. For any $\epsilon>0,$ there exists $\lambda_0=\lambda_0(v,\alpha,\epsilon)>0$ such that for all $\lambda>\lambda_0$, the following hold for each $i\geq 2$ and $x\in I_i$.
             $$|c_{i,j}(t)-c_{i+1,j}(t)|<C\lambda^{-\frac{3}{4}r_{i-1}},j=1,2$$ and
$$\|A_{\pm r_{i}}(x,t)\|>\lambda^{(1-\epsilon)r_{i}}.
$$

         For $X,Y\in \{x,t\}$ it holds that
$$\left\vert \|A_{\pm r_i}(x,t)\|^{-1}\partial_X\|A_{\pm r_i}(x,t)\|\right\vert\leq  r_ie^{(\log \|A_{\pm r_i}(x,t)\|)^c}, \left\vert \|A_{\pm r_i}(x,t)\|^{-1}\partial^2_{XY}\|A_{\pm r_i}(x,t)\|\right\vert\leq  r_i^2e^{(\log \|A_{\pm r_i}(x,t)\|)^c}.$$

    Furthermore  one of the following three conclusions for $g_{i+1}$ holds true by induction:
    \begin{enumerate} \item
  For the case of Type $\mathbf{I}_{i}$, $g_i$ satisfies~$\lambda^{- (\log \mathcal{N}_{i})^C}$-\textbf{non-resonant}~condition~on~$D_i(t).$ And it holds that $$\ \partial_x g_{i}(c_{i,1},t)\cdot \partial_y g_{i}(c_{i,2},t)<0,~\|g_{i}-g_{i-1}\|_{C^2(D_{i-1})}\leq C\lambda^{-\frac{3}{2}r_{i-1}}.$$

   \vskip 0.4cm


                \item In the case of Type $\mathbf{II}^k_{i}$ with $0\le |k|<q^2_{N+i-1}$, there exists  $ l_k\in \R$ satisfying $\lambda^{\frac{1}{2}|k|}\leq l_k\leq \lambda^{2|k|}$ such that
                \begin{equation}\label{AKbound}I_{i,2}(t)\cap(I_{i,1}(t)+k\alpha)\neq\emptyset,\quad              \|A_{k}(x,t)\|\sim_{0,\mathcal{N}_{s(k)+1}^{-1}} l_k.\end{equation} More precisely, it holds that
                \begin{equation}\label{esszero}|[(I_{i,1}(t)+k\alpha)-I_{i,2}]\bigcup [I_{i,2}-(I_{i,1}(t)+k\alpha)]|\leq \lambda^{-\frac{1}{100}r_i}.\end{equation}
                For each $j=1,2,$ it holds that
                $1\le |\{x\in I_{i,j} \vert \partial_xg_i(x,t)=0\}|\leq 2.$
                Let $(\tilde{c}_{i,j}(t),\ \tilde{c}^*_{i,j}(t))=\{x\in I_{i,j} \vert \partial_xg_i(x,t)=0\}$  such that $|g_{i+1}(\tilde{c}_{i,j}(t),\ t)\ {\rm mod\ }\pi|\le |g_{i+1}(\tilde{c}^*_{i,j}(t),\ t)\ {\rm mod\ }\pi|$ (it is possible $\tilde{c}_{i,j}(t)=\tilde{c}^*_{i,j}(t)$), then \begin{equation}\label{partxj*}c\leq \vert \partial_t g_{i}(\tilde{c}_{i,j}(t),t)\vert\leq q^C_{N+i-1}.\end{equation}
               Furthermore, for
         $\tilde{I}_{i,j}(t)=(\tilde{c}_{i,j}(t)-l_{k}^{-1}\mathcal{N}_{i}^{-2\hat{\epsilon}^{-1}},\tilde{c}_{i,j}(t)+l_{k}^{-1}\mathcal{N}_{i}^{-2\hat{\epsilon}^{-1}}),$
there exists $\tilde{d}_i\in \R$ satisfying $\lambda^{\frac{1}{2}|k|}\leq \tilde{d}_i\leq \lambda^{2 |k|}$ such that
$$\begin{array}{ll}&\vert g_{i}(\tilde{c}_{i,j}(t),t)-g_{i}(x,t) \vert\sim_{2,\mathcal{N}_{i}^{-1}} \tilde{d}_i(x-\tilde{c}_{i,j}(t))^2.
\end{array}$$
Moreover, for $j = 1, 2$, we have
\begin{equation}\label{gnrange}\begin{array}{ll}&\max\limits_{x \in I_{i,j}} g_{i+1}(x,t) - \min\limits_{x \in I_{i,j}} g_{i+1}(x,t) \leq \pi - c\lambda^{-100|k|}, \\
&\pi - C\lambda^{-\frac{1}{100}|k|} \le\max\limits_{x \in I_{s(k),j}} g_{s(k)+1}(x,t) - \min\limits_{x \in I_{s(k),j}} g_{s(k)+1}(x,t).
\end{array}\end{equation}

Finally on $I_{i,j}-\tilde{I}_{i,j},$ we have
$$|\partial_x g_{i,j}(x,t)|,\quad |g_{i,j}(x,t)|>[\min\{\mathcal{N}_{i}^{-\hat{\epsilon}^{-1}},l_{k}^{-8}\}]^2.$$

 \end{enumerate}
\end{theorem}

 \begin{remark}
 \eqref{esszero} implies that as $i$-th step $critical~points,$ $c'_{i,1}(t)$ essentially is the $ (-k)$-iteration of $c_{i,2}(t)$ while $c'_{i,2}(t)$ essentially is the $k$-iteration of $c_{i,1}(t)$ under $x\longmapsto x+\alpha$ on $\R/\Z.$ Thus we call $c'_{n,i},i=1,2,$ {\it non-essential} critical points since they share the same dynamical behaviors with $c_{n,j}, j\not=i$.
 \end{remark}

\subsection{The Large Deviation Theorem and Avalanche Principle}\label{4.2}For our purpose, LDT in \cite{wz1} is needed. We state it without proof.
\begin{theorem}\label{Th18} Let $v$ and $\alpha$ be as in Theorem \ref{Th1}. Then there exist $\lambda_1=\lambda(v,\alpha),\ i_0=i_0(\alpha)\in\Z^+$ and $0<\sigma<1,$ such that for each $\lambda>\lambda_1$ and each $i\geq i_0,$ it holds that
$$Leb\{x\in \R/{\Z} | \frac{1}{i}\log\|A_i(x)\|\geq \frac{9}{10}\log\lambda\}<\lambda^{-i^{\sigma}}.$$\label{ldt}
\end{theorem}

\begin{remark} Note that the LDT above is weak in the sense that $\sigma<1$. However, in our proof, this weak LDT or even a much  weaker one as follows is sufficient for a sharp estimate on the LE: $$Leb\{x\in \R/{\Z} | \frac{1}{i}\log\|A_i(x)\|\geq \frac{9}{10}\log\lambda\}<\lambda^{-\left(\log i\right)^{\hat{\epsilon}^{-1}}}.$$
\end{remark}
Combining Theorem \ref{Th1} with the Avalanche Principle, see \cite{goldsteinschlag}, \cite{bourgainjitomirskaya}, \cite{young}, we can  obtain the following lemma, the proof of which can be found in \cite{wz1}.
     \begin{lemma}\label{lm27} Let $v,\alpha, \lambda ~and~ \sigma$ be as in Theorem \ref{Th18}. Then for all large $n\in \Z^+$ and all $E\in [\lambda\inf v -2,\lambda\sup v +2],$ it holds that
     $$\big\vert L_n(E)+L(E)-2L_{2n}(E)\big\vert<\lambda^{-\frac{n^{\sigma}}{2}}.$$\label{ap}
     \end{lemma}
     The following result is obtained with the help of Lemma \ref{lm27}, see \cite{wz1} and \cite{LWY}.
\begin{corollary}\label{log-conti}
Let $v,\alpha, \lambda ~and~ \sigma$ be as in Theorem \ref{Th18}. Then it holds that
$L(E)$ is continuous.
  \end{corollary}

\section{The resonance and the spectral gaps}\label{class}

In this section, we will study the relationship between the resonance and the properties of the spectral gaps. First by Theorem \ref{theorem12}, we have

 \begin{lemma}\label{IIIheI} For $t\in [\inf v-\frac{2}{\lambda},\sup v+\frac{2}{\lambda}],$ denote $$\{\hat{k}_1(t),\hat{k}_2(t),\cdots\}=\{k\in \Z \vert ~There~exists~i\in \Z_+~such~that~step~i~belongs~to~\textbf{II}_i^{k}\}$$ with $|\hat{k}_i|< |\hat{k}_{i+1}|.$ Then there must exist $j_1<j_2<j_3$ such that step $j_1$ belongs to $\textbf{II}_{j_1}^{\hat{k}_i}, $ step $j_2$ belongs to $\textbf{I}_{j_2}$ and step $j_3$ belongs to $\textbf{II}^{\hat{k}_{i+1}}_{j_3}.$
 \end{lemma}

The following lemma directly follows from Theorem \ref{theorem12}, which shows the  fact that
if we move the parameter $t$ in a small neighborhood, the image of the angle function $g_n(x,t)$ also moves a  distance of a similar order.
\begin{lemma}\label{derivative-t-gn}
For any fixed $t_0\in [\inf v-\frac{2}{\lambda},\sup v+\frac{2}{\lambda}],$  the following several properties hold true.
\begin{enumerate}
\item[a:] For any $n\in \Z_+,$ it holds that \begin{equation}\label{lm17-main}c<\left\vert\frac{\partial g_n(x,t)}{\partial t}\right\vert_{C^0(D_n(t_0))}<q_{N+n-1}^C.\end{equation}
\item[b:]\label{boflem27}     Let $\hat{k}_i(t)$ be defined in Lemma \ref{IIIheI}. If there exist $n_1,n_2$ such that step $n_1$ belongs to Type $\textbf{II}_{n_1}^{\hat{k}_i(t)}$ and step $n_2$ belongs to Type $\textbf{II}_{n_2}^{\hat{k}_{i+1}(t)}.$ Then for any $n_1<n<n_2$ such that step $n$ belongs to Type $\textbf{I},$ and we have
    \begin{equation}\label{gnjinsss}g_n(x,t)\sim_{1,|I_n|^C} a_{n,i}(x-c_{n,i})+b_{n,i}(t-t_0),\quad (x,t)\in  [D_n(t_0)]^C\end{equation} with $\hat{k}_{i+1}^{-C}\leq |a_{n,i}|\leq C,$~$sgn(a_{n,1})=-sgn(a_{n,2})$ and $c\leq |b_{n,i}|\leq q_{N+n-1}^C$.

\item[c:] It holds true
  \begin{equation}\label{k2k1}\log |\hat{k}_{i+1}|\geq |\hat{k}_i|^{\hat{\epsilon}}.\end{equation}

\end{enumerate}

\label{lmg}

\end{lemma}

\begin{corollary} \label{lmg'} For any fixed $t_0\in [\inf v-\frac{2}{\lambda},\sup v+\frac{2}{\lambda}]$ and $i\in\Z_+,$ it holds that
\begin{equation}\label{lmg'-1}\left\vert\frac{d(c_{i,1}(t)-c_{i,2}(t))}{dt}\right\vert>c {\rm\ for\ } ~t\in Q,\end{equation} where $$Q=B(t_0,\lambda^{-q_{N+n-1}})\bigcap \{t\in [\inf v-\frac{2}{\lambda},\sup v+\frac{2}{\lambda}]: g_{i}(c_{i,j}(t),t)=0,j=1,2\}.$$

Moreover if step $i$ belongs to Type $\textbf{I}_1$ for$\ any\ ~t\in B(t_0,\lambda^{-q_{N+i-1}})$, then it holds that
\begin{equation}\label{lmg'-2}\left\vert\frac{d(c_{i,1}(t)-c_{i,2}(t))}{dt}\right\vert< \lambda^{q_{N+i-2}^c}.\end{equation}

\end{corollary}

\begin{proof} By the help of \eqref{lm17-main} of Lemma \ref{derivative-t-gn} and Implicit Function Theorem, to obtain \eqref{lmg'-1} we only need to give the upper bound with some absolute constant $C,$ which directly follows from Theorem \ref{theorem12}. Similarly, since $g_i(x,t)$ is of  $type_I$, from Theorem \ref{theorem12}, we $$\left\vert\partial_x g_{i}\right\vert_{C^0(I_i)}\geq \lambda^{-r^c_{i-1}}\geq \lambda^{-q^c_{N+i-2}},$$
which yields \eqref{lmg'-2}.
\hfill\qed\end{proof}

 The following theorem give some estimates on the spectral gaps.

 \begin{theorem}\label{15} For each $\lambda>\lambda_0$, there exists a subset $\mathcal{K}(\lambda)\subset \mathbb{Z}$ such that $\mathbb{R}-\Sigma^{\lambda}=\bigcup_{k\in \mathcal{K}(\lambda)} G_k,$ where $G_k(\lambda)=(t^k_-(\lambda), t^k_+(\lambda))$  such that  the following  conclusions hold true.

\begin{itemize}
\item[]{\rm (1)} For each $k\in \mathcal{K}(\lambda)$, there exists some $l(k)\in \N$ such that  \begin{equation}\label{var-ep} \left\vert(c_{n,1}(t^{\pm}_k(\lambda))+k\alpha)- c_{n,2}(t^{\pm}_k(\lambda)){\rm \ (mod}\ 1{\rm)}\right\vert\leq \lambda^{-{r^c_{n-1}}},\quad   n\geq l(k).
 \end{equation}

    \vskip 0.2cm
    \noindent

\item[]{\rm (2)} For each $k\in \mathcal{K}(\lambda)$, it holds that\begin{equation}\label{pugapguj}\lambda^{-C|k|} \leq \left|G_{k}(\lambda)\right|\leq \lambda^{-c|k|}.\end{equation}
\vskip 0.2cm
\item[]{\rm (3)} For $k,\ k'\in \mathcal{K}(\lambda)$ with $k\not=k'$, it holds that \begin{equation}\label{dgkij}dist(G_{k},G_{k'})\geq (\max\{\left\vert |k|-|k'|\right\vert, \lambda^{\left(\min\{|k|,|k'|\}\right)^{c}}\})^{-C}. \end{equation}
\label{lm15}
\end{itemize}
\end{theorem}


\begin{proof}

First we list some classic results on \textbf{uniformly hyperbolic} systems ($\mathcal{UH}$ for short) without proof, one can see \cite{z1} for details.
\begin{proposition}[\cite{johnson}]\label{johnson} For irrational $\alpha,$ it holds that
$$\Sigma^{\lambda}=\{t \vert (\alpha,A^{\lambda(t-v)})\notin \mathcal{UH}\}.$$
\end{proposition}

\begin{proposition}[\cite{yoc}]\label{yocc} $(\alpha,A^{\lambda(t-v)})\in \mathcal{UH}$ if and only if there exists $d>0$ and $\rho>1$ such that
$$\|A_{n}(x,t)\|\geq d\rho^{|n|}$$ for all $n\in \Z$ and for all $x\in \R/\Z.$
\end{proposition}

\begin{proposition}[\cite{z1}, Lemma 11]\label{uh} Let $\{B^{(k)}\}_{k\in \Z}\subset SL(2;\R)$ be a bounded sequence,  $\beta=\inf\limits_{k\in \Z}\|B^{(k)}\|$

and $\gamma=\inf\limits_{k\in \Z}\left\vert \tan[s(B^{(k)})-u(B^{(k-1)})] \right\vert$. Assume
$$\beta\gg\frac{1}{\gamma}\gg 1>\frac{2}{\beta-\beta^{-1}}.$$ Then for each $k\in \Z$ and each $n\geq 1$, it holds that
$$\|B^{(k+n-1)}\cdots B^{(k)}\|\geq (c\beta\gamma)^n.$$
\end{proposition}

By the help of above propositions, the following holds true.
\begin{lemma}\label{lem20} For $t\in [-\frac{2}{\lambda}+\inf v,\frac{2}{\lambda}+\sup v],$ if $~\liminf\limits_{j\rightarrow +\infty}\min\limits_{x\in I_{j}}\left\vert g_{j}(x,t)\right\vert>0,$
then $t\notin \Sigma^{\lambda}.$

\end{lemma}

\begin{proof}

Note $~\liminf\limits_{j\rightarrow +\infty}\min\limits_{x\in I_{j}}\left\vert g_{j}(x,t)\right\vert>0$ implies there exists some large $N^*\in \Z_+$ such that for each $j\geq N^*$ $$\min\limits_{x\in I_{j}}\left\vert g_{j}(x,t)\right\vert> \lambda^{-r^{\frac{1}{100}}_{N^*-1}}.$$ Thus we can guarantee that there exists some $k\in \Z$ such that $for~any~j\geq N^*$,
 \begin{equation}\label{54}\left\{\begin{array}{ll}&~\min\limits_{x\in I_{j}}\left\vert g_{j}(x,t)\right\vert>\lambda^{-r^{\frac{1}{100}}_{N^*-1}}\\
& Step~j~is~of~Type~\mathbf{II}^k_{j}.\end{array}\right.\end{equation}
 If for some $t,$ step $l_j$ is of Type $\textbf{II}_{l_j}^{k(l_j)}$ with $|k(l_j)|\rightarrow +\infty$ as $j\rightarrow +\infty$, then there exists $l_j<l^*_j<l_{j+1}$ such that step $l^*_j$ belongs to Type $\textbf{I}_{l^*_j}.$ It implies $\liminf\limits_{j\rightarrow +\infty}\min\limits_{x\in I_{j}}\left\vert g_{j}(x,t)\right\vert=0$, which leads to a contradiction.

By the diophantine condition,  for each $x\in \R/\Z,$ it holds that $x+m\alpha \in I_{N^*}$ for some $m\leq |I_{N^*}|^{-C}:=M_1.$ Set $M_2=\max\limits_{x\in I_{N^*}}[r^{\pm}_{N^*}(x,t)]^2.$ For each $x\in I_{N^*}$ and  each $M\geq \max\{M_2,M_1\},$ let $ 1\leq j_p\leq M,$ $1\leq p\leq m$ be all the times such that
$$j_p-j_{p-1}>k>q^2_{N+s(k)-1},~x+j_p\alpha \in I_{N^*},$$ where $j_0=0.$ Then
$$\begin{array}{ll}&\| A_{M}(x,t_0)\|=\left\|A_{M-j_m}(x+j_m\alpha,t) \prod\limits_{p=1}^m A_{j_p-j_{p-1}}(x+j_{p-1}\alpha,t)\right\|\\&\geq \left\|A_{M-j_m}(x+j_m\alpha,t)\right\|^{-1}\prod\limits_{p=1}^m \left\|A_{j_p-j_{p-1}}(x+j_{p-1}\alpha,t)\right\|^{\frac{3}{5}}(by~Proposition~\ref{uh}~and~\eqref{54})\\&\geq \lambda^{-M^{\frac{1}{2}}}\lambda^{\frac{1}{2}j_m}\geq \lambda^{-M^{\frac{1}{2}}}\lambda^{\frac{1}{2}M}\geq \lambda^{\frac{1}{3}M}.\end{array}$$

Now for any $M\geq \max\{M_2,M_1\}$ and $x\in \R/\Z,$ let $j_1$ be the first time such that $x+j_1\alpha\in I_{N^*},$ then we have
$$\|A_{M}(x,t)\|\geq \|A_{j_1}(x,t)\|^{-1}\|A_{M-j_1}(x+j_1\alpha,t)\|\geq \lambda^{-M}\lambda^{\frac{1}{3}(M-j_1)}\geq \lambda^{\frac{1}{4}M}.$$

Similarly,  for any $M\geq \max\{M_2,M_1\}$ and $x\in \R/\Z,$ we have $\|A_{-M}(x,t)\|\geq \lambda^{\frac{1}{4}M}.$

Therefore by taking $d=\frac{1}{\lambda^{\frac{1}{4}\max\{M_2,M_1\}}}$ and $\rho=\lambda^{\frac{1}{4}}$,  for each $n\in \Z$ and for each $x\in \R/\Z$, we have
 $$\|A_{n}(x,t)\|\geq d\rho^n.$$ Then Proposition \ref{johnson} and Proposition \ref{yocc} complete the proof.
\hfill\qed\end{proof}

\begin{lemma}\label{typeiiiyzx} Let ~$G^*:=(t^*_-,t^*_+)$ be a spectral gap. Then for $t\in G^*$, we have
\begin{equation}\label{fact222'}~\liminf_{j\rightarrow +\infty}\min\limits_{x\in I_{j}(t)}\left\vert g_{j}(x,t)\right\vert>0.\end{equation}

Moreover, there exists a minimal $k=k(G^*)\in \Z$ such that for\ any $t\in G^*$ and $i\geq s(k)+2$, we have
\begin{equation}\label{fact111'} each~i-th~step~is~of~Type~\mathbf{II}^k_{i}.\end{equation}

\end{lemma}

\begin{proof}

Fix some $t\in G^*.$

\textbf{Proof of \eqref{fact222'}:}
 Assume \eqref{fact222'} does not hold true, then there exists a subsequence $j_i$ such that \begin{equation}\label{valid}\lim\limits_{i\rightarrow +\infty}\min\limits_{x\in I_{j_i}(t)}\left\vert g_{j_i}(x,t)\right\vert=0.\end{equation}

Then we claim that for each $i$, it holds that
\begin{equation}\label{snun}\min\limits_{x\in I_{j_i}(t)}\left\vert g_{j_i}(x,t)\right\vert\leq \lambda^{-\frac{4}{3}r_{j_{i-1}}}.\end{equation}

Otherwise, there exists some $j^*$ such that $\min\limits_{x\in I_{j^*}(t)}\left\vert g_{j^*}(x,t)\right\vert> \lambda^{-\frac{4}{3}r_{j^*-1}}.$ Then by Theorem \ref{theorem12}, for any $j\geq j^*,$ $$\begin{array}{ll}\min\limits_{x\in I_{j}(t)}\left\vert g_{j}(x,t)\right\vert&\geq \min\limits_{x\in I_{j^*}(t)}\left\vert g_{j^*}(x,t)\right\vert-\sum\limits_{j\geq j^*}\|g_{j}-g_{j-1}\|_{C^0(I_{j-1})}\\&\geq \lambda^{-\frac{4}{3}r_{j^*-1}}-\sum\limits_{j\geq  j^*}\lambda^{-{\frac{3}{2}}r_{j-1}}\geq \lambda^{-\frac{5}{3}r_{j^*-1}},\end{array}$$ which implies
$\lim\limits_{i\rightarrow +\infty}\min\limits_{x\in I_{j_i}(t)}\left\vert g_{j_i}(x,t)\right\vert\geq \lambda^{-\frac{5}{3}r_{j^*-1}}>0.$~This~conflicts~with~\eqref{valid}.

Note $I_{j_{i+1}}\subset I_{j_i}$ with $|I_{j_i}|\rightarrow 0$ as $j_i\rightarrow +\infty.$ Hence there exist $c_{\infty,l},l=1,2$ such that $$\{c_{\infty,1},c_{\infty,2}\}\subset I_{j_i}(=I_{j_i,1}\bigcup I_{j_i,2}).$$

Recall in Theorem \ref{theorem12} that
$g_{j_i}(x,t)=s(A_{r_{j_{i-1}}}(x,t))-s(A_{-r_{j_{i-1}}}(x,t))$ with
$$\|s(A_{r_{j_{i-1}}})-s(A_{r_{j_i}})\|_{C^0(I_{j_{i-1}})}+\|s(A_{-r_{j_{i-1}}})-s(A_{-r_{j_i}})\|_{C^0(I_{j_{i-1}})}\leq C\lambda^{-\frac{3}{2}r_{j_{i-1}}}.$$ Then from \eqref{snun},   there exists $s^*_l\in \R/\Z,~l=1,2$ such that
$$\lim\limits_{i\rightarrow +\infty}s(A_{r_{j_{i-1}}}(c_{\infty,l},t))=\lim\limits_{i\rightarrow +\infty}s(A_{-r_{j_{i-1}}}(c_{\infty,l},t))=s^*_l.$$

This implies \begin{equation}\label{cannot}(\alpha, A^{\lambda(t-v)})~\notin~\mathcal{UH}\end{equation} (note the uniform hyperbolicity of $(\alpha, A^{\lambda(t-v)})$ implies for each $x\in \R/\Z$ $\lim\limits_{p\rightarrow +\infty}s(A_{p}(x,t))$ and $\lim\limits_{p\rightarrow +\infty}s(A_{-p}(x,t))$ exists with $\lim\limits_{p\rightarrow +\infty}s(A_{p}(x,t))\neq \lim\limits_{p\rightarrow +\infty}s(A_{-p}(x,t)) )$.

\eqref{cannot} and Proposition \ref{johnson} clearly yield $t\in \Sigma^{\lambda},$ which leads to a contradiction with $t \in G^*\subset \R-\Sigma^{\lambda}.$

Hence \eqref{valid} is not true and this completes the proof of \eqref{fact222'}.

\

\textbf{Proof of \eqref{fact111'}}\quad
For each $t\in G^*$ we denote
$\liminf\limits_{j\rightarrow +\infty}\min\limits_{x\in I_{j}(t)}\left\vert g_{j}(x,t)\right\vert=c(t)>0.$

Hence for each $t\in G^*$, there exists some $\tilde{N}(t)>0$ such that \begin{equation}\label{>1}\min\limits_{x\in I_{j}(t)}\left\vert g_{j}(x,t)\right\vert>\frac{1}{2}c(t),~ j\geq \tilde{N}(t).\end{equation}

We claim that  for any $l\geq \tilde{N}=\tilde{N}(t)$,
\begin{equation}\label{each}each~l-th~step~is~of~Type~\mathbf{II}^{k(t)}_{l}.\end{equation}
Otherwise, there exists some $j^*>\tilde{N}$ such that $j^*-th~step~is~of~Type~\mathbf{I}_{j^*},$ which implies $\min\limits_{x\in I_{j^*}(t)}\left\vert g_{j^*}(x,t)\right\vert=0$ by the definition of Type $~\mathbf{I}$. This conflicts with \eqref{>1}.



Without loss of generality, we assume $\bar{N}(t)$ is the smallest one such that the above hold true.

Next we claim that there exists some $~k^*~\in~\Z$ such that
\begin{equation}\label{xiangd}~k(t')=k(t''):=k^*,~{\rm \ for\ any\ }~t',t''\in G^*;\end{equation}
\begin{equation}\label{bound}~\bar{N}(t')=s(k^*)+2,~{\rm \ for\ any\ }~t'\in G^*.\end{equation}

\textbf{Proof of \eqref{xiangd}}

Let $\tilde{N}(\cdot),$ $\bar{N}(\cdot),$ $k(\cdot)$ and $c(\cdot)$ be defined as above.
By taking $j^*\gg \max\{\tilde{N}(t'),\bar{N}(t'),\tilde{N}(t''),\bar{N}(t'')\}$, we can guarantee that for $~X=t'~or~t''$ and each $j\geq j^*$,
$$j-th~step~is~of~Type~\mathbf{II}^{k(X)}_{j},$$ and
\begin{equation}\label{mini}\left( \min\limits_{x\in I_{j}(X)}\left\vert g_{j}(x,X)\right\vert \mod \pi\right)>\frac{1}{2}c(X)>\lambda^{-r^{\frac{1}{100}}_{j^*-1}}.\end{equation}

It follows from \textbf{(a)} of Lemma \ref{lmg} that if  step $j$ belongs to Type~$\textbf{II}_j^{k(t')}$ for some $t'\in G^*$, then\begin{equation}\label{dandia} g_{j}(\tilde{c}_{j,l}(t),t)~is~monotonic~on~G^*~with~respect~to~t,~l=1,2.\end{equation}

\eqref{gnrange} of Theorem \ref{theorem12} implies $\tilde{c}_{j,p}(t),~p=1,2$ are two extreme points of $g_j(x,t)$ with respect to $x$ with \begin{equation}\label{g121212}\left\vert g_{j}(\tilde{c}_{j,1}(t),t)- g_{j}(\tilde{c}_{j,2}(t),t)\right\vert<\pi-\lambda^{-r^{\frac{1}{1000}}_{j^*-1}}.\end{equation}

Then by \eqref{dandia} and the continuity of $g_{j}(\tilde{c}_{j,p}(t),t)(p=1,2)$ on $t\in (t^*_--\lambda^{-q_{N+j-1}},t^*_++\lambda^{-q_{N+j-1}})$ with respect to $t,$ we have

$$J^*:=\{t\vert \lambda^{-r^{\frac{1}{100}}_{j^*-1}}<g_{j}(\tilde{c}_{j,p}(t),t)(p=1,2)<\pi-\lambda^{-r^{\frac{1}{100}}_{j^*-1}} \ (mod\  \pi)\}~is~an~open~interval.$$ Note for any $t\in J^*$ and $j\geq j^*$,
$$\begin{array}{ll}&\min\limits_{x\in I_l(t)}|g_{j}(x,t)|=\left\vert\min\limits_{p=1,2}g_{j}(\tilde{c}_{j,p}(t),t)\ (mod\  \pi)\right\vert\\&\geq \lambda^{-r^{\frac{1}{100}}_{j^*-1}}-\sum\limits_{l\geq j^*}\|g_{l}-g_{l+1}\|_{C^0(I_l)}\geq \lambda^{-r^{\frac{1}{100}}_{j^*-1}}-\sum\limits_{j\geq j^*}C\lambda^{-r_{j-1}}>\lambda^{-r^{\frac{1}{50}}_{j^*-1}},\end{array}$$
and there exists a uniform $k^*$ such that for any $t\in J^*$ and $j\geq j^*,$
\begin{equation}\label{argu}step~j~belongs~to~Type~\textbf{II}_{j}^{k^*}.\end{equation}

Since \eqref{mini} and \eqref{g121212} imply
$\{t',t''\}\subset J^*,$ by the argument as above, \eqref{argu} immediately leads
$$k(t')=k(t'')=k^*.$$

\

\textbf{Proof of \eqref{bound}:}

Let $k^*$ be as in \eqref{xiangd}. By the help of Theorem \ref{theorem12}, for any $t'\in G^*$, we have
$$step~s(k^*)+2~belongs~to~Type~\textbf{II}^{k^*}_{s(k^*)+2}.$$
Otherwise, if Step~$s(k^*)+2$~belongs~to~Type~$\textbf{I}_{s(k^*)},$ then for some $l>s(k^*)+2,$ Step~$l$~is~of~Type~$\textbf{II}^{k(t')}_{l},$ where $k(t')\not=k^*$ is defined in \eqref{each}. More precisely, here we obtain $|k(t')|>q^2_{N+s(k^*)}>|k^*|,$ which conflicts with \eqref{fact111'} and \eqref{xiangd}.

Since $\bar{N}(t')$ is  the smallest integer such that \eqref{each} holds true, we have $\bar{N}(t')\leq s(k^*)+2.$

On the other hand, Theorem \ref{theorem12} implies if $ Step~l~belongs~to~Type~\textbf{II}^{k^*}_{l},$ then $|k^*|\leq q^2_{N+l}.$

Recall the definition of $s(k^*):$ $q^2_{N+s(k^*)-1}< |k^*|\leq q^2_{N+s(k^*)},$ therefore $ l\geq s(k^*)+2.$ Then \eqref{each} implies
$\bar{N}(t')\geq s(k^*)+2.$ Hence
$\bar{N}(t')=s(k^*)+2,~{\rm \ for\ any\ }~t'\in G^*.$ This completes the proof of \eqref{bound}, which together with \eqref{each} leads to \eqref{fact111'}.\hfill\qed\end{proof}

Next we show
\begin{lemma}\label{EP} Let ~$G^*:=(t^*_-,t^*_+)$ be a spectral gap and $k(G^*)$ be as in Lemma \ref{typeiiiyzx}.
 Then for $t=t^*_-$ or $t^*_+$, it holds that
$$\label{4666}each~i-th~step~is~of~Type~\mathbf{II}^k_{i}~for~i\geq s(k(G^*)).$$

Moreover,
the following limits exist:
$$c_{\infty,l}(t^*_{X}):=\lim\limits_{i\rightarrow +\infty}c_{i+1,l}(t^*_{X}),~l=1,2$$ with
\begin{equation}\label{ci+1111}c_{\infty,2}(t^*_{X})-c_{\infty,1}(t^*_{X})=k(G^*)\alpha\ {\rm {\rm\ (mod\ 1)}},~X\in\{+,-\}.\end{equation}

\end{lemma}
\begin{proof}

Without loss of generality, we assume $t=t^*_+,$ thus  $(t-\delta,t)\subset \R-\Sigma^{\lambda}$ with some small $\delta>0.$


 \eqref{fact111'} of Lemma \ref{typeiiiyzx} implies

\begin{equation}\label{each11}for \ any~t'\in (t-\delta,t),~each~i-th~step~is~of~Type~\mathbf{II}^k_{i}~for~i\geq s(k)+2.\end{equation}

Now we claim that $$for~t,~each~i-th~step~is~also~of~Type~\mathbf{II}^k_{i}~for~i\geq s(k)+2.$$

In fact, suppose for $t$, there exists $j_i\rightarrow +\infty$ as $i\rightarrow +\infty$ such that $$each~j_i-th~step~is~of~Type~\mathbf{I}_{j_i}.$$ Note that \eqref{lmg'-2} of Corollary \ref{lmg'} implies for sufficiently small $\delta>\delta_i>0$, $g_{j_i}(x,t-\delta_i)~\text{are~of~the~same~type~as}~g_{j_i}(x,t).$
 In other words,
 $$~for~each~i~and~(t-\delta_i),~j_i-th~step~is~of~Type~\mathbf{I}_{j_i},$$ which conflicts with \eqref{each11}. This completes the proof of the necessity.

\

It remains to prove \eqref{ci+1111}.

Note we have proved that there exists $k(G^*)\in \Z$ such that for $t^*_{\pm},$

$$each~i-th~step~is~of~Type~\mathbf{II}^k_{i}~for~any~i\geq s(k)+2.$$

Then \eqref{esszero} of Theorem \ref{theorem12} implies for each $i\geq s(k)+2,~X\in\{+,-\},$
$$\begin{array}{ll}&\|c_{i,1}(t^*_{X})+k\alpha-c_{i,2}(t^*_{X})\|_{\R/\Z}<\|c_{i,2}(t^*_{X})-c'_{i,2}(t^*_{X})\|_{\R/\Z}+\|c_{i,1}(t^*_{X})+k\alpha-c'_{i,2}(t^*_{X})\|_{\R/\Z}
\\&\leq |I_i|+\lambda^{-\frac{1}{30}r_{i-1}}.\end{array}$$

Therefore for $X\in\{+,-\},$

\begin{equation}\label{jxxd}\|c_{i,1}(t^*_{X})+k\alpha-c_{i,2}(t^*_{X})\|_{\R/\Z}\rightarrow 0~as~i\rightarrow~+\infty.\end{equation}

On the other hand, from Theorem \ref{theorem12}, it holds that $$\sum\limits_{i\geq j}\|c_{i,1}(t^*_{X})-c_{i+1,1}(t^*_{X})\|_{\R/\Z}\leq \lambda^{-r_{j-1}}\rightarrow 0~as~j~\rightarrow~+\infty.$$
Hence the following limits exist: $c_{\infty,l}(t^*_{X}):=\lim\limits_{i\rightarrow +\infty}c_{i,l}(t^*_{X}),~l=1,2.$
Then from \eqref{jxxd}, we obtain
$$c_{\infty,2}(t^*_{X})-c_{\infty,1}(t^*_{X})=k\alpha {\rm\ (mod\ 1)},$$ which yields \eqref{ci+1111} as desired.\hfill\qed\end{proof}
Now we are in a position to prove Theorem \ref{15}.

{\bf Proof of (1):}

By  Lemma \ref{EP}, we have
$$\{\rm{\ all\ (openning)~gaps\ of }\  \lambda\!\cdot\!\! v\}=\{G_k^{\lambda}=(t^k_-(\lambda), t^k_+(\lambda))| k\in K(\lambda)\subset\mathbb{Z}\}.$$

For $X\in\{+,-\}$, \eqref{ci+1111} of Lemma \ref{EP} and Theorem \ref{theorem12} imply
\begin{equation}\label{50}c_{\infty,2}(t^k_{X})-c_{\infty,1}(t^k_{X})=k\alpha {\rm\ (mod\ 1)}\end{equation}
 and there exists $l(k)\gg 1$ such that
\begin{equation}\label{diedd}\|c_{i,l}(t^k_{X})-c_{i+1,l}(t^k_{X})\|_{\R/\Z}\leq \lambda^{-\frac{1}{100}{r_{i-1}}}\le c\lambda^{-{r^c_{i-1}}},\quad  l=1,2.\end{equation}

Therefore if there exists some $i> l(k)$ such that $$\|c_{i,1}(t^k_{X})+k\alpha-c_{i,2}(t^k_{X})\|_{\R/\Z}\geq C\lambda^{-{r^c_{i-1}}},$$ then \eqref{50} and \eqref{diedd} imply
$$\begin{array}{ll}&0=\|k\alpha-k\alpha\|=\| c_{\infty,2}(t^k_{X})-c_{\infty,1}(t^k_{X})-k\alpha\|_{\R/\Z}\\&>C\lambda^{-{r^c_{i-1}}}-\sum\limits_{l=1}^2\sum\limits_{p\geq i-1}\|c_{p,l}(t^k_{X})-c_{p+1,l}(t^k_{X})\|_{\R/\Z}\\&>\lambda^{-{r^c_{i-1}}}>0,\end{array}$$ which is impossible.

Hence for each $i\geq l(k)$ and $X\in\{+,-\},$ $\|c_{i,1}(t^k_{X})+k\alpha-c_{i,2}(t^k_{X})\|_{\R/\Z}<\lambda^{-{r^c_{i-1}}}.$

This completes the proof of \eqref{var-ep}.\hfill\qed\end{proof}

{\bf Proof of (2):}

Our target is to estimate the scale of
$G_k=(t^k_-,t^k_+),~k\in K(\lambda).$ By Theorems \ref{typeiiiyzx}  and \ref{EP}, for $t\in G_k,$ $each~i-th~step~is~of~Type~\mathbf{II}^k_{i}~for~i\geq s(k)+2.$ Therefore step $s(k)+2$ belongs to $\textbf{II}_{s(k)+2}.$ By the help of Theorem \ref{theorem12}, in this case we can precisely write
\begin{equation}\label{thetat22}g_{s(k)+2,1}(x,t)= \tan^{-1}(\|A_{k}(x,t)\|^2[\tan \left(\hat{g}_{s(k)+1,2}(x+k\alpha,t)\right)])-\frac{\pi}{2}+\hat{g}_{s(k)+1,1}(x,t),\end{equation}
where $\hat{g}_{s(k)+1,j},j=1,2$ corresponds to the \textbf{nonresonant} case at step $s(k)+1$  and $\|A_k(x,t)\|$ satisfies \begin{equation}\label{akmod}\|A_k(x,t)\|\geq c\lambda^{\frac{7k}{8}}, \|A_k(x,t)\|_{C^2}\leq \|A_k(x,t)\|^{\frac{9}{8}}.\end{equation}
 We denote the zeros of $\hat{g}_{s(k)+1,i}$ by $\hat{c}_{s(k)+1,i},i=1,2.$ Then \eqref{lmg'-1} of Corollary \ref{lmg'} implies that there exists a unique $\hat{t}_k$ such that $\hat{c}_{s(k)+1,2}(\hat{t}_k)-\hat{c}_{s(k)+1,1}(\hat{t}_k)=k\alpha{\rm\ (mod\ 1)}.$
Let $$x-\hat{c}_{s(k)+1,i}(\hat{t}_k)=h_{x,i},\quad~t-\hat{t}_k=l_t.$$ Then Lemma \ref{lmg} and Theorem \ref{theorem12} imply that for $i=1,\ 2$, \begin{equation}\label{gyii}\begin{array}{ll}\hat{g}_{s(k)+1,i}(x,t)=\bar {u}_i h_{x,i}+\bar{v}_i l_t+\frac{1}{2}(2X_i(x,t)h_{x,i}l_t+Y_i(x,t)h_{x,i}^2+Z_i(x,t)l_t^2)\end{array}\end{equation} with $|h_{x,i}|,\ |l_t|\leq \lambda^{-\frac{k}{10}}$ and $$0<\bar{v}_i,\bar{v}_i^{-1},|\bar{u}_i|,|\bar{u}_i^{-1}|,\|X_i\|_{C^0},\|Y_i\|_{C^0},\|Z_i\|_{C^0}\leq (\log k)^{C}; sgn(\bar{u}_1)=-1=- sgn(\bar{u}_2).$$

\textbf{Proof of the lower bound}

By Theorem \ref{theorem12}, $$\min\limits_{x\in I_{s(k)+2}}\left\vert\tan (g_{s(k)+2})(x,t)\right\vert\geq \lambda^{-cr_{s(k)+1}},$$ which implies
$$\liminf\limits_{j}\min\limits_{x\in I_{j}}\left\vert\tan (g_{j})(x,t)\right\vert\geq \lambda^{-2cr_{s(k)+1}}>0. $$ It then
leads $t\notin \Sigma^{\lambda}$ and for $i\geq s(k)+2$, $each~step~i~belongs~to~Type~\textbf{II}_{i}^k.$

Therefore

\begin{equation}\label{minnn1}\{t\vert \min\limits_{x\in I_{s(k)+2}}\left\vert\tan (g_{s(k)+2})(x,t)\right\vert\geq \lambda^{-cr_{s(k)+1}}\}\subset(t^k_-,t^k_+).\end{equation}


Then \eqref{minnn1} and (vi) of Lemma \ref{thetat1fenjie} yield
\begin{equation}\label{minnn}\{t\vert \min\limits_{x\in I_{s(k)+2,1}}\left\vert\tan (g_{s(k)+2,1})(x,t)\right\vert\geq \lambda^{-2cr_{s(k)+1}}\}\subset(t^k_-,t^k_+).\end{equation}

Let $$\Xi_{\hat{g}}=\left\vert\frac{\tan \left(\hat{g}_{s(k)+1,2}(x+k\alpha,t)\right)\tan \left(\hat{g}_{s(k)+1,1}(x,t)\right)-\|A_k(x,t)\|^{-2}}{\tan \left(\hat{g}_{s(k)+1,2}(x+k\alpha,t)\right)+\|A_k(x,t)\|^{-2}\tan \left(\hat{g}_{s(k)+1,1}(x,t)\right)}\right\vert.$$
Note \eqref{thetat22} implies
\begin{equation}\label{minmin}\min\limits_{x\in I_{s(k)+2,1}}\left\vert\tan (g_{s(k)+2,1})\right\vert=\min\limits_{x\in I_{s(k)+2,1}}\Xi_{\hat{g}}.\end{equation}

 Since $\left\vert \tan \left(\hat{g}_{s(k)+1,2}(x+k\alpha,t)\right)+\|A_k(x,t)\|^{-2}\tan \left(\hat{g}_{s(k)+1,1}(x,t)\right) \right\vert<C,$ we have

\begin{equation}\label{minn}\begin{array}{ll}&\Xi_{\hat{g}}>c\left\vert \tan \left(\hat{g}_{s(k)+1,2}(x+k\alpha,t)\right)\tan \left(\hat{g}_{s(k)+1,1}(x,t)\right)-\|A_k(x,t)\|^{-2}\right\vert.\end{array}\end{equation}

\eqref{minmin}, \eqref{minn} and \eqref{minnn} imply
\begin{equation}\label{subset}\{t\vert \min\limits_{x\in I_{s(k)+2,1}}\left\vert{\tan \left(\hat{g}_{s(k)+1,2}(x+k\alpha,t)\right)\tan \left(\hat{g}_{s(k)+1,1}(x,t)\right)-\|A_k(x,t)\|^{-2}}\right\vert\geq c\lambda^{-cr_{s(k)+1}}\}\end{equation}
$$\subset\{t\vert \min\limits_{x\in I_{s(k)+2}}\left\vert\tan (g_{s(k)+2})(x,t)\right\vert\geq \lambda^{-cr_{s(k)+1}}\}\subset(t^k_-,t^k_+)=G_k.$$

 $\eqref{gyii}$ implies (note $h_{x+k\alpha,2}=h_{x,1}$) for $i=1,2,$ $$\hat{g}_{s(k)+1,i}(x,t)= (\bar {u}_i+o(\lambda^{-\frac{k}{100}})) h_{x,i}+(\bar{v}_i+o(\lambda^{-\frac{k}{100}}) )l_t,$$
\begin{equation}\label{67}\tan \left(\hat{g}_{s(k)+1,2}(x+k\alpha,t)\right)\tan \left(\hat{g}_{s(k)+1,1}(x,t)\right)=h^2_{x,1}+\beta h_{x,1}l_t+\gamma l_t^2.\end{equation}
where \begin{equation}\label{abg}\eta=\bar{u}_1\bar{u}_2+o(\lambda^{-\frac{k}{1000}}),\quad \beta=\bar{v}_1\bar{u}_2+\bar{v}_2\bar{u}_1+o(\lambda^{-\frac{k}{1000}}),\quad \gamma=\bar{v}_1\bar{v}_2+o(\lambda^{-\frac{k}{1000}}).\end{equation}
Therefore

$$\min\limits_{x\in I_{s(k)+1,1}}\left\vert{\tan \left(\hat{g}_{s(k)+1,2}(x+k\alpha,t)\right)\tan \left(\hat{g}_{s(k)+1,1}(x,t)\right)-\|A_k(x,t)\|^{-2}}\right\vert\geq \lambda^{-cr_{s(k)+1}}.$$ It then implies
\begin{equation}\label{subs}\min\limits_{x\in I_{s(k)+1,1}}\left\vert \eta h^2_{x,1}+\beta h_{x,1}l_t+\gamma l_t^2-\|A_k(x,t)\|^{-2}\right\vert\geq \lambda^{-cr_{s(k)+1}}.\end{equation}
 We denote the set of  such $t$  satisfying
\eqref{subs} by $\mathcal{H}_t.$

Since $I_{s(k)+2,1}\subset I_{s(k)+1,1},$ \eqref{subset} yields
\begin{equation}\label{htgt}\mathcal{H}_t\subset G_k.\end{equation}

Write $$\begin{array}{ll}&\left\vert\eta h^2_{x,1}+\beta h_{x,1}l_t+\gamma l_t^2-\|A_k(x,t)\|^{-2}\right\vert
=\left\vert \eta\left( h_{x,i}+\frac{\beta l_t}{2\eta}\right)^2+\gamma l_t^2-\frac{\beta^2l_t^2}{4\eta}-\|A_k(x,t)\|^{-2}\right\vert.\end{array}$$

Note Lemma \ref{lmg} leads that $\eta<0.$ Set
$$\mathcal{Q}_t:=\{t\vert -\gamma l_t^2+\frac{\beta^2l_t^2}{4\eta}+\min\limits_{x\in I_{s(k)+1,1}}\|A_k(x,t)\|^{-2}>c\lambda^{-cr_{s(k)+1}}\}.$$

Thus for $t\in \mathcal{Q}_t,$ it holds that $\gamma l_t^2-\frac{\beta^2l_t^2}{4\eta}-\|A_k(x,t)\|^{-2}<-c\lambda^{-cr_{s(k)}}<0.$ Hence for $t\in \mathcal{Q}_t,$ $$\begin{array}{ll}&\min\limits_{x\in I_{s(k)+1,1}}\left\vert\eta h^2_{x,1}+\beta h_{x,1}l_t+\gamma l_t^2-\|A_k(x,t)\|^{-2}\right\vert\geq \min\limits_{x\in I_{s(k)+1,1}}\vert\gamma l_t^2-\frac{\beta^2l_t^2}{4\eta}-\|A_k(x,t)\|^{-2}\vert\\&\geq \min\limits_{x\in I_{s(k)+1,1}}\left(-\gamma l_t^2+\frac{\beta^2l_t^2}{4\eta}+\|A_k(x,t)\|^{-2}\right)=-\gamma l_t^2+\frac{\beta^2l_t^2}{4\eta}+\min\limits_{x\in I_{s(k)+1,1}}\|A_k(x,t)\|^{-2},\end{array}$$ which implies
$$\min\limits_{x\in I_{s(k)+1,1}}\left\vert\eta h^2_{x,1}+\beta h_{x,1}l_t+\gamma l_t^2-\|A_k(x,t)\|^{-2}\right\vert\geq \lambda^{-cr_{s(k)+1}}.$$ Then we obtain \eqref{subs}.
Hence
\begin{equation}\label{qt}\mathcal{Q}_t\subset \mathcal{H}_t.\end{equation} Note $$(\log k)^{-C}<\gamma-\frac{\beta^2}{4\eta}\leq (\log k)^{C},\quad~r_{s(k)}\gg k;~\|A_k\|\leq \lambda^{Ck}.$$
Then a direct calculation yields

\begin{equation}\label{qtda}\{l_t^2\leq \frac{\frac{1}{2}\lambda^{-2Ck}}{(\log k)^C}\}\subset\{l_t^2\leq \frac{\lambda^{-2Ck}-c\lambda^{-cr_{s(k)+1}}}{\gamma-\frac{\beta^2}{4\eta}}\}\subset \{l_t^2\leq \frac{\min\limits_{x\in I_{s(k)+1,1}}\|A_k(x,t)\|^{-2}-c\lambda^{-cr_{s(k)+1}}}{\gamma-\frac{\beta^2}{4\eta}}\}\subset \mathcal{Q}_t.\end{equation}

Since $(t-\hat{t}_k)=l_t,$ we obtain \begin{equation}\label{qtx}|\{t\vert l_t^2\leq \frac{\frac{1}{2}\lambda^{-2Ck}}{(\log k)^C}\}|\geq \lambda^{-\frac{3}{2}Ck}.\end{equation}

Combining \eqref{qtda} and \eqref{qtx}, we obtain
$|\mathcal{Q}_t|\geq \lambda^{-\frac{3}{2}Ck}.$

Finally \eqref{htgt}, \eqref{qt} and the above inequality yield the lower bound of $|G_k|.$

\

\textbf{Proof of the upper bound}

Recall Theorem \ref{theorem12} shows that for Type $\textbf{II}_{s(k)+2}^k,$ if $$\{x\in I_{s(k)+1,1} \vert \tan (g_{s(k)+2,1})=0\}\neq \emptyset {\quad\rm(}~\{x\in I_{s(k)+1,2} \vert \tan (g_{s(k)+2,2})=0\}\neq \emptyset{\rm)},$$ then it has two elements $\{c_{s(k)+2,1},c'_{s(k)+2,2}\}$ ($\{c_{s(k)+2,2},c'_{s(k)+2,1}\}$) with
\begin{equation}\label{rsk}\|c_{s(k)+2,1}+k\alpha-c'_{s(k)+2,1}\|_{\R/\Z}+\|c_{s(k)+2,2}-k\alpha-c'_{s(k)+2,2}\|_{\R/\Z}\leq \lambda^{-\frac{1}{100}r_{s(k)+1}}.\end{equation}

Let $\eta,\beta,\gamma$ as in \eqref{abg}, by \eqref{minmin} and \eqref{67}, for $t\in (\hat{t}_k-\lambda^{-q^c_{N+s(k)}},\hat{t}_k+\lambda^{-q^c_{N+s(k)}})$, we have $$\begin{array}{ll}&\{c_{s(k)+2,1}(t),c'_{s(k)+2,2}(t)\}=\{x\in I_{s(k)+1}\vert \tan (g_{s(k)+2,1})=0\}\\&=\{x\vert \left\vert{\tan \left(\hat{g}_{s(k)+1,2}(x+k\alpha,t)\right)\tan \left(\hat{g}_{s(k)+1,1}(x,t)\right)-\|A_k(x,t)\|^{-2}}\right\vert=0\}\\&=\{x\vert \eta h^2_{x,1}+\beta h_{x,1}l_t+\gamma l_t^2-\|A_k(x,t)\|^{-2}=0\}\\&=\{x\vert \eta\left( h_{x,i}+\frac{\beta l_t}{2\eta}\right)^2+\gamma l_t^2-\frac{\beta^2l_t^2}{4\eta}-\|A_k(x,t)\|^{-2}=0\}.\end{array}$$

Now we take $$t\in \{t \vert \gamma l_t^2-\frac{\beta^2l_t^2}{4\eta}= \lambda^{-\frac{1}{1000}k}\}(=\{\hat{t}_k-\frac{\lambda^{-\frac{1}{2000}k}}{\sqrt{\gamma-\frac{\beta^2}{4\eta}}},~\hat{t}_k+\frac{\lambda^{-\frac{1}{2000}k}}{\sqrt{\gamma-\frac{\beta^2}{4\eta}}}\}).$$ For all $x\in I_{s(k)+1},~t\in (\hat{t}_k-\lambda^{-q^c_{N+s(k)}},\hat{t}_k+\lambda^{-q^c_{N+s(k)}})$, by \eqref{AKbound} we have $\|A_k(x,t)\|\geq \lambda^{\frac{1}{2}k}.$

Thus the following inequality holds true: \begin{equation}\label{lttt}\gamma l_t^2-\frac{\beta^2l_t^2}{4\eta}-\|A_k(x,t)\|^{-2}>\lambda^{-\frac{1}{1000}k}-\lambda^{-k}>\lambda^{-\frac{1}{900}k}.\end{equation}

Note \eqref{akmod} implies \begin{equation}\label{shangjjj}\|A_k(x,t)\|^{-2}\leq \lambda^{-\frac{7}{4}k}\ll\lambda^{-\frac{1}{900}k} , \left\vert\partial_x(\|A_k(x,t)\|^{-2})\right\vert\leq \frac{\|A_k\|^{\frac{9}{8}}}{\|A_k\|^3}\leq \lambda^{-\frac{5}{3}k}\ll (\log k)^{-C}\leq |\partial_x h_{x,i}|.\end{equation} Then \eqref{lttt} and \eqref{shangjjj}  imply
$$\{c_{s(k)+2,1}(t),c'_{s(k)+2,2}(t)\}=\{x\vert \eta\left( h_{x,i}+\frac{\beta l_t}{2\eta}\right)^2+\gamma l_t^2-\frac{\beta^2l_t^2}{4\eta}-\|A_k(x,t)\|^{-2}=0\}\neq \emptyset.$$
With the help of the fact that $\left\vert c_{s(k)+2,1}(t)-c'_{s(k)+2,2}(t)\right\vert< \frac12$, \eqref{lttt} implies
$$\begin{array}{ll}&\|c_{s(k)+2,1}(t)-c'_{s(k)+2,2}(t)\|_{\R/\Z}=\left\vert c_{s(k)+2,1}(t)-c'_{s(k)+2,2}(t)\right\vert\ \geq 2\lambda^{-\frac{1}{900}k}.\end{array}$$

Since $r(s(k))\gg k,$ the above inequality and \eqref{rsk} imply for $t\in \{\hat{t}_k-\frac{\lambda^{-\frac{1}{2000}k}}{\sqrt{\gamma-\frac{\beta^2}{4\eta}}},~\hat{t}_k+\frac{\lambda^{-\frac{1}{2000}k}}{\sqrt{\gamma-\frac{\beta^2}{4\eta}}}\},$
\begin{equation}\label{cskk}\|c_{s(k)+2,1}(t)+k\alpha-c_{s(k)+2,2}(t)\|_{\R/\Z}\geq \lambda^{-\frac{1}{800}k}\gg C\lambda^{-c{r^{c}_{s(k)+1}}}.\end{equation}

Note (1) of Theorem \ref{15}, which has been proved, shows that for two endpoints of $G_k,$ the above inequality is invalid. Thus $t \in (t^k_-,t^k_+)~or~t \in \R-[t^k_-,t^k_+].$ Now we claim that \begin{equation}\label{puwai}t \in \R-[t^k_-,t^k_+].\end{equation} Actually, by the help of \eqref{fact111'} of Lemma \ref{typeiiiyzx}, if $t\in (t^k_-,t^k_+)$, then for each $j>s(k)+2$, \begin{equation}\label{jstep}step~j~belongs~to~Type~\textbf{II}_{j}^{k}.\end{equation}
 On the other hand, for $j^*\gg s(k)$, \eqref{cskk} implies $$\begin{array}{ll}&\|c_{j^*,1}(t)+k\alpha-c_{j^*,2}(t)\|_{\R/\Z}\\&\geq \lambda^{-{r^{c}_{s(k)+1}}}-\sum\limits_{l\geq s(k)+2}^{j^*-1}\|c_{l,1}(t)-c_{l+1,1}(t)\|_{\R/\Z}-\sum\limits_{l\geq s(k)+2}^{j^*-1}\|c_{l,2}(t)-c_{l+1,2}(t)\|_{\R/\Z}\\&>\lambda^{-{r^{c}_{s(k)+1}}}-C\lambda^{-{r_{s(k)+1}}}>\lambda^{-{r^{c}_{s(k)+1}}}\gg |I_{j^*}|,\end{array}$$ which leads that $I_{j^*,1}+k\alpha \bigcap I_{j^*,2}=\emptyset.$ Hence either there exists some $k^*\neq k$ with $|k^*|\leq q^2_{N+s(k^*)}$ such that $I_{j^*,1}+k^*\alpha \bigcap I_{j^*,2}\neq \emptyset,$ which implies $$step~j^*~belongs~to~Type~\textbf{II}^{k^*}_{j^*};$$ or for $|p|\leq q^2_{N+s(k^*)}$, we have $I_{j^*,1}+p\alpha \bigcap I_{j^*,2}= \emptyset,$ which means $$step~j^*~belongs~to~Type~\textbf{I}_{j^*}.$$
Anyway, it leads to a contradiction with \eqref{jstep}.
Therefore \eqref{puwai} holds true.

By \eqref{puwai}, $$\{\hat{t}_k-\frac{\lambda^{-\frac{1}{2000}k}}{\sqrt{\gamma-\frac{\beta^2}{4\eta}}},~\hat{t}_k+\frac{\lambda^{-\frac{1}{2000}k}}{\sqrt{\gamma-\frac{\beta^2}{4\eta}}}\}\subset \R-G_k.$$

Since $$c(\log k)^{-C}<\gamma-\frac{\beta^2}{4\eta}\leq C(\log k)^{C},$$ we obtain
\begin{equation}\label{maozi}\{\hat{t}_k-\lambda^{-\frac{1}{1000}k},~\hat{t}_k+\lambda^{-\frac{1}{1000}k}\}\subset \R-G_k.\end{equation}
Note that $\hat{t}\in \mathcal{H}_t$ (recall $\mathcal{H}_t$ is denoted by such $t$ that satisfies
$$\min\limits_{x\in I_{s(k)+1,1}}\left\vert \eta h^2_{x,1}+\beta h_{x,1}l_t+\gamma l_t^2-\|A_k(x,t)\|^{-2}\right\vert\geq \lambda^{-cr_{s(k)+1}}).$$ Then \eqref{qt} implies $\hat{t}\in G_{k}.$

Combining this with \eqref{maozi}, we obtain

$$|G_k|\leq \left(\hat{t}_k+\lambda^{-\frac{1}{1000}k}\right)-\left(\hat{t}_k-\lambda^{-\frac{1}{1000}k}\right)\leq 2\lambda^{-\frac{1}{1000}k}.$$ This completes the proof of the upper bound for $|G_k|.$

\

{\bf Proof of (3):}

We have already proved that all gaps can be determined by some set $\mathcal{K}(\lambda)\subset \Z$ with
$$\R-\Sigma^{\lambda}=\bigcup \limits_{k\in \mathcal{K}(\lambda)}(t^k_-,t^k_+).$$

\

Recall the definition $$I_{j,l}:=(c_{j,l},\lambda^{-\hat{\epsilon}q^{\hat{\epsilon}}_{N+j-1}}),~j\in \Z_+,~l=1,2.$$
And diophantine condition guarantees that there exists some $C_{\alpha}>0$ such that \begin{equation}\label{qc}q_{N+j}\leq q^{C_{\alpha}}_{N+j-2},\quad~\|P\alpha\|_{\R/\Z}\geq P^{-C_{\alpha}} .\end{equation}

\

{\rm We first prove the following lemma.}

\begin{lemma}\label{lem1}~ Given $k\in \mathcal{K}(\lambda)$, then for any $t\in (t^k_--\lambda^{-|k|^{100c}},t^k_-)$ {\rm(resp. $(t^k_+,t^k_++\lambda^{-|k|^{\frac{1}{100}c}})$)},
$$step~s(k)+2~belongs~to~Type~\textbf{II}_{s(k)+2}^k.$$

Moreover, for $\lambda^{-q_{N+s(k)}}<\eta<\lambda^{-|k|^{100c}},$ it holds that \begin{equation}\label{bigvert}\min\{|s|\big\vert G_{s}\bigcap (t^{k}_--\eta,t^k_-)\neq \emptyset,~s\in \mathcal{K}(\lambda)\}\geq \eta^{-C}\end{equation}
{\rm(resp.} $\min\{|s|\big\vert G_{s}\bigcap (t^k_+,t^{k}_++\eta),~s\in \mathcal{K}(\lambda)\}\geq \eta^{-C}${\rm).}
\end{lemma}

\begin{proof} Without loss of generality, we assume $k\geq 0.$ For $$A_{r_{s(k)+1}-k}(x+k\alpha,t)A_{k}(x,t)A_{r_{s(k)+1}}(x-r_{s(k)+1}\alpha,t),$$ let $$\tilde{g}_{s(k)+1,1}(x,t)=\frac{\pi}{2}-s(A_k(x))+s(A_{-r_{s(k)+1}}),~x\in~I_{s(k)+1,1}(t)$$ and $$\tilde{g}_{s(k)+1,2}(x,t)=\frac{\pi}{2}+s(A_{-k}(x))-s(A_{r_{s(k)+1}-k}(x+k\alpha)),~x\in~I_{s(k)+1,2}(t).$$ Here $\tilde{g}_{s(k)+1,j}(x,t)$ corresponds to the \textbf{nonresonant} case at step $s(k)+1$ and has a unique zero $\bar{c}_{s(k)+1,j}(t)\in I_{s(k)+1,j}(t)$ for $j=1,2.$ Since for $t_-^k$, $step~s(k)+2~belongs~to~Type~\textbf{II}^k_{s(k)+2},$ we have
\begin{equation}\label{98}\|\bar{c}_{s(k)+1,1}(t^k_-)+k\alpha-\bar{c}_{s(k)+1,2}(t^k_-)\|_{\R/\Z}\leq |I_{s(k)+1}|.
\end{equation}

By \eqref{lmg'-2} of Lemma \ref{lmg'}, for $j=1,2$ and for any $t\in (t^k_--\lambda^{-|k|^{100c}},t^k_-),$ $$\begin{array}{ll}&\left|\bar{c}_{s(k)+1,j}(t)-\bar{c}_{s(k)+1,j}(t^k_-)\right|\leq \lambda^{cq^c_{N+s(k)-1}}\cdot \lambda^{-|k|^{100c}}\leq \lambda^{ck^c}\cdot \lambda^{-|k|^{{{100}c}}}\\
\\&\leq \lambda^{-|k|^{50c}}\leq \lambda^{-q^{50c}_{N+s(k)-1}} \ll \lambda^{-q^c_{N+s(k)-1}} \leq |I_{s(k)+1}|.\end{array}$$
Then by \eqref{98} and the above inequality,
$$\begin{array}{ll}&\|\bar{c}_{s(k)+1,1}(t)+k\alpha-\bar{c}_{s(k)+1,2}(t)\|_{\R/\Z}\\& \leq \|\bar{c}_{s(k)+1,1}(t^k_-)+k\alpha-\bar{c}_{s(k)+1,2}(t^k_-)\|_{\R/\Z}+\sum\limits_{j=1}^2\left|\bar{c}_{s(k)+1,j}(t)-\bar{c}_{s(k)+1,j}(t^k_-)\right|\\& \leq |I_{s(k)+1}|+|I_{s(k)+1}|\leq 2|I_{s(k)+1}|.\end{array}$$
Hence  we have $2I_{s(k),1}(t)+k\alpha\bigcap 2I_{s(k),2}(t)\neq \emptyset,$ which implies for any $t\in (t^k_--\lambda^{-c|k|^{c}},t^k_-),$
$$step~s(k)+2~belongs~to~Type~\textbf{II}_{s(k)+2}^k.$$ This completes the proof of the first part.

\

For the second part, we first claim $$\min\limits_{x\in I_{s(k)+2}}|g_{{s(k)+2}}(x,t^k_-)|\leq \lambda^{-r^c_{{s(k)+1}}}\leq \lambda^{-\lambda^{q^c_{N+s(k)+1}}}\ll \lambda^{-q_{N+s(k)}}.$$

Otherwise, $\min\limits_{x\in I_{s(k)+2}}|g_{{s(k)+2}}(x,t^k_-)|> \lambda^{-r^c_{{s(k)+1}}}$. Then by Theorem \ref{theorem12}, for any $j\geq s(k)+2,$ $$\begin{array}{ll}
\min\limits_{x\in I_{j}}\left\vert g_{j}(x,t^k_-)\right\vert&\geq \min\limits_{x\in I_{s(k)+2}}\left\vert g_{s(k)+2}(x,t^k_-)\right\vert-\sum\limits_{j>s(k)+2}\|g_{j}-g_{j-1}\|_{C^0(I_{j-1})}\\&\geq \lambda^{-r^c_{{s(k)+1}}}-\sum\limits_{j>  s(k)+2}\lambda^{-{\frac{3}{2}}r_{j-1}}\geq \lambda^{-r_{s(k)+1}},\end{array}.$$ Then Lemma \ref{lem20} implies $t^k_-\notin \Sigma^{\lambda}.$ Thus we obtain the claim.

Let $k^*$ satisfy $$|k^*|=\min\{|s|\big\vert G_{s}\bigcap (t^{k}_--\lambda^{-|k|^{100\hat{\epsilon}}},t^k_--\lambda^{-q_{N+s(k)}}),~s\in \mathcal{K}(\lambda)\}.$$
Then $t^{k^*}_+\in (t^{k}_--\lambda^{-|k|^{100\hat{\epsilon}}},t^k_--\lambda^{-q_{N+s(k)}}).$

By Lemma \ref{lmg}, for $t\in(t^{k}_--\lambda^{-c|k|^{c}},t^k_--\lambda^{-q_{N+s(k)}}),$ the local minimum of $g_{s(k)+2}(x,t),$ denoted by $m^*(t),$ satisfies
\begin{equation}\label{zhybp}\begin{array}{ll}&m^*\geq -\lambda^{-r^c_{s(k)+1}}- q^C_{N+s(k)+1}|t^k_--t|\geq -\lambda^{-r^c_{s(k)+1}}-|t^k_--t|^{\frac{1}{2}}\geq -|t^k_--t|^{\frac{1}{3}}.\end{array}
\end{equation}
On the other hand, Theorem \ref{theorem12} implies for any $x\in I_{s(k)+1}$, \begin{equation}\label{fth}\left\vert\frac{\partial g_{j}(x,t^k_-)}{\partial x}\right\vert+\left\vert\frac{\partial^2 g_{j}(x,t^k_-)}{\partial x^2}\right\vert>|k|^{-C}.\end{equation}

And $g_{j,l}$ has at most two zeros on $I_{j,l},$ which is denoted by $c_{j,1}$ and $c'_{j,2}$ (resp. $c_{j,2}$ and $c'_{j,1}$). Combining \eqref{zhybp} and \eqref{fth}, one can see that $\left\vert g_j(x,t^k_-)-g_{j}(\tilde{c}_{j,1},t^k_-)\right\vert\geq |k|^{-C}\left\vert x-\tilde{c}_{j,1}\right\vert^2$ (~resp.~$\left\vert g_j(x,t^k_-)-g_{j}(\tilde{c}_{j,2},t^k_-)\right\vert\geq |k|^{-C}\left\vert x-\tilde{c}_{j,2}\right\vert^2$) with $\partial_x g_j(\tilde{c}_{j,l},t^k_-)=0,~l=1,2.$
Then  it holds that for $j\geq s(k)+2$ and $t\in(t^{k}_--\lambda^{-c|k|^{c}},t^k_--\lambda^{-q_{N+s(k)}}),$
\begin{equation}\label{bz}\|c'_{j,1}(t)-c_{j,2}(t)\|_{\R/\Z}+\|c'_{j,2}(t)-c_{j,1}(t)\|_{\R/\Z}\leq 2\frac{|t^k_--t|^{\frac{1}{3}}}{k^{-C}}\leq |t^k_--t|^{\frac{1}{10}}.\end{equation}

It follows from \eqref{esszero} of  Theorem \ref{theorem12} that
\begin{equation}\label{lmg13}\|c_{s(k)+1,1}(t^{k}_-)+k\alpha-c_{s(k)+1,2}(t^{k}_-)\|_{\R/\Z}\leq C\lambda^{-r_{s(k)+1}}.\end{equation}
By \eqref{lmg'-2}, for $j=1,2,$ \begin{equation}\label{lmg12} \|c_{s(k)+1,j}(t^{k}_-)-c_{s(k)+1,j}(t^{k^*}_+)\|_{\R/\Z}\leq C\cdot\lambda^{q^c_{N+s(k)-1}}|t^k_--t^{k^*}_+|\leq C\lambda^{k^c}|t^k_--t^{k^*}_+|\leq |t^k_--t^{k^*}_+|^{\frac{1}{2}}.
\end{equation}
Since $t^{k^*}_+\in (t^{k}_--\lambda^{-|k|^{100c}},t^k_--\lambda^{-q_{N+s(k)}}),$ it holds from \eqref{lmg12} and  \eqref{lmg13} that
\begin{equation}\label{yjslh}\begin{array}{ll}\|c_{s(k)+1,1}(t^{k^*}_+)+k\alpha-c_{s(k)+1,2}(t^{k^*}_+)\|_{\R/\Z}&\leq C\lambda^{-r_{s(k)+1}}+|t^k_--t^{k^*}_+|^{\frac{1}{2}}.\end{array}\end{equation}

On the other hand, since for $t^{k^*}_+,$ $step~s(k^*)+2~belongs~to~Type~\textbf{II}^{k^*}_{s(k^*)+2},$
it follows from Lemma \ref{EP} that
$|k^*|>k.$

It follows from Theorem \ref{theorem12},
\begin{equation}\label{tjyhe}\|c_{s(k^*)+1,1}(t^{k^*}_+)+k^*\alpha-c_{s(k^*)+1,2}(t^{k^*}_+)\|_{\R/\Z}\leq C\lambda^{-r_{s(k^*)+1}}\end{equation} and for $j=1,2,$ \begin{equation}\label{tjehe}\|c_{s(k)+1,j}(t^{k^*}_+)-c_{s(k^*)+1,j}(t^{k^*}_+)\|_{\R/\Z}\leq \sum\limits_{i\geq s(k)+1}C\lambda^{-\frac{1}{2}r_{i}}\leq 2C\lambda^{-\frac{1}{2}r_{s(k)+1}}.\end{equation} Combining \eqref{tjyhe} with \eqref{tjehe}, it is clear that \begin{equation}\label{kjsl}\|c_{s(k)+1,1}(t^{k^*}_+)+k^*\alpha-c_{s(k)+1,2}(t^{k^*}_+)\|_{\R/\Z}\leq 3C\lambda^{-\frac{1}{2}r_{s(k)+1}}.\end{equation} Then, it follows from \eqref{yjslh} and \eqref{kjsl} that
\begin{equation}\label{101}\|k\alpha-k^*\alpha\|_{\R/\Z}\leq 4C\lambda^{-\frac{1}{2}r_{s(k)+1}}+|t^k_--t^{k^*}_+|^{\frac{1}{2}}.\end{equation} On the other hand, \eqref{qc} implies
\begin{equation}\label{102}\|k\alpha-k^*\alpha\|_{\R/\Z}\geq |k-k^*|^{-C_{\alpha}}.\end{equation}

\noindent Therefore \eqref{101} and \eqref{102} yield that
$|k-k^*|\geq \left(4C\lambda^{-\frac{1}{2}r_{s(k)+1}}+|t^k_--t^{k^*}_+|^{\frac{1}{2}}\right)^{-c},$ which implies $$|k^*|\geq \frac{1}{2}\left(4C\lambda^{-\frac{1}{2}r_{s(k)+1}}+|t^k_--t^{k^*}_+|^{\frac{1}{2}}\right)^{-c}.$$

Combining the above inequality and the fact $\lambda^{-\frac{1}{2}r_{s(k)+1}}\ll \lambda^{-q_{N+s(k)}}$, we obtain \eqref{bigvert}.\hfill\qed\end{proof}
\begin{lemma}\label{lem255}\label{lemmapart1} For $m,k\in \mathcal{K}(\lambda)$ with $k>m,$ if $dist\{G_{m},G_k\}\leq \lambda^{-q_{N+s(k)}},$ then it holds that $dist\{G_m,G_k\}\geq |k|^{-C}.$
\end{lemma}

\begin{proof}

Without loss of generality, we only consider the case $G_{m}\bigcap (t^{k}_--\eta,t^k_-)\neq \emptyset$ and assume $|k|$ is large enough.

Recall that Lemma \ref{EP} implies for $t\in \{t^k_-,t^k_+\}$ and each $j\geq s(k)+2$,  $step~j~belongs~to~Type~II_{j}^{k}.$

 It follows from Theorem \ref{theorem12} that for $j\geq s(k)+2$ and $t\in (t^k_--\lambda^{-q_{N+j-2}},t^k_--\lambda^{-q_{N+j-1}}),$ \begin{equation}\label{fbz}\begin{array}{ll}\|c'_{j,1}(t)-c_{j,1}(t)-k\alpha\|_{\R/\Z}+\|c'_{j,2}(t)-c_{j,2}(t)+k\alpha\|_{\R/\Z}&\leq C\lambda^{-\frac{1}{200}r_{j-1}}\ll \lambda^{-\frac{1}{10}q_{N+j}}\leq |t^k_--t|^{\frac{1}{10}}.\end{array}\end{equation}

Then, by \eqref{bz} and \eqref{fbz} we have
\begin{equation}\label{fbz1}\|c_{j,2}(t)-c_{j,1}(t)-k\alpha\|_{\R/\Z}\leq 2|t^k_--t|^{\frac{1}{10}}.\end{equation}

Now we consider the following set $$X_j:=\{t^m_+ \vert (t^m_-,t^m_+)\bigcap (t^{k}_--\lambda^{-q_{N+j-2}},t^k_--\lambda^{-q_{N+j-1}})\neq \emptyset,~m\in \mathcal{K}(\lambda)\}.$$

For $t^m_+\in X_j,$ \eqref{fbz1} yields
\begin{equation}\label{ts+}\|c_{j,2}(t^m_+)-c_{j,1}(t^m_+)-k\alpha\|_{\R/\Z}\leq 2|t^k_--t_+^m|^{\frac{1}{10}}.\end{equation}

By Lemma \ref{EP}, $\lim\limits_{j\rightarrow +\infty}c_{j,l}(t^m_+),~l=1,2$ exists and $\lim\limits_{j\rightarrow +\infty}c_{j,2}(t^m_+)-\lim\limits_{j\rightarrow +\infty}c_{j,1}(t^m_+)=m\alpha {\rm\ (mod\ 1)}.$

Then by \eqref{ts+}, we have
$$\begin{array}{ll}&\|m\alpha-k\alpha\|_{\R/\Z}=\|c_{\infty,2}(t^m_+)-c_{\infty,1}(t^m_+)-k\alpha\|_{\R/\Z}\leq 2|t^k_--t^m_+|^{\frac{1}{10}}+\sum\limits_{p=1}^2\sum\limits_{l\geq j}|c_{l,p}-c_{l+1,p}|\\&\leq 2|t^k_--t^m_+|^{\frac{1}{10}}+C\lambda^{-r_{j-1}}\leq 3|t^k_--t^m_+|^{\frac{1}{10}}.\end{array}$$

By \eqref{qc}, $\|m\alpha-k\alpha\|_{\R/\Z}\geq |m-k|^{-C}.$
Therefore
\begin{equation}\label{ksks}|m-k|>\left(3|t^k_--t^m_+|^{\frac{1}{10}}\right)^{-c}.\end{equation}

On the other hand, again by the help of Lemma \ref{EP},
\begin{equation}\label{131}for~each~i\geq s(m)+2,~step~i~with~respect~to~t^m_+~belongs~to~Type~\textbf{II}^{m}_{i}.\end{equation}
By Lemma \ref{lem1},
\begin{equation}\label{132}step~s(k)+2~with~respect~to~t^m_+~belongs~to~Type~\textbf{II}_{s(k)+2}^k.\end{equation}

Since $m\neq k,$ \eqref{131} and \eqref{132} imply $|m|>|k|.$
Therefore \eqref{ksks} yields
$2|m|>\left(3|t^k_--t^m_+|^{\frac{1}{10}}\right)^{-c}.$
Then for $j\geq s(k)+2$ and $t^m_+\in X_j$, it holds that

$$|m|>|t^k_--t^m_+|^{-c}=dist\{G_k,G_m\}^{-c}.\hfill\qed$$
\end{proof}

By the help of \eqref{bigvert} of Lemma \ref{lem1} and Lemma \ref{lem255}, it holds that
for $k,m\in \mathcal{K}(\lambda)$, if $dist\{G_m,G_k\}\leq \lambda^{-|k|^{100c}},$ then \begin{equation}\label{zuihoujielun1}dist\{G_k,G_m\}\geq |m|^{-c}.\end{equation}

\

\textbf{Final proof of (3) of Theorem \ref{15}:}

\begin{proof}

Given $k_1,k_2\in \mathcal{K}(\lambda)$ with $|k_1|\leq |k_2|,$

\begin{enumerate}
\item[1:] if $$\big\vert |k_2|-|k_1|\big\vert<|k_2|<\lambda^{|k_1|^{100c}},$$ then we claim that
$$dist\{G_{k_1},G_{k_2}\}\geq \lambda^{-|k_1|^{100c}}.$$ In fact
\eqref{zuihoujielun1} together with $$dist\{G_{k_1},G_{k_2}\}< \lambda^{-|k_1|^{100c}}$$ implies $$|k_2|\geq dist\{G_{k_1},G_{k_2}\}^{-c}\geq \lambda^{|k_1|^{100c}}>|k_2|.$$ This leads to a contradiction.

\item[2:] if $$|k_2|\geq \lambda^{|k_1|^{100c}}+|k_1|,$$ then $\big\vert |k_1|-|k_2| \big\vert\approx |k_2|.$

If $$dist\{G_{k_1},G_{k_2}\}< \lambda^{-|k_1|^{100c}},$$ Lemma \ref{lem255} implies $dist\{G_{k_1},G_{k_2}\}\geq |k_2|^{-c}\geq \big\vert |k_1|-|k_2| \big\vert^{-c}.$
If $$dist\{G_{k_1},G_{k_2}\}\geq \lambda^{-|k_1|^{100c}},$$  we have $$dist\{G_{k_1},G_{k_2}\}\geq \lambda^{-|k_1|^{100c}}\geq |k_2|^{-1}\geq |k_2|^{-C}\geq \frac{1}{2}\big\vert|k_1|-|k_2| \big\vert^{-C}.$$
\end{enumerate}

In summary, if $|k_2|-|k_1|<\lambda^{|k_1|^{100c}}$, we have
$dist\{G_{k_1},G_{k_2}\}\geq \lambda^{-|k_1|^{100c}};$
if $|k_2|-|k_1|\geq \lambda^{|k_1|^{100c}}$, we have
$dist\{G_{k_1},G_{k_2}\}\geq c\big\vert|k_1|-|k_2| \big\vert^{-C}.$

Therefore for $|k_2|\geq |k_1|$ with $k_i\in \mathcal{K}(\lambda),i=1,2,$ it holds that
$$dist\{G_{k_1},G_{k_2}\}\geq \max\{ |k_2|-|k_1|,\lambda^{|k_1|^{100c}}\}^{-C}.$$

Hence for any $k_i\in \mathcal{K}(\lambda),i=1,2,$
$$dist\{G_{k_1},G_{k_2}\}\geq \max\{\left\vert|k_2|-|k_1|\right\vert,\lambda^{\min \{|k_1|,|k_2|\}^{100c}}\}^{-C}.$$
This ends the proof of Theorem \ref{15}.
\hfill\qed\end{proof}
\noindent Lemma \ref{typeiiiyzx} allows us to give the following definition of the (new) label of a spectral gap:
\begin{definition}\noindent Each spectral gap $G$ of $\lambda v$ can be identified by a unique $k=k(G)\in \mathbb{Z}$ such that \eqref{var-ep}--\eqref{dgkij} and the conclusion in Theorem \ref{15} hold true, which is called the (new) label of the gap. Thus $\mathcal{K}(\lambda)$ is a set of labels for gaps.
\end{definition}

\section{The proof of the main theorem based on a sharp estimate of the derivative on FLE (Lemma \ref{lemma18})}
In this section, based on a sharp estimate on the derivatives of the finite Lyapunov exponent (FLE) proved later, we will provide a sharp estimate on the regularity of LE with the help of LDT and AP.

 It is worth noting that the key for our purpose is not a sharp LDT, but a sharp estimate on the derivative of FLE. In fact, even a `weak' LDT is sufficient, see Subsection \ref{4.2}.
 At different energies, the magnitudes of the derivatives of FLE may be much different and thus the local regularity of LE is quite different. It sources from the fact that at different energies, the measure of the set of `bad' phase $x$ for which $\frac{\partial_E\|A_{N_0}(x,E)\|}{\|A_{N_0}(x,E)\|}$ does not have a good upper bound, may be of different orders in magnitude. For example, the measure of the set of `bad' $x$ for  $t\in\mathcal{FR}$  (the definition is given below) of full measure is much less than the one for ${EP}$. It comes from the degeneration of the function $g(x,\cdot)$ of $t\in {EP}$ and the nondegeneration of the one of $t\in \mathcal{FR}$, respectively. Moreover, the regularity of LE at other energies is in between according to the best approximation of them by the set ${EP}$.

We will first obtain the local regularity of LE for ${EP}$. Based on them, we can obtain the local regularity for other energies. Finally, we will show the absolutely $\frac12$-H\"older continuity of LE.

\subsection{A sharp estimate on the derivative of FLE}\label{goodcontrol} We will show that  the derivative of FLE is essentially governed by the resonances.
In the remaining part of this paper, we always assume  $G_{k}\triangleq(t_{-}^{k},t_{+}^{k})$, $k\in \mathcal{K}(\lambda).$  We give several definitions as follows.

Note \eqref{sk_df} and (2) of Theorem \ref{lm15} imply that \begin{equation}\label{GKLEQQ}|G_k|\leq C\lambda^{-c|k|}\leq C\lambda^{-cq^2_{N+s(k)-1}}\ll \lambda^{-q_{N+s(k)-1}}.\end{equation}

\begin{definition} Let   $j,\ k,\ S\in\Z_+$.
\begin{enumerate}

\item $\mathcal{B}_{j}(t)\triangleq B(t,\lambda^{-{q_{N+j-1}}})- B(t,\lambda^{-q^{\log q_{N+j-1}}_{N+j-1}})$ for all $j\in\Z_+$ (note that $\mathcal{B}_{j}(t)\bigcap \mathcal{B}_{j+1}(t)\neq \emptyset$).
\item ${\mathcal{B}}^k=(t^k_--\lambda^{-q_{N+s(k)-1}},t^k_-+\lambda^{-q_{N+s(k)-1}})\bigcup (t^k_+-\lambda^{-q_{N+s(k)-1}},t^k_++\lambda^{-q_{N+s(k)-1}})$; $\mathcal{B}_{\xi_k}^{k}:=\bigcup\limits_{\cdot=\pm}B(t_{\cdot}^{k},|G_k|^{1+\xi_k}),$ where $\xi_k:=q^{-\frac{1}{4}}_{N+s(k)-1}$.

\item $$K_{strong}(t)\triangleq \left\{\begin{matrix}\left\{k| t\in 2\mathcal{B}^{k}\right\}=\{k_1(t),k_2(t),\cdots, k_i(t),\cdots\}, & \left\{k| t\in 2\mathcal{B}^{k}\right\}\neq \emptyset\\ \{0\}, & \left\{k| t\in 2\mathcal{B}^{k}\right\}=\emptyset\end{matrix}\right.$$ If $|K_{strong}(t)|<\infty$, we define $k_{last}(t)\triangleq \max \{ |k_i(t)|| k_i\in K_{strong}(t)\}.$

\item $J(t)\triangleq\{j_1(t),j_2(t),\cdots, j_i(t),\cdots\}$, where $$j_i(t):=\left\{\begin{matrix}1+\max\{j\geq s(k_i)\vert t\in (2\mathcal{B}_j(t^{k_{i}(t)}_{-})\bigcup 2\mathcal{B}_j(t^{k_{i}(t)}_{+}))\},& \{j\geq s(k_i)\vert t\in (2\mathcal{B}_j(t^{k_{i}(t)}_{-})\bigcup 2\mathcal{B}_j(t^{k_{i}(t)}_{+}))\}\neq \emptyset \\ 1 ,& \{j\geq s(k_i)\vert t\in (2\mathcal{B}_j(t^{k_{i}(t)}_{-})\bigcup 2\mathcal{B}_j(t^{k_{i}(t)}_{+}))\}=\emptyset\end{matrix}\right.$$ If $|J(t)|<+\infty.$, we define $j_{last}(t)\triangleq \max \{j_i(t), j_i\in J(t)\}$

\item Semi-finite resonance points: $${\mathcal{FR}}\triangleq \bigcup\limits_{j\geq 0}\bigcap\limits_{|k|\geq j}\left(2\mathcal{B}^{k}\right)^c\triangleq \bigcup\limits_{j\geq 0}\mathcal{F}_j;$$
\item
Semi-infinite resonance points: $${\mathcal{IR}}=\Sigma^{\lambda} -\left({\mathcal{FR}}\cup EP\right).$$

\item Given $k\in \mathcal{K}(\lambda),$ for any interval $I=(a,b)\subset 2{\mathcal{B}}^k,$ let $$d(I,k):=dist\{\frac{a+b}{2},G_k\}$$ and for $1>d(I,k)>0$, $$\varsigma(I,k):=\lambda^{-|\log d(I,k)|^{\log |\log d(I,k)| }}.$$

\item $\eta=\lambda^{-(\log n)^C}.$


\end{enumerate}

\end{definition}

 The definitions (5) and (6) in the above describe the approximation of the fixed point $t\in \Sigma^\lambda$ by the endpoints of the gaps. Moreover, we will find that the local regularity of LE at ${\mathcal{FR}}$, ${\mathcal{IR}}$ or $EP$ are different from each other.

\




Based on the above definitions, with the help of Theorem \ref{lm15}, we have the following  properties.
\begin{proposition}\label{prop14}

\begin{enumerate}
\item If $t\in EP,$ then $|K_{strong}(t)|<\infty;$
\item If $|K_{strong}(t)|<\infty,$ then $t\in EP\bigcup {\mathcal{FR}};$

\item $Leb\{\mathcal{IR}\}=0.$\end{enumerate}
\end{proposition}
\begin{proof}

{(1):} It's enough to show $k_{last}(t)<+\infty.$ If $t\in EP$, by (1) of Theorem \ref{15} there exists some $m\in\Z$ such that \begin{equation}\label{tep} \text{all step}~j>s(m)+2~\text{belongs to}~\textbf{II}^{m}_j ~\text{and}~t\in \bar{G}_{m}\subset \mathcal{B}^{m}.\end{equation}  Therefore $k_{last}(t)\geq |m|.$  If $k_{last}(t)> |m|,$ then, b of Lemma \ref{boflem27} implies there must exist some $s(m)+2<s<s(m')+2$ such that step $s$ belongs to Type \textbf{I}, which contradicts with \eqref{tep}. Hence $k_{last}(t)=|m|<+\infty$ as desired.

{(2):} If $t\in {\mathcal{IR}},$ then by the definition, we can find a sequence $\{k_{i_j}(t)\}$ satisfying $|k_{i_j}(t)|\rightarrow +\infty$ as $j\rightarrow \infty$ such that $t\in \mathcal{B}^{k_{i_j}(t)},$ which implies that $K_{strong}=+\infty.$

{(3):} By Borel Contelli Lemma, it is enough to show that $\sum\limits_{k\in \Z} Leb\{\mathcal{B}^k\}<+\infty.$ Note that (2) of Theorem \ref{15} implies that $|t^k_+-t^k_-|\leq C\lambda^{-c k}$. Hence $$|\mathcal{B}^k|\subset |(t^k_--\lambda^{-q_{N+s(k)-1}},t^k_++\lambda^{-q_{N+s(k)-1}})|\subset |(t^k_--\lambda^{-ck},t^k_++\lambda^{-ck})|\leq 2C\lambda^{-ck},$$ we immediately have $\sum\limits_{i\in \Z}Leb\{\mathcal{B}^k\}<\lambda^{-\frac{c}{2} k}<+\infty,$ as desired.
\hfill\qed\end{proof}


The most crucial lemma in this study, which provides a precise estimate of FLE, is as follows. We will prove it in later sections.
\begin{lemma}\label{lemma18}

Under the same condition as in Theorem \ref{Th1}, let $G_k=(t_-^k, t^k_+),\ k\in \mathcal{K}$ be defined in Theorem \ref{15} and $s(k)$ be as in (\ref{sk_df}). For
$t\not\in \{t^{k}_-,t^{k}_+\}$, denote  $$H(k,t)= \frac{(sgn(t-t^k_+)sgn(t-t^k_-))\left(2t-(t^k_++t^k_-)\right)}{\sqrt{|t-t^k_+|\cdot |t-t^k_-|}}.$$
 Then there exists $\left(\log k\right)^{-C}<C_k\leq (\log k)^C$  such that
\begin{enumerate}
\item  $for\ any\ t\in \mathcal{B}_j(t^{k}_{-})\bigcup\mathcal{B}_j(t^{k}_{+}),\ j\geq s(k),\ m=1,2,$ we have
\begin{equation}\label{6.1.1}\left|\frac{d L_{m\cdot \mathcal{N}_j}(t)}{d t}\right|\leq  C_k\cdot\left\vert H(k,t)\right\vert+\lambda^{\left(\log (m\cdot \mathcal{N}_j)\right)^{{C}}};\end{equation}

\item $for\ \  t\in \left(\mathcal{B}_j(t^{k}_{-})\bigcup\mathcal{B}_j(t^{k}_{+})\right)\bigcap G_{k}, j\geq s(k), m=1,2$, we have
\begin{equation}\label{6.1.2}\left\vert\frac{d L_{m\cdot \mathcal{N}_j}(t)}{d t}-C_k\cdot H(k,t)\right\vert\leq \lambda^{\left(\log (m\cdot \mathcal{N}_j)\right)^{{C}}};\end{equation}

\item $for\ \  t_0\in \Sigma^{\lambda}-{EP},$ $m=1,2$ and $i\in \Z_+,$ it holds that
\begin{equation}\label{6.1.3}
\left|\frac{d L_{m\cdot \mathcal{N}_{l}}(t)}{d t}\right|\leq C_{k_i}\cdot\left\vert H(k_i(t_0),t)\right\vert+\lambda^{\left(\log (m\cdot \mathcal{N}_{j_i(t_0)})\right)^{{C}}},~
  \ \ ~t\in~\frac{1}{2}\mathcal{B}_{l}(t_0),\  j_{i}(t_0)\leq l< s(k_{i+1}(t_0));\end{equation}
  \item $for\  t\in \left(\mathcal{B}_j(t^{k}_{-})\bigcup\mathcal{B}_j(t^{k}_{+})\right)\bigcap G_{k}, m=1,2,$ fixed $i\in \Z_+,$ we have \begin{equation}\label{6.1.2'}\left\vert\frac{d L_{m\cdot \mathcal{N}_l}(t')}{d t}-C_k\cdot H(k,t')\right\vert\leq \lambda^{\left(\log (m\cdot \mathcal{N}_{j_i(t)})\right)^{{C}}}~
, ~t'\in~\frac{1}{2}\mathcal{B}_{l}(t), ~j_{i}(t)~\leq~l.\end{equation}

\end{enumerate}\label{lemma 6.1}
\end{lemma}

\begin{remark}\label{remark19} By the proof of Lemma \ref{lemma18}, if we replace $\mathcal{B}_{\cdot}$ with $2\mathcal{B}_{\cdot}$, then all the results as above also hold true.
\end{remark}
\begin{remark}
If $t\in \sum$ is close to ${EP}$ in some sense,  the right hand side of (\ref{6.1.1}) or (\ref{6.1.2})
is dominated by the first term. For such a $t$, the regularity of LE is strictly weaker than Lipschitz continuity. In particular, with the help of LDT and AP, (\ref{6.1.1}) or (\ref{6.1.2}) lead to exact $\frac12$-H\"older continuity at $t=t^k_{\pm}$. Otherwise, if  $t\in \sum$ is approximated by points in ${EP}$  at a slow speed (e.g. $t\in \mathcal{FR}$), the second term on  the right-hand side of (\ref{6.1.1}), which is similar to the right-hand side of (\ref{6.1.3}), will become the dominant one. In particular $\{j_i(t)\}_i$ is a finite set for $t\in \mathcal{FR}$, which leads to a finite upper bound for the derivatives of all FLE at $t$ and thus the Lipschitz continuity of LE there.

The difference between (\ref{6.1.1})+(\ref{6.1.2}) for $EP$ and (\ref{6.1.3}) for $\mathcal{FR}$  sources from the difference between $g_{s(\mathcal{N}_j)}$ of them. Roughly speaking, the measure of `bad' $x$ (see the beginning of this section) is decided by the measure of the set $\{x|\ |g_{s(\mathcal{N}_j)}|\le \lambda^{-Cq_{N+{s(\mathcal{N}_j)}-1}}\}$. In $EP$ case, Due to continuing resonances, $g_{s(\mathcal{N}_j)}$ is {\bf degenerate} and the measure of `bad' $x$ is about $\sqrt{\lambda^{-Cq_{N+{s(\mathcal{N}_j)}-1}}}$ for all large $j$. In contrast, in $\mathcal{FR}$ case, resonance occurs for finite times, which implies that $g_n$ is {\bf nondegenerate} for $n\gg 1$ and thus the measure of `bad' $x$ is about $\lambda^{-Cq_{N+n-1}}$.

For $t$ in the case $\mathcal{IR}$,  $\beta(t)$ defined in Section 3  describes the speed of approximation of $t$ by endpoints of gaps and determines the regularity of LE on $t$,  which is between $\frac{1}{2}$-H\"older continuity and Lipschitz continuity. LE for those $t$ with the fastest approaching speed by endpoints of gaps possesses a regularity close to $\frac{1}{2}$-H\"older continuity, while the ``slowest" ones correspond to   a regularity close to Lipschitz continuity.
\end{remark}
\vskip0.3cm

\subsection{Several useful inequalities}
\begin{lemma}\label{genhaobds} Given $K\in \Z_+$ and a sequence $a\leq a_1<b_1<a_2<b_2<\cdots\leq a_K\leq b_K\leq b$, a monotonic, concave and absolutely continuous function $f$ on $[a, b],$
 the following hold true.

\begin{enumerate}

\item If $f$ is monotonic increasing, then
$$0\leq \sum\limits_{1\leq i\leq K}(f(b_i)-f(a_i)) \leq f(\sum\limits_{i\geq 1}(b_i-a_i)+a)-f(a).$$

\item If $f$ is monotonic decreasing, then
\begin{equation}\label{danj}0\geq \sum\limits_{1\leq i\leq K}(f(b_i)-f(a_i)) \geq f(b)-f(b-\sum\limits_{1\leq i\leq K}(b_{i}-a_{i})).\end{equation}
\end{enumerate}

\end{lemma}
\begin{proof} Note that the concavity implies $f(x)$ is derivable almost everywhere and

\begin{equation}\label{yijd} f'(x)\geq f'(y),~{\rm \ for\ any\ } x\leq y\in [a,b]-X\end{equation} with $Leb\{X\}=0.$

The proof of (1) and (2) are similar, here we only prove (1).
 If $f$ is monotonic increasing, then $$f'(x)\geq 0,~{\rm \ for\ any\ } x\in [a,b]-X$$ with the same $X$ as above.
Hence \eqref{yijd} implies

\begin{equation}\label{concavee}0\leq \int_{c}^{d} f'(x) dx\leq \int_{c'}^{d'} f'(x) dx\end{equation} for any $c,c',d,d' \in [a,b]$ with $c'\leq c, d'\leq d$ and $c-c'=d-d'.$

Note $a_j<b_j<a_{j+1}<b_{j+1}$ implies
$$\sum\limits_{1\leq j\leq i-1}(b_j-a_j)+a\leq \sum\limits_{1\leq j\leq i-1}(a_{j+1}-a_j)+a_1\leq  a_{i},~i\geq 2;$$
$$\sum\limits_{1\leq j\leq i}(b_j-a_j)+a\leq \sum\limits_{1\leq j\leq i}(b_j-b_{j-1})+b_1\leq  b_{i},~i\geq 2.$$
Therefore it holds from the absolute continuity of $f$ and \eqref{concavee} that

\begin{equation}\label{concavee1}f(b_1)-f(a_1)=\int^{b_1}_{a_1} f'(x) dx\leq \int^{b_1-a_1+a}_{a} f'(x) dx=f(b_1-a_1+a)-f(a);\end{equation} \begin{equation}\label{concavee2}~f(b_i)-f(a_i)=\int^{b_i}_{a_i} f'(x) dx\leq \int^{\sum\limits_{1\leq j\leq i}(b_j-a_j)+a}_{\sum\limits_{1\leq j\leq i-1}(b_j-a_j)+a} f'(x) dx= f(\sum\limits_{1\leq j\leq i}(b_j-a_j)+a)-f(\sum\limits_{1\leq j\leq i-1}(b_j-a_j)+a),~i\geq 2.\end{equation}

Then, \eqref{concavee1}  and  \eqref{concavee2} imply
$$0\leq f(b_1)-f(a_1)+\sum\limits_{i\geq 2}f(b_i)-f(a_i)\leq f(b_1-a_1+a)-f(a)+\sum\limits_{i\geq 2}\left(f(\sum\limits_{1\leq j\leq i}(b_j-a_j)+a)-f(\sum\limits_{1\leq j\leq i-1}(b_j-a_j)+a)\right),$$ which implies
$$0\leq \sum\limits_{1\leq i\leq K}(f(b_i)-f(a_i))\leq f(\sum\limits_{1\leq i\leq K}(b_i-a_i)+a)-f(a).\hfill\qed$$
\end{proof}




Recall that $G_k=(t_k^-, t_k
^+)$. The following results hold from Lemma \ref{genhaobds}.
\begin{lemma}\label{jiandanjf1} Let $H(k,t)$ be defined in Lemma \ref{lemma18}. Denote $\sum\limits_{j=s(k)}^{+\infty}\mathcal{N}_{j}(\hat{\epsilon})\cdot \chi_{S_{k,j}}(t)$ by $\Xi(k,t)$ with the set $S_{k,j}=\{t|\lambda^{-q_{N+j}}< dist\{t, G_k\}\leq\lambda^{-q_{N+j-1}}\}$. For any pairwise disjoint intervals $(a_i,b_i)\subset \mathcal{B}^k,\ {i\geq 1}$ satisfying $\sum\limits_{i}|a_i-b_i|=X$,  it holds that
\begin{equation}\label{jfbd1}\int_{\bigcup\limits_{i}(a_i,b_i)} |H(k,t)| dt\leq 4\sqrt{|G_k|+X}\sqrt{X};\end{equation}

\begin{equation}\label{jfbd2}\int_{\bigcup\limits_{i}(a_i,b_i)} \Xi(k,t) dt \leq 4X^{1-2\xi_k}.\end{equation}

Particularly, for $(a-b,a+b)\subset 2\mathcal{B}^k$ with $1>b>0$, it holds that \begin{equation}\label{jiandanjf3}\int_{a-b}^{a+b} (|H(k,t)|+\Xi(k,t)) dt\leq  4b^{\min\{\beta_k(a),1-2\xi_k\}}.\end{equation}
\end{lemma}
\begin{proof}
We assume that there is no $i$ satisfying  $t^k_{\pm}$ or $\frac{t^k_-+t^k_+}{2}\in (a_i,b_i).$  For  other cases, for instance $(a_i,b_i)\ni t^k_{-}$, we only need to consider two  intervals $(a_i,t^k_-)$ and $(b_i,t^k_-)$ and the remaining proof are similar.

Thus there exist $i_l,\ l=1,\ 2,\ 3$ such that $$a_1<b_1<\cdots<a_{i_1}<b_{i_1}<t^k_{-}<\cdots<a_{i_2}<b_{i_2}<\frac{t^k_-+t^k_+}{2}<\cdots <a_{i_3}<b_{i_3}<t^k_{+}<\cdots.$$

Denote
$$\begin{array}{ll}&I_1=\{i \vert 1\le i\le i_1\},\ I_2=\{i|i_1<i\le i_2 \},\ I_3=\{i \vert i_2<i\le i_3\}, \ I_4=\{i \vert i>i_3\};\\
\\
&J_1:=(-\infty,t^k_-),\ J_2:=(t^k_-,\frac{t^k_-+t^k_+}{2}),\ J_3:=(\frac{t^k_-+t^k_+}{2},t^k_+),\ J_4:=(t^k_+,+\infty).
\end{array}
$$

 Let $$X_l=\sum_{i\in I_l}(b_i-a_i),~1\le l\le 4.$$
\textbf{The proof of \eqref{jfbd1}:}

 Note \begin{equation}\label{-} -\int_{a}^{b}|H(t,k)|dt =\int_{a}^{b}H(t,k)dt =\sqrt{|b-t^k_-||b-t^k_+|}-\sqrt{|a-t^k_-||a-t^k_+|},~a,b\in J_1~or~J_3;\end{equation} $$ \int_{a}^{b}|H(t,k)|dt =\int_{a}^{b}H(t,k)dt =\sqrt{|b-t^k_-||b-t^k_+|}-\sqrt{|a-t^k_-||a-t^k_+|},~a,b\in J_2~or~J_4.$$

 Therefore $\sqrt{|t-t^k_-||t-t^k_+|}$ is absolutely continuous on each $J_l,~l=1,2,3,4.$ On the other hand, by a direct calculation,
\begin{equation}\label{htd}\frac{d H(k,t)}{ dt}=\frac{d^2 \sqrt{|t-t^k_-||t-t^k_+|}}{d t^2}=-\frac{(t-t^k_+)^2(t-t^k_-)^2}{|t-t^k_+|^{\frac{7}{2}}|t-t^k_-|^{\frac{7}{2}}}<0\end{equation} for $t\in \bigcup\limits_{l=1}^4 J_l.$ Hence $\sqrt{|t-t^k_-||t-t^k_+|}$ is concave on each $J_l,~l=1,2,3,4.$

 Note $\sqrt{|t-t^k_-||t-t^k_+|}$ is monotonic decreasing on $J_1$ and $J_3.$ And
 $a_1<b_1<\cdots<a_{i_1}<b_{i_1}<t^k_{-}$ with $a_1,b_1,\cdots,a_{i_1},b_{i_1}\in J_1.$
 Hence by the help of \eqref{-} and \eqref{danj} of Lemma \ref{genhaobds}, we obtain

 $$\begin{array}{ll}&-\int_{\bigcup_{i\in I_1}(a_i, b_i)}|H(t,k)| dt=\int_{{\bigcup_{i\in I_1}(a_i, b_i)}}H(t,k) dt=\sum\limits_{l\in I_1}\left(\sqrt{|b_l-t^k_-||b_l-t^k_+|}-\sqrt{|a_l-t^k_-||a_l-t^k_+|}\right)\\&\geq \sqrt{|t^k_--t^k_-||t^k_--t^k_+|}-\sqrt{|(t^k_--X_1)-t^k_-||(t^k_--X_1)-t^k_+|}= 0-\sqrt{X_1(t^k_+-t^k_-+X_1)}.\end{array}$$

 Therefore
 $$\int_{X_1}|H(t,k)| dt\leq \sqrt{X_1(t^k_+-t^k_-+X_1)}.$$

 Similarly,  we have
$$\int_{X_l}|H(t,k)| dt\leq \sqrt{X_l(t^k_+-t^k_-+X_l)},~l=2,3,4.$$
 Then $$\begin{array}{ll}&\int_{\bigcup\limits_{i}(a_i,b_i)} |H(k,t)| dt =\bigcup_{l=1}^4 \bigcup_{i\in I_l}\int_{(a_i,b_i)} |H(k,t)| dt\\
 &\leq \sum\limits_{l=1,4}\sqrt{X_l(t_+^k-t^k_-+X_l)}+\sum\limits_{l=2,3}\sqrt{X_l(t_+^k-t^k_-+X_l)}\\ &\le \sqrt{(\sum\limits_{i=1,4}X_i)(\sum\limits_{i=1,4}(t_+^k-t^k_-+X_i))}+\sqrt{(\sum\limits_{i=2,3}X_i)(\sum\limits_{i=2,3}(t_+^k-t^k_-+X_i))}\quad {\rm (by\  Cauchy-Schwarz~ inequality)}\\&
\leq \sqrt{(\sum\limits_{i=1,2,3,4}X_i)(4(t_+^k-t^k_-)+X_1+X_4+X_2+X_3))}\quad {\rm (by\  Cauchy-Schwarz~ inequality)}\\& \leq \sqrt{X(4(t^k_+-t^k_-)+X)} \leq 4\sqrt{X((t^k_+-t^k_-)+X)}.
\end{array}$$
Thus we have \eqref{jfbd1}.

\

\textbf{The proof of \eqref{jfbd2}:}

For $t\notin G_k,$ let $j_t$ satisfy \begin{equation}\label{defgk}\lambda^{-q_{N+j_t}}< dist\{t, G_k\}\leq\lambda^{-q_{N+j_t-1}}.\end{equation}

Denote $$ F(t)=\sum\limits_{j\geq j_t+1} \mathcal{N}_j(\lambda^{-q_{N+j-1}}-\lambda^{-q_{N+j}})+\mathcal{N}_{j_t}(dist\{x, G_k\}-\lambda^{-q_{N+j_t}}).$$
One notes that for any $j^*\geq s(k),$ since $\mathcal{N}_{j}\leq \mathcal{N}_{j+1}$ it holds that
\begin{equation}\label{j^*}\begin{array}{ll}&F(t)\leq \sum\limits_{j\geq j^*+1} \mathcal{N}_j(\lambda^{-q_{N+j-1}}-\lambda^{-q_{N+j}})+\mathcal{N}_{j^*}(dist\{t, G_k\}-\lambda^{-q_{N+j^*}})\end{array}\end{equation}
(The case $j^*\geq j_t$ is trivial. For the case $j^*< j_t,$  one notes
$$\begin{array}{ll}
&\mathcal{N}_{j_t}(\lambda^{-q_{N+j_t-1}}-\text{dist}\{x,G_k\})+\sum\limits_{j^*\leq j\leq j_t-1}\mathcal{N}_j(\lambda^{-q_{N+j-1}}-\lambda^{-q_{N+j}})\\&\geq
\mathcal{N}_{j^*}(\lambda^{-q_{N+j_t-1}}-\text{dist}\{x,G_k\})+\sum\limits_{j^*\leq j\leq j_t-1}\mathcal{N}_{j^*}(\lambda^{-q_{N+j-1}}-\lambda^{-q_{N+j}})
\\&= \mathcal{N}_{j^*}(\lambda^{-q_{N+j^*}}-dist\{t, G_k\}),\end{array}$$ implies what we desire).

Hence for any fixed $0\leq a\leq 1,$ $l=1,2,3,4,$ and $t,x_1,x_2\in J_l$  taking $j^*=j_{ax_1+(1-a)x_2}$ in \eqref{j^*}, it holds that
\begin{equation}\label{ftftft}\begin{array}{ll}&F(t)\leq \sum\limits_{j\geq j_{ax_1+(1-a)x_2}+1} \mathcal{N}_j(\lambda^{-q_{N+j-1}}-\lambda^{-q_{N+j}})+\mathcal{N}_{j_{ax_1+(1-a)x_2}}(dist\{t, G_k\}-\lambda^{-q_{N+j_{ax_1+(1-a)x_2}-1}}).\end{array}\end{equation}

Note $$dist\{x,G_k\}=t^k_--x,~x\in J_1;dist\{x,G_k\}=x-t^k_-,~x\in J_2;dist\{x,G_k\}=t^k_+-x,~x\in J_3;dist\{x,G_k\}=x-t^k_+,~x\in J_4.$$

Hence for any fixed $l=1,2,3,4$ and $x_1,x_2\in J_l$, we have
$$a(dist\{x_1, G_k\})+(1-a)dist\{x_2, G_k\}=\text{dist}\{ax_1+(1-a)x_2, G_k\}.$$
Therefore by $\eqref{ftftft}$ we obtain
$$\begin{array}{ll}&aF(x_1)+(1-a)F(x_2)\leq \sum\limits_{j\geq j_{ax_1+(1-a)x_2}+1} \mathcal{N}_j(\lambda^{-q_{N+j-1}}-\lambda^{-q_{N+j}})\\&+\mathcal{N}_{j_{ax_1+(1-a)x_2}}(\text{dist}\{ax_1+(1-a)x_2, G_k\}-\lambda^{-q_{N+j_{ax_1+(1-a)x_2}-1}})
\\
\\&=F(ax_1+(1-a)x_2).\end{array}$$

It implies
$F(t)$ is concave in each $J_l,~l=1,2,3,4.$

On the other hand, note for $a<b \in J_2~or~J_4,$
$$\begin{array}{ll}&\int_{a}^b \Xi(k,t) dt=\int_{a}^b \sum\limits_{j=s(k)}^{+\infty}\mathcal{N}_{j}\cdot \chi_{S_{k,j}}(t)dt=\sum\limits_{j_b+1\leq j\leq j_a-1} \mathcal{N}_j(\lambda^{-q_{N+j-1}}-\lambda^{-q_{N+j}})\\&+\mathcal{N}_{j_b}(dist\{b,G_k\}-\lambda^{-q_{N+j_b}})+\mathcal{N}_{j_a}(\lambda^{-q_{N+j_a-1}}-dist\{a,G_k\})
=F(b)-F(a)\end{array}$$
and similarly, for $a,b \in J_1~or~J_3,$ it holds that $\begin{array}{ll}&\int_{a}^b \Xi(k,t) dt=F(a)-F(b).\end{array}$

Therefore $F(t)$ is absolutely continuous with $$F'(t)=\Xi(k,t),~a.e.~t\in J_2\bigcup J_4~\text{and}~F'(t)=-\Xi(k,t),~a.e.~t\in J_1\bigcup J_3.$$

Note
$F(t)$ is monotonic increasing on $J_2$ and $J_4$ and monotonic decreasing on $J_1$ and $J_3.$

Then for $l=1$ and~$t\in J_1,$ by the help of \eqref{danj} of Lemma \ref{genhaobds}, we have

$$\begin{array}{ll}&0\geq -\int_{\bigcup\limits_{l\in I_1} (a_l,b_l)} \Xi(k,t) dt= \sum\limits_{l\in I_1}(F(b_l)-F(a_l))\\&\geq \lim\limits_{t\rightarrow t^k_-}F(t)-F(t-\sum\limits_{l\in I_1}(b_l-a_l))=0-F(R)\\&=-\left(\sum\limits_{j\geq j_{R}+1} \mathcal{N}_j(\lambda^{-q_{N+j-1}}-\lambda^{-q_{N+j}})+\mathcal{N}_{j_{R}}(\sum\limits_{l\in I_1}(b_l-a_l)-\lambda^{-q_{N+j_{R}}})\right)\\&\geq
-\left(2\mathcal{N}_{j_{R}+1}\lambda^{-q_{N+j_{R}}}+\mathcal{N}_{j_{R}}(\sum\limits_{l\in I_1}(b_l-a_l))\right),\end{array}$$
where $R=t^k_--\sum\limits_{l\in I_l}(b_l-a_l).$

Recall $\mathcal{N}_j=[\lambda^{q^{\hat{\epsilon}}_{N+j-1}}]$ and $\xi_k=q^{-\frac{1}{4}}_{N+s(k)-1}.$
By choosing $\hat{\epsilon}$ in $\mathcal{N}_j$ sufficiently small,
we have
\begin{equation}\label{NJJJ}\begin{array}{ll}&\mathcal{N}_{j}\leq (\lambda^{-q_{N+j-1}})^{-\xi_k}\end{array},\quad j\geq s(k).\end{equation}

Since $dist(x, G_k)=t_-^k-x$ for $x<t_-^k$, we have \begin{equation}\label{xsxsxs111}\lambda^{-q_{N+R}}<\text{dist}(R, G_k)=\sum\limits_{l\in I_1}(b_l-a_l)<\lambda^{-q_{N+R-1}}\end{equation} by \eqref{defgk}.
Then \eqref{NJJJ} implies
$\mathcal{N}_{j_{R}}\leq \left(\sum\limits_{l\in I_1}(b_l-a_l)\right)^{-\xi_k}.$

Then \eqref{xsxsxs111} implies
$$\begin{array}{ll}&\left(2\mathcal{N}_{j_{R}+1}\lambda^{-q_{N+j_{R}}}+\mathcal{N}_{j_{R}}(\sum\limits_{l\in I_1}(b_l-a_l))\right)\leq
2\left(\lambda^{-q_{N+j_{R}}}\right)^{1-\xi_k}+(\sum\limits_{l\in I_1}(b_l-a_l))^{1-\xi_k}\\
\\&\leq 3(\sum\limits_{l\in I_1}(b_l-a_l))^{1-\xi_k}\leq 3(\sum\limits_{l\in \bigcup\limits_{l=1}^4 I_1}(b_l-a_l))^{1-\xi_k}\leq 3X^{1-\xi_k}\leq X^{1-2\xi_k}.\end{array}$$

Similarly,
$\int_{\bigcup\limits_{l\in I_l} (a_l,b_l)} \Xi(k,t) dt\leq X^{1-2\xi_k},\quad 2\le l\le 4.$
Thus
$\int_{\bigcup\limits_{l\in \bigcup\limits_{l=1}^4I_l} (a_l,b_l)} \Xi(k,t) dt\leq 4X^{1-2\xi_k.}$

\

\textbf{The proof of \eqref{jiandanjf3}:}

\

Without loss of generality, we assume that $a<t^k_-$ and the other cases are similar. Then we have
$dist\{a,G_{k}\}=t^k_--a.$
By \eqref{jfbd2} we have \begin{equation}\label{b11}\int_{a-b}^{a} \Xi(k,t) dt,\ \int_{a}^{a+b} \Xi(k,t) dt\leq 4b^{1-2\xi_k}.\end{equation}

\begin{enumerate}

\item {\bf The case $b\geq |a-t^k_-|(=t^k_--a)$}

\begin{lemma}\label{usefullem}\noindent The following holds true.

$$\left(\frac{\log (x+\epsilon)}{\log x}\right)'>0,~{\rm \ for\ any\ }~0<x<\frac{1}{10}~and~0<\epsilon<\frac{1}{10}.$$

\end{lemma}

\begin{proof} Consider the set $$\mathcal{Z}:=\{x\in \R_+ \vert \left(\frac{\log (x+\epsilon)}{\log x}\right)'=0\}.$$
Note
$$\left(\frac{\log (x+\epsilon)}{\log x}\right)'=0,~(x,~\epsilon>0)$$ is equivalent to
$$x\log x=(x+\epsilon)\log (x+\epsilon),~(x,~\epsilon>0).$$

By a direct calculation, we have $$\frac{d\left(x\log x-(x+\epsilon)\log (x+\epsilon)\right)}{dx}=\log x- \log (x+\epsilon)<0,\ (x>0).$$

Since $x\log x$ is monotonic decreasing from $0$ to $-e^{-1}$ in $(0,e^{-1}]$ and monotonic increasing from $-e^{-1}$ to $+\infty$ in $(e^{-1},+\infty)$ and
$-e^{-1}$ is the unique minimum point of $x\log x,$ one has
$$(e^{-1}-\epsilon)\log (e^{-1}-\epsilon)>e^{-1}\log(e^{-1})=(e^{-1}-\epsilon+\epsilon)\log(e^{-1}-\epsilon+\epsilon);~(e^{-1}+\epsilon)\log (e^{-1}+\epsilon)>e^{-1}\log(e^{-1}).$$

Therefore $\frac{\log (x+\epsilon)}{\log x}$ possesses a unique extreme point $x^*$ satisfying $e^{-1}-\epsilon<x^*\leq e^{-1}.$

Furthermore  since $$\lim_{\delta\rightarrow0+}\frac{\log (\delta+\epsilon)}{\log \delta}=0<\frac{\log (\delta+\epsilon)}{\log \delta},\ {\rm\ if}\ 0<\delta\ll 1,$$
it follows that $\frac{\log (x+\epsilon)}{\log x}$ is monotonic increasing in $$(0,x^*)\supseteq (0,e^{-1}-\epsilon)\supseteq (0,e^{-1}-\frac{1}{10})\supseteq (0,\frac{1}{10}).\hfill\qed$$
\end{proof}

The above lemma shows that
$${\frac{\log(b+|G_k|)}{2\log b}}\geq {\frac{\log(|a-t^k_-|+|G_k|)}{2\log |a-t^k_-|}}={\frac{\log|a-t^k_+|}{2\log |a-t^k_-|}}.$$

Then by the help of \eqref{jfbd1}, we have for $0<b\ll 1$
\begin{equation}\label{b22}\int_{a-b}^{a} |H(k,t)| dt,\int_{a}^{a+b} |H(k,t)| dt\leq \sqrt{|G_k|+b}\sqrt{b}\leq b^{\frac{1}{2}+\frac{\log(b+|G_k|)}{2\log b}}\leq b^{\frac{1}{2}+\frac{\log|a-t^k_+|}{2\log |a-t^k_-|}}= b^{\beta_k(a)}.\end{equation}
Finally \eqref{b11} and \eqref{b22} imply (\ref{jiandanjf3}).

\item {\bf The case $b< |a-t^k_-|(=t^k_--a)$}

 One notes $a,b\leq t^k_-.$
Moreover, \eqref{htd} tells us $\frac{d H(t,k)}{dt}<0.$ Thus the definition of $H(k,t)$ implies $H(k,t)<0~(t\leq t^k_-).$
Therefore
$$\begin{array}{ll}&\int_{a-b}^{a} |H(k,t)| dt<\int_{a}^{a+b} |H(k,t)| dt
\\&= \sqrt{t^k_+-a-b}\sqrt{t^k_--a-b}-\sqrt{t^k_+-a}\sqrt{t^k_--a}
\\&= \sqrt{|G_k|+|a-t^k_-|+b}\sqrt{|a-t^k_-|+b}-\sqrt{|G_k|+|a-t^k_-|}\sqrt{|a-t^k_-|}.
\end{array}$$

A direct computations shows
$$\begin{array}{ll}&\sqrt{|G_k|+|a-t^k_-|+b}\sqrt{|a-t^k_-|+b}-\sqrt{|G_k|+|a-t^k_-|}\sqrt{|a-t^k_-|}\\&=\sqrt{(|G_k|+|a-t^k_-|)(1+\frac{b}{(|G_k|+|a-t^k_-|)})}\sqrt{|a-t^k_-|(1+\frac{b}{|a-t^k_-|})}-\sqrt{|G_k|+|a-t^k_-|}\sqrt{|a-t^k_-|}\\&\leq (1+\frac{b}{|a-t^k_-|})\sqrt{|G_k|+|a-t^k_-|}\sqrt{|a-t^k_-|}-\sqrt{|G_k|+|a-t^k_-|}\sqrt{|a-t^k_-|}\\&\leq \frac{b}{|a-t^k_-|}\sqrt{|G_k|+|a-t^k_-|}\sqrt{|a-t^k_-|}=b |a-t^k_-|^{-\frac{1}{2}}\sqrt{|a-t^k_+|}\\&=b^{\frac{1}{2}+\frac{\log|a-t^k_+|+\log b- \log |a-t^k_-|}{2\log b}}\leq b^{\frac{1}{2}+\frac{\log|a-t^k_+|}{2\log |a-t^k_-|}}=b^{\beta_k(a)}.\end{array}$$
 Combining this with \eqref{b11}, we obtain what we desire.\hfill\qed
\end{enumerate}
\end{proof}

\subsection{The proof of Theorem \ref{Th1}}$$
$$

\begin{definition}\label{ti}
Given $t^*\in \R,$ $r\in \R_+$ and $t\in (t^*-r,t^*+r).$ We define a sequence of $t_i\rightarrow t$ and $n_i\rightarrow +\infty$ as follows.
\begin{enumerate}

\item If $t\geq t^*,$ then we set $$t_0=t,~t_i=t^*+\frac{t-t^*}{2^i},~i\in \Z_+;$$
\item If $t<t^*,$ then we set $$t_0=t,~t_i=t^*-\frac{t^*-t}{2^i},~i\in \Z_+;$$
\item $n_i$ satisfies $$\lambda^{-q_{N+n_i}}\leq |t_i-t^*|< \lambda^{-q_{N+n_i-1}},~i\in \N.$$
 \item Let  $\eta_i= q^{-\frac{1}{2}}_{N+n_i-1},~i\in \N.$

\end{enumerate}
\end{definition}
\

\begin{definition} $$L_{\mathcal{N}_{n_i}}(t_i)+L(t_i)-2L_{2{\mathcal{N}_{n_i}}}(t_i):=X_i;~L_{\mathcal{N}_{n_i}}(t_{i+1})+L(t_{i+1})-2L_{2{\mathcal{N}_{n_i}}}(t_{i+1}):=X_{i+1};$$
$$L_{\mathcal{N}_{n_{i}}}(t_{i})-L_{\mathcal{N}_{n_{i}}}(t_{i+1}):=Y_i;~L_{2\mathcal{N}_{n_{i}}}(t_{i})-L_{2\mathcal{N}_{n_{i}}}(t_{i+1}):=Z_{i}.$$
\end{definition}
Without loss of generality, we assume that $t>t^*.$ Thus the definition of $t_i$ implies \begin{equation}\label{wolgti}t_i>t_{i+1}\rightarrow t^*\quad {\rm (as\ }i\rightarrow +\infty).\end{equation}

\begin{lemma}\label{basiclemmaldt} Let $\sigma$ be from Lemma \ref{lm27}. Then for $M\in \Z_+,$ the following  hold true.
\begin{enumerate}

\item

For $t^k_+\leq t'\leq t,$
\begin{align}\label{universal*}
\vert L(t)-L(t'))\vert
\leq & \lambda^{-\frac{c}{10}\mathcal{N}^{\sigma}_{n_0}}+\left( \int^{t}_{t'} \left\vert\frac{d L_{\mathcal{N}_{n_i}}(t)}{d t}\right\vert dt+2\cdot\int^{t}_{t'} \left\vert\frac{d L_{2\mathcal{N}_{n_i}}(t)}{d t}\right\vert dt\right).
\end{align}

Particularly,
\begin{align}\label{universal}
\vert L(t)-L(t_M))\vert
\leq & \lambda^{-\frac{c}{10}\mathcal{N}^{\sigma}_{n_0}}+\left( \int^{t}_{t_{M}} \left\vert\frac{d L_{\mathcal{N}_{n_i}}(t)}{d t}\right\vert dt+2\cdot\int^{t}_{t_{M}} \left\vert\frac{d L_{2\mathcal{N}_{n_i}}(t)}{d t}\right\vert dt\right).
\end{align}

\item If \begin{equation}\label{tiaoj*}sgn(Y_i)=sgn(Y_{i+1})=sgn(Z_i)=sgn(Z_{i+1}),~i\in \N\end{equation}and
\begin{equation}\label{tiaoj}\max\{|Y_i-Z_i|,|X_i|,|X_{i+1}|\}\ll \min\{|Y_i|,|Z_i|\},~i\in \N,\end{equation}

 then
\begin{align}\label{universal2}
\vert L(t)-L(t_M)\vert\geq -\lambda^{-\frac{c}{10}\mathcal{N}^{\sigma}_{n_0}}+\sum_{i=0}^{M-1}\left\vert- \int^{t_i}_{t_{i+1}} \frac{d L_{\mathcal{N}_{n_i}}(t)}{d t}dt+2\cdot\int^{t_i}_{t_{i+1}}\frac{d L_{2\mathcal{N}_{n_i}}(t)}{d t}dt\right\vert.
\end{align}

\item It holds that \begin{equation}\label{liptail}\sum\limits_{0\leq i\leq M-1}\lambda^{\left(\log(\mathcal{N}_{n_i})\right)^{C}}|t_i-t_{i+1}|\leq \frac{5}{4}|t_0-t_M|^{1-2q^{-\frac{1}{2}}_{N+n_0-1}}\end{equation}
\end{enumerate}

\end{lemma}

\begin{proof}







Note that for $m=1,2,$
\begin{equation}\label{flejf*} |L_{m\mathcal{N}_{n_i}}(x)-L_{m\mathcal{N}_{n_i}}(y)|=\left\vert\int^{x}_{y}\frac{d L_{m\mathcal{N}_{n_i}}(t)}{d t} dt\right\vert\leq \int^{x}_{y} \left\vert\frac{d L_{m\mathcal{N}_{n_i}}(t)}{d t}\right\vert dt.
\end{equation}

Hence
\begin{equation}\label{flejf} |Y_i|\leq \int^{t_i}_{t_{i+1}} \left\vert\frac{d L_{\mathcal{N}_{n_i}}(t)}{d t}\right\vert dt;~|Z_i|\leq \int^{t_i}_{t_{i+1}} \left\vert\frac{d L_{2\mathcal{N}_{n_i}}(t)}{d t}\right\vert dt.
\end{equation}

By Lemma \ref{lm27},
\begin{equation}\label{xi} |X_i|\leq \lambda^{-c\mathcal{N}^{\sigma}_{n_i}},~i\in \N.
\end{equation}

\

\textbf{The proof of (1):}
By the definition, there exists some $K\in \N$ such that $t_{K+1}\leq t'\leq t_{K}.$

By \eqref{flejf} and \eqref{xi}, we have
$$\begin{array}{ll}
&\vert L(t)-L(t_{K})\vert \leq \sum\limits_{i=0}^{K-1}\vert L(t_i)-L(t_{i+1})\vert =\sum\limits_{i=0}^{K-1}|X_i-X_{i+1}+2Z_i-Y_i|\leq \sum\limits_{i=0}^{K-1}\left(|X_i|+|X_{i+1}|+2|Z_i|+|Y_i|\right)\\& \leq 2|Z_i|+|Y_i|+2\sum\limits_{i=0}^{K-1}\lambda^{-c\mathcal{N}^{\sigma}_{n_i}}\leq 2\sum\limits_{i=0}^{K-1}\lambda^{-c\mathcal{N}^{\sigma}_{n_i}}+\sum\limits_{i=1}^{K-1}\left( \int^{t_i}_{t_{i+1}} \left\vert\frac{d L_{\mathcal{N}_{n_i}}(t)}{d t}\right\vert dt+2\cdot\int^{t_i}_{t_{i+1}} \left\vert\frac{d L_{2\mathcal{N}_{n_i}}(t)}{d t}\right\vert dt\right).
\end{array}
.$$

Similarly, by \eqref{flejf*} and \eqref{xi}, we obtain
$$\vert L(t')-L(t_{K})\vert \leq 2\lambda^{-c\mathcal{N}^{\sigma}_{n_K}}+\left( \int^{t_K}_{t'} \left\vert\frac{d L_{\mathcal{N}_{n_i}}(t)}{d t}\right\vert dt+2\cdot\int^{t_K}_{t'} \left\vert\frac{d L_{2\mathcal{N}_{n_i}}(t)}{d t}\right\vert dt\right).$$

Therefore $$\begin{array}{ll}\vert L(t)-L(t')\vert&\leq \left( \int^{t}_{t'} \left\vert\frac{d L_{\mathcal{N}_{n_i}}(t)}{d t}\right\vert dt+2\cdot\int^{t}_{t'} \left\vert\frac{d L_{2\mathcal{N}_{n_i}}(t)}{d t}\right\vert dt\right)+2\sum\limits_{i=0}^{K}\lambda^{-c\mathcal{N}^{\sigma}_{n_i}}\\& \leq \left( \int^{t}_{t'} \left\vert\frac{d L_{\mathcal{N}_{n_i}}(t)}{d t}\right\vert dt+2\cdot\int^{t}_{t'} \left\vert\frac{d L_{2\mathcal{N}_{n_i}}(t)}{d t}\right\vert dt\right)+\lambda^{-\frac{c}{10}\mathcal{N}^{\sigma}_{n_0}}.\end{array}$$

Then \eqref{universal} directly follows from \eqref{universal*} by taking $t'=t_M,~M\in \N$.

\

\textbf{The proof of (2):}

Note that
$$L(t_i)-L(t_{i+1})=X_i-X_{i+1}+2Z_i-Y_i.$$
Hence
$$\begin{array}{ll}\left\vert \left(L(t_i)-L(t_{i+1})\right)-Z_i\right\vert &\leq |X_i-X_{i+1}|+|Z_i-Y_i|\leq 2\max\{|X_i|,|X_{i+1}|\}+|Z_i-Y_i|.\end{array}$$

Then

\begin{equation}\label{zifh}Z_i-\left(2\max\{|X_i|,|X_{i+1}|\}+|Z_i-Y_i|\right)\leq \left(L(t_i)-L(t_{i+1})\right)\leq Z_i+2\max\{|X_i|,|X_{i+1}|\}+|Z_i-Y_i|.\end{equation}
Note by \eqref{tiaoj}, we have $2\max\{|X_i|,|X_{i+1}|\}+|Z_i-Y_i|\ll |Z_i|.$
Therefore \eqref{zifh} implies $$ sgn(Z_i)=sgn(L(t_i)-L(t_{i+1})),~i\in \N.$$

Then by \eqref{tiaoj*},
\begin{equation}\label{fhb}sgn(L(t_i)-L(t_{i+1}))=sgn(Z_i)=sgn(Z_{i+1})=sgn(L(t_{i+1})-L(t_{i+2})),~i\in \N.\end{equation}

Then, \eqref{fhb} and $L(t)-L(t_{M})=\sum\limits_{i=0}^{M-1}(L(t_i)-L(t_{i+1}))$ imply

$$\left\vert L(t)-L(t_{M})\right\vert=\sum\limits_{i=0}^{M-1}\left\vert L(t_i)-L(t_{i+1})\right\vert.$$

Therefore
$$\begin{array}{ll}
&\vert L(t)-L(t_{M})\vert = \sum\limits_{i=0}^{M-1}\vert L(t_i)-L(t_{i+1})\vert =\sum\limits_{i=0}^{M-1}|X_i-X_{i+1}+2Z_i-Y_i|\geq \sum\limits_{i=0}^{M-1} \left(|2Z_i-Y_i\vert-|X_i|-|X_{i+1}|\right)\\& \geq \sum\limits_{i=0}^{M-1}|2Z_i-Y_i|-2\sum\limits_{i=0}^{M-1}\lambda^{-c\mathcal{N}^{\sigma}_{n_i}}\geq -\lambda^{-\frac{c}{10}\mathcal{N}^{\sigma}_{n_0}}+\sum\limits_{i=0}^{M-1}\left\vert- \int^{t_i}_{t_{i+1}} \frac{d L_{\mathcal{N}_{n_i}}(t)}{d t}dt+2\cdot\int^{t_i}_{t_{i+1}}\frac{d L_{2\mathcal{N}_{n_i}}(t)}{d t}dt\right\vert.
\end{array}
$$

\textbf{The proof of (3):}

Since $q_{N-n_i+1}\gg 1,$ one note that
$$|t_i-t_{i+1}|^{-1}\geq 2\lambda^{q_{N+n_i-1}}\gg \lambda^{q_{N+n_i-1}^{\frac{5}{6}}}.$$
Hence we can choose a suitable small $\hat{\epsilon}$ in $\mathcal{N}_{n_i}=\left[\lambda^{q_{N+n_i-1}^{\hat{\epsilon}}}\right]$ such that $$\begin{array}{ll}\lambda^{\left(\log\mathcal{N}_{n_i}\right)^{C}}&\leq \lambda^{q_{N+n_i-1}^{\frac{1}{3}}}=\left(\lambda^{q_{N+n_i-1}^{\frac{5}{6}}}\right)^{q^{-\frac{1}{2}}_{N+n_i-1}}<|t_i-t_{i+1}|^{-\eta_i}.\end{array}$$
Hence
\begin{equation}\label{etai}\lambda^{\left(\log\mathcal{N}_{n_i}\right)^{C}}|t_i-t_{i+1}|<|t_i-t_{i+1}|^{1-\eta_i}.\end{equation}

Recall the definition of $t_i$ that (note $t_0=t$)
$$|t_i-t_{i+1}|=\frac{1}{2}|t_i-t^k_+|=\frac{|t-t^k_+|}{2^{i+1}}.$$

Note $$0<\eta_i=q^{-\frac{1}{2}}_{N+n_i-1}\leq q^{-\frac{1}{2}}_{N+n_i-2}=\eta_{i-1}\leq \cdots =\eta_0\leq 2\eta_0\ll 1,~{\rm \ for\ any\ } i\in \Z_+.$$
Therefore $|t-t^M|<1$ implies \begin{equation}\label{etai1}\begin{array}{ll}&\sum\limits_{0\leq i\leq M-1}|t_i-t_{i+1}|^{1-\eta_i}=\sum\limits_{0\leq i\leq M-1}\left(\frac{|t-t^M|}{2^{i+1}}\right)^{1-\eta_i} \leq \sum\limits_{0\leq i\leq M-1}\left( \max\limits_{0\leq i\leq M-1}|t-t^M|^{1-\eta_i}\right)\left(\frac{1}{2^{i+1}}\right)^{1-\eta_i}\\\\&\leq |t-t^M|^{1-2\eta_0}\sum\limits_{0\leq i\leq M-1}\left(\frac{1}{2^{i+1}}\right)^{1-\eta_i}<\frac{5}{4}|t-t^M|^{1-2\eta_0}.\hfill\qed\end{array}\end{equation}
\end{proof}

\

In the following, we will prove all the conclusions in Theorem \ref{Th1} case by case based on (\ref{universal}) and (\ref{universal2}).
\subsubsection{{Exactly local $\frac{1}{2}$-H\"older continuity for $t\in EP$}}
\noindent

\

In \eqref{wolgti}, we take $t^*=t^k_+$ and $t\in \mathcal{B}^k(=(t^k_--\lambda^{-q_{N+s(k)-1}},t^k_++\lambda^{-q_{N+s(k)-1}}))$.

The definition of $t_i$ implies
\begin{equation}\label{titititi}2|t_{i+1}-t^k_+|=2|t_i-t_{i+1}|=|t^k_+-t_i|\end{equation}
and
the definition of $n_i$ implies $n_0\geq s(k)$ and
\begin{equation}\label{jied*}\lambda^{-q^{\log q_{N+n_i-1}}_{N+n_i-1}}\ll \lambda^{-q_{N+n_i}}\leq |t^k_+-t_i|\leq \lambda^{-q_{N+n_i-1}}.\end{equation}

Therefore \begin{equation}\label{jied}\lambda^{-q^{\log q_{N+n_i-1}}_{N+n_i-1}}\ll |t_i-t_{i+1}|<\lambda^{-q_{N+n_i-1}}.\end{equation}

Without loss of generality, we assume that $t_i>t^k_+.$ Then \eqref{titititi} implies $t_i>t_{i+1}\rightarrow t^k_+$ as $i\rightarrow +\infty.$

Then, \eqref{jied*} and \eqref{jied} show that

$$\lambda^{-q^{\log q_{N+n_i-1}}_{N+n_i-1}}<t_{i+1}-t^k_+<t_i-t^k_+<\lambda^{-q_{N+n_i-1}}.$$

Recall that $$\mathcal{B}_{n_i}(t^k_+)=(t^k_+-\lambda^{-q_{N+n_i-1}},t^k_++\lambda^{-q_{N+n_i-1}})-[t^k_+-\lambda^{-q^{\log q_{N+n_i-1}}_{N+n_i-1}},t^k_++\lambda^{-q^{\log q_{N+n_i-1}}_{N+n_i-1}}].$$

Therefore we have $(t_{i+1},t_{i})\subset \mathcal{B}_{n_i}(t^k_+).$

\

\textbf{The upper bound of the local $\frac{1}{2}$-H\"older continuity:}

By the help of \eqref{6.1.1} in (1) of Lemma \ref{lemma 6.1}, for $m=1,2$ and $t\in (t_{i+1},t_{i})\subset \mathcal{B}_{n_i}(t^k_+)$, we have

\begin{equation}\label{6.1.1*}\left|\frac{d L_{m\cdot \mathcal{N}_{n_i}}(t)}{d t}\right|\leq  C_k\cdot\left\vert H(k,t)\right\vert+\lambda^{\left(\log (m\cdot \mathcal{N}_{n_i})\right)^{C}}\end{equation}
with
\begin{equation}\label{cksc} (\log k)^{-C}\leq C_k \leq (\log k)^C\end{equation}
and
$$H(t,k)=\frac{(sgn(t-t^k_+)sgn(t-t^k_-))\left(2t-(t^k_++t^k_-)\right)}{\sqrt{|t-t^k_+|\cdot |t-t^k_-|}}.$$
By the continuity of $L(t)$ (see~\cite{wz1}), for any fixed $t$ and $t_+^k$ there exists $M=M(t,t_+^k)>0$ such that
\begin{equation}\label{tail}
|L(t_M)-L(t_+^k)|\leq |t-t^{k}_{+}|^{100};~\sqrt{|t_M-t^k_+||t_M-t^k_-|}\leq \frac{1}{2}\sqrt{|t-t^k_+||t-t^k_-|}.
\end{equation}

Now we set $$P(t,k)\triangleq 2C_k\cdot\sqrt{|(t-t^k_+)(t-t^k_-)|}.$$

By \eqref{6.1.1*}, for $m=1,2$, we have
\begin{align}\label{l6.1}
\sum_{0\leq i\leq M-1}\int^{t_i}_{t_{i+1}} \left\vert\frac{d L_{m\cdot\mathcal{N}_{n_i}}(t)}{d t}\right\vert dt\leq & (P(t,k)-P(t_M,k))+\sum\limits_{0\leq i\leq M-1}\lambda^{\left(\log(\mathcal{N}_{n_i})\right)^{C}}|t_i-t_{i+1}|.
\end{align}

By the help of \eqref{liptail} of Lemma \ref{basiclemmaldt}, we have

$$\sum\limits_{0\leq i\leq M-1}\lambda^{\left(\log(\mathcal{N}_{n_i})\right)^{C}}|t_i-t_{i+1}|\leq \frac{5}{4}|t-t^k_+|^{1-2q^{-\frac{1}{2}}_{N+n_0-1}}.$$

Since $n_0\geq s(k),$ \eqref{l6.1} yields

\begin{align}\label{l6.3}
\sum_{0\leq i\leq M-1}\int^{t_i}_{t_{i+1}} \left\vert\frac{d L_{m\cdot\mathcal{N}_{n_i}}(t)}{d t}\right\vert dt\leq & (P(t,k)-P(t_M,k))+\frac{5}{4}|t-t^k_+|^{1-2q^{-\frac{1}{2}}_{N+s(k)-1}}.
\end{align}
Then setting $m=1,\ 2$ in \eqref{l6.3}, we obtain
\begin{equation}\label{l6.4}\begin{array}{ll}
&\sum\limits_{i=1}^{M-1}\int^{t_i}_{t_{i+1}} |\frac{d L_{2\cdot\mathcal{N}_{n_i}}(t)}{d t}|dt+2\sum\limits_{i=1}^{M-1}\int^{t_i}_{t_{i+1}} |\frac{d L_{\mathcal{N}_{n_i}}(t)}{d t}|dt\\&\leq (P(t,k)-P(t_M,k))+\frac{5}{4}|t-t^k_+|^{1-2q^{-\frac{1}{2}}_{N+s(k)-1}}+2\left((P(t,k)-P(t_M,k))+\frac{5}{4}|t-t^k_+|^{1-2q^{-\frac{1}{2}}_{N+s(k)-1}}\right)\\& \leq 3(P(t,k)-P(t_M,k))+\frac{15}{4}|t-t^k_+|^{1-2q^{-\frac{1}{2}}_{N+s(k)-1}}.
\end{array}\end{equation}
Then \eqref{universal} and \eqref{l6.4} imply
\begin{equation}\label{universal3}\begin{array}{ll}
\vert L(t)-L(t_M)\vert &\leq \lambda^{-\frac{c}{10}\mathcal{N}^{\sigma}_{n_0}}+\sum_{i=1}^{M-1}\left( \int^{t_i}_{t_{i+1}} |\frac{d L_{\mathcal{N}_{n_i}}(t)}{d t}|dt+2\cdot\int^{t_i}_{t_{i+1}} |\frac{d L_{2\mathcal{N}_{n_i}}(t)}{d t}|dt\right)\\& \leq 3(P(t,k)-P(t_M,k))+\frac{15}{4}|t-t^k_+|^{1-2q^{-\frac{1}{2}}_{N+s(k)-1}}+\lambda^{-\frac{c}{10}\mathcal{N}^{\sigma}_{n_0}}.
\end{array}\end{equation}
Note $\lambda^{-q_{N+n_i}}\leq |t_i-t^k_+|,~i\in \N$. Then the fact $q_{N+n_i}\leq q^{C}_{N+n_i-1}$ implies
$$|t_i-t^k_+|^2>\lambda^{-2q_{N+n_i}}\gg \lambda^{-\frac{c}{10}\left[\lambda^{ q_{N+n_i-1}^{\hat{\epsilon}}}\right]^{\sigma }}=\lambda^{-\frac{c}{10}\mathcal{N}^{\sigma}_{n_i}},~{\rm \ for\ any\ } i\in \N.$$
Hence \begin{equation}\label{dpc} |t-t^k_+|^2\geq \lambda^{-\frac{c}{10}\mathcal{N}^{\sigma}_{n_0}}.
\end{equation}
Therefore \eqref{dpc},~\eqref{l6.3} and \eqref{universal3} yield
\begin{align}\label{ep121}
\vert L(t)-L(t_M)\vert
&\leq 3(P(t,k)-P(t_M,k))+\frac{15}{4}\sum\limits_{0\leq i\leq M-1}|t_{i}-t_{i+1}|^{1-\eta_i}+\lambda^{-\frac{c}{10}\mathcal{N}^{\sigma}_{s(k)}}\\ \nonumber
&<3P(t,k)+\frac{15}{4}|t-t^{k}_{+}|^{1-2q^{-\frac{1}{2}}_{N+s(k)-1}}+|t-t^{k}_{+}|^{2}\\ \nonumber
&=\left(6C_k|t-t^k_-|^{\frac{1}{2}}+\frac{15}{4}|t-t^{k}_{+}|^{\frac{1}{2}-2q^{-\frac{1}{2}}_{N+s(k)-1}}+|t-t^{k}_{+}|^{\frac{3}{2}}\right)|t-t^{k}_{+}|^{\frac{1}{2}}
.
\end{align}

Then \eqref{tail}, \eqref{ep121} and the fact $|t-t^k_+|, q^{-\frac{1}{2}}_{N+s(k)-1}\ll 1$ lead to
\begin{equation}\label{ep122}\begin{array}{ll}\vert L(t)-L(t^k_+)\vert&\leq \vert L(t)-L(t^k_+)\vert+\vert L(t)-L(t_M)\vert\\&\leq \left(6C_k|t-t^k_-|^{\frac{1}{2}}+\frac{15}{4}|t-t^{k}_{+}|^{\frac{1}{2}-2q^{-\frac{1}{2}}_{N+s(k)-1}}+|t-t^k_+|^{\frac{3}{2}}+|t-t^k_+|^{\frac{199}{2}}\right)|t-t^{k}_{+}|^{\frac{1}{2}}\\& \leq \left(6C_k|t-t^k_-|^{\frac{1}{2}}+4|t-t^{k}_{+}|^{\frac{1}{2}-2q^{-\frac{1}{2}}_{N+s(k)-1}}\right)|t-t^{k}_{+}|^{\frac{1}{2}}.\end{array}\end{equation}
Set
$$\tilde{C}_{k,t}:=\left(6C_k|t-t^k_-|^{\frac{1}{2}}+4|t-t^{k}_{+}|^{\frac{1}{2}-2q^{-\frac{1}{2}}_{N+s(k)-1}}\right).$$
Clearly, \eqref{ep122} means \begin{equation}\label{shangjiehe}\vert L(t)-L(t^k_+)\vert\leq \tilde{C}_{k,t}|t-t^{k}_{+}|^{\frac{1}{2}},~t\in B(t^k_+,\lambda^{-q_{N+s(k)-1}}),\end{equation} which yields the local $\frac{1}{2}-$ H\"older continuity of $L(t)$ at $t^k_+$ by the fact $$\limsup\limits_{t\in B(t^k_+,\lambda^{-q_{N+s(k)-1}})}\tilde{C}_{k,t}\leq 24 C_k<C(\log k)^C.$$ Thus we obtain the local $\frac{1}{2}-$ H\"older continuity of $L(t)$ at any $t\in EP.$

\

To get absolute $\frac{1}{2}-$ H\"older continuity, we have to precisely estimate  $ \tilde{C}_{k,t}$.

\

\textbf{Quantitative estimation on $ \tilde{C}_{k,t}$}

\

\begin{lemma}\label{ckt}
For $k\in \mathcal{K}(\lambda)$ {\rm (thus $|G_k|>0$)}, the following hold true.
\begin{enumerate}

\item  For $t\in \mathcal{B}^k,$ $\tilde{C}_{k,t}\leq C\lambda^{-\frac{1}{4}q_{N+s(k)-1}}$
\item  For $t\in \mathcal{B}_{\xi_k}^{k},$ \begin{equation}\label{ep21} \tilde{C}_{k,t}\leq 24C_k\sqrt{|G_k|}\end{equation}
and
\begin{equation}\label{upperbound-frac123}\vert L(t)-L(t^{k}_{X})\vert\leq 24C_k\sqrt{|G_k|}|t-t^{k}_{X}|^{\frac{1}{2}} ,\  t\in B(t^{k}_{X},|G_k|^{1+{\xi}_k}),~X\in\{+,-\}.\end{equation}


\end{enumerate}

\end{lemma}

\begin{proof} 

\textbf{The proof of (1):}

The definition of $\mathcal{B}^k=(t^k_--\lambda^{-q_{N+s(k)-1}},t^k_++\lambda^{-q_{N+s(k)-1}})$ implies
\begin{equation}\label{bkscale}|\mathcal{B}^k|=2\lambda^{-q_{N+s(k)-1}}.\end{equation}




Recall \eqref{cksc} implies $C_k\leq (\log |k|)^C.$
Then it follows from $0<\eta_0\ll 1$ and \eqref{bkscale} that
$$\begin{array}{ll}\tilde{C}_{k,t} &=6C_k|t-t^k_-|^{\frac{1}{2}}+4|t-t^{k}_{+}|^{\frac{1}{2}-2\eta_0}\leq (\log |k|)^C(2|\mathcal{B}^k|)^{\frac{1}{3}}\leq C\lambda^{-\frac{1}{4}q_{N+s(k)-1}}.\end{array}$$

\

\textbf{The proof of (2):}

Recall the definition of $\xi_k$ and $\mathcal{B}_{\xi_k}^{k}$:

$$\mathcal{B}_{\xi_k}^{k}=\bigcup\limits_{\cdot=\pm}B(t_{\cdot}^{k},|G_k|^{1+\xi_k}),\quad \xi_k:=q^{-\frac{1}{4}}_{N+s(k)-1}.$$
We only consider the case $t\in B(t_{+}^{k},|G_k|^{1+\xi_k}).$
Recall that $t_0=t$ and $n_0$ satisfies $$\lambda^{-q_{N+n_0}}\leq \vert t_0-t^k_+\vert\leq \lambda^{-q_{N+n_0-1}}\leq \lambda^{-q_{N+s(k)-1}}.$$

Note that for $t\in B(t_{+}^{k},|G_k|^{1+\xi_k})\subset B(t_{+}^{k},\lambda^{-q_{N+s(k)-1}})$,
it holds that
\begin{equation}\label{eta0}3q^{-\frac{1}{2}}_{N+s(k)-1}\ll \frac{1}{2}q^{-\frac{1}{4}}_{N+s(k)-1}=\frac{1}{2}\xi_k\ll 1.\end{equation}

In addition, \eqref{cksc} implies \begin{equation}\label{ckxj}C_k\geq (\log |k|)^{-C},\end{equation}

\eqref{pugapguj} implies \begin{equation}\label{gksj}|G_k|\leq C\lambda^{-c|k|},\end{equation}

and the definition of $s(k)$ implies $$\xi_k^{-4C}\geq q^{C}_{N+s(k)-1}\geq q_{N+s(k)}\geq |k|^{\frac{1}{2}}\geq q_{N+s(k)-1}=\xi_k^{-4}.$$ Hence
\begin{equation}\label{xikxiaj}|k|^{-\frac{1}{8}}\leq \xi_k\le |k|^{-\frac{1}{8C}}.\end{equation}
Therefore for $t\in B(t_{+}^{k},|G_k|^{1+\xi_k}),$
\begin{equation}\label{gkk1}\begin{array}{ll}
&|t-t^k_+|^{\frac{1}{2}-2q^{-\frac{1}{2}}_{N+s(k)-1}}\leq |G_k|^{(1+\xi_k)(\frac{1}{2}-2q^{-\frac{1}{2}}_{N+s(k)-1})}
=|G_k|^{(\frac{1}{2}+\frac{1}{2}\xi_k-2q^{-\frac{1}{2}}_{N+s(k)-1}-2q^{-\frac{1}{2}}_{N+s(k)-1}\xi_k)}\\&\leq |G_k|^{(\frac{1}{2}+\frac{1}{2}\xi_k-3q^{-\frac{1}{2}}_{N+s(k)-1})}\quad (by~0<\xi_k<1)
\leq |G_k|^{\frac{1}{2}+\frac{1}{3}\xi_k}\quad(by~\eqref{eta0})\leq C\lambda^{-\frac{1}{3}c|k|\xi_k}\sqrt{|G_k|}\quad(by~\eqref{gksj})
\\&<C\lambda^{-\frac{1}{3}c|k|^{\frac{7}{8}}}\sqrt{|G_k|}(by~\eqref{xikxiaj})\\&\leq \lambda^{-c|k|^{\frac{6}{7}}}\sqrt{|G_k|}\leq (\log |k|)^{-C}\sqrt{|G_k|}<C_k\sqrt{|G_k|}\quad (by~\eqref{ckxj}).\end{array}
\end{equation}

Note
\begin{equation}\label{gkk2}|t-t^k_-|\leq |t-t^k_+|+|G_k|\leq |G_k|^{1+\xi_k}+|G_k|<2|G_k|.\end{equation}

Combining \eqref{gkk1} with \eqref{gkk2}, we obtain $$\begin{array}{ll}\tilde{C}_{k,t} &=6C_k|t-t^k_-|^{\frac{1}{2}}+4|t-t^{k}_{+}|^{\frac{1}{2}-2q^{-\frac{1}{2}}_{N+s(k)-1}}\leq 6\sqrt{2}C_k\sqrt{|G_k|}+4C_k\sqrt{|G_k|}\leq 24C_k\sqrt{|G_k|}.\hfill\qed\end{array}$$
\end{proof}
\



\

\textbf{The lower bound of $\frac{1}{2}$-H\"older continuity:}

\

 We consider the case $t\in (t^k_+-|G_k|^{1+\xi_k},t^k_+),$ which implies $t\in G_k.$ We will prove
$$\frac{C_k}{2}\sqrt{|G_k|}|t-t^{k}_{+}|^{\frac{1}{2}}\leq \vert L(t)-L(t^k_+)\vert\leq 24C_k\sqrt{|G_k|}|t-t^{k}_{+}|^{\frac{1}{2}},~{\rm \ \ } t\in (t^k_+-|G_k|^{1+\xi_k},t^k_+).$$

 For $t\in (t^k_+-|G_k|^{1+\xi_k},t^k_+)$,  the definition of $t_i$\ (~in \eqref{wolgti})  implies
\begin{equation}\label{recalltiti}t^k_+>t_{i+1}>t_{i}\rightarrow t^k_+\ (~as~i\rightarrow +\infty),\quad~\frac{t^k_++t_i}{2}=t_{i+1}\end{equation}
 and
 $$t_i\geq t_0=t>\frac{t^k_++t^k_-}{2},\ i\in \N,$$ which implies $H(t,k)>0.$

 Note that to use (2) of Lemma \ref{basiclemmaldt}, we have to check that \eqref{tiaoj*} and \eqref{tiaoj} are valid.
 Recall the definition $P(t,k)\triangleq 2C_k\cdot\sqrt{|(t-t^k_+)(t-t^k_-)|}.$

A direct calculation yields $$\frac{dP(t,k)}{dt}=H(t,k)=\frac{(sgn(t-t^k_+)sgn(t-t^k_-))(t-\frac{t^k_++t^k_-}{2})}{\sqrt{|t-t^k_+||t-t^k_-|}}=\frac{-(t-\frac{t^k_++t^k_-}{2})}{\sqrt{|t-t^k_+||t-t^k_-|}}.$$
Hence
\begin{equation}\label{chang}\begin{array}{ll}\int_{t_{i}}^{{t_{i+1}}} C_kH(t,k) dt&= P(t_i,k)-P(t_{i+1},k)\\&=2C_k\cdot\left(\sqrt{(t^k_+-t_i)(t_i-t^k_-)}-\sqrt{(t^k_+-t_{i+1})(t_{i+1}-t^k_-)}\right)\\&=2C_k\cdot\frac{(t^k_+-t_i)(t_i-t^k_-)-(t^k_+-t_{i+1})(t_{i+1}-t^k_-)}{\left(\sqrt{(t^k_+-t_i)(t_i-t^k_-)}+\sqrt{(t^k_+-t_{i+1})(t_{i+1}-t^k_-)}\right)}\\&
=2C_k\cdot\frac{-\frac{1}{4}(t_i-t^k_+)^2+(t^k_+-t_i)(\frac{1}{2}(t_{i}-t^k_+)+\frac{1}{2}(t^k_+-t^k_-))}{\left(\sqrt{(t^k_+-t_i)(t_i-t^k_-)}+\sqrt{(t^k_+-t_{i+1})(t_{i+1}-t^k_-)}\right)}
\\&=\frac{1}{2}C_k\cdot\frac{-3(t_i-t^k_+)^2+2(t^k_+-t_i)(t^k_+-t^k_-)}{\left(\sqrt{(t^k_+-t_i)(t_i-t^k_++t^k_+-t^k_-)}+\sqrt{(t^k_+-t_{i+1})(t_{i+1}-t^k_++t^k_+-t^k_-)}\right)}
\\&=\frac{1}{2}C_k\cdot\frac{-3(t_i-t^k_+)^2+2(t^k_+-t_i)|G_k|}{\left(\sqrt{(t^k_+-t_i)(t_i-t^k_++|G_k|)}+\frac{1}{2}\sqrt{(t^k_+-t_i)(t_{i}-t^k_++2|G_k|)}\right)}
\\&=\frac{1}{2}C_k \sqrt{t^k_+-t_i}\cdot\frac{-3(t^k_+-t_i)+2|G_k|}{\left(\sqrt{(t_i-t^k_++|G_k|)}+\frac{1}{2}\sqrt{(t_{i}-t^k_++2|G_k|)}\right)}
\\&\geq \frac{1}{2}C_k \sqrt{t^k_+-t_i}\cdot\frac{\frac{199}{100}|G_k|}{\left(\sqrt{(-|G_k|^{1+\xi_k}+|G_k|)}+\frac{1}{2}\sqrt{(-|G_k|^{1+\xi_k}+2|G_k|)}\right)} \quad (note~t\in (t^k_+-|G_k|^{1+\xi_k},t^k_+))
\\&=\frac{1}{2}C_k \sqrt{t^k_+-t_i}\cdot\sqrt{|G_k|}\cdot\frac{\frac{199}{100}}{\left(\sqrt{(-|G_k|^{\xi_k}+1)}+\frac{1}{2}\sqrt{(-|G_k|^{\xi_k}+2)}\right)}
\\&>\frac{1}{2}C_k\cdot\sqrt{|G_k|}\cdot \sqrt{t^k_+-t_i}\frac{\frac{199}{100}}{1+\frac{\sqrt{2}}{2}}
\\&>\frac{1}{2}C_k\cdot\sqrt{|G_k|}\cdot \sqrt{t^k_+-t_i}
\\&=\left[C_k\cdot\sqrt{|G_k|}\cdot(t^k_+-t_i)^{-\frac{1}{2}}\right]\cdot{(t_{i+1}-t_i)}(~\text{by}~\eqref{recalltiti}).\end{array}
\end{equation}

Now we claim that
\begin{equation}\label{liangji1}C_k\sqrt{|G_k|}(-t_i+t^k_+)^{-\frac{1}{2}}\gg \lambda^{\left(\log(\mathcal{N}_{n_i})\right)^{3C}}.\end{equation}


In fact,
\begin{enumerate}
\item If $|k|\leq q^{\frac{1}{2}}_{N+n_i-1},$ then $C_k\sqrt{|G_k|}\geq (\log |k|)^{-C}\cdot\left(\lambda^{-\frac{1}{2}C|k|}\right)\geq \left(\lambda^{-C|k|}\right)\geq \lambda^{-Cq^{\frac{1}{2}}_{N+n_i-1}}.$

    Note $|t_i-t^k_+|\leq \lambda^{-q_{N+n_i-1}}.$
    Therefore
    $$C_k\sqrt{|G_k|}(-t_i+t^k_+)^{-\frac{1}{2}}\geq \lambda^{-Cq^{\frac{1}{2}}_{N+n_i-1}}\cdot \lambda^{\frac{1}{2}q_{N+n_i-1}}\geq \lambda^{\frac{1}{4}q_{N+n_i-1}}\gg \lambda^{q_{N+n_i-1}^{\frac{1}{3}}}\geq \lambda^{\left(\log(\mathcal{N}_{n_i})\right)^{C}}.$$
\item If $|k|\geq q^{\frac{1}{2}}_{N+n_i-1},$ then
with the help of $|t_i-t^k_+|<|G_k|^{1+\xi_k},$ we have
$$C_k\sqrt{|G_k|}(-t_i+t^k_+)^{-\frac{1}{2}}\geq C_k |G_k|^{-\xi_k}\geq c(\log k)^{-C}(\lambda^{-ck})^{-k^{-\frac{1}{8}}}\gg \lambda^{k^{\frac{1}{6}}}\geq \lambda^{q^{\frac{1}{3}}_{N+s(k)-1}}\geq \lambda^{\left(\log(\mathcal{N}_{s(k)})\right)^{C}}.$$
\end{enumerate}

Note \eqref{chang} and \eqref{liangji1} imply
\begin{equation}\label{fha}\int_{t_{i}}^{{t_{i+1}}} C_kH(t,k) dt\gg \lambda^{\left(\log(\mathcal{N}_{s(k)})\right)^{C}}(t_{i+1}-t_{i})(>0).\end{equation}

By \eqref{GKLEQQ}, we have $ B(t^k_+,|G_k|^{1+\xi_k})\subset \frac{1}{2}\mathcal{B}_{s(k)}.$ Hence \eqref{6.1.2'} is valid for $t\in B(t^k_+,|G_k|^{1+\xi_k})$.
Thus by \eqref{6.1.2'}, for $m=1,2$ and $t\in B(t^k_+,|G_k|^{1+\xi_k}),$ it holds that
\begin{equation}\label{ddbds}\left\vert\int^{t_{i+1}}_{t_{i}}\frac{d L_{m\mathcal{N}_{n_i}}(t)}{d t}dt-\int^{t_{i+1}}_{t_{i}}C_k H(t,k) dt\right\vert\leq \lambda^{\left(\log(\mathcal{N}_{n_i})\right)^{C}}(t_{i+1}-t_{i}).\end{equation}

Therefore

\begin{equation}\label{htk}\int^{t_{i+1}}_{t_{i}}C_k H(t,k) dt-\lambda^{\left(\log(\mathcal{N}_{n_i})\right)^{C}}|t_{i+1}-t_{i}|\leq \int^{t_{i+1}}_{t_{i}}\frac{d L_{m\mathcal{N}_{n_i}}(t)}{d t}dt\leq \int^{t_{i+1}}_{t_{i}}C_k H(t,k) dt+\lambda^{\left(\log(\mathcal{N}_{n_i})\right)^{C}}|t_{i+1}-t_{i}|.\end{equation}

Combining this with \eqref{fha}, we immediately have
$$sgn(L_{m\mathcal{N}_{n_i}}(t_i)-L_{m\mathcal{N}_{n_i}}(t_{i+1}))=sgn(\int^{t_i}_{t_{i+1}}\frac{d L_{m\mathcal{N}_{n_i}}(t)}{d t}dt)=sgn(\int^{t_i}_{t_{i+1}}C_k H(t,k) dt),~m=1,2.$$

From (\ref{fha}), we have\begin{equation}\label{zheng}\int^{t_{i+1}}_{t_{i}}C_k H(t,k) dt>0,~ i\in \N,\end{equation}. Hence we obtain
$$sgn(L_{m\mathcal{N}_{n_i}}(t_{i+1})-L_{m\mathcal{N}_{n_i}}(t_{i}))=1,~{\rm \ for\ any\ } i\in \N,~m=1,2$$
(recall $Z_i=L_{2\mathcal{N}_{n_i}}(t_i)-L_{2\mathcal{N}_{n_i}}(t_{i+1})$ and $Y_i=L_{\mathcal{N}_{n_i}}(t_i)-L_{\mathcal{N}_{n_i}}(t_{i+1})$ in Lemma \ref{basiclemmaldt}).
Thus \eqref{tiaoj*} is valid.

On the other hand,
\eqref{fha}, \eqref{htk} and \eqref{zheng} lead that for $m=1,2,$
\begin{equation}\label{109}\frac{10}{9}\int^{t_{i+1}}_{t_{i}}C_k H(t,k) dt>\int^{t_{i+1}}_{t_{i}}\frac{d L_{m\mathcal{N}_{n_i}}(t)}{d t}dt>\frac{9}{10}\int^{t_{i+1}}_{t_{i}}C_k H(t,k) dt.\end{equation}

In addition, \eqref{ddbds} implies
$$\left\vert\int^{t_{i+1}}_{t_{i}}\frac{d L_{2\mathcal{N}_{n_i}}(t)}{d t}dt-\int^{t_{i+1}}_{t_{i}}\frac{d L_{\mathcal{N}_{n_i}}(t)}{d t}dt\right\vert\leq 2\lambda^{\left(\log(\mathcal{N}_{n_i})\right)^{\hat{\epsilon}^{-1}}}|t_{i+1}-t_{i}|.$$

Therefore $$\begin{array}{ll}&\min\{\left\vert\int^{t_{i+1}}_{t_{i}}\frac{d L_{2\mathcal{N}_{n_i}}(t)}{d t}dt\right\vert,\left\vert\int^{t_{i+1}}_{t_{i}}\frac{d L_{\mathcal{N}_{n_i}}(t)}{d t}dt\right\vert\}\geq \frac{9}{10}\int_{t_{i}}^{{t_{i+1}}} C_kH(t,k) dt\gg 2\lambda^{\left(\log(\mathcal{N}_{s(k)})\right)^{C}}(t_{i+1}-t_{i})\\
\\& \geq\left\vert\int^{t_{i+1}}_{t_{i}}\frac{d L_{2\mathcal{N}_{n_i}}(t)}{d t}dt-\int^{t_{i+1}}_{t_{i}}\frac{d L_{\mathcal{N}_{n_i}}(t)}{d t}dt\right\vert.\end{array}$$

Hence $\min\{|Z_i|,|Y_i|\}\gg |Z_i-Y_i|,$ which is the first condition of \eqref{tiaoj}.

For the second condition of \eqref{tiaoj}, we have to check
$|X_i|\ll \min\{|Z_i|,|Y_i|\}.$

By~Theorem~\ref{Th18}, we have $$|X_i|\leq \lambda^{-c\mathcal{N}^{\sigma}_{n_i}}\quad= \lambda^{-c\left[ \lambda^{q^{{\hat{\epsilon}}}_{N+n_i-1}}\right]^{\sigma}}\leq \lambda^{-c\left[ \lambda^{q^{{\hat{\epsilon}}}_{N+s(k)-1}}\right]^{\sigma}}\ll \lambda^{-q^{1000C}_{N+n_i-1}}.$$  Note
$$\begin{array}{ll}&\lambda^{-q^{1000C}_{N+n_i-1}}\ll \lambda^{-q^{1000}_{N+n_i}}\leq \lambda^{-q^{1000}_{N+n_0}}\\
\\&\leq \lambda^{-q^{1000}_{N+s(k)}}\leq \lambda^{-|k|^{500}}\ll c(\lambda^{-Ck})\leq c(\log |k|)^{-C} \cdot (c\lambda^{-\frac{1}{2}Ck})\\
\\
&\leq \frac{1}{2}C_k \sqrt{|G_k|}\leq \int_{t_{i+1}}^{{t_i}} C_kH(t,k) dt. \quad(by~\eqref{chang})\end{array}$$
Then \eqref{109} shows
$$|X_i|\ll \min\{|Y_i|,|Z_i|\},~ i\in \N,$$ which is the second condition of \eqref{tiaoj}.

In summary, both \eqref{tiaoj*} and \eqref{tiaoj} hold true. Hence \eqref{universal2} is available.
i.e.
\begin{equation}\label{lowbb}\vert L(t)-L(t_M)\vert\geq -\lambda^{-\frac{c}{10}\mathcal{N}^{\sigma}_{n_0}}+\sum_{i=0}^{M-1}\left\vert- \int^{t_{i+1}}_{t_{i}} \frac{d L_{\mathcal{N}_{n_i}}(t)}{d t}dt+2\cdot\int^{t_{i+1}}_{t_{i}}\frac{d L_{2\mathcal{N}_{n_i}}(t)}{d t}dt\right\vert.\end{equation}

By \eqref{6.1.2'}, \eqref{etai} and \eqref{etai1}, for $m=1,2$ and $t\in B(t^k_+,|G_k|^{1+\xi_k}),$ it holds that
\begin{equation}\label{xiajieam}\begin{array}{ll}&\sum\limits_{i=0}^{M-1}\left\vert\int^{t_{i+1}}_{t_{i}}\frac{d L_{m\mathcal{N}_{n_i}}(t)}{d t}dt-\int^{t_{i+1}}_{t_{i}}C_k H(t,k) dt\right\vert\leq \sum\limits_{0\leq i\leq M-1}\lambda^{\left(\log(\mathcal{N}_{n_i})\right)^{\hat{\epsilon}^{-1}}}|t_{i+1}-t_{i}|\\
&\leq \sum\limits_{0\leq i\leq M-1}|t_{i+1}-t_{i}|^{1-\eta_i}\leq \frac{5}{4}|t-t^k_+|^{1-2q^{-\frac{1}{2}}_{N+s(k)-1}}.\end{array}
\end{equation}

Recall \eqref{tail} guarantees that
$\sqrt{|t_M-t^k_+||t_M-t^k_-|}\leq \frac{1}{2}\sqrt{|t-t^k_+||t-t^k_-|}.$
Hence
\begin{equation}\label{xiajieak}\begin{array}{ll}&\sum\limits_{i=0}^{M-1}\left\vert\int^{t_{i+1}}_{t_{i}} C_k H(t,k) dt\right\vert\geq \left\vert P(t,k)-P(t_M,k)\right\vert>|P(t,k)|-|P(t_M,k)|\\&>2C_k\sqrt{|t-t^k_+||t-t^k_-|}-2C_k\sqrt{|t_M-t^k_+||t_M-t^k_-|}\geq C_k\sqrt{|t-t^k_+||t-t^k_-|}.\end{array}\end{equation}

Furthermore  \eqref{dpc} implies
\begin{equation}\label{dpc1}\lambda^{-\frac{c}{10}\mathcal{N}^{\sigma}_{n_0}}\leq |t-t^k_+|^{2}.\end{equation}

By \eqref{tail}, \eqref{lowbb},\eqref{xiajieam}, \eqref{xiajieak} and \eqref{dpc1}, we obtain
\begin{equation}\label{ep12'1}\begin{array}{ll}
&\vert L(t)-L(t^k_+)\vert
\geq \vert L(t)-L(t_M)\vert-\vert L(t_M)-L(t^k_+)\vert\geq \vert L(t)-L(t_M)\vert-|t-t^k_+|^{100} \\&\geq\left(C_k|t-t^k_-|^{\frac{1}{2}}-\frac{15}{4}|t-t^{k}_{+}|^{\frac{1}{2}-2q^{-\frac{1}{2}}_{N+s(k)-1}}-|t-t^{k}_{+}|^{\frac{3}{2}}-|t-t^k_+|^{\frac{199}{2}}\right)|t-t^{k}_{+}|^{\frac{1}{2}}.
\end{array}
\end{equation}

Note \eqref{gkk1} implies \begin{equation}\label{gk1}|G_k|^{(\frac{1}{2}+\frac{1}{3}\xi_k)}\geq |G_k|^{(1+\xi_k)(\frac{1}{2}-2q^{-\frac{1}{2}}_{N+s(k)-1})}\geq |t-t^k_+|^{\frac{1}{2}-2q^{-\frac{1}{2}}_{N+s(k)-1}}.\end{equation}
Moreover, $t\in (t^k_+-|G_k|^{1+\xi_k},t^k_+)$ and $|G_k|\leq \lambda^{-c|k|}$ imply that
$$\begin{array}{ll}&|t-t^k_-|=|G_k|-|t-t^k_+|>|G_k|-|G_{k}|^{1+\xi_k}=(1-|G_k|^{\xi_k})|G_{k}|\geq (1-C\lambda^{-c|k| \xi_k})|G_{k}|\\&\geq (1-\lambda^{-c q^2_{N+s(k)-1}\cdot q^{-\frac{1}{4}}_{N+s(k)-1}})|G_{k}|>\frac{1}{16}|G_{k}|.\end{array}$$
In the above, we use  the fact $q^2_{N+s(k)-1}\leq |k|\leq q^2_{N+s(k)}$ and the definition $\xi_k=q^{-\frac{1}{4}}_{N+s(k)-1}$.

Therefore
\begin{equation}\label{gk2}|G_{k}|^{\frac{1}{2}}\leq 4|t-t^k_-|^{\frac{1}{2}}.\end{equation}
On the other hand,
$|G_k|\leq \lambda^{-c|k|}$ and $|k|^{-\frac{1}{8}}\leq \xi_k\leq |k|^{-\frac{1}{8C}}$ (by \eqref{xikxiaj}) imply
\begin{equation}\label{gk3}\begin{array}{ll}|G_k|^{\frac{1}{2}\xi_k}\leq \lambda^{-C|k|\frac{\xi_k}{2}}&\leq C^{\frac{\xi_k}{2}}\lambda^{-\frac{C}{2}|k|^{\frac{7}{8}}}\le\lambda^{-c|k|^{\frac{6}{7}}}<\frac{1}{150}(\log |k|)^{-C}\leq \frac{C_k}{150}.\end{array}\end{equation}

And $t\in (t^k_+-|G_k|^{1+\xi_k},t^k_+)\subset G_k$ implies
\begin{equation}\label{gk4} \begin{array}{ll}&|t-t^k_+|^{\frac{3}{2}}\leq |G_{k}|^{1+\xi_k}|t-t^k_+|^{\frac{1}{2}}\leq |G_k||t-t^k_+|^{\frac{1}{2}}\leq \lambda^{-c|k|}|t-t^k_+|^{\frac{1}{2}}\leq \frac{1}{100}(\log k)^{-C}|t-t^k_+|^{\frac{1}{2}}\leq \frac{1}{100}C_{k}|t-t^k_+|^{\frac{1}{2}}.\end{array}\end{equation}
Then, \eqref{gk1}, \eqref{gk2},~\eqref{gk3} and \eqref{gk4} yield that

\begin{equation}\label{ep22}\begin{array}{ll}&C_k|t-t^k_-|^{\frac{1}{2}}-\frac{15}{4}|t-t^{k}_{+}|^{\frac{1}{2}-2q^{-\frac{1}{2}}_{N+s(k)-1}}-|t-t^{k}_{+}|^{\frac{3}{2}}-|t-t^k_+|^{\frac{199}{2}}\\&\geq C_k|t-t^k_-|^{\frac{1}{2}}-\frac{15}{4}|G_k|^{\frac{1}{2}+\frac{1}{2}\xi_k}-2|t-t^{k}_{+}|^{\frac{3}{2}}\quad (~by~|t-t^k_+|<1~and~\eqref{gk1})\\&\geq C_k|t-t^k_-|^{\frac{1}{2}}-\frac{1}{40}C_k|G_k|^{\frac{1}{2}}-\frac{1}{100}C_{k}|t-t^k_+|^{\frac{1}{2}}(~by~\eqref{gk3}~and~\eqref{gk4})\\&\geq C_k|t-t^k_-|^{\frac{1}{2}}-\frac{1}{10}C_k|t-t^k_-|^{\frac{1}{2}}-\frac{1}{100}C_{k}|t-t^k_+|^{\frac{1}{2}}(b~\eqref{gk2})\geq \frac{C_k}{2}|t-t^k_-|^{\frac{1}{2}}\geq \frac{C_k}{2}|t^k_+-t^k_-|^{\frac{1}{2}}=\frac{C_k}{2}|G_k|^{\frac{1}{2}}.\end{array}\end{equation}

Finally, \eqref{ep12'1} and \eqref{ep22} imply

\begin{equation}\label{ep12'}
\vert L(t)-L(t^k_+)\vert
\ge \frac{C_k}{2}\sqrt{|G_k|}|t-t^{k}_{+}|^{\frac{1}{2}},\quad t\in (t^k_+-|G_k|^{1+\xi_k},t^k_+).
\end{equation}

Similarly,
\begin{equation}\label{ep12'*}
\vert L(t)-L(t^k_-)\vert
\ge \frac{C_k}{2}\sqrt{|G_k|}|t-t^{k}_{-}|^{\frac{1}{2}},\quad t\in (t^k_-,t^k_-+|G_k|^{1+\xi_k}).
\end{equation}

We finish the proof of the lower bound.
\

\textbf{Proof of exactly $\frac{1}{2}-$ H\"older continuity in $EP$}

Let us recall that combining \eqref{upperbound-frac123} with \eqref{ep12'}, we have already obtained

$$\frac{C_k}{2}\sqrt{|G_k|}|t-t^{k}_{+}|^{\frac{1}{2}}\leq \vert L(t)-L(t^k_+)\vert\leq 24C_k\sqrt{|G_k|}|t-t^{k}_{+}|^{\frac{1}{2}},~ t\in (t^k_+-|G_k|^{1+\xi_k},t^k_+)$$

and

$$\vert L(t)-L(t^k_+)\vert\leq 24C_k\sqrt{|G_k|}|t-t^{k}_{+}|^{\frac{1}{2}}, t\in (t^k_+-|G_k|^{1+\xi_k},t^k_++|G_k|^{1+\xi_k}).$$

This implies $L(t)$ is exactly $\frac{1}{2}-$ H\"older continuous at $t^k_+.$ The proof for $t^k_-$ is similar.

\

\subsubsection{\text{Local Lipschitz continuity for $t\in \mathcal{FR}$}}\label{lipproof}
\noindent
Note that for $t\in{\mathcal{FR}}$, we can obtain $(1-\epsilon)$-H\"older continuity of $L(t)$ for any fixed $\epsilon>0$  by a quite similar proof as for $\frac{1}{2}$-H\"older continuity. However, it is much more difficult to improve the regularity from $(1-\epsilon)$-H\"older continuity to Lipschitz continuity since for the latter we need a much sharper upper bound for FLE, that is $\left\|\frac{\partial_E\|A_{n}(x,E)\|}{\|A_{n}(x,E)\|}\right\|_{\mathcal{L}^1(\R/\Z)}\leq C(E)$ with $C(E)>0$ independent of $n$.
Our key observation on the proof of Lipschitz continuity is that the function $W$ defined as in Lemma \ref{lemma8} is odd on $\theta$. Thus if $I$ is of the form $(-a,a)$ with $a\ll 1$ and $\lambda_i$ satisfies some suitable conditions, then $|\int_I W(\theta(x),  \lambda_1(x),\lambda_2(x))dx| \ll \int_I |W(\theta(x),  \lambda_1(x),\lambda_2(x))|dx$, see Lemma \ref{usefullemma1}.

Instead of directly proving Local Lipschitz continuity for $t\in \mathcal{FR},$ we will show a stronger result as follows.
Recall that $\mathcal{B}_{l}(\tilde{t})=(\tilde{t}-\lambda^{-q_{N+j_{last}-1}},\tilde{t}+\lambda^{-q_{N+j_{last}-1}})-\{\tilde{t}\}$
and  $dist\{\tilde{t},G_{K}\}=\min\{|\tilde{t}-t^K_-|,|\tilde{t}-t^K_+|\}$.
\begin{lemma}\label{liplemma} For any fixed $\tilde{t}\in {\mathcal{FR}},$

 \begin{enumerate}

 \item $$\left|L(t)-L(\tilde{t})\right|\leq 4C_{k_{last}(\tilde{t})}\int_{t}^{\tilde{t}}\vert H(k_{last}(\tilde{t}),t)\vert dt+ 4\lambda^{(\log(\mathcal{N}_{j_{last}(\tilde{t})}))^{C}}|t-\tilde{t}|,\ t\in \bigcup\limits_{l\geq j_{last}(\tilde{t})}\frac{1}{2}\mathcal{B}_{l}(\tilde{t}).$$

\item For any $K\in \Z$ and $t\in \{t' \vert k_{last}(t')=K\},$ there exists $\tilde\epsilon(K)>0$ such that for any $0<|t-\bar{t}|\leq \tilde{\epsilon}$, \begin{equation}\label{lipschitz-1} \frac{|L(t)-L(\tilde{t})|}{|t-\tilde{t}|}\leq C\lambda^{-q^{\frac{1}{4}}_{N+s(K)-1}}(dist\{\tilde{t},G_{K}\})^{-\frac{1}{2}}.\end{equation}

\end{enumerate}
\end{lemma}

\begin{proof}

\textbf{The proof of (1) of Lemma \ref{liplemma}:}

By the definition of $j_{last}(\tilde{t})$ and part (3) of Lemma \ref{lemma 6.1}, we have that
\begin{equation}\label{fr111}\left|\frac{d L_{m\cdot \mathcal{N}_{l}}(t)}{d t}\right|\leq C_{k_{last}}\cdot \left\vert H(k_{last}(\tilde{t}),t)\right\vert+\lambda^{(\log(\mathcal{N}_{j_{last}(\tilde{t})}))^{C}},
m=1,2,\ l\geq j_{last}(\tilde{t}),\ t\in \frac{1}{2}\mathcal{B}_{l}(\tilde{t})
\end{equation}
with $\left(\log k_{last}\right)^{-C}<C_{k_{last}}\leq (\log k_{last})^C$.

Note that $$\bigcup\limits_{l\geq j_{last}(\tilde{t})}\frac{1}{2}\mathcal{B}_{l}(\tilde{t})=(\tilde{t}-\frac{1}{2}\lambda^{-q_{N+j_{last}(\tilde{t})-1}},\tilde{t}+\frac{1}{2}\lambda^{-q_{N+j_{last}(\tilde{t})-1}}).$$

Without loss of generality, we assume that
$ t>\tilde{t}\geq t^k_+.$

 Recall the definition of $j_{last}(\tilde{t})$ implies that
$$(\tilde{t}-\frac{1}{2}\lambda^{-q_{N+j_{last}(\tilde{t})-1}},\tilde{t}+\frac{1}{2}\lambda^{-q_{N+j_{last}(\tilde{t})-1}})\bigcap G_{k_{last}(\tilde{t})}=\emptyset.$$
Therefore for any $t\in (\tilde{t}-\frac{1}{2}\lambda^{-q_{N+j_{last}(\tilde{t})-1}},\tilde{t}+\frac{1}{2}\lambda^{-q_{N+j_{last}(\tilde{t})-1}}),$ we have $t>t^k_+,$ which implies
$$H(k_{last}(\tilde{t}),t)>0.$$
By the help of (1) of Lemma \ref{basiclemmaldt}, taking $t^*=\tilde{t}$ in \eqref{wolgti}, we obtain
\begin{equation}\label{lml}\begin{array}{ll}\vert L(t)-L(t_M))\vert
\leq & \lambda^{-\frac{c}{10}\mathcal{N}^{\sigma}_{n_0}}+\sum_{i=1}^{M-1}\left( \int^{t_i}_{t_{i+1}} \left\vert\frac{d L_{\mathcal{N}_{n_i}}(t)}{d t}\right\vert dt+2\cdot\int^{t_i}_{t_{i+1}} \left\vert\frac{d L_{2\mathcal{N}_{n_i}}(t)}{d t}\right\vert dt\right).\end{array}\end{equation}

For any fixed $t\neq \tilde{t},$ by the continuity of $L(t)$ at $\tilde{t}$ (\cite{wz1}), we can take $M\in \Z_+$ such that
$|L(\tilde{t})-L(t_M)|\leq |t-\tilde{t}|^{100}.$
Then, \eqref{fr111} and \eqref{lml} imply
\begin{equation}\label{1deguoc}\begin{array}{ll}
&|L(t)-L(\tilde{t})|\leq |L(t)-L(t_M)|+|L(t_M)-L(\tilde{t})| \\&\leq  \lambda^{-\frac{c}{10}\mathcal{N}^{\sigma}_{n_0}}+\sum_{i=1}^{M-1}\left( \int^{t_i}_{t_{i+1}} \left\vert\frac{d L_{\mathcal{N}_{n_i}}(t)}{d t}\right\vert dt+2\cdot\int^{t_i}_{t_{i+1}} \left\vert\frac{d L_{2\mathcal{N}_{n_i}}(t)}{d t}\right\vert dt\right))+|t-\tilde{t}|^{100}\\&\leq
\sum\limits_{i=0}^{M-1}3\left(\int_{t_{i+1}}^{t_{i}}C_{k_{last}}\cdot \left\vert H(k_{last}(\tilde{t}),t)\right\vert dt+\int_{t_{i+1}}^{t_i}\lambda^{(\log(\mathcal{N}_{j_{last}(\tilde{t})}))^{C}}dt\right)+|t-\tilde{t}|^{100}\\&\leq 3\left(\int_{t_M}^{t}C_{k_{last}}\cdot \left\vert H(k_{last}(\tilde{t}),t)\right\vert dt+\int_{t_M}^{t}\lambda^{(\log(\mathcal{N}_{j_{last}(\tilde{t})}))^{C}}dt\right)+|t-\tilde{t}|^{100}
\\&\leq 3\left(\int_{\tilde{t}}^{t}C_{k_{last}}\cdot \left\vert H(k_{last}(\tilde{t}),t)\right\vert dt+\int_{\tilde{t}}^{t}\lambda^{(\log(\mathcal{N}_{j_{last}(\tilde{t})}))^{C}}dt\right)+|t-\tilde{t}|^{100}\\&\leq 4\left(\int_{\tilde{t}}^{t}C_{k_{last}}\cdot \left\vert H(k_{last}(\tilde{t}),t)\right\vert dt+\int_{\tilde{t}}^{t}\lambda^{(\log(\mathcal{N}_{j_{last}(\tilde{t})}))^{C}}dt\right).\quad (by~|t-\tilde{t}|<1~and~\int_{\tilde{t}}^{t}\lambda^{(\log(\mathcal{N}_{j_{last}(\tilde{t})}))^{C}}dt>|t-\tilde{t}|)\end{array}\end{equation}

\textbf{Proof of (2) of Lemma \ref{liplemma}:}

Similarly to \eqref{1deguoc}, we have
$$\frac{|L(t)-L(\tilde{t})|}{t-\tilde{t}}\leq 3 \left(C_{k_{last}}\cdot \frac{1}{t-\tilde{t}}\int_{\tilde{t}}^{t}\left\vert H(k_{last}(\tilde{t}),t)\right\vert dt+\frac{1}{t-\tilde{t}}\int_{\tilde{t}}^{t}\lambda^{(\log(\mathcal{N}_{j_{last}(\tilde{t})}))^{C}}dt\right)+|t-\tilde{t}|^{99}.$$

Note that $$\lim\limits_{t\rightarrow \tilde{t}}\frac{1}{t-\tilde{t}}\int_{\tilde{t}}^{t}\left\vert H(k_{last}(\tilde{t}),t)\right\vert dt=|H(k_{last}(\tilde{t}),\tilde{t})|;$$
$$\lim\limits_{t\rightarrow \tilde{t}}\frac{1}{t-\tilde{t}}\int_{\tilde{t}}^{t}\lambda^{(\log(\mathcal{N}_{j_{last}(\tilde{t})}))^{C}}dt=
\lambda^{(\log(\mathcal{N}_{j_{last}(\tilde{t})}))^{C}}.$$
Then there exists $\tilde{\epsilon}(k_{last}(\tilde{t}))>0$ such that for any $0<|t-\bar{t}|\leq \tilde{\epsilon}(k_{last}(\tilde{t}))$, it holds that
\begin{equation}\label{277}\frac{|L(t)-L(\tilde{t})|}{|t-\tilde{t}|}\leq 30C_{k_{last}(\tilde{t})}\cdot |H(k_{last}(\tilde{t}),\tilde{t})|+30\lambda^{(\log(\mathcal{N}_{j_{last}(\tilde{t})}))^{C}}.\end{equation}

Recall $$H(\tilde{t},{k_{last}(\tilde{t})})= \frac{2\tilde{t}-(t^{k_{last}(\tilde{t})}_++t^{k_{last}(\tilde{t})}_-)}{\sqrt{(\tilde{t}-t^{k_{last}(\tilde{t})}_+)\cdot (\tilde{t}-t^{k_{last}(\tilde{t})}_-)}},~t^{k_{last}(\tilde{t})}_+>t^{k_{last}(\tilde{t})}_-.$$

With the assumption $\tilde{t}>t^{k_{last}(\tilde{t})}_+$, we have $dist\{\tilde{t},G_{k_{last}(\tilde{t})}\}=\tilde{t}-t^{k_{last}(\tilde{t})}_+$.

Hence \begin{equation}\label{hktgj}
\begin{array}{ll}
H(\tilde{t},{k_{last}(\tilde{t})})&=(dist\{\tilde{t},G_{k_{last}(\tilde{t})}\})^{-\frac{1}{2}}\cdot \frac{2\tilde{t}-(t^{k_{last}(\tilde{t})}_++t^{k_{last}(\tilde{t})}_-)}{\sqrt{(\tilde{t}-t^{k_{last}(\tilde{t})}_-)}}
\\&\leq (dist\{\tilde{t},G_{k_{last}(\tilde{t})}\})^{-\frac{1}{2}}\cdot \frac{2\tilde{t}-(t^{k_{last}(\tilde{t})}_-+t^{k_{last}(\tilde{t})}_-)}{\sqrt{(\tilde{t}-t^{k_{last}(\tilde{t})}_-)}}
\\&=(dist\{\tilde{t},G_{k_{last}(\tilde{t})}\})^{-\frac{1}{2}}\cdot 2\sqrt{(\tilde{t}-t^{k_{last}(\tilde{t})}_-)}\\&= (dist\{\tilde{t},G_{k_{last}(\tilde{t})}\})^{-\frac{1}{2}}\cdot 2\sqrt{(\tilde{t}-t^{k_{last}(\tilde{t})}_+)+(t^{k_{last}(\tilde{t})}_+-t^{k_{last}(\tilde{t})}_-)}
\\&\leq 2(dist\{\tilde{t},G_{k_{last}(\tilde{t})}\})^{-\frac{1}{2}}\cdot \sqrt{\lambda^{-q_{N+s(k_{last}(\tilde{t}))-1}}+C\lambda^{-c|k_{last}(\tilde{t})|}}\\&\leq
C(dist\{\tilde{t},G_{k_{last}(\tilde{t})}\})^{-\frac{1}{2}}\cdot \sqrt{\lambda^{-q_{N+s(k_{last}(\tilde{t}))-1}}+\lambda^{-cq^2_{N+s(k_{last}(\tilde{t}))-1}}}\\&\leq C\lambda^{{-q^{\frac{1}{4}}_{N+s(k_{last}(\tilde{t}))-1}}}(dist\{\tilde{t},G_{k_{last}(\tilde{t})}\})^{-\frac{1}{2}}.\end{array}\end{equation}

On the other hand,
$\lambda^{(\log(\mathcal{N}_{j_{last}(\tilde{t})}))^{C}}\leq \lambda^{(\log \lambda)^{C}q^{{C}{\hat{\epsilon}}}_{N+j_{last}(\tilde{t})-1}}\leq \lambda^{q^{{\frac{1}{3}}}_{N+j_{last}(\tilde{t})-1}}$
and
$$\begin{array}{ll}&\lambda^{-q^{\frac{1}{4}}_{N+s(k_{last}(\tilde{t}))-1}}(dist\{\tilde{t},G_{k_{last}(\tilde{t})}\})^{-\frac{1}{2}}=\lambda^{-q^{\frac{1}{4}}_{N+s(k_{last}(\tilde{t}))-1}}(\tilde{t}-t^k_+)^{-\frac{1}{2}}\\&\geq \lambda^{-q^{\frac{1}{4}}_{N+s(k_{last}(\tilde{t}))-1}}\lambda^{\frac{1}{2}q_{N+j_{last}(\tilde{t})-1}}\geq \lambda^{-q^{\frac{1}{4}}_{N+j_{last}(\tilde{t})-1}}\lambda^{\frac{1}{2}q_{N+j_{last}(\tilde{t})-1}}\\&\gg \lambda^{q^{{\frac{1}{3}}}_{N+j_{last}(\tilde{t})-1}}.\end{array}$$
Then
\begin{equation}\label{lambdagj}\lambda^{(\log(\mathcal{N}_{j_{last}(\tilde{t})}))^{C}}\ll C\lambda^{-q^{\frac{1}{4}}_{N+s(k_{last}(\tilde{t}))-1}}(dist\{\tilde{t},G_{k_{last}(\tilde{t})}\})^{-\frac{1}{2}}.\end{equation}

Finally \eqref{277}, \eqref{hktgj} and \eqref{lambdagj} complete the proof.
\hfill\qed\end{proof}

\subsubsection{Some preparations for the proof of the absolutely H\"older-continuity}
We  need a slightly stronger version of Lemma \ref{liplemma}. For this purpose, we need the following definitions.
\begin{definition}\label{def4.2}

 Set $${\mathcal{FR}}^*=\bigcup\limits_{j\geq 0}\bigcap\limits_{|k|\geq j}\left(\mathcal{B}^{k}\right)^c.$$

$$K^*_{strong}(t)\triangleq \left\{\begin{matrix}\left\{k| t\in \mathcal{B}^{k}\right\}=\{k^*_1(t),k^*_2(t),\cdots, k^*_i(t),\cdots\}, & \left\{k| t\in \mathcal{B}^{k}\right\}\neq \emptyset\\ \{0\}, & \left\{k| t\in \mathcal{B}^{k}\right\}=\emptyset\end{matrix}\right.$$
$k^*_{last}(t)\triangleq \max \{k^*_i(t), k^*_i\in K^*_{strong}(t)\}$ if $|K^*_{strong}(t)|<+\infty.$

 $J^*(t)\triangleq\{j^*_1(t),j^*_2(t),\cdots, j^*_i(t),\cdots\}$, where $$j^*_i(t):=\left\{\begin{matrix}1+\max\{j\geq s(k_i)\vert t\in (\mathcal{B}_j(t^{k_{i}(t)}_{-})\bigcup \mathcal{B}_j(t^{k_{i}(t)}_{+}))\},& \{j\geq s(k^*_i)\vert t\in (\mathcal{B}_j(t^{k^*_{i}(t)}_{-})\bigcup \mathcal{B}_j(t^{k^*_{i}(t)}_{+}))\}\neq \emptyset \\ 1 ,& \{j\geq s(k^*_i)\vert t\in (\mathcal{B}_j(t^{k^*_{i}(t)}_{-})\bigcup \mathcal{B}_j(t^{k^*_{i}(t)}_{+}))\}=\emptyset\end{matrix}\right.$$
$j^*_{last}(t)\triangleq \max \{j^*_i(t), j^*_i\in J(t)\}$ if $|J(t)|<+\infty.$\end{definition}
\vskip 0.3cm
 Recall  ${\mathcal{FR}}=\bigcup\limits_{j\geq 0}\bigcap\limits_{|k|\geq j}\left(2\mathcal{B}^{k}\right)^c$. Thus from the definition, we have ${\mathcal{FR}}\subset \mathcal{FR}^*.$

 Note for $k^*\in \Z$ and $t\notin \bigcup\limits_{|k|>|k^*|} \mathcal{B}_k,$ we have \begin{equation}\label{k***}k^*_{last}(t)\leq |k^*|;~j^*_{last}(t)\leq s(|k^*|)+3.\end{equation}

 Thus similar as the proof  of Lemma \ref{liplemma}, we obtain
\begin{lemma}\label{liplemma*} For  $k^*\in \Z$ and for $\tilde{t}\in 2\mathcal{B}_{k^*}-\bigcup\limits_{|k|>|k^*|} \mathcal{B}_k$. It holds that
$$\left|L(t)-L(\tilde{t})\right|\leq 4C_{k^*_{last}(\tilde{t})}\int_{t}^{\tilde{t}}\vert H(k^*_{last}(\tilde{t}),t)\vert dt+ 4\lambda^{(\log(\mathcal{N}_{j^*_{last}(\tilde{t})}))^{C}}|t-\tilde{t}|,\ t\in \bigcup\limits_{l\geq j^*_{last}(\tilde{t})}\frac{1}{2}\mathcal{B}_{l}(\tilde{t}).$$

\end{lemma}

By Lemma \ref{liplemma*} and \eqref{k***}, we have
\begin{equation}\label{differ*} \left|L(t)-L(\tilde{t})\right|\leq 4C_{k^*}\int_{t}^{\tilde{t}}\vert H(k^*,t)\vert dt+ 4\lambda^{(\log(\mathcal{N}_{s(k^*)}))^{C}}|t-\tilde{t}|,\ t\in \bigcup\limits_{l\geq s(|k^*|)+3}\frac{1}{2}\mathcal{B}_{l}(\tilde{t}).\end{equation}

\begin{remark} The definitions of $K^*_{strong}(t)$ and $J^*(t)$ are essentially the same as those of $K_{strong}(t)$ and $J(t)$. In fact,  in the definition of $K^*_{strong}(t)$ and $J^*(t)$, replacing the coefficients in front of $\mathcal{B}^k$ and $\mathcal{B}_k$ does not affect the correctness of the above lemma; the only difference is the constant $C$ on the right-hand side of the inequality.
\end{remark}

\noindent\textsf{\large An overview on the most crucial point for absolutely H\"older-continuity of LE}
\vskip 0.2cm

Given any interval $(a,b)\in \R,$  to estimate $|L(a)-L(b)|.$ It is enough to consider the following two cases.
\begin{enumerate}
\item $(a,b)- \bigcup\limits_{k} \mathcal{B}^k\neq \emptyset,$
\item $(a,b)\subset \bigcup\limits_{k} \mathcal{B}^k.$

\end{enumerate}

For case (1), note that $k^*_{last}(t)=0$ for all $t\in (a,b)- \bigcup\limits_{k} \mathcal{B}^k(\subset \mathcal{FR}^*).$ Since the definition guarantees that $\{t\vert k^*_{last}(t)=0\}$ is closed, we have $k^*_{last}(t)=0$ for all $t\in [a,b]- \bigcup\limits_{k} \mathcal{B}^k(\subset \mathcal{FR}^*).$ Then there exists $\epsilon_0>0$ such that
$$ |L(t)-L(t')|\leq C|t-t'|,~{\rm \ for\ any\ } |t-t'|\leq \epsilon_0.$$ Therefore we can find a finite sequence $t_1,t_2,\cdots,t_K\in [a,b]$ such that $$[a,b]-\bigcup\limits_{k} \mathcal{B}^k\subset \bigcup\limits_{i=1}^K(t_i-\epsilon_0,t_i+\epsilon_0)$$ and
\begin{equation}\label{lipschitzjieshi} |L(t_i)-L(t')|\leq C|t_i-t'|,~ |t_i-t'|\leq \epsilon_0,~i=1,2,\cdots,K.\end{equation}

By \eqref{lipschitzjieshi}, it remains to consider the case $t\in (a,b)-\bigcup\limits_{i=1}^K[t_i-\epsilon_0,t_i+\epsilon_0],$ which consists of finite open intervals $\{(a_i,b_i)\}_{i=1}^{K^*}.$ Note that $\{(a_i,b_i)\}_{i=1}^{K^*}\subset \bigcup\limits_{k} \mathcal{B}^k.$ Hence we only need to consider the case $(a_i,b_i)\subset \mathcal{B}^{k(i)}$ for some fixed $k(i)\in \Z,$ which leads us to consider case (2).

\

For case (2), by taking $t^*=t^k_+$ and $t=b$ in \eqref{universal}, we have
$$\vert L(a)-L(t_M))\vert
\leq \lambda^{-\frac{c}{10}\mathcal{N}^{\sigma}_{n_0}}+\sum_{i=1}^{M-1}\left( \int^{t_i}_{t_{i+1}} \left\vert\frac{d L_{\mathcal{N}_{n_i}}(t)}{d t}\right\vert dt+2\cdot\int^{t_i}_{t_{i+1}} \left\vert\frac{d L_{2\mathcal{N}_{n_i}}(t)}{d t}\right\vert dt\right)$$ with $n_0$ determined by
$$q_{N+n_0-1}\log \lambda\leq \left\vert\log |b-t^k_+|\right\vert\leq q_{N+n_0}\log \lambda.$$

Then by the help of \eqref{6.1.1}, we obtain for $m=1,2$ and $i=0,1,\cdots,M-1,$

$$\left|\frac{d L_{m\cdot \mathcal{N}_{n_i}}(t)}{d t}\right|\leq  C_k\cdot\left\vert H(k,t)\right\vert+\lambda^{\left(\log (m\cdot \mathcal{N}_{n_i})\right)^{C}}
,~ t\in \mathcal{B}_{n_i}(t^k_+).$$

On the other hand, the definition of $t_i$ implies that $(t_{i+1},t_i)\subset \mathcal{B}_{n_i}(a).$ Thus

$$\int_{t_{i+1}}^{t_i}\left|\frac{d L_{m\cdot \mathcal{N}_{n_i}}(t)}{d t}\right| dt\leq  C_k\cdot\int_{t_{i+1}}^{t_i}\left\vert H(k,t)\right\vert dt+\lambda^{\left(\log (m\cdot \mathcal{N}_{n_i})\right)^{C}}(t_i-t_{i+1})
.$$

Combining all these as above and letting $t_M\rightarrow b,$ we have

\begin{equation}\label{jishujiangjie}\vert L(a)-L(b))\vert
\leq 3C_k\int_{a}^b\vert H(k,s)\vert ds+\int_{a}^b\Xi(k,s) ds+\lambda^{-\frac{c}{10}\mathcal{N}^{\sigma}_{n_0}}\end{equation} with $$\Xi(k,s)=\sum\limits_{j=s(k)}^{+\infty}\mathcal{N}_{j}(\hat{\epsilon})\cdot \chi_{S_{k,j}}(s);~S_{k,j}=\{t|\lambda^{-q_{N+j}}< dist\{t, G_k\}\leq\lambda^{-q_{N+j-1}}\}$$
and
$$H(t,k)=\frac{2t-(t^k_++t^k_-)}{\sqrt{|t-t^k_+|\cdot |t-t^k_-|}}.$$

To get the absolutely H\"older continuity, we hope to apply \eqref{jishujiangjie} to prove \begin{equation}\label{hope}\lambda^{-\frac{c}{10}\mathcal{N}^{\sigma}_{n_0}}\ll 3C_k\int_{a}^b\vert H(k,s)\vert ds+\int_{a}^b\Xi(k,s) ds.\end{equation}

Unfortunately, \textbf{\eqref{hope} is not always valid if $t^k_+,t^k_-\notin (a,b).$} The reason is as follows.

Without loss of generality, we assume
$t^k_+\notin (a,b)$ and $~a>t^k_+.$

If $t_1<a\leq t_0=b,$ we have
$\max\limits_{t\in (a,b)}\vert H(k,t)\vert\leq \vert H(k,b)\vert$
and
$$\int_{a}^b\Xi(k,s)ds\leq \int_{a}^{t_0(=b)}\Xi(k,s)ds\leq \int_{a}^{t_0}\sup_{s\in (a,t_0)}\Xi(k,s)ds\le \int_{a}^{t_0}\sup_{s\in (t_1,t_0)}\Xi(k,s)ds\leq \lambda^{\left(\log(\mathcal{N}_{n_0})\right)^{C}}|a-b|.$$

Therefore
$$\left\vert 3C_k\int_{a}^b\vert H(k,s)\vert ds+\int_{a}^b\Xi(k,s) ds\right\vert\leq \left(3C_k\vert H(k,b)\vert+\lambda^{\left(\log(\mathcal{N}_{n_0})\right)^{C}}\right)|a-b|.$$

Note that both $3C_k\vert H(k,b)\vert+\lambda^{\left(\log(\mathcal{N}_{n_0})\right)^{C}}$ and $\lambda^{-\frac{c}{10}\mathcal{N}^{\sigma}_{n_0}}$ are independent on $a.$
Hence if $$|a-b|\ll \frac{\lambda^{-\frac{c}{10}\mathcal{N}^{\sigma}_{n_0}}}{\left(3C_k\vert H(k,b)\vert+\lambda^{\left(\log(\mathcal{N}_{n_0})\right)^{C}}\right)},$$ \eqref{hope} is invalid.

For simplification, we denote this problem by $\textbf{Pr}_{k}.$ Here $k$ corresponds with some spectral gap $G_k.$
We will see that  $\textbf{Pr}_k$  is the most difficult point in the whole proof.

\

\noindent\textsf{ The idea to solve $\textbf{Pr}_k$}

\
Indeed, for fixed open interval $(a,b),$ instead of directly solving $\textbf{Pr}_k,$ we only need to seek another $k'\neq k$ such that $\textbf{Pr}_{k'}$ does not occur, which implies \eqref{hope} is valid for $k'.$ Hence the key point is to prove the existence of such $k'$.
\vskip 0.5cm

\noindent\textsf{The details of the proof}

Recall Theorem \ref{15} implies there exists $\mathcal{K}\subset \Z$ such that
$\mathbb{R}-\Sigma^{\lambda}=\bigcup_{k\in \mathcal{K}} G_k.$
Now we set \begin{equation}\label{def-ki}\mathcal{K}=\{K_i\}_{i\in \Z}~\text{with}~|K_{i}|\leq |K_{i+1}|.\end{equation}
\begin{lemma}\label{c3lemma} For $I=(a,b)\subset 2{\mathcal{B}}^k$ satisfying $I-\bigcup\limits_{|K_i|> k}\mathcal{B}^{K_i}\neq \emptyset,$ it holds that
 \begin{equation}\label{gap_uniform6'}\left\vert L(a)-L(b)\right\vert\leq 6C_k\int_{I}\vert H(k,t)\vert dt+\int_{I}\Xi(k,t) dt.\end{equation}
\end{lemma}
\begin{proof}
Recall $G_k=(t^k_-,t^k_+).$ The definition of $\mathcal{B}^k$ shows that $$2{\mathcal{B}}^k=(t^k_--2\lambda^{-q_{N+s(k)-1}},t^k_-+2\lambda^{-q_{N+s(k)-1}})\bigcup (t^k_+-2\lambda^{-q_{N+s(k)-1}},t^k_++2\lambda^{-q_{N+s(k)-1}}).$$
Thus $$2{\mathcal{B}}^k-G_k=(t^k_--2\lambda^{-q_{N+s(k)-1}},t^k_-]\bigcup [t^k_+,t^k_++2\lambda^{-q_{N+s(k)-1}}).$$

We consider the following two cases:
\textbf{Case A:} $I\bigcap \{t^k_-,t^k_+\}=\emptyset$ and
\textbf{Case B:} $I\bigcap \{t^k_-,t^k_+\}\neq \emptyset.$
\begin{itemize}
\item \textbf{Case A:} $I\bigcap \{t^k_-,t^k_+\}=\emptyset$

In this case, one of the followings holds true: $$I\subset (t^k_--2\lambda^{-q_{N+s(k)-1}},t^k_-),$$ $$I\subset (t^k_-,t^k_-+2\lambda^{-q_{N+s(k)-1}}),$$ $$I\subset (t^k_+-2\lambda^{-q_{N+s(k)-1}},t^k_+),$$  $$I\subset  (t^k_+,t^k_++2\lambda^{-q_{N+s(k)-1}}).$$ Since these four cases are quite similar,
without loss of generality, we only consider $$I=(a,b)\subset (t^k_+,t^k_++2\lambda^{-q_{N+s(k)-1}}).$$

Next we consider the following two subcases:

\

\textbf{Case A.1:}\ $|I|\geq \varsigma(I,k).$

\

With the definition \ref{def4.2} and taking $t^*=t^k_+$ and $t=b$ in \eqref{wolgti}, namely,
\begin{equation}\label{ti10}t_0=b,\quad t_1=\frac{t_0+t^k_+}{2},\quad t_2=\frac{t_1+t^k_+a}{2},\cdots,~t_s=\frac{t_{s-1}+t^k_+}{2},~\cdots,\end{equation}
we have $$b-t^k_+\leq 2\lambda^{-q_{N+n_0-1}}~with~some~n_0\geq s(k).$$
Hence $$d(I,k)=dist(G_k,I)=\frac{b+a}{2}-t^k_+<b-t^k_+\leq 2\lambda^{-q_{N+n_0-1}}.$$
Therefore by the assumption of Case A.1,
\begin{equation}\label{Ixj11}\begin{array}{ll}&|I|\geq \varsigma(I,k)= \lambda^{-|\log d(I,k)|^{\log |\log d(I,k)|}}\geq \lambda^{-|\log \lambda^{-q_{N+n_0-1}}|^{\log |\log \lambda^{-q_{N+n_0-1}}| }}
\\
\\&\geq \lambda^{-(C\cdot q_{N+n_0-1})^{C\log\log q_{N+n_0-1}}}\gg \lambda^{-\frac{c}{20}\lambda^{\frac{{\sigma}}{C}q_{N+n_0-1}}}= \lambda^{-\frac{c}{20}\mathcal{N}^{\sigma}_{n_0}}.\end{array}\end{equation}

On the other hand, since $C_k>(\log |k|)^{-C}$ and $|H(k,t)|=\frac{t-t^k_++t-t^k_-}{\sqrt{(t-t^k_+)(t-t^k_-)}}\geq 2~for~t\geq t^k_+(>t^k_-),$ we have
$$\int_{a}^b C_k|H(k,t)|dt\geq 2C_k|b-a|\geq c(\log |k|)^{-C}|I|.$$

Note that $$c(\log |k|)^{-C}\geq c(\log q^2_{N+n_0})^{-C}\geq c((2C)\log q_{N+n_0-1})^{-C}\gg \lambda^{-\frac{c}{20}\mathcal{N}^{\sigma}_{n_0}}.$$ Then \eqref{Ixj11} implies
\begin{equation}\label{abjf}\int_{a}^b C_k|H(k,t)|dt\gg \lambda^{-\frac{c}{10}\mathcal{N}^{\sigma}_{n_0}}.\end{equation}

The definition of $t_i$ guarantees $[t_{i+1},t_{i})\subset \mathcal{B}_i(t^k_+).$
Then by the help of \eqref{6.1.1} (for the case $I\subset G_k,$ we apply \eqref{6.1.2}), for $m=1,2,$ and $t'\in [t_{i+1},t_{i})$, it holds that \begin{equation}\label{t'''}\int^{t_{i}}_{t'}\left|\frac{d L_{m\cdot \mathcal{N}_{n_i}}(t)}{d t}\right| dt\leq  C_k\int_{t'}^{t_i}\left\vert H(k,t)\right\vert dt+\lambda^{\left(\log (m\cdot \mathcal{N}_{n_0})\right)^{C}}(t_i-t')=C_k\int_{t'}^{t_i}\left\vert H(k,t)\right\vert dt+\int_{t'}^{t_i}\Xi(k,t)dt.\end{equation}


\

Note there exists $P=P(a)\in \N$ such that $t_{P+1}\leq a\leq t_{P}\leq t_{P-1}\leq \cdots \leq b=t_0$ where $t_P$ is defined in \eqref{ti10}. Hence \eqref{t'''} implies

\begin{equation}\label{dapianc}\begin{array}{ll}&\int^{b}_{a}\left|\frac{d L_{m\cdot \mathcal{N}_{n_i}}(t)}{d t}\right| dt=\sum\limits_{i=0}^{P-1}\int^{t_{i}}_{t_{i+1}}\left|\frac{d L_{m\cdot \mathcal{N}_{n_i}}(t)}{d t}\right| dt+\int^{t_{P}}_{a}\left|\frac{d L_{m\cdot \mathcal{N}_{n_i}}(t)}{d t}\right| dt\\&\leq  C_k\cdot\int_{a}^{b}\left\vert H(k,t)\right\vert dt+\int_{a}^{b}\Xi(k,t)dt.\quad(~\text{note}~b=t_0)\end{array}\end{equation}

Note \eqref{universal*} implies
$$\vert L(b)-L(a))\vert
\leq \lambda^{-\frac{c}{10}\mathcal{N}^{\sigma}_{n_0}}+\int^{b}_{a} \left\vert\frac{d L_{\mathcal{N}_{n_0}}(t)}{d t}\right\vert dt+2\cdot\int^{b}_{a} \left\vert\frac{d L_{2\mathcal{N}_{n_0}}(t)}{d t}\right\vert dt.$$

Combining this with \eqref{abjf} and \eqref{dapianc}, we immediately get
$$\begin{array}{ll}&\vert L(b)-L(a))\vert \leq \lambda^{-\frac{c}{10}\mathcal{N}^{\sigma}_{n_0}}+3C_k\int_{I}\vert H(k,t)\vert dt+\int_{I}\Xi(k,t) dt\quad(by~\eqref{dapianc})\\
\\&\leq 6C_k\int_{I}\vert H(k,t)\vert dt+\int_{I}\Xi(k,t) dt\quad(by~\eqref{abjf}).\end{array}$$

\

\textbf{Case A.2:}~$|I|< \varsigma(I,k).$

\

Since $I=(a,b)-\left(\bigcup\limits_{|i|> |k|}{\mathcal{B}}^{i}\right)\not=\emptyset$ and $I\subset 2\mathcal{B}^k$  there exists $\hat{t}\in (a,b)$ such that $\hat{t}\in  2\mathcal{B}^k- \left(\bigcup\limits_{|i|> |k|}{\mathcal{B}}^{i}\right).$

     Now we apply the definition \ref{def4.2} and set $t^*=a$ and $t=b$ in \eqref{wolgti}. Then $$b-t^k_+\geq 2\lambda^{-q_{N+n_0}};$$
     $$d(I,k)=dist(G_k,I)=(\frac{b+a}{2}-t^k_+)= \frac{a-t^k_+}{2}+\frac{b-t^k_+}{2}>\frac{b-t^k_+}{2}\geq \lambda^{-q_{N+n_0}}.$$

     Note the fact $\hat{t}\in 2\mathcal{B}_k$ implies $n_0\geq s(k).$
     Hence
     \begin{equation}\label{Ixj1}\begin{array}{ll}&\frac{b-a}{2}=\frac{1}{2}|I|\leq \frac{1}{2}\varsigma(I,k)= \frac{1}{2}\lambda^{-|\log d(I,k)|^{\log |\log d(I,k)|}}\leq \lambda^{-|\log \lambda^{-q_{N+n_0}}|^{\log |\log \lambda^{-q_{N+n_0}}| }}\\\\&\leq \lambda^{-(c\cdot q_{N+n_0-1})^{c\log\log q_{N+n_0-1}}}
     \ll \lambda^{-q_{N+n_0+10}}
     \leq \lambda^{-q_{N+s(k)+10}}.\end{array}\end{equation}

\eqref{Ixj1} clearly implies that $$(a,b)\subset(\tilde{t}-\frac{1}{2}\lambda^{-q_{N+s(k)+4}},\tilde{t}+\frac{1}{2}\lambda^{-q_{N+s(k)+4}})=\bigcup\limits_{l\geq s(k)+3}\frac{1}{2}\mathcal{B}_{l}(\tilde{t}).$$

     Then by \eqref{differ*}, for any $t\in I=(a,b)$ it holds that
     $$\begin{array}{ll}\left|L(t)-L(\hat{t})\right|&\leq 4C_{k}\int^{t}_{\hat{t}}\vert H(k,t)\vert dt+ 4\lambda^{(\log(\mathcal{N}_{s(k)}))^{C^*}}|t-\hat{t}|\\&\leq 4C_{k}\int^{t}_{\hat{t}}\vert H(k,t)\vert dt+ \int^{t}_{\hat{t}}\Xi(k,t)dt.\end{array}$$

\item \textbf{Case B:}

If $I\bigcap \{t^k_-,t^k_+\}\neq \emptyset,$ then $I-\{t^k_-,t^k_+\}$ consists of at most $3$ open intervals as follows:
$$I_1=(t^k_+,b),\quad I_2=(t^k_-,t^k_+),\quad I_3=(a,t^k_-).$$

Clearly, the definition implies \begin{equation}\label{casea}I_j\bigcap \{t^k_-,t^k_+\}=\emptyset.\end{equation}
Next we will check that
\begin{equation}\label{123}|I_j|\geq \varsigma(I_j,k),~j=1,2,3.\end{equation}

We only check $|I_1|\geq \varsigma(I_1,k)$ and the left two are similar.

Setting $t^*=t^k_+$ and $t=b,$ with the definition \ref{ti}, we have $$b-t^k_+\geq 2\lambda^{-q_{N+n_0}}~with~some~n_0\geq s(k).$$
 Recall $~d(I_1,k)=dist\{G_k,\frac{t^k_++b}{2}\}=\frac{t^k_++b}{2}-t^k_+$. Hence
$$\begin{array}{ll}&\varsigma(I_1,k)= \lambda^{-|\log d(I_1,k)|^{\log |\log d(I_1,k)|}}\leq \lambda^{-|\log \lambda^{-q_{N+n_0}}|^{\log |\log \lambda^{-q_{N+n_0}}| }}
\\
\\&\leq \lambda^{-(C\cdot q_{N+n_0})^{C\log\log q_{N+n_0}}}\ll \lambda^{-q_{N+n_0}}\leq \frac{b-t^k_+}{2}=d(I_1,k)\le  2|I_1|.\end{array}$$

Then similar as in {\bf Case A.1}, \eqref{casea} and \eqref{123} together imply
$$\left\vert L(b)-L(t^k_+)\right\vert\leq 6C_k\int_{I_1}\vert H(k,t)\vert dt+\int_{I_1}\Xi(k,t) dt.$$

Similarly,
$$\left\vert L(t^k_-)-L(t^k_+)\right\vert\leq 6C_k\int_{I_2}\vert H(k,t)\vert dt+\int_{I_2}\Xi(k,t) dt.$$
$$\left\vert L(a)-L(t^k_-)\right\vert\leq 6C_k\int_{I_3}\vert H(k,t)\vert dt+\int_{I_3}\Xi(k,t) dt.$$

Therefore
$$\begin{array}{ll}&\left\vert L(b)-L(a)\right\vert \leq \left\vert L(b)-L(t^k_+)\right\vert+\left\vert L(t^k_-)-L(t^k_+)\right\vert+\left\vert L(a)-L(t^k_-)\right\vert\\&\leq 6C_k\int_{I_1\bigcup I_2 \bigcup I_3}\vert H(k,t)\vert dt+\int_{I_1\bigcup I_2 \bigcup I_3}\Xi(k,t) dt=6C_k\int_{a}^b\vert H(k,t)\vert dt+\int_{a}^b\Xi(k,t) dt.\hfill\qed\end{array}$$
\end{itemize}
\end{proof}

\

\begin{lemma}\label{c1lemma}For any pairwise~disjointed intervals $J_i=(a_i,b_i)\subset 2{\mathcal{B}}^k,\ i\ge 1$ satisfying $$J_i-\bigcup\limits_{|j|> |k|}\mathcal{B}^{j}\neq \emptyset,$$ it holds that

\begin{equation}\label{gap_uniform5}\sum\limits_{i}\left\vert L(a_i)-L(b_i)\right\vert\leq C\lambda^{-q^{\frac{1}{4}}_{N+s(k)-1}}\left(\sum\limits_{i}|a_i-b_i|\right)^{\frac{1}{2}}.\end{equation}
\end{lemma}
\begin{proof}
Recall $$|G_k|\leq C\lambda^{-ck}\leq C\lambda^{-cq^2_{N+s(k)-1}}\leq C\lambda^{-cq_{N+s(k)-1}}.$$
Since $J_i\bigcap J_j=\emptyset$ and $J_i\subset 2{\mathcal{B}}^k,$ we obtain
$$\begin{array}{ll}\sum\limits_{i}|a_i-b_i|&\leq 2|\mathcal{B}^k|=4\lambda^{-q_{N+s(k)-1}}+2|G_k|\leq 6C\lambda^{-cq_{N+s(k)-1}}.\end{array}$$
Then \eqref{jfbd1} and \eqref{jfbd2} of Lemma \ref{jiandanjf1} show $$\begin{array}{ll}\int_{\bigcup\limits_{i}(a_i,b_i)} |H(k,t)| dt&\leq 4\sqrt{|G_k|+\sum\limits_{i}|a_i-b_i|}\sqrt{\sum\limits_{i}|a_i-b_i|}\leq 12C\lambda^{-\frac{1}{2}cq_{N+s(k)-1}}\sqrt{\sum\limits_{i}|a_i-b_i|}\end{array};$$

$$\begin{array}{ll}\int_{\bigcup\limits_{i}(a_i,b_i)} \Xi(k,t) dt &\leq 4(\sum\limits_{i}|a_i-b_i|)^{1-2\xi_k}<4(\sum\limits_{i}|a_i-b_i|)^{\frac{1}{6}}(\sum\limits_{i}|a_i-b_i|)^{\frac{1}{2}}.\end{array}$$
By the help of Lemma \ref{c3lemma}, it holds that
 $$\begin{array}{ll}\sum\limits_{i}\left\vert L(a_i)-L(b_i)\right\vert&\leq 6C_k\sum\limits_{i}\int_{J_i}\vert H(k,t)\vert dt+\sum\limits_i\int_{J_i}\Xi(k,t) dt.\end{array}$$
 Thus \eqref{gap_uniform5} immediately follows from  $C_k\leq (\log |k|)^C\ll \lambda^{\frac 12 q^{\frac{1}{4}}_{N+s(k)-1}}.$
\hfill\qed\end{proof}

On the other hand, we have the following Lemma.

\begin{lemma}\label{homo} Let $\mathcal{K}$ be defined in \eqref{def-ki}. Given a sequence $\{K_{i_j}\}_j\subset \mathcal{K}$ with $|K_{i_1}|\leq |K_{i_2}|\leq |K_{i_3}|\leq \cdots,$ if ${\bigcup\limits_{j\geq 1} \mathcal{B}^{K_{i_j}}}$ is connected (i.e. an open interval), then it holds that
\begin{equation}\label{ki12}|K_{i_2}|>|K_{i_1}|\end{equation}
and
\begin{equation}\label{kij}{\bigcup\limits_{j\geq 1} \mathcal{B}^{K_{i_j}}} \subseteq 2\mathcal{B}^{K_{i_1}}.\end{equation}
\end{lemma}
\begin{proof}

\textbf{The proof of \eqref{ki12}:}

First, we have to rule out the case $$|K_{i_1}|=|K_{i_2}|\quad (~i.e.~K_{i_2}=-K_{i_1}).$$

Otherwise, recall \eqref{dgkij} implies \begin{equation}\label{k+-}dist(G_{K_{i_1}},G_{-K_{i_1}})\geq \lambda^{-|K_{i_1}|^c}\gg 10|\mathcal{B}^{K_{i_1}}|+10|\mathcal{B}^{-K_{i_1}}|.\end{equation}
Without loss of generality, we assume $G_{K_{i_1}}<G_{-K_{i_1}},$ i.e. $t^{K_{i_1}}_+<t^{-K_{i_1}}_-$. Since $G_{K_{i_1}},G_{-K_{i_1}}\subset \bigcup\limits_{j\geq 1}\mathcal{B}^{K_{i_j}}$ with $\bigcup\limits_{j\geq 1}\mathcal{B}^{K_{i_j}}$ being connected,  we have
$$(t^{K_{i_1}}_+,t^{-K_{i_1}}_-)\subset \bigcup\limits_{j\geq 1}\mathcal{B}^{K_{i_j}}.$$
Note \eqref{k+-} implies $$\vert t^{K_{i_1}}_+-t^{-K_{i_1}}_-\vert \gg 10|\mathcal{B}^{K_{i_1}}|+10|\mathcal{B}^{-K_{i_1}}|.$$
Therefore
\begin{equation}\label{baoh1}(t^{K_{i_1}}_+,t^{K_{i_1}}_++8|\mathcal{B}^{K_{i_1}}|)\subset(t^{K_{i_1}}_+,t^{-K_{i_1}}_-)\subset \bigcup\limits_{j\geq 1}\mathcal{B}^{K_{i_j}}.\end{equation}
Since $$(t^{K_{i_1}}_+,t^{K_{i_1}}_++8|\mathcal{B}^{K_{i_1}}|)\bigcap \left(\mathcal{B}^{-K_{i_1}}(=\mathcal{B}^{K_{i_2}})\right)=\emptyset,$$ \eqref{baoh1} implies
$$(t^{K_{i_1}}_+,t^{K_{i_1}}_++8|\mathcal{B}^{K_{i_1}}|)\subset \bigcup\limits_{j\neq 2}\mathcal{B}^{K_{i_j}}.$$
We rewrite $$\{{j\neq 2|\mathcal{B}^{K_{i_j}}\bigcap (t^{K_{i_1}}_+,t^{K_{i_1}}_++8|\mathcal{B}^{K_{i_1}}|)\neq \emptyset}\}=\{j_s\}_{s\in \Z_+},\quad |K_{j_s}|\le |K_{j_{s+1}}|,~s\in \Z_+.$$
Therefore
$$(t^{K_{i_1}}_+,t^{K_{i_1}}_++8|\mathcal{B}^{K_{i_1}}|)\subset \bigcup\limits_{\{j_s\}_{s\in \Z_+}}\mathcal{B}^{K_{i_{j_s}}}.$$

Since \begin{equation}\label{minmin220}\{K_{i_j}\vert j\neq 2\}=\{K_{i_1},K_{i_3},\cdots\}\end{equation} satisfies $$|K_{i_1}|<|K_{i_3}|\leq |K_{i_4}|\leq \cdots,$$ we have
$\min\left(\{j_s\}_{s\in \Z_+}-\{1\}\right)\geq 3.$
Therefore $$(t^{K_{i_1}}_+,t^{K_{i_1}}_++8|\mathcal{B}^{K_{i_1}}|)\subset \left(\bigcup\limits_{s\geq 1}\mathcal{B}^{K_{i_{j_s}}}\right)\subset \left(\bigcup\limits_{j\geq 1}\mathcal{B}^{K_{i_{1}+j}}\right).$$
Then \eqref{pugapguj} implies
$$8|\mathcal{B}^{K_{i_1}}|\leq \left\vert\left(\bigcup\limits_{j\geq 1}\mathcal{B}^{K_{i_{1}+j}}\right)\right\vert\leq 3|\mathcal{B}^{K_{i_1}}|$$
leading to a contradiction.

\

\textbf{The proof of \eqref{kij}:}

\eqref{ki12} implies
$$|K_{i_1}|<|K_{i_2}|\leq |K_{i_3}|\leq \cdots.$$

If $$\lambda^{\left(|K_{i_1}|\right)^{\frac{1}{4C}}}+|K_{i_1}|>|K_{i_j}|> |K_{i_1}|,$$ then  (1) of Lemma \ref{zblemma} and \eqref{wolgg} imply
$$\begin{array}{ll} &dist(G_{K_{i_1}},G_{K_{i_j}})>\lambda^{-|K_{i_1}|^{\frac{1}{4C}}}>\lambda^{-|q_{N+s(K_{i_1})}^{2}|^{\frac{1}{4C}}}\geq \lambda^{-q_{N+s(K_{i_1})-1}^{\frac{1}{2}}}\\&\gg 6\lambda^{-q_{N+s(K_{i_1})-1}}+6\lambda^{-q_{N+s(K_{i_j})-1}}\geq 2\left\vert\mathcal{B}^{K_{i_1}}\right\vert+2\left\vert\mathcal{B}^{K_{i_j}}\right\vert.\end{array}$$

Therefore $$\min\{|K_{i_j}|>|K_{i_1}| \vert \mathcal{B}^{K_{i_j}}\bigcap 2\mathcal{B}^{K_{i_1}}\neq \emptyset\}\geq \lambda^{c\left(|K_{i_1}|\right)^{\frac{1}{4C}}}.$$
Recall $q^2_{N+s(K_i)-1}\leq |K_i|\leq q^2_{N+s(K_i)}$. Then
\begin{equation}\label{ll}\begin{array}{ll}&\sum\limits_{|K_{i_j}|>|K_{i_1}|, \mathcal{B}^{K_{i_j}}\bigcap 2\mathcal{B}^{K_{i_1}}\neq \emptyset}|\mathcal{B}^{K_{i_j}}|\leq \sum\limits_{|K_i|\geq \lambda^{\left(|K_{i_1}|\right)^{\frac{1}{4C}}}}|\mathcal{B}^{K_i}|\leq \sum\limits_{|K_i|\geq \lambda^{\left(|K_{i_1}|\right)^{\frac{1}{4C}}}}\lambda^{-q_{N+s(K_{i})-1}}
\\
\\&\leq \sum\limits_{|K_i|\geq \lambda^{\left(|K_{i_1}|\right)^{\frac{1}{4C}}}}\lambda^{-|K_i|^c}\ll \lambda^{-|K_{i_1}|^C}\leq |\mathcal{B}^{K_{i_1}}|.\end{array}\end{equation}

Now we suppose that \begin{equation}\label{jias}\bigcup\limits_{j\geq 1} \mathcal{B}^{K_{i_j}}-2\mathcal{B}^{K_{i_1}}\neq \emptyset.\end{equation}
Without loss of generality, we assume there exists some $t^*<t_-^{K_{i_1}}$ such that $t^*\in \bigcup\limits_{j\geq 1} \mathcal{B}^{K_{i_j}}-2\mathcal{B}^{K_{i_1}}.$
Then since ${\bigcup\limits_{j\geq 1} \mathcal{B}^{K_{i_j}}}$ is connected and $t^*,t^{K_{i_1}}_-\in {\bigcup\limits_{j\geq 1} \mathcal{B}^{K_{i_j}}}$, we have $$(t^{K_{i_1}}_--|\mathcal{B}^{K_{i_1}}|,t^{K_{i_1}}_--\frac 12|\mathcal{B}^{K_{i_1}}|)\subset(t^{K_{i_1}}_--|\mathcal{B}^{K_{i_1}}|,t^{K_{i_1}}_-)\subset (t^*,t^{K_{i_1}}_-)\subset{\bigcup\limits_{j\geq 1} \mathcal{B}^{K_{i_j}}}= {\bigcup\limits_{|K_{i_j}|>|K_{i_1}|} \mathcal{B}^{K_{i_j}}}.$$

Note if $(a,b)\subset \bigcup\limits_{i} J_i$, where each $J_i$ is an open interval, then $(a,b)\subset \bigcup\limits_{i\in \{i \vert J_i\bigcap(a,b)\neq \emptyset\}} J_i.$
Therefore
$$\begin{array}{ll}&(t^{K_{i_1}}_--|\mathcal{B}^{K_{i_1}}|,t^{K_{i_1}}_--\frac12|\mathcal{B}^{K_{i_1}}|)\subset \bigcup\limits_{|K_{i_j}|>|K_{i_1}|, ~(t^{K_{i_1}}_--|\mathcal{B}^{K_{i_1}}|,t^{K_{i_1}}_--\frac12|\mathcal{B}^{K_{i_1}}|)\bigcap \mathcal{B}^{K_{i_j}}\neq \emptyset}\mathcal{B}^{K_{i_j}}\\& \subset\bigcup\limits_{|K_{i_j}|>|K_{i_1}|, ~\left(2{\mathcal{B}}^{K_{i_1}}\right)\bigcap \mathcal{B}^{K_{i_j}}\neq \emptyset}\mathcal{B}^{K_{i_j}}\end{array},$$
which yields
\begin{equation}\label{ll*}|\mathcal{B}^{K_{i_1}}|=Leb\{(t^{K_{i_1}}_--|\mathcal{B}^{K_{i_1}}|,t^{K_{i_1}}_--\frac{1}{2}|\mathcal{B}^{K_{i_1}}|)\}\leq \left\vert\bigcup\limits_{|K_{i_j}|>|K_{i_1}|, ~2\mathcal{B}^{K_{i_1}}\bigcap \mathcal{B}^{K_{i_j}}\neq \emptyset}\mathcal{B}^{K_{i_j}}\right\vert.\end{equation}
Then \eqref{ll} and \eqref{ll*} imply
$$|\mathcal{B}^{K_{i_1}}|\leq \left\vert\bigcup\limits_{|K_{i_j}|>|K_{i_1}|, ~2\mathcal{B}^{K_{i_1}}\bigcap \mathcal{B}^{K_{i_j}}\neq \emptyset}\mathcal{B}^{K_{i_j}}\right\vert\ll |\mathcal{B}^{K_{i_1}}|,$$ which leads to a contradiction.
Therefore \eqref{jias} is invalid and we finish the proof.
\hfill\qed\end{proof}

\begin{lemma}\label{homg} Let $\mathcal{K},K_i$ be defined in \eqref{def-ki}. Given $k\in \mathcal{K}$ and an open interval $I\subset \R$, if $I\subset \bigcup\limits_{|K_i|\geq |k|}{\mathcal{B}}^{K_i},$
then there exists $K_{i^*}\in \mathcal{K}$ with $|K_{i^*}|\geq |k|$ such that
    $I\subset 2\mathcal{B}^{K_{i^*}}\quad {\rm and}\quad  I-\bigcup\limits_{|j|> |K_{i^*}|}{\mathcal{B}}^j\neq \emptyset.$
\end{lemma}
\begin{proof}

First, we have the following fact.

\

\textbf{The fact:}~If $I\subset \bigcup\limits_{K_i\in \mathcal{K}}{\mathcal{B}}^{K_i},$ then there exists $K_{i'}\in \mathcal{K}$ such that
$I\subset 2\mathcal{B}^{K_{i'}}.$

\

\textbf{The proof of the fact:}

Since $\bigcup\limits_{K_i \in \mathcal{K}}{\mathcal{B}}^{K_i}$ is open, all  connected components of it correspond with  subsets of $\mathcal{K}$, denoted by $\{ K_{i^j_k}\}_{k=1}^{s_j},~j=1,2,\cdots,S,$  such that $$\left(\bigcup\limits_{m=1}^{s_p} \mathcal{B}^{K_{i^p_m}}\right)~is~connected~and~\left(\bigcup\limits_{m=1}^{s_p} \mathcal{B}^{K_{i^p_m}}\right)\bigcap \left(\bigcup\limits_{m=1}^{s_q} \mathcal{B}^{K_{i^q_m}}\right)=\emptyset,~1\leq p\not=q\leq S.$$
$$|K_{i^j_{m}}|\leq |K_{i^j_{m+1}}|,~m=1,2,\cdots,s_j,~j=1,2,\cdots,S.$$

\

Hence we have $$I\subset \bigcup\limits_{K_i\in \mathcal{K}}{\mathcal{B}}^{K_i}=\bigcup\limits_{p=1}^S\left(\bigcup\limits_{m=1}^{s_p} \mathcal{B}^{K_{i^p_m}}\right).$$
Since $I$ is connected, there exists some $1\leq p^*\leq S$ such that
$I\subset \left(\bigcup\limits_{m=1}^{s_{p^*}} \mathcal{B}^{K_{i^{p^*}_m}}\right).$

By the help of \eqref{kij} of Lemma \ref{homo},
$\left(\bigcup\limits_{m=1}^{s_{p^*}} \mathcal{B}^{K_{i^{p^*}_m}}\right)\subset 2\mathcal{B}^{K_{i^{p^*}_1}}.$
Therefore $I\subset 2\mathcal{B}^{K_{i^{p^*}_1}},$ which completes the proof of the {\bf fact}.

\vskip 0.2cm
Thus we have
$\{K_i\subset \mathcal{K} \vert I\subset 2\mathcal{B}^{K_i}\}\neq \emptyset.$
Since $|\mathcal{B}^{K_i}|\rightarrow 0,~as~|K_i|\rightarrow+\infty,$
we have
\begin{equation}\label{zuid}|K_{i^{**}}|:=\max\{|K_i| \vert I\subset 2\mathcal{B}^{K_i}\}<+\infty.\end{equation}

To end  the proof of the lemma, we only need to prove
$$I-\bigcup\limits_{|j|> |K_{i^{**}}|}{\mathcal{B}}^j\neq \emptyset.$$

In fact, set $M=\min \{|K_i| \vert |K_i|>|K_{i^{**}}|\}.$
If $I-\bigcup\limits_{|j|> |k_{i^{**}}|}{\mathcal{B}}^j= \emptyset,$ then $I\subset \bigcup\limits_{|K_i|> |K_{i^{**}}|}{\mathcal{B}}^{K_i}=\bigcup\limits_{|K_i|\geq M}{\mathcal{B}}^{K_i}.$ Consequently, the {\bf fact} implies there exists some $k_{i^{***}}\in \{K_i \vert |K_{i^{***}}|\geq M\}$ such that
$I \subset 2\mathcal{B}^{K_{i^{***}}}.$ Therefore

$$|K_{i^{**}}|<M\leq |K_{i^{***}}|\in \{|K_i| \vert I\subset 2\mathcal{B}^{K_i}\},$$ which contradicts \eqref{zuid}.
\hfill\qed\end{proof}

\

The following result directly follows from Lemma \ref{homg}.
\begin{corollary}\label{generalca} If $I\subset\bigcup\limits_{K_i\in \mathcal{K}}{\mathcal{B}}^{K_i},$ then there exists some $K_{i^*}\in \mathcal{K}$ such that $I\subset 2{\mathcal{B}}^{K_{i^*}}$ and $~I-\bigcup\limits_{|K_i|> |K_{i^*}|}{\mathcal{B}}^{K_i}\neq \emptyset.$
\end{corollary}

The following conclusion follows from Corollary \ref{generalca}.
\begin{lemma}\label{c2lemma} Given a sequence of pairwise disjointed open intervals $\{(a_i,b_i)\}_{i\in \Omega}\subset \bigcup\limits_{i\in \Z_+}{\mathcal{B}}^{K_i}$ with $a_i,b_i\in \R$ and $\Omega$ being some index set, it holds that
$$ \sum\limits_{i}|L(a_i)-L(b_i)| \leq C\left(\sum\limits_{i}|a_i-b_i|\right)^{\frac{1}{2}}.
$$
\end{lemma}

\begin{proof} It follows from Corollary \ref{generalca} that there exists some $K_s\in \{k_i\}_{i\in \Z}$ such that $$(a_i,b_i)\subset 2\mathcal{B}^{K_s}\quad {\rm and }\quad(a_i,b_i)-\bigcup\limits_{|K_i|>K_s}\mathcal{B}^{K_i}\ne \emptyset.$$ For this situation, we say $(a_i,b_i)$ satisfies condition $\mathcal{L}(K_s).$

For any fixed $s\geq 1,$ we denote $\mathcal{I}_{s}:=\{i \vert (a_i,b_i)~\text{satisfies~condition}~\mathcal{L}(K_s)\}.$
Then $ \Omega\subset\bigcup\limits_{j\in \Z} \mathcal{I}_{j}.$
Note Lemma \ref{c1lemma} implies for $j\geq 1$ we have
$$\sum\limits_{i\in \mathcal{I}_j}|L(a_i)-L(b_i)| \leq C\lambda^{-q^{\frac{1}{4}}_{N+s(K_j)-1}}\left(\sum\limits_{i\in \mathcal{I}_j}|a_i-b_i|\right)^{\frac{1}{2}}.$$

Therefore by Cauchy-Schwartz inequality and the fact $$\sum\limits_{|k|=0}^{\infty}\lambda^{-2q^{\frac{1}{4}}_{N+s(k)-1}}<C<+\infty,$$ we obtain

$$\begin{array}{ll}&\sum\limits_{i\in \Omega}|L(a_i)-L(b_i)|\leq \sum\limits_{j\geq 1}\sum\limits_{i\in \mathcal{I}_j}|L(a_i)-L(b_i)| \leq C(\sum\limits_{j\geq 1}\lambda^{-2q^{\frac{1}{4}}_{N+s(k_j)-1}})^{\frac{1}{2}}\left(\sum\limits_{i}|a_i-b_i|\right)^{\frac{1}{2}}\leq 2C\left(\sum\limits_{i}|a_i-b_i|\right)^{\frac{1}{2}}.\hfill\qed\end{array}$$
\end{proof}

\

\subsubsection{\text{Absolutely $\frac{1}{2}$-H\"older continuity}}
\noindent

\begin{lemma}\label{ab11}

For $t_0\in \R-{\left(\bigcup\limits_{K_i\in \mathcal{K}}{\mathcal{B}}^{K_i}\right)},$ it holds that $$\limsup_{t\rightarrow t_0}\frac{|L(t)-L(t_0)|}{|t-t_0|}\leq C.$$
\end{lemma}
\begin{proof}By the help of \eqref{lipschitz-1}, for $t_0\in {\R-\left(\bigcup\limits_{K_i\in \mathcal{K}}{\mathcal{B}}^{K_i}\right)},$ which implies $k_{last}(t_0)=0$ and $j_{last}(t_0)=1,$ we have
$$\limsup\limits_{t\rightarrow t_0}\frac{|L(t)-L(t_0)|}{|t-t_0|}\leq C\lambda^{-q^{\frac{1}{4}}_{N-1}}(dist\{t_0,G_{0}\})^{-\frac{1}{2}}.$$
Note $0\in \mathcal{K}$ implies $ t_0\notin \mathcal{B}^0.$ Therefore $dist\{t_0,G_{0}\}\geq c.$
Finally, $\lambda^{-q^{\frac{1}{4}}_{N-1}}(dist\{t_0,G_{0}\})^{-\frac{1}{2}}$ is independent on $t_0$ as desired.
\hfill\qed\end{proof}
Recall that we say a function $F$ defined on an open interval $B$ is absolutely $\frac{1}{2}-H\ddot{o}lder$ continuous, if there exists $C_1>0$ such that for any $N\in \Z_+$ and pair-wise disjoint intervals $(a_i,b_i)\subset B,\ 1\le i\le N$,  it holds
$$\sum\limits_{i=1}^N|F(a_i)-F(b_i)|\leq C_1\left(\sum\limits_{i=1}^N|a_i-b_i|\right)^{\frac{1}{2}}.$$We have the following result.

\begin{lemma}\label{ab12} Given $M>0,~P\in \Z_+$ and $F\in C^0([-M,M]).$ Let $B=\{(a_i,b_i)\}_{i=1}^P\subset [-M,M],$ if $$(\mathbf{A}):F~is~absolutely~\frac{1}{2}-H\ddot{o}lder~continuous~in~B;$$

$$(\mathbf{B}):\limsup\limits_{t\rightarrow t_0}\frac{|F(t)-F(t_0)|}{|t-t_0|}<+\infty,~{\rm \ for\ any\ }~t_0\in [-M,M]-B,$$

then $$F~is~absolutely~\frac{1}{2}-H\ddot{o}lder~continuous~in~[-M,M].$$
\end{lemma}

\begin{proof}
Given a sequence open intervals $\{(a_i,b_i)\}_{i=1}^P\subset [-M,M]$ with $P\in \Z_+,$ since $F\in C^0([-M,M]),$ there exists $\delta^*>0$ such that
\begin{equation}\label{yizhi}\vert F(x)-F(y)\vert<\frac{C_1}{2P}\left(\sum\limits_{i=1}^P(b_i-a_i)\right)^{\frac{1}{2}},~x,y\in [-M,M]~with~|x-y|<3\delta^*.\end{equation}

 The condition ($\mathbf{B}$) implies for any $t\in [-M,M]-B$ and $\delta^*\gg \epsilon^*>0,$ there exists \begin{equation}\label{tij}0<\delta(t,\epsilon^*)<\epsilon^*(\ll \delta^*)\end{equation} such that $$\sup\limits_{t'\in(t-\delta(t,\epsilon^*),t+\delta(t,\epsilon^*))}\frac{|F(t)-F(t')|}{|t-t'|}\leq 2C(t),$$ with $$C(t):=\limsup\limits_{x\rightarrow t}\frac{|F(x)-F(t)|}{|x-t|}.$$

Since $[-M,M]-B$ is a compact set, for $\epsilon^*>0,$ there exists $S(\epsilon^*)\in \Z_+$ and a sequence $t_1,t_2,\cdots,t_S$ such that
$[-M,M]-B\subset\bigcup\limits_{j=1}^{S(\epsilon^*)}(t_j-\delta(t_j,\epsilon^*),t_j+\delta(t_j,\epsilon^*)).$

We denote
$C_2=\max\limits_{1\leq j\leq S(\epsilon^*)}C(t_j).$ Then for each $1\leq j\leq S(\epsilon^*),$ it holds that
\begin{equation}\label{supa}\sup\limits_{t'\in(t_j-\delta(t_j,\epsilon^*),t_j+\delta(t_j,\epsilon^*))}\frac{|F(t_j)-F(t')|}{|t_j-t'|}<2C_2.\end{equation}

Note $\bigcup\limits_{j=1}^{S(\epsilon^*)}(t_j-\delta(t_j,\epsilon^*),t_j+\delta(t_j,\epsilon^*))$ has finitely many
 connected components.
Hence we can write $$\bigcup\limits_{j=1}^{S(\epsilon^*)}(t_j-\delta(t_j,\epsilon^*),t_j+\delta(t_j,\epsilon^*))=\bigcup\limits_{j=1}^{\tilde{S}(\epsilon^*)}(c_i,d_i)$$ with some $1\leq \tilde{S}(\epsilon^*)\leq S(\epsilon^*)$ and $(c_i,d_i)\bigcap (c_j,d_j)=\emptyset$ for any $1\leq i<j\leq \tilde{S}(\epsilon^*),$ where each $(c_i,d_i)$ corresponds with a sequence (at least two terms)$\{t_{l^i_l}\}_{l=1}^{s_i}\subset  \{t_i\}_{i=1}^{S(\epsilon^*)}$  satisfying $t_{l^i_1}<t_{l^i_{2}}<\cdots<t_{l^i_{s_i}}$ and \begin{equation}\label{cdi}(c_i,d_i)=\bigcup\limits_{j=1}^{s_i}(t_{l^i_j}-\delta(t_{l^i_j},\epsilon^*),t_{l^i_j}+\delta(t_{l^i_j},\epsilon^*)).\end{equation}

We claim that for $1\leq i\leq \tilde{S}(\epsilon^*)$, and $\text{for}~x\in [t_{l^i_k}-\delta(t_{l^i_k},\epsilon^*),t_{l^i_k}],~y\in[t_{l^i_{m}},d_i+\delta(t_{l^i_m},\epsilon^*)],~1\leq k\leq m\leq s_i,$ it holds that

\begin{equation}\label{zhonx*} \vert F(x)-F(y) \vert\leq 2C_2(x-y),~.~
\
\end{equation}

\

\textbf{The proof of \eqref{zhonx*}}:

\begin{proof}

First we prove
\begin{equation}\label{zhonx} \vert F(t_{l^i_j})-F(t_{l^i_{j+1}}) \vert\leq 2C_2|t_{l^i_{j+1}}-t_{l^i_{j}}|.
\end{equation}

Without~loss~of~generality, we assume $\delta(t_{l^i_j},\epsilon^*)\ge \delta(t_{l^i_{j+1}},\epsilon^*)$.

If $t_{l^i_{j+1}}-t_{l^i_j}\leq \delta(t_{l^i_j},\epsilon^*),$ then
$t_{l^i_{j+1}}\in (t_{l^i_j}-\delta(t_{l^i_j},\epsilon^*),t_{l^i_j}+\delta(t_{l^i_j},\epsilon^*)).$
Hence \eqref{supa} implies \eqref{zhonx}.

If $t_{l^i_{j+1}}-t_{l^i_j}> \delta(t_{l^i_j},\epsilon^*),$
then $t_{l^i_{j+1}}-\delta(t_{l^i_{j+1}},\epsilon^*)>t_{l^i_{j}};~t_{l^i_{j}}+\delta(t_{l^i_{j}},\epsilon^*)<t_{l^i_{j+1}}.$

Note that \eqref{cdi} implies $$(t_{l^i_{j}}-\delta(t_{l^i_{j}},\epsilon^*),t_{l^i_{j}}+\delta(t_{l^i_{j}},\epsilon^*))\bigcap(t_{l^i_{j+1}}-\delta(t_{l^i_{j+1}},\epsilon^*),t_{l^i_{j+1}}+\delta(t_{l^i_{j+1}},\epsilon^*))\neq \emptyset,$$ which implies
\begin{equation}\label{gc1} \left(t_{l^i_{j+1}}-\delta(t_{l^i_{j+1}},\epsilon^*)\right)-\left(t_{l^i_{j}}+\delta(t_{l^i_{j}},\epsilon^*)\right)<0.\end{equation}
We denote $$\frac{\left(t_{l^i_{j+1}}-\delta(t_{l^i_{j+1}},\epsilon^*)\right)+\left(t_{l^i_{j}}+\delta(t_{l^i_{j}},\epsilon^*)\right)}{2}:=t^*.$$
By a direct calculation, $$\begin{array}{ll}&t^*-t_{l^i_j}=\frac{\left(t_{l^i_{j+1}}-\delta(t_{l^i_{j+1}},\epsilon^*)\right)+\left(t_{l^i_{j}}+\delta(t_{l^i_{j}},\epsilon^*)\right)}{2}-t_{l^i_{j}}
\\
\\&
=\frac{\left(t_{l^i_{j+1}}-\delta(t_{l^i_{j+1}},\epsilon^*)\right)-\left(t_{l^i_{j}}+\delta(t_{l^i_{j}},\epsilon^*)\right)}{2}+\delta(t_{l^i_{j}},\epsilon^*)
<\delta(t_{l^i_{j}},\epsilon^*)\quad (by~\eqref{gc1})\end{array}$$

and

$$\begin{array}{ll}&t_{i_{j+1}}-t^*=t_{l^i_{j+1}}-\frac{\left(t_{l^i_{j+1}}-\delta(t_{l^i_{j+1}},\epsilon^*)\right)+\left(t_{l^i_{j}}+\delta(t_{l^i_{j}},\epsilon^*)\right)}{2}
\\
\\&
=\frac{\left(t_{l^i_{j+1}}-\delta(t_{l^i_{j+1}},\epsilon^*)\right)-\left(t_{l^i_{j}}+\delta(t_{l^i_{j}},\epsilon^*)\right)}{2}+\delta(t_{l^i_{j+1}},\epsilon^*)
<\delta(t_{l^i_{j+1}},\epsilon^*)\quad (by~\eqref{gc1}).\end{array}$$

Therefore

$$t^*\in (t_{l^i_{j}}-\delta(t_{l^i_{j}},\epsilon^*),t_{l^i_{j}}+\delta(t_{l^i_{j}},\epsilon^*))\bigcap(t_{l^i_{j+1}}-\delta(t_{l^i_{j+1}},\epsilon^*),t_{l^i_{j+1}}+\delta(t_{l^i_{j+1}},\epsilon^*))$$
with
$t_{l^i_j}<t^*<t_{l^i_{j+1}}.$

Then by \eqref{supa}, we obtain
$$\begin{array}{ll}&\vert F(t_{l^i_j})-F(t_{l^i_{j+1}})\vert\leq  \vert F(t_{l^i_j})-F(t^*)\vert+\vert F(t_{l^i_{j+1}})-F(t^*)\vert
\\
\\&\leq 2C_2(t_{l^i_{j+1}}-t^*)+2C_2(t^*-t_{l^i_j})\leq 2C_2(t_{l^i_{j+1}}-t_{l^i_{j}}),\end{array}$$ which yields \eqref{zhonx}.

Note that $x\in [t_{l^i_k}-\delta(t_{l^i_k},\epsilon^*),t_{l^i_k}]\subset [t_{l^i_k}-\delta(t_{l^i_k},\epsilon^*),t_{l^i_k}+\delta(t_{l^i_k},\epsilon^*)]$ and $y\in [t_{l^i_{m}},t_{l^i_{m}}+\delta(t_{l^i_{m}},\epsilon^*)]\subset [t_{l^i_{m}}-\delta(t_{l^i_{m}},\epsilon^*),t_{l^i_{m}}+\delta(t_{l^i_{m}},\epsilon^*)].$
Hence \eqref{supa} implies
$$\vert F(x)-F(t_{l^i_k})\vert+\vert F(y)-F(t_{l^i_{m}})\vert\leq 2C_2(t_{l^i_k}-x)+2C_2(t_{l^i_{m}}-y).$$
Therefore

$$\begin{array}{ll}\vert F(x)-F(y)\vert&\leq \vert F(x)-F(t_{l^i_k})\vert +\vert F(t_{l^i_k})-F(t_{l^i_{m}})\vert+\vert F(t_{l^i_{m}}-F(y) \vert
\\&\leq 2C_2(t_{l^i_k}-x)+2C_2(y-t_{l^i_{m}})+\sum\limits_{j=k}^{m-1}\vert F(t_{l^i_j})-F(t_{l^i_{j+1}})\vert\\&\leq 2C_2(t_{l^i_k}-x)+2C_2(y -t_{l^i_{m}})+\sum\limits_{j=k}^{m-1}2C_2(t_{l^i_{j+1}}-t_{l^i_{j}})\quad (by~\eqref{zhonx})\\&\leq 2C_2|x-y|.\hfill\qed\end{array}$$
\end{proof}

\

For each $(a_i,b_i),$ $1\leq i\leq P,$ we denote \begin{equation}\label{feng}(a_i,b_i)\bigcap \left(\bigcup\limits_{j=1}^{\tilde{S}(\epsilon^*)}(c_j,d_j)\right)=\bigcup\limits_{j=1}^{Q(i)}(a^i_j,b^i_j)\end{equation} with $Q(i)\in \Z_+$ and $ b^i_j<a^i_{j+1},~ 1\leq j\leq Q(i)-1.$

Since $a^i_1\geq a_i$ and $b^i_Q\leq b_i$, we have

     \begin{equation}\label{feng*}\begin{array}{ll}&(a_i,a^i_1]\bigcup\left[\left(\bigcup\limits_{j=1}^{Q(i)}(a^i_j,b^i_j)\right)\bigcup \left(\bigcup\limits_{j=1}^{Q(i)-1}[b^i_j,a^i_{j+1}] \right)\right]\bigcup [b^i_Q,b_i)
     \\\\&=(a_1,a^i_1]\bigcup(a^i_1,b^i_Q)\bigcup [b^i_Q,b_i)= (a_i,b_i)\end{array}\end{equation}
(here $(a_i,a^i_1],[b^i_Q,b_i)=\emptyset$ if $a_i=a^i_1$ and $b^i_Q=b_i$).

Then \eqref{feng*} and \eqref{feng} show (we set $b^i_0:=a_i$ and $a^i_{Q(i)+1}:=b_i$)
$$(a_i,b_i)-\bigcup\limits_{j=1}^{Q(i)}(a^i_j,b^i_j)=(a_i,a^i_1]\bigcup[b^i_Q,b_i)\bigcup\left(\bigcup\limits_{j=1}^{Q(i)-1}[b^i_j,a^i_{j+1}]\right).$$
Hence \begin{equation}\label{ll1}(a_i,a^i_1]\bigcup[b^i_Q,b_i)\bigcup\left(\bigcup\limits_{j=1}^{Q(i)-1}[b^i_j,a^i_{j+1}]\right)\subset [a_i,b_i]\bigcap \left([-M,M]-\bigcup\limits_{j=1}^{\tilde{S}(\epsilon^*)}(c_j,d_j)\right).\end{equation}

Note $[-M,M]-B\subset \bigcup\limits_{j=1}^{\tilde{S}(\epsilon^*)}(c_j,d_j).$ Hence $\left([-M,M]-\bigcup\limits_{j=1}^{\tilde{S}(\epsilon^*)}(c_j,d_j)\right)\subset B.$
Therefore for each $1\leq i\leq P,$
$$(a_i,a^i_1]\bigcup[b^i_Q,b_i)\bigcup\left(\bigcup\limits_{j=1}^{Q(i)-1}[b^i_j,a^i_{j+1}]\right)\subset B$$ with $(a_i,a^i_1],[b^i_Q,b_i),[b^i_j,a^i_{j+1}]$ being disjointed with each other.

Denote $$S_1:=\sum\limits_{i=1}^P\left((a^i_1-a_i)+(b_i-b^i_Q)+\sum\limits_{j=1}^{Q(i)-1}(b^i_j-a^i_{j+1})\right).$$
\eqref{ll1} implies $$\left((a^i_1-a_i)+(b_i-b^i_Q)+\sum\limits_{j=1}^{Q(i)-1}(b^i_j-a^i_{j+1})\right)\leq (b_i-a_i).$$

Therefore
\begin{equation}\label{S1}S_1\leq \sum\limits_{i=1}^P(b_i-a_i).\end{equation}

Then

\begin{equation}\label{ll2}\begin{array}{ll}&\sum\limits_{i=1}^P\left(|F(a^i_1)-F(a_i)|+|F(b^i_Q)-F(b_i)|+\sum\limits_{j=1}^{Q(i)-1}|F(b^i_j)-F(a^i_{j+1})|\right)\leq C_1S_1^{\frac{1}{2}}\ (by~(\mathbf{A})~of~Lemma~\ref{ab12})\\&\leq C_1\left(\sum\limits_{i=1}^P(b_i-a_i)\right)^{\frac{1}{2}}\quad \ (by~\eqref{S1}).\end{array}\end{equation}

Now we claim that following inequality holds true:
\begin{equation}\label{hold}\sum\limits_{i=1}^P\sum\limits_{j=1}^{Q(i)}\vert F(a^i_j)-F(b^i_j) \vert\leq 4(C_1+C_2)\left(\sum\limits_{i=1}^P(b_i-a_i)\right)^{\frac{1}{2}}.\end{equation}

\textbf{The final proof of Lemma \ref{ab12}:}

By the help of \eqref{ll2} and \eqref{hold}, we obtain
$$\begin{array}{ll}\sum\limits_{i=1}^P\vert F(a_i)-F(b_i)\vert&= \sum\limits_{i=1}^P\vert (F(a^i_1)-F(a_i))+(F(b^i_Q)-F(b_i))+\sum\limits_{j=1}^{Q(i)-1}(F(b^i_j)-F(a^i_{j+1}))+\sum\limits_{j=1}^{Q(i)}( F(b^i_j)-F(a^i_j))\vert\\&\leq \sum\limits_{i=1}^P\left(|F(a^i_1)-F(a_i)|+|F(b^i_Q)-F(b_i)|+\sum\limits_{j=1}^{Q(i)-1}|F(b^i_j)-F(a^i_{j+1})|+\sum\limits_{j=1}^{Q(i)}\vert F(a^i_j)-F(b^i_j) \vert\right)
\\&=\sum\limits_{i=1}^P\left(|F(a^i_1)-F(a_i)|+|F(b^i_Q)-F(b_i)|+\sum\limits_{j=1}^{Q(i)-1}|F(b^i_j)-F(a^i_{j+1})|\right)+\sum\limits_{i=1}^P\sum\limits_{j=1}^{Q(i)}\vert F(a^i_j)-F(b^i_j) \vert\\&\leq (6C_1+6C_2)\left(\sum\limits_{i=1}^P(b_i-a_i)\right)^{\frac{1}{2}}.
\end{array}$$
This implies  $$F~is~absolute~\frac{1}{2}-H\ddot{o}lder~continuous~in~[-M,M].$$

\

Hence it remains to show \eqref{hold}.

\textbf{The proof of \eqref{hold}}

\begin{proof} For any $1\leq i\leq P,$
note that \begin{equation}\label{jiaofk}\{(a^i_j,b^i_j)\vert j=1,2,\cdots,Q(i)\}=\{(c_j,d_j)\bigcap (a_i,b_i) \vert (c_j,d_j)\bigcap (a_i,b_i)\neq \emptyset,~j=1,2,\cdots,\tilde{S}(\epsilon^*) \}.\end{equation}

\eqref{jiaofk} allow us to denote \begin{equation}\label{zd}(a^i_1,b^i_1)=(\max\{a_i,c_{i^{*}}\},d_{i^*})~with~some~1\leq i^*\leq \tilde{S}(\epsilon^*)\end{equation} and \begin{equation}\label{yd}(a^i_{Q},b^i_Q)=(c_{i^{**}},\min\{d_{i^{**}},b_i\})~with~some~1\leq i^*\leq i^{**}\leq \tilde{S}(\epsilon^*).\end{equation} And \begin{equation}\label{zj}(a^i_j,b^i_j)=(c_{j_i},d_{j_i}),~j=2,3,\cdots,Q-1;~1\leq j_i\leq \tilde{S}(\epsilon^*)\end{equation} satisfies
$$(a^i_1,b^i_1)<(a^i_2,b^i_2)<\cdots<(a^i_Q,b^i_Q),$$
i.e.~$b_j^1<a^j_{i+1}~{\rm \ for\ any\ }~j=1,2,\cdots,Q-1.$

Then we have to consider the following three cases.

\begin{enumerate}
\item If $c_{i^*}=\max\{a_i,c_{i^{*}}\}$ and $d_{i^{**}}=\min\{d_{i^{**}},b_i\},$ then $(a^i_1,b^i_1)=(c_{i^*},d_{i^*})$ and $~(a^i_{Q},b^i_Q)=(c_{i^{**}},d_{i^{**}}).$
Hence \eqref{zhonx*} implies
\begin{equation}\label{1}\sum\limits_{1\leq j\leq Q}\vert F(a^i_j)-F(b^i_j) \vert\leq 2C_2\sum\limits_{1\leq j\leq Q}(b^i_j-a^i_j)\leq 2C_2(b_i-a_i).\end{equation}
\item If $a_i=\max\{a_i,c_{i^{*}}\}$ and $b_i=\min\{d_{i^{**}},b_i\},$ then
$a_i=a^i_1$ and $b_i=b^i_Q.$

Recall that $$(c_{i^*},d_{i^*})=\bigcup\limits_{j=1}^{s_{i^*}}(t_{l^{i^*}_j}-\delta(t_{l^{i^*}_j},\epsilon^*),t_{l^{i^*}_j}+\delta(t_{l^{i^*}_j},\epsilon^*))$$
$$(c_{i^{**}},d_{i^{**}})=\bigcup\limits_{j=1}^{s_{i^{**}}}(t_{l^{i^{**}}_j}-\delta(t_{l^{i^{**}}_j},\epsilon^*),t_{l^{i^{**}}_j}+\delta(t_{l^{i^{**}}_j},\epsilon^*)).$$
Define
$$t_{l^{i^*}_p}< a_i\leq t_{l^{i^*}_{p+1}},~1\leq p\leq s_{i^*}$$ and $$t_{l^{i^{**}}_q}< b_i\leq t_{l^{i^{**}}_{q+1}},~1\leq q\leq s_{i^{**}}.$$

And for  fixed $p,q$ we set $$\max\{t_{l^{i^*}_{p+1}}-\delta(t_{l^{i^*}_{p+1}},\epsilon^*),a_i\}=c';$$ $$\min\{t_{l^{i^{**}}_{q}}+\delta(t_{l^{i^{**}}_{q}},\epsilon^*),b_i\} =d'.$$

Then
$$(a_i,d_{i^*})=(a_i,c']\bigcup(c',d_{i^*}); (c_{i^{**}},b_i)=(c_{i^{**}},d')\bigcup[d',b_i).$$

Note $$t_{l^{i^*}_{p+1}}-\delta(t_{l^{i^*}_{p+1}},\epsilon^*)\leq c'\leq t_{l^{i^*}_{p+1}}$$ and $$t_{l^{i^{**}}_{q}}\leq d'\leq t_{l^{i^{**}}_{q}}+\delta(t_{l^{i^{**}}_{q}},\epsilon^*).$$
Then \eqref{zhonx*} implies
\begin{equation}\label{c'd}\vert F(c')-F(d_{i^*})\vert \leq 2C_2(d_{i^*}-c')<2C_2(d_{i^*}-c_{i^*}),\end{equation}
\begin{equation}\label{d'c}\vert F(d')-F(c_{i^{**}})\vert \leq 2C_2(d'-c_{i^{**}})<2C_2(d_{i^{**}}-c_{i^{**}}).\end{equation}

On the other hand, one notes $$(c'-a_i)=\left\{\begin{matrix}0, & if~c'=a_i\\ t_{l^{i^*}_{p+1}}-\delta(t_{l^{i^*}_{p+1}},\epsilon^*)-a_i & otherwise. \end{matrix}\right.$$

Therefore by the fact $t_{l^{i^*}_{p+1}}-\delta(t_{l^{i^*}_{p+1}},\epsilon^*)<t_{l^{i^*}_{p}}+\delta(t_{l^{i^*}_{p}},\epsilon^*)$, we obtain

$$0\leq (c'-a_i)\leq t_{l^{i^*}_{p+1}}-\delta(t_{l^{i^*}_{p+1}},\epsilon^*)-a_i<t_{l^{i^*}_{p}}+\delta(t_{l^{i^*}_{p}},\epsilon^*)-a_i<\delta(t_{l^{i^*}_{p}},\epsilon^*)<\delta^*(~\text{by}~\eqref{tij}).$$
Then by \eqref{yizhi},

$$ \vert F(c')-F(a_i) \vert,\ \vert F(d')-F(b_i) \vert\leq \frac{C_1}{2P}\left(\sum\limits_{i=1}^P(b_i-a_i)\right)^{\frac{1}{2}}.$$

Combining this with \eqref{c'd} and \eqref{d'c}, we have
$$\vert F(a_i)-F(d_{i^*})\vert \leq 2C_2(d_{i^*}-a_i)+\frac{C_1}{2P}\left(\sum\limits_{i=1}^P(b_i-a_i)\right)^{\frac{1}{2}}$$and
$$\vert F(c_{i^{**}})-F(b_i)\vert \leq 2C_2(b_i-c_{i^{**}})+\frac{C_1}{2P}\left(\sum\limits_{i=1}^P(b_i-a_i)\right)^{\frac{1}{2}}.$$

Then, \eqref{zd} and \eqref{yd} imply
\begin{equation}\label{z}\vert F(a^i_1)-F(b^i_1)\vert \leq 2C_2(b^i_1-a^i_1)+\frac{C_1}{2P}\left(\sum\limits_{i=1}^P(b_i-a_i)\right)^{\frac{1}{2}}\end{equation}and
\begin{equation}\label{y}\vert F(a^i_Q)-F(b^i_Q)\vert \leq 2C_2(b^i_Q-a^i_Q)+\frac{C_1}{2P}\left(\sum\limits_{i=1}^P(b_i-a_i)\right)^{\frac{1}{2}}.\end{equation}

The definition \eqref{zj} and  \eqref{zhonx*} allow us to obtain
\begin{equation}\label{zo}\vert F(a^i_j)-F(b^i_j)\vert =\vert F(c_{j_i})-F(d_{j_i})\vert\leq 2C_2(d_{j_i}-c_{j_i})=2C_2(b^i_j-a^i_j), j=2,3,\cdots,Q-1.\end{equation}

By \eqref{z}, \eqref{y} and \eqref{zo}, it holds that for $i=1,2,\cdots,P,$
\begin{equation}\label{2}\begin{array}{ll}&\sum\limits_{j=1}^{Q(i)}\vert F(a^i_j)-F(b^i_j)\vert\leq 2C_2\sum\limits_{j=1}^{Q(i)}(b^i_j-a^i_j)+\frac{C_1}{P}\left(\sum\limits_{i=1}^P(b_i-a_i)\right)^{\frac{1}{2}}\leq 2C_2(b_i-a_i)+\frac{C_1}{P}\left(\sum\limits_{i=1}^P(b_i-a_i)\right)^{\frac{1}{2}}\\\\&(recall~a_i\leq a^i_{1}< b^i_1<\cdots<a^i_j<b^i_j<\cdots<a^i_Q<b^i_Q\leq b_i).\end{array}\end{equation}

\item If $c_{i^*}< a_i,~d_{i^{**}}\leq b_i$ or $d_{i^{**}}> b_i,c_{i^*}\geq a_i,$  without loss of generality, we consider the case $c_{i^*}< a_i,~d_{i^{**}}\leq b_i.$
Then $a_i=a^i_1$ and $~d_{i^{**}}=b^i_Q.$

Recall the definition
$(a^i_1,b^i_1)=(\max\{a_i,c_{i^{*}}\},d_{i^*})$ and
$(a^i_Q,b^i_Q)=(c_{i^{**}},\min\{d_{i^{**}},b_i\})$
and
\begin{equation}\label{zj'}(a^i_j,b^i_j)=(c_{j_i},d_{j_i}),~j=2,3,\cdots,Q(i)-1.\end{equation}

In the current case, $(a^i_Q,b_Q^i)=(c_{i^{**}},d_{i^{**}}).$ Then, \eqref{zhonx*} implies
\begin{equation}\label{4}\vert F(a^i_Q)-F(b^i_Q)\vert=\vert F(c_{i^{**}})-F(d_{i^{**}})\vert\leq 2C_2(d_{i^{**}}-c_{i^{**}})=2C_2(b^i_Q-a^i_Q).\end{equation}

Similar as the proof of \eqref{z} and \eqref{zo}, we have
\begin{equation}\label{z'}\vert F(a^i_1)-F(b^i_1)\vert \leq 2C_2(b^i_1-a^i_1)+\frac{C_1}{2P}\left(\sum\limits_{i=1}^P(b_i-a_i)\right)^{\frac{1}{2}}\end{equation}
and
\begin{equation}\label{zo'}\vert F(a^i_j)-F(b^i_j)\vert =\vert F(c_{j_i})-F(d_{j_i})\vert\leq 2C_2(d_{j_i}-c_{j_i})=2C_2(b^i_j-a^i_j), j=2,3,\cdots,Q(i)-1.\end{equation}

By \eqref{4}, \eqref{z'} and \eqref{zo'}, we obtain
\begin{equation}\label{3}\begin{array}{ll}\sum\limits_{j=1}^{Q(i)}\vert F(a^i_j)-F(b^i_j)\vert&\leq 2C_2(b_i-a_i)+\frac{C_1}{2P}\left(\sum\limits_{i=1}^P(b_i-a_i)\right)^{\frac{1}{2}}.\end{array}\end{equation}
\end{enumerate}

\

In summary, \eqref{1}, \eqref{2} and \eqref{3} together show that in any case we always have
$$\sum\limits_{j=1}^{Q(i)}\vert F(a^i_j)-F(b^i_j)\vert\leq 2C_2(b_i-a_i)+\frac{C_1}{P}\left(\sum\limits_{i=1}^P(b_i-a_i)\right)^{\frac{1}{2}},\quad  1\leq i\leq P.$$
Therefore
$$\begin{array}{ll}&\sum\limits_{i=1}^{P}\sum\limits_{i=1}^{Q(i)}\vert F(a^i_i)-F(b^i_i)\vert\leq 2C_2\sum\limits_{i=1}^P(b_i-a_i)+P\cdot\frac{C_1}{P}\left(\sum\limits_{i=1}^P(b_i-a_i)\right)^{\frac{1}{2}}\\&\leq 2C_2\sum\limits_{i=1}^P(b_i-a_i)+C_1\left(\sum\limits_{i=1}^P(b_i-a_i)\right)^{\frac{1}{2}}\leq 4(C_1+C_2)\left(\sum\limits_{i=1}^P(b_i-a_i)\right)^{\frac{1}{2}},\end{array}$$
which yield \eqref{hold}.
\hfill\qed\end{proof}

\hfill\qed\end{proof}
Finally, by the help of Lemma \ref{c2lemma}, \ref{ab11} and \ref{ab12}, we completes the proof of absolutely $\frac{1}{2}$-H\"older continuity of LE.

\subsubsection{\text{Almost everywhere differentiability}}

Note we have already obtained that $L(t)$ is absolutely $\frac{1}{2}$-H\"older continuous on $[-M,M]$ with any $M>0$, hence

\begin{corollary}\label{coro11} $L(t)$ is almost everywhere differentiable on $[-M,M]$ with any $M>0.$ \end{corollary}

Furthermore  we have the following estimate on the measure of energies for which LE possesses a large derivative. Denote $$\Sigma^*=\{t\in \Sigma^{\lambda}\vert L'(t)~exists \}$$ and $$F_{L}=\{t\in\Sigma^{*}\bigcap \mathcal{FR} \vert |L'(t)|>L\}$$ with $L>0$.

Note $$ \Sigma^{\lambda} \subset [-2+2\inf v, 2+2\sup v]$$ and Corollary \ref{coro11} (taking $M=\max \{|-2+2\inf v|,|2+\sup v|\}$) shows $ Leb\{\Sigma^*\}=Leb\{\Sigma^{\lambda}\}.$
On the other hand, (3) of Proposition \ref{prop14} implies
$Leb\{\mathcal{IR}\}=Leb\{\Sigma^{\lambda}-\mathcal{FR}\}=0,$ which leads that
$$Leb\{\mathcal{FR}\}=Leb\{\Sigma^{\lambda}\bigcap \mathcal{FR}\}=Leb\{\Sigma^{\lambda}\}.$$

Then \begin{equation}\label{almosteve}Leb(\{\Sigma^*\bigcap \mathcal{FR}\})= Leb(\{\Sigma^{\lambda}\bigcap \mathcal{FR}\})\end{equation}
and
     $\lim\limits_{L\rightarrow 0}Leb\{F_{L}\}=Leb\{\Sigma^{\lambda}\}.$

Recall that \eqref{lipschitz-1} of Lemma \ref{liplemma} implies for any $\tilde{t}\in \mathcal{FR},$ it holds that
$$\limsup\limits_{t\rightarrow \tilde{t}}\frac{|L(t)-L(\tilde{t})|}{|t-\tilde{t}|}\leq C\lambda^{-q^{\frac{1}{4}}_{N+s(k_{last}(\tilde{t}))-1}}(dist\{\tilde{t},G_{k_{last}(\tilde{t})}\})^{-\frac{1}{2}}.$$
Therefore for $\tilde{t}\in \Sigma^*\bigcap \mathcal{FR},$
$$L'(\tilde{t})=\limsup\limits_{t\rightarrow \tilde{t}}\frac{|L(t)-L(\tilde{t})|}{|t-\tilde{t}|}\leq C\lambda^{-q^{\frac{1}{4}}_{N+s(k_{last}(\tilde{t}))-1}}(dist\{\tilde{t},G_{k_{last}(\tilde{t})}\})^{-\frac{1}{2}}.$$

Hence \begin{equation}\label{ttttt}\{t\in \Sigma^{*}\bigcap \mathcal{FR} \vert |L'(t)|>L\}\subset \{\tilde{t} \vert L\leq C\lambda^{-q^{\frac{1}{4}}_{N+s(k_{last}(\tilde{t}))-1}}(dist\{\tilde{t},G_{k_{last}(\tilde{t})}\})^{-\frac{1}{2}}\}.\end{equation}

Then \eqref{almosteve} and \eqref{ttttt} imply
$$\begin{array}{ll}&Leb\{F_L\}=Leb \{t\in \Sigma^{\lambda}\bigcap \mathcal{FR} \vert |L'(t)|>L\}=Leb \{t\in \Sigma^{*}\bigcap \mathcal{FR} \vert |L'(t)|>L\}\\
&\leq \sum\limits_{k\in \Z}Leb\{\tilde{t} \vert L\leq C\lambda^{-q^{\frac{1}{4}}_{N+s(k)-1}}(dist\{\tilde{t},G_{k}\})^{-\frac{1}{2}}\}
\\&\leq \sum\limits_{k\in \Z} C\left(\lambda^{-2q^{\frac{1}{4}}_{N+s(k)-1}}L^{-2} \right)
\leq C\left(\sum\limits_{k\in \Z} \lambda^{-2q^{\frac{1}{4}}_{N+s(k)-1}}\right)L^{-2}\leq {C}^* L^{-2}.\end{array}$$

\subsubsection{\text{H\"older continuity for $t\in {\mathcal{IR}}$}}\label{4.3.3}

\begin{definition}For any $t\in \R,$ let
$$\label{defbeta}\beta(t)\triangleq\left\{\begin{matrix}\liminf\limits_{n\rightarrow +\infty} \beta_n(t) & t\notin \bigcup\limits_{i\in \Z_+}\{t^{k_i}_-,t^{k_i}_+\}\\ \frac{1}{2}& t\in \bigcup\limits_{i\in \Z_+}\{t^{k_i}_-,t^{k_i}_+\}\end{matrix}\right.,$$
 where  $\beta_n(t)\triangleq\frac{1}{2}+\min\{\frac{\log|t-t^{k_n}_+|}{2\log|t-t^{k_n}_-|},\frac{\log|t-t^{k_n}_-|}{2\log|t-t^{k_n}_+|}\}$ and $t^{k_n}_{\pm}$ is from (2) of Theorem \ref{15}.

\end{definition}

\begin{remark} It is easy to check that $\frac{1}{2}\leq \beta(t)\le 1.$ Later, we will prove $\beta$ is related to the local H\"older exponent of $L(t).$ More precisely, for $t\in \bigcup\limits_{i\in \Z_+}\{t^{k_i}_-,t^{k_i}_+\},$ $L(t)$ is $\frac{1}{2}$-H\"older continuous; for $t\notin \bigcup\limits_{i\in \Z_+}\{t^{k_i}_-,t^{k_i}_+\},$ the exponent is between $\beta(t)-\epsilon$ and $\beta(t)+\epsilon.$
\end{remark}

\begin{theorem}\label{zhibiaoeq}
For any $\beta\in [\frac{1}{2},1],$ it holds that $\{t\in \Sigma^{\lambda}|\beta(t)=\beta\}\neq \emptyset.$

\end{theorem}

By Theorem \ref{15}, all opening gaps on $[\inf \Sigma^{\lambda}, \sup \Sigma^{\lambda}]$ can be labeled by $\{G_{k_i}\}_{i\in \Z_+}$ and
\begin{equation} \label{pugapp} \lambda^{-C|k_i|}\leq |G_{k_i}|\leq \lambda^{-c|k_i|}.
\end{equation}

To obtain Theorem \ref{zhibiaoeq}, we have to do the following preparations.

\begin{lemma}\label{zblemma} Given any $a,b\in \mathcal{K}(\lambda).$
\begin{enumerate}
\item If $\lambda^{|b|^{c}}+|b|>|a|\geq |b|,$ then
$ dist(G_{a},G_{b})>\lambda^{-|b|^c}.$

\item If $\lambda^{|b|^{c}}+|b|\leq |a|,$ then
$ dist(G_{a},G_{b})>|a|^{-C}.$

\end{enumerate}
\end{lemma}

\begin{proof} Take $\eta=\left[\max\{|a|-|b|, \lambda^{|b|^{c}}\}\right]^{-1}$ and apply \eqref{dgkij} of Theorem \ref{15}.

If $\lambda^{|b|^{c}}+|b|>|a|\geq |b|,$ we have
$$dist(G_{a},G_{b})>\left[\max\{|a|-|b|, \lambda^{|b|^{c}}\}\right]^{-C}\ge \lambda^{-C\left(|b|\right)^{c}}$$  and if $\lambda^{|b|^{c}}+|b|\leq |a|$, we have $$dist(G_{a},G_{b})>\left[\max\{|a|-|b|, \lambda^{|b|^{c}}\}\right]^{-C}=(|a|-|b|)^{-C}>c|a|^{-C}.\hfill\qed$$
\end{proof}

Given any $\kappa,\gamma>0,$ we set $$\tilde{G}_{k_j,\kappa,\gamma}=(t^{k_j}_--\gamma|G_{k_j}|^{\kappa},t^{k_j}_++\gamma|G_{k_j}|^{\kappa}).$$

\begin{lemma}\label{pugapg} For any fixed $\kappa,\gamma,$ there exists $i^*(\kappa,\gamma)\in \Z_+$ such that for any $s$ with $|k_s|\geq |k_{i^*}|,$ it holds that \begin{equation}\label{pugapp1}\Lambda_{\kappa,\gamma,s}\subset \{k_j \vert |k_j|> \lambda^{\left(|k_{s}|\right)^{c}}\}\end{equation}and
\begin{equation}\label{pugapp2}\Lambda^*_{\kappa,\gamma,s} \subset \{k_j \vert |k_j|> \lambda^{\left(|k_{s}|\right)^{c}}\},\end{equation}

 where $$\Lambda_{\kappa,\gamma,s}:=\{k_j~\vert~j\neq s,~{G}_{k_j}\bigcap \tilde{G}_{k_{s,\kappa,\gamma}}\neq \emptyset\},\quad~\Lambda^*_{\kappa,\gamma,s}:=\{k_j \vert j\neq s,~|k_j|\geq |k_{s}|,~\tilde{G}_{k_j,\kappa,\gamma}\bigcap \tilde{G}_{k_{s},\kappa,\gamma}\neq \emptyset\}.$$
\end{lemma}
\begin{proof}

By \eqref{pugapp}, there exists $i^*$ such that for any $|k_j|\geq |k_s|\geq |k_{i^*}|$, we have

\begin{equation}\label{wolgg}\lambda^{-\frac{1}{10}c\kappa|k_{s}|}\geq 2\lambda^{-c\kappa|k_s|}\geq \gamma|G_{k_j}|^{c\kappa}+\gamma|G_{k_{s}}|^{c\kappa}.\end{equation}

For any fixed $|k_s|\geq |k_{i^*}|,$ we consider the following two sets
$$\Lambda_1:=\{k_j \vert j\neq s,~|k_j|< |k_{s}|,~{G}_{k_j}\bigcap \tilde{G}_{k_{s},\kappa,\gamma}\neq \emptyset\},$$
$$\Lambda_2:=\{k_j \vert j\neq s,~|k_j|\geq |k_{s}|,~{G}_{k_j}\bigcap \tilde{G}_{k_{s},\kappa,\gamma}\neq \emptyset\}.$$

Clearly, $\Lambda_{\kappa,\gamma,s}=\Lambda_1\bigcup\Lambda_2$ with $\Lambda_1\bigcap \Lambda_2=\emptyset.$

For $\Lambda_2,$ we apply (1) of Lemma \ref{zblemma} with $b=k_{s}.$
Then if $\lambda^{|k_{s}|^{c}}+|k_{s}|>|k_{j}|\geq |k_{s}|,$ \eqref{wolgg} implies
$ \text{dist}(G_{k_j},G_{k_{s}})>\lambda^{-|k_{s}|^c}\gg \lambda^{-\frac{1}{10}c\kappa|k_{s}|}\geq \gamma|G_{k_j}|^{c\kappa}+\gamma|G_{k_{s}}|^{c\kappa},$
which implies \begin{equation}\label{til}\tilde{G}_{{k_j},\kappa,\gamma}\bigcap \tilde{G}_{k_{s},\kappa,\gamma}= \emptyset.\end{equation}
Therefore $\min\Lambda_2\geq \lambda^{|k_{s}|^{c}}+|k_{s}|>\lambda^{|k_{s}|^{c}}.$ Hence \begin{equation}\label{lambda2} \Lambda_2\subset \{k_j \vert |k_j|> \lambda^{|k_{s}|^{c}}\}.\end{equation}

For $\Lambda_1,$ taking $a=k_{s}$ in Lemma \ref{zblemma}, we have

\begin{enumerate}

\item If $\lambda^{|k_{j}|^{c}}+|k_{j}|>|k_{s}|\geq |k_{j}|,$ then
$ dist(G_{k_{s}},G_{k_{j}})>\lambda^{-|k_{j}|^c}\geq \lambda^{-|k_{s}|^c}\gg \gamma|G_{k_{s}}|^{c\kappa}.$
\item If $\lambda^{|k_{j}|^{c}}+|k_{j}|\leq |k_{s}|,$ then
$ dist(G_{k_{s}},G_{k_{j}})>|k_{s}|^{-C}\geq \lambda^{-C|k_{s}|^c}\gg \gamma|G_{k_{s}}|^{c\kappa}.$
\end{enumerate}

Therefore \begin{equation}\label{lambda1}\Lambda_1=\emptyset.\end{equation}

By \eqref{lambda2} and \eqref{lambda1}, we have (\ref{pugapp1}).

And \eqref{til} yields (\ref{pugapp2}).\hfill\qed\end{proof}

\begin{lemma}\label{pugape} There exists $i^*(\kappa,\gamma)\in \Z_+$ such that for any $|k_s|\geq |k_{i^*}|$ and an open interval
$$I\subset \left[(t^{k_{s}}_+,t^{k_{s}}_++\gamma|G_{k_{s}}|^{\kappa})\bigcup (t^{k_{s}}_--\gamma|G_{k_{s}}|^{\kappa},t^{k_{s}}_-)\right]$$ satisfying $|I|\geq |G_{k_{s}}|^{100\kappa},$ it holds that
\begin{equation}\label{pugape1}Leb\{{I\bigcap \Sigma^{\lambda}}\}>0.\end{equation}

\end{lemma}

\begin{proof}

By the help of Lemma \ref{pugapg}, there exists $i^*$ such that
\begin{equation}\label{pugape0}\Lambda_{\kappa,\gamma,s} \subset \{k_j \vert |k_j|> \lambda^{|k_{s}|^{c}}\},~{\rm \ for\ any\ }~|k_s|\geq |k_{i^*}|.\end{equation}
In the following proof, we omit the dependence of $\Lambda$ for simplification.
Note the upper bound of \eqref{pugapp} and \eqref{pugape0} imply
\begin{equation}\label{pugapc}\sum\limits_{k_j\in \Lambda}|\tilde{G}_{k_j,\kappa,\gamma}|\leq \sum\limits_{k_j\in \Lambda}\left(2\gamma|G_{k_j}|^{\kappa}+|G_{k_j}|\right)\leq \lambda^{-\lambda^{|k_{s}|^{c}}}\ll \gamma^{1000}|G_{k_{s}}|^{1000\kappa}.\end{equation}

Then \eqref{pugapc} implies for any open interval $I\subset \tilde{G}_{k_{s},\kappa,\gamma}-G_{k_{s}}$ with $|I|\geq |G_{k_{s}}|^{100\kappa},$
\begin{equation}\label{pugapd} Leb\{I-\bigcup\limits_{k_j\in \Lambda} \tilde{G}_{k_j,\kappa,\gamma}\}>|I|-\gamma^{1000}|G_{k_{s}}|^{1000\kappa}>(1-|G_{k_{s}}|^{800\kappa})|I|>\frac{2}{3}|I|.
\end{equation}


By the definition of $\Lambda$, for $k_j\notin \Lambda,$ we have $G_{k_j}\bigcap \tilde{G}_{k_{s},\kappa,\gamma}=\emptyset,$ which shows $$I\bigcap G_{k_j}=\emptyset,~{\rm \ for\ any\ }~k_j\notin \Lambda.$$

Thus \eqref{pugapd} is equivalent to
$$Leb\{I-\bigcup\limits_{k_j\in \Lambda} \tilde{G}_{k_j,\kappa,\gamma}-\bigcup\limits_{k_j\notin \Lambda} {G}_{k_j}\}>\frac{2}{3}|I|.$$

Hence $$\begin{array}{ll}
&Leb\{{I\bigcap \Sigma^{\lambda}}\}
=Leb\{I-\bigcup\limits_{k_j\in \Lambda} G_{k_j}-\bigcup\limits_{k_j\notin \Lambda} {G}_{k_j}\}
>Leb\{I-\bigcup\limits_{k_j\in \Lambda} \tilde{G}_{k_j,\kappa,\gamma}-\bigcup\limits_{k_j\notin \Lambda} {G}_{k_j}\}>\frac{2}{3}|I|>0.\hfill\qed
\end{array}$$
\end{proof}




\begin{lemma}\label{betat} Given $t\in \Sigma^{\lambda},$  the following  hold true.

\

\begin{enumerate}

\item Given $\gamma>0$ and $\kappa>1,$ if there exists some $K\in \Z_+$ and a sequence $|k_{n_j}|\rightarrow +\infty$ such that $$dist(t, G_{k_j})\geq \frac{\gamma}{2}|G_{k_j}|^{\kappa},\quad |k_j|\geq K;$$
$$\frac{3\gamma}{4}|G_{k_{n_j}}|^{\kappa}\geq dist(t, G_{k_{n_j}})\geq \frac{\gamma}{2}|G_{k_{n_j}}|^{\kappa},~j\in \Z_+,$$

then $\beta(t)=\frac{\kappa+1}{2\kappa}.$

\

\item Given $\gamma>0$ and $0<\kappa<1,$ if there exists some $K\in \Z_+$ such that $$dist(t, G_{k_j})\geq \frac{\gamma}{2}|G_{k_j}|^{\kappa},\quad |k_j|\geq K,$$

then $\beta(t)=1.$
\end{enumerate}
\end{lemma}

\begin{proof}Since $t\in \Sigma^{\lambda},$ by the symmetric expression in the definition of $\beta(t)$, without loss of generality, we assume that $t\leq t^{k_j}_{-}.$

\textbf{The proof of (1):}\ Note for large $j$ we have $$1\gg |t^{k_j}_--t^{k_j}_+|>|t-t^{k_j}_-|=dist(t, G_{k_j}).$$
 For any $|k_j|\geq K,$ we have
\begin{equation}\label{xiajiea1}\begin{array}{ll}\left\vert \frac{\log|t-t^{k_j}_+|}{2\log|t-t^{k_j}_-|}\right\vert= \left\vert \frac{\log|t-t^{k_j}_+|}{2\log|t-t^{k_j}_-|}\right\vert= \frac{\log(t^{k_j}_--t+|G_{k_j}|)}{2\log(t^{k_j}_--t)}&\geq \frac{\log(\frac{\gamma}{2}|G_{k_{j}}|^{\kappa}+|G_{k_{j}}|)}{2\log(\frac{\gamma}{2}|G_{k_{j}}|^{\kappa})}=\frac{\log(1+\frac{\gamma}{2}|G_{k_j}|^{\kappa -1})+1}{\frac{2\log (\frac{\gamma}{2})}{\log |G_{k_j}|}+2\kappa}.\end{array}\end{equation}

For any $j\in \{{n_i}\}_{i\in \Z_+},$
\begin{equation}\label{shangjiea1}\begin{array}{ll}\left\vert \frac{\log|t-t^{k_{n_j}}_+|}{2\log|t-t^{k_{n_j}}_-|}\right\vert= \left\vert \frac{\log|t-t^{k_{n_j}}_+|}{2\log|t-t^{k_{n_j}}_-|}\right\vert= \frac{\log(t^{k_{n_j}}_--t+|G_{k_{n_j}}|)}{2\log(t^{k_{n_j}}_--t)}&\leq \frac{\log(\frac{3\gamma}{4}|G_{k_{n_j}}|^{\kappa}+|G_{k_{n_j}}|)}{2\log(\frac{3\gamma}{4}|G_{k_{n_j}}|^{\kappa})}=\frac{\log(1+\frac{3\gamma}{4}|G_{k_j}|^{\kappa -1})+1}{\frac{2\log (\frac{3\gamma}{4})}{\log |G_{k_j}|}+2\kappa}.\end{array}\end{equation}

Then \eqref{xiajiea1} and \eqref{shangjiea1} lead to
$$\frac{1}{2\kappa} \leq \liminf\left\vert \frac{\log|1+\frac{|t^{k_j}_--t^{k_j}_+|}{|t-t^{k_j}_-|}|}{2\log|t-t^{k_j}_-|}\right\vert\leq  \frac{1}{2\kappa}.$$

Hence
$$\begin{array}{ll}&\liminf \left(\min\{\frac{\log|t-t^{k_j}_+|}{2\log|t-t^{k_j}_-|},\frac{\log|t-t^{k_j}_-|}{2\log|t-t^{k_j}_+|}\}\right)=\liminf\frac{\log|t-t^{k_j}_+|}{2\log|t-t^{k_j}_-|}=\liminf\frac{\log|t-t^{k_j}_-+t^{k_j}_--t^{k_j}_+|}{2\log|t-t^{k_j}_-|}\\
&=\liminf\frac{\log(|t-t^{k_j}_-|+|t^{k_j}_--t^{k_j}_+|)}{2\log|t-t^{k_j}_-|}
=\frac{1}{2}+\liminf\left\vert\frac{\log|1+\frac{|t^{k_j}_--t^{k_j}_+|}{|t-t^{k_j}_-|}|}{2\log|t-t^{k_j}_-|}\right\vert= \frac{1}{2}+\frac{1}{2\kappa}.\end{array}$$

Therefore $$\beta(t)=\liminf\beta_{j}(t)=\liminf\left(\frac{1}{2}+\min\{\frac{\log|t-t^{k_j}_+|}{2\log|t-t^{k_j}_-|},\frac{\log|t-t^{k_j}_-|}{2\log|t-t^{k_j}_+|}\}\right)= \frac{1}{2}+\frac{1}{2\kappa}.$$

\

\textbf{The proof of (2):}
Since $t\leq t^{k_j}_{-},$ we have

$$\left\vert \frac{\log|1+\frac{|t^{k_j}_--t^{k_j}_+|}{|t-t^{k_j}_-|}|}{2\log|t-t^{k_j}_-|}\right\vert\leq \left\vert\frac{\frac{|t^{k_j}_--t^{k_j}_+|}{|t-t^{k_j}_-|}}{2\log|t-t^{k_j}_-|}\right\vert\leq C\left\vert\frac{\frac{|G_{k_j}|}{|G_{k_j}|^{\kappa}}}{2\log||G_{k_j}|^{\kappa}|}\right\vert \leq C |G_{k_j}|^{1-\kappa}\rightarrow 0,$$ as $j\rightarrow +\infty.$

Therefore
$$\begin{array}{ll}&\min\{\frac{\log|t-t^{k_j}_+|}{2\log|t-t^{k_j}_-|},\frac{\log|t-t^{k_j}_-|}{2\log|t-t^{k_j}_+|}\}=\frac{\log|t-t^{k_j}_+|}{2\log|t-t^{k_j}_-|}=\frac{\log|t-t^{k_j}_-+t^{k_j}_--t^{k_j}_+|}{2\log|t-t^{k_j}_-|}=\frac{\log(|t-t^{k_j}_-|+|t^{k_j}_--t^{k_j}_+|)}{2\log|t-t^{k_j}_-|}
\\& =\frac{1}{2}+\frac{\log|1+\frac{|t^{k_j}_--t^{k_j}_+|}{|t-t^{k_j}_-|}|}{2\log|t-t^{k_j}_-|}\rightarrow \frac{1}{2},\quad {\rm as~j\rightarrow~}+\infty.\end{array}$$
Hence
$$\beta(t)=\liminf\beta_{j}(t)=\liminf\left(\frac{1}{2}+\min\{\frac{\log|t-t^{k_j}_+|}{2\log|t-t^{k_j}_-|},\frac{\log|t-t^{k_j}_-|}{2\log|t-t^{k_j}_+|}\}\right)= \frac{1}{2}+\frac{1}{2}=1.\hfill\qed$$
\end{proof}

\

Now we are in a position to give the proof of Theorem \ref{zhibiaoeq}.

\begin{proof}[The proof of Theorem \ref{zhibiaoeq}]

\

The definition of $\beta$ implies $\bigcup\limits_{i\in \Z_+}\{t^{k_i}_-,t^{k_i}_+\}\subset \{t\in \Sigma^{\lambda}|\beta(t)=\frac{1}{2}\},$ which shows
$ \{t\in \Sigma^{\lambda}|\beta(t)=\frac{1}{2}\}\neq \emptyset.$

Therefore to obtain what we desire, we only need to prove that for any $\beta\in (\frac12, 1)$,
$ \{t\in \Sigma^{\lambda}|\beta(t)=\beta\}\neq \emptyset$
and
$ \{t\in \Sigma^{\lambda} | \beta(t)=1\}\neq \emptyset.$

Therefore it is enough to  find some $t\in \Sigma^{\lambda}$ such that
$\beta(t)=\beta$ for any fixed $\frac{1}{2}<\beta<1$ and find  $t\in \Sigma^{\lambda}$ such that
$\beta(t)=1.$

Next, we will find such a $t$ {\bf by induction}. More precisely, we claim that for any fixed $\kappa>0$, there exists $t\in \Sigma^{\lambda},$ a monotonically increasing sequence~$|\hat{k}_n|\rightarrow +\infty$ and $K>0$ such that
\begin{equation}\label{xiajie*}dist(t, G_{k_j})\geq 2|G_{k_j}|^{\kappa},~{\rm \ for\ any\ }~|k_j|\geq K\end{equation}
and
\begin{equation}\label{shangjie*}3|G_{\hat{k}_{n+1}}|^{\kappa}\geq dist(t, G_{\hat{k}_{n+1}})\geq 2|G_{\hat{k}_{n+1}}|^{\kappa},~ n\in \Z_+.\end{equation}


\textbf{From the original step to the first step.}

\

For any fixed $\kappa,\gamma,$ we can take a uniform $i^*$ such that both Lemma \ref{pugapg} and Lemma \ref{pugape} hold true.

Now we set $$\hat{k}_0:=k_{i^*},\quad I_0:=(t^{\hat{k}_0}_--\frac{3\gamma}{4}|G_{\hat{k}_0}|^{\kappa},t^{\hat{k}_0}_--\frac{\gamma}{2}|G_{\hat{k}_0}|^{\kappa}).$$
Then
\eqref{pugape1} of Lemma \ref{pugape} implies $\Sigma^{\lambda}\bigcap I_0\neq \emptyset.$


Recall that $\tilde{G}_{k_j,\kappa,\gamma}=(t^{k_j}_--\gamma|G_{k_j}|^{\kappa},t^{k_j}_++\gamma|G_{k_j}|^{\kappa})$. Let
\begin{equation}\label{defhatk}|\hat{k}_1|=\min{\{k_j \vert  |k_j|\geq |k_{i^*}|, \tilde{G}_{k_j,\kappa,\gamma} \bigcap I_0\neq \emptyset \}}\end{equation} (if $\hat{k}_1,-\hat{k}_1\in \{k_j \vert  |k_j|\geq |k_{i^*}|, \tilde{G}_{k_j,\kappa,\gamma} \bigcap I\neq \emptyset \},$ then we choose the positive one).

Without loss of generality, we assume that \begin{equation}\label{wolg}\frac{t^{\hat{k}_1}_-+t^{\hat{k}_1}_+}{2}<\frac{\inf I_0+\sup I_0}{2}.\end{equation} That is, the center of $\tilde{G}_{k_j,\kappa,\gamma}$ is located at the left side of $I_0.$

By the help of \eqref{pugapp2}, $|\hat{k}_1|\gg |\hat{k}_0|(=|k_{i^*}|).$
Then $|G_{\hat{k}_1}|,|G_{\hat{k}_1}|^{\kappa}\ll |G_{\hat{k}_0}|=|I_0|~(by~\eqref{pugapp}).$ Combining the above with \eqref{wolg} we have $$(t^{\hat{k}_1}_++\frac{\gamma}{2}|G_{\hat{k}_1}|^{\kappa},t^{\hat{k}_1}_++\frac{3\gamma}{4}|G_{\hat{k}_1}|^{\kappa})\subset I_0.$$ Now we denote
$$I_1:=(t^{\hat{k}_1}_++\frac{\gamma}{2}|G_{\hat{k}_1}|^{\kappa},t^{\hat{k}_1}_++\frac{3\gamma}{4}|G_{\hat{k}_1}|^{\kappa}).$$ Note that $\tilde{G}_{\hat{k}_1,\kappa,\gamma}\bigcap I_0\neq \emptyset$ and \eqref{wolg} imply $$dist(I_1, I_0^c)\geq \frac{\gamma}{2}|G_{\hat{k}_1}|^{\kappa}.$$
Since $|I_1|=\frac{\gamma}{4}|G_{\hat{k}_1}|^{\kappa}\geq |G_{\hat{k}_1}|^{100\kappa},$ by the help of \eqref{pugape1} of Lemma \ref{pugape}, we obtain
$$I_1\bigcap \Sigma^{\lambda}\neq \emptyset.$$
And the definition of $\hat{k}_1$ in \eqref{defhatk} implies that for any $|\hat{k}_1|>|k_j|\geq |\hat{k}_{0}|,$ $$\tilde{G}_{k_j,\kappa,\gamma} \bigcap I_0= \emptyset.$$
Therefore
$$dist(I_1, G_{k_j})\geq \frac{\gamma}{2}|G_{\hat{k}_1}|^{\kappa}+\gamma|G_{k_j}|^{\kappa}>\frac{\gamma}{2}|G_{k_j}|^{\kappa} ,~|\hat{k}_1|\geq |k_j|\geq |\hat{k}_{0}|,~$$

(Here, if $\{\hat{k}_1,-\hat{k}_1\}\subset \{k_j \vert  |k_j|\geq |k_{i^*}|, \tilde{G}_{k_j,\kappa,\gamma} \bigcap I_0\neq \emptyset \},$ then we have to check $dist(I_1, G_{-\hat{k}_1})\geq \gamma|G_{\hat{k}_1}|^{\kappa}.$ In fact, by the help of (1) of Lemma \ref{zblemma} and taking $b=\hat{k}_1,$ we have $dist(G_{\hat{k}_1},G_{-\hat{k}_1})\geq \lambda^{-|\hat{k}_1|^c}\gg \gamma|G_{\hat{k}_1}|^{\kappa}.$ Therefore $\text{dist}(I_1, G_{-\hat{k}_1})\geq dist(G_{-\hat{k}_1}, G_{-\hat{k}_1})-dist(G_{\hat{k}_1},I_1)>\lambda^{-|\hat{k}_1|^c}-\gamma|G_{\hat{k}_1}|^{\kappa}\gg \gamma|G_{\hat{k}_1}|^{\kappa}.)$

Hence we have
$$|\hat{k}_1|\gg |\hat{k}_0|;$$
$$I_1\bigcap \Sigma^{\lambda}\neq \emptyset,~|I_1|=\frac{\gamma}{4}|G_{\hat{k}_1}|^{\kappa}~and~dist(I_1,I_0^c)>\frac{\gamma}{2}|G_{\hat{k}_1}|^{\kappa};$$
$$dist(I_1, G_{k_j})\geq \frac{\gamma}{2}|G_{k_j}|^{\kappa},~|\hat{k}_1|\geq |k_j|\geq |\hat{k}_{0}|.$$

\

\textbf{From $n$-th step to $(n+1)$-th step}

Now suppose that we have obtained
$$|\hat{k}_n|\gg |\hat{k}_{n-1}|;$$
$$I_n\bigcap\Sigma^{\lambda}\neq \emptyset,~|I_n|=\frac{\gamma}{4}|G_{\hat{k}_n}|^{\kappa}~and~dist(I_n,I_{n-1}^c)>\frac{\gamma}{2}|G_{\hat{k}_n}|^{\kappa};$$
$$dist(I_n, G_{k_j})\geq \gamma|G_{k_j}|^{\kappa},~|\hat{k}_n|\geq |k_j|\geq |\hat{k}_{n-1}|.$$
Set $$|\hat{k}_{n+1}|=\min{\{k_j \vert  |k_j|\geq |\hat{k}_{n}|, \tilde{G}_{k_j,\kappa,\gamma} \bigcap I_n\neq \emptyset \}};$$


$$I_{n+1}:=(t^{\hat{k}_{n+1}}_++\frac{\gamma}{2}|G_{\hat{k}_{n+1}}|^{\kappa},t^{\hat{k}_{n+1}}_++\frac{3\gamma}{4}|G_{\hat{k}_{n+1}}|^{\kappa}).$$






By the same argument as the previous case, we have
$$|\hat{k}_{n+1}|\gg |\hat{k}_n|;$$
$$I_{n+1}\bigcap \Sigma^{\lambda}\neq \emptyset,~|I_{n+1}|=\frac{\gamma}{4}|G_{\hat{k}_{n+1}}|^{\kappa}~and~dist(I_{n+1},I_n^c)>\frac{\gamma}{2}|G_{\hat{k}_{n+1}}|^{\kappa};$$
$$dist(I_{n+1}, G_{k_j})\geq \frac{\gamma}{2}|G_{k_j}|^{\kappa},~|\hat{k}_{n+1}|\geq |k_j|\geq |\hat{k}_{n}|.$$

\textbf{Finishing the induction.}

By induction as above, we define a sequence of $\{\hat{k}_n\}_{n\geq 0}$ and $\{I_{n}\}_{n\geq 0}$ satisfying
\begin{equation}\label{zhiz}|\hat{k}_{n+1}|\gg |\hat{k}_n|;\end{equation}
\begin{equation}\label{bhgx}I_{n+1}\subset I_n;\end{equation} \begin{equation}\label{puf}~I_{n+1}\bigcap \Sigma^{\lambda}\neq \emptyset;\end{equation}~\begin{equation}\label{quc}|I_{n+1}|=\frac{\gamma}{4}|G_{\hat{k}_{n+1}}|^{\kappa};\end{equation}~
\begin{equation}\label{jul}dist(I_{n+1}, G_{k_j})\geq \frac{\gamma}{2}|G_{k_j}|^{\kappa},~|\hat{k}_{n+1}|\geq |k_j|\geq |\hat{k}_{n}|.\end{equation}
\begin{equation}\label{jul'}\frac{3\gamma}{4}|G_{\hat{k}_{n+1}}|^{\kappa}\geq dist(I_{n+1}, G_{\hat{k}_{n+1}})\geq \frac{\gamma}{2}|G_{\hat{k}_{n+1}}|^{\kappa}.\end{equation}

Note that \eqref{bhgx} implies
$dist(I_{n+1}, G_{k_j})\geq dist(I_{n}, G_{k_j}),~ n\geq 0.$
Then \eqref{jul} implies
$$dist(I_{n+1}, G_{k_j})\geq dist(I_{n}, G_{k_j})\geq \cdots \geq dist(I_{l}, G_{k_j})\geq \frac{\gamma}{2}|G_{k_j}|^{\kappa},~|\hat{k}_{l}|\geq |k_j|\geq |\hat{k}_{l-1}|,~{\rm \ for\ any\ }~ 0\leq l\leq n.$$

Therefore \begin{equation}\label{jula}dist(I_{n+1}, G_{k_j})\geq \frac{\gamma}{2}|G_{k_j}|^{\kappa},~|\hat{k}_{n+1}|\geq |k_j|\geq |\hat{k}_{0}|.\end{equation}

By the principle of nested intervals,~\eqref{zhiz},\ \eqref{bhgx}, \eqref{puf},\ \eqref{quc} and \eqref{jula} yield that there exists a unique $t\in \Sigma^{\lambda}$ such that
$$dist(t, G_{k_j})\geq \frac{\gamma}{2}|G_{k_j}|^{\kappa},~|\hat{k}_{n+1}|\geq |k_j|\geq |\hat{k}_{0}|,~{\rm \ for\ any\ } n\geq 0.$$

Hence \eqref{zhiz} implies
\begin{equation}\label{xiajie}dist(t, G_{k_j})\geq \frac{\gamma}{2}|G_{k_j}|^{\kappa},~|k_j|\geq |\hat{k}_{0}|.\end{equation}

And \eqref{jul'} implies
\begin{equation}\label{shangjie}\frac{3\gamma}{4}|G_{\hat{k}_{n+1}}|^{\kappa}\geq dist(t, G_{\hat{k}_{n+1}})\geq \frac{\gamma}{2}|G_{\hat{k}_{n+1}}|^{\kappa},~{\rm \ for\ any\ } n\in \Z_+.\end{equation}

Taking $\gamma=4$ in \eqref{xiajie} and \eqref{shangjie}, we obtain \eqref{xiajie*} and \eqref{shangjie*} as desired.

\

{\bf The proof of the existence of $t\in \Sigma^{\lambda}$ satisfying $\frac{1}{2}<\beta(t)<1$}

Taking $\kappa=\frac{1}{2\beta-1}$ in \eqref{xiajie*} and \eqref{shangjie*}, (1) of Lemma \ref{betat} immediately yields that $\beta(t)=\frac{\frac{1}{2\beta-1}+1}{\frac{2}{2\beta-1}}=\beta.$

\

{\bf The proof of the existence of $t\in \Sigma^{\lambda}$ satisfying $\beta(t)=1$}

Taking $\kappa=\frac{1}{2}(~or~any~other~0<\kappa<1)$ in \eqref{xiajie*}, (2) of Lemma \ref{betat} immediately yields that $\beta(t)=1.$
\hfill\qed\end{proof}
\noindent
\vskip 0.3cm
Recall that $$\mathcal{K}(\lambda)=\{k_i\}_{i\in \Z_+}=\{k\in \Z\vert G_k~is~open\};$$
$$\xi_{k_i}=q^{-\frac{1}{4}}_{N+s(k_i)-1};$$
$$ \mathcal{B}_{\xi_{k_i}}^{k_i}=(t^{k_i}_--|G_{k_i}|^{1+\xi_{k_i}},t^{k_i}_++|G_{k_i}|^{1+\xi_{k_i}}),$$

$$\beta(t)\triangleq\left\{\begin{matrix}\liminf_{n\rightarrow +\infty} \beta_n(t) & t\notin \bigcup\limits_{i\in \Z_+}\{t^{k_i}_-,t^{k_i}_+\}\\ \frac{1}{2}& t\in \bigcup\limits_{i\in \Z_+}\{t^{k_i}_-,t^{k_i}_+\}\end{matrix}\right.$$ and
$$\beta_i(t)=\min\{\frac{1}{2}+\frac{\log|\bar{t}-t^{k_i}_-|}{2 \log |\bar{t}-t^{k_i}_+|},\frac{1}{2}+\frac{\log|\bar{t}-t^{k_i}_+|}{2 \log |\bar{t}-t^{k_i}_-|}\}.$$

In the following proof, we set
$$ \mathcal{B}_{\xi^*_{k_i}}^{k_i}=(t^{k_i}_--|G_{k_i}|^{1+\xi^*_{k_i}},t^{k_i}_++|G_{k_i}|^{1+\xi^*_{k_i}})$$ with
$\xi^*_{k_i}=\xi^{\frac{1}{2}}_{k_i}=q^{-\frac{1}{8}}_{N+s(k_i)-1}.$
Clearly, $ \mathcal{B}_{\xi^*_{k_i}}^{k_i}\subset \mathcal{B}_{\xi_{k_i}}^{k_i}.$

\begin{lemma}\label{lem43}
Given $\bar{t}\in \Sigma^{\lambda},$ we have
\begin{enumerate}

\item If $\bar{t}\notin \mathcal{B}^{k_i}_{\xi^*_{k_i}},$ then $\beta_{k_i}(\bar{t})>1-\lambda^{-q^{\frac{3}{2}}_{N+s(k_i)-1}}.$

\item if $\beta(\bar{t})<1,$ then $\bar{t}\in \bigcap\limits_{j\geq 1}\bigcup\limits_{i\geq j} \mathcal{B}_{\xi^*_{k_i}}^{k_i}.$

\end{enumerate}
\end{lemma}

\begin{proof}

\textbf{The proof of (1):}

\

$\bar{t}\notin \mathcal{B}^{k_i}_{\xi^*_{k_i}}$ implies $\bar{t}<t^{k_i}_--|G_{k_i}|^{1+\xi^*_{k_i}}$ or $\bar{t}>t^{k_i}_++|G_{k_i}|^{1+\xi^*_{k_i}}.$

By symmetry, we only show the former case. By a direct computation, we have

$$\begin{array}{ll}&\frac{1}{2}+\frac{\log|\bar{t}-t^{k_i}_-|}{2 \log |\bar{t}-t^{k_i}_+|}>\frac{1}{2}+\frac{\log|\bar{t}-t^{k_i}_+|}{2 \log |\bar{t}-t^{k_i}_-|}= \frac{1}{2}+\frac{\log |\bar{t}-t^{k_i}_-+|G_{k_i}||}{2 \log|\bar{t}-t^{k_i}_-|}\\&\geq \frac{1}{2}+\frac{\log (|G_{k_i}|(1+|G_{k_i}|^{\xi^*_{k_i}}))}{2\log||G_{k_i}|^{1+\xi^*_{k_i}}|}(~\text{by}~Lemma~\ref{usefullem})\\&=\frac{1}{2}+\frac{\log |G_{k_i}|+\log (1+|G_{k_i}|^{\xi^*_{k_i}})}{2(1+\xi^*_{k_i})\log|G_{k_i}|}
\\&=1+\frac{-\xi^*_{k_i}\log |G_{k_i}|+\log (1+|G_{k_i}|^{\xi^*_{k_i}})}{2(1+\xi^*_{k_i})\log|G_{k_i}|} \\&> 1-\xi^*_{k_i}=1-q^{-\frac{1}{8}}_{N+s(k_i)-1}.\end{array}$$

Therefore we have

$$\beta_{k_i}(\bar{t})=\min\{\frac{1}{2}+\frac{\log|\bar{t}-t^{k_i}_-|}{2 \log |\bar{t}-t^{k_i}_+|},\frac{1}{2}+\frac{\log|\bar{t}-t^{k_i}_+|}{2 \log |\bar{t}-t^{k_i}_-|}\}>1-q^{-\frac{1}{8}}_{N+s(k_i)-1}.$$

\

\textbf{The proof of (2):}

Suppose $\bar{t}\notin \bigcap\limits_{j\geq 1}\bigcup\limits_{i\geq j} \mathcal{B}_{\xi^*_{k_i}}^{k_i}.$ Then
 $\bar{t}\in \bigcup\limits_{j\geq 1}\bigcap\limits_{i\geq j} (\R-\mathcal{B}_{\xi^*_{k_i}}^{k_i}).$
Hence there exists $i_0\in \Z_+$ such that
$\bar{t}\notin \mathcal{B}_{\xi^*_{k_i}}^{k_i},~~i\geq i_0.$

By (1), we have already obtained
$$\beta_{k_i}(\bar{t})>1-q^{-\frac{1}{8}}_{N+s(k_i)-1}.$$

Note $\bar{t}\notin \mathcal{B}_{\xi^*_{k_i}}^{k_i}$ implies $\bar{t}\notin \bigcup\limits_{i\geq i_0}\{t^{k_i}_-,t^{k_i}_+\},$ hence
$$\begin{array}{ll}&\beta(\bar{t})=\liminf_{n\rightarrow +\infty}\beta_n(\bar{t})=\liminf_{n\rightarrow +\infty}\min\{\frac{1}{2}+\frac{\log|\bar{t}-t^{k_i}_-|}{2 \log |\bar{t}-t^{k_i}_+|},\frac{1}{2}+\frac{\log|\bar{t}-t^{k_i}_+|}{2 \log |\bar{t}-t^{k_i}_-|}\}\\&\geq \liminf_{n\rightarrow +\infty}(1-C\lambda^{-cq^{\frac{7}{4}}_{N+s(k_i)-1}})=1.\end{array}$$
This conflicts with the fact $\beta(\bar{t})<1.$
\hfill\qed\end{proof}

\



Let $\bar{t}$ satisfy $\frac12<\beta(\bar{t})<1$.  Then there must exist some $1\gg \delta_0(\bar{t})>0$ satisfying $\frac12+\delta_0<\beta(\bar{t})<1-\delta_0.$
 We need to show the following conclusion\\
$$\liminf\limits_{t\rightarrow \bar{t}}\frac{\log |L(t)-L(\bar{t})|}{\log |t-\bar{t}|}= \beta(\bar{t}).$$

\noindent

For this purpose, we need consider the following two parts $$\liminf\limits_{t\rightarrow \bar{t}}\frac{\log |L(t)-L(\bar{t})|}{\log |t-\bar{t}|}\geq \beta(\bar{t})\quad(~\bf Part_1)$$ and
$$\liminf\limits_{t\rightarrow \bar{t}}\frac{\log |L(t)-L(\bar{t})|}{\log |t-\bar{t}|}\leq \beta(\bar{t})\quad(~\bf Part_2).$$

\

Note for $t\in \mathcal{FR}$, it holds that $$|\left\{k\in \{k_i\}| t\in \mathcal{B}_{\xi^*_k}^{k}\right\}|=+\infty.$$ Then we denote
$$K_1:=\left\{k\in \{k_i\}_{i\in \Z_+}| t\in \mathcal{B}_{\xi^*_{k}}^{k}\right\}=\{\hat{k}_1(\bar{t}),\hat{k}_2(\bar{t}),\cdots, \hat{k}_i(\bar{t}),\cdots\}$$ and $$K_2=\{k_i\}_{i\in \Z_+}\backslash K_1.$$

\

If $|K_2|=\infty$,  then (1) of Lemma \ref{lem43} implies
$$\liminf_{i\rightarrow +\infty; k_i\in K_2}\beta_{k_i}(\bar{t})=1>1-\delta_0>\beta(\bar{t}).$$

Therefore
\begin{equation}\label{kil}\beta(\bar{t})=\liminf_{i\rightarrow +\infty; k_i\in K_1} \beta_{k_i}(\bar{t})=\liminf_{i\rightarrow +\infty} \beta_{{k}_i}(\bar{t}).\end{equation}

If $|K_2|<\infty,$ then \eqref{kil} automatically holds true.

Note \eqref{kil} implies for any subsequence $\{k_{n_j}\}\subset \{k_i\}_{i\in \Z_+}$ with $|k_{n_j}|\rightarrow +\infty,$ it holds that
$\liminf_{j\rightarrow +\infty} \beta_{k_{n_j}}(\bar{t})=\liminf_{j\rightarrow +\infty; k_{n_j}\in K_1} \beta_{k_{n_j}}(\bar{t})\geq \liminf_{i\rightarrow +\infty} \beta_{{k}_i}(\bar{t})\geq \beta(\bar{t}).$

\


\textbf{The proof of $\bf Part_1$:}

\

For any sequence $t_n\rightarrow \bar{t},$ we denote

\begin{equation}\label{sn}p(n):=\max\{ i \vert t_n\in \mathcal{B}_{\xi^*_{\hat{k}_i}}^{{k}_i}(\bar{t})\}.\end{equation}

Here $p(n)$ is well-defined for the following reason. If $t_n\neq \bar{t}$ and $|\{ i \vert t_n\in \mathcal{B}_{\xi^*_{\hat{k}_i}}^{{k}_i}(\bar{t})\}|=+\infty,$
then $\bar{\mathcal{B}}_{\xi^*_{{k}_i}}^{{k}_i}(\bar{t})\rightarrow \{t_n\}$ which,
 by the definition of ${k}_i(\bar{t}),$
  conflicts with
$\bar{\mathcal{B}}_{\xi^*_{{k}_i}}^{{k}_i}(\bar{t})\rightarrow \{\bar{t}\}.$

Therefore $|\{ i \vert t_n\in \mathcal{B}_{\xi^*_{{k}_i}}^{{k}_i}(\bar{t})\}|<+\infty$ for any $t_n\neq \bar{t},$ which implies \eqref{sn} is well-defined.

Note $t_n\rightarrow \bar{t}$ implies
\begin{equation}\label{pn}p(n)\rightarrow +\infty~as~n\rightarrow +\infty.\end{equation}

We denote $T_n:=(t_n,\bar{t})(~or~(\bar{t},t_n)).$

By the definition \eqref{sn}, clearly we have $T_n\subset \mathcal{B}_{\xi^*_{{k}_{p(n)}}}^{{k}_{p(n)}}.$
Since $$\lambda^{-q_{N+{k}_{p(n)}-1}}\geq \lambda^{-{k}^{\frac{1}{2}}_{p(n)}}\gg C\lambda^{-c{k}_{p(n)}}\geq |G_{{k}_{p(n)}}|,$$ we have~$$\begin{array}{ll}&\mathcal{B}_{\xi^*_{{k}_{p(n)}}}^{{k}_{p(n)}}\subset(t^{{k}_{p(n)}}_--|G_{{{k}_{p(n)}}}|^{1+\xi^*_{{{k}_{p(n)}}}},t^{{k}_{p(n)}}_++|G_{{{k}_{p(n)}}}|^{1+\xi^*_{{{k}_{p(n)}}}})\subset (t^{{k}_{p(n)}}_--|G_{{{k}_{p(n)}}}|,t^{{k}_{p(n)}}_++|G_{{{k}_{p(n)}}}|)\\&\subset (t^{{k}_{p(n)}}_--\lambda^{-q_{N+{k}_{p(n)}-1}},t^{{k}_{p(n)}}_++\lambda^{-q_{N+{k}_{p(n)}-1}})= \mathcal{B}^{{{k}_{p(n)}}}.\end{array}$$
Hence $T_n\subset \mathcal{B}^{{k}_{p(n)}}.$

By the help of Lemma \ref{homg}, for any $n\in \Z_+,$ there exists $|k^*_n|\geq {k}_{p(n)}$ such that
$$T_n\subset 2\mathcal{B}^{k^*_n}{\rm\ and\ } T_n-\bigcup\limits_{|k_i|>k_n^*}\mathcal{B}^{k_i}\neq \emptyset.$$

Then \eqref{gap_uniform6'} of Lemma \ref{c3lemma} shows that\begin{equation}\label{c31}\left\vert L(t)-L(\bar{t})\right\vert\leq 6C_{k^*_n}\int_{T_n}\vert H(k^*_n,x)\vert dx+\int_{T_n}\Xi(k^*_n,x) dx\leq (6C_{k^*_n}+1)\int_{T_n} (|H(k^*_n,x)|+\Xi(k^*_n,x)) dx.\end{equation}

On the other hand, \eqref{jiandanjf3} of Lemma \ref{jiandanjf1} implies

\begin{equation}\label{c32}\begin{array}{ll}&\int_{T_n} (|H(k^*_n,t)|+\Xi(k^*_n,t)) dt\leq \int_{\bar{t}-|\bar{t}-t_n|}^{\bar{t}+|\bar{t}-t_n|} (|H(k^*_n,t)|+\Xi(k^*_n,t)) dt \leq 2 (|t-\bar{t}|)^{\min\{\beta_{k^*_n}(\bar{t}),1-\xi^*_{k^*_n}\}}.\end{array}\end{equation}

Combining \eqref{c31} with \eqref{c32}, we have
$$\left\vert L(t_n)-L(\bar{t})\right\vert\leq  (100C_{k^*_n}+1)|t-\bar{t}|^{\min\{\beta_{k^*_n}(\bar{t}),1-\xi^*_{k^*_n}\}}.$$

Note $T_n\subset 2\mathcal{B}^{k^*_n(\bar{t})}$ implies $$|t_n-\bar{t}|\leq 2Leb\{\mathcal{B}^{k^*_n}\}\leq 6\lambda^{-q_{N+s(k^*_n)-1}}.$$
Recall that \begin{equation}\label{CK}C_{k^*_n}\leq (\log (k^*_n))^C\leq \lambda^{|{k}_{_{p(n)}}|^c}\leq \left(3\lambda^{q^{\frac{3}{4}}_{N+s(k^*_n)-1}}\right)\le\left(3\lambda^{-q_{N+s(k^*_n)-1}}\right)^{-10\xi^*_{k^*_n}}\leq |t_n-\bar{t}|^{-10\xi^*_{k^*_n}}.\end{equation}
Then
$$\left\vert L(t_n)-L(\bar{t})\right\vert\leq  (\log ({k}_i))^C|t_n-\bar{t}|^{\min\{\beta_{k^*_n}(\bar{t}),1-\xi^*_{k^*_n}\}}\leq |t_n-\bar{t}|^{-10\xi^*_{k^*_n}}||t_n-\bar{t}|^{\min\{\beta_{k^*_n}(\bar{t}),\beta(\bar{t})-\xi^*_{k^*_n}\}}\leq |t_n-\bar{t}|^{\beta_{k^*_n}(\bar{t})-20\xi^*_{k^*_n}}.$$

Therefore \eqref{pn} and \eqref{kil} imply for any $t_n\rightarrow t$, it holds that
\begin{equation}\label{fin}\liminf\limits_{n\rightarrow +\infty}\frac{\log |L(t_n)-L(\bar{t})|}{\log |t_n-\bar{t}|}\geq \liminf\limits_{n\rightarrow +\infty}\left(\beta_{k^*_n}(\bar{t})-20\xi^*_{k^*_n(\bar{t})}\right)=\liminf\limits_{n\rightarrow +\infty}\beta_{k^*_n}(\bar{t})\geq \beta(\bar{t}).\end{equation}

Then \eqref{fin} implies
$$\liminf\limits_{t\rightarrow \bar{t}}\frac{\log |L(t)-L(\bar{t})|}{\log |t-\bar{t}|}\geq \beta(\bar{t}).$$

\

\noindent
\textbf{The proof of $\bf Part_2$:}

\

Since $\bar{t}\in \mathcal{B}_{\xi^*_{{k}_i}}^{{k}_i}-G_{{k}_i}\subset \mathcal{B}^{{k}_i},$ we have
$\bar{t}<t^{{k}_i}_-~~or~~\bar{t}>t^{{k}_i}_+.$
We  construct a sequence  as follows.

\begin{enumerate}
\item If $\bar{t}< t^{{k}_i}_-,$ then we denote $t^*_{i}:=t^{{k}_i}_-+10^{10}|\bar{t}-t^{{k}_i}_-|.$
\item If $\bar{t}> t^{{k}_i}_+,$ then we denote $t^*_{i}:=t^{{k}_i}_+-10^{10}|\bar{t}-t^{{k}_i}_+|.$
\end{enumerate}

Since $$dist\{\bar{t},G_{{k}_i}\}<|G_{{k}_i}|^{1+\xi^*_{{k}_i}}\ll |G_{{k}_i}|^{1+\xi_{{k}_i}},$$ one has $t^*_i\in G_{{k}_i}$ and $$t^*_{i}\in (t^{{k}_i}_-,t^{{k}_i}_-+|G_{{k}_i}|^{1+\xi_{{k}_i}})~~if~~\bar{t}< t^{{k}_i}_-,$$ $$t^*_{i}\in (t^{{k}_i}_+-|G_{{k}_i}|^{1+\xi_{{k}_i}},t^{{k}_i}_+)~~if~~\bar{t}> t^{{k}_i}_+.$$

Recall \eqref{ep12'} and \eqref{ep12'*} imply
$$
\vert L(t)-L(t^{{k}_i}_+)\vert
\ge \frac{C_{{k}_i}}{2}\sqrt{|G_{{k}_i}|}|t-t^{{{k}_i}}_{+}|^{\frac{1}{2}},~ t\in (t^{{k}_i}_+-|G_{{k}_i}|^{1+\xi_{{k}_i}},t^{{k}_i}_+)
$$
and
$$
\vert L(t)-L(t^{{k}_i}_-)\vert
\ge \frac{C_{{k}_i}}{2}\sqrt{|G_{{k}_i}|}|t-t^{{{k}_i}}_{-}|^{\frac{1}{2}},~  t\in (t^{{k}_i}_-,t^{{k}_i}_-+|G_{{k}_i}|^{1+\xi_{{k}_i}})
.$$

On the other hand, by \eqref{shangjiehe} and \eqref{ep21} of Lemma \ref{ckt},
we have
$$
\vert L(t)-L(t^{{k}_i}_+)\vert
\le 24{C_{{k}_i}}\sqrt{|G_{{k}_i}|}|t-t^{{{k}_i}}_{+}|^{\frac{1}{2}},~  t\in (t^{{k}_i}_+-|G_{{k}_i}|^{1+\xi_{{k}_i}},t^{{k}_i}_+)
$$
and
$$
\vert L(t)-L(t^{{k}_i}_-)\vert
\le 24{C_{{k}_i}}\sqrt{|G_{{k}_i}|}|t-t^{{{k}_i}}_{-}|^{\frac{1}{2}},~  t\in (t^{{k}_i}_-,t^{{k}_i}_-+|G_{{k}_i}|^{1+\xi_{{k}_i}})
.$$

Therefore

\begin{enumerate}

\item if $t_i^*=t^{{k}_i}_-+10^{10}|\bar{t}-t^{{k}_i}_-|$ (corresponding with the case $\bar{t}< t^{{k}_i}_-$), since $C_{{k}_i}\geq |t-\bar{t}|^{10\xi_{{k}_i}}$,  by \eqref{CK} we have
\begin{equation}\begin{array}{ll}&\ \ \ \ |L(t_i^*)-L(\bar{t})|\geq |L(t_i^*)-L(t^{{k}_i}_-)|-|L(t^{{k}_i}_-)-L(\bar{t})|\\&\geq  \frac{1}{2}C_{{k}_i}\cdot \sqrt{|G_{{k}_i}}|\sqrt{10^{10}|\bar{t}-t_-^{{k}_i}|}-24C_{{k}_i}\sqrt{|G_{{{k}_i}}|}\sqrt{|\bar{t}-t_-^{{k}_i}|}\\&\geq 10^4 \cdot C_{{k}_i}\cdot \sqrt{|\bar{t}-t_-^{{k}_i}|\cdot |G_{{k}_i}|}  \geq 10^3 C_{{k}_i}\cdot \sqrt{|\bar{t}-t_-^{{k}_i}|\cdot |\bar{t}-t^{{k}_i}_+|}\\ &\geq 10^3 \cdot C_{{k}_i} |\bar{t}-t^{{k}_i}_-|^{\beta_{{k}_i}(\bar{t})}
\geq |\bar{t}-t^{{k}_i}_-|^{\beta_{{k}_i}(\bar{t})+20\xi_{{k}_i}}= \left(\frac{|t_i^*-\bar{t}|}{10^{10}}\right)^{\beta_{{k}_i}(\bar{t})+20\xi_{{k}_i}}.\end{array}
\end{equation}
Then $$\frac{\log |L(t^*_i)-L(\bar{t})|}{\log |t^*_i-\bar{t}|}\leq \left(\beta_{{k}_i}(\bar{t})+20\xi_{{k}_i}\right)(1-\frac{100\log 10}{\log |t^*_i-\bar{t}|} ).$$

\item The case $t_i^*=t^{{k}_i}_+-10^{10}|\bar{t}-t^{{k}_i}_+|$ (corresponding with the case $\bar{t}> t^{{k}_i}_+$) is similar as above.
\end{enumerate}

In summary, we obtain for any $i\in \Z_+,$
$$\frac{\log |L(t^*_i)-L(\bar{t})|}{\log |t^*_i-\bar{t}|}\leq \left(\beta_{{k}_i}(\bar{t})+20\xi_{{k}_i}\right)(1-\frac{100\log 10}{\log |t^*_i-\bar{t}|} ).$$

Then it follows from $$\liminf_{i\rightarrow+\infty}\left(\beta_{{k}_i}(\bar{t})+20\xi_{{k}_i}\right)(1-\frac{100\log 10}{\log |t^*_i-\bar{t}|} )= \liminf_{i\rightarrow+\infty}\beta_{{k}_i}(\bar{t})=\beta(\bar{t})$$ that
$$\liminf_{i\rightarrow+\infty}\frac{\log |L(t^*_i)-L(\bar{t})|}{\log |t^*_i-\bar{t}|}\leq \beta(\bar{t}).$$

Hence
$$\liminf\limits_{t\rightarrow \bar{t}}\frac{\log |L(t)-L(\bar{t})|}{\log |t-\bar{t}|}\leq \beta(\bar{t}).$$

\

Combining Theorem \ref{zhibiaoeq} and the arguments above, we complete the proof of (4) of Theorem \ref{Th1}.

\section{Some preparation for the proof of Lemma \ref{lemma 6.1}}\label{section5.1}
In this section, we give a technical lemma (Lemma \ref{niceset1}) which is crucial for the proof of Lemma \ref{lemma 6.1}.
Let $t,v,\alpha$ be as in Theorem 2. For any $n>i\in \Z_+, x\in \R/\Z,$ if $\|A_{i}(x,t)\|,\|A_{n-i}(x,t)\|>1,$  we define \begin{equation}\label{theta-up-low}\theta^n_{i}(x,t)\triangleq u(A_{i}(x,t))-s(A_{n-i}(x,t)).\end{equation}

For simplification, sometimes we omit the dependence of $\theta_i^n$ on $t$ in the remaining part of this paper.



\

\begin{definition}\label{lemma29} Given
$D\subset \R/\Z,$ which consists of union of some intervals, $n\in \Z_+$ and $~\eta_n\ll 1$, for some $~t\in [\inf v-\frac{2}{\lambda},\sup v+\frac{2}{\lambda}],$ we say $A_n(x,t)$ is \textbf{${(\eta_n,+)-nice}$} on $D$, if there exists some universal constant $C>0$ (independent on $n$ and $m$) such that

\begin{equation}\label{result2}\min\limits_{x\in D}\|A_n(x,t)\|\geq \lambda^{\frac{9}{10}n},\end{equation}
\begin{equation}\label{result3}\int_{x\in D}\left\vert\frac{\partial_t \|A_n(x,t)\|}{\|A_n(x,t)\|}\right\vert dx \leq  e^{|\log \eta_n|^{C}}, \end{equation}
\begin{equation}\label{result4}\int_{x\in D}\left\vert\frac{\partial_x \|A_n(x,t)\|}{\|A_n(x,t)\|}\right\vert dx \leq e^{|\log \eta_n|^{C}},\end{equation}
and
\begin{equation}\label{result1}\max\limits_{x\in D}\sum\limits_{p=0}^2|(s(A_n(x,t))-s_{r_m}(x,t))^{(p)}|\leq \|A_n\|^{-2} e^{|\log \eta_n|^{\hat{\epsilon}^{-2}}}.\end{equation}
($\hat{\epsilon}$ comes from \eqref{HN})

 \label{lemma6.1}

${(\eta_n,-)-nice}$ are defined similarly to above by replacing $n$ with $-n$ and $r_m$ with $-r_m.$

 \end{definition}
For $w\in \N$ and $k\in \Z,$ we set $$f_w(x)=\sum\limits_{s(k)+1\leq w}\frac{1}{\sqrt{\lambda^{-8|k|}+(x-c_{w,1})^2}}.$$
Given $n\in \N,$ let $j_y, 1\le y\le m$ be all the forward returning time of $x$  back to $I_{n-1}-I_{n}$ after $q_{N+n-2}^2-1.$ We will show the following result.
\begin{lemma}\label{niceset0} Given $n\in \N,$ $X\in \{+,-\}$ and $x\in I_{n,1},~t\in  [\inf v-\frac{2}{\lambda},\sup v+\frac{2}{\lambda}],$ it holds that for $y,z\in \{1,2,\cdots,m\}$
\begin{equation}\label{66-0} \|A_{j_z}\|\geq \lambda^{(1-c)j_z},
\end{equation}
\begin{equation}\label{66-1}\frac{\|A_{j_z}\|^{(1)}}{\|A_{j_z}\|}\leq \left(e^{ (\log j_z)^{C}}+f_n(x)\right)|I_n|^{-C}\end{equation}
and
\begin{equation}\label{66-2}\sum\limits_{l=0}^2|(s(A_{r^{X}_n(x)})-s(A_{j_y}))^{(l)}|\leq \|A_{j_y}\|^{-2} \left(e^{ (\log j_y)^{C}}+f_n(x)\right)|I_n|^{-C}.\end{equation}
\end{lemma}

\begin{proof} \eqref{66-0} directly follows from Theorem \ref{theorem12}. We prove the remaining one by induction. $n=0$ is trivial. For $n\in \Z_+,$ by symmetry, we only show the case $X=+.$  For convenience, here we assume that $z=m$ (i.e.~$j_z=r^{+}_n(x).$) We have to separately consider the following several cases. Set $\epsilon_n=\sum\limits_{j\leq n}j^{-2},~n\in \Z_+;~\epsilon_0=0.$

Set the inductive hypothesi as follow. For $w\leq n-1,$ it holds that
$\|A_{r^+_w}\|\geq \lambda^{(1-\epsilon_w)j_z}$ and $\frac{\|A_{r^+_w}\|^{(1)}}{\|A_{r^+_w}\|}\leq \left(e^{(\log r^+_w)^{(1+\epsilon_w)C}}+f_w(x)\right)|I_w|^{-C}.$

\begin{enumerate}
\item Step $n$ belongs to Type \textbf{I}. We consider $$A_{j_m}(x)=A_{j_{m}-j_{m-1}}A_{j_{m-1}-j_{m-2}}\cdots A_{j_1}(x).$$ Then by the inductive hypothesis, for $y=1,2,\cdots,m,$ $$\|A_{j_y-j_{y-1}}\|\geq \lambda^{(1-\epsilon_n)(j_y-j_{y-1})}\gg |I_{n-1}|^{-C};$$ $$\frac{\|A_{j_y-j_{y-1}}\|^{(1)}}{\|A_{j_y-j_{y-1}}\|}\leq \left(e^{(\log r^+_{n-1})^{(1+\epsilon_{n-1})C}}+f_{n-1}(x)\right)|I_{n-1}|^{-C}\ll |I_{n}|^{-C}.$$
    Hence there exists $\mu\geq \lambda^{(1-\epsilon_n)(r^+_{n-1})}$ such that $$\|A_{j_y-j_{y-1}}\|\sim_{0,|I_n|^c} \mu;~$$
    It follows from Theorem \ref{theorem12} (which implies the non-degeneracy of $\theta_{{j_{y-1}-i_{v-2}}}^{j_{y}-j_{y-2}}$) and the fact $\text{dist}\{x+j_y\alpha, I_n\}>0$ that
  $$~|\theta_{{j_{y-1}-i_{v-2}}}^{j_{y}-j_{y-2}}|\geq |I_{n}|^C$$ and $$\sum\limits_{l=1}^2|\left(\theta_{{j_{y-1}-j_{y-2}}}^{j_{y}-j_{y-2}}\right)^{(l)}|\leq |I_{n-1}|^{-C}.$$
    Then by repeatedly using Lemma \ref{lemma8} and inductive hypothesis we have \begin{equation}\label{case111}\begin{array}{ll}\sum\limits_{l=1}^2\frac{\|A_{r^{+}_n(x)}\|^{(l)}}{\|A_{r^{+}_n(x)}\|}&\leq (\sum\limits_{v=1}^w e^{[\log (j_y-j_{y-1})]^{(1+\epsilon_{n-1})C}}+f_{n-1}(x))(|I_{n-1}|^{-C}+|I_{n}|^{-C})\\&\leq \left(e^{(\log r^+_n)^{(1+\epsilon_n)C}}+f_{n-1}(x)\right)|I_{n}|^{-C}\\& \leq \left(e^{(\log r^+_n)^{(1+\epsilon_n)C}}+f_{n-1}(x)\right)|I_{n}|^{-C}.\end{array}\end{equation}

\item Step $n$ belongs to $\textbf{II}_{k'}$ with $k'\in \Z$ satisfying $I_{n,1}+k'\alpha \bigcap I_{n,2} \neq \emptyset$ with $1\leq |k'|<q^2_{N+n-1}.$ Then we consider $A=A_{r_n^+-k'}A_{k'}.$ For $A_{k'}$ and $A_{r_n^+-k'},$ similarly to previous condition, we have $$\frac{\|A_{k'}\|^{(1)}}{\|A_{k'}\|}\leq \left(e^{(\log k')^{(1+\epsilon_n)k'}}+f_{n-1}(x)\right)|I_{n}|^{-C},$$ $$\frac{\|A_{r^+_n-k}\|^{(1)}}{\|A_{r^+_n-k'}\|}\leq \left(e^{(\log (r^+_n-k'))^{(1+\epsilon_n)C}}+f_{n-1}(x)\right)|I_{n}|^{-C},$$  $$\|A_{k'}\|\geq \lambda^{(1-\epsilon_n)k'},~\|A_{r^+_n-k'}\|\geq \lambda^{(1-\epsilon_n)(r^+_n-k')},$$ $$\sum\limits_{l=1}^2|\left(\theta^{r^+_n}_{k'}\right)^{(l)}|\leq |I_{n}|^{-C}.$$
    Then combining the above and Lemma \ref{lemma8} implies\begin{equation}\label{case2222}\begin{array}{ll}&\sum\limits_{l=1}^2\frac{\|A_{r^{+}_n(x)}\|^{(l)}}{\|A_{r^{+}_n(x)}\|}\\&\leq 2\left(e^{(\log (r^+_n))^{(1+\epsilon_n)C}}+f_{n-1}(x)\right)|I_{n}|^{-C}+|W(\|A_{r_n^+}\|,\|A_{r_n^+-k'}\|,\theta_{k'}^{r_n^+})|\cdot |\left(\theta_{k'}^{r_n^+}\right)^{(1)}|\\&\leq 2\left(e^{(\log (r^+_n))^{(1+\epsilon_n)C}}+f_{n-1}(x)\right)|I_{n}|^{-C}+\frac{|I_{n}|^{-C}}{\sqrt{\lambda^{-4k'}+\tan^2(\theta^{r_n^+}_{k'})}}.
    \end{array}\end{equation}
    Note in this case, we have $n\geq s(k')+1,$ hence \eqref{case2222} implies
    \begin{equation}\label{case222}\sum\limits_{l=1}^2\frac{\|A_{r^{+}_n(x)}\|^{(l)}}{\|A_{r^{+}_n(x)}\|}\leq 2\left(e^{(\log (r^+_n))^{(1+\epsilon_n)C}}+f_{n}(x)\right)|I_{n}|^{-C}.\end{equation}

\end{enumerate}

 By \eqref{case111} and \eqref{case222} we completes the proof of \eqref{66-1}.

 For \eqref{66-2}, note that by induction as above, we have obtained $|\theta^{r^{+}_n}_{j_y}|\geq |I_n|^C,$ $\sum\limits_{l=1}^2|\left(\theta^{r^+_n}_{j_y}\right)^{(l)}|\leq |I_{n}|^{-C},$ $\|A_{j_y-j_{y-1}}\|\geq \lambda^{(1-\epsilon_n)(j_y-j_{y-1})}\gg |I_{n}|^{-C},$ and $\sum\limits_{l=1}^2\frac{\|A_{j_y}\|^{(1)}}{\|A_{j_y}\|}\leq \left(e^{(\log r^+_{n-1})^{(1+\epsilon_{n-1})C}}+f_{n}(x)\right)|I_{n}|^{-C}.$

Then by the help of remark \ref{remark10}, we have $$\sum\limits_{0}^2\left\vert(s(A_{r^{X}_n(x)})-s(A_{j_y}))^{(l)}\right\vert \leq n^2\|A_{j_y}\|^{-2} \left(e^{ (\log j_y)^{C}}+f_n(x)\right)|I_n|^{-C}\leq \|A_{j_y}\|^{-2} \left(e^{ (\log j_y)^{C}}+f_n(x)\right)|I_n|^{-2C}$$ as desired.
\hfill\qed\end{proof}

More precisely we have,
\begin{lemma}\label{niceset0'} Let everything be defined as in Lemma \ref{niceset0}. Given $y\in \{1,2,\cdots,m-1\}$ and $1\leq b \leq j_{y+1}-j_y,$ it holds that
\begin{equation}\label{66-0'} \|A_{j_y+b}\|\geq \lambda^{(1-\epsilon)(j_y+b)},
\end{equation}
\begin{equation}\label{66-1'}\frac{\|A_{j_y+b}\|^{(1)}}{\|A_{j_y+b}\|}\leq \left(e^{(\log (j_y+b))^{C}}+f_n(x)\right)|I_n|^{-C\hat{\epsilon}^{-1}}\end{equation}
and
\begin{equation}\label{66-2'}\sum\limits_{l=0}^2|(s(A_{r^{X}_n(x)})-s(A_{j_y+b}))^{(l)}|\leq \|A_{j_y+b}\|^{-2}\left(e^{(\log (j_y+b))^{C}}+f_n(x)\right) |I_n|^{-C\hat{\epsilon}^{-1}}.\end{equation}
\end{lemma}

\begin{proof} We prove it by induction. Consider $$A_{j_y+b}(x)=A_{b}(x+j_y\alpha)A_{j_y}(x).$$ $n=0$ is trivial. For $n-1\in \N,$ we assume that everything holds true.
Let $n^*$ satisfy that
$$r^+_{n^*-1}(x+j_y\alpha)\leq b\leq r^+_{n^*}(x+j_y\alpha).$$ Note that $n^*\leq n-1.$
Let $j_{y+1}$ satisfy $j_{y+1}-j_{y}=r^+_{n-1}(x+j_{y}\alpha).$
 By inductive hypothesis we have \begin{equation}\label{66-01}\|A_b(x+j_y\alpha)\|\geq \lambda^{(1-\epsilon_{n-1})b};~\frac{\|A_b(x+j_y\alpha)\|^{(1)}}{\|A_b(x+j_y\alpha)\|}\leq \left(e^{(\log b)^{(1+\epsilon_{n-1})C}}+f_{n-1}(x)\right)|I_{n^*}|^{-C\hat{\epsilon}^{-1}};\end{equation}
\begin{equation}\label{66-001}\sum\limits_{l=0}^2|(s(A_{j_{y+1}-j_y})-s(A_{b}))^{(l)}|\leq \|A_{b}\|^{-2} \left(e^{(\log b)^{(1+\epsilon_{n-1})C}}+f_{n-1}(x)\right)|I_{n^*}|^{-(1+\epsilon_{n-1})C\hat{\epsilon}^{-1}}.\end{equation}
By \eqref{66-001} we have
\begin{equation}\label{66-03}\sum\limits_{l=1}^2\vert(\theta^{j_{y+1}}_{j_y}-\theta^{j_{y}+b}_{j_y})^{(l)}\vert\leq \|A_{b}\|^{-2}e^{(\log b)^{2C}}|I_{n^*}|^{-2C\hat{\epsilon}^{-1}}.\end{equation}
By Lemma \ref{niceset0} we have
\begin{equation}\label{66-02}\|A_{j_y}(x)\|\geq \lambda^{(1-\epsilon_{n-1})j_y};~\frac{\|A_{j_y}(x)\|^{(1)}}{\|A_{j_y}(x)\|}\leq \left(e^{(\log j_y)^{C}}+f_{n-1}(x)\right)|I_{n}|^{-C\hat{\epsilon}^{-1}}.\end{equation}
\begin{equation}\label{66-002}\sum\limits_{l=0}^2|(s(A_{r^+_n})-s(A_{j_{y}}))^{(l)}|\leq \|A_{j_{y}}\|^{-2} |I_{n}|^{-(1+\epsilon_{n-1})C}e^{(\log  j_y)^C}.\end{equation}

Next, we have to consider the following two cases.
\begin{enumerate}
\item $\|A_b(x+j_y\alpha)\|\leq |I_{n-1}|^{-\hat{\epsilon}^{-1}}.$ Then $\|A_b\|\leq |I_{n-1}|^{\hat{\epsilon}^{-1}}\ll c\lambda^{\frac{1}{2}r_{n-1}^+}\leq \|A_{j_y}\|.$ Note that by Theorem \ref{theorem12}, \begin{equation}\label{them1212}\sum\limits_{l=1}^2\vert \left(\theta^{j_{y+1}}_{j_y}\right)^{(l)}\vert\leq |I_n|^{-C}.\end{equation} Then \eqref{66-03} implies $\sum\limits_{l=1}^2\vert \left(\theta^{j_y+b}_{j_y}\right)^{(l)}\vert\leq 2|I_n|^{-C}.$ Then by \eqref{66-01}, \eqref{66-02} and Lemma \ref{lemma8} (similar analysis to \eqref{case2222}) we get \begin{equation}\label{moddd0}\|A_{j_y+b}\|\geq \|A_{j_y}\|\|A_{b}\|^{-1}\geq \lambda^{(1-\epsilon_n)(j_y+b)};\end{equation}
    \begin{equation}\label{moddd}\frac{\|A_{j_y+b}\|^{(1)}}{\|A_{j_y+b}\|}\leq  \left(e^{(\log (j_y+b))^{ (1+\epsilon_n)C}}+f_n(x)\right)|I_n|^{-C{\hat{\epsilon}^{-1}}}.\end{equation}
    On the other hand, (ii)-(2a) of Lemma \ref{lemma8*} implies \begin{equation}\label{jiaoddd'}\begin{array}{ll}&\sum\limits_{l=1}^2|(s(A_{j_{y}})-s(A_{j_{y}+b}))^{(l)}|\\&\leq \|A_{j_{y}}\|^{-2} \|A_b\|^2e^{(\log b)^c}\leq \|A_{j_{y}+b}\|^{-2} \|A_b\|^{C}\leq \|A_{j_{y}+b}\|^{-2} e^{(\log (j_y+b))^{(1+\epsilon_n)C}}|I_n|^{-(1+\epsilon_n)C{\hat{\epsilon}^{-1}}}.\end{array}\end{equation}
    Then \eqref{66-002} and \eqref{jiaoddd'} implies
    \begin{equation}\label{jiaoddd}\sum\limits_{l=1}^2|(s(A_{r^+_n})-s(A_{j_{y}+b}))^{(l)}|\leq \|A_{j_{y}+b}\|^{-2} e^{(\log (j_y+b))^{(1+\epsilon_n)C}}|I_n|^{-(1+\epsilon_n)C{\hat{\epsilon}^{-1}}}.\end{equation}

\item $\|A_b(x+j_y\alpha)\|> |I_{n-1}|^{-\hat{\epsilon}^{-1}}.$ Then \eqref{66-03} implies
$$\sum\limits_{l=1}^2\vert(\theta^{j_{y+1}}_{j_y}-\theta^{j_{y}+b}_{j_y})^{(l)}\vert\leq |I_{n-1}|^{\hat{\epsilon}^{-1}}\ll |I_{n}|^{C}.$$ Thus
\begin{equation}\label{xxxjjj}|\theta^{j_{y}+b}_{j_y}|\geq |\theta^{j_{y+1}}_{j_y}|-|I_{n-1}|^{\hat{\epsilon}^{-1}}>|I_{n}|^{-C}-|I_{n-1}|^{\hat{\epsilon}^{-1}}>c|I_{n}|^{-C}.\end{equation} Note \eqref{them1212}, follows from Theorem \ref{theorem12}, still holds true. In summary, by the help of \eqref{66-01}, \eqref{66-02} and (i) of Lemma \ref{lemma9*} we get \begin{equation}\label{moddd''}\|A_{j_y+b}\|\geq \lambda^{(1-\epsilon_n)(j_y+b)}.\end{equation} and
\begin{equation}\label{moddd'}\sum\limits_{l=1}^2\frac{\|A_{j_y+b}\|^{(l)}}{\|A_{j_y+b}\|}\leq \left(e^{(\log (j_y+b))^{(1+\epsilon_n)C}}+f_{n-1}(x)\right)|I_n|^{-(1+\epsilon_n)C{\hat{\epsilon}^{-1}}}\leq \left(e^{(\log (j_y+b))^{(1+\epsilon_n)C}}+f_{n}(x)\right)|I_n|^{-(1+\epsilon_n)C{\hat{\epsilon}^{-1}}}.\end{equation}
On the other hand, \eqref{66-01}, \eqref{66-02}, \eqref{them1212} and \eqref{xxxjjj} imply \begin{equation}\label{jiaoddd''}\begin{array}{ll}&\sum\limits_{l=1}^2|(s(A_{j_{y}})-s(A_{j_{y}+b}))^{(l)}|\\&\leq \|A_{j_{y}}\|^{-2} (\left(e^{(\log b)^{(1+\epsilon_{n})C}}+f_{n}(x)\right)|I_{n^*}|^{-(1+\epsilon_{n})C\hat{\epsilon}^{-1}}+e^{(\log j_y)^{C}}|I_{n}|^{-C\hat{\epsilon}^{-1}}+|I_n|^{-C\hat{\epsilon}^{-1}})\\&\leq \|A_{j_{y}+b}\|^{-2} \left(e^{(\log b)^{(1+\epsilon_{n})C}}+f_{n}(x)\right)|I_n|^{-(1+\epsilon_n)C\hat{\epsilon}^{-1}}.\end{array}\end{equation}
Then \eqref{66-002} and \eqref{jiaoddd''} implies
\begin{equation}\label{jiaoddd'''}\sum\limits_{l=1}^2|(s(A_{r^+_n})-s(A_{j_{y}+b}))^{(l)}|\leq \|A_{j_{y}+b}\|^{-2} \| \left(e^{(\log b)^{(1+\epsilon_{n})C}}+f_{n}(x)\right)|I_n|^{-(1+\epsilon_n)C\hat{\epsilon}^{-1}}.\end{equation}
\end{enumerate}

Combining \eqref{jiaoddd}, \eqref{moddd0}, \eqref{moddd}, \eqref{moddd''}, \eqref{moddd'} and \eqref{jiaoddd'''}, we finish the induction for the case $n.$
\hfill\qed\end{proof}

By the help of Lemma \ref{niceset0'}, we have the following several results hold true.

Let $p(\cdot): \Z_+\rightarrow \Z_+$ satisfy \begin{equation}\label{qcscale}|I_{p(n)-1}|^{-1}\leq  n\leq |I_{p(n)}|^{-1}.\end{equation}

\begin{lemma}\label{niceset1}For $t\in [\inf v-\frac{2}{\lambda},\sup v+\frac{2}{\lambda}],$ $i\in \Z_+,$ $s(k_i(t))+2\le j< s(k_{i+1}(t))+2,$
 the following hold true.
\begin{enumerate}
\item[1:] For $r^-_j(x,t)\geq n\geq \lambda^{(\log k_i)^c},$ it holds that $A_{n}(x,t)$ is $(\lambda^{-(\log n)^C},-)-\text{nice}$ on $I_{p(n),1}$.
\item[2:] For $k_i\geq n\geq \lambda^{(\log k_i)^c},$ it holds that $A_{n}(x,t)$ is $(\lambda^{-(\log n)^C},+)-\text{nice}$ on $I_{p(n),1}$ and $A_{n}(x,t)$ is $(\lambda^{-(\log n)^C},-)-\text{nice}$ on $I_{p(n),2}$.
\item[3:] For $r_j^+(x_0,t_0)\geq n\geq \lambda^{(\log k_i)^c},$ it holds that $A_{n}(x,t)$ is $(\lambda^{-(\log n)^C},+)-\text{nice}$ on $I_{p(n),2}.$
\item[4:] For $(s(k_i(t))+2\leq )j_i(t)\le j< s(k_{i+1}(t))+2,$ $r_j(x,t)\geq n\geq q^C_{N+j-1},$ it holds that $A_{n}(x,t)$ is $(\lambda^{-(\log n)^C},\pm)-\text{nice}$ on $I_{p(n)}.$
\item[5:] Set $D^{\pm}_{1,l}=\{x\in \R/\Z \vert x\pm s\alpha\notin {I}_{j,l},\  s=1, 2, \cdots, n\}$ and $D^{\pm}_{2,l}=\{x\in \R/\Z \vert x\pm s\alpha\notin {I}_{j,l},\  s=1, 2, \cdots, k_i-1\},$ $l=1,2,$ it holds that $A_{n}(x,t)$ is $(\lambda^{-(\log n)^C},-)-\text{nice}$ on $D^{-}_{1,1}$ and $D^{-}_{2,2};$ $A_{n}(x,t)$ is $(\lambda^{-(\log n)^C},+)-\text{nice}$ on $D^{+}_{1,2}$ and $D^{+}_{2,1}.$
\item[6:] Set $\lambda^{-\left\vert\log |I_j|\right\vert^{\hat{\epsilon}^{-1}}}\cdot I_j=\bar{I}_j,$ for all of the above conclusions, the upper bound of \eqref{result3} and \eqref{result4} can be strengthened to $1$ if we replace the integration interval with $\bar{I}_j.$
\end{enumerate}

\end{lemma}
\begin{proof} The proofs of $1,2,3,4$ and $5$ are quite similar, for the convenience, we only give the proof of 3 for convenience.
Note the definition $k_{i+1}(x_0,t_0)$ implies step $s(k_{i+1}(t_0))+1$ belongs to Type \textbf{I} and step $s(k_{i+2}(t_0))+1$ belongs to Type \textbf{II}. Thus \begin{equation}\label{ki+1low}k_{i+1}(t_0)>r^{+}_{s(k_{i+1}(t_0))+1}(x_0,t_0)\end{equation} (otherwise step $s(k_{i+1}(t_0))+1$ must belong to Type \textbf{II} leading to a contradiction).
On the other hand, note \eqref{k2k1} implies $k_i(t_0)\geq e^{k^c_{i-1}(t_0)}.$ Hence
\begin{equation}\label{ki+1upper} \lambda^{(\log k_i)^c}>\lambda^{k^{c^2}_{i-1}(t_0)}\gg k_{i-1}(t_0). \end{equation} Combining \eqref{ki+1low}, \eqref{ki+1upper} and the fact $r_j^+(x_0,t_0)\geq n\geq \lambda^{(\log k_i)^c}$, we have
\begin{equation}\label{ki+1t00}k_{i+1}(t_0)> n>k_{i-1}(t_0).\end{equation} Combining \eqref{ki+1t00} and \eqref{qcscale} yields that step $p(n)$ is located between $s(k_{i-1})+2$ and $s(k_{i+1})+2.$ Therefore except $s(k_i(t_0))+2,$ no resonant case occurs.


  \eqref{qcscale} implies
 $ cq^{\hat{\epsilon}}_{N+p(n)-2}\leq \log n \leq Cq^{\hat{\epsilon}}_{N+p(n)-1}.$
 Hence
$$ c((\log n))^{\hat{\epsilon}^{-1}}\leq q_{N+p(n)-1}\leq C((\log n))^{\hat{\epsilon}^{-1}C^*}.$$

 And the diophantine condition implies there exists some constant $C^*>0$ such that
$$R_{l}\leq |I_{l}|^{-C^*};~q_{N+l}\leq q^{C^*}_{N+l-1},~l\in \Z_+.$$


Let $i_l, 1\le l\le w$ be all the returning time of $x$  back to $I_{p(n)-1}-I_{p(n)}$ after $q_{p(n)-1}^2-1$ (this means that when \textbf{I} (non-resonance) occurs, we take the first return time from $I_{p(n)-1}$ to $I_{p(n)-1}$. When \textbf{III} occurs, we take the second return time). Without loss of generality, we assume that $w>2$ (otherwise by the diophantine condition we can always take a suitable constant $\hat{\epsilon}$ in the scale of $I_{p(n)}$ such that $w=3$).

By the help of \eqref{66-0'}, \eqref{66-1'} and \eqref{66-2'} of Lemma \ref{niceset0'}, we have
\begin{equation}\label{66-00'} \|A_{n}\|\geq \lambda^{(1-c)(n)},
\end{equation}
\begin{equation}\label{66-11'}\begin{array}{ll}
&\int_{x\in I_{p(n)}}\frac{\|A_{n}\|^{(1)}}{\|A_{n}\|} dx\leq \int_{x\in I_{p(n)}}\left(e^{(\log n)^{C}}+f_{p(n)}\right)|I_{p(n)}|^{-C\hat{\epsilon}^{-1}} dx\\&\leq
\left(e^{(\log n)^{C}}+\int_{x\in I_{p(n)}}f_{p(n)}(x) dx\right)|I_{p(n)}|^{-C\hat{\epsilon}^{-1}}\leq \left(e^{(\log n)^{C}}+Ck_i\right)|I_{p(n)}|^{-C\hat{\epsilon}^{-1}}\\&\leq \left(e^{(\log n)^{C}}+Cr_{p(n)}\right)|I_{p(n)}|^{-C\hat{\epsilon}^{-1}}\leq \left(e^{(\log n)^{C}}+|I_{p(n)}|^{-C}\right)|I_{p(n)}|^{-C\hat{\epsilon}^{-1}}\leq e^{(\log (n))^{2C}}
\end{array}\end{equation}
and
\begin{equation}\label{66-22'}\sum\limits_{l=0}^2|(s(A_{r^{+}_{p(n)}(x)})-s(A_{n}))^{(l)}|\leq \|A_{n}\|^{-2}e^{(\log n)^{C}} |I_{p(n)}|^{-C\hat{\epsilon}^{-1}}\leq \|A_{n}\|^{-2} e^{(\log n)^{2C}}.\end{equation}

Then \eqref{66-00'}, \eqref{66-11'} and \eqref{66-22'} imply $A_n(x_0,t_0)$ is ${(e^{(\log n)^C},+)-\text{nice}}$ on $I_{p(n)}.$

For $6,$ the proof for the first part is similar to above. For the remaining part, one notes that $e^{\left\vert\log |I_j|\right\vert^{\hat{\epsilon}^{-1}}}\gg \lambda^{c(\log r_j)^C}\geq \lambda^{Ck_i}.$ Then $$\begin{array}{ll}
\int_{x\in I_{p(n)}}\frac{\|A_{n}\|^{(1)}}{\|A_{n}\|} dx&\leq e^{-\left\vert\log |I_j|\right\vert^{\hat{\epsilon}^{-1}}}\cdot \left(\lambda^{(\log r_j)^C}+\lambda^{Ck_i}\right)<1.\hfill\qed
\end{array}$$
\end{proof}


The key to prove Lemma \ref{lemma18} is to estimate the finite Lyapunov exponent (FLE). Recall  that  for any $n>0$, FLE is defined as
$$
L_{n}(E)=\frac1{n}\int_{\R/\Z}\log\|A_{n}(x,t)\|dx,
$$
where $t=\frac{E}{\lambda}$. Thus we have
\begin{equation}\label{derivative_finite_LE}
\lambda|L'_{n}(E)|=\left|\frac{1}{n}\int_{\R/\Z}\frac{\partial_t\|A_{n}(x,t)\|}{\|A_{n}(x,t)\|}dx\right|.
\end{equation}

We will estimate it in the following sections.

\

\section{The proof of (1),(2) in Lemma \ref{lemma 6.1}} \label{proof of 6.1}

\vskip 0.2cm
In this section, we fix  $k\in\Z,Y\in \Z_+$ satisfying $Y\geq s(k)+1.$  $s(k)$  is defined as in \eqref{sk_df}, i.e.~$ q^2_{N+s(k)-1}<|k|\leq q^2_{N+s(k)}.$ Let $t\in \mathcal{B}_{Y+1}(t^{k}_{-})\bigcup\mathcal{B}_{Y+1}(t^{k}_{+}).$ Without loss of generality we assume that $k\geq 0$
and $I_{Y+1,2}(t)\bigcap T^k\left(I_{Y+1,1}(t)\right)\not=\emptyset.$  Next we give some notations which frequently appear in the later proof.

Denote $\tilde{D}_{Y+1}=\left(\bigcup\limits_{t\in J} I_{Y+1}(t)\right)\times J$ with $J=\mathcal{B}_{Y+1}(t^{k}_{-})\bigcup\mathcal{B}_{Y+1}(t^{k}_{+}).$

Let
$$ n= m\cdot\mathcal{N}_{Y+1}(= \lambda^{q_{N+Y}^{\hat{\epsilon}}}\gg k),\quad
n_0=n^{\frac{1}{100}},\quad
  \zeta_n= e^{-(\log n)^{c}},
\quad m=1,2.$$

It is clear that, to prove the first two results of Lemma \ref{lemma 6.1}, it is enough to prove that both \eqref{6.1.1} and \eqref{6.1.2} hold true for $L'_n(t).$
Without loss of generality, we only consider the case $$|t-t^k_+|\leq |t-t^k_-|.$$ Recall that
$$ \mathcal{B}_{Y+1}(t^{k}_{-})\bigcup\mathcal{B}_{Y+1}(t^{k}_{+})\triangleq \left(B(t^k_-,\lambda^{-{q_{N+Y}}})- B(t^k_-,\lambda^{-q^{\log q_{N+Y}}_{N+Y}})\right)\bigcup \left(B(t^k_+,\lambda^{-{q_{N+Y}}})- B(t^k_+,\lambda^{-q^{\log q_{N+Y}}_{N+Y}})\right).$$ It gives
$$\lambda^{-e^{(\log \log n)^C}}\leq \lambda^{-e^{(\log q_{N+Y})^2}}\leq \lambda^{-q_{N+Y}^{\log q_{N+Y}}}\leq |t-t^k_+|\leq \lambda^{-q_{N+Y}}\leq \lambda^{-(\log n)^{\hat{\epsilon}^{-1}}}.$$

Therefore
\begin{equation}\label{t-tk}\lambda^{-e^{(\log \log n)^{\hat{\epsilon}^{-1}}}} \leq |t-t^k_+|\leq \lambda^{-(\log n)^{\hat{\epsilon}^{-1}}}.\end{equation}

Notice that for all \( x \in \mathbb{R}/\mathbb{Z} \), there are at most two values of \( 1\leq l\leq \mathcal{N}_{Y+1} \) such that \( T^{l}x \in I_{Y+1} \). (Recall $|I_{Y+1}|=2\lambda^{-\hat{\epsilon}^{-1}q_{N+Y}^{\hat{\epsilon}}}$) Thus
we can rewrite the right side of \eqref{derivative_finite_LE} as follows:
\begin{equation}\label{qheshi}\begin{array}{ll}
\left|\frac{1}{n}\int_{\R/\Z}\frac{\partial_t\|A_{n}(x,t)\|}{\|A_{n}(x,t)\|}dx\right|\le&\left|\frac{1}{n}\int_{S_1(t)}
\frac{\partial_t\|A_{n}(x,t)\|}{\|A_{n}(x,t)\|}dx\right|+\left|\frac{1}{n}\int_{S_2(t)}\frac{\partial_t\|A_{n}(x,t)\|}
{\|A_{n}(x,t)\|}dx\right|+\left|\frac{1}{n}\int_{S_3(t)}\frac{\partial_t\|A_{n}(x,t)\|}{\|A_{n}(x,t)\|}dx\right|\\
\\
&:=I+II+III,
\end{array}\end{equation}
where
     $$
     \begin{array}{ll}
     S_1(t)&=\{x\in \R/\Z|~T^{l}(x)\not\in I_{Y+1}\ ~for\ any\ ~ 0\leq l\leq n \};\\
      \\
      S_2(t)&=\{x\in \R/\Z|~there\ exists\ \frac{1}{2}n_0\le l\leq n-\frac{1}{2}n_0 ~such~that~ T^{l}(x)\in I_{Y+1,1}\ \};\\
      \\
           S_3(t)&=\R/\Z-S_2-S_1.
      \end{array}
      $$

\begin{remark}Generally, by the analysis as above, if we deal with integrals of the form $$INT(A_n)\triangleq\int_{\R/\Z}\frac{\partial_t \|A_n(x,t)\|}{\|A_n(x,t)\|}dx,$$ where $A_{n}(x,t)=A_{n_2}A_{n_1}$ with $\|A_{n_1}\|, \|A_{n_2}\|\gg 1$, we can always divide $INT(A_n)$ into three terms. Two of these terms $\frac{\partial_t\|A_{n_i}(x,t)\|}{\|A_{n_i}(x,t)\|},i=1,2$ can be easily settled. By defining $\theta\triangleq \frac{\pi}{2}-s(A_{n_2})+u(A_{n_1}),$ the remaining one $INT_{M}\triangleq \int_{\R/\Z}\frac{1}{\|A_n(x,t)\|}\frac{\partial\|A_n\|}{\partial \theta}\partial_t\theta dx$ is very complicated, which is basically the most difficult point of this paper. In the following, we will directly calculate the $INT_{M}$ by (to a certain extent) simplifying $INT_{M}$ into a form of the integral of rational functions dependent on $\|A_{n_1}\|,\|A_{n_2}\|, \theta, t-t^k_{\pm},$ which can be quantitatively estimated precisely. As one will see, these three terms will form a certain competitive relationship between each other in the sense that each of them has the possibility to be the dominant term (if $\min\{n_1,n_2\}$ is too small, then $t-t^k_+$ will meaningless since the error caused by $\|A_{\min\{n_1,n_2\}}\|$ might cover it. In this case, it is sufficient to consider only $\|A_{\min\{n_1,n_2\}}\|$).  This is why we give the definition of $S_1,S_2,S_3$ such that in each of them the dominant term is definite.
\end{remark}
\subsection{The estimate on $I$}

For $x\in S_1(t),$ since there is no $0<l<n$ such that $x+l\cdot\alpha\in I_{Y+1},$ by the help of (5) of Lemma \ref{niceset1}, we indeed obtain that $A_n(x,t)$ is $(\lambda^{-({\log n})^C},+)-\text{nice}$ on $S_1.$ Then \eqref{result3} implies
\begin{equation}\label{subi}I\leq \lambda^{(\log n)^{C}}.\end{equation}
\subsection{The estimate on $II$ and $III$}\label{sec5.2.2} \label{section estimate23}
\subsubsection{\text{ An ease case of  $III$}}\quad
First we prove an easy case denoted by $ S_3^*\subset S_3$: there exists only one $0\leq m_0\leq n$ such that $x+m_0\alpha \in I_{Y+1}.$ The uniqueness of $m_0$ implies that $m_0\le k$ and $x+m_0\alpha\in I_{Y+1,2}.$

\begin{lemma}\label{S*3}
It holds that
$$|\int_{S_3^*}\frac{\partial_t\|A_n\|}{\|A_n\|} dx|\leq \lambda^{(\log n)^{C}}.
$$
\end{lemma}
\begin{proof}
Consider $A_n=A_{n-m_0}A_{m_0}.$
Note that Lemma \ref{niceset1} implies that $A_{n-m_0}(x,t)~\text{is}~(\lambda^{-(\log n)^{C}},+)-\text{nice}~\text{on}~S_3^*$ and $A_{m_0}(x,t)~\text{is}~(\lambda^{-(\log n)^{C}},-)-\text{nice}~\text{on}~S_3^*.$ Hence we obtain $\|A_{n-m_0}\|\geq \lambda^{\frac{1}{2}(n-m_0)}\gg \lambda^{Cm_0}\geq  \|A_{m_0}\|.$ Therefore for the case $m_0\leq \lambda^{(\log k)^c},$ with the help of $k<(\log n)^{\hat{\epsilon}^{-1}}$, (ii) of Lemma \ref{lemma9*} implies what we desire.

For the case $m_0\geq \lambda^{|\log k|^c},$
the definition of nice set and the fact $m_0~\le~k~$ yield for each $x\in S^*_3$
\begin{equation}\label{easycase} \int_{x-m_0\alpha\in S_3^*}\frac{\vert\|A_{n-m_0}\|^{(1)}\vert}{\|A_{n-m_0}\|} dx+\int_{x\in S_3^*}\frac{\vert\|A_{m_0}\|^{(1)}\vert}{\|A_{m_0}\|} dx\leq \lambda^{(\log n)^{C}}\end{equation} and \begin{equation}\label{easycase1}\sum\limits_{j=0}^2|\left(\theta^{n}_{m_0}-g_{p(m_0)}\right)^{(j)}|\leq \|A_{m_0}\|^{-1}+\|A_{n-m_0}\|^{-1}\leq \lambda^{-|\log m_0|^{C}}\leq \lambda^{-|\log k|^{C}} .\end{equation}
Lemma \ref{lmg} implies that
$ |\partial_t g_{p(m_0)}|\leq |I_{p(m_0)}|^{-C}\leq m_0^C\leq k^C\leq \lambda^{|\log k|^C}$ and $ |\partial_x g_{p(m_0)}|\geq |I_{p(m_0)}|^C>k^{-C}.$ Therefore \eqref{easycase1} implies (recall the definition of $\theta^{n}_{i}$ in (\ref{theta-up-low})) \begin{equation}\label{easycase2}|\partial_t \theta^{n}_{m_0}|\leq \lambda^{|\log k|^{C}},\quad |\partial_x \theta^{n}_{m_0}|\geq k^{-C}.\end{equation}  Combining \eqref{easycase}, \eqref{easycase2}, Lemma \ref{lemma8} and Lemma~\ref{usefullemma2}, we have
\begin{equation}\label{zuieasy}\begin{array}{ll}&|\int_{S_3^*}\frac{\partial_t\|A_n\|}{\|A_n\|} dx|\\
\\&\leq \sum\limits_{1\leq m_0\leq k}\vert\int_{x+m_0\alpha \in I_{Y+1,2}}\partial_{\theta_{m_0}^{n}}\|A_n\|\cdot \|A_n\|^{-1}\cdot \partial_t \theta_{m_0}^{n}\vert+\|A_{m_0}\|^{-1}|\partial_t \|A_{m_0}\||+\|A_{n-m_0}\|^{-1}|\partial_t \|A_{n-m_0}\||\\&\leq \sum\limits_{1\leq m_0\leq k}\vert\int_{x+m_0\alpha \in I_{Y+1,2}}W(\|A_{m_0}\|,\|A_{n-m_0}\|,\theta_{m_0}^{n})\cdot (\lambda^{|\log k|^C})+6+(\lambda^{(\log n)^{C}})|\\&\leq \lambda^{(\log n)^{C}}\cdot \sum\limits_{1\leq m_0\leq k}\int_{x+m_0\alpha \in I_{Y+1,2}}|W(\|A_{m_0}\|,\|A_{n-m_0}\|,\theta_{m_0}^{n})|\\&\leq  \lambda^{(\log n)^{C}}\cdot \sum\limits_{1\leq m_0\leq k}\int_{x+m_0\alpha \in I_{Y+1,2}}\frac{1}{\sqrt{\lambda^{-Cn}+\tan^2 \theta^{n}_{m_0}}} dx\\&\leq (\lambda^{(\log n)^{C}})\cdot k\cdot \lambda^{|\log k|^{C}}\cdot n~~~\leq \lambda^{(\log n)^{C}}~\qquad (by~\eqref{easycase2}~and~Lemma~\ref{usefullemma2}).\hfill\qed\end{array}\end{equation}
\end{proof}

\begin{remark}\label{zhongk} For $s(k_i(t))+2\leq i^*< j_i(t),$ by the assumption that $I_{s(k_i(t)),1}+k_i\alpha\bigcap I_{s(k_i(t)),2}\neq \emptyset,$ we have  $I_{i^*,1}+k_i\alpha\bigcap I_{i^*,2}\neq \emptyset.$ Moreover, note that
$r^+_{i^*}(x)$ is the second returning time with respect to $I_{i^*},$ which implies that ${r}^+_{i^*}(x)\geq |I_{i^*}|^{-C^*}$. Then by the help of Lemma \ref{lemma8}, there exist $M\in \mathbb{N}$ and $A_{Q_i},~i=1,2,\cdots,M$ such
that $\frac{\partial_t\|A_{{r}^+_{i^*}(x)}\|}{\|A_{{r}^+_{i^*}(x)}\|}\le \sum\limits_{1\le i\le M}\frac{\partial_t\|A_{Q_i}\|}{\|A_{Q_i}\|}+|I_{i^*}|^{-C}.$ Moreover, by the above proof, we indeed obtain that
$\vert \int_{x\in \mathcal{I}_{i^*,1}} \frac{\partial_t\|A_{Q_i}\|}{\|A_{Q_i}\|} dx\vert \leq \lambda^{(\log Q_i)^C}.$ Hence $\vert \int_{x\in \mathcal{I}_{i^*,1}} \frac{\partial_t\|A_{\tilde{r}^+_{i^*}(x)}\|}{\|A_{\tilde{r}^+_{i^*}(x)}\|} dx\vert  \leq M|I_{i^*}|\sum\limits_{i=1}^M\lambda^{(\log Q_i)^C}+|I_{i^*}|^{-C}\leq \lambda^{q^{C\hat{\epsilon}}_{N+i^*-1}}\leq \lambda^{q^{c}_{N+i^*-1}}.$
\end{remark}
\begin{remark}\label{remark45} If $s(k_i(t))+2\leq l < s(k_{i+1}(t))+2$ and $Y_1\leq e^{(\log n)^{c}}\leq Y_2\leq r_{l},$  by the similar analysis as above, Lemma~\ref{usefullemma2} implies that $$\int_{{~I}_{l,1}}|\frac{\partial_t\theta^{Y_1+Y_2}_{Y_1}\tan\theta^{Y_1+Y_2}_{Y_1}}{\sqrt{(\tan\theta^{Y_1+Y_2}_{Y_1})^4+o_1\cdot (\tan\theta^{Y_1+Y_2}_{Y_1})^2+o_2^2}} |dx\leq 2Y_1\lambda^{(\log k_i)^{C}}k_i^C\leq e^{(\log k_i)^{C}+(\log n)^{c}}\leq e^{2(\log n)^{c}},$$
 where $ o_1=\frac{2}{\|A_{Y_1}\|^4}+\frac{2}{\|A_{Y_2}(\cdot+Y_1\alpha)\|^4}$ and\ $o_2=\vert\frac{1}{\|A_{Y_1}\|^4}-\frac{1}{\|A_{Y_2}(\cdot+Y_1\alpha)\|^4}\vert.$
 By the same argument as Remark \ref{zhongk}, the above also holds true if we replace $\theta^{Y_1+Y_2}_{Y_1}$ by $\theta^{{r}^-_l(x)+Y_1}_{{r}^-_l(x)}$ or $\theta^{{r}^+_l(x)+Y_2}_{Y_2}.$
\end{remark}
\subsubsection{\text{ The reduction of the remaining case}}\quad

Now we consider the remaining cases, that is $S_2\bigcup (S_3-S^*_3)$. Based on the above case, by symmetry, in the following we only need to consider the case for which there exist $0\leq n_1<n_2\leq n$ with $n_2-n_1=k (< n^{\frac{1}{2}})$ such that $x+n_i\alpha \in I_{Y+1},i=1,2.$  We write $A_{n-n_2}A_{n_2-n_1}A_{n_1}=A_{n-n_2}A_{k}A_{n_1}.$ The symmetry allows us only need to consider the condition  \begin{equation}\label{baseassume} n_1<n-n_2.\end{equation}  Thus from $n=n_1+(n_2-n_1)+(n-n_2)\le 2(n-n_2)+k$, we have $n-n_2>\frac n3$.

For the case $n_1\leq \lambda^{|\log n|^{c}},$ our target is to prove Lemma \ref{S3-S*3} holds true. With the help of Remark \ref{remark45}, we have $\int_{I_{Y+1,1}}\frac{\partial_t \|(A_{n-n_2}A_{k})(A_{n_1})\|}{\|(A_{n-n_2}A_{k})(A_{n_1})\|} dx\leq \int_{I_{Y+1,1}}\frac{\partial_t \|(A_{n-n_2}A_{k})\|}{\|(A_{n-n_2}A_{k})\|}dx+\int_{I_{Y+1,1}}\frac{\partial_t \|A_{n_1}\|}{\|A_{n_1}\|}dx +e^{(\log n)^{C}},$ which is dominated by $(A_{n-n_2}A_{k})$ and thus is indeed the case $S_3^*$.   Hence we only need to consider the case \begin{equation}\label{tjz}\lambda^{|\log n|^{c}}\leq n_1<n-n_2.\end{equation} For this case, it follows from Lemma \ref{niceset1} that $$\begin{array}{ll}
&A_{n_1}~is~(\lambda^{-|\log n_1|^{C}},-)-nice~on~I_{Y+1,1},\\
 &A_{n_2-n_1}~is~(\lambda^{-|\log k|^{C}},+)-nice~on~I_{Y+1,1},\\
  &A_{n_2-n_1}~is~(\lambda^{-|\log k|^{C}},-)-nice~on~I_{Y+1,2},\\
   &A_{n-n_2}~is~(\lambda^{-|\log n|^{C}},+)-nice~on~I_{Y+1,2}.\end{array}$$

Let $\hat{g}_{X,1}\triangleq u(A_{r_{X}}(x,t))-s(A_k(x,t))$ on $I_{X,1}$ and $\hat{g}_{X,2}\triangleq u({A_{k}(x+k\alpha,t)})-s(A_{r_{X}}(x+k\alpha,t))$ on $I_{X,2},~X=p(n_1)$ or $Y.$
Then the definition of a nice set yields for each $x\in S_2\bigcup (S_3-S^*_3)$,
\begin{equation}\label{diffcase}\begin{array}{ll}& \int_{x\in S_2\bigcup (S_3-S^*_3)}\frac{|\|A_{n_1}\|^{(1)}|}{\|A_{n_1}\|} dx+\int_{x-n_1\alpha\in S_2\bigcup (S_3-S^*_3)}\frac{|\|A_{n_2-n_1}\|^{(1)}|}{\|A_{n_2-n_1}\|} dx+\int_{x-n_2\alpha\in S_2\bigcup (S_3-S^*_3)}\frac{|\|A_{n-n_2}\|^{(1)}|}{\|A_{n-n_2}\|} dx\\&\leq \lambda^{(\log n)^{C}},\end{array}\end{equation}
\begin{equation}\label{diffcase1} \sum\limits_{j=0}^2|(\theta^{n-n_1}_{n_2-n_1}-\hat{g}_{Y,2})^{(j)}|\leq \|A_{n-n_2}\|^{-2}\lambda^{(\log (n-n_2))^{C}},\end{equation}
and
  \begin{equation}\label{bargpn0}\sum\limits_{l=0}^2|\left(u(A_{r_{p(n_1)}})-u(A_{n_1})\right)^{(l)}|\leq \|A_{n_1}\|^{-2}\lambda^{(\log n_1)^{C}}.\end{equation}

 On the other hand, by Theorem \ref{theorem12} and the fact $n_1<cr_{p(n_1)}$ (from \eqref{qcscale}),  we have
 \begin{equation}\label{bargpn1}\sum\limits_{l=0}^2|\left(u(A_{r_{p(n_1)}})-u(A_{r_Y})\right)^{(l)}|\leq \lambda^{-r_{p(n_1)}}\leq \lambda^{-Cn_1}\leq \|A_{n_1}\|^{-2}\lambda^{(\log n_1)^C}.\end{equation}

It then follows from \eqref{bargpn0} and \eqref{bargpn1} that
\begin{equation}\label{diffcase2} \sum\limits_{j=0}^2|(\theta^{n_2}_{n_1}-\hat{g}_{Y,1})^{(j)}|=\sum\limits_{j=0}^2|(u(A_{n_1})-u(A_{r_{Y}}))^{(j)}|\leq \|A_{n_1}\|^{-2}\lambda^{(\log n_1)^C}.\end{equation}

\begin{definition}
For simplification, in the following we define
$$\theta^{n}_{n_1}:=\theta_{n_1},$$
where the function $\theta_i^n$ is defined as (\ref{theta-up-low}) and $n_1$ satisfies (\ref{tjz}).
\end{definition}
\begin{lemma}\label{S3-S*3} Assume (\ref{tjz}) holds true. It holds that
$$|\int_{X}\frac{\partial_t\|A_n\|}{\|A_n\|} dx|\le \lambda^{(\log n)^{C}}+\sum\limits_{k<n_1<n-k}|\mathcal{S}_{I_{Y+1,1},n-n_1,n_1}|,$$
where $\mathcal{X}=S_2\bigcup (S_3-S_3^*)$ and \begin{equation}\label{tarobtain}\mathcal{S}_{I_{Y+1,1},n-n_1,n_1}= \int_{T^{-n_1}(I_{Y+1,1})}\frac{\frac{\partial \|A_{n}\|}{ \partial{\theta_{n_1}}}\frac{\partial {\theta_{n_1}}}{\partial t}}{\|A_{n}\|}dx.\end{equation}
\end{lemma}
\begin{proof} Considering $A_{n-n_1}=A_{n-n_2}A_{n_2-n_1}$. Similar to the argument in the previous case in Subsection 6.2.1, it holds that
\begin{equation}\label{precase}\left|\int_{x-n_1\alpha \in \mathcal{X}}\frac{\partial_t \|A_{n-n_1}\|}{ \|A_{n-n_1}\|}dx\right|\leq  \lambda^{(\log n)^{C}}.\end{equation}
\eqref{diffcase} shows
\begin{equation}\label{precase1}\left|\int_{\mathcal{X}}\frac{\partial_t \|A_{n_1}\|}{ \|A_{n_1}\|}dx\right|\leq \lambda^{(\log n)^{C}}.\end{equation}
Then Lemma \ref{lemma8},~\eqref{precase} and \eqref{precase1} imply that \begin{equation}\label{7.2}\begin{array}{ll}|\int_{X}\frac{\partial_t\|A_n\|}{\|A_n\|} dx|&\leq |\int_{X}\frac{\partial_t\|A_{n-n_1}\|}{\|A_{n-n_1}\|}dx|+|\int_{X}\frac{\frac{\partial \|A_n\|}{ \partial{\theta_{n_1}}}\frac{\partial {\theta_{n_1}}}{\partial t}}{\|A_n\|}dx|+|\int_{X}\frac{\partial_t\|A_{n_1}\|}{\|A_{n_1}\|}dx| \\
\\
&\leq \lambda^{(\log n)^{C}}+\sum\limits_{k<n_1<n-k}|\mathcal{S}_{I_{Y+1,1},n-n_1,n_1}|.\hfill\qed\end{array}\end{equation}
\end{proof}
\vskip 0.2cm \noindent Thus for the cases $II$ and $III$, it is sufficient to estimate \eqref{tarobtain}.
 To simplify the estimate on \eqref{tarobtain}, we use the concept of the equivalent term ``$\sim_{l,\epsilon}$" defined in Section 2.1.
In the remaining part of this paper, the domain of ``$\sim_{l,\epsilon}.$" is often $\tilde{D}_{Y+1},$ $\Pi_1\tilde{D}_{Y+1}$ or $\Pi_2\tilde{D}_{Y+1}.$  In this case, we will not mention the domain of ``$\sim_{l,\epsilon}.$".

By the help of Lemma \ref{lemma8}, we have
\begin{equation}\label{antilde}\frac{\frac{\partial \|A_{n}\|}{ \partial{\theta_{n_1}}}\frac{\partial {\theta_{n_1}}}{\partial t}}{\|A_{n}\|} \sim_{0,\sup\limits_{(x,t)\in \tilde{D}_{Y+1}}\tilde{\epsilon}^2_1} \partial_t(\theta_{n_1})\cdot \mathcal{M}(\tan\theta_{n_1})\end{equation}
with
\begin{equation} \label{k(f)} \mathcal{M}(x)=\frac{x}{\sqrt{x^4+\tilde{\epsilon}_1^4\cdot X^2+\tilde{\epsilon}_2^8}}, \tilde{\epsilon}_1^4=\frac{2}{\|A_{n_1}\|^4}+\frac{2}{\|A_{n-n_1}(\cdot+n_1\alpha)\|^4}, \ \tilde{\epsilon}_2^4=\left\vert\frac{1}{\|A_{n_1}\|^4}-\frac{1}{\|A_{n-n_1}(\cdot+n_1\alpha)\|^4}\right\vert.\end{equation}
We need the following lemma to simplify \eqref{tarobtain}.
\begin{lemma}\label{dominate1}Let $\epsilon>0, F,F_1\in C^0(I), I\subset \R/\Z\times \R$. If $F\sim_{0,\epsilon} F_1,$ then $$\vert\int_{I} F-F_1 dx\vert\leq \epsilon\cdot \int_{I}\vert F_1\vert dx.$$
\end{lemma}

\begin{proof} Note $F\sim_{0,\epsilon} F_1$ implies $\left\vert F-F_1 \right\vert\leq \epsilon|F_1|,$ which immediately implies what we desire.
\hfill\qed\end{proof}

\eqref{antilde} and Lemma \ref{dominate1} imply that

\begin{equation}\label{jifendejianjin}\begin{array}{ll}&\left\vert\int_{T^{-n_1}(I_{Y+1,1})}\frac{\frac{\partial \|A_{n}\|}{ \partial{\theta_{n_1}}}\frac{\partial {\theta_{n_1}}}{\partial t}}{\|A_{n}\|}dx-\int_{T^{-n_1}(I_{Y+1,1})}\partial_t(\theta_{n_1})\cdot \mathcal{M}(\tan\theta_{n_1})dx\right\vert\\&\leq C\sup\limits_{(x,t)\in \tilde{D}_{Y+1}}(\tilde\epsilon_1^2(x,t))\cdot \int_{T^{-n_1}(I_{Y+1,1})}\left\vert\partial_t(\theta_{n_1})\cdot \mathcal{M}(\tan\theta_{n_1})\right\vert dx.\end{array}\end{equation}

Hence to estimate \eqref{tarobtain}, we only need to consider the following two integrals: \begin{equation}\label{tar1}\vert\int_{T^{-n_1}(I_{Y+1,1})} (\partial_t \theta_{n_1})\mathcal{M}(\tan\theta_{n_1}) dx\vert;~\int_{T^{-n_1}(I_{Y+1,1})}\left\vert (\partial_t \theta_{n_1})\mathcal{M}(\tan\theta_{n_1})\right\vert dx.\end{equation}

\

\
\subsubsection{\text{The reduction of $\mathcal{M}(\tan\theta_{n})$}}

Note that $g_{Y+1}$ belongs to Type $\textbf{II}^k_{Y+1}.$ Without loss of generality we assume that step $Y$ belongs to Type $\textbf{I}.$ Our target is to estimate \eqref{tar1}. For this purpose, we will find the equivalent terms of $\theta_{n_1}(x,t)$ and $\tilde{\epsilon}_1$, respectively.

 Now we simplify $\tilde{\epsilon}_1$. Denote the zeros of $\hat{g}_{Y,i}$ by $\hat{c}_{Y,i},i=1,2.$ Then \eqref{lmg'-1} of Corollary \ref{lmg'} implies that there exists a unique $\hat{t}_k\in (t^k_-,t^k_+)$ such that $\hat{c}_{Y,1}(\hat{t}_k)+k\alpha=\hat{c}_{Y,2}(\hat{t}_k).$  Recall the notation $\zeta_n= e^{-(\log n)^c},$
Then Lemma \ref{lmg} and Theorem \ref{theorem12} imply that \begin{equation}\label{gyi}\begin{array}{ll}\hat{g}_{Y,i}(x,t)\sim_{1,\zeta_k}\bar {u}_i (x-\hat{c}_{Y,i}(\hat{t}_k))+\bar{v}_i (t-\hat{t}_k)\end{array};~|\hat{g}_{Y,i}(x,t)|\leq C(\zeta_n^{c})+C\lambda^{-\frac{1}{2}k};~\vert\frac{\partial^2 \hat{g}_{Y,i}(x,t)}{\partial X^2}\vert\leq C(\log k)^C,\end{equation} with $~X\in \{t,~x\}~,~i=1,2,$ where \begin{equation}\label{uhv}0<\bar{v}_i,\bar{v}_i^{-1}\leq (\log k)^C; 0<|\bar{u}_i|,|\bar{u}_i^{-1}|\leq (\log k)^{C}; sgn(\bar{u}_1)=-1=- sgn(\bar{u}_2).\end{equation}
Moreover, $ \log\|A_k(\hat{c}_{Y}(\hat{t}_k),\hat{t}_k)\|\geq ck(\log \lambda).$

Note that for each $x\in S_2\bigcup (S_3-S^*_3),$ it follows from \eqref{diffcase} that
$$\log \|A_k(x,t)\|\sim_{1,\zeta_n} \log\|A_k(\hat{c}_{Y}(\hat{t}_k),\hat{t}_k)\|+p^k_1(x-\hat{c}_{Y}(\hat{t}_k))+p_2^k(t-\hat{t}_k);~|p^k_i|\leq \lambda^{(\log k)^C},~i=1,2.$$ Let \begin{equation}\label{lkdf}\|A_k(\hat{c}_{Y}(\hat{t}_k),\hat{t}_k)\|:=l_k.\end{equation} Clearly $\left\vert\log \|A_k(x,t)\|-\log l_k\right\vert\leq \lambda^{C(\log k)}\cdot \zeta_n\leq C\zeta_n^{\frac{1}{2}}\ll 1$, which implies \begin{equation}\label{def-lklk}\|A_k(x,t)\|\sim_{0,\zeta_n^{\frac{1}{2}}} l_k~\text{on}~\tilde{D}_{Y+1}.\end{equation}


Similarly, we have
\begin{equation}\label{lambda1df}\|A_{n_1}(x,t)\|\sim_{0,\zeta_n^{\frac{1}{2}}} \|A_{n_1}(\hat{c}_{Y}(\hat{t}_k)-n_1\alpha,\hat{t}_k)\|\triangleq \lambda_1~\text{on}~\tilde{D}_{Y+1}\end{equation} and

\begin{equation}\label{lambdadf2}\|A_{n-n_1}(x,t)\|\sim_{0,\zeta_n^{\frac{1}{2}}}\|A_{n-n_1}(\hat{c}_{Y}(\hat{t}_k),\hat{t}_k)\|\triangleq \lambda_2~\text{on}~\tilde{D}_{Y+1}.\end{equation}

Now for simplification, we define
 \begin{equation}\label{epsilon1} \epsilon^4_1:=\frac{2}{\lambda_1^4}+\frac{2}{\lambda_2^4};\end{equation}

 \begin{equation}\label{translation-notation} \ \omega:=\frac{1}{2}(-\frac{\bar{v}_1}{\bar{u}_1}+\frac{\bar{v}_2}{\bar{u}_2}).
 \end{equation}

 \eqref{tjz}, \eqref{def-lklk}, \eqref{lambda1df}, \eqref{epsilon1}, Lemma \ref{niceset1} imply that \begin{equation}\label{ktheta}\tilde{\epsilon}_1\sim_{0,\zeta_n}
\epsilon_1\end{equation} and
 \begin{equation} \label{ktheta1-1}|\partial_{x}\tilde{\epsilon}_2|+|\partial_{t}\tilde{\epsilon}_2|\leq \tilde{\epsilon}_1^{1-}\leq \epsilon_1^{1-}.\end{equation}

 It holds from \eqref{ktheta} that \begin{equation}\label{k-reduction}\mathcal{M}(\tan\theta_{n_1})\sim_{0,\zeta_n} \frac{\tan\theta_{n_1}}{\sqrt{\tan^4\theta_{n_1}+\epsilon_1^4\tan\theta^2_{n_1}+\tilde{\epsilon}_2^8}}~\text{on}~\tilde{D}_{Y+1}.\end{equation}
 To estimate (\ref{tar1}), we only need to consider the following case
 \begin{equation}\label{jiashe000}|\theta_{n_1}(x,t)|\leq \zeta_n.\end{equation}
 (Note if the above case is invalid, then \eqref{tar1} is dominated by $\lambda^{(\log n)^{C}},$ which is ignorable in comparison with the error estimate in Lemma \ref{lemma18}.)

 By the help of (ii)-(3c) of Lemma \ref{lemma8*}, \eqref{jiashe000} implies \begin{equation}\label{jiashe001} Cl_k^{-2}< \min\limits_{C^0(\tilde{D}_{Y+1})}|\hat{g}_{Y,2}|
 \end{equation}

By (ii)-(2b) of Lemma \ref{lemma8*}, \eqref{jiashe000}, \eqref{diffcase1} and \eqref{diffcase2}, we can precisely write
\begin{equation}\label{thetat1}\theta_{n_1}(x,t)\sim_{2,\zeta_n} \tan^{-1}(\|A_{k}(x,t)\|^2[\tan \left(\hat{g}_{Y,2}+\Delta_{n-n_1}(x,t)\right)])-\frac{\pi}{2}+\hat{g}_{Y,1}+\Delta_{n_1}(x,t) \end{equation} with \begin{equation}\label{thetat11}\|\Delta_{j}(x,t)\|_{C^2(\tilde{D}_{Y+1})}\leq C\|A_{j}\|^{-2}\lambda^{(\log j)^{\hat{\epsilon}^{-2}}},~j=n_1~or~n-n_1.\end{equation}


By~\eqref{baseassume}, \eqref{thetat11} and \eqref{jiashe001}, it holds that \begin{equation}\label{theta111}\|\Delta_{n-n_1}\|_{C^0(\tilde{D}_{Y+1})}\leq C\lambda^{-\frac{n}{3}}\ll C\lambda^{-2k}\leq C\|A_k\|^{-2}\leq Cl_k^{-2}\leq \min\limits_{C^0(\tilde{D}_{Y+1})}|\hat{g}_{Y,2}|.\end{equation}

\subsubsection{\text{The reduction of $\partial_t\theta_{n_1}$}}\quad
In this sub-subsection, we will reduce the estimate on $\partial_t\theta_{n_1}$ to the one on $\partial_x\theta_{n_1}$, see Lemma
\ref{lemmathetat1}.

Recall the notation $\zeta_n= e^{-(\log n)^{c}}$. By the help of Lemma \ref{thetat1fenjie}, we can estimate $|G_k|$ precisely.
\begin{lemma}\label{pugapgj} Let $\omega$ be defined in \eqref{translation-notation} and ${l}_k$ be as in \eqref{lkdf}. Moreover, let $\bar{u}_1,\ \bar{u}_2$ be as in (\ref{gyi}) and (\ref{uhv}). Then it holds that
 \begin{equation}\label{gapgja}|G_k|=|t^k_--t^k_+|\sim_{0,\zeta_n^{\frac{1}{2}}}\frac{2}{\omega\sqrt{|\bar{u}_1\bar{u}_2|}}\cdot l_k^{-1}.\end{equation}
\end{lemma}

\begin{proof} In Lemma \ref{thetat1fenjie}, we set $$D=\tilde{D}_{Y+1},~F=g_{Y+1,1},~L=\|A_k(x,t)\|,~h_i=g_{Y,i},~i=1,2,~\epsilon=\zeta_n,~l=l_k,~a_i=\bar{u}_i,
~b_i=\bar{v}_i,i=1,2.$$ Then by the help of iii of Lemma \ref{thetat1fenjie}, we obtain two points $x^*_i(t), i=1,2$ satisfying that $$\{x\vert \partial_x g_{Y+1,1}(x,t)=0\}=\{x^*_1(t),x^*_2(t)\}.$$
And v yields that  $$\vert\partial_t g_{Y+1,1}(x^*_i(t),t)\vert\sim_{0,\zeta_n^{\frac{1}{2}}} 2\omega|\bar{u}_1|>0\ {\rm\ on\ }{\Pi_2 \tilde{D}_{Y+1}},~i=1,2,~$$ which implies that $g_{Y+1,1}(x^*_i(t),t)$ is strictly monotonic with respect to $t$ on $\Pi_2\tilde{D}_{Y+1}.$
Then \eqref{gnrange} of Theorem \ref{theorem12} allows us to define two one-element sets
$\{t\vert g_{Y+1,1}(x^*_i(t),t)=0\}:=\{E_i^*\},\ i=1, 2.$ Then (\rm v) of Lemma \ref{thetat1fenjie} implies that $$\vert g_{Y+1,1}(x^*_2(E^*_1),E^*_1)\vert=\vert g_{Y+1,1}(x^*_1(E^*_1),E^*_1)-g_{Y+1,1}(x^*_2(E^*_1),E^*_1)\vert\sim_{0,\zeta_n^{\frac{1}{2}}} 4\sqrt{\frac{|\bar{u}_1|}{|\bar{u}_2|}}l_k^{-1}.$$

Then Newton-Leibniz's Formula yields that $$|E^*_1-E^*_2|\cdot(2\omega|\bar{u}_1|)\sim_{0,\zeta_n^{\frac{1}{2}}} \vert\int_{E^*_1}^{E^*_2} \partial_t g_{Y+1,1}(x^*_1(t),t) dt\vert\sim_{0,\zeta_n^{\frac{1}{2}}} 4\sqrt{\frac{|\bar{u}_1|}{|\bar{u}_2|}}l_k^{-1}.$$

Hence $|E^*_1-E^*_2|\sim_{0,\zeta_n^{\frac{1}{2}}} 2\frac{l_k^{-1}}{\omega\sqrt{|\bar{u}_1\bar{u}_2|}}.$
Therefore we have $|t^k_--t^k_+|\sim_{0,\zeta_n^{\frac{1}{2}}}|E^*_1-E^*_2|\sim_{0,\zeta_n^{\frac{1}{2}}} 2\frac{l_k^{-1}}{\omega\sqrt{|\bar{u}_1\bar{u}_2|}}$ as desired.
\hfill\qed\end{proof}

The following  conclusions are indeed direct consequences of Lemma \ref{thetat1fenjie}.

Recall (\ref{tjz}) : $\lambda^{|\log n|^c}\leq n_1<n-n_2.$
\begin{lemma}\label{lemmathetat1}

\textbf{A:}~
If $\eqref{tjz}$ holds true, then the following properties hold true.

\begin{enumerate}

\item[i:] For any fixed $t_0\in \Pi_2\tilde{D}_{Y+1},$ $\tan\theta_{n_1}(x,t_0)$ has at most two zeros on $\Pi_1 \tilde{D}_{Y+1}.$
\item[ii:] There exist two functions $h(x,t),w(x,t)\in C^1(\tilde{D}_{Y+1})$ such that
\begin{equation}\label{theta_t1}\partial_t\theta_{n_1}=h(x,t)\cdot\frac{\partial_x (\tan\theta_{n_1})}{1+\tan^2\theta_{n_1}}+w(x,t),~(x,t)\in~\tilde{D}_{Y+1}\end{equation} with $\max\limits_{(x,t)\in \tilde{D}_{Y+1}}|h(x,t)|,\max\limits_{(x,t)\in \tilde{D}_{Y+1}}|w(x,t)|\leq C(\log k)^C.$

\

\item[iii:] There exist two intervals $J_1(t),\ J_2(t)\subset \R/\Z$ satisfying $Leb\{J_i\}\leq Cl_k^{-1}(\zeta_n^{c}+l_k^{-c})^{\frac{1}{4}}$ such that $\partial_x\theta_{n_1}(x,t)$  has exactly one zero $z_i(t)$   on each $J_i,\ i=1,\ 2$ and the following hold true.

\begin{equation}\label{lem50-1-} \{x\vert (x,t)\in \tilde{D}_{Y+1},\ |\partial_x \theta_{n_1}(x,t)|\leq \zeta_n^{c}+l_k^{-c}\}=J_1\bigcup J_2;\end{equation}

\begin{equation}\label{lem50-1'} \vert\frac{\partial^2 \theta_{n_1} }{\partial x^2}(x,t)\vert \sim_{0,(\zeta_n^{c}+l_k^{-c})^{\frac{1}{2}}} 2\bar{u}_2\sqrt{\vert\bar{u}_1 \bar{u}_2\vert}l_k\ {\rm\ on\ }{J_1\bigcup J_2};\end{equation}

 \begin{equation}\label{lem50-1''}\sum_{X,Y\in \{x,t\}}\vert\frac{\partial^2 \theta_{n_1} }{\partial X\partial Y}(x,t)\vert\leq C(\log k)^C l^8_k,\ (x,t)\in \tilde{D}_{Y+1};\end{equation}

\begin{equation}\label{lem50-0} ~|h(x,t)|\sim_{0,(\zeta_n^{c}+l_k^{-c})^{\frac{1}{2}}} \left\vert\frac{\bar{v}_2}{\bar{u}_2}\right\vert,~w(x,t)\sim_{0,(\zeta_n^{c}+l_k^{-c})^{\frac{1}{2}}} (-\frac{\bar{v}_1}{\bar{u}_1}+\frac{\bar{v}_2}{\bar{u}_2})\bar{u}_1
\quad {\rm on\ } {\tilde{D}_{Y+1}};
\end{equation}

\begin{equation}\label{lem50-2} ~\big\vert\tan\theta_{n_1}(z_1(t),t) -\tan\theta_{n_1}(z_2(t),t)\big\vert\sim_{0,(l_k^{-c}+\zeta_n^c)^{\frac{1}{2}}} \pi-4\sqrt{\frac{|\bar{u}_1|}{|\bar{u}_2|}}l_k^{-1}\ {\rm\ on\ } {\Pi_2\tilde{D}_{Y+1}};\end{equation}

\begin{equation}\label{lem50-3} \max\limits_{i=1,2}\inf\limits_{x\in J_i}\vert \tan\theta_{n_1} \vert(~\text{mod}~\pi)\geq \sqrt
{\frac{|\bar{u}_1|}{|\bar{u}_2|}}l_k^{-1}.
\end{equation}
\end{enumerate}

\textbf{B:}~Let $t\in \R,~\tilde{i}\in \Z_+$ and $s(k_{\tilde{i}}(t))+2\leq j<j_{\tilde{i}}(t)$. Then \eqref{theta_t1}-\eqref{lem50-3} also hold true if $\theta_{n_1}$ is replaced with $g_{{j}}(x,t)$. Furthermore  if $k_{\tilde{i}}=k$ and~$j_{\tilde{i}}(t)\leq j< s(k_{\tilde{i}+1}(t))+2$, then  \eqref{theta_t1},~\eqref{lem50-1''} and \eqref{lem50-0} also hold true if $\theta_{n_1}$ is replaced with $g_{{j}}(x,t)$.

Particularly, it holds that \begin{equation}\label{lem49-zh}\frac{l_{k_{\tilde{i}}}}{|\partial_x g_{\bar{j}}(c_{\bar{j},m}(t),t)|}\leq \lambda^{-(\log k_{\tilde{i}})^{C}}\cdot|~I_{j}|^{-1}, \ ~s(k_{\tilde{i}+1}(t))+2> j\geq j_{\tilde{i}}(t),~m=1,2.\end{equation}
\end{lemma}
\begin{proof}
We take in Lemma \ref{thetat1fenjie}

 $D=\tilde{D}_{Y+1},~F=\theta_{n_1}(x,t),~L=\|A_k(x,t)\|,~h_1=g_{Y,1}+\Delta_{n_1},~h_2=g_{Y,2}+\Delta_{n-n_1},~\epsilon=\zeta_n^{c}+l_k^{-c},~l=l_k;~\bar{u}_i=a_i,~\bar{v}_i=b_i,i=1,2.$
Note that \eqref{tjz} and \eqref{thetat11} imply that $$\vert\frac{\partial \Delta_{n_1}(x,t)}{\partial t}\vert+\vert\frac{\partial \Delta_{n_1}(x,t)}{\partial x}\vert\leq C(\zeta_n^{c}+l_k^{-c})\ll \vert u_1\vert,\vert v_1\vert,$$ which leads that $$\frac{\partial ({g}_{Y,1}+\Delta_{n_1}(x,t))}{\partial x}\sim_{0,\epsilon} \bar{u}_1,\quad\frac{\partial ( {g}_{Y,1}+\Delta_{n_1}(x,t))}{\partial t}\sim_{0,\epsilon} \bar{v}_1.$$
Then one can check that \eqref{tjz},\eqref{diffcase},\eqref{thetat11} and \eqref{gyi} imply \eqref{lem48-tj}. Hence by  Lemma \ref{thetat1fenjie}, we conclude \textbf{A}.

Note that by the help of \eqref{iii-3} of Lemma \ref{thetat1fenjie}, we have (recall that $g_{{j}}(\tilde{c}_{{j},m}(t))=0$) \begin{equation}\label{jiest0}|\partial_x g_{j_{\tilde{i}}(t)}(c_{j_{\tilde{i}}(t),m}(t),t)|\geq \lambda^{(\log k_{\tilde{i}})^c}|\tilde{c}_{j_{\tilde{i}}(t),m}(t)-c_{j_{\tilde{i}}(t),m}(t)|,~m=1,2.\end{equation} On the other hand, $|~I_{j_{\tilde{i}}(t)}|\leq \lambda^{-|\log|\tilde{c}_{j_{\tilde{i}}(t),m}(t)-c_{j_{\tilde{i}}(t),m}(t)||^C},$ which leads that
\begin{equation}\label{jiest}|~I_{j_{\tilde{i}}(t)}|^{-1}\geq \lambda^{|\log|\tilde{c}_{j_{\tilde{i}}(t),m}(t)-c_{j_{\tilde{i}}(t),m}(t)||^C}.\end{equation} By Theorem \ref{theorem12},
$$\begin{array}{ll}&|\partial_x g_{j_{\tilde{i}}(t)}(c_{j_{\tilde{i}}(t),m}(t),t)|
\leq |\partial_x g_{s(k_{\tilde{i}})+2}(c_{j_{\tilde{i}}(t),m}(t),t)|+\lambda^{-r_{s(k_{\tilde{i}})+2}}
\\&<C|\partial_x g_{s(k_{\tilde{i}})+2}(c_{j_{\tilde{i}}(t),m}(t),t)|<Cl^{-C}_{k_{\tilde{i}}}\quad(note~that~step~s(k_{\tilde{i}})+2~belongs~to~\textbf{II}^{k_{\tilde{i}}}_{s(k_{\tilde{i}})+2}).\end{array}$$
Then \eqref{jiest0} yields $$|\tilde{c}_{j_{\tilde{i}}(t),m}(t)-c_{j_{\tilde{i}}(t),m}(t)|\leq |I_{j_{\tilde{i}}(t)}|\leq l^{-c}_{k_{\tilde{i}}}$$ and \eqref{jiest} implies $|~I_{j_{\tilde{i}}(t)}|^{-1}\gg l^C_{k_{\tilde{i}}}.$ Combining this with \eqref{jiest0}, for each $s(k_{\tilde{i}+1}(t))+2> l\geq j_{\tilde{i}}(t)$, we have $$|~I_{l}|^{-1}\geq |~I_{j_{\tilde{i}}(t)}|^{-1}\gg l^C_{k_{\tilde{i}}}\cdot \lambda^{c|\log|\tilde{c}_{j_{\tilde{i}}(t),m}(t)-c_{j_{\tilde{i}}(t),m}(t)||^C}\gg l^C_{k_{\tilde{i}}}\cdot |\tilde{c}_{j_{\tilde{i}}(t),m}(t)-c_{j_{\tilde{i}}(t),m}(t)|^{-C}\gg \frac{\lambda^{(\log k_{\tilde{i}})^{C}}l_{k_{\tilde{i}}}}{|\partial_x g_{j}(c_{j,m}(t),t)|},$$ which immediately implies \eqref{lem49-zh}. The remaining part for the proof of \textbf{B} is similar as the proof of \textbf{A}.
\hfill\qed\end{proof}


\vskip 0.6cm
Let $J_i,i=1,2$ be defined in Lemma \ref{lemmathetat1}.
By \eqref{lem50-3}, in the following proof, without loss of generality, we always assume that \begin{equation}\label{153}\min\limits_{x\in J_2(t)}\vert\tan\theta_{n_1}(x,t)\vert= \vert\tan\theta_{n_1}(z_2(t),t)\vert\geq (\log k)^{-C}\cdot l_k^{-1}.\end{equation}

By the help of Lemma \ref{lemmathetat1} and \eqref{tar1}, to finish the estimate on II and III as in (\ref{qheshi}), we only need to consider the following three integrals:

\begin{equation}\label{intg1}\left\vert\int_{T^{-n_1}(I_{Y+1,1})} h(x,t)\frac{\partial_x(\tan\theta_{n_1})\mathcal{M}(\tan\theta_{n_1})}{(1+tan^2\theta_{n_1})} dx\right\vert;~\int_{T^{-n_1}(I_{Y+1,1})}\vert h(x,t)\frac{\partial_x(\tan\theta_{n_1})\mathcal{M}(\tan\theta_{n_1})}{(1+tan^2\theta_{n_1})}\vert dx;\end{equation} \begin{equation}\label{intg2}\left\vert\int_{T^{-n_1}(I_{Y+1,1})} w(x,t)\mathcal{M}(\tan\theta_{n_1})dx\right\vert;~\int_{T^{-n_1}(I_{Y+1,1})} \left\vert w(x,t)\mathcal{M}(\tan\theta_{n_1})\right\vert dx\end{equation}
with $n_1$ satisfying $\eqref{tjz}$ and  $h(x,t)$, and $w(x,t)$ defined as in (\ref{theta_t1}) satisfying (\ref{lem50-0}).
\vskip 0.3cm
\noindent \subsubsection{\text{The estimate on \eqref{intg1}}}\quad
\vskip 0.2cm
Denote the endpoints of $I_{Y+1}$ by $a_1$ and $a_2,$ then by the uniform non degeneracy of $\tan \theta_{n_1}(x,t),$ we have

 \begin{equation}\label{bjtj}|\tan \theta_{n_1}(a_1,t)|,|\tan \theta_{n_1}(a_2,t)|\geq c|(I_{Y+1,1})|^C\geq c\lambda^{-q^{\hat{\epsilon}}_{N+Y}}\geq c\lambda^{-|\log n|^{\hat{\epsilon}}}\gg c\lambda^{-e^{|\log n|^{\hat{\epsilon}}}}\geq  \epsilon_1.\end{equation}

On the other hand, by \eqref{k-reduction} and \eqref{bjtj}, for simplification, we can assume that \begin{equation}\label{bjtj0}|\tan \theta_{n_1}(a_1,t)|=|\tan \theta_{n_1}(a_2,t)|,\end{equation} since the difference between $\tan \theta_{n_1}(a_1,t)$ and $\tan \theta_{n_1}(a_2,t)$  only brings out an error term $C\lambda^{-|\log n|^{C}}$, which is ignorable in comparison with the error estimate in Lemma \ref{lemma18}.
\begin{lemma}\label{dierjiandan}If $\eqref{tjz}$ holds true, then we have
\begin{equation}\label{dier-1}\left\vert\int_{T^{-n_1}(I_{Y+1,1})}h(x,t)\frac{\partial_x(\tan\theta_{n_1})\mathcal{M}(\tan\theta_{n_1})}{1+\tan^2\theta_{n_1}} dx\right\vert\leq C(\log k)^C\end{equation}

and

\begin{equation}\label{dier-2}\int_{T^{-n_1}(I_{Y+1,1})}\left\vert h(x,t)\frac{\partial_x(\tan\theta_{n_1})\mathcal{M}(\tan\theta_{n_1})}{1+\tan^2\theta_{n_1}}\right\vert dx\leq C\cdot n^2.\end{equation}
\end{lemma}

\begin{proof}
Note that $\vert \mathcal{M}(x)\vert\leq \frac{C}{\sqrt{x^2+\epsilon_1^4}}$ ~(recall the definition of $\epsilon_1\ll 1$ in \eqref{epsilon1}). Therefore $$\begin{array}{ll}&\int_{T^{-n_1}(I_{Y+1,1})}\vert\frac{\partial_x(\tan\theta_{n_1})\mathcal{M}(\tan\theta_{n_1})}{1+\tan^2\theta_{n_1}}\vert dx \leq C\int_{T^{-n_1}(I_{Y+1,1})}\frac{|\partial_x(\tan\theta_{n_1})|}{(1+\tan^2\theta_{n_1})(\sqrt{\tan^2\theta_{n_1}+\epsilon_1^4})} dx\\&=C\int_{\tan\theta_{n_1}(T^{-n_1}(I_{Y+1,1}))}\frac{1}{(1+y^2)(\sqrt{y^2+\epsilon_1^4})} dy=C\left.\frac{\log \left\vert\frac{\sqrt{y^2+\epsilon_1^4}+\sqrt{1-\epsilon_1^4}y}{\sqrt{y^2+\epsilon_1^4}-\sqrt{1-\epsilon_1^4}y}\right\vert}{2\sqrt{1-\epsilon_1^4}}\right\vert_{a}^{b}\leq C\left.\log \left\vert\frac{\sqrt{y^2+\epsilon_1^4}+\sqrt{1-\epsilon_1^4}y}{\sqrt{y^2+\epsilon_1^4}-\sqrt{1-\epsilon_1^4}y}\right\vert\right\vert_{a}^{b}\end{array}$$ with some numbers $a, b$ satisfying $\epsilon_1\ll |T^{-n_1}(I_{Y+1,1})|^C\leq |a|,|b|\ll 1$ (by \eqref{bjtj}).

Hence it follows from the fact $\epsilon_1^C\leq \left\vert\frac{\sqrt{y^2+\epsilon_1^4}+\sqrt{1-\epsilon_1^4}y}{\sqrt{y^2+\epsilon_1^4}-\sqrt{1-\epsilon_1^4}y}\right\vert\leq \epsilon_1^{-C}$ and (\ref{tjz}) that $$\int_{T^{-n_1}(I_{Y+1,1})}\left\vert\frac{\partial_x(\tan\theta_{n_1})\mathcal{M}(\tan\theta_{n_1})}{1+\tan^2\theta_{n_1}}\right\vert dx\leq C|\log \epsilon_1|<Cn.$$ Finally, $|h(x,t)|\leq C(\log k)^C<n,$ which follows from Lemma \ref{lemmathetat1}, concludes \eqref{dier-2}.


Next, we denote $\tilde{\epsilon}_2(c_{Y+1,1},t):=\epsilon_2.$ Then \eqref{ktheta1-1} and Lemma \ref{AZH} show that for any interval $I\subset \R/\Z$ centered by $c_{Y+1,1}$,

\begin{equation}\label{tildeep2}\left\vert \int_{I}\frac{\frac{XdX}{1+X^2}}{\sqrt{X^4+\epsilon_1^4 X^2 +\tilde{\epsilon}^8_2}}-\frac{\frac{XdX}{1+X^2}}{\sqrt{X^4+\epsilon_1^4 X^2 +\epsilon^8_2}}\right\vert\leq C\epsilon_1^{\frac{1}{4}}.\end{equation}

\eqref{tildeep2} together with \eqref{k-reduction} implies that \begin{equation}\label{lem53-1}\vert\int_{T^{-n_1}(I_{Y+1,1})}\frac{\partial_x(\tan\theta_{n_1})\mathcal{M}(\tan\theta_{n_1})dx}{1+\tan^2\theta_{n_1}}-\frac{\frac{\tan\theta_{n_1}\partial_x \tan\theta_{n_1}dx}{1+\tan\theta_{n_1}^2}}{\sqrt{\tan\theta_{n_1}^4+\epsilon_1^4 \tan\theta_{n_1}^2 +\epsilon^8_2}}\vert<C \epsilon_1^{\frac{1}{4}}.\end{equation}
(\ref{bjtj0}) then together with (\ref{lem53-1}) shows
$$\begin{array}{ll}\left\vert\int_{T^{-n_1}(I_{Y+1,1})}\frac{\frac{\tan\theta_{n_1}\partial_x\tan\theta_{n_1}dx}{1+\tan\theta_{n_1}^2}}{\sqrt{\tan\theta_{n_1}^4+\epsilon_1^4 \tan\theta_{n_1}^2 +\epsilon^8_2}}\right\vert &=\int_{\tan\theta_{n_1}(a_1(t),t)}^{\tan\theta_{n_1}(a_2(t),t)}\frac{\frac{\tan\theta_{n_1} d \tan\theta_{n_1}}{1+\tan\theta_{n_1}^2}}{\sqrt{\tan\theta_{n_1}^4+\epsilon_1^4 \tan\theta_{n_1}^2 +\epsilon^8_2}}=0.\end{array}$$

Therefore \begin{equation}\label{partial-tan-k}\vert\int_{T^{-n_1}(I_{Y+1,1})}\frac{\partial_x(\tan\theta_{n_1})\mathcal{M}(\tan\theta_{n_1})dx}{1+\tan^2\theta_{n_1}}\vert\leq C\epsilon_1^{\frac{1}{4}}.\end{equation}

Next, \eqref{lem50-0} implies that $h(x,t)\sim_{0,|I_{Y+1,1}|^c} \frac{\bar{v}_2}{\bar{u}_2}.$ Combining this with \eqref{uhv}, \eqref{dier-2}, \eqref{lem53-1}  and \eqref{partial-tan-k}, we have
$$\begin{array}{ll}\vert\int_{T^{-n_1}(I_{Y+1,1})}\frac{h(x,t)\partial_x(\tan\theta_{n_1})\mathcal{M}(\tan\theta_{n_1})dx}{1+\tan^2\theta_{n_1}}\vert &\leq \vert\int_{T^{-n_1}(I_{Y+1,1})}\frac{\frac{\bar{u}_2}{\bar{v}_2}\partial_x(\tan\theta_{n_1})\mathcal{M}(\tan\theta_{n_1})dx}{1+\tan^2\theta_{n_1}}\vert+C|I_{Y+1,1}|^{c}(\log k)^Cn^2\\&\leq C(\log k)^C\epsilon_1^{\frac{1}{4}}+C\lambda^{-(\log n)^C}\leq C(\log k)^C.\hfill\qed\end{array}$$ yielding \eqref{dier-1}.
\end{proof}

\begin{remark} For $j_i(t)\leq l< s(k_{i+1}(t))+2,$ $q^{C}_{N+l-1}\leq Y_1,Y_2\leq {r}_{l},$ Lemma \ref{niceset1} leads that  $A_{Y_1}(x-Y_1\alpha)$ is $(\lambda^{-|\log q_{N+l-1}|^{C}},-)-\text{nice}$ on $I_{l}$ and $A_{Y_2}(x)$ is $(\lambda^{-|\log q_{N+l-1}|^{C}},+)-\text{nice}$ on $I_{l,}$ which implies $\|A_{Y_i}\|\geq \lambda^{\frac{1}{2}Y_1}\geq \lambda^{q^{c}_{N+l-1}}\gg |~I_{l}|^{-1}.$ We have the same argument as \eqref{bjtj0} for $\tan\theta^{Y_1+Y_2}_{Y_1}$. Then by a similar calculation to Lemma \ref{dierjiandan}, we have
\begin{equation}\label{htheta}\left\vert\int_{{~I_{l,j}}}h^{l,j}(x,t)\frac{\partial_x(\tan\theta^{Y_1+Y_2}_{Y_1})\mathcal{M}(\tan\theta^{Y_1+Y_2}_{Y_1})}{1+\tan^2\theta^{Y_1+Y_2}_{Y_1}} dx\right\vert \leq C\cdot (\log k_i)^C,~j=1,2.\end{equation} By the same argument as in Remark \ref{zhongk}, the above also holds true if we replace $\theta^{Y_1+Y_2}_{Y_1}$ with $\theta^{{r}^-_l(x)+Y_1}_{{r}^-_l(x)}$ or $\theta^{{r}^+_l(x)+Y_2}_{Y_2}.$
\end{remark}
\begin{remark}\label{remark51}  Suppose $s(k_{i}(t))+2\leq l< j_i(t),$ $Y_1,Y_2\in \Z_+$ satisfy $\max\{Y_1,Y_2\}\geq q^{C}_{N+l-1}$ and $\min\{Y_1,Y_2\}\geq c\lambda^{|\log k_i|^c}$. Then by the fact $\epsilon_1^C\leq \left\vert\frac{\sqrt{y^2+\epsilon_1^4}+\sqrt{1-\epsilon_1^4}y}{\sqrt{y^2+\epsilon_1^4}-\sqrt{1-\epsilon_1^4}y}\right\vert\leq \epsilon_1^{-C}$ and Lemma \ref{niceset1}, one can check that by the similar proof of Lemma \ref{dierjiandan}, we indeed have
\begin{equation}\label{htheta10}\int_{{~I_{l,j}}}\vert(h^{l,j}(x,t))\frac{\partial_x(\tan\theta^{Y_1+Y_2}_{Y_1})\mathcal{M}(\tan\theta^{Y_1+Y_2}_{Y_1})}{1+\tan^2\theta^{Y_1+Y_2}_{Y_1}}\vert dx\leq C\cdot (\log k_i)^C( \min\{Y_1,Y_2\})^C,~j=1,2.\end{equation} By the same argument as Remark \ref{zhongk}, the above also holds true if we replace $\theta^{Y_1+Y_2}_{Y_1}$ by $\theta^{{r}^-_l(x)+Y_1}_{{r}^-_l(x)}$ or $\theta^{{r}^+_l(x)+Y_2}_{Y_2}.$
\end{remark}

\noindent \subsubsection{\text{The estimate on \eqref{intg2}}}\quad

To estimate \eqref{intg2},  we separately consider the following two different cases (recall that $\zeta_n= e^{-(\log n)^c}$):

$\mathcal{P}_1(t):=\{x\in T^{-n_1}(I_{Y+1,1})|\vert\partial_x \tan\theta_{n_1}\vert\geq \zeta_n\}$ and $~\mathcal{P}_2(t):=T^{-n_1}(I_{Y+1,1})-\mathcal{P}_1.$
\vskip 0.2cm

\noindent{\it $\lozenge$ The estimate on $\mathcal{P}_1(t)$: }\quad

\begin{lemma}\label{case1easy'}If $\eqref{tjz}$ holds true, for $w$ and $\mathcal{M}(\tan\theta_{n_1})$ defined as in
(\ref{theta_t1}), (\ref{lem50-0}) and (\ref{k(f)}), respectively, we have $$ \int_{\mathcal{P}_1}|w(x,t)\mathcal{M}(\tan\theta_{n_1}) |dx\leq C\cdot n\cdot \zeta_n^{-C}.$$
\end{lemma}

\begin{proof} On one hand, by the help of Lemma \ref{lemmathetat1}, we have $\vert w(x,t) \vert\leq C(\log k)^C.$ On the other hand, for\ any $ x\in \mathcal{P}_1,$  we have $\vert\partial_x \tan\theta_{n_1}\vert=(1+\tan^2\theta_{n_1})\vert\partial_x \theta_{n_1}\vert\geq \zeta_n.$ By the help of (\rm iii) of Lemma \ref{lemmathetat1}, there exist  at most $3$ intervals denoted by $H_j, 1\leq j\leq N\le 3$ such that $\mathcal{P}_1=\bigcup\limits_{j=1}^{N} H_j.$ Hence there exists a piecewise linear function $P(x)$ such that $$|\tan\theta_{n_1}|\geq |P(x)|,\quad |P'(x)|\geq \zeta_n^{2},\ \ \ \ \ \ \  ~x\in \bigcup\limits_{j=1}^{N} H_j.$$ Then, by the fact $\vert k(x)\vert\leq \frac{1}{\sqrt{x^2+\epsilon_1^4}}$ and Lemma \ref{usefullemma2}, we have $$\int_{\mathcal{P}_1}|w(x,t)\mathcal{M}(\tan\theta_{n_1}) |dx\leq  C\cdot |\log k|^C\cdot\left( 3C\zeta_n^{-2}|\log \epsilon_1|\right)\leq Cn\zeta_n^{-C}.\hfill\qed$$
\end{proof}
\begin{remark} By the similar proof and Lemma \ref{lemmathetat1}, it is not difficult to see that for $s(k_{i}(t))+2\leq l< j_i(t),$ we have
\begin{equation}\label{htheta11}\begin{array}{ll}&\int_{{~I_{l,j}}\bigcap \{\vert\partial_x \theta_{Y_1}^{Y_1+Y_2}\vert\geq \zeta_{\mathcal{N}_l}\}}\vert(w^{l,j}(x,t))\mathcal{M}(\tan\theta_{Y_1}^{Y_1+Y_2})\vert dx\\&\leq C\cdot (\log k_i)^C \min\{Y_1,Y_2\}^C \cdot \zeta_{\mathcal{N}_l}^{-C}\leq C\min\{Y_1,Y_2\}^C\cdot (\lambda^{(\log \mathcal{N}_{j_i(t)})^{C}}),~j=1,2.\end{array}\end{equation}  By the same argument as Remark \ref{zhongk}, the above also holds true if we replace $\theta^{Y_1+Y_2}_{Y_1}$ with $\theta^{{r}^-_l(x)+Y_1}_{{r}^-_l(x)}$ or $\theta^{{r}^+_l(x)+Y_2}_{Y_2}.$
\end{remark}

\noindent{\it $\lozenge$ The estimate on $\mathcal{P}_2(t)$: }\quad
Note that by \eqref{jiashe001} and \eqref{theta111} we have $$g_{Y+1,1}\sim_{2,\zeta_n}\tan^{-1}(\|A_{k}(x,t)\|^2\tan \left(\hat{g}_{Y,2}\right))-\frac{\pi}{2}+\hat{g}_{Y,1}:=\breve{g}_{Y+1,1},$$ which implies \begin{equation}\label{gcwc}\|\breve{g}_{Y+1,1}-\theta_{n_1}\|_{C^0(\tilde{D}_{Y+1})}\leq C(\|\Delta_{n_1}\|_{C^0(\tilde{D}_{Y+1})}+\|\Delta_{n-n_1}\|_{C^0(\tilde{D}_{Y+1})}).\end{equation}
\begin{lemma}\label{property1}Let $z_1(t)$ and $\omega$ be defined as in Lemma \ref{lemmathetat1} and (\ref{translation-notation}), respectively. If
$\|\Delta_{n_1}\|_{C^0}\leq \zeta_n|t-t^k_+|$ with $\Delta_{n_1}$  defined as in (\ref{thetat1}) and (\ref{thetat11}), then \begin{equation}\label{duandianshuzhi} |\tan \theta_{n_1}(z_1(t),t)|\sim_{0,\zeta_n} 2\omega\bar{u}_1|t-t^k_{+}|.\end{equation}
\end{lemma}

\begin{proof}Note that $\vert \tan\theta_{n_1}\vert+\vert \partial_x\tan\theta_{n_1}\vert\leq \zeta^c_{n}$ for any $x\in \mathcal{P}_{2}.$ Combining this with \eqref{theta_t1} and \eqref{lem50-0}, we have
$$\begin{array}{ll}
\partial_t \tan\theta_{n_1}(x,t)&=(-1-\tan^2\theta_{n_1})\partial_t\theta_{n_1}=h(x,t)\partial_x\tan\theta_{n_1}-(1+\tan^2\theta_{n_1})w(x,t)\\
\\&\sim_{0,\zeta^c_n}-w(x,t)\sim_{0,\epsilon^c_{n}}-2\omega\bar{u}_1.\end{array}$$ Therefore
\begin{equation}\label{gcwc1}|\tan \theta_{n_1}(z_1(t),t)-\tan \theta_{n_1}(z_1(t^k_+),t^k_+)|\sim_{0,\zeta_n^c} 2\omega\bar{u}_1|t-t^k_{+}|.\end{equation} On the other hand, by \eqref{gcwc} we have $$\|g_{Y+1,1}-\theta_{n_1}\|_{C^2}\leq C\left(\|\Delta_{n_1}\|_{C^0}+\|\Delta_{n-n_1}\|_{C^0}\right) \leq 2C\|\Delta_{n_1}\|_{C^0}\leq C\zeta_n|t-t^k_+|.$$  Moreover it is clear that $$r_{Y}\gg n;~\zeta_n,~|t-t^k_+|\geq c\lambda^{-e^{(\log\log n)^{\hat{\epsilon}^{-1}}}}\gg \lambda^{-\hat{\epsilon}^{-1}n}.$$ Combining these with (1) of Theorem \ref{15}, we have \begin{equation}\label{gcwc2}\begin{array}{ll}\vert \tan\theta_{n_1}(z_1(t^k_+),t^k_+)\vert &\leq |g_{Y+1,1}(\tilde{c}_{Y+1,1}(t^k_+))|+|g_{Y+1,1}(\tilde{c}_{Y+1,1}(t^k_+))-\tan\theta_{n_1}(z_1(t^k_+),t^k_+)|\\& \leq \lambda^{-\hat{\epsilon}^{-1}n}+C\zeta_n|t-t^k_+|\leq 2\hat{\epsilon}^{-1}\zeta_n|t-t^k_+|.\end{array}\end{equation}
Then, by \eqref{gcwc1} and \eqref{gcwc2}, we immediately obtain \eqref{duandianshuzhi}.\hfill\qed\end{proof}

\begin{lemma}\label{property2}If $\|\Delta_{n_1}\|_{C^0}> \zeta_n|t-t^k_+|$, then $\epsilon_1^2\geq c\lambda^{-(\log n)^c}|t-t^k_+|,$  where $\epsilon_1$ is defined by (\ref{epsilon1}).\end{lemma}
\begin{proof} By \eqref{thetat11}, we have $$ c(\lambda^{-(\log n)^c})|t-t^k_+|\leq \zeta_n|t-t^k_+|<\|\Delta_{n_1}\|_{C^0}\leq  C\lambda_1^{-2}\lambda^{(\log n_1)^{\hat{\epsilon}^{-2}}}.$$ Hence \begin{equation}\label{lambda12}\lambda_1^{-2}\geq c\lambda^{-(\log n_1)^{\hat{\epsilon}^{-3}}}|t-t^k_+|.\end{equation} On the other hand, by $\zeta_n|t-t^k_+|\leq \|\Delta_{n_1}\|_{C^0}\leq C\lambda^{-n_1}$ and $\lambda^{-e^{(\log \log n)^{\hat{\epsilon}^{-2}}}} \leq |t-t^k_+|\leq \lambda^{-(\log n)^{\hat{\epsilon}^{-2}}}$ (from \eqref{t-tk}), we have $$n_1\leq C(|\log (\zeta_n)|+e^{(\log \log n)^{\hat{\epsilon}^{-2}}})\leq 2Ce^{(\log \log n)^{\hat{\epsilon}^{-2}}}.$$ Therefore \eqref{lambda12} implies that $\lambda_1^{-2}\geq ce^{-(\log\log n)^{{\hat{\epsilon}^{-10}}}}|t-t^k_+|\geq c\lambda^{-(\log n)^{2c}}|t-t^k_+|$ and the definition of $\epsilon_1$ implies that $\epsilon_1^2\geq c(\lambda_1^{-2})\geq c\lambda^{-(\log n)^{c}}|t-t^k_+|$ as desired.  \hfill\qed\end{proof}

Combining \eqref{153}, Lemma \ref{property1} with Lemma \ref{property2}, we immediately have

\begin{lemma}\label{property3}Let $z_1(t)$ and $z_2(t)$ be defined as above. We have
$$\sum\limits_{i=1,2}\frac{1}{\sqrt{\epsilon_1^2+\vert\tan\theta_{n_1}(z_i(t),t)\vert}}\leq \lambda^{(\log n)^c}\cdot \frac{1}{\sqrt{\vert t- t^k_+ \vert}}.$$
\end{lemma}

Now we estimate $\int_{\mathcal{P}_{2}}\vert w(x,t)\mathcal{M}(\tan\theta_{n_1}) \vert dx$ and $|\int_{\mathcal{P}_{2}} w(x,t)\mathcal{M}(\tan\theta_{n_1})dx|$.
\begin{lemma}\label{hexinl}Let $z_i(t),i=1,2$ be defined as above and $l_k$ is defined as in (\ref{lkdf}). It holds that

\begin{equation}\label{eq312121}\int_{\mathcal{P}_{2}}\vert w(x,t)\mathcal{M}(\tan\theta_{n_1}) \vert dx\leq C(\vert\log \epsilon_1\vert)\cdot(\log k) \sum\limits_{i=1,2}\left(\sqrt{l_k\cdot \vert\tan\theta_{n_1}(z_i(t),t)\vert+l_k\cdot \epsilon_1^2}\right)^{-1}.\end{equation}

Moreover, with the notation $\Omega(t)=\sqrt{\left|{\tan\theta_{n_1}}(z_1(t),t)\cdot (\frac{\partial^2 \tan\theta_{n_1}}{\partial x^2}(z_1(t),t))^{-1}\right|}$, we have

\begin{enumerate}

\item[Case 1:] If  $\vert\tan\theta_{n_1}(z_1(t),t)\vert \geq \zeta_{n}l_k^{-1},$ then \begin{equation}\label{eq312}\int_{\mathcal{P}_{2}}|w(x,t)\mathcal{M}(\tan\theta_{n_1}) |dx\leq C\zeta_n^{-C}\cdot n ;\end{equation}
\begin{equation}\label{eq312''}\int_{\mathcal{P}_{2}\bigcap\{|\tan\theta_{n_1}|\geq \zeta_{n}\cdot          \Omega(t)\}}(|w(x,t)\cot\theta_{n_1}|+|\cot\theta_{n_1}|)dx\leq C\zeta_n^{-C}\cdot n;\end{equation}
\item[Case 2:] If  $\vert\tan\theta_{n_1}(z_1(t),t)\vert< \zeta_{n}l_k^{-1}$ and $\frac{\epsilon_1^2}{\vert\tan\theta_{n_1}(z_1(t),t)\vert}\leq (\lambda^{-cn^c})$, then it holds that

\begin{enumerate}
\item[$2_a:$] If $\tan\theta_{n_1}$ has zeros on $P_{2}$, then we have \begin{equation}\label{eq31212}\vert\int_{\mathcal{P}_{2}}w(x,t)\mathcal{M}(\tan\theta_{n_1}) dx\vert\leq \zeta_n^{-\frac{1}{10}}(l_k\cdot \vert\tan\theta_{n_1}(z_1(t),t)\vert)^{-\frac{1}{2}}+C\zeta_n^{-C}\cdot n;\end{equation}
    \begin{equation}\label{eq3121211}\vert\int_{\mathcal{P}_{2}\bigcap\{|\tan\theta_{n_1}|\geq \zeta_{n}\cdot \Omega(t)\}}w(x,t)\cot\theta_{n_1}dx\vert \leq \zeta_n^{-\frac{1}{10}}(l_k\cdot \vert\tan\theta_{n_1}(z_1(t),t)\vert)^{-\frac{1}{2}}+C\zeta_n^{-C}\cdot n;\end{equation}
      \begin{equation}\label{eq31212111}\int_{\mathcal{P}_{2}\bigcap\{|\tan\theta_{n_1}|\geq \zeta_{n}\cdot \Omega(t)\}}\vert \cot\theta_{n_1}\vert dx \leq C(\log l_k+\vert \log \tan\theta_{n_1}(z_1(t),t)\vert)(l_k\cdot \vert\tan\theta_{n_1}(z_1(t),t)\vert)^{-\frac{1}{2}}+C\zeta_n^{-C}\cdot n.\end{equation}

\item[$2_b:$]  if $\tan\theta_{n_1}$ has no zeros on $P_{2}$, then we have \begin{equation}\label{eq31212'}\vert\int_{\mathcal{P}_{2}}w(x,t)\mathcal{M}(\tan\theta_{n_1}) dx-M_k(t)\vert \leq C\zeta_k^{-C}\cdot n,\quad \vert M_k(t)\vert \sim_{0,\zeta_n^{c}} (l_k\cdot \vert\tan\theta_{n_1}(z_1(t),t)\vert)^{-\frac{1}{2}}, \end{equation} where $\vert M_k(t)\vert= C_{k}(t)\frac{2\vert t-\frac{t^k_-+t^k_+}{2}\vert}{\sqrt{|(t-t^k_+)(t-t^k_-)|}}$ and $C_{k}(t)\sim_{0, ck^{-1}} C_k^*$ with $0<(C_{k}^*)^{-1}\leq C(\log k)^C.$
    \end{enumerate}
\end{enumerate}
\end{lemma}

\begin{proof} Note that $\vert w(x,t)\vert\leq C(\log k)^C\leq Ck \leq C\vert\log \epsilon_1\vert$ and $|\mathcal{M}(\tan\theta_{n_1})|\leq \frac{1}{\sqrt{\tan^2\theta_{n_1}+\epsilon_1^4}}\leq \frac{\sqrt{2}}{|\tan\theta_{n_1}|+\epsilon_1^2}.$ By the help of Lemma \ref{useful111}, to obtain \eqref{eq312121}, it is enough to prove that \begin{equation}\label{mubiaogj}\vert\tan \theta_{n_1}\vert\geq c\sum_{i=1}^2\vert A l_k(x-z^*)^2+B_i \vert\end{equation} for $x\in \mathcal{P}_{2}$ with some $z^*\in \R/\Z,$ $(\log k)^{-C}l_k\leq |A|\leq (\log k)^{C}l_k$ and $|B_i|\sim_{0,\zeta_n^c} |\tan(\theta_{n_1}(z_i))|,~i=1,2.$

Then Lemma \ref{lm104} implies \eqref{mubiaogj} as desired.

\eqref{eq312}-\eqref{eq31212'} can be obtained from an easy corollary from Lemma \ref{usefullemma1} and \eqref{lem50-1'}.
\hfill\qed\end{proof}

Recall that $n_0=n^{\frac{1}{100}}$. By Lemma \ref{hexinl}, we have the following precise estimates.
\begin{corollary}\label{hexinle}Let $z_i(t),i=1,2$ be defined as above and assume that $\eqref{tjz}$ holds true.  Then  the following hold true.

\begin{enumerate}

\item[A:]If $n_1\leq n_0,$ then it holds that

\begin{equation}\label{eq312121*}\int_{\mathcal{P}_{2}}\vert w(x,t)\mathcal{M}(\tan\theta_{n_1}) \vert dx\leq Cn^{\frac{1}{50}}\cdot \frac{\vert t-\frac{t^k_-+t^k_+}{2}\vert}{\sqrt{|(t-t^k_+)(t-t^k_-)|}}.\end{equation}

\item[B:]If $n_1> n_0$, then the following two conclusions hold true.

\begin{enumerate}
\item[$B_1:$] If $\vert\tan\theta_{n_1}(z_1(t),t)\vert > \zeta_{n}l_k^{-1}$, then
\begin{equation}\label{142} \vert\int_{\mathcal{P}_{2}}w(x,t)\mathcal{M}(\tan\theta_{n_1}) dx\vert\leq C(\lambda^{(\log n)^C});\end{equation}
\begin{equation}\label{143}\int_{\mathcal{P}_{2}\bigcap\{|\tan\theta_{n_1}|\geq \zeta_{n}\cdot \Omega(t)\}}(\vert w(x,t)\cot\theta_{n_1} \vert +\vert\cot\theta_{n_1}\vert) dx\leq C(\lambda^{(\log n)^C}).\end{equation}
\item[$B_2:$] If $\vert\tan\theta_{n_1}(z_1(t),t)\vert \leq \zeta_{n}l_k^{-1}$, then it holds that

\begin{equation}\label{eq31212*}\vert\int_{\mathcal{P}_{2}}w(x,t)\mathcal{M}(\tan\theta_{n_1}) dx\vert\leq \lambda^{-\frac{1}{20}k}\frac{\vert t-\frac{t^k_-+t^k_+}{2}\vert}{\sqrt{|(t-t^k_+)(t-t^k_-)|}}+C\lambda^{(\log n)^C};\end{equation}
\begin{equation}\label{eq144}\vert\int_{Q}w(x,t)\cot\theta_{n_1} dx\vert \leq \lambda^{-\frac{1}{20}k}\frac{\vert t-\frac{t^k_-+t^k_+}{2}\vert}{\sqrt{|(t-t^k_+)(t-t^k_-)|}}+C\lambda^{(\log n)^C};\end{equation}
\begin{equation}\label{eq149}\int_{Q}|\cot\theta_{n_1}| dx\leq C(\vert\log|t-t_+^k|\vert+\vert\log|t-t_-^k|\vert)\frac{\vert t-\frac{t^k_-+t^k_+}{2}\vert}{\sqrt{|(t-t^k_+)(t-t^k_-)|}}+C\lambda^{(\log n)^C},
\end{equation}
where $Q=\mathcal{P}_{2}\bigcap\{|\tan\theta_{n_1}|\geq \zeta_{n}\cdot \Omega(t)\}.$
\vskip 0.15cm
\noindent Particularly, for $t\in (t^k_-,t^k_+)\bigcap \left(\mathcal{B}_{Y+1}(t^{k}_{-})\bigcup\mathcal{B}_{Y+1}(t^{k}_{+})\right)$, we have \begin{equation}\label{eq31212'*}\vert\int_{\mathcal{P}_{2}}w(x,t)\mathcal{M}(\tan\theta_{n_1}) dx-M_k(t)\vert \leq C\lambda^{(\log n)^C}, \end{equation} where $ M_k(t)$ is defined as in Lemma \eqref{hexinl}.
\end{enumerate}\end{enumerate}
\end{corollary}

\begin{proof} Lemma \ref{property3}, Lemma \ref{hexinl} and the fact $\epsilon_1\leq C\lambda^{-n_0}= C\lambda^{-n^{\frac{1}{100}}}$ clearly imply that \begin{equation}\label{jichugj}\int_{\mathcal{P}_{2}}\vert w(x,t)\mathcal{M}(\tan\theta_{n_1}) \vert dx\leq Cn^{\frac{1}{100}}(\lambda^{C(\log n)^c})\cdot \frac{1}{\sqrt{l_k|t-t^k_+|}}\leq Cn^{\frac{1}{80}}\cdot \frac{1}{\sqrt{l_k|t-t^k_+|}}.\end{equation} Note that if $|t-t^k_+|\leq \zeta_{n}l_k^{-1}$, then it follows from \eqref{gapgja} of Lemma \ref{pugapgj} that \begin{equation}\label{jinsishuzhi}\frac{|t-\bar{t}^k|}{\sqrt{|t-t^k_-|}}\sim_{0,\zeta_n} d^*l_k^{-\frac{1}{2}}\end{equation} with $c(\log k)^{-C}<d^*,(d^*)^{-1}\leq C(\log k)^C.$ Then the right side of \eqref{jichugj} is less than $Cn^{\frac{1}{50}}\cdot \frac{\vert t-\frac{t^k_-+t^k_+}{2}\vert}{\sqrt{|(t-t^k_+)(t-t^k_-)|}}.$ If $|t-t^k_+|> \epsilon_{n}l_k^{-1}$, then the right side is less than $Cn^{\frac{1}{50}}\cdot \zeta_n^{-C}\leq C\lambda^{(\log n)^C}.$ Therefore \eqref{eq312121*} holds true.

For the remaining  inequalities, on one hand we have $$\|\Delta_{n_1}\|_{C^0} \leq C\epsilon_1\leq C\lambda^{-n^{\frac{1}{100}}}\ll c\lambda^{-e^{(\log \log n)^{C}}}\leq  \zeta_n|t-t^k_+|.$$ (Note \eqref{t-tk} implies $\lambda^{-e^{(\log \log n)^{C}}} \leq |t-t^k_+|\leq \lambda^{-(\log n)^{C}}$) Then, by the help of Lemma \ref{property1}, \eqref{duandianshuzhi} holds true.  On the other hand, we have $$\epsilon_1^2\leq C\lambda^{-n^{\frac{1}{100}}}\ll c\lambda^{-e^{(\log \log n)^{C}}}\leq  C|t-t^k_+|,$$ which implies that \eqref{eq31212} holds true. Combining \eqref{duandianshuzhi} and our assumption $\vert\tan\theta_{n_1}(z_i(t),t)\vert \leq \zeta_nl_k^{-1}$, we immediately have $|t-t^k_+|\leq \zeta_nl_k^{-1},$ which yields \eqref{jinsishuzhi}. Then, combining \eqref{duandianshuzhi}, \eqref{eq31212} and \eqref{jinsishuzhi}, we immediately obtain \eqref{eq31212*}, \eqref{eq144} and \eqref{eq149}. Combining \eqref{duandianshuzhi}, \eqref{eq31212'} and \eqref{jinsishuzhi}, we obtain \eqref{eq31212'*}. \eqref{142} and \eqref{143} directly follow from case 1 of Lemma \ref{hexinl} by the fact $\zeta_n^{-1}\leq n.$
\hfill\qed\end{proof}
\begin{remark}\label{remark-62} Note that by a proof similar to the one for \eqref{eq312121*} (replace $n,n_0,n_1$ with $r_l+Y,Y,r_l$ and apply the fact that $r_l\geq \lambda^{q^c_{N+l-1}}\gg q^C_{N+l-1}$), it indeed holds that for $q^c_{N+l-1}\leq Y\leq q^C_{N+l-1}$ and $s(k_i(t))+2\leq l<s(k_{i+1}(t))+2,~j=1,2$: \begin{equation}\label{1234511}\int_{~I_{l,j}\bigcap \{\vert\partial_x \tan\theta^{{r}_l+Y}_{{r}_l}\vert\leq \zeta_{\mathcal{N}_l}\}}\vert w^{l,j}(x,t)\mathcal{M}(\tan\theta^{{r}_l+Y}_{{r}_l}) \vert dx\leq Cq^{C}_{N+l-1}\cdot \frac{\vert t-\frac{t^{k_i}_-+t^{k_i}_+}{2}\vert}{\sqrt{|(t-t^{k_i}_+)(t-t^{k_i}_-)|}}.\end{equation}
Combining \eqref{htheta11} with \eqref{1234511} (note that in the case $j_i(t)\leq l<s(k_{i+1}(t))+2,$ we do not need to consider $\mathcal{P}_1$ since $\vert\partial_x g_{\bar{l}}\vert\geq \vert~I_l \vert^c\geq \zeta_{\mathcal{N}_l},\  x\in ~I_l,$ which follows from the definition of $\bar{l}$ at the beginning of Section 4), we obtain \begin{equation}\label{1234512}\int_{~I_{l,j}}\vert w^{l,j}(x,t)\mathcal{M}(\tan\theta^{{r}_l+Y}_{{r}_l}) \vert dx\leq Cq^{C}_{N+l-1}\cdot \left(\frac{\vert t-\frac{t^{k_i}_-+t^{k_i}_+}{2}\vert}{\sqrt{|(t-t^{k_i}_+)(t-t^{k_i}_-)|}}+\lambda^{(\log \mathcal{N}_{j_i(t)})^C}\right).\end{equation}

Then \eqref{htheta10} and \eqref{1234512} imply
\begin{equation}\label{12345121}\int_{~I_{l,j}}\vert \partial_t \theta^{{r}_l+Y}_{{r}_l} \mathcal{M}(\tan\theta^{{r}_l+Y}_{{r}_l}) \vert dx\leq Cq^{C}_{N+l-1}\cdot \left(\frac{\vert t-\frac{t^{k_i}_-+t^{k_i}_+}{2}\vert}{\sqrt{|(t-t^{k_i}_+)(t-t^{k_i}_-)|}}+\lambda^{(\log \mathcal{N}_{j_i(t)})^C}\right).\end{equation}

And Remark \ref{remark45} implies the above holds true for all $Y\leq q^C_{N+l-1}.$ By the same argument as Remark \ref{zhongk}, the above also holds true if we replace $\theta^{{r}_l+Y}_{{r}_l}$ with $\theta^{{r}^-_l(x)+Y}_{{r}^-_l(x)}$ or $\theta^{{r}^+_l(x)+Y}_{Y}.$
\end{remark}

\begin{remark}\label{remark-60} By \textbf{B} of Lemma \ref{lemmathetat1}, Corollary \ref{hexinle} also holds true if we replace $\theta_{n_1}$ with $g_{\bar{j}},$ where $s(k_{\tilde{i}}(t))+2\leq j< j_{\tilde{i}}(t)$ with $\tilde{i}$ satisfying $k_{\tilde{i}}=k.$ Particularly, it is not difficult to see that \begin{equation}\label{pxgj}\sum\limits_{l=1}^2\frac{1}{|\partial_x g_{\bar{j}}(c_{\bar{j},l}(t),t)|}\leq C(\log k_i)^C\frac{\vert t-\frac{t^{k_i}_-+t^{k_i}_+}{2}\vert}{\sqrt{|(t-t^{k_i}_+)(t-t^{k_i}_-)|}},\  j_i(t)\leq j< s(k_{i+1}(t))+2.\end{equation}

(The case $c_{\bar{j},l}(t)\in \mathcal{P}_2$ follows from the proof of Corollary \ref{hexinle} and the case $c_{\bar{j},l}(t)\in \mathcal{P}_1$ is trivial since the left summation is dominated by the second term of the right side.)
\end{remark}
\subsubsection{The estimate on II and III}
 Combining Lemma \ref{dierjiandan}, Lemma \ref{case1easy'} and Corollary \ref{hexinle}, we have that \eqref{eq312121*}, \eqref{eq31212*} and \eqref{eq31212'*} hold true by replacing $w(x,t)\mathcal{M}(\tan\theta_{n_1})$ with $\partial_t\theta_{n_1}\mathcal{K}(\tan\theta_{n_1})$ and $\mathcal{P}_{2}$ with $~I_{Y+1}.$ Then, \eqref{jifendejianjin} and the fact $\max\limits_{(x,t)\in \tilde{D}_{Y+1}}\bar{\epsilon}_1(x,t)\leq \epsilon_1^c\ll C(\log n)$ if $n_1\geq n_0$ imply the following Lemma.

\begin{lemma}\label{zhongyao}

For $0\leq n_1<n^{\frac{1}{100}}$ and $n-n^{\frac{1}{100}}<n_2<n-k,$ it holds that
 \begin{equation}\label{universalestimate}
\vert\mathcal{S}_{I_{Y+1,1},n-n_1,n_1}(t)\vert \leq C(n^{\frac{1}{50}})\cdot \frac{\vert t-\frac{t^k_-+t^k_+}{2}\vert}{\sqrt{|(t-t^k_+)(t-t^k_-)|}}+C\lambda^{(\log n)^C}.\end{equation}

Moreover, if $n^{\frac{1}{100}}<n_1\leq n-n^{\frac{1}{100}}$, then it holds that

\begin{equation}\label{pujideguji}\vert\mathcal{S}_{I_{Y+1,1},n-n_1,n_1}(t)\vert\leq C\lambda^{-\frac{1}{20}k}\frac{\vert t-\frac{t^k_-+t^k_+}{2}\vert}{\sqrt{|(t-t^k_+)(t-t^k_-)|}}+C\lambda^{(\log n)^C},\end{equation}
\begin{equation}\label{pujideguji2}\vert\int_{T^{-n_1}(I_{Y+1,1})\bigcap \{\vert\theta_{n_1}\vert\geq \zeta_{n}\cdot \Omega(t)\}}\cot \theta_{n_1}\cdot \partial_t\theta_{n_1} dx\vert\leq C\lambda^{-\frac{1}{20}k}\frac{\vert t-\frac{t^k_-+t^k_+}{2}\vert}{\sqrt{|(t-t^k_+)(t-t^k_-)|}}+C\lambda^{(\log n)^C},\end{equation} and
\begin{equation}\label{pujideguji3}\int_{T^{-n_1}(I_{Y+1,1})\bigcap \{\vert\theta_{n_1}\vert\geq \zeta_{n}\cdot \Omega(t)\}}\vert\cot \theta_{n_1}\vert dx\leq C(\vert\log|t-t_+^k|\vert+\vert\log|t-t_-^k|\vert)\frac{\vert t-\frac{t^k_-+t^k_+}{2}\vert}{\sqrt{|(t-t^k_+)(t-t^k_-)|}}+C\lambda^{(\log n)^C}.\end{equation}
Particularly, for $n^{\frac{1}{100}}<n_1\leq n-n^{\frac{1}{100}}$ and $t\in (t^k_-,t^k_+)\bigcap \left(\mathcal{B}_{Y+1}(t^{k}_{-})\bigcup\mathcal{B}_{Y+1}(t^{k}_{+})\right)$, we have \begin{equation}\label{pugapdeguji}\vert\mathcal{S}_{I_{Y+1,1},n-n_1,n_1}(t)-M_k(t)\vert \leq C\lambda^{(\log n)^C} \end{equation} where $\vert M_k(t)\vert= C_{k}(t)\frac{2\vert t-\frac{t^k_-+t^k_+}{2}\vert}{\sqrt{|(t-t^k_+)(t-t^k_-)|}}$ and $C_{k}(t)\sim_{0, (Ck)^{-1}} C_k^*$ with $0<C^*_k,(C_{k}^*)^{-1}\leq C(\log k)^C.$
\end{lemma}

\begin{remark}\label{remark-61} By the help of \textbf{B} of Lemma \ref{lemmathetat1} and all that arguments in the proof of Lemma \ref{zhongyao}, \eqref{pujideguji2} and \eqref{pujideguji3} also hold true if we replace $\theta_{n_1}$ with $g_j,$ $s(k_{\tilde{i}}(t))+2\leq j< j_{\tilde{i}}(t)$ where $k_{\tilde{i}}=k.$ In particular, let $\tilde{i}=i$ and choosing $j=j_i(t)-1$ and $\zeta_{\mathcal{N}_j}$ satisfying $\zeta_{\mathcal{N}_j}\cdot \sqrt{\frac{\vert\tan g_{\bar{j}}(z^j_1(t),t)\vert}{\frac{\partial^2 \tan g_{\bar{j}}}{\partial x^2}(z^j_1(t),t)}}=|~I_{j_i}|$,
from \eqref{pujideguji2}, \eqref{pujideguji3} and \eqref{pugapdeguji}, we obtain that

\begin{equation}\label{lip-11} \vert\int_{~I_{j_i(t)-1}-~I_{j_i(t)}}\cot g_{{j}_i(t)-1}\cdot \partial_t g_{{j}_i(t)-1} dx\vert\leq C\frac{\vert t-\frac{t^{k_i}_-+t^{k_i}_+}{2}\vert}{\sqrt{|(t-t^{k_i}_+)(t-t^{k_i}_-)|}}+C\lambda^{\left(\log (\mathcal{N}_{j_i(t)})\right)^{C}}.
\end{equation}
\end{remark}

By Lemma \ref{lipuseful} and a similar proof as above, we have the following result.

  \begin{lemma}\label{lipschitz111} For $t\in [\inf v-\frac{2}{\lambda},\sup v+\frac{2}{\lambda}],$ $j_i(t)\leq j< s(k_{i+1}(t))+2$ and $r_{j}\geq Y_1,Y_2\geq q^C_{N+j-2}$, it holds that \begin{equation}\label{lip-12}\vert\int_{~I_j-~I_{j+1}} \cot g_{{j}}\partial_t g_{{j}} dx\vert,\ \vert\int_{~I_j}\frac{\partial_{\theta_{Y_1}^{Y_1+Y_2}} \|A_{Y_1+Y_2}\|}{\|A_{Y_1+Y_2}\|}\partial_t \theta_{Y_1}^{Y_1+Y_2} dx\vert \leq C\frac{\vert t-\frac{t^{k_i}_-+t^{k_i}_+}{2}\vert}{\sqrt{|(t-t^{k_i}_+)(t-t^{k_i}_-)|}}+C\lambda^{(\log k_i)^C}.\end{equation}
  \end{lemma}

  \begin{proof} For $~j_i(t)\leq j< s(k_{i+1}(t))+2,$ by \textbf{B} of Lemma \ref{lemmathetat1}, we have \begin{equation}\label{gjt}\partial_t g_{j}\vert_{I_{j,l}}=h^{j,l}(x,t)\cdot \partial_x g_{j}+w^{j,l}(x,t),l=1,2\end{equation} with \begin{equation}\label{gjxt}h^{j,l}(x,t)\sim_{0,\zeta_{\mathcal{N}_j}} h^{j,l}(c_{j,l}(t),t)\ {\rm\ on\ } {I_{j,l}};~w^{j,l}(x,t)\sim_{0,\zeta_{\mathcal{N}_j}} w^{j,l}(c_{j,l}(t),t)\ {\rm \ on\ }{I_{j,l}};~|h^{j,l}(c_{j,l}(t),t)|,|w^{j,l}(c_{j,l}(t),t)|\leq C(\log k)^C\end{equation}

  and

  \begin{equation}\label{12190}|\frac{\partial^2 g_{j}}{\partial x^2}|+|\frac{\partial^2 g_{j}}{\partial x\partial t}|+|\frac{\partial^2 g_{j}}{\partial t^2}|\leq C(\log k_i)^C l_{k_i},\ (x,t)\in I_{j,l}\times B(t_0,\lambda^{-q_{N+j-1}}).\end{equation}

 Note that \eqref{12190} implies that \begin{equation}\label{par2gj1}\vert g_{j}(x,t)-\partial_x g_{j}(c_{j,l}(t),t)(x-c_{j,l}(t))\vert \leq C(\log k_i)^Cl_{k_i}\cdot |x-c_{j,l}(t)|^2;
~\partial_x g_{j}(x,t)\sim_{0,\zeta_{\mathcal{N}_j}}\partial_x g_{j}(c_{j,l}(t),t), {\rm\ on\ } {I_{j,l}}~\end{equation}$l=1,2.$

  On the other hand, it follows from \eqref{lem49-zh} that \begin{equation}\label{12191}\frac{l_{k_i}}{|\partial_x g_{j}(c_{j,m}(t),t)|}\leq C\lambda^{-(\log k_i)^{C}} |~I_{j}|^{-1},~m=1,2,~j_i(t)\leq j< s(k_{i+1}(t))+2.\end{equation}

  Note that for any fixed $x\in ~I_j$, it holds from result 4 of Lemma \ref{niceset1} that $A_{Y_1}$ is $(\lambda^{-|\log q_{N+j-1}|^{C}},+)-\text{nice}$ on $I_j$ and $A_{Y_2}$ is $(\lambda^{-|\log q_{N+j-1}|^{C}},-)-\text{nice}$ on $I_j.$ Hence it follows from \eqref{result1} that $$\vert(s(A_{Y_1})-u(A_{Y_2}))-g_{j}\vert_{C^2(~I_j)}\leq C\left(\min\{\|A_{Y_1}\|,\|A_{Y_2}\|\}\right)^{-1}\leq C\lambda^{-q^c_{N+j-2}}\ll |~I_j|\leq \zeta_{\mathcal{N}_j},|~I_{j_i(t)}|.$$ Therefore \eqref{gjt}-\eqref{12191} also hold true if we replace $g_{j}$ with $\theta_{Y_1}^{Y_1+Y_2}.$

  Combining Lemma \ref{lemma8}, \eqref{htheta} and Lemma \ref{lipuseful}, we immediately get

  $$\begin{array}{ll}&\vert\int_{~I_j-~I_{j+1}}\cot g_{j}\partial_t g_{j} dx\vert,\vert\int_{~I_j}\frac{\partial_{\theta_{Y_1}^{Y_1+Y_2}} \|A_{Y_1+Y_2}\|}{\|A_{Y_1+Y_2}\|}\partial_t \theta_{Y_1}^{Y_1+Y_2} dx\vert \leq \sum\limits_{m=1}^2C(\log {k_i})^C\frac{l_{k_i}|~I_{j}|}{(\partial_x g_{j}(c_{j,m}(t),t))^2}+C(\log k_i)^C\\
  &\leq \sum\limits_{m=1}^2C(\lambda^{-(\log k_i)^C})\frac{1}{\vert\partial_x g_{j}(c_{j,m}(t),t)\vert}+C(\log k_i)^C\qquad (by ~\eqref{gjt}-~\eqref{12191}).\end{array}$$ Finally, we finish the proof by \eqref{pxgj}.
 \hfill\qed\end{proof}

 Note that \eqref{lip-11} and \eqref{lip-12} imply that $\sum\limits_{j=j_i(t)-1}^{l}\vert\int_{~I_j-~I_{j+1}} \cot g_{j}\partial_t g_{j} dx\vert \leq C\frac{\vert t-\frac{t^{k_i}_-+t^{k_i}_+}{2}\vert}{\sqrt{|(t-t^{k_i}_+)(t-t^{k_i}_-)|}}+C\lambda^{(\log k_i)^C}$ for any $j_i(t)\leq l<s(k_{i+1}(t))+2.$
 By a similar argument to Remark \ref{zhongk}, we indeed get
  \begin{lemma}\label{remark67}  For $t\in [\inf v-\frac{2}{\lambda},\sup v+\frac{2}{\lambda}],$ $j_i(t)\leq l< s(k_{i+1}(t))+2$ and $r_{j}\geq Y_1,Y_2\geq q^C_{N+j-2}$, it holds that $$\sum\limits_{j=j_i(t)-1}^{l}\left\vert\int_{~I_j-~I_{j+1}} \cot \theta^{{r}^-_j(x)+{r}^+_j(x)}_{{r}^-_j(x)}\partial_t \theta^{{r}^-_j(x)+{r}^+_j(x)}_{{r}^-_j(x)}dx\right\vert,\quad \left\vert\int_{~I_l}\frac{\partial_{\theta_{{r}_l^-(x)}^{Y_1+{r}_l^-(x)}} \|A_{Y_1+{r}_l^-(x)}\|}{\|A_{Y_1+{r}_l^-(x)}\|}\partial_t \theta_{{r}_l^-(x)}^{Y_1+{r}_l^-(x)} dx\right\vert,$$ $$ \left\vert\int_{~I_l}\frac{\partial_{\theta_{Y_2}^{Y_2+{r}_l^+(x)}} \|A_{Y_2+{r}_l^+(x)}\|}{\|A_{Y_2+{r}_l^+(x)}\|}\partial_t \theta_{Y_2}^{Y_2+{r}_l^+(x)} dx\right\vert\leq C\frac{\vert t-\frac{t^{k_i}_-+t^{k_i}_+}{2}\vert}{\sqrt{|(t-t^{k_i}_+)(t-t^{k_i}_-)|}}+C\lambda^{(\log k_i)^C}.$$
  \end{lemma}

By \eqref{zuieasy}, \eqref{7.2} and \eqref{universalestimate}, it holds that
\begin{equation}\label{III1} n\cdot III\leq 2n_0\cdot \left(C(n^{\frac{1}{50}})\cdot \frac{\vert t-\frac{t^k_-+t^k_+}{2}\vert}{\sqrt{|(t-t^k_+)(t-t^k_-)|}}+C\lambda^{(\log n)^{C}}\right)\leq n^{\frac{1}{10}}\cdot \frac{\vert t-\frac{t^k_-+t^k_+}{2}\vert}{\sqrt{|(t-t^k_+)(t-t^k_-)|}}+C\lambda^{(\log n)^{C}}.\end{equation}
By \eqref{7.2} and \eqref{pujideguji}, it holds that
\begin{equation}\label{II1}n\cdot II\leq (n-2n_0)\cdot (C\lambda^{-\frac{1}{20}k}\frac{\vert t-\frac{t^k_-+t^k_+}{2}\vert}{\sqrt{|(t-t^k_+)(t-t^k_-)|}}+C\lambda^{(\log n)^{C}})\leq n\cdot (M_k(t)+C\lambda^{(\log n)^{C}}),\end{equation}
where $M_k(t)$ is defined as in Lemma \ref{hexinl}.

\vskip 0.4cm
{\it The proof of (1) and (2) of Lemma \ref{lemma 6.1}}
\vskip 0.2cm
Now we are in a position  to prove (1) and (2) of Lemma \ref{lemma 6.1}.

By \eqref{qheshi},~\eqref{subi}, Lemma \ref{S*3}, \eqref{II1} and \eqref{III1}, we obtain that
$$\begin{array}{ll}&\left|\frac{1}{n}\int_T\frac{\partial_t\|A_{n}(x,t)\|}{\|A_{n}(x,t)\|}dx\right|\leq I+II+III\\&\leq (1+o(n^{-\frac{1}{5}}))\cdot |M_k(t)|+C\lambda^{(\log n)^{C}}~(note~\bar{C}_k\geq (\log k)^{-C}\gg n^{-1})\\&=\tilde{C}_k(t)\cdot\frac{2|t-\frac{t^k_-+t^k_+}{2}|}{\sqrt{|(t-t^k_-)(t^k_+-t)|}}+C\lambda^{(\log n)^{C}},\end{array}$$
 where $\tilde{C}_k$ satisfies  $\tilde{C}_k\sim_{0, (Ck)^{-1}} C_k^*$ with $0<C^*_k,(C_{k}^*)^{-1}\leq C(\log k)^C,$ which is indeed \eqref{6.1.1} of Lemma \ref{lemma 6.1}.
\eqref{6.1.2} of Lemma \ref{lemma 6.1} can be similarly obtained.


\section{The proof of (3),~(4) of Lemma \ref{lemma 6.1}}\label{section5.3}

\subsection{Some lemmas}
We denote
$$
\frac{\vert t_0-\frac{t^{k_i}_-+t^{k_i}_+}{2}\vert}{\sqrt{|(t_0-t^{k_i}_+)(t_0-t^{k_i}_-)|}}+\lambda^{\left(\log (\mathcal{N}_{j_i(t_0)})\right)^{C}}:=\mathcal{G}_{j_i(t_0)}.$$
Recall that $$~I_i(t):=B(c_{i,1}(t),\lambda^{-\hat{\epsilon}^{-1}q^{\hat{\epsilon}}_{N+i-1}})\bigcup B(c_{i,2}(t),\lambda^{-\hat{\epsilon}^{-1}q^{\hat{\epsilon}}_{N+i-1}}):=~I_{i,1}\bigcup ~I_{i,2}.$$

Our main target is to prove that
\begin{equation}\label{tarobtain-lip}\left|\int_{x\in \R/\Z}\frac{d L_{\mathcal{N}_{l}}(t_0)}{d t} dx\right|\leq C\cdot \mathcal{G}_{j_i(t_0)}\end{equation} for any $j_i(t_0)-1\leq l< s(k_{i+1}(t_0))+2$ (the estimate on $\int_{x\in \R/\Z}\frac{d L_{2\mathcal{N}_{l}}(t_0)}{d t} dx$ can be similarly obtained).

 For this purpose, we need the following  lemmas.
\begin{lemma}\label{dAdlambda}

Consider
$A_n(x,t)=A_{n-l}(x,t)A_l(x).$ We define $\|A_{n-l}(x,t)\|:=\lambda_1(x,t),~\|A_l(x,t)\|:=\lambda_2(x,t)$~and $~\theta_l(x,t):=\frac{\pi}{2}-s(A_{n-l}(x,t))+u(A_l(x,t)).$
Assume that there exists $\mu_i\gg 1,~i=1,2$ such that
$$\mu_i^{\frac{100}{99}}\geq \lambda_i\geq \mu_i^{\frac{99}{100}},~i=1,2.$$
Then the following several results hold true.
\begin{enumerate}
\item It holds that
\begin{equation}\label{tj0}\frac{1}{\|A_n\|}\left|\frac{d\|A_n\|}{d\lambda_i}\right|<C\lambda_i^{-1},\quad i=1,\ 2.\end{equation}

\item If \begin{equation}\label{tj1}\lambda_2^2\leq \lambda_1\ (respetively\ \lambda_1^2\leq\lambda_2),\end{equation}

then we have
\begin{equation}\label{tj3}\frac{1}{\|A_n\|}\frac{d\|A_n\|}{d\lambda_1}
\sim_{0,\mu_1^{-1}}\lambda_1^{-1}\quad  (respetively\ \frac{1}{\|A_n\|}\frac{d\|A_n\|}{d\lambda_2}
\sim_{0,\mu_2^{-1}}\lambda_2^{-1}).\end{equation}

\item If
\begin{equation}\label{tj2}|\theta_l|\geq \lambda_m^{-\eta}, ~where~ \lambda_m=\min\{\lambda_1,\lambda_2\}~ and ~\eta\ll 1,\end{equation}

then we have
\begin{equation}\label{tj3'}\frac{1}{\|A_n\|}\frac{d\|A_n\|}{d\lambda_1}
\sim_{0,\mu_1^{-1}}\lambda_1^{-1}\quad  and~\frac{1}{\|A_n\|}\frac{d\|A_n\|}{d\lambda_2}
\sim_{0,\mu_2^{-1}}\lambda_2^{-1}.\end{equation}
\end{enumerate}

\end{lemma}
\begin{proof}
Without loss of generality, we only consider the estimates on $\frac{1}{\|A_n\|}\frac{d\|A_n\|}{d\lambda_1}$. By \eqref{lemma8,1} and a direct calculation, we have

\begin{align*}
&\frac{1}{\|A_n\|}\frac{\p\|A_n\|}{\p\lambda_1}
= & \frac{sgn(\cot \theta_l )\lambda_1^{-1}\left((1-\frac{1}{\lambda_1^4\lambda_2^4})\cot\theta_l+(\frac{1}{\lambda_2^4}-\frac{1}{\lambda_1^4})\tan\theta_l\right)}{\sqrt{(1-\frac{1}{\lambda_1^4\lambda_2^4})^2
\cot^2\theta_l+(\frac{1}{\lambda_1^4}-\frac{1}{\lambda_2^4})^2\tan^2\theta_l+2(1+\frac{1}{\lambda_1^4\lambda_2^4})
(\frac{1}{\lambda_1^4}+\frac{1}{\lambda_2^4})-\frac{8}{\lambda_1^4\lambda_2^4}}}
=\lambda_1^{-1}\frac{a_1\cdot sgn(\cot\theta_l)}{\sqrt{a_1^2+a_2^2}},
\end{align*}
where $a_1=(1-\frac{1}{\lambda_1^4\lambda_2^4})\cot\theta_l+(\frac{1}{\lambda_2^4}-\frac{1}{\lambda_1^4})\tan\theta_l$
and $a_2=\frac{2}{\lambda_1^2}(1-\frac{1}{\lambda_2^4})$.

Thus it is easy to obtain $\left|\frac{1}{\|A_n\|}\frac{d\|A_n\|}{d\lambda_1}\right|<C\lambda_1^{-1}.$

Assume the condition (\ref{tj1}) holds, then $\frac{1}{\lambda_1^4}\ll\frac{1}{\lambda_2^4}$, which implies
two terms of $a_1$ share the same sign. Thus
$$
a_1=\cot\theta_l+\frac{1}{\lambda_2^4}\tan\theta_l+o(a_1)>\frac{1}{2\lambda_2^2}.
$$
Since $a_2=\frac{2}{\lambda_1^2}+o(a_2)$, with the help of (\ref{tj1}), we have $a_2\ll|a_1|,$ which implies (\ref{tj3}) holds true.
On the other hand, assume the condition (\ref{tj2}) holds, then $$\left|(1-\frac{1}{\lambda_1^4\lambda_2^4})\cot\theta_l\right|\gg
\max\{\left|(\frac{1}{\lambda_2^4}-\frac{1}{\lambda_1^4})\tan\theta_l\right|,\ a_2\}.$$ Thus (\ref{tj3'}) holds true.
\hfill\qed\end{proof}

Next we fix $l$ with $j_{i^*}(t_0)-1\leq l< s(k_{i^*+1}(t_0))+2.$
Now we give some definitions.

\begin{definition} Let $s\in \Z_+,~x_0\in~\R/\Z$ and $t_0\in [\inf v-\frac{2}{\lambda},\sup v+\frac{2}{\lambda}].$

  If step $s$ is of Type ${I_s}$ ,   we denote $\{x_0+l\alpha \vert x_0+l\alpha\in ~I_{s},~1\leq l\leq n \}$ by $\{x_{i^j_{s}}\}_{j=1}^{h(s,n,x_0,t_0)}$ with some $h(s,n,x_0,t_0)\in \Z_+.$

  If step $s$ is of Type~${II^{k^*}_s}$ satisfying $~I_{s,1}+k^*\alpha\bigcap ~I_{s,2}\neq \emptyset$ for some $k^*\in \Z_+$, then we denote $\{x_0+l\alpha~\vert x_0+l\alpha\in ~I_{s,1},~1\leq l\leq n \}$ by $\{x_{i^j_{s}}\}_{j=1}^{h(s,n,x_0,t_0)}$ with some $h(s,n,x_0,t_0)\in \Z_+.$

   If step $s$ is of Type~${II^{k^*}_s}$ satisfying $~I_{s,1}-k^*\alpha\bigcap ~I_{s,2}\neq \emptyset$ for some $k^*\in \Z_+$, then we denote $\{x_0+l\alpha \vert x_0+l\alpha\in ~I_{s,2},~1\leq l\leq n\}$ by $\{x_{i^j_{s}}\}_{j=1}^{h(s,n,x_0,t_0)}$ with some $h(s,n,x_0,t_0)\in \Z_+.$
\end{definition}

For simplicity, we sometimes omit the dependence of $h$ on $n,t,x$. Moreover without loss of generality, in the following proof, we suppose
\begin{equation}\label{jiashe1*}i_s^1<i_{s+1}^1;~i_s^{h(s,n,x,t)}>i_{s+1}^{h(s,n,x,t)},\  n>s\in \Z_+,~x\in \R/\Z,~t\in [\inf v-\frac{2}{\lambda},\sup v+\frac{2}{\lambda}].\end{equation}

Define $\mathcal{\Theta}_{i}^{j}:=\frac{\partial_{\theta^j_{i}} \|A_{j}\|}{\|A_{j}\|} \cdot \partial_t(\theta^j_{i})\cdot sgn(\theta^{j}_{i}),\ i, j\in \Z_+.$ Let $\hat{\Theta}_{s,j}=\cot \theta^{i^{j+1}_{s}-i^{j-1}_{s}}_{i^{j}_{s}-i^{j-1}_{s}}(x_{i^j_{s}})\cdot\partial_t \theta^{i^{j+1}_{s}-i^{j-1}_{s}}_{i^{j}_{s}-i^{j-1}_{s}}(x_{i^j_{s}}),$ $ {\Theta}_{i_s^1}=\mathcal{\Theta}_{i^1_{s}}^{i^2_{s}}.$
\begin{lemma}\label{jib''} For any fixed $m\in \Z_+,~x\in~\R/\Z$ and $t\in [\inf v-\frac{2}{\lambda},\sup v+\frac{2}{\lambda}],$ let $\{x_{i^j_{s}}\}_{j=1}^{h(s)}$ be defined as above. Assume that there exists some $i^*\in \Z_+$ such that $k_{i^*}$ satisfies $~I_{u,1}+k_{i^*}\alpha \bigcap ~I_{u,2}\neq \emptyset$ with $u=j_{i^*}(t)-1$. Suppose $s(k_{i^*+1}(t))+2>m\geq u.$ Then it holds that
$$ \begin{array}{ll}&\left\vert \frac{\partial_t\|A_{{r}^+_m(x,t)}\|}{\|A_{{r}^+_m(x,t)}\|}-\sum\limits_{s=u}^{m-1}(\sum\limits_{j=2}^ {h(s)}\hat\Theta_{s,j}+\Theta_{i_s^1})\right\vert\leq C\lambda^{-cq_{N+u-1}}+C\sum\limits_{j=0}^ {h(s)}\frac{\vert\partial_t\|A_{i_{u}^{j+1}-i_{u}^{j}}\|\vert}{\|A_{i_{u}^{j+1}-i_{u}^{j}}\|}.\end{array}
$$
\end{lemma}
\begin{proof} Note  $s(k_{i^*+1}(t_0))+2>m\geq u$ means there is no  resonance occur between the $u$-th step and $m$-th step.
We consider $A_{{r}^+_m}=A_{i^{h(m-1)+1}_{m-1}-i_{m-1}^{h(m-1)}}\cdots A_{i_{m-1}^{2}-i_{m-1}^{1}}A_{i_{m-1}^{1}},$ where $i^{h(m-1)+1}_{m-1}:={r}^+_m.$~The assumption \eqref{jiashe1*} implies that $h(p)\geq 2,~u\leq p\leq m-1.$
  We have to separately consider the following two cases.

We set $i_{m-1}^{0}=0.$
Next, we will gradually decompose the large matrix $A_{r^+_m}.$ To reduce repetitive statements, we will only expand on two steps.

\textbf{Step 1:}
If $i_{m-1}^1(x)\geq q^C_{N+m-2},$ then it holds from Lemma \ref{niceset1} that for $j\geq 0,$
$$A_{i_{m-1}^{j+1}-i_{m-1}^{j}}(x_{i_{m-1}^j})\ {\rm is}\ (\lambda^{-|\log q_{N+m-1}|^{C}},+)-\text{nice~on}~I_m.$$ Hence \eqref{result1} and the fact $|I_m|=2\lambda^{q^{\hat{\epsilon}}_{N+m-1}}$ with $\hat{\epsilon}\ll c \leq 1$ imply $$|\Theta_{i_{m-1}^1}(x_{i_{m-1}^1})|>|I_m|^C-\lambda^{-q_{N+m-2}^c}>|I_m|^{2C}$$ and $$\|A_{i_{m-1}^{j+1}-i_{m-1}^{j}}\|\geq c\lambda^{(1-c){r}_{m-1}}\geq c\lambda^{(1-c)q^2_{N+m-2}}\gg c\lambda^{Cq^{\hat{\epsilon}}_{N+m-2}}\geq |I_m|^{-C}.$$ Then by the help of \eqref{tj3'} and repeatedly using Lemma \ref{lemma8}, we obtain

\begin{equation}
\label{jib}\begin{array}{ll}&\left\vert \frac{\partial_t\|A_{{r}^+_m(x_0,t_0)}\|}{\|A_{{r}^+_m(x_0,t_0)}\|} -\sum\limits_{j=0}^{h(m-1)}\frac{\partial_t\|A_{i_{m-1}^{j+1}-i_{m-1}^{j}}\|}{\|A_{i_{m-1}^{j+1}-i_{m-1}^{j}}\|}-\sum\limits_{j=2}^{h(m-1)}\left(\hat\Theta_{m-1,j} +\Theta_{i^1_{m-1}} \right)\right\vert\\&\leq C(h(m-1))\lambda^{-\frac{1}{4}{r}_{m-1}}\end{array}.\end{equation}
If $i_{m-1}^1(x)< q^C_{N+m-2},$   then we consider $A_{i^{h(m-1)+1}_{m-1}-i_{m-1}^{h(m-1)}}\cdots A_{i_{m-1}^2}.$ By diophantine condition, we have $i_{m-1}^2\geq |I_{m-1}|^{-C^*}\gg q^C_{N+m-2}>i_{m-1}^1.$ Then $$c\lambda^{\frac{9}{10}i_{m-1}^1}\leq \|A_{i_{m-1}^1}\|\leq C\lambda^{i_{m-1}^1}\ll C\lambda^{q^C_{N+m-2}}\ll\|A_{i_{m-1}^2-i_{m-1}^1}\|.$$ Therefore, we obtain that $\|A_{i_{m-1}^2}\|\geq \lambda^{\frac{1}{2}{r}_{m-1}}\gg \lambda^{cq_{N+m-2}}.$ Hence by (ii)-(2a) of Lemma \ref{lemma8*}, $|\theta^{i_s^3}_{i_s^2}(x_{i_s^2})|$ and $|\theta^{i_s^3}_{i_s^2-i^s_1}(x_{i_s^2})|$ has the same lower bound $|I_m|^{C}$. Then, again by the help of \eqref{tj3'},~Lemma \ref{lemma8} and induction, we obtain
\begin{equation}\begin{array}{ll}&
\label{jib11}\left\vert \frac{\partial_t\|A_{{r}^+_m(x_0,t_0)}\|}{\|A_{{r}^+_m(x_0,t_0)}\|} -\sum\limits_{j=2}^{h(m-1)}\frac{\partial_t\|A_{i_{m-1}^{j+1}-i_{m-1}^{j}}\|}{\|A_{i_{m-1}^{j+1}-i_{m-1}^{j}}\|}-\frac{\partial_t\|A_{i_{m-1}^{2}}\|}{\|A_{i_{m-1}^{2}}\|}-\sum\limits_{j=2}^{h(m-1)}\hat\Theta_{m-1,j} \right\vert\\&\leq C(h(m-1))\lambda^{-\frac{1}{4}{r}_{m-1}}\leq C\lambda^{-cq_{N+m-2}}.\end{array}\end{equation}

\textbf{Step 2:}
Now we repeat the above process on each $\frac{\partial_t\|A_{i_{m-1}^{j+1}-i_{m-1}^{j}}\|}{\|A_{i_{m-1}^{j+1}-i_{m-1}^{j}}\|},~j\geq 0$ in \eqref{jib} and $\frac{\partial_t\|A_{i_{m-1}^{j+1}-i_{m-1}^{j}}\|}{\|A_{i_{m-1}^{j+1}-i_{m-1}^{j}}\|},~j\geq 1$  in \eqref{jib11}. For $\frac{\partial_t\|A_{i_{m-1}^{2}}\|}{\|A_{i_{m-1}^{2}}\|}$ in \eqref{jib11}. Note that if $i_{m-2}^1\geq q^C_{N+m-3}$ then we consider $$A_{i_{m-1}^1}=A_{i_{m-1}^1-i_{m-2}^{j^*}}\cdots A_{i_{m-2}^2-i_{m-2}^1}A_{i_{m-2}^1}$$ for some $j^*\geq 1$ (from \eqref{jiashe1*}). Similarly to the case $i_{m-1}^1\geq q^C_{N+m-2}$ in Step 1. We can decompose $A_{i_{m-1}^1}$ like $\frac{\partial_t\|A_{i_{m-1}^{j+1}-i_{m-1}^{j}}\|}{\|A_{i_{m-1}^{j+1}-i_{m-1}^{j}}\|}.$ For the case $i_{m-2}^1< q^C_{N+m-3},$ similarly to the second case $i_{m-1}^1< q^C_{N+m-2}$ in Step 1. We keep $A_{i_{m-2}^2}$ and will consider how to decompose it in the next step.

Repeat the above process we get either



\begin{equation}\label{im0}\begin{array}{ll}\left\vert \frac{\partial_t\|A_{{r}^+_m(x_0,t_0)}\|}{\|A_{{r}^+_m(x_0,t_0)}\|}-\sum\limits_{s=u}^{m-1}(\sum\limits_{j=2}^ {h(s)}\hat\Theta_{s,j}+\Theta_{i_s^1})-\sum\limits_{j\geq 0}\frac{\partial_t\|A_{i_{u}^{j+1}-i_{u}^{j}}\|}{\|A_{i_{u}^{j+1}-i_{u}^{j}}\|}\right\vert &\leq C\lambda^{-cq_{N+u-1}}\end{array}\end{equation} for case $i_u^1\geq q^C_{N+u-1}$
or
\begin{equation}\label{im0*}\begin{array}{ll}\left\vert \frac{\partial_t\|A_{{r}^+_m(x_0,t_0)}\|}{\|A_{{r}^+_m(x_0,t_0)}\|}-\sum\limits_{s=u}^{m-1}(\sum\limits_{j=2}^ {h(s)}\hat\Theta_{s,j})-\sum\limits_{s=u+1}^{m-1}\Theta_{i_s^1}-\sum\limits_{j\geq 2}\frac{\partial_t\|A_{i_{u}^{j+1}-i_{u}^{j}}\|}{\|A_{i_{u}^{j+1}-i_{u}^{j}}\|}-\frac{\partial_t\|A_{i^2_{u}}\|}{\|A_{i^2_{u}}\|}\right\vert &\leq C\lambda^{-cq_{N+u-1}}\end{array}\end{equation} for the case $i_u^1<q^C_{N+u-1}.$
Since $\|A_{i_{u}^1}\|\leq \lambda^{q^C_{N+u-1}}\ll \lambda^{c\lambda^{q_{N+u-1}^{\hat{\epsilon}}}}\leq \|A_{i_{u}^2-i_{u}^1}\|,$ by the help of Lemma \ref{dAdlambda} we obtain that
\begin{equation}\label{im1}\begin{array}{ll}&\left\vert \frac{\partial_t\|A_{i_{u}^{2}}\|}{\|A_{i_{u}^{2}}\|} -\Theta_{i_{u}^1}\cdot\chi_{\mathcal{S}_1}(u)-\frac{\partial_t\|A_{i_{u}^{1}}\|}{\|A_{i_{u}^{1}}\|}-\frac{\partial_t\|A_{i_{u}^{2}-i_{u}^{1}}\|}{\|A_{i_{u}^{2}-i_{u}^{1}}\|}\right\vert\\&\leq C\lambda^{-cq_{N+u-1}}+C\left\vert\frac{\partial_t\|A_{i_{u}^{1}}\|}{\|A_{i_{u}^{1}}\|}\right\vert
\\&=C\lambda^{-cq_{N+u-1}}+C\left\vert\frac{\partial_t\|A_{i_{u}^{1}-i_{u}^{0}}\|}{\|A_{i_{u}^{1}-i_{u}^{0}}\|}\right\vert .\end{array}\hfill\qed\end{equation}\end{proof}

By the help of \eqref{im0}, \eqref{im0*} and \eqref{im1}, we completes the proof.

\begin{remark}\label{remark--} Lemma \ref{jib''} also holds true for ${r}_m^-$ by the same proof as above if the orbit $\{T^jx,\ j\ge 0\}$ is replaced by $\{T^jx,\ j\le 0\}$.
\end{remark}
\begin{remark} \label{remark--'} Consider the case $n< {r}^+_{m}$ satisfying $h(m-1,n,x_0,t_0)\geq 2$ and the case $~i_m^2\geq n\geq {r}^+_{m}$ satisfying $h(m-1,n,x_0,t_0)\geq 2$ and $\min\{{r}^+_{m},n-{r}^+_{m}\}\leq q^C_{N+m-1}$. By the proof of Lemma \ref{jib''} and Remark \ref{remark--}, it is not difficult to see that for these two cases, by considering $A_n=A_{n-i_{m-1}^{h}}\cdots A_{i_{m-1}^3-i_{m-1}^2}A_{i_{m-1}^2}$ or $ A_n= A_{n-i_{m-1}^{h-1}}\cdots A_{i_{m-1}^2-i_{m-1}^1}A_{i_{m-1}^1}$, it holds that
$$ \begin{array}{ll}&\left\vert \frac{\partial_t\|A_{n}\|}{\|A_{n}\|}-\sum\limits_{s= u}^{m-1}\sum\limits_{j=2}^{h(s)-1}\hat\Theta_{s,j}-\sum\limits_{s=u}^{m-1}(\Theta_{i_s^1}+{\Theta}_{i_s^h})\right\vert \leq C\lambda^{-cq_{N+u-1}}+C\sum\limits_{j\geq 0}\frac{\vert\partial_t\|A_{i_{u}^{j+1}-i_{u}^{j}}\|\vert}{\|A_{i_{u}^{j+1}-i_{u}^{j}}\|},\end{array}
$$
where ${\Theta}_{i_s^h}=\mathcal{\Theta}_{i^{h-1}_{s}}^{i^{h}_{s}+i^{h-1}_{s}}$ and $\mathcal{S}_h=\{s|i_s^h> \mathcal{N}_l- q^C_{N+s-1}\}.$
\end{remark}

 Let $\Theta_{j}^{\mathcal{N}_l}=\mathcal{\Theta}_{j}^{\mathcal{N}_l}$ for any $q^C_{N+l-2}\leq j\leq \mathcal{N}_l-q^C_{N+l-2}.$ By the fact $j_{i^*}(t_0)-1\leq l< s(k_{i^*+1}(t_0))+2,$
there exists at most one $j$ with $k_{i^*}< j\leq \mathcal{N}_l\leq \lambda^{q^{\hat{\epsilon}}_{N+l-1}}\leq \lambda^{q^{\hat{\epsilon}}_{N+s(k_{i^*+1}(t_0))}}<r_{s(k_{i^*+1})+2}\leq k_{i^*+1}$ such that $x+j\alpha\in ~I_{l}$ (Otherwise, there would exist some \( \tilde{k} \) between \( k_{i^*} \) and \( k_{i^*+1} \) such that some step \( \tilde{j} \) belongs to \( III_{\tilde{j}}^{\tilde{k}} \), which contradicts the definition of \( k_{i^*} \) and \( k_{i^*+1} \).).
Let $\mathcal{S}_{l}=\{q^C_{N+l-2}\le s\le \mathcal{N}_l-q^C_{N+l-2}|s=i_{l}^1\}.$
We will show that
\begin{lemma}\label{lip-m} For $x_0\in \R/\Z$ and $t_0\in [\inf v-\frac{2}{\lambda},\sup v+\frac{2}{\lambda}],$ it holds that
\begin{equation}\label{r-r-r-r} \begin{array}{ll}&\left\vert \frac{\partial_t\|A_{\mathcal{N}_l}\|}{\|A_{\mathcal{N}_l}\|}-\sum\limits_{s= j_i(t)-1}^{l-1}(\sum\limits_{j=2}^{h(s)-1}\hat\Theta_{s,j}+
\Theta_{i_s^1}+{\Theta}_{i_s^h})-\sum\limits_{s\in \mathcal{S}_{l}}\Theta_{s}^{\mathcal{N}_l}
\right\vert\\&\leq C\lambda^{-cq_{N+j_i(t)-2}}+C\sum\limits_{j\geq 0}\left\vert\frac{\partial_t\|A_{i_{j_i(t)-1}^{j+1}-i_{j_i(t)-1}^{j}}
\|}{\|A_{i_{j_i(t)-1}^{j+1}-i_{j_i(t)-1}^{j}}\|}\right\vert+\chi_{x_s\in \bar{I}_{l}}\cdot\left(\left\vert\frac{\partial_t\|A_{i_{l}^1}\|}{\|A_{i_{l}^1}\|}\right\vert+
    \left\vert\frac{\partial_t\|A_{\mathcal{N}_l-i_{l}^1}\|}{\|A_{\mathcal{N}_l-i_{l}^1}\|}\right\vert\right).\end{array}
\end{equation}
\end{lemma}
\begin{proof}
We have to consider the following 3 cases:

\begin{enumerate}

\item If there is no $1\leq s \leq \mathcal{N}_l$ such that $x_s\in ~I_{l}$ , then Remark \ref{remark--'} implies \eqref{r-r-r-r}.
\item If there do exist one $s<q^C_{N+l-2}$ or  $\mathcal{N}_l-s<q^C_{N+l-2}$ such that $x_s\in ~I_{l,1},$ then Remark \ref{remark--'} also implies \eqref{r-r-r-r}.
\item If there do exist one $q_{N+l-2}^C\leq s \leq \mathcal{N}_l-q_{N+l-2}^C$ such that $x_s\in ~I_{l,1},$ then by the fact $ \mathcal{N}_l\ll r_{{l}}$ we know that $q_{N+l-2}^C\leq i_{l}^1\leq \mathcal{N}_l-q_{N+l-2}^C$ and $i_l^2$ is absent. Then we consider $A_{\mathcal{N}_l}=A_{i_{l}^1}A_{\mathcal{N}_l-i_{l}^1}.$ By the help of Lemma \ref{lemma8}, \eqref{tj0} and the facts $A_{\mathcal{N}_l-i_{l}^1}(x_{i_{l}^1})$ is $(\lambda^{-|\log q_{N+l-1}|^{C}},+)-\text{nice}$ on $I_{l}$ and $A_{i_{l}^1}(x_{i_{l}^1})$ is $(\lambda^{-|\log q_{N+l-1}|^{C}},-)-\text{nice}$ on $I_{l
    }$( by Lemma \ref{niceset1}),
    it holds that for $x\notin \bar{I}_{l}:=\lambda^{-\left\vert \log |I_l| \right\vert^{\hat{\epsilon}^{-1}}}\cdot I_{l},$
    $$|\Theta^{\mathcal{N}_l}_{s}(x)|>|\bar{I}_{l,1}|^{C}-\lambda^{-cq^C_{N+l-1}}\geq |\bar{I}_{l,1}|^{C}-\lambda^{-\left\vert \log |I_l| \right\vert^{C\hat{\epsilon}^{-1}}}\geq |\bar{I}_{l,1}|^{\frac{1}{2}C}\gg \left(\min\{\|A_{s}\|,\|A_{\mathcal{N}_l-s}\|\}\right)^{-c}.$$
    By (3) of Lemma \ref{dAdlambda},
    $$\left\vert\frac{\partial_t\|A_{\mathcal{N}_l}\|}{\|A_{\mathcal{N}_l}\|}-\frac{\partial_t\|A_{i_{l}^1}\|}{\|A_{i_{l}^1}\|}-
    \frac{\partial_t\|A_{\mathcal{N}_l-i_{l}^1}\|}{\|A_{\mathcal{N}_l-i_{l}^1}\|}-\Theta_{s}^{\mathcal{N}_l}
    \right\vert\leq C\lambda^{-cq_{N+l-1}}+\chi_{x_s\in \bar{I}_{l}}\cdot\left(\left\vert\frac{\partial_t\|A_{i_{l}^1}\|}{\|A_{i_{l}^1}\|}\right\vert+
    \left\vert\frac{\partial_t\|A_{\mathcal{N}_l-i_{l}^1}\|}{\|A_{\mathcal{N}_l-i_{l}^1}\|}\right\vert\right).$$


     Finally, by applying Lemma \ref{jib''} and Remark \ref{remark--} on $A_{\mathcal{N}_l-i_{l}^1}$ and $A_{i^1_{l}},$ we obtain \eqref{r-r-r-r} as desired.\hfill\qed
\end{enumerate}
\end{proof}

\subsection{The proof of \eqref{tarobtain-lip}}

Now we come back to our target \eqref{tarobtain-lip}, which will implies (3) of Lemma \ref{section5.3}. Note that from $i_s^1\leq q^C_{N+s-1}$, it follows that $x+M\alpha\in ~I_s$ with some $M\leq q^C_{N+s-1}.$ By this fact and  \eqref{r-r-r-r} of Lemma \ref{lip-m}, we have
\begin{equation}\label{1-1-1-1}  \begin{array}{ll}&\left\vert \int_{\R/\Z} \frac{\partial_t\|A_{\mathcal{N}_l}\|}{\|A_{\mathcal{N}_l}\|} dx\right\vert  \leq \left\vert \int_{\R/\Z}(\sum\limits_{s= j_i(t)-1}^{l-1}\sum\limits_{j=2}^{h(s)-1}\hat\Theta_{s,j}+\sum\limits_{s\in [j_i(t)-1, l-1]-\mathcal{S}_1} {\Theta}_{s,1}+\sum\limits_{s\in [j_i(t)-1, l-1]-\mathcal{S}_h}{\Theta}_{s,h})dx\right\vert \\ &+\left\vert\int_{\R/\Z}(\sum\limits_{s\in [j_i(t)-1, l-1]\bigcap\mathcal{S}_1}\Theta_{s,1}+\sum\limits_{s\in \mathcal{S}_h\bigcap[j_i(t)-1, l-1]}\Theta_{s,h})dx\right\vert+\left\vert\int_{\R/\Z}\sum\limits_{s\in \mathcal{S}_{l}}\Theta_{s}^{\mathcal{N}_l}dx\right\vert\\&+C\int_{\R/\Z} \sum\limits_{j\geq 0}\frac{\left\vert\partial_t\|A_{i_{j_i(t)-1}^{j+1}-i_{j_i(t)-1}^{j}}\|\right\vert}{\|A_{i_{j_i(t)-1}^{j+1}-i_{j_i(t)-1}^{j}}\|} dx +\left(C\lambda^{-cq_{N+j_i(t)-2}}+\left(\int_{\bar{I}_{l}}\left\vert\frac{\partial_t\|A_{i_{l}^1}\|}{\|A_{i_{l}^1}\|}\right\vert dx+
    \int_{\bar{I}_{l}}\left\vert\frac{\partial_t\|A_{\mathcal{N}_l-i_{l}^1}\|}{\|A_{\mathcal{N}_l-i_{l}^1}\|}\right\vert dx\right)\right)\\& := P_1+P_2+P_3+P_4+P_5.\end{array}
\end{equation}

First by (6) of Lemma \ref{niceset1} we have $P_5\leq 1\leq C\cdot \mathcal{G}_{j_i(t_0)}.$

Next note that the definition implies that for $2\leq j\leq h-1.$ Thus we have $\hat{\Theta}_{s,j}=\theta^{{r}^+_s(x_{i_s^j})+{r}^-_s(x_{i_s^j})}_{{r}^-_s(x_{i_s^j})}.$  On the other hand, for any $X\in \{q^C_{N+s-1},q^C_{N+s-1}+1,\cdots,\mathcal{N}_l-q^C_{N+s-1}\}$ and $~j_i(t)-1\leq s\leq l-1$, there must exist some $x'\in \R/\Z$ and an unique $b$ with $1\leq b(X,x',s)\leq h(s)$ such that $i_s^{b(X,x',s)}(x')=X.$ If $\min\{X,\mathcal{N}_l-X\}<q^C_{N+l-2}$, let $m=m(X)$ satisfy that $q^C_{N+m-1}\leq \min\{X,\mathcal{N}_l-X\}< q^C_{N+m}$ and $R_{l}=+\infty.$ If $\min\{X,\mathcal{N}_l-X\}\geq q^C_{N+l-2}$, we define $m(X):=l-1$. Denote $S_{X,1}=\{s|i_s^1=X\}$
and $S_{X,b}=\{s|i_s^b=X\}$.
By Lemma \ref{remark67}, for any $j_i(t)-1\leq p\leq l-1$ it holds that $$\vert\sum\limits_{s= j_i(t)-1}^{p}\int_{\R/\Z}F_{X,s}(x)dx\vert=\vert\sum\limits_{j_i(t)-1\leq s \leq p-1}\int_{~I_s-~I_{s+1}}F_{X,s}(x) dx+\int_{~I_p}F_{X,p}(x) dx\vert\le C\cdot \mathcal{G}_{j_i(t_0)},$$

 where $F_{X,s}(x)=\hat\Theta_{s,b(X,x,s)}(x_{i_s^b})\cdot\chi_{\mathcal{S}_{X,b}}(s)+\Theta_{i_s^1}\cdot
 \chi_{\mathcal{S}_{X,1}-\mathcal{S}_1}(s)+\Theta_{i_s^h}\cdot \chi_{\mathcal{S}_{X,h}-\mathcal{S}_h}(s)$. Then the definition of $b(X,x,s)$ implies that $$\begin{array}{ll}P_1&
 =\vert\int_{x\in \R/\Z}\sum\limits_{s=j_i(t)-1}^{l-1}\sum\limits_{X=R_s}^{\mathcal{N}_l-R_s}F_{X,s}(x)dx\vert\leq \sum\limits_{X=R_{j_i(t)-1}}^{\mathcal{N}_l-R_{j_i(t)-1}}\vert\int_{x\in \R/\Z}\sum\limits_{s=j_i(t)-1}^{m(X)}F_{X,s}(x) dx\vert  \leq C\cdot \mathcal{G}_{j_i(t_0)}.\end{array}$$

By \eqref{12345121} of Remark \ref{remark-62} and the definitions of $\Theta_{s,1}$ and $\Theta_{s,h}$ for $j_i(t)-1\le s\le l-1$, we have $$\vert\int_{\R/\Z}(\sum\limits_{i_s^1=1}^{q^C_{N+s-1}}\Theta_{s,1}(x_{i^1_s})+\sum\limits_{i_s^{h}=\mathcal{N}_l-q^C_{N+s-1}}^{\mathcal{N}_l}\Theta_{s,h}(x_{i^h_s})) dx \vert \leq C\cdot \left( 2q^C_{N+s-1}\right)\cdot (q^C_{N+s-1}) \mathcal{G}_{j_i(t_0)}\leq C\cdot (\log \mathcal{N}_l)^C\mathcal{G}_{j_i(t_0)}.$$

Hence it holds that $P_2\leq C\cdot \sum\limits_{s=j_i(t)-1}^{l-1}q^C_{N+s-1}\cdot (\log \mathcal{N}_l)^C\cdot \mathcal{G}_{j_i(t_0)}\leq C(\log \mathcal{N}_l)^{2C}\cdot \mathcal{G}_{j_i(t_0)}.$

By Lemma \ref{lipschitz111}, taking $Y_1=s,~Y_2=\mathcal{N}_l-s,$ we have $\vert\int_{\R/\Z}\Theta_{i_l^1}^{\mathcal{N}_l} dx\vert\leq C\cdot \mathcal{G}_{j_i(t_0)}.$ Hence $$P_3=\left\vert\int_{\R/\Z}\sum\limits_{s\in \mathcal{S}_{l}}\Theta_{s}^{\mathcal{N}_l} dx\right\vert\leq C\cdot \mathcal{N}_l\cdot \mathcal{G}_{j_i(t_0)}.$$

By remark \ref{zhongk}, we have $P_4=\int_{\R/\Z} \frac{\left\vert\partial_t\|A_{i_{j_i(t)-1}^{j+1}-i_{j_i(t)-1}^{j}}\|\right\vert}{\|A_{i_{j_i(t)-1}^{j+1}-i_{j_i(t)-1}^{j}}\|} dx\leq C\lambda^{(\log {\mathcal{N}_{j_i(t)}})^{C}}\leq C\cdot \mathcal{G}_{j_i(t_0)}.$

Combining all these above with \eqref{1-1-1-1}, we immediately obtain $\eqref{tarobtain-lip}$.

\vskip 0.3cm
\subsection{The proof of (3),~(4) of Lemma \ref{lemma 6.1}:}
 $\eqref{tarobtain-lip}$ directly implies (3) of Lemma \ref{lemma 6.1}. For (4), note that $\mathcal{UH}$ implies that $\|u(A_n)-u(A_{n+1})\|_{C^2}+\|s(A_n)-s(A_{n+1})\|_{C^2}\leq C\lambda^{-n},\  n\geq \mathcal{N}_{j_i(t)}.$ Then by all the analysis as above, (4) holds true.

\section{{Appendix}}
\subsection{The proof of Theorem \ref{theorem12} (the induction Theorem)}

\subsubsection{The second step}\
 From the definition, it is not difficult to obtain the following results for $r_1^{\pm}(x,t)$.

\begin{lemma}\label{step1r}

\begin{enumerate}

\item For Type $\textbf{II}_{1},$ it holds that $r_1^{\pm}(x,t)=\min\{|n| \vert (I_{1}+n\alpha)\bigcap I_{1}\neq \emptyset\}.$

\item \quad For Type $\textbf{I}_{1},$ one of the following three cases holds true.
\begin{enumerate}
\item $r_1^{\pm}(x,t)=m_1^{\pm}(t)\ {\rm for}\ x\in I_{1,1}$
and
$r_1^{\pm}(x,t)=m_1^{\mp}(t)\ {\rm\ for}\ x\in I_{1,2}.$\vskip 0.2cm
\item $r_1^{+}(x,t)=\min\{n \vert (I_{1}+n\alpha)\bigcap I_{1}\neq \emptyset,~n> m_1^+(t)\};~r_1^-(x,t)=m_1^-(t)\
\ {\rm for}\ x\in I_{1,1}$ and $$r_1^{-}(x,t)=\min\{n \vert (I_{1}+n\alpha)\bigcap I_{1}\neq \emptyset,~n> m_1^-(t)\};~r_1^+(x,t)=m_1^+(t)\ {\rm for}\ x\in I_{1,2}.$$

\item $r_1^{+}(x,t)=\min\{n \vert (I_{1}+n\alpha)\bigcap I_{1}\neq \emptyset,~n> m_1^-(t)\};~r_1^-(x,t)=m_1^+(t)\ {\rm\ for}\ x\in I_{1,2}$ and $$r_1^{-}(x,t)=\min\{n \vert (I_{1}+n\alpha)\bigcap I_{1}\neq \emptyset,~n> m_1^+(t)\};~r_1^+(x,t)=m_1^-(t)\ {\rm for}\ x\in I_{1,1}.$$

\end{enumerate}
\end{enumerate}
\end{lemma}

\begin{remark}
    (a) means $r_1^{\pm}(x,t)$ is exactly the first forward~(backward~respectively) returning time of $x\in I_1(t)$.

    (b) means for $x\in I_{1,1},$ $m_1< q_N^2\le r_1^{+}(x,t)$ is exactly the second forward returning time of $x\in I_1(t)$ and $r_1^{-}(x,t)$ is exactly the first backward returning time of $x\in I_1(t)$. For $x\in I_{1,2},$ $r_1^{+}(x,t)$ is exactly the first forward returning time of $x\in I_1(t)$ and $r_1^{-}(x,t)$ is exactly the second backward returning time of $x\in I_1(t)$.
The situation of (c) is similar.
Particularly, in case (b) and (c), by the Diophantine condition,
    $$r_1(t)=\min\limits_{X\in\{+,-\}}\min\limits_{x\in I_1(t)}\{r^{X}_1(x,t)\}>\mathcal{N}_1^c.$$

\end{remark}

\
We give the following definitions.
\begin{enumerate}
\item  The \textbf{angle function} for the second step
$$g_2(x,t)=s_{r_1(t)}(x,t)-u_{r_1(t)}(x,t):D_1\rightarrow\R\mathbb{P}^1,$$
where we define
$$D_1:=\{(x,t):x\in I_1(t),t\in [\inf v-\frac{2}{\lambda},\sup v+\frac{2}{\lambda}]\}.$$

We define the \textbf{critical~points} for the $2$nd step $(j=1,2)$
$$C_2(t)=\{x'\in I_1\vert |g_2(x',t)|=\min\limits_{x\in I_1}|g_2(x,t)|\}=\{c_{2,1},c_{2,2}\}.$$
For the one-element case, we assume $c_{2,1}=c_{2,2}$. (This will be shown in Lemma \ref{step12})

\item The \textbf{critical interval} of 2nd step is denoted by
 $$I_{2,j}(t)=\{x:|x-c_{2,j}(t)|\leq {\mathcal{N}_2}^{-{1}}\},\ j=1, 2~and~I_2(t)=I_{2,1}(t)\cup I_{2,2}(t).$$
For $t\in  [\inf v-\frac{2}{\lambda},\sup v+\frac{2}{\lambda}]$ $$D_2(t):=\{(x,t'):x\in I_2(t'),t'\in (t-\lambda^{-q_{N+1}},t+\lambda^{-q_{N+1}})\}.$$

\item The $\textbf{returning~time}$ for the third step:

Let
$$ r_2^{\pm}(x,t)\ge \max\{q_{N+1}^2,r_1\}:I_2(t)\rightarrow \Z^+$$
be the first forward (backward) returning time of $x\in I_2(t)$ back to $I_2(t)$ {\it after $\max\{q_{N+1}^2,r_1\}-1.$}  Let $r_2(t)=\min\{r_2^+(t),r_2^-(t)\}$ with $r_2^{\pm}(t)=\min_{x\in I_2(t)}r_2^{\pm}(x,t).$
\end{enumerate}

Here we have to guarantee that $s_{r_1(t)}(x,t)$ and $u_{r_1(t)}(x,t)$ is well defined. It's sufficient to prove $$\|A_{r_1}(x,t)\|>1.$$ More precisely, we have

\begin{lemma}\label{step12}
It holds that
\begin{equation}\label{step1-2}\|A_{\pm r_1}(x,t)\|\geq \lambda^{(1-(\log \lambda_0)^{-\frac{1}{2}})r_1}.\end{equation}

Let $X,Y\in\{x,t\}$ and $j=1,2$. Then we have
\begin{equation}\label{step1-2*}
\begin{aligned}
\left\vert \|A_{\pm r_1}(x,t)\|^{-1}\partial_X \|A_{\pm r_1}(x,t)\| \right\vert
&\leq r_1 e^{(\log \|A_{\pm r_1}(x,t)\|)^c}, \\
\left\vert \|A_{\pm r_1}(x,t)\|^{-1} \partial^2_{XY} \|A_{\pm r_1}(x,t)\| \right\vert
&\leq r_1^2 e^{(\log \|A_{\pm r_1}(x,t)\|)^c}.
\end{aligned}
\end{equation}

\begin{equation}\label{ci-1i}|c_{1,j}(t)-c_{2,j}(t)|<C\lambda^{-\frac{3}{4}}.\end{equation}

Moreover, the following hold true.

\begin{enumerate}

\item If the first step belongs to $\textbf{I}_1$ and (a) of (2) of Lemma \ref{step1r} holds true, then $g_2(x,t)$ has exactly two zeros $c_{2,1}$ and $c_{2,2}$ with $g_2(x,t')$~satisfies~$\mathcal{N}_2^{-{\hat{\epsilon}^{-2}}}-\textbf{non-resonant}~condition~on~D_2(t)$
\ {\rm and}\ $$\partial_x g_{2}(c_{2,1},t)\cdot \partial_y g_{2}(c_{2,2},t)<0.$$

\item If the first step belongs to $\textbf{II}^0_1,$ then  $\left\vert\{x\in I_{2} \vert |\partial_x g_2(x,t)|=0\}\right\vert=1$ and $g_2(x,t)$ has at most $2$ zeros.

     \begin{enumerate}

    \item If $I_{2,1}\bigcap I_{2,2}=\emptyset,$ then $g_2(x,t')~satisfies~\mathcal{N}_2^{-{\hat{\epsilon}^{-2}}}-\textbf{non-resonant}~condition~on~D_2(t)$
\ {\rm and}\ $$\partial_x g_{2}(c_{2,1},t)\cdot \partial_y g_{2}(c_{2,2},t)<0.$$

\item If $I_{2,1}\bigcap I_{2,2}\neq \emptyset,$ then $\{x\in I_{2} \vert |\partial_x g_2(x,t)|=0\}=\{\tilde{c}_2\}$ and on~$I_{2,1}\bigcup I_{2,2}$ we have
$$|g_2(x,t)-g_2(\tilde{c}_2,t)|>\frac{c_0}{2}(1-(\log \lambda_0)^{-\frac{1}{2}})(x-\tilde{c}_2)^2.$$

\end{enumerate}

   \item If the first step belongs to $\textbf{I}_1$ and (b) or (c) of (2) of Lemma \ref{step1r} holds true, then $\partial_x g_2(x,t)$ possesses one or two zeroes on $I_{2,j}$,
     denoted by $\tilde{c}_{2,j}$ and $\tilde{c}^*_{2,j}$ (it is possible that $\tilde{c}_{2,j}=\tilde{c}^*_{2,j}$). That is,
         \begin{equation}\label{g22*}\{x\in I_{2,j} \vert |\partial_x g_2(x,t)|=0\}=\{\tilde{c}_{2,j},\tilde{c}^*_{2,j}\}.\end{equation} Moreover,
         \begin{equation}\label{c0xj}c_0\leq \vert \partial_t g_2(\tilde{c}_{2,j}(t),t)\vert,\partial_t g_2(\tilde{c}^*_{2,j}(t),t)\vert\leq C_0^3.\end{equation}

     On each interval $I_{2,j}$, \begin{equation}\label{g21}g_2(x,t)~has~at~most~2~zeros:~c_{2,j},c'_{2,i},~1\leq j\neq i\leq 2.\end{equation} And there exists $1\leq |k_1(t)|<q_N^2$ ($|k_1(t)|$ is indeed the first returning time) such that for the case $c_{2,j}\neq c'_{2,i}$, we have
         \begin{equation}\label{g22}|c_{2,1}+k_1\alpha-c'_{2,1}|+|c_{2,2}-k_1\alpha-c'_{2,2}|\leq \lambda^{-\frac{1}{100}r_1}.\end{equation}
         In addition, there exists some $l_{k_1}>0$ such that for any $(x,t')\in D_2(t),$ \begin{equation}\label{lkmo}\frac{l_{k_1}}{2}<\|A_{k_1}\|< 2l_{k_1}~with~\lambda^{\frac{4}{3}|k_1|}>l_{k_1}>\lambda^{\frac{3}{4}|k_1|}\end{equation}
       and
\begin{equation}\label{g21c2}\|g_{2,j}\|_{C^2}\leq Cl^8_{k_1}.\end{equation}

        Moreover, one of the following inequality holds true. $$c_{2,2}'\leq \tilde{c}_{2,1}\leq c_{2,1}; c_{2,2}\leq \tilde{c}_{2,2}\leq c_{2,1}'\quad~or~\quad c_{2,2}'\geq \tilde{c}^*_{2,1}\geq c_{2,1}; c_{2,2}\geq \tilde{c}^*_{2,2}\geq c_{2,1}'.$$ (if $\{x\in I_{2,j} \vert g_2(x,t)=0\}=\emptyset,$ $c_{2,i}'=c_{2,j}=\tilde{c}_{2,j}~or~\tilde{c}^*_{2,j},~j\neq i$)

\vskip 0.2cm
       \noindent  Let
         $\tilde{I}_{2,j}=(\tilde{c}_{2,j}-l_{k_1}^{-1}\mathcal{N}_2^{-2\hat{\epsilon}^{-1}},\tilde{c}_{2,j}+l_{k_1}^{-1}\mathcal{N}_2^{-2\hat{\epsilon}^{-1}}),\quad
         \tilde{I}^*_{2,j}=(\tilde{c}^*_{2,j}-l_{k_1}^{-1}\mathcal{N}_2^{-2\hat{\epsilon}^{-1}},\tilde{c}^*_{2,j}+l_{k_1}^{-1}\mathcal{N}_2^{-2\hat{\epsilon}^{-1}}).$
Then
\begin{equation}\label{g24}\begin{array}{ll}&\vert g_2(\tilde{c}_{2,j}(t),t)-g_2(x,t) \vert\sim_{2,\mathcal{N}_2^{-1}} d'_2(x-\tilde{c}_{2,j}(t))^2,\quad x\in \tilde{I}_{2,j},
\end{array}\end{equation} where $\lambda^{\frac{1}{2}|k_1|}\leq d'_2\leq \lambda^{2 |k_1|}$
and
\begin{equation}\label{g24*}\begin{array}{ll}&\vert g_2(\tilde{c}^*_{2,j}(t),t)-g_2(x,t) \vert\sim_{2,\mathcal{N}_2^{-1}} d''_2(x-\tilde{c}^*_{2,j}(t))^2,\quad x\in\tilde{I}^*_{2,j},
\end{array}\end{equation} where $\lambda^{\frac{1}{2}|k_1|}\leq d''_2\leq \lambda^{2 |k_1|}.$

Finally,
\begin{equation}\label{g25}\pi-\lambda^{-(\log \mathcal{N}_2)^C}l_{k_1}^{-1}\leq \vert g_2(\tilde{c}^*_{2,j}(t),t)-g_2(\tilde{c}_{2,j}(t),t)\vert<\pi-\lambda^{(\log \mathcal{N}_2)^C}l_{k_1}^{-1}\end{equation}
and the following hold true.
          \begin{enumerate}
          \item If $c_{2,2}'\leq \tilde{c}_{2,1}\leq c_{2,1},c_{2,2}\leq \tilde{c}_{2,2}\leq c_{2,1}'$ and $|g_2(\tilde{c}_{2,1}(t),t)|\leq \min\{\mathcal{N}_2^{-\hat{\epsilon}^{-1}},l_{k_1}^{-8}\},$ then
          $$|g_2(\tilde{c}^*_{2,j})|>l_{k_1}^{-2}>|g_2(\tilde{c}_{2,j})|$$ and on $I_{2,j}-\tilde{I}_{2,j},$ we have
$$|\partial_x g_{2,1}(x,t)|,|g_{2,1}(x,t)|>[\min\{\mathcal{N}_2^{-\hat{\epsilon}^{-1}},l_{k_1}^{-8}\}]^2.$$

\item If $c_{2,2}'\leq \tilde{c}_{2,1}\leq c_{2,1}; c_{2,2}\leq \tilde{c}_{2,2}\leq c_{2,1}'$ and $|g_2(\tilde{c}_{2,1}(t),t)|>  \min\{\mathcal{N}_2^{-\hat{\epsilon}^{-1}},l_{k_1}^{-8}\},$ then

    \begin{enumerate}

    \item if $\{x\in I_{2,j} \vert g_2(x,t)=0\}\neq \emptyset,$ $g_2(x,t')$~satisfies~$$ \min\{\mathcal{N}_2^{-\hat{\epsilon}^{-2}},l_{k_1}^{-8}\}-\textbf{non-resonant}~condition~on~D_2(t)
\ {\rm and}\ \partial_x g_{2}(c_{2,1},t)\cdot \partial_y g_{2}(c_{2,2},t)<0.$$

\item if $\{x\in I_{2,j} \vert g_2(x,t)=0\}= \emptyset,$ then
$$\min\limits_{(x,t')\in D_2(t)}|g_2(x,t')|>\min\{\mathcal{N}_2^{-\hat{\epsilon}^{-1}},l_{k_1}^{-8}\}.$$

\end{enumerate}

 \item If $c_{2,2}'\geq \tilde{c}^*_{2,1}\geq c_{2,1}; c_{2,2}\geq \tilde{c}^*_{2,2}\geq c_{2,1}'$ and $|g_2(\tilde{c}^*_{2,j}(t),t)|\leq  \min\{\mathcal{N}_2^{-\hat{\epsilon}^{-1}},l_{k_1}^{-8}\},$ then
    $$|g_2(\tilde{c}^*_{2,j})|<|g_2(\tilde{c}_{2,j})|.$$
    And
on $I_{2,j}-\tilde{I}^*_{2,j},$ we have
$$|\partial_x g_{2,j}(x,t)|,|g_{2,j}(x,t)|> [\min\{\mathcal{N}_2^{-\hat{\epsilon}^{-1}},l_{k_1}^{-8}\}]^2.$$

\item If $c_{2,2}'\leq \tilde{c}^*_{2,1}\leq c_{2,1}; c_{2,2}\leq \tilde{c}^*_{2,2}\leq c_{2,1}'$ and $|g_2(\tilde{c}^*_{2,j}(t),t)|>  \min\{\mathcal{N}_2^{-\hat{\epsilon}^{-1}},l_{k_1}^{-8}\},$ then

    \begin{enumerate}
    \item if $\{x\in I_{2,j} \vert g_2(x,t)=0\}\neq \emptyset,$ $g_2(x,t')$~satisfies~$$ \min\{\mathcal{N}_2^{-\hat{\epsilon}^{-2}},l_{k_1}^{-8}\}-\textbf{non-resonant}~condition~on~D_2(t)
\ {\rm and}\ \partial_x g_{2}(c_{2,1},t)\cdot \partial_y g_{2}(c_{2,2},t)<0.$$
\item if $\{x\in I_{2,j} \vert g_2(x,t)=0\}= \emptyset,$ then
$$\min\limits_{(x,t')\in D_2(t)}|g_2(x,t')|>\min\{\mathcal{N}_2^{-\hat{\epsilon}^{-1}},l_{k_1}^{-8}\}.$$
\end{enumerate}
\end{enumerate}

\end{enumerate}

\end{lemma}
\subsubsection{Proof of Lemma \ref{step12}}
\begin{proof} Clearly, we have to consider the following two cases: the first step belongs to Type $\textbf{II}^0_{1}$ and Type $\textbf{I}_{1}$.
\vskip 0.2cm
\textbf{Proof for Type \(\textbf{II}^0_1\)}

We begin by considering the case where the first step belongs to Type \(\textbf{II}^0_1\). From Lemma \ref{step1r}, we have the following definitions for \(r_1^+(x,t)\) and \(r_1^-(x,t)\)
\[
r_1^+(x,t) = \min\{n \, | \, (I_1 + n\alpha) \cap I_1 \neq \emptyset, ~ n \geq 1\}, \quad r_1^-(x,t) = \min\{|n| \, | \, (I_1 + n\alpha) \cap I_1 \neq \emptyset, ~ n \leq -1\}.
\]
Thus it follows that \(I_1 + l\alpha \cap I_1 = \emptyset\) for any \(1 \leq |l| < r_1\), where
\[
r_1 = \min\limits_{X = +,-} \min\limits_{x \in I_1} \{r_1^+(x,t), r_1^-(x,t)\}.
\]
Therefore for any \(x \in I_1\) and \(1 \leq |l| < r_1\), we have \(x + l\alpha \notin I_1\). Consequently, by \eqref{2jfttt}, we obtain
\[
|g_1(x + l\alpha)| \geq c_0 |I_1|^2 \geq \mathcal{N}_1^{-3\hat{\epsilon}^{-1}} = e^{-3\hat{\epsilon}^{-1}(\log \lambda)^{\hat{\epsilon}}} \geq e^{-(\log \lambda)^{2\hat{\epsilon}}}.
\]

Now, consider the matrix \(A_{r_1(t)}(x,t)\) given by:
\[
A_{r_1(t)}(x,t) = \prod_{l=0}^{r_1-1} \Lambda_l(x,t) R_{\frac{\pi}{2} - g_1(x + l\alpha,t)},
\]
where:
\[
\Lambda_l(x,t) = \left[\begin{array}{cc} \|A(x + l\alpha,t)\| & 0 \\ 0 & \|A(x + l\alpha,t)\|^{-1} \end{array} \right].
\]

From Lemma \ref{lemma4}, equation \eqref{oc}, we recall that for any \(x \in \mathbb{R}/\mathbb{Z}\) and \(t \in \mathbb{R}\), the following holds:
\[
\|A(x,t)\| = (1 + O(\lambda^{-4})) \lambda, \quad \left| \partial_X \|A(x,t)\| \right| + \left| \partial^2_{XY} \|A(x,t)\| \right| \leq C \lambda, \quad X, Y = x, t.
\]

Define \(\lambda_0 = \min_{1 \leq l < r_1} \min_{x \in I_0} \|A(x + l\alpha, t)\|\). If \((a, b) \subset \mathbb{R}/\mathbb{Z}\) with \(b - a = \eta \ll 1\) and \(\tilde{x} \in (a, b)\) fixed, the above implies
\[
\|A(x,t)\| \sim_{0,\eta^{\frac{1}{2}}} \|A(\tilde{x},t)\| \geq \lambda_0 \quad \text{on } (a, b) \times \mathbb{R},
\]
and
\[
\sup_{x \in I_0} \left| \|A(x,t)\|^{-1} \partial_X \|A(x,t)\| \right|, \quad \sup_{x \in I_0} \left| \|A(x,t)\|^{-1} \partial^2_{XY} \|A(x,t)\| \right| \leq C < e^{(\log \lambda_0)^{c}}.
\]

From the cos-type condition, we also know
\[
\sup_{x \in I_0} \left| \partial_x g_1 \right|, \sup_{x \in I_0} \left| \partial_t g_1 \right|, \sup_{x \in I_0} \left| \partial^2_{xx} g_1 \right|, \sup_{x \in I_0} \left| \partial^2_{xt} g_1 \right|, \sup_{x \in I_0} \left| \partial^2_{tt} g_1 \right| < C_0 < e^{(\log \lambda)^c}.
\]
Thus conditions \eqref{jfact1}, \eqref{jfact2}, \eqref{jfact3}, and \eqref{jfact4} from Lemma \ref{jishu1} are satisfied. Consequently, from \eqref{re1} and \eqref{re2} of Lemma \ref{jishu1}, we obtain
\begin{equation}\label{step1m}
\|A_{r_1(t)}(x,t)\| \sim_{0,r_1 l_0^{-1}} \prod_{l=0}^{r_1-1} \|A(x + l\alpha,t)\| \cdot \prod_{l=1}^{r_1-1} |\sin g_1(x + l\alpha,t)| \geq \lambda^{(1 - (\log \lambda_0)^{-\frac{1}{2}}) r_1},
\end{equation}
which implies \(\boxed{\eqref{step1-2}}\).

For \(X, Y = x, t\), we also have
\begin{equation}\label{step1md}
\begin{aligned}
\left| \partial_X \|A_{r_1}\| \right| &\leq r_1 \cdot \|A_{r_1}\| \cdot e^{(\log \lambda_0)^{5\hat{\epsilon}}} \leq r_1 \cdot \|A_{r_1}\| \cdot e^{(\log \|A_{r_1}\|)^{\hat{\epsilon}}}, \\
\left| \frac{\partial^2 \|A_{r_1}\|}{\partial X \partial Y} \right| &\leq r_1^2 \cdot \|A_{r_1}\| \cdot e^{(\log \lambda_0)^{5\hat{\epsilon}}} \leq r_1^2 \cdot \|A_{r_1}\| \cdot e^{(\log \|A_{r_1}\|)^{\hat{\epsilon}}},
\end{aligned}
\end{equation}
which implies \(\boxed{\eqref{step1-2*}}\).

Furthermore  from \eqref{re3} of Lemma \ref{jishu1}, we have
\begin{equation}\label{step1jd}
\|s_{r_1}(x,t) - g_1(x,t)\|_{C^2(D_1)}, \quad \|u_{r_1}(x,t)\|_{C^2(D_1)} \leq \lambda^{-2} e^{(\log \lambda_0)^{5\hat{\epsilon}}},
\end{equation}
which leads to:
\begin{equation}\label{g1-g2}
\begin{aligned}
\|g_2 - g_1\|_{C^2(D_1)} &\leq \|s_{r_1}(x,t) - u_{r_1}(x,t) - g_1\|_{C^2(D_1)} \\
&\leq \|s_{r_1}(x,t) - g_1(x,t)\|_{C^2(D_1)} + \|u_{r_1}(x,t)\|_{C^2(D_1)} \\
&\leq 2 \lambda^{-2} e^{(\log \lambda_0)^{5\hat{\epsilon}}}.
\end{aligned}
\end{equation}
This gives \(\boxed{\eqref{ci-1i}}\).

Moreover, \eqref{g1-g2}, together with \eqref{2jfttt}, shows that
\begin{equation}\label{step22}
\min_{x \in I_1(t)} \left| \frac{\partial^2 g_2(x,t)}{\partial x^2} \right| > \frac{1}{2} c_0 - \lambda^{-1} > \frac{1}{2} (1 - (\log \lambda_0)^{-\frac{1}{2}}) c_0.
\end{equation}
Thus by the definition of Type \(\textbf{II}^0_1\) and \eqref{step22}, \(g_2(x,t)\) has only one extreme point \(\tilde{c}_{2,j}\) (for \(j = 1, 2\)) on \(I_{1,j}\), with
\[
|\tilde{c}_{2,j} - \tilde{c}_{1,j}| \leq C \lambda^{-2} e^{(\log \lambda_0)^{5\hat{\epsilon}}}.
\]

Finally, the cos-type condition and \eqref{g1-g2} yield
$$
\frac{\partial g_2(x,t)}{\partial t} > c_0 - \lambda^{-1} > (1 - (\log \lambda_0)^{-\frac{1}{2}}) c_0.
$$

Recall that \(c_{2,1}(t), c_{2,2}(t) \in I_2 \subset I_1\) satisfy \( |g(c_{2,j}(t),t)| = \min_{x \in I_{2,j}} |g(x,t)| \). If \( |c_{2,1}(t) - c_{2,2}(t)| > 2 \mathcal{N}_2^{-\hat{\epsilon}^{-1}} \), then \(g_2(x,t)\) satisfies the \( \mathcal{N}_2^{-\hat{\epsilon}^{-2}}\)-\textbf{non-resonant} condition, and \(\partial_x g_2(c_{2,1},t) \cdot \partial_y g_2(c_{2,2},t) < 0\), which implies \(\boxed{(a)~\text{of case (2)}}\) of Lemma \ref{step12}.

Otherwise, if \( |c_{2,1}(t) - c_{2,2}(t)| \leq 2 \mathcal{N}_2^{-\hat{\epsilon}^{-1}} \), then we have \( \left| \{x \in I_2 \, | \, |\partial_x g_2(x,t)| = 0\} \right| = 1 \), which, together with \eqref{step22}, implies \(\boxed{(b)~\text{of case (2)}}\) of Lemma \ref{step12}.

This concludes the proof for the case of Type \(\textbf{II}^0_1\).

\

\textbf{Proof for Type~$I_1$:}

Now, we consider the case where the first step belongs to Type $\textbf{I}_{1}$. To address case (1) of Lemma \ref{step12}, it is sufficient to examine case (a) of (2) in Lemma \ref{step1r}, which is denoted by \textbf{lm13.2.a}. Similarly, to address case (3) of Lemma \ref{step12}, it is sufficient to examine cases (b) and (c) of (2) in Lemma \ref{step1r}, which are denoted by \textbf{lm13.2.b} and \textbf{lm13.2.c}, respectively. Let $m_1^+$ and $m_1^-$ be defined as in Lemma \ref{step1r}.
     \begin{enumerate}
     \item[\textbf{lm13.2.a}:] In this case, we have $r_1 = \min\{m_1^+, m_1^-\}$. Therefore for any $x \in I_1$ and $1 \leq |l| < r_1$, we have $x + l \alpha \notin I_1$. This, together with \eqref{1jdfthh0} and \eqref{1jdfthh}, implies that
$$|g_1(x + l \alpha)| \geq c_0 \mathcal{N}_1^{-1} |I_1| \geq \mathcal{N}_1^{-3}.$$

Thus similar to the first case, \eqref{step1m}, \eqref{mojinsi11}, \eqref{mojinsi11*}, \eqref{step1md}, \eqref{step1jd}, and \eqref{g1-g2} hold true. Hence we obtain \boxed{\eqref{step1-2},~\eqref{step1-2*}~and~\eqref{ci-1i}}. By \eqref{1jdfthh0}, \eqref{1jdfthh}, and \eqref{g1-g2}, we have
$$|\partial_x g_2(c_{2,j}(t),t)|,\ |\partial_y g_2(c_{2,j}(t),t)| > c_0 - \lambda^{-1} > (1 - \lambda^{-\frac{1}{2}})c_0.$$
Furthermore  $g_2(x,t)$ satisfies the $\mathcal{N}_1^{-2}$- \textbf{non-resonant} condition on $D_2$ and $\partial_x g_{2}(c_{2,1},t) \cdot \partial_y g_{2}(c_{2,2},t) < 0$. This implies \boxed{case~(1)} of Lemma \ref{step12}.

    \item[\textbf{lm13.2.b}:] In this case, for $x \in I_{1,1}(t)$, there exists an integer $0 < k_1 < q_N^2$ such that $I_{1,1} + k_1 \alpha \cap I_{1,2} \neq \emptyset$. Moreover, for $1 \leq l < k_1$ or $k_1 < l < r_1$, we have
\begin{equation}\label{ffff}
I_{1,1} + l \alpha \cap I_{1,2} = \emptyset.
\end{equation}

Consider
\[ A_{r_1}(x,t) = [A_{r_1-k_1}(x+k_1\alpha,t)][A_{k_1}(x,t)] \text{ and } A_{-r_1}(x,t). \]
Clearly, \eqref{ffff} implies that $A_{r_1-k_1}(x+k_1\alpha,t)$ and $A_{k_1}(x,t)$ possess estimates similar to the first case. For example, we have
\begin{equation}\label{akmo}
\|A_{r_1-k_1}(x+k_1\alpha,t)\| \geq \lambda^{(r_1-k_1)(1 - (\log \lambda_0)^{-\frac{1}{2}})}, \quad \|A_{k_1}(x,t)\| \geq \lambda^{k_1(1 - (\log \lambda_0)^{-\frac{1}{2}})}.
\end{equation}
Since $r_1 > \mathcal{N}_1^c \gg q_N^2 > k_1$, we have
\[ \lambda^{k_1(1 - (\log \lambda_0)^{-\frac{1}{2}})} \leq \|A_{k_1}(x,t)\| \leq \lambda^{2k_1} \ll \lambda^{(r_1-k_1)(1 - (\log \lambda_0)^{-\frac{1}{2}})} \leq \|A_{r_1-k_1}(x+k_1\alpha,t)\|. \]

For $\eta > 0$ and $\tilde{I}_2 \subset I_1$ with $\eta, |\tilde{I}_2| \leq e^{-(\log \lambda_0)^{5000\hat{\epsilon}}}$, we define
\[ \tilde{D}_2(t) = \{ (x, t') \mid x \in \tilde{I}_2, t' \in (t - \eta, t + \eta) \}. \]
Notice that
\[ e^{(\log \lambda_0)^{5000\hat{\epsilon}}} \gg e^{5(\log \lambda_0)^{\hat{\epsilon}}} e^{(\log \lambda_0)^{50\hat{\epsilon}}}. \]

Then, for any fixed $(\tilde{x}, \tilde{t}) \in \tilde{D}_2$, it follows from \eqref{axyxy} that
\begin{equation}\label{mojinsi11}
\|A_{r_1}(x,t)\| \sim_{0, e^{-\frac{1}{2} (\log \lambda_0)^{5000\hat{\epsilon}}}} \|A_{r_1}(\tilde{x}, \tilde{t})\| \text{ on } \tilde{D}_2.
\end{equation}

In particular, since $D_2(t) = \{ (x, t') \mid x \in I_2(t'), t' \in (t - \lambda^{-q_{N+1}}, t + \lambda^{-q_{N+1}}) \}$, we have
\begin{equation}\label{mojinsi11*}
\|A_{r_1}(x,t)\| \sim_{0, \mathcal{N}_2^{-\frac{1}{2}}} \|A_{r_1}(c_{2,1}(\tilde{t}), \tilde{t})\| \text{ on } D_2 \text{ for fixed } \tilde{t}.
\end{equation}
Similar to \eqref{mojinsi11}, for any fixed $(\tilde{x}, \tilde{t}) \in \tilde{D}_2$, we have
$$
\|A_{k_1}\| \sim_{0, e^{-\frac{1}{2} (\log \lambda_0)^{5000\hat{\epsilon}}}} \|A_{k_1}(\tilde{x}, \tilde{t})\|, \quad \|A_{r_1-k_1}\| \sim_{0, e^{-\frac{1}{2} (\log \lambda_0)^{5000\hat{\epsilon}}}} \|A_{r_1-k_1}(\tilde{x} + k_1 \alpha, \tilde{t})\| \text{ on } \tilde{D}_2.
$$

Hence we denote
\begin{equation}\label{r1k1}
\|A_{k_1}(\tilde{x}, \tilde{t})\| = l_{k_1}, \quad \|A_{r_1-k_1}(\tilde{x} + k_1 \alpha, \tilde{t})\| = l_{r_1 - k_1}.
\end{equation}

Clearly, \eqref{akmo} implies
\[ l_{k_1} > \lambda^{\frac{3}{4} |k_1|}, \]
which leads to \boxed{\eqref{lkmo}}.

By (2) of Lemma \ref{lemma9*} and the fact that $r_1 > \mathcal{N}_1^c \gg q_N^2 > k_1$, we obtain
\[
\begin{array}{ll}
\|A_{r_1}(x,t)\| & \geq (1 - \|A_k(x,t)\|^{-1}) \|A_{r_1-k_1}(x+k_1 \alpha, t)\| \|A_k(x,t)\|^{-1} \\
& \geq \lambda^{(r_1-k_1)(1 - (\log \lambda_0)^{-\frac{1}{2}}) - 2k_1} \\
& \geq \lambda^{(r_1-k_1)(1 - (\log \lambda_0)^{-\frac{1}{3}})},
\end{array}
\]
which implies \boxed{\eqref{step1-2}}. On the other hand, similar to the previous case, since for $l \neq k_1$, $0 < l \leq r_1$ we have $I_{1,1} + l \alpha \cap I_{1,2} = \emptyset$, it follows that
\begin{equation}\label{partxxx}
\begin{array}{ll}
\left\vert \partial_X \|A_{k_1}\| \right\vert & \leq k_1 \cdot \|A_{k_1}\| \cdot e^{(\log l_{k_1})^{\hat{\epsilon}}}, \\
\left\vert \frac{\partial^2 \|A_{r_1-k_1}\|}{\partial X \partial Y} \right\vert & \leq (r_1 - k_1)^2 \cdot \|A_{r_1-k_1}\| \cdot e^{(\log l_{r_1 - k_1})^{\hat{\epsilon}}}.
\end{array}
\end{equation}

Now, we denote
\[
\hat{g}_{2,2} = s(A_{r_1-k_1}(x+k_1\alpha, t)) - s(A_{-k_1}(x+k_1\alpha, t)), \quad \hat{g}_{2,1} = s(A_{k_1}(x, t)) - s(A_{-r_1}(x, t)).
\]
Let
\[
\Lambda^m_k(x, t) = \left[\begin{array}{cc}
\|A_{m}(x + k \alpha, t)\| & 0 \\
0 & \|A_{m}(x + k \alpha, t)\|^{-1}
\end{array}\right].
\]
Then
\[
\begin{array}{ll}
A_{r_1}(x, t) &= R_{s(A_{-(r_1-k_1)}(x+k_1\alpha, t))} \cdot \Lambda^{r_1-k_1}_{k_l}(x, t) \cdot R_{\frac{\pi}{2} - \hat{g}_{2,2}} \cdot \Lambda^{k_1}_{0}(x, t) \\
& \cdot R_{\frac{\pi}{2} - \hat{g}_{1,2}} \cdot \Lambda^{-r_1}_{0}(x, t) \cdot R_{\frac{\pi}{2} - s((A_{-r_1})^{-1})}.
\end{array}
\]

By \eqref{jishu1},
\begin{equation}\label{g22h}
\|\hat{g}_{2,2} - g_{1,2}\|_{C^2(I_{2,2})} \leq \|\frac{\pi}{2} - s(A_{r_1-k_1}(x+k_1\alpha, t))\|_{C^2(I_{2,2})} + \|s(A_{-r_1})\|_{C^2(I_{2,2})} \leq 2\lambda^{-2} e^{(\log \lambda_0)^{5{\hat{\epsilon}}}}.
\end{equation}
This implies
\begin{equation}\label{g222}
\|\hat{g}_{2,2}\|_{C^2} \leq 2\lambda^{-2} e^{(\log \lambda_0)^{5{\hat{\epsilon}}}} + \|g_{1,2}\|_{C^2} \leq C.
\end{equation}
Combining \eqref{partxxx} with \eqref{g222}, Lemma \ref{lemma9*} implies \boxed{\eqref{step1-2*}} holds true.

Similarly to \eqref{g22h}, we have
\begin{equation}\label{g22h*}
\|\hat{g}_{2,1} - g_{1,1}\|_{C^2(I_{2,1})} \leq 2\lambda^{-2} e^{(\log \lambda_0)^{5{\hat{\epsilon}}}}.
\end{equation}

By \eqref{1jdfthh0} and \eqref{1jdfthh}, for $j = 1, 2$, there exists $\hat{c}_{2,j}$ such that
\[ \{\hat{c}_{2,2}\} = \{x \in (I_{1,1} + k_1 \alpha) \cup I_{1,2} \mid \hat{g}_{2,j}(x,t) = 0\}, \]
\[ \{\hat{c}_{2,1}\} = \{x \in (I_{1,2} - k_1 \alpha) \cup I_{1,1} \mid \hat{g}_{2,j}(x,t) = 0\}, \]
and
\begin{equation}\label{xiajjja}
(1 + \lambda^{-\frac{1}{2}}) C_0 > |\partial_x \hat{g}_{2,j}(c_{1,j}(t), t)| > (1 - \lambda^{-\frac{1}{2}}) c_0.
\end{equation}
Moreover, $\hat{g}_{2,j}(x, t')$ satisfies the $\mathcal{N}_1^{-{\hat{\epsilon}}^{-2}}$- \textbf{non-resonant} condition on $[(\hat{c}_{2,j} - \eta, \hat{c}_{2,j} + \eta) \times (t - \eta, t + \eta)]$, which yields that for $0 < \eta \leq \mathcal{N}_1^{-{\hat{\epsilon}}^{-2}}$ and on $[(\hat{c}_{2,j} - \eta, \hat{c}_{2,j} + \eta) \times (t - \eta, t + \eta)]$,
\begin{equation}\label{eta111}
\hat{g}_{2,j}(x, t') \sim_{1, \eta^{\frac{1}{2}}} \partial_x \hat{g}_{2,j}(\hat{c}_{2,j}(t), t) (x - \hat{c}_{2,j}(t)) + \partial_t \hat{g}_{2,j}(\hat{c}_{2,j}(t), t) (t' - t),
\end{equation}
and
\[ \partial_x \hat{g}_{2,1}(\hat{c}_{2,1}(t), t) \cdot \partial_x \hat{g}_{2,2}(\hat{c}_{2,2}(t), t) < 0, \]
\[ \partial_t \hat{g}_{2,j}(\hat{c}_{2,j}(t), t) > (1 - \lambda^{-\frac{1}{2}}) c_0 > 0. \]

Now, consider
\[
\Lambda^{r_1-k_1}_{k_l}(x,t) \cdot R_{\frac{\pi}{2} - \hat{g}_{2,2}} \cdot \Lambda^{k_1}_{0}(x,t).
\]

Since \(\|A_{r_1}(x,t)\| > 1\), it can be uniquely expressed as
\[
R_{\phi_1} \cdot \Lambda^{r_1}_{0}(x,t) \cdot R_{\frac{\pi}{2} - \phi_2},
\]
where
\[
\phi_2 = s \left( \Lambda^{r_1-k_1}_{k_l}(x,t) \cdot R_{\frac{\pi}{2} - \hat{g}_{2,2}} \cdot \Lambda^{k_1}_{0}(x,t) \right),
\]
and
\[
\phi_1 = s \left( \left( \Lambda^{r_1-k_1}_{k_l}(x,t) \cdot R_{\frac{\pi}{2} - \hat{g}_{2,2}} \cdot \Lambda^{k_1}_{0}(x,t) \right)^{-1} \right).
\]

Since \( l_{r_1 - k_1} \gg l_{k_1} \), applying Lemma \ref{lemma8*} gives us the following results:

\begin{enumerate}
    \item If \( |\tan \hat{g}_{2,2}(x,t)| > e^{-(\log l_{k_1})^{2\hat{\epsilon}}} \), then
    \begin{equation}
    \label{jhbbx}
    \|\frac{\pi}{2} - \phi_2\|_{C^2}, \|\phi_1\|_{C^2} \leq C e^{(\log l_{k_1})^{3\hat{\epsilon}}} l_{k_1}^{-2}.
    \end{equation}

    \item If \( |\tan \hat{g}_{2,2}(x,t)| \leq e^{-(\log l_{k_1})^{2\hat{\epsilon}}} \), then, given that \( l_{r_1 - k_1} \gg l_{k_1} \), we have
    \[
    \|\phi_1\|_{C^2} \leq l_{k_1}^{-\frac{3}{2}}.
    \]

    \item If \( e^{-(\log l_{k_1})^{2\hat{\epsilon}}} \geq |\tan \hat{g}_{2,2}| \geq l_{k_1}^{-2} e^{(\log l_{k_1})^{2\hat{\epsilon}}} \), then
    \begin{equation}
    \label{sar1}
    \begin{array}{ll}
    \phi_2 \pmod{\pi} & \sim_{2, l_{r_1 - k_1}^{-1}} \frac{\pi}{2} - \arctan[l_{k_1}^{-2} \cot \hat{g}_{2,2}] \pmod{\pi} \\
    & = \arctan[l_{k_1}^{2} \tan \hat{g}_{2,2}] \pmod{\pi}.
    \end{array}
    \end{equation}
    Since \( \|g_{1,2}\|_{C^2} \leq C \), combining \eqref{sar1} with \eqref{g22h} yields
    \begin{equation}
    \label{phi22}
    \|\phi_2\|_{C^2} \leq C l_{k_1}^8.
    \end{equation}

    \item If \( |\tan \hat{g}_{2,2}| \leq l_{k_1}^{-2} e^{(\log l_{k_1})^{2\hat{\epsilon}}} \), then
    \begin{equation}
    \label{daosxjj}
    \left|\frac{\pi}{2} - \phi_2\right| \geq e^{-(\log l_{k_1})^{4\hat{\epsilon}}}, \quad \left|\partial_X \phi_2\right| \geq l_{k_1}.
    \end{equation}
\end{enumerate}

Recall the definition of \( g_2 \):
\begin{equation}
\label{g2bds}
g_{2,1} = \phi_2 - \frac{\pi}{2} + \hat{g}_{2,1}.
\end{equation}

Noting that \eqref{phi22} and \eqref{g2bds} imply
\[
\|g_{2,1}\|_{C^2} \leq C l_{k_1}^8,
\]
which leads to
$
\boxed{\eqref{g21c2}}.
$

If \( l_{k_1}^{-2} e^{(\log l_{k_1})^{2\hat{\epsilon}}} \leq |\tan \hat{g}_{2,2}(x,t)| \leq e^{-(\log l_{k_1})^{2\hat{\epsilon}}} \), then \eqref{g2bds} and \eqref{sar1} imply
$$
g_{2,1}(x,t) + \frac{\pi}{2} - \hat{g}_{2,1}(x,t) \sim_{2, l_{r_1 - k_1}^{-1}} \arctan[l_{k_1}^{2} \tan \hat{g}_{2,2}(x + k_1 \alpha, t)].
$$
If \( l_{k_1}^{-2} e^{(\log l_{k_1})^{2\hat{\epsilon}}} > |\tan \hat{g}_{2,2}(x,t)| \), then by \(\left|\partial_X g_1\right| \ll l_0 \leq l_{k_1}\), \eqref{daosxjj} and \eqref{g2bds} yield
\begin{equation}
\label{g2xjjjj}
|g_{2,1} - \hat{g}_{2,1}| > e^{-(\log l_{k_1})^{4\hat{\epsilon}}}, \quad \left|\partial_X g_{2,1}\right| \geq l_{k_1}.
\end{equation}

Additionally, if \( |\tan \hat{g}_{2,2}(x,t)| > e^{-(\log l_{k_1})^{2\hat{\epsilon}}} \), then \eqref{jhbbx} and \eqref{g2bds} yield
\begin{equation}
\label{hatgg}
\|g_{2,1} - \hat{g}_{2,1}\|_{C^2} \leq C e^{(\log l_{k_1})^{3\hat{\epsilon}}} l_{k_1}^{-2}.
\end{equation}

Now, let \( 0 < \bar{\epsilon} = 1000000\hat{\epsilon} \ll 1 \). Notice that
\[
e^{(\log l_{k_1})^{8\bar{\epsilon}}} \geq e^{(\log \lambda_0)^{8\bar{\epsilon}}} \geq e^{(\log \lambda_0)^{8000000\hat{\epsilon}}} \geq \mathcal{N}_1.
\]

Hence \eqref{eta111} implies that for \( j = 1, 2 \),
\[
\hat{g}_{2,j}(x,t) \sim_{1, e^{-(\log l_{k_1})^{8\bar{\epsilon}}}} \partial_x \hat{g}_{2,j}(\hat{c}_{2,j}(t), t) (x - \hat{c}_{2,j}(t)) \text{ on } (\hat{c}_{2,j} - e^{-(\log l_{k_1})^{8\bar{\epsilon}}}, \hat{c}_{2,j} + e^{-(\log l_{k_1})^{8\bar{\epsilon}}}).
\]

Define
\[
\hat{I}_{2,j} = (\hat{c}_{2,j} - e^{-(\log l_{k_1})^{8\bar{\epsilon}}}, \hat{c}_{2,j} + e^{-(\log l_{k_1})^{8\bar{\epsilon}}}),
\]
\[
\hat{D}_{2,1}(t) = (\hat{I}_{2,1} \cup \hat{I}_{2,2} - k_1 \alpha) \times (t - e^{-(\log l_{k_1})^{8\bar{\epsilon}}}, t + e^{-(\log l_{k_1})^{8\bar{\epsilon}}}),
\]
\[
\hat{D}_{2,2}(t) = (\hat{I}_{2,2} \cup \hat{I}_{2,1} + k_1 \alpha) \times (t - e^{-(\log l_{k_1})^{8\bar{\epsilon}}}, t + e^{-(\log l_{k_1})^{8\bar{\epsilon}}}).
\]

Clearly, by the choice of \(\bar{\epsilon}\), we have \(\hat{I}_{2,j} \subset I_{1,j}\) for \(j = 1, 2\).

Therefore
\[
(\hat{I}_{2,1} + k_1 \alpha) \cup \hat{I}_{2,2} \subset [(I_{1,1} + k_1 \alpha) \cup I_{1,2}].
\]

Then we consider the following two cases.
\begin{enumerate}
\item \((\hat{I}_{2,1} + k_1 \alpha) \cap \hat{I}_{2,2} = \emptyset\)

In this case, for any \(x \in \hat{I}_{2,1}\), it follows from \eqref{xiajjja} that
\[
|\tan \hat{g}_{2,2}(x + k_1 \alpha, t)| > e^{-(\log l_{k_1})^{9 \bar{\epsilon}}}.
\]
Applying \eqref{hatgg} and noting that \(l_{k_1} > \lambda^{\frac{1}{2} k_1} \gg e^{(\log \lambda_0)^{9 \bar{\epsilon}}}\), we find that \(g_2\) has exactly one zero on each interval \(\hat{I}_{2,j}\) for \(j = 1, 2\). These zeros are denoted by \(c_{2,j}\) and satisfy
\begin{equation}\label{c2jjj}
|c_{2,j} - \hat{c}_{2,j}| \leq l_{k_1}^{-\frac{3}{2}}.
\end{equation}
Combining \eqref{g22h}, \eqref{g22h*}, and \eqref{c2jjj}, we obtain \(\boxed{\eqref{ci-1i}}\).

Additionally, for \(X = x, t\), we have
\[
\partial_X g_{2} \sim_{0, e^{-(\log l_{k_1})^{\bar{\epsilon}}}} \partial_X \hat{g}_{2,1} \text{ on } \hat{I}_{2,1},
\]
\[
\partial_X g_{2} \sim_{0, e^{-(\log l_{k_1})^{\bar{\epsilon}}}} \partial_X \hat{g}_{2,2} \text{ on } \hat{I}_{2,2}.
\]
Thus \(g_2 \vert_{\hat{D}_{2,1} \cup \hat{D}_{2,2}}\) satisfies the \(e^{-2(\log l_{k_1})^{\bar{\epsilon}}}\)-\textbf{non-resonant} condition.

\item \((\hat{I}_{2,1} + k_1 \alpha) \cap \hat{I}_{2,2} \neq \emptyset\)

In this scenario, the interval \(\hat{I}_{2,1} \cup (\hat{I}_{2,2} - k_1 \alpha)\) has length
\begin{equation}\label{case22222}
\vert \hat{I}_{2,1} \cup (\hat{I}_{2,2} - k_1 \alpha) \vert \leq 2 e^{-(\log l_{k_1})^{8 \bar{\epsilon}}}.
\end{equation}

We aim to show that the zeros of \(g_{2,1}\) occur within the intersection \((\hat{I}_{2,1} + k_1 \alpha) \cap \hat{I}_{2,2}\). Using \eqref{g22h}, \eqref{g22h*}, and \eqref{case22222}, we deduce \(\boxed{\eqref{ci-1i}}\).

Additionally, the upper bound from \eqref{xiajjja} gives:
\[
|\tan \hat{g}_{2,2}(x + k_1 \alpha, t)| \leq 2 \mathcal{N}_1^2 e^{-(\log l_{k_1})^{8 \bar{\epsilon}}} \leq e^{-(\log l_{k_1})^{7 \bar{\epsilon}}}.
\]
For \(j = 1, 2\), the interval
\[
\hat{I}^*_{2,j} := \{ x \mid l_{k_1}^{-2} e^{(\log l_{k_1})^{2 \bar{\epsilon}}} > |\tan \hat{g}_{2,j}(x, t)| \}
\]
is contained within \(\hat{I}_{2,j}\). From \eqref{g2xjjjj}, for \(x \in \hat{I}^*_{2,2} - k_1 \alpha\), we have:
\[
\begin{array}{ll}
|g_{2,1}(x, t)| & > \left| \arctan[l_{k_1}^{2} \tan \hat{g}_{2,2}(x + k_1 \alpha, t)] - \frac{\pi}{2} \right| - |\hat{g}_{2,1}| \\
& > e^{-(\log l_{k_1})^{4 \bar{\epsilon}}} - e^{-(\log l_{k_1})^{8 \bar{\epsilon}}} > e^{-(\log l_{k_1})^{5 \bar{\epsilon}}}.
\end{array}
\]
Thus for \(x \in \hat{I}^*_{2,2} - k_1 \alpha\), it follows that
\[
|g_{2,1}(x, t)| > e^{-(\log l_{k_1})^{5 \bar{\epsilon}}}.
\]
Additionally,
\[
\vert \partial_x g_{2,1}(x, t) \vert > l_{k_1}.
\]
These information helps us determine the geometric properties of \(g_{2,1}\) on \(\hat{I}^*_{2,2} - k_1 \alpha\).

For \(x \notin \hat{I}^*_{2,2} - k_1 \alpha\), we have
\[
l_{k_1}^{-2} e^{(\log l_{k_1})^{2 \bar{\epsilon}}} \leq |\tan \hat{g}_{2,2}(x + k_1 \alpha, t)| \leq e^{-(\log l_{k_1})^{7 \bar{\epsilon}}}.
\]
On the interval \([\hat{I}_{2,1} \cup (\hat{I}_{2,2} - k_1 \alpha)] - [\hat{I}^*_{2,2} - k_1 \alpha]\), we have
\begin{equation}\label{g2zx}
g_2 = (1 + o(l_{r_1 - k_1}^{-1})) \arctan[l_{k_1}^{2} \tan \hat{g}_{2,2}] - \frac{\pi}{2} + \hat{g}_{2,1}.
\end{equation}
Since \(l_{r_1 - k_1} \gg l_{k_1}\), this implies that \(g_2\) and \(\arctan[l_{k_1}^{2} \tan \hat{g}_{2,2}] - \frac{\pi}{2} + \hat{g}_{2,1}\) share the same geometric properties.

According to \eqref{eta111} for \(j = 1, 2\) and \((x, t') \in \hat{D}_{2,j}\), we have:
\begin{equation}\label{hatggg}
\hat{g}_{2,j}(x, t') \sim_{1, \eta} \partial_x \hat{g}_{2,j}(\hat{c}_{2,j}(t), t) (x - \hat{c}_{2,j}(t)) + \partial_t \hat{g}_{2,j}(\hat{c}_{2,j}(t), t) (t' - t),
\end{equation}
\[
(1 + \lambda^{-1}) C_0 > \partial_t \hat{g}_{2,j}(x, t'), \quad \left\vert \partial_x \hat{g}_{2,j}(x, t') \right\vert \geq (1 - \lambda^{-1}) c_0 > 0,
\]
\[
|\partial^2_{xt} \hat{g}_{2,j}| + |\partial^2_{x^2} \hat{g}_{2,j}| + |\partial^2_{t^2} \hat{g}_{2,j}| < (1 + \lambda_1^{-1}) C_0.
\]

These, along with \eqref{mojinsi11}, \eqref{r1k1}, and \eqref{partxxx}, confirm (i) of Lemma \ref{thetat1fenjie}. Thus for any \(t' \in (t - e^{-(\log l_{k_1})^{8 \bar{\epsilon}}}, t + e^{-(\log l_{k_1})^{8 \bar{\epsilon}}})\), the function \(\arctan[l_{k_1}^{2} \tan \hat{g}_{2,2}(x + k_1 \alpha, t')] - \frac{\pi}{2} + \hat{g}_{2,1}(x, t')\) has at most two zeros on \(\hat{I}_{2,1} \cup [\hat{I}_{2,2} - k_1 \alpha]\), denoted by \(d_{2,1}\) and \(d'_{2,2}\) (if they exist) with \(d_{2,1} \leq d'_{2,2}\).

Similarly, \(\arctan[l_{k_1}^{2} \tan \hat{g}_{2,1}(x - k_1 \alpha, t')] - \frac{\pi}{2} + \hat{g}_{2,2}(x, t')\) also has at most two zeros on \(\hat{I}_{2,2} \cup [\hat{I}_{2,1} + k_1 \alpha]\), denoted by \(d_{2,2}\) and \(d'_{2,1}\) (if they exist) with \(d_{2,2} \geq d'_{2,1}\). This implies \(\boxed{\eqref{g21}}\).

From \eqref{g2zx}, we conclude that \(g_2\) has at most two zeros on \(\hat{I}_{2,1} \cup [\hat{I}_{2,2} - k_1 \alpha]\), denoted by \(c_{2,1}\) and \(c'_{2,2}\) (if they exist) with:
\begin{equation}\label{d21}
|d_{2,1} - c_{2,1}|, \quad |d'_{2,2} - c'_{2,2}| \leq l_{r_1 - k_1}^{-\frac{1}{2}}.
\end{equation}

Similarly, \(g_2\) also has at most two zeros on \(\hat{I}_{2,2} \cup [\hat{I}_{2,1} + k_1 \alpha]\), denoted by \(c_{2,2}\) and \(c'_{2,1}\) (if they exist) with:
\begin{equation}\label{d22}
|d_{2,2} - c_{2,2}|, \quad |d'_{2,1} - c'_{2,1}| \leq l_{r_1 - k_1}^{-\frac{1}{2}}.
\end{equation}

On the other hand, for \(y =x + k_1 \alpha\), the following are equivalent:
\[
\begin{array}{ll}
& \arctan[l_{k_1}^{2} \tan \hat{g}_{2,2}(x + k_1 \alpha, t')] - \frac{\pi}{2} + \hat{g}_{2,1}(x, t') = 0 \\
\Leftrightarrow & \arctan[l_{k_1}^{2} \tan \hat{g}_{2,1}(x, t')] - \frac{\pi}{2} + \hat{g}_{2,2}(x + k_1 \alpha, t') = 0 \\
\Leftrightarrow & \arctan[l_{k_1}^{2} \tan \hat{g}_{2,1}(y - k_1 \alpha, t')] - \frac{\pi}{2} + \hat{g}_{2,2}(y, t') = 0.
\end{array}
\]
Therefore
\begin{equation}\label{d2122}
d_{2,1} + k_1 \alpha = d'_{2,2}, \quad d_{2,2} - k_1 \alpha = d'_{2,1}.
\end{equation}

Using \eqref{d21}, \eqref{d22}, and \eqref{d2122}, we obtain:
\[
|c_{2,1} + k_1 \alpha - c'_{2,1}|, \quad |c_{2,2} - k_1 \alpha - c'_{2,2}| < l_{r_1 - k_1}^{-\frac{1}{3}} < \lambda^{-\frac{1}{4}(r_1 - k_1)} < \lambda^{-\frac{1}{100}r_1}.
\]
This suggests \(\boxed{\eqref{g22}}\).

According to (iii-a) of Lemma \ref{thetat1fenjie}, we have:
\[
\{x \in \hat{I}_{2,1} \cup [\hat{I}_{2,2} - k_1 \alpha] \mid |\partial_x g_2(x, t)| = 0\} = \{\tilde{c}_{2,1}, \tilde{c}^*_{2,1}\}, \quad \tilde{c}_{2,1} \geq \tilde{c}^*_{2,1},
\]
\[
\{x \in \hat{I}_{2,2} \cup [\hat{I}_{2,1} + k_1 \alpha] \mid |\partial_x g_2(x, t)| = 0\} = \{\tilde{c}_{2,2}, \tilde{c}'_{2,2}\}, \quad \tilde{c}_{2,2} \leq \tilde{c}^*_{2,2}.
\]
Thus \(\boxed{\eqref{g22*}}\).

By applying (v) of Lemma \ref{thetat1fenjie}, on \(\Pi_2 D_2\), we obtain:
\begin{equation}\label{jinsijieguo}
\left| g_2(\tilde{c}_{2,1}(t), t) - g_2(\tilde{c}^*_{2,1}(t), t) \right| \sim_{0, \mathcal{N}_1^{-1}} \pi - 4 \sqrt{\frac{|\partial_x \hat{g}_{2,1}(\hat{c}_{2,1}(t), t)|}{|\partial_x \hat{g}_{2,2}(\hat{c}_{2,2}(t), t)|}} l_{k_1}^{-1}.
\end{equation}
Combining this with \eqref{hatggg} gives \(\boxed{\eqref{g25}}\).

Finally, (iv) of \eqref{thetat1fenjie} implies:
\[
\partial^2_{x^2} g_2(\tilde{c}_{2,j}(t), t) = -2 (|\partial_x \hat{g}_{2,2}(\hat{c}_{2,2}(t), t)|)^{\frac{1}{2}} (|\partial_x \hat{g}_{2,1}(\hat{c}_{2,1}(t), t)|)^{\frac{3}{2}} l_{k_1},
\]
\[
\partial^2_{x^2} g_2(\tilde{c}^*_{2,j}(t), t) = 2 (|\partial_x \hat{g}_{2,2}(\hat{c}_{2,2}(t), t)|)^{\frac{1}{2}} (|\partial_x \hat{g}_{2,1}(\hat{c}_{2,1}(t), t)|)^{\frac{3}{2}} l_{k_1}.
\]
Combining these results with \eqref{hatggg}, we obtain \(\boxed{\eqref{g24}}\) and \(\boxed{\eqref{g24*}}\).

If \(\{x \in [\hat{I}_{2,1} \cup [\hat{I}_{2,2} - k_1 \alpha]] \cup [\hat{I}_{2,2} \cup [\hat{I}_{2,1} + k_1 \alpha]] \mid g_2(x, t) = 0\} \neq \emptyset\), then we have either:
\begin{equation}\label{youduandian}
\tilde{c}^*_{2,1} \leq c'_{2,2} \leq \tilde{c}_{2,1} \leq c_{2,1}, \quad c_{2,2} \leq \tilde{c}_{2,2} \leq c'_{2,1} \leq c^*_{2,2},
\end{equation}
or:
\begin{equation}\label{zuoduandian}
c_{2,1} \leq \tilde{c}^*_{2,1} \leq c'_{2,2} \leq \tilde{c}_{2,1}, \quad c^*_{2,2} \leq c'_{2,1} \leq \tilde{c}_{2,2} \leq c_{2,2}.
\end{equation}

Therefore we need to consider the following cases:

\begin{enumerate}
\item[i:] Suppose \(\{x \in [\hat{I}_{2,1} \cup (\hat{I}_{2,2} - k_1 \alpha)] \cup [\hat{I}_{2,2} \cup (\hat{I}_{2,1} + k_1 \alpha)] \mid g_2(x,t) = 0\} \neq \emptyset\). In this scenario, either \(\eqref{youduandian}\) or \(\eqref{zuoduandian}\) holds. Consequently, (iii-c) of Lemma \ref{thetat1fenjie} provides the bound:
$$ c_0 \eta \leq |c_{2,1} - c'_{2,2}| \leq \min\{4 C_0 \mathcal{N}_1^{-3\hat{\epsilon}^{-1}}, 2 \mathcal{N}_1^{\hat{\epsilon}} \eta^{\frac{1}{2}}\}, $$
where
$$ \eta = \min\{\left\vert g_2(\tilde{c}_{2,1}(t), t)\right\vert, \left\vert g_2(\tilde{c}^*_{2,1}(t), t)\right\vert\}. $$

\item[ii:] Suppose \(\{x \in [\hat{I}_{2,1} \cup (\hat{I}_{2,2} - k_1 \alpha)] \cup [\hat{I}_{2,2} \cup (\hat{I}_{2,1} + k_1 \alpha)] \mid g_2(x,t) = 0\} = \emptyset\). In this case, there are no zeros of \(g_2\) in the specified intervals.

\end{enumerate}

Therefore
\begin{enumerate}
\item In case (i), without loss of generality, assume
$$|g_2(\tilde{c}_{2,1})| = \min\{\left\vert g_2(\tilde{c}_{2,1}(t),t)\right\vert, \left\vert g_2(\tilde{c}^*_{2,1}(t),t)\right\vert\}.$$
Let
$$\min\{\mathcal{N}_2^{-2\hat{\epsilon}^{-1}}, l_{k_1}^{-8}\} = m^*.$$

    \textbf{A1:} If \( |g_2(\tilde{c}_{2,1})| > m^* \), it follows that
    $$|c_{2,1} - c'_{2,2}| > c_0 m^* > 2(m^*)^2.$$
    On the other hand, (iii-d) of Lemma \ref{thetat1fenjie} yields
    $$|\partial_x g_2(c_{2,1}(t),t)| > \mathcal{N}_1^{-\hat{\epsilon}^{-1}} |\tilde{c}_{2,1}(t) - c_{2,1}(t)| > c_0 m^* \mathcal{N}_1^{-\hat{\epsilon}^{-1}} > (m^*)^2.$$
    Thus by (iii-e) of Lemma \ref{thetat1fenjie}, for \(X = x, t\),
    $$|\partial_X g_2(x,t)| \sim_{0,m^*} |\partial_X g_2(c_{2,1}(t),t)| \text{ on } \tilde{I}_{2,1},$$
    where \(\tilde{I}_{2,1} = (c_{2,1} - l_{k_1}^{-9} \mathcal{N}_2^{-2\hat{\epsilon}^{-1}]}, c_{2,1} + l_{k_1}^{-9} \mathcal{N}_2^{-2\hat{\epsilon}^{-1}}
    \).
    This implies that \(g_2\) satisfies the \(m^*\)-\textbf{non-resonant} condition, completing the proof for \boxed{case~(b)-(i)~and~(d)-(i)}.

    \textbf{A2:} If \( |g_2(\tilde{c}_{2,1})| \leq m^* \), the condition
    $$[I_{2,1}(t) + k_1 \alpha] \cap I_{2,2}(t) \neq \emptyset$$
    may occur. However, with \(\tilde{I}_{2,1} = (\tilde{c}_{2,1} - l_{k_1}^{-1} \mathcal{N}_2^{-2\hat{\epsilon}^{-1}}, \tilde{c}_{2,1} + l_{k_1}^{-1} \mathcal{N}_2^{-2\hat{\epsilon}^{-1}})\), we have the following estimate.

    On \(\tilde{I}_{2,1}\),
    $$\begin{array}{ll}
    \vert g_2(\tilde{c}_{2,1}(t), t) - g_2(x, t) \vert \\
    \sim_{2, \mathcal{N}_2^{-1}} \left[2 \left(|\partial_x \hat{g}_{2,2}(\hat{c}_{2,2}(t), t)|\right)^{\frac{1}{2}} \left(|\partial_x \hat{g}_{2,1}(\hat{c}_{2,1}(t), t)|\right)^{\frac{3}{2}} l_{k_1}\right] (x - \tilde{c}_{2,1}(t))^2
    \end{array}.$$

    On \(I_{2,1} - \tilde{I}_{2,1}\), we have
    $$|\partial_x g_{2,1}(x, t)|, |g_{2,1}(x, t)| > (m^*)^2.$$
    Furthermore  \eqref{jinsijieguo} yields
    $$|g_2(\tilde{c}_{2,1}(t), t)| < l_{k_1}^{-8} < l_{k_1}^{-2} < |g_2(\tilde{c}^*_{2,1}(t), t)|.$$
    This completes the proof for \boxed{case~(a)~and~(c)}.

    Additionally, in both Case (A1) and Case (A2), by (iii-b) of Lemma \ref{thetat1fenjie}, for \(1 \leq i \neq j \leq 2\),
    $$c_0 < |\partial_t \hat{g}_{2,j}(\hat{c}_{2,j}(t))| \leq \vert \partial_t g_2(\tilde{c}_{2,j}(t), t)\vert \leq |\partial_t \hat{g}_{2,j}(\hat{c}_{2,j}(t))| + |\partial_t \hat{g}_{2,i}(\hat{c}_{2,i}(t))| \frac{|\partial_x \hat{g}_{2,i}(\hat{c}_{2,i}(t))|}{|\partial_x \hat{g}_{2,j}(\hat{c}_{2,j}(t))|} < C_0 + C_0 \frac{C_0}{c_0} < C_0^3.$$
    Similarly,
    $$c_0 \leq \vert \partial_t g_2(\tilde{c}^*_{2,j}(t), t)\vert \leq C_0^3.$$
    This completes the proof of \boxed{\eqref{c0xj}}.

\item If \(\{x \in [\hat{I}_{2,1} \cup (\hat{I}_{2,2} - k_1 \alpha)] \cup [\hat{I}_{2,2} \cup (\hat{I}_{2,1} + k_1 \alpha)] \mid g_2(x, t) = 0\} = \emptyset\), then
$$\min\limits_{(x, t') \in D_2(t)} |g_{2,1}| = \min\{g_{2,1}(\tilde{c}^*_{2,j}(t), t), g_{2,1}(\tilde{c}_{2,j}(t), t)\}.$$
Thus \eqref{g25}, which was previously obtained, implies \boxed{case~(b)-(ii)~and~(d)-(ii)}.

\vskip 0.3cm
\end{enumerate}

\end{enumerate}

\item[\textbf{lm13.2.c}:] The proof is similar to the case \textbf{lm13.2.b}.

    \end{enumerate}
Then we finish the proof.
\hfill\qed\end{proof}

\begin{remark} Although $g_1$  is first-order and second-order non-degenerate at each non-extreme or extreme point from the cos-type condition, to ensure the induction can go forward, we have to worry about whether $g_i,\ i\ge 2$ is still non-degenerate due to the existence of so-called resonance between critical points. In fact, the resonance is inevitable by the ergodicity of base dynamics.
To see the relationship between the resonance and the shape of $g_{i+1}$, let us consider the following baby model:
 \begin{example} Let  $h_1(x)= -(x+\frac14),\ h_2(x)=x-\frac14$. Then $c_i=\mp \frac14$
 is the unique zero of $h_i$,\ $i=1,\ 2$. Then $\arctan(\lambda^2\tan h_2(x+d))-\frac{\pi}{2}$ is a translation of the graph of  $\arctan(\lambda^2\tan  x)$ with $-\frac12<x<0$, $0\le d<\frac12$ and $\lambda\gg 1$. In particular,  $\arctan(\lambda^2\tan(h_2(x+d))-\frac{\pi}{2}\thickapprox 0\ {\rm\ (mod\ \pi)}$ out of a O($(1/\lambda)^{2-}$)-neighbor  of $\frac14-d$. while it jumps rapidly from $-\pi$ to $0$ in this small neighbor (we call it a pulse). Now for $0<\lambda^{-1}\ll \delta\ll 1$, we consider
 $$h(x)=h_1(x)+\arctan(\lambda^2\tan(h_2(x+d))-\frac{\pi}{2}\ {\rm\ (mod\ \pi)},\qquad -\frac14-\delta\le x <-\frac14+\delta.$$
 Then the shape of the graph of $h$ depends heavily on $d$. In fact, if $d\not\thickapprox \frac12$, that is, $|c_1+d-c_2|\gtrsim\delta$, then $h(x)\thickapprox h_1(x)$ on the whole interval $(-\frac14-\delta, -\frac14+\delta)$. However, if $d\thickapprox \frac12$, that is  $c_1+d\thickapprox c_2$, then $h(x)$ is a superposition of  $h_1$ and a pulse in some O($(1/\lambda)^{2-}$) subinterval although $h(x)\thickapprox h_1(x)$ elsewhere. We will find that the role of the resonance plays on the shape of $g_{i+1}$ is similar  as $d\thickapprox \frac12$ does on $h$ here.

  \end{example}

 \end{remark}

By the help of Lemma \ref{step12}, we give the following definition, which determines the types of step 2.

\

\


\subsubsection{The $(i+1)$-step}\
Now we will show the following holds true for step $(i+1).$
\begin{lemma}\label{stepi+11}

$$\|A_{\pm r_i}(x,t)\|\geq \lambda^{(1-\sum\limits_{j=0}^{i-1}(r_{j}\log \lambda_0)^{-\frac{1}{2}})r_i}.$$

For $X,Y\in \{x,t\},$ it holds that
\begin{align*}
&\left\vert \|A_{\pm r_i}(x,t)\|^{-1} \partial_X \|A_{\pm r_i}(x,t)\| \right\vert \leq r_i e^{(\log \|A_{\pm r_i}(x,t)\|)^{\hat{\epsilon}}}, \\
&\left\vert \|A_{\pm r_i}(x,t)\|^{-1} \partial^2_{XY} \|A_{\pm r_i}(x,t)\| \right\vert \leq r_i^2 e^{(\log \|A_{\pm r_i}(x,t)\|)^{\hat{\epsilon}}}.
\end{align*}

$$|c_{i-1,j}(t)-c_{i,j}(t)|<C\lambda^{-\frac{3}{4}r_{i-2}},j=1,2.$$

For $g_{i+1}$, there exists the following several cases.

\begin{enumerate}

\item The $i$th step belongs to $\textbf{I}_i$ and both $r_i^+(x,t)$ and $r_i^-(x,t)$ are exactly the first returning time after $\min\{q^2_{N+i-1},r_{i-1}\}$ back to $I_i$. It holds that $g_{i+1}(x,t)$ exactly has two zeros $c_{i+1,1}$ and $c_{i+1,2}$ with $g_{i+1}$ satisfies $\lambda^{-(\log \mathcal{N}_{i+1})^C}-\textbf{non-resonant}~condition~on~D_{i+1}(t)$
\ {\rm and}\ $$\partial_x g_{i+1}(c_{i+1,1},t)\cdot \partial_y g_{i+1}(c_{i+1,2},t)<0.$$







   \item The $i$th step belongs to $\textbf{I}_i,$ and $r_i^+(x,t)$ or $r_i^-(x,t)$ is the first returning time after $\min\{q^2_{N+i-1},r_{i-1}\}-1$ back to $I_{i+1}$. Then the following hold true.


        For each $t$ and $j$, the function $|g_{i+1,j}(x,t){\rm\ mod}\ \pi|$ composes of
one or two minimum points, denoted by $C^{(n+1,j)}=\{c_{n+1,j},\ \ c'_{n+1,j}\}$ (for the one-element case, we assume
$c_{n+1,j}=c'_{n+1,j})$. For $C^{(n+1,j)}$, there exists $k_*(t)\in \R$ satisfying $\max\{r_{i-1},q^2_{N+i-2}\}\leq |k_*(t)|<q^2_{N+i-1}$ such that
         $$|c_{i+1,1}+k_*\alpha-c'_{i+1,1}|+|c_{i+1,2}-k_*\alpha-c'_{i+1,2}|\leq \lambda^{-\frac{1}{100}r_{i}}.$$
        Furthermore, $\partial_x g_{i+1}(x,t)$ possesses one or two zeroes for each $t$ and $j$, denoted by
         $\{\tilde{c}_{i+1,j},\tilde{c}^*_{i+1,j}\}$ (for the one-element case, we assume
$\tilde{c}_{n+1,j}=\tilde{c}^*_{n+1,j})$. Then
         $$c_0\leq \vert \partial_t g_{i+1}(\tilde{c}_{i+1,j}(t),t)\vert, \vert\partial_t g_{i+1}(\tilde{c}^*_{i+1,j}(t),t)\vert\leq q^C_{N+i}.$$
         Moreover,  we have
         either $c_{i+1,2}'\leq \tilde{c}_{i+1,1}\leq c_{i+1,1}; c_{i+1,2}\leq \tilde{c}_{i+1,2}\leq c_{i+1,1}'$ or $c_{i+1,2}'\geq \tilde{c}^*_{i+1,1}\geq c_{i+1,1}; c_{i+1,2}\geq \tilde{c}^*_{i+1,2}\geq c_{i+1,1}'$ (if $\{x\in I_{i+1,j} \vert g_{i+1}(x,t)=0\}=\emptyset,$ $c_{i+1,j'}'=c_{i+1,j}=\tilde{c}_{i+1,j}~or~\tilde{c}^*_{i+1,j},~j\neq j'$).

          In addition, there exists ${k_*}>0$ satisfying $k_*\le\min\{q^2_{N+i-1},r_{i-1}\}$ such that for any $(x,t')\in D_{i+1}(t),$ it follows that $$\left\{\begin{array}{ll}
          &\frac{l_{k_*}}{2}<\|A_{k_*}\|< 2l_{k_*}~with~\lambda^{\frac{4}{3}|k_*|}>l_{k_*}>\lambda^{\frac{3}{4}|k_*|},\\
          &\|g_{i+1,j}\|_{C^2}\leq Cl^8_{k_*},\\
          &\vert g_{i+1}(\tilde{c}^*_{i+1,j}(t),t)-g_{i+1}(\tilde{c}_{i+1,j}(t),t)\vert<\pi-\lambda^{(\log \mathcal{N}_{i+1})^C}l_{k_*}^{-1},\\
          &\pi-\lambda^{-(\log \mathcal{N}_{s(k_*)+1})^C}l_{k_*}^{-1}\leq \vert g_{i+1}(\tilde{c}^*_{s(k_*)+1,j}(t),t)-g_{i+1}(\tilde{c}_{s(k_*)+1,j}(t),t)\vert.\end{array}\right.$$

   Additionally,       on~$\tilde{I}_{i+1,j}=(\tilde{c}_{i+1,j}-l_{k_*}^{-1}\mathcal{N}_{i+1}^{-2\hat{\epsilon}^{-1}},\tilde{c}_{i+1,j}+l_{k_*}^{-1}\mathcal{N}_{i+1}^{-2\hat{\epsilon}^{-1}})$
there exists $d'_{i+1}\in \R$ satisfying $\lambda^{\frac{1}{2}|k_*|}\leq d'_{i+1}\leq \lambda^{2 |k_*|}$ such that $$\begin{array}{ll}&\vert g_{i+1}(\tilde{c}_{i+1,j}(t),t)-g_{i+1}(x,t) \vert\sim_{2,\mathcal{N}_{i+1}^{-1}} d'_{i+1}(x-\tilde{c}_{i+1,j}(t))^2
\end{array}$$
and
on~$\tilde{I}^*_{i+1,j}=(\tilde{c}^*_{i+1,j}-l_{k_*}^{-1}\mathcal{N}_{i+1}^{-2\hat{\epsilon}^{-1}},\tilde{c}^*_{i+1,j}+l_{k_*}^{-1}\mathcal{N}_{i+1}^{-2\hat{\epsilon}^{-1}})$
there exists $d''_{i+1}\in \R$ satisfying $\lambda^{\frac{1}{2}|k_*|}\leq d''_{i+1}\leq \lambda^{2 |k_*|}$ such that $$\begin{array}{ll}&\vert g_{i+1}(\tilde{c}^*_{i+1,j}(t),t)-g_{i+1}(x,t) \vert\sim_{2,\mathcal{N}_{i+1}^{-1}} d''_{i+1}(x-\tilde{c}^*_{i+1,j}(t))^2.
\end{array}$$

Finally, we have the following four cases.

          \begin{enumerate}
          \item If $c_{i+1,2}'\leq \tilde{c}_{i+1,1}\leq c_{i+1,1}; c_{i+1,2}\leq \tilde{c}_{i+1,2}\leq c_{i+1,1}'$ and $|g_{i+1}(\tilde{c}_{i+1,1}(t),t)|\leq \min\{\mathcal{N}_{i+1}^{-\hat{\epsilon}^{-1}},l_{k_*}^{-8}\},$ then
          $$|g_{i+1}(\tilde{c}^*_{i+1,j})|>|g_{i+1}(\tilde{c}_{i+1,j})|>l_{k_*}^{-2}$$ and on $I_{i+1,j}-\tilde{I}_{i+1,j},$ we have
$$|\partial_x g_{i+1,j}(x,t)|,|g_{i+1,j}(x,t)|>[\min\{\mathcal{N}_{i+1}^{-\hat{\epsilon}^{-1}},l_{k_*}^{-8}\}]^2.$$

\item If $c_{i+1,2}'\leq \tilde{c}_{i+1,1}\leq c_{i+1,1},\ c_{i+1,2}\leq \tilde{c}_{i+1,2}\leq c_{i+1,1}'$ and $|g_{i+1}(\tilde{c}_{i+1,1}(t),t)|>  \min\{\mathcal{N}_{i+1}^{-\hat{\epsilon}^{-1}},l_{k_*}^{-8}\},$

     then $g_{i+1}(x,t')~satisfies~\min\{\mathcal{N}_{i+1}^{-\hat{\epsilon}^{-2}},l_{k_*}^{-8}\}-\textbf{non-resonant}~condition~on~D_2(t)
\ {\rm and}$\ $$\partial_x g_{i+1}(c_{i+1,1},t)\cdot \partial_y g_{i+1}(c_{i+1,2},t)<0.$$

 \item If $c_{i+1,2}'\geq \tilde{c}^*_{i+1,1}\geq c_{i+1,1},\ c_{i+1,2}\geq \tilde{c}^*_{i+1,2}\geq c_{i+1,1}'$ and $|g_{i+1}(\tilde{c}^*_{i+1,j}(t),t)|\leq  \min\{\mathcal{N}_{i+1}^{-\hat{\epsilon}},l_{k_*}^{-8}\},$ then
    $$|g_{i+1}(\tilde{c}^*_{i+1,j})|>l^{-2}_{k^*}>|g_{i+1}(\tilde{c}_{i+1,j})|.$$
    And
on $I_{i+1,j}-\tilde{I}^*_{i+1,j},$ we have
$$|\partial_x g_{i+1,j}(x,t)|,|g_{i+1,j}(x,t)|> [\min\{\mathcal{N}_{i+1}^{-\hat{\epsilon}^{-1}},l_{k_*}^{-8}\}]^2.$$

\item If $c_{i+1,2}'\leq \tilde{c}_{i+1,1}\leq c_{i+1,1}, c_{i+1,2}\leq \tilde{c}_{i+1,2}\leq c_{i+1,1}'$ and $|g_{i+1}(\tilde{c}^*_{i+1,j}(t),t)|>  \min\{\mathcal{N}_{i+1}^{-\hat{\epsilon}^{-1}},l_{k_*}^{-8}\},$

     then $g_{i+1}(x,t')~satisfies~ \min\{\mathcal{N}_{i+1}^{-\hat{\epsilon}^{-2}},l_{k_*}^{-8}\}-\textbf{non-resonant}~condition~on~D_2(t)
\ {\rm and}$\ $$\partial_x g_{i+1}(c_{i+1,1},t)\cdot \partial_y g_{i+1}(c_{i+1,2},t)<0.$$
\end{enumerate}

\item The $i$th step belongs to $\textbf{II}^{k^*}_i.$ Then all results in case (2) still hold true by replacing $k_*$ by $k^*$.
\end{enumerate}
Moreover we have
\begin{equation}\label{k^*ki} \log |k_*|\geq 2c\hat{\epsilon}^{-1}(k^*)^{\frac{\hat{\epsilon}}{2C_{\alpha}}}.
\end{equation}
\end{lemma}

\begin{proof} We first consider the case $i=s(k_*)$. Then other than \eqref{k^*ki}, one can see that the proof is quite similar to the proof of Lemma \ref{step12}. In fact we only need to consider the induction that from type $\textbf{I}_i$  and $\textbf{II}^{k^*}_i$ to the type $\textbf{I}_{i+1}$  and $\textbf{II}^{k_*}_{i+1}.$ And all the analysis is similar to the induction from step 1 to step 2.

For  \eqref{k^*ki}, one notes the diophantine condition guarantee that
 $$|k_*|>r_{i^*-1}>|\mathcal{N}_{i^*-1}^{2c\hat{\epsilon}^{-1}}|\geq e^{2c\hat{\epsilon}^{-1}q^{\hat{\epsilon}}_{N+i^*-2}}\geq e^{2c\hat{\epsilon}^{-1}q^{\frac{\hat{\epsilon}}{C_{\alpha}}}_{N+i^*-1}}\geq e^{2c\hat{\epsilon}^{-1}(k^*)^{\frac{\hat{\epsilon}}{2C_{\alpha}}}}$$ as desired.\hfill\qed\end{proof}\vskip 0.3cm


\subsubsection{ The proof of Theorem \ref{theorem12}}
Theorem \ref{theorem12} can be obtained directly from Lemma \ref{stepi+11} by setting $k=k_*$ or $k=k^*$ for the case Type $\mathbf{II}^k_{i}$.
\qed

\subsection{Proof of Lemma \ref{thetat1fenjie}}

\begin{proof}
\textbf{The proof of ${\rm i}$:}
We are tasked with proving that the equation
\[
F(x, y) = \tan^{-1}\left(L^2(x, y) \tan h_2(x, y)\right) - \frac{\pi}{2} + h_1(x, y) = 0 \ (\text{mod} \ \pi)
\]
has at most two solutions for any fixed \( y \in \Pi_2 D \). This equation can be rewritten as
\[
T(x, y) := \tan h_1(x, y) \tan h_2(x, y) - L^{-2}(x, y) = 0.
\]

From assumption \eqref{epepep}, we know that for \( i = 1, 2 \),
\[
\max_{(x, y) \in D} |\tan h_i(x, y)| \leq \epsilon^{\frac{2}{3}} \leq \lambda^{-\frac{2}{3}(\log k)^{\hat{\epsilon}^{-1}}} \ll (\log k)^{-10} \leq \Gamma^{-10}.
\]

Next, we analyze the second derivative of \( T(x, y) \) with respect to \( x \).
\[
\frac{\partial^2 T(x, y)}{\partial x^2} = 2(\tan h_1)(\tan h_2)(1 + \tan^2 h_1)(h_1')^2  + (1 + \tan^2 h_1)\tan h_2 h_1''  + 2(1 + \tan^2 h_1)(1+\tan^2 h_2) h_1' h_2'
\]
\[
+ 2(\tan h_2)(\tan h_1)(1 + \tan^2 h_2)(h_2')^2 + (1 + \tan^2 h_2)\tan h_1h_2''  - 6 \frac{L'^2}{L^4} + 2 \frac{L''}{L^3}.
\]
\[:=2(1 + \tan^2 h_1)(1+\tan^2 h_2) h_1' h_2'+\mathcal{R}.
\]

Note that \( |h_1'|, |h_2'|, |h_1'|^{-1}, |h_2'|^{-1} \leq \Gamma \) and $|\tan h_j(x,y)|\leq \epsilon^{\frac{2}{3}}\leq e^{-\frac{2}{3}\Gamma}=o(\Gamma^{-6}),~j=1,2.$ Thus from the assumptions of the lemma, we have
\[\left| \frac{\partial^2 T(x, y)}{\partial x^2} \right| \geq 2 |h_1' h_2'|-\mathcal{R}
\geq \Gamma^{-2} - o(\Gamma^{-6})- \left\vert 6 \frac{L'^2}{L^4}\right\vert - \left\vert 2 \frac{L''}{L^3}\right\vert.
\]
Simplifying further with \( \eqref{lem48-tj111} \) and \( \eqref{lem48-tj} \), we obtain a lower bound for the above
\[
\geq \Gamma^{-2} - o(\Gamma^{-6}) - l^{-2} \cdot e^{|\log \epsilon|^{C}}.
\]
Since \( l = e^k \gg (\log k)^C \geq \Gamma \) and \( e^{|\log \epsilon|^{C}} \leq e^{|\log k|^{C}} \ll e^{Ck} \leq l \), we conclude that
$
\left| \frac{\partial^2 T(x, y)}{\partial x^2} \right| \geq c \Gamma^{-2}.
$

Thus \( T(x, y) \) has at most two zeros on \( D\), which completes the proof of part {\rm i}.

\hfill \(\square\)

\textbf{Proof of ${\rm ii}$:}
Using the fact that \( \Gamma \ll \epsilon^{-1} \), we obtain the following bound
\[
|\tan h_2|^2 + \left| \frac{2 (\partial_X \log L) \tan h_2}{\partial_X h_2} \right| \leq \epsilon^{\frac{4}{3}} + \frac{e^{(\log k)^{C}}\epsilon^{\frac{2}{3}}}{\Gamma^{-1}} \leq \epsilon^{\frac{4}{3}} + \frac{e^{\Gamma^{c}}\epsilon^{\frac{2}{3}}}{\Gamma^{-1}} \leq \epsilon^{\frac{1}{2}} (~\text{note}~\epsilon= e^{-\Gamma}),
\]
which leads to the following relation {on} $D$
\begin{equation}\label{RX}
R_X(x, y) := \partial_X h_2 \left( 1 + \tan^2 h_2 + \frac{2 (\partial_X \log L) \tan h_2}{\partial_X h_2} \right) \sim_{0, \epsilon^{\frac{1}{2}}} \partial_X h_2, \quad X\in\{x,\ y\}.
\end{equation}

Note \eqref{RX} and the fact $a_2,b_2>0$ imply
\begin{equation}\label{RXRXRX}\Gamma^{-1}\leq R_X(x, y)\leq \Gamma \end{equation}

Now let us compute the derivative \( \partial_Y R_X(x, y) \) for \( X, Y \in \{x, y\} \). A direct calculation yields:
\[
\partial_Y R_X(x, y) = 2 \frac{\partial^2 \log L}{\partial X \partial Y} \tan h_2 + 2 (\partial_X \log L + \tan h_2)(1 + \tan^2 h_2) \partial_X h_2 + (1 + \tan^2 h_2) \frac{\partial^2 h_2}{\partial X \partial Y}.
\]
By using the assumptions, we estimate this as follows
\[
\left| \partial_Y R_X(x, y) \right| \leq 2 e^{|\log k|^C} \epsilon^{\frac{2}{3}} + 2(e^{|\log k|^C} + \epsilon^{\frac{2}{3}})(1 + \epsilon^{\frac{4}{3}}) \Gamma + (1 + \epsilon^{\frac{4}{3}}) \Gamma.
\]
By $\epsilon=e^{-\Gamma}$ and $\Gamma=(\log k)^{\hat{\epsilon}^{-1}}$, we obtain
\begin{equation}\label{partialxRX}
\left| \partial_Y R_X(x, y) \right| \leq 100 e^{2|\log k|^{\hat{\epsilon}^{-1}}} \ll e^k = l.
\end{equation}

Next, we compute the first and second derivatives of \( F(x, y) \). By a direct calculation, we have
\begin{equation}\label{parx}
\partial_X F = \frac{L^2 R_X}{1 + L^4 \tan^2 h_2} + \partial_X h_1, \quad X \in \{x, y\},
\end{equation}
and
\begin{equation}\label{parx2}
\begin{aligned}
\frac{\partial^2 F}{\partial X\partial Y} &= \frac{L^2 (2 \partial_Y (\log L)) R_X}{1 + L^4 \tan^2 h_2}
+ \frac{L^2 (\partial_Y R_X)}{1 + L^4 \tan^2 h_2}  - R_X \frac{4L^6 (\partial_Y (\log L)) \tan^2 h_2 + 2(1 + \tan^2 h_2)(\partial_Y h_2) L^6 \tan h_2}{(1 + L^4 \tan^2 h_2)^2} \\
&\quad + \frac{\partial^2 h_1}{\partial X\partial Y}.
\end{aligned}
\end{equation}

 Note $L\leq Cl=Ce^{k}$, $\Gamma\leq (\log k)^{\hat{\epsilon}^{-1}}$ and $\epsilon= e^{-\Gamma}$. Thus by the estimates from \eqref{RXRXRX}, \eqref{partialxRX}, \eqref{parx} and \eqref{parx2}, we obtain
\[
\left| F \right| + \left| \frac{\partial F}{\partial X} \right| + \left| \frac{\partial^2 F}{\partial X \partial Y} \right| \leq \pi + (\Gamma L^2 + \Gamma) + C \Gamma (e^{|\log \epsilon|^{\hat{\epsilon}}}) L^6 \epsilon^{\frac{4}{3}} \leq C l^8.
\]
This implies \eqref{2jiedaos} holds true.

Next, we turn to the estimates for the terms involving \( R_x \) and \( \partial_x h_1 \). Recall that we already have
\begin{equation}\label{RX123}
-\Gamma<\partial_x h_1 \sim_{0, \epsilon^{\frac{1}{2}}} a_1 < -\Gamma^{-1}, \quad \Gamma>\partial_x h_2 \sim_{0, \epsilon^{\frac{1}{2}}} a_2 > \Gamma^{-1}, \quad R_x \sim_{0, \epsilon^{\frac{1}{2}}} \partial_x h_2 = a_2(~by~\eqref{RX}).
\end{equation}
Assume that
\begin{equation}\label{assum}
\left| \partial_x F \right| < \Gamma^{-2}.
\end{equation}
One can obtain from \eqref{parx} and \eqref{RX123} that
\[
\frac{1}{2} \Gamma^{-1} <-\Gamma^{-2}+\Gamma^{-1}< -\Gamma^{-2} - \partial_x h_1 < \frac{L^2 R_x}{1 + L^4 \tan^2 h_2}=\partial_X F-\partial_X h_1 < \Gamma^{-2} - \partial_x h_1 < \Gamma^{-2} + \Gamma < 2 \Gamma.
\]
This yields (by \eqref{RXRXRX} and \eqref{RX123})
\[
\frac{1}{2} l^2 \Gamma^{-2} = \frac{1}{2} \Gamma^{-1} l^2 \Gamma^{-1} \leq \frac{1}{2}\Gamma^{-1} L^2R_x\leq 1 + L^4 \tan^2 h_2\leq 2\Gamma L^2R_x \leq 2 \Gamma l^2 \Gamma \leq 2 \Gamma^2 l^2,
\]
implying that
\begin{equation}\label{yjdth}
\frac{\sqrt{3}}{3} \Gamma^{-1} l^{-1} \leq |\tan h_2| \leq \sqrt{2} \Gamma l^{-1}.
\end{equation}

Next we estimate $\left| \frac{\partial^2 F}{\partial x^2} \right|$ by taking $X=Y=x$ in \eqref{parx2}. By \eqref{yjdth}, we have
\[
\mathcal{M}_1:=\left| R_x \frac{2(1 + \tan^2 h_2)(\partial_x h_2) L^6 \tan h_2}{(1 + L^4 \tan^2 h_2)^2} \right| \geq \frac{2(1 + \frac{1}{3} \Gamma^{-2} l^{-2})(\Gamma^{-1}) l^6 \frac{\sqrt{3}}{3} \Gamma^{-1} l^{-1}}{2(1+ 4\Gamma^4 l^{4})} \Gamma^{-1} \geq \frac{\Gamma^{-7} l}{100} > l^{\frac{1}{2}}.
\]
Similarly for the remaining terms of \eqref{parx2}, we obtain
$$\begin{array}{ll}&\mathcal{M}_2:=\left\vert\frac{L^2(2\partial_x(\log L))R_x}{1+L^4\tan^2h_2}+\frac{L^2(\partial_x R_x)}{1+L^4\tan^2h_2}-R_x\frac{4L^6(\partial_x(\log L))\tan^2h_2}{(1+L^4\tan^2h_2)^2}+\frac{\partial^2 h_1}{\partial x^2}\right\vert \\&\leq
 \left\vert\frac{2l^2(e^{|\log \epsilon|^{\hat{\epsilon}}})\Gamma}{1+l^4\frac{1}{3}\Gamma^{-2} l^{-2}}\right\vert+ \left\vert\frac{l^2100e^{|\log k|^C}}{1+l^4\frac{1}{3}\Gamma^{-2} l^{-2}}\right\vert(by~\eqref{partialxRX})+\left\vert \Gamma\frac{4l^6(e^{|\log \epsilon|^{\hat{\epsilon}}})2\Gamma^{2} l^{-2}}{1+l^8\frac{1}{9}\Gamma^{-4} l^{-4}}\right\vert+\Gamma
 \quad \\&\leq 100\Gamma^{100}(e^{|\log \epsilon|^{\hat{\epsilon}}})(e^{|\log k|^C})\leq 100\Gamma^{100}(e^{200|\log k|^C})\ll e^{\frac{1}{3}k}\leq l^{\frac{1}{2}}.\end{array}$$

Therefore combining these terms gives that if \eqref{assum} holds, then
\begin{equation}\label{erjdxj}
\left| \frac{\partial^2 F}{\partial x^2} \right| \geq \mathcal{M}_1-\mathcal{M}_2\geq \frac{l^{\frac{1}{2}}}{100} \gg \Gamma^2 \geq \Gamma^{-2}.
\end{equation}

Finally, \eqref{erjdxj} shows
\[
\min_{x \in \Pi_1 D} \left( \left| \frac{\partial F}{\partial x} \right| + \left| \frac{\partial^2 F}{\partial x^2} \right| \right) > \Gamma^{-2},
\]
which proves \eqref{fthx2} as required.

Using equation \eqref{parx} and the fact that \( R_X \neq 0 \), which follows from \eqref{RX}, we can write
\[
\begin{aligned}
\partial_t F &= \frac{L^2 R_t}{1 + L^4 \tan^2 h_2} + \partial_t h_1 = \frac{L^2 R_x \frac{R_t}{R_x}}{1 + L^4 \tan^2 h_2} + \frac{R_t}{R_x} \partial_x h_1 - \frac{R_t}{R_x} \partial_x h_1 + \partial_t h_1 \\
&= \frac{R_t}{R_x} \left( \frac{L^2 R_x}{1 + L^4 \tan^2 h_2} + \partial_x h_1 \right) + \partial_t h_1 - \frac{R_t}{R_x} \partial_x h_1 = \frac{R_t}{R_x} \partial_x F + \partial_t h_1 - (\partial_x h_1) \frac{R_t}{R_x}:= A \cdot \partial_x F + B,
\end{aligned}
\]
where
\[
A(x, y) = \frac{R_t}{R_x}, \quad B(x, y) = \partial_t h_1 - (\partial_x h_1) \frac{R_t}{R_x}.
\]

By \eqref{RX}, we obtain the following estimates for \( (x, y) \in D \):
\[
A \sim_{0, \epsilon^{\frac{1}{2}}} \frac{b_2}{a_2}, \quad B \sim_{0, \epsilon^{\frac{1}{2}}} -a_1 \left( \frac{b_2}{a_2} - \frac{b_1}{a_1} \right),
\]
which corresponds to  \eqref{ABAB}. This completes the proof of part {\rm ii}.

\

\textbf{Proof of {\rm v}}

For a fixed \( y \in \Pi_2 D \), we solve the equation \( \partial_x F = 0 \). This is equivalent to solving
$$
(\partial_x F=)L^2 R_x + (1 + L^4 \tan^2 h_2) \partial_x h_1 = 0.
$$
Substituting $R_x$ by \eqref{RX} and divide both sides by $L^2$, the above equation is equivalent to
\begin{equation}\label{shiz121}
((\partial_x h_1) L^2 + \partial_x h_2) \tan^2 h_2 + 2 (\partial_x \log L) \tan h_2 + (\partial_x h_1) L^{-2} + \partial_x h_2 = 0.
\end{equation}

This is a problem involving finding the roots of a quadratic function in \( \tan h_2 \).

By $sgn(a_1)=-sgn(a_2)=-1,$ we observe that
\begin{equation}\label{parx113}- (\partial_x h_1)(\partial_x h_2)(L^2 + L^{-2})\sim_{0,\epsilon^{\frac{1}{2}}}|a_1||a_2|(l^2),\end{equation}
\begin{equation}\label{parx111}
\begin{aligned}
&\frac{(\partial_x (\log L))^2+|\partial_x h_1|^2+|\partial_x h_2|^2}{|a_1||a_2|l^2}
\leq \frac{e^{2|\log \epsilon|^{C}}+|a_1|^2+|a_2|^2}{ \frac{1}{2}\Gamma^{-2}e^{2k}}
\\&\leq \frac{e^{2|\log k|^C}+\Gamma^2}{ \frac{1}{2}\Gamma^{-2}e^{2k}}
\leq \frac{e^{2|\log k|^C}+(\log k)^2}{ \frac{1}{2}\Gamma^{-2}e^{2k}}
\leq e^{-\frac{1}{2}k}\ll \epsilon^{\frac{1}{2}}.
\end{aligned}
\end{equation}

\eqref{parx113} and \eqref{parx111} imply
\begin{equation}\label{delta} (\partial_x (\log L))^2+|\partial_x h_2|^2-|\partial_x h_1|^2- (\partial_x h_1)(\partial_x h_2)(L^2 + L^{-2})\sim_{0,\epsilon^{\frac{1}{2}}}|a_1||a_2|l^2.\end{equation}

Similarly, we also have
\begin{equation}\label{a1a2}-\partial_x (\log L) \pm \sqrt{(\partial_x (\log L))^2 - |\partial_x h_1|^2 - |\partial_x h_2|^2 - (\partial_x h_1)(\partial_x h_2)(L^2 + L^{-2})}\sim_{0,\epsilon^{\frac{1}{2}}}\pm \sqrt{|a_1||a_2|} l\end{equation}
and
\begin{equation}\label{a1a11}(\partial_x h_1) L^2 + \partial_x h_2\sim_{0,\epsilon^{\frac{1}{3}}} a_1l^2.\end{equation}
\eqref{delta} guarantees that the quadratic equation has two different roots. \eqref{a1a2} and \eqref{a1a11} imply the roots of this quadratic, denoted by \( H_{\pm} \), i.e.
\[
H_{\pm} = \frac{-\partial_x (\log L) \pm \sqrt{(\partial_x (\log L))^2 - |\partial_x h_1|^2 - |\partial_x h_2|^2 - (\partial_x h_1)(\partial_x h_2)(L^2 + L^{-2})}}{(\partial_x h_1) L^2 + \partial_x h_2},
\]
satisfy that for any \( (x, y) \in D \),
\begin{equation}\label{H+-}
H_{\pm} \sim_{0,\epsilon^{\frac{1}{2}}}\pm \frac{\sqrt{|a_2|}}{\sqrt{|a_1|}} l^{-1}.
\end{equation}

From \eqref{shiz121}, \( \tan h_2 = 0 \) implies \( \partial_x F =L^2R_x> 0 \) (~recall $R_x\sim_{0,\epsilon^{\frac{1}{2}}} a_2>0$). Thus to solve $\partial_x F=0$, we only need to consider the case $\tan h_2 \neq 0$ and $\partial_x F=0,$ which implies $-G(x,y):=\frac{\partial_x F}{\tan h_2}=0.$
From \eqref{shiz121},
$$~G(x,y)=-[(\partial_x h_1) L^2 + \partial_x h_2] \tan h_2 - 2 (\partial_x \log L)  - [(\partial_x h_1) L^{-2} + \partial_x h_2]\cot h_2=0.$$

Then a direct calculation yields
\begin{equation}\label{abab}
\begin{aligned}
\partial_x G &= ((-\partial_x h_1) L^2 - \partial_x h_2) (1 + \tan^2 h_2) \partial_x h_2 + \partial_x((-\partial_x h_1) L^2 - \partial_x h_2) \tan h_2 \\
&\quad - 2 \frac{\partial^2 \log L}{\partial x^2} + (1 + \cot^2 h_2) \partial_x h_2 ((\partial_x h_1) L^{-2} + \partial_x h_2) + \partial_x((-\partial_x h_1) L^{-2} - \partial_x h_2) \cot h_2 \\
&= [(-\partial_x h_1) L^2 - \partial_x h_2] \partial_x h_2 + \tan h_2 H(x,t) - 2 \frac{\partial^2 \log L}{\partial x^2} + (\partial_x h_2)^2 (1 + \cot^2 h_2) \\
&\quad + L^{-2} M(x,t) - \frac{\partial^2 h_2}{\partial x^2} \cot h_2,
\\&=\left[(-\partial_x h_1)(\partial_x h_2) L^2+(\partial_x h_2)^2\cot^2 h_2\right]+\left[\tan h_2 H(x, t) - 2 \frac{\partial^2 \log L}{\partial x^2} + L^{-2} M(x, t) - \frac{\partial^2 h_2}{\partial x^2} \cot h_2\right],
\end{aligned}
\end{equation}
where
\[
H(x, t) = [(-\partial_x h_1) L^2 - \partial_x h_2] \tan h_2 - \frac{\partial^2 h_1}{\partial x^2} L^2 - 2 L^2 (\partial_x (\log L)) \partial_x h_1 - \frac{\partial^2 h_2}{\partial x^2}
\]
and
\[
M(x, t) = (1 + \cot^2 h_2) \partial_x h_2 (\partial_x h_1) - \frac{\partial^2 h_1}{\partial x^2} + 2 \partial_x h_1 (\partial_x \log L).
\]

Since for \( i = 1, 2 \), \( |\partial_x h_i|, \left| \frac{\partial^2 h_i}{\partial x^2} \right| \leq \Gamma \), \( |\tan h_i| \leq \epsilon^{\frac{2}{3}}\leq Ce^{-\frac{2}{3}\Gamma} \)~(~which~implies~$|\cot h_i|\geq \epsilon^{-\frac{2}{3}}$) and \( |\partial_x (\log L)| \leq e^{(\log k)^C}\leq e^{|\log \epsilon|^{C\hat{\epsilon}}} \), we have
\(
|H(x, y)| \leq 100 \Gamma L^2 e^{|\log \epsilon|^{C\hat{\epsilon}}}
\)
and
\[
\left\vert\frac{M(x, y)}{(\partial_x h_1)(\partial_x h_2)\cot^2 h_2}-1\right\vert \leq \frac{ |(\partial_x h_2)(\partial_x h_1)|+|\frac{\partial^2 h_1}{\partial x^2}|+|2 \partial_x h_1 (\partial_x \log L)|}{(\partial_x h_1)(\partial_x h_2)\cot^2 h_2}\leq \frac{\Gamma^2+\Gamma+2\Gamma e^{|\log \epsilon|^{\hat{\epsilon}}}}{\epsilon^{-\frac{4}{3}}\Gamma^2}\ll \epsilon^{\frac{1}{2}}.
\]
It implies
$$M\sim_{0,\epsilon^{\frac{1}{2}}} (\partial_x h_1)(\partial_x h_2)\cot^2 h_2.$$
Thus we obtain
\[
\begin{array}{ll}&
\left| \frac{\tan h_2 H(x, t) - 2 \frac{\partial^2 \log L}{\partial x^2} + L^{-2} M(x, t) - \frac{\partial^2 h_2}{\partial x^2} \cot h_2}{(-\partial_x h_1)(\partial_x h_2) L^2 + \cot^2 h_2 (\partial_x h_2)^2} \right|\\& \leq C\frac{L^2\Gamma^2 |\tan h_2|+L^{-2}\Gamma^2 |\cot h_2|^2+\Gamma |\cot h_2|}{\Gamma^{-2}L^2+\Gamma^{-2}|\cot h_2|^2}
\\&\leq \epsilon^{\frac{1}{2}}.
\end{array}
\]

Then from \eqref{abab}, we conclude that
\[
\partial_x G \sim_{0,\epsilon^{\frac{1}{2}}} (-\partial_x h_1)(\partial_x h_2) L^2 + \cot^2 h_2 (\partial_x h_2)^2
\sim_{0,\epsilon^{\frac{1}{2}}} |a_1| |a_2| l^2 + |a_2|^2 \cot^2 h_2(x, y) > 0.
\]
Hence
\begin{equation}\label{pg}\partial_x G>0.\end{equation}
On the other hand,
\begin{equation}\label{a222}
\partial_x (\tan h_2) = (1 + \tan^2 h_2) \partial_x h_2 \sim_{0,\epsilon^{\frac{1}{2}}} a_2 > 0,
\end{equation}
and \( [-\epsilon, \epsilon] \subseteq \text{Ran}(\tan h_2) \) implies there exist points \( z_- < z_0 < z_+ \) such that
\[
\tan h_2(z_-) = -\epsilon, \quad \tan h_2(z_0) = 0, \quad \tan h_2(z_+) = \epsilon.
\]
Note $\lim\limits_{\tan h_2\rightarrow 0^{\mp}}G=\pm\infty.$
Thus \( G(z_-) < 0 \), \( \lim\limits_{z\rightarrow z_0^-}G(z) = +\infty \), \( \lim\limits_{z\rightarrow z_0^+}G(z) = -\infty \), and \( G(z_+) > 0 \). Then \eqref{pg} implies for any \( y \in \Pi_2 D \), \( \partial_x F \) has exactly two roots $x^*_1$ and $ x^*_2$ in \( \Pi_1 D \) with \( z_- < x^*_1 < z_0 < x^*_2 < z_+ \).

From equation \eqref{H+-}, we deduce that
\begin{equation}\label{tanh2}
\tan h_2(x^*_1(y), y)\sim_{0,\epsilon^{\frac{1}{2}}}h_2(x^*_1(y), y) \sim_{0,\epsilon^{\frac{1}{2}}} -\frac{\sqrt{a_2}}{\sqrt{|a_1|}} l^{-1};
 \end{equation}
 \begin{equation}\label{tanh2'}
 \tan h_2(x^*_2(y), y)\sim_{0,\epsilon^{\frac{1}{2}}}h_2(x^*_2(y), y) \sim_{0,\epsilon^{\frac{1}{2}}} \frac{\sqrt{a_2}}{\sqrt{|a_1|}} l^{-1}.
\end{equation}

By \eqref{tanh2} and \eqref{tanh2'},
\begin{equation}\label{tanh2''}h_2(x^*_1(y), y)-h_2(x^*_1(y), y)\sim_{0,\epsilon^{\frac{1}{2}}} -2\frac{\sqrt{a_2}}{\sqrt{|a_1|}} l^{-1}.
\end{equation}

Since $\partial_x h_2\sim_{0,\epsilon^{\frac{1}{2}}} a_2,$ \eqref{tanh2''} shows
\begin{equation}\label{x1x2}x^*_1(y)-x^*_2(y)\sim_{0,\epsilon^{\frac{1}{2}}} \frac{-2\frac{\sqrt{a_2}}{\sqrt{|a_1|}} l^{-1}}{a_2}=-2\frac{l^{-1}}{\sqrt{|a_1||a_2|}}.   \end{equation}
Next, let us consider
\[
\begin{array}{ll}
&\left\vert F(x^*_1(y), y) - F(x^*_2(y), y) \right\vert \\
&= \left\vert \tan^{-1}\left( L^2(x^*_1(y), y) \tan h_2(x^*_1(y), y) \right) - \tan^{-1}\left( L^2(x^*_2(y), y) \tan h_2(x^*_2(y), y) \right) \right. \\& \left.+ h_1(x^*_1(y), y) - h_1(x^*_2(y), y) \right\vert.
\end{array}
\]

If we choose the domain of \( \tan^{-1}(\cdot) \) as \( (0, \pi] \), then applying \eqref{tanh2} and \eqref{x1x2} for \( y \in \Pi_2 D \), we have
\[
\tan^{-1}\left(L^{-2}(x^*_1(y), y) \cot h_2(x^*_1(y), y)\right) \sim_{0, \epsilon^{\frac{1}{2}}} \pi - \frac{\sqrt{|a_1|}}{\sqrt{|a_2|}} l^{-1},
\]
\[
\tan^{-1}\left(L^{-2}(x^*_2(y), y) \cot h_2(x^*_2(y), y)\right) \sim_{0, \epsilon^{\frac{1}{2}}} \frac{\sqrt{|a_1|}}{\sqrt{|a_2|}} l^{-1},
\]
\[
h_1(x^*_1(y), y) - h_1(x^*_2(y), y) \sim_{0, \epsilon^{\frac{1}{2}}}a_1(x^*_1(y)-x^*_2(y))\sim_{0,\epsilon^{\frac{1}{2}}} 2\frac{\sqrt{|a_1|}}{\sqrt{|a_2|}} l^{-1}.
\]

Therefore (note \( a_1 < 0 \)) for \( y \in \Pi_2 D \),
$$
\left\vert F(x^*_1(y), y) - F(x^*_2(y), y) \right\vert \sim_{0, \epsilon^{\frac{1}{2}}} \pi - 4 \frac{\sqrt{|a_1|}}{\sqrt{a_2}} l^{-1}.
$$

In other words, for any \( y \in \Pi_2 D \),
\[
\left| \{ F(x, y) \mid x \in \Pi_1 D \} \right| \sim_{0, \epsilon^{\frac{1}{2}}} \pi - 4 \sqrt{\frac{|a_1|}{|a_2|}} l^{-1},
\]
which indeed confirms statement {\rm v}.

\

\textbf{Proof of {\rm vi}:}

Recall that
\begin{equation}\label{F}
F(x,y) = \tan^{-1}\left( L^2(x,y) \tan h_2(x,y) \right) - \frac{\pi}{2} + h_1(x,y) := F(h_2,h_1),
\end{equation}
where we emphasize that \( F \) is a function with respect to \( h_1 \) and \( h_2 \). We denote $
\tilde{F} = F(h_1,h_2).
$

We claim that
\begin{equation}\label{abcdef} \{ y \mid \min\limits_{x} |F| \geq l^{-100C} \}\subset \{ y \mid \min\limits_{x} |\tilde{F}| \geq l^{-200C} \}\end{equation}
 and $$ \{ y \mid \min\limits_{x} |\tilde{F}| \geq l^{-100C} \}\subset \{ y \mid \min\limits_{x} |F| \geq l^{-200C} \}.$$
By symmetry, we only shows \eqref{abcdef}. For fix \( y \in \{ y \mid \min\limits_{x} |F| \geq l^{-100C} \} \), if
$
\min\limits_{x} |\tilde{F}(x,y)| \leq l^{-200C},
$
then there exists some \( |\delta| \leq l^{-200C} \) such that
\[
\{x \mid F(h_1,h_2 + \delta) = 0\} = \{x \mid F(h_1,h_2) + \delta = 0\} = \{x \mid \tilde{F}(x,y) + \delta = 0\} \neq \emptyset.
\]
Note that
\[
0 = F(h_1, h_2 + \delta) = \tan^{-1} \left( L^2(x,y) \tan h_1(x,y) \right) - \frac{\pi}{2} + h_2(x,y) + \delta
\]
is equivalent to
\[
(\tan h_1)(\tan (h_2 + \delta)) = L^{-2},
\]
which is also equivalent to
\[
0 = F(h_2 + \delta, h_1) = \tan^{-1} \left( L^2(x,y) \tan (h_2(x,y) + \delta) \right) - \frac{\pi}{2} + h_1(x,y).
\]

Note \eqref{F} implies $$|\partial_{h_2} F|=\left\vert \frac{L^2(1+\tan^2 h_2)}{1+L^4\tan^2 h_2}\right\vert\leq Cl^2.$$ Then we have
\[
|F(h_2 + \delta, h_1) - F(h_2, h_1)| \leq l^3 \delta.
\]
By the inequality above and the fact that \( |\delta| < l^{-200C} \), we obtain
\[
l^{-150C} < l^{-100C} - l^3 \delta = \min\limits_{x} |F(x,y)| - l^3 \delta = \min\limits_{x} |F(h_2,h_1)| - l^3 \delta \leq \min\limits_{x} |F(h_2 + \delta, h_1)| = 0,
\]
which leads to a contradiction. Therefore \eqref{abcdef} holds true. This completes the proof of {\rm vi}.

\

\textbf{Proof of {\rm iii-a} and {\rm iv}:}

By \eqref{a222}, we observe that for \((x,y) \in D\),
\[
|\tan h_2(x,y) - \tan h_2(x^*_i(y),y)| \leq a_2 |x - x^*_i(y)|.
\]
Therefore for \(x \in \left(x_1^*(y) - \epsilon^{\frac{1}{2}} \frac{l^{-1}}{\sqrt{|a_1||a_2|}}, x_1^*(y) + \epsilon^{\frac{1}{2}} \frac{l^{-1}}{\sqrt{|a_1||a_2|}}\right)\) and \(y \in \Pi_2D\), we have
\[
\tan h_2(x,y) \sim_{0,\epsilon^{\frac{1}{2}}} \tan h_2(x^*_1(y),y) \sim_{0,\epsilon^{\frac{1}{2}}} -\frac{\sqrt{a_2}}{\sqrt{|a_1|}} l^{-1},
\]
and for \(x \in \left(x_2^*(y) - \epsilon^{\frac{1}{2}} \frac{l^{-1}}{\sqrt{|a_1||a_2|}}, x_2^*(y) + \epsilon^{\frac{1}{2}} \frac{l^{-1}}{\sqrt{|a_1||a_2|}}\right)\) and \(y \in \Pi_2D\),
\begin{equation}\label{tan222}
\tan h_2(x,y) \sim_{0,\epsilon^{\frac{1}{2}}} \tan h_2(x^*_2(y),y) \sim_{0,\epsilon^{\frac{1}{2}}} \frac{\sqrt{a_2}}{\sqrt{|a_1|}} l^{-1}.
\end{equation}

For \(x \in \left(x_1^*(y) - \epsilon^{\frac{1}{2}} \frac{l^{-1}}{\sqrt{|a_1||a_2|}}, x_1^*(y) + \epsilon^{\frac{1}{2}} \frac{l^{-1}}{\sqrt{|a_1||a_2|}}\right)\) and \(y \in \Pi_2D\), using \eqref{RX} and direct calculation, we obtain
\begin{equation}\label{h221}
\begin{aligned}
\left| \frac{L^2(2 \partial_x (\log L)) R_x}{1 + L^4 \tan^2 h_2(x,y)} \right| &\leq \frac{l^2 e^{|\log \epsilon|^{{\hat{\epsilon}}}} \Gamma}{l^4 \frac{\sqrt{a_2}}{\sqrt{|a_1|}} l^{-1}} \leq \frac{l^2 e^{|\log \epsilon|^{{\hat{\epsilon}}}} \Gamma}{l^4 \frac{\Gamma^{-1}}{\Gamma} l^{-1}} \leq \Gamma^3 e^{|\log \epsilon|^{{\hat{\epsilon}}}} l^{-1} \leq l^{-\frac{1}{2}} \leq l^{\frac{1}{2}},
\end{aligned}
\end{equation}

\begin{equation}\label{h221'}
\left| \frac{L^2(2 \partial_x (R_x))}{1 + L^4 \tan^2 h_2(x,y)} \right| \leq \frac{l^3}{l^4 \frac{\sqrt{a_2}}{\sqrt{|a_1|}} l^{-1}} \leq \Gamma^2 \leq l^{\frac{1}{2}},
\end{equation}
\begin{equation}\label{h221''}
\frac{\partial^2 h_1(x,y)}{\partial x^2} \leq \Gamma \leq l^{\frac{1}{2}},
\end{equation}
\begin{equation}\label{h221'''}
\left| R_x \frac{4 L^6 (\partial_x (\log L)) \tan^2 h_2(x,y)}{(1 + L^4 \tan^2 h_2(x,y))^2} \right| \leq \Gamma C(e^{|\log \epsilon|^{\hat{\epsilon}}}) l^{-2} \cot^2 h_2(x,y) \leq \Gamma C(e^{|\log \epsilon|^{\hat{\epsilon}}}) \Gamma^2 (\text{by}~\eqref{tan222})\leq l^{\frac{1}{2}},
\end{equation}
\begin{equation}\label{h222}
R_x \frac{2(1 + \tan^2 h_2(x,y)) (\partial_x h_2(x,y)) L^6 \tan h_2(x,y)}{(1 + L^4 \tan^2 h_2(x,y))^2} \sim_{0,\epsilon^{\frac{1}{2}}}\frac{2 a_2^2 l^6 \left( \frac{\sqrt{|a_2|}}{\sqrt{|a_1|}} l^{-1} \right)}{l^8 \left( \frac{\sqrt{|a_2|}}{\sqrt{|a_1|}} l^{-1} \right)^4} = 2 |a_2|^{\frac{1}{2}} |a_1|^{\frac{3}{2}} l.
\end{equation}
Thus from \eqref{h221}-\eqref{h222} we obtain
\begin{equation}\label{l2222}
|\mathcal{L}|:=\left| \frac{\frac{L^2(2 \partial_x (\log L)) R_x}{1 + L^4 \tan^2 h_2} + \frac{L^2(\partial_x R_x)}{1 + L^4 \tan^2 h_2} + \frac{\partial^2 h_1}{\partial x^2} - R_x \frac{4 L^6 (\partial_x (\log L)) \tan^2 h_2}{(1 + L^4 \tan^2 h_2)^2}}{\frac{2 R_x(1 + \tan^2 h_2)(\partial_x h_2) L^6 \tan h_2}{(1 + L^4 \tan^2 h_2)^2}} \right| \leq \frac{4l^{\frac{1}{2}}}{2\Gamma^{-2}l} \leq l^{-\frac{1}{3}} \ll \epsilon^{\frac{1}{2}}.
\end{equation}
On the other hand, from \eqref{parx2} we have
\begin{equation}\label{l3333}
\begin{aligned}
\frac{\partial^2 F}{\partial x^2}(x,y) &= \frac{L^2(2 \partial_x (\log L)) R_x}{1 + L^4 \tan^2 h_2}
+ \frac{L^2(\partial_x R_x)}{1 + L^4 \tan^2 h_2} \\
&\quad - R_x \frac{4 L^6 (\partial_x (\log L)) \tan^2 h_2
+ 2(1 + \tan^2 h_2)(\partial_x h_2) L^6 \tan h_2}{(1 + L^4 \tan^2 h_2)^2}
+ \frac{\partial^2 h_1}{\partial x^2} \\
&= -\frac{2 R_x(1 + \tan^2 h_2)(\partial_x h_2)L^6 \tan h_2}{(1 + L^4 \tan^2 h_2)^2}\ \left(
1 -\mathcal{L}
\right).
\end{aligned}
\end{equation}

Combining \eqref{h222}, \eqref{l2222} and \eqref{l3333}, we conclude that for \(x \in \left(x_1^*(y) - \epsilon^{\frac{1}{2}} \frac{l^{-1}}{\sqrt{|a_1||a_2|}}, x_1^*(y) + \epsilon^{\frac{1}{2}} \frac{l^{-1}}{\sqrt{|a_1||a_2|}}\right)\) and \(y \in \Pi_2D\),
\begin{equation}\label{+++}
\frac{\partial^2 F(x,y)}{\partial x^2} \sim_{0,\epsilon^{\frac{1}{2}}} -2 |a_2|^{\frac{1}{2}} |a_1|^{\frac{3}{2}} l.
\end{equation}
Similarly, for \(x \in \left(x_2^*(y) - \epsilon^{\frac{1}{2}} \frac{l^{-1}}{\sqrt{|a_1||a_2|}}, x_2^*(y) + \epsilon^{\frac{1}{2}} \frac{l^{-1}}{\sqrt{|a_1||a_2|}}\right)\) and \(y \in \Pi_2D\),
$$
\frac{\partial^2 F(x,y)}{\partial x^2} \sim_{0,\epsilon^{\frac{1}{2}}} 2 |a_2|^{\frac{1}{2}} |a_1|^{\frac{3}{2}} l.
$$

This implies that for any \(y \in \Pi_2D\), the function \(F\) has one local maximum and one local minimum. More precisely, \(F(x,y)\) is strictly decreasing with respect to $x$ on \(\Pi_1D - (x_1^*(y), x_2^*(y))\) and strictly increasing on \((x_1^*(y), x_2^*(y))\).

Therefore a direct calculation shows that for any \(y \in \Pi_2D\),
\begin{equation}\label{34}
\left\{ x \mid |\partial_x F(x,y)| \leq \epsilon^{\frac{3}{4}} \right\} = (x^*_1 - \epsilon^*_{1,-} l^{-1}, x^*_1 + \epsilon^*_{1,+} l^{-1}) \cup (x^*_2 - \epsilon^*_{2,-} l^{-1}, x^*_2 + \epsilon^*_{2,+} l^{-1}) := J_1 \cup J_2,
\end{equation}
with
\[
\epsilon^*_{j,\Xi} \sim_{0,\epsilon^{\frac{1}{2}}} \frac{1}{2} |a_2|^{-\frac{1}{2}} |a_1|^{-\frac{3}{2}} \epsilon^{\frac{3}{4}}, \quad j = 1, 2, \, \Xi \in\{+,-\}.
\]

This completes the proof of {\rm iii-a} and {\rm iv}.

\

\textbf{Proof of {\rm iii-b}:}

Recall that
\begin{equation}\label{RXZX}
\partial_X F = \frac{L^2 R_X}{1 + L^4 \tan^2 h_2} + \partial_X h_1, \quad X \in \{x, y\},
\end{equation}
with the following asymptotics:
\[
R_x \sim_{0, \epsilon^{\frac{1}{2}}} \partial_x h_2 \sim_{0, \epsilon^{\frac{1}{2}}} a_2 > 0, \quad R_y \sim_{0, \epsilon^{\frac{1}{2}}} \partial_y h_2 \sim_{0, \epsilon^{\frac{1}{2}}} b_2 > 0,
\]
\[
\partial_x h_1 \sim_{0, \epsilon^{\frac{1}{2}}} a_1 < 0, \quad \partial_y h_1 \sim_{0, \epsilon^{\frac{1}{2}}} b_1 > 0.
\]

Note that for \(x < x_1^*(y)\) or \(x > x_2^*(y)\), \(F(x, y)\) is strictly decreasing.

Hence for \(x > x_2^*(y)\),
$
\partial_x F \leq \partial_x F(x_2^*(y), y) = 0.
$ Therefore for \(x > x_2^*(y)\), using the fact that \(\partial_x h_1 \sim_{0, \epsilon^{\frac{1}{2}}} a_1 < 0\), we have
\[
\partial_x F =  \frac{L^2 R_x}{1 + L^4 \tan^2 h_2} - (-\partial_x h_1)\leq 0.
\]
Thus
\[
|\partial_x F| = \left| \frac{L^2 R_x}{1 + L^4 \tan^2 h_2} - (-\partial_x h_1) \right| = (-\partial_x h_1) - \frac{L^2 R_x}{1 + L^4 \tan^2 h_2}.
\]
Since \(\partial_x h_1 \sim_{0, \epsilon^{\frac{1}{2}}} a_1 < 0\) and $R_x>0$, we obtain
$
|\partial_x F| < -\partial_x h_1\sim_{0, \epsilon^{\frac{1}{2}}} |a_1|.
$
This implies inequality \eqref{iii-2}.

Furthermore  taking $X=y$ in \eqref{RXZX} and by the fact $\partial_x h_2 \sim_{0,\epsilon^{\frac{1}{2}}} a_2>0$ (which implies $\tan^2 h_2(x,,y)>\tan^2 h_2(x_2^*(y),y)$),  we have
\[
\partial_y h_1 < \partial_y F < \frac{l^2 b_2}{1 + l^4 \tan^2 h_2(x_2^*(y), y)} + b_1\sim_{0, \epsilon^{\frac{1}{2}}} \frac{l^2 b_2}{l^4 \frac{a_2}{|a_1|} l^{-2}} + b_1 = b_1 + b_2 \cdot \frac{|a_1|}{a_2}.
\]
This implies inequality \eqref{iii-2'}.
The case \(x < x_1^*(y)\) follows similarly as above.

Thus we complete the proof of {\rm iii-b}.

\

\textbf{Proof of {\rm iii-c}:}

The result of {\rm iii-a} shows that for any \(y \in \Pi_2 D\), the function \(F(x, y)\) is strictly decreasing on \(\left( -\infty, x_1^*(y) \right) \cap \Pi_1 D\), strictly increasing on \(\left( x_1^*(y), x_2^*(y) \right)\), and strictly decreasing on \(\left( x_2^*(y), +\infty \right) \cap \Pi_1 D\).

On the other hand, by result \(\rm i\), we know that \(F = 0 \, (\mathrm{mod} \, \pi)\) has at most two roots, denoted by $\tilde{z}_1(y)$ and $\tilde{z}_2(y)$. Hence we must have one of the following two cases:
\[
x_2^*(y) < \tilde{z}_1(y) \leq x_2^*(y) \leq \tilde{z}_2(y),
\]
or
\[
\tilde{z}_1(y) \leq x_1^*(y) \leq \tilde{z}_2(y) < x_2^*(y).
\]

Without loss of generality, in the following proof, we assume in the following proof that
\begin{equation}\label{assf}
x_2^*(y) < \tilde{z}_1(y) \leq x_2^*(y) \leq \tilde{z}_2(y).
\end{equation}

From \eqref{34}, we have shown that for \(J_i = \left( x_i^* - \epsilon_{i,-}^* l^{-1}, x_i^* + \epsilon_{i,+}^* l^{-1} \right)\), \(i = 1, 2\), the following holds:
\[
\left( \Pi_1 D - J_1 \cup J_2 \right) \subset \left\{ x \mid |\partial_x F| > \epsilon^{\frac{3}{4}} \right\}.
\]
Moreover, by \eqref{+++}, we have for \(x \in J_1 \cup J_2\),
\begin{equation}\label{>>>>}
\frac{\partial^2 F(x, y)}{\partial x^2} \sim_{0, \epsilon^{\frac{1}{2}}} -2 |a_2|^{\frac{1}{2}} |a_1|^{\frac{3}{2}} l.
\end{equation}

Thus for \(x \in J_2=\left( x_2^* - \epsilon_{2,-}^* l^{-1}, x_2^* + \epsilon_{2,+}^* l^{-1} \right)\), we obtain
\[
\left| \partial_x F(x, y) \right| = \left| \partial_x F(x, y) - \partial_x F(x_2^*(y), y) \right| \geq 2 |a_2|^{\frac{1}{2}} |a_1|^{\frac{3}{2}} l |x - x_2^*(y)|\geq \Gamma^{-10} l (x - x_2^*(y)) > \epsilon |x - x_2^*(y)|.
\]

For \(x \in (x_1^* + \epsilon_{1,+}^* l^{-1}, x_2^* - \epsilon_{2,-}^* l^{-1}) \cup \left( (x_2^* + \epsilon_{2,+}^* l^{-1}, +\infty) \cap \Pi_1 D \right)\), by \eqref{>>>>}, we have
\[
\left| \partial_x F(x, y) \right| > \epsilon^{\frac{3}{4}} > \epsilon^{\frac{3}{4}} |x - x_2^*(y)| > \epsilon |x - x_2^*(y)|.
\]

Therefore for \(\tilde{z}_2(y) > x_2^*(y)\) and \(x_1^*(y) \leq \tilde{z}_1(y) \leq x_2^*(y)\), we have
\begin{equation}\label{partialxf}
|\partial_x F(\tilde{z}_i(y), y)| \geq \epsilon |x_2^*(y) - \tilde{z}_i(y)|, \quad i = 1, 2,
\end{equation}
which completes the proof of \(\rm iii-c\).

\

\textbf{Proof of \(\rm iii-d\):}
Recall that $$
\left\vert F(x^*_1(y), y) - F(x^*_2(y), y) \right\vert \sim_{0, \epsilon^{\frac{1}{2}}} \pi - 4 \frac{\sqrt{|a_1|}}{\sqrt{a_2}} l^{-1}.
$$
Without loss of generality, we assume
\[
\min \left\{ \left| F(x_1^*(y), y) \, (\mathrm{mod} \, \pi) \right|, \left| F(x_2^*(y), y) \, (\mathrm{mod} \, \pi) \right| \right\} = \left| F(x_2^*(y), y) \, (\mathrm{mod} \, \pi) \right|.
\]
Denote \(F(x_2^*(y), y) = \eta\).

From \eqref{partialxf}, we have the following
\begin{equation}\label{eta11}
\eta = |\eta - 0| = \left| F(x_2^*(y), y) - F(\tilde{z}_2(y), y) \right| \geq \frac{1}{2} \epsilon |x_2^*(y) - \tilde{z}_2(y)|^2.
\end{equation}

On the other hand, using \eqref{iii-2}, we obtain
\begin{equation}\label{eta22}
\eta = |\eta - 0| = \left| F(x_2^*(y), y) - F(\tilde{z}_2(y), y) \right| \leq |a_2| |x_2^*(y) - \tilde{z}_2(y)| \leq \Gamma |x_2^*(y) - \tilde{z}_2(y)|.
\end{equation}

By \eqref{eta22} and  \eqref{assf} we obtain
\begin{equation}\label{assf1}
|\tilde{z}_2(y) - \tilde{z}_1(y)| >|\tilde{z}_2(y) - x_2^*(y)| \geq \eta \Gamma^{-1}.\end{equation}

By \eqref{eta11} and \eqref{assf}, it holds that
\begin{equation}\label{assf2}
\begin{array}{ll}&|\tilde{z}_2(y) - \tilde{z}_1(y)| \leq |\tilde{z}_2(y) - x_2^*(y)| + |x_2^*(y) - \tilde{z}_1(y)| \leq |\tilde{z}_2(y) - x_2^*(y)| + |x_2^*(y) - x_1^*(y)|\\
&\leq \sqrt{2}\eta^{\frac{1}{2}} \epsilon^{-\frac{1}{2}} + \Gamma^5 l^{-1} \leq 2 \eta^{\frac{1}{2}} \epsilon^{-\frac{1}{2}}.
\end{array}\end{equation}

On the other hand, by \eqref{epepep}, we  have
\begin{equation}\label{assf3}
\eta \Gamma^{-1} \leq |\tilde{z}_2(y) - \tilde{z}_1(y)| \leq |\Pi_1 D| \leq 4 \Gamma^2 \epsilon^{\frac{2}{3}}.
\end{equation}

Therefore we conclude from \eqref{assf1}, \eqref{assf2} and \eqref{assf3} that
$
\eta \Gamma^{-1} \leq |\tilde{z}_2(y) - \tilde{z}_1(y)| \leq \min \left\{ 4 \Gamma^2 \epsilon^{\frac{2}{3}}, 2 \eta^{\frac{1}{2}} \epsilon^{-\frac{1}{2}} \right\}.
$
This completes the proof of \(\rm iii-d\).

\

\textbf{Proof of \(\rm iii-e\):}

Recall the definition $\tilde{I}(y) := \left\{ x \in \Pi_1 D \mid |F(x,y)| < \epsilon^3 l^{-16} e^{-|\log |x^*_2(y) -\tilde{z}_2(y)| |^C} \right\}.$ For convenience, we take $C=8.$ By \eqref{assf} and \eqref{assf3}, \begin{equation}\label{x2*}|x_2^*(y) - \tilde{z}_2(y)|\leq |\tilde{z}_1(y) - \tilde{z}_2(y)|\leq  \Gamma^2 \epsilon^{\frac{2}{3}}(\ll 1).\end{equation}

Set $U:=\left\vert\log |x_2^*(y) - \tilde{z}_2(y)|\right\vert.$ We obtain from \eqref{x2*} that
$$X\geq \frac{2}{3}|\log \epsilon|-2\log \Gamma\geq \frac{2}{3}\Gamma-2\log \Gamma>\frac{1}{3}\Gamma(\gg 1).$$
Hence
$$X^8-2X\geq \frac{1}{2}X^8\geq \frac{1}{3^8}\Gamma^8>0>-2\Gamma+2\ln 2-16 k$$
$$=3\log \epsilon-\log \epsilon-\log 2- 16 \log l=\log (\epsilon^3)-\log (l^{16})-\log (\frac{\epsilon}{2}).$$
Then applying the results from \eqref{eta11}, \eqref{eta22} and above computation, we obtain
\[
\Gamma |x_2^*(y) - \tilde{z}_2(y)| \geq \left| F(x_2^*(y), y) - F(\tilde{z}_2(y), y) \right| \geq \frac{\epsilon}{2} |x_2^*(y) - \tilde{z}_2(y)|^2 \gg \epsilon^3 l^{-16} e^{-|\log |x_2^*(y) - \tilde{z}_2(y)||^8}.
\]
Thus \(\tilde{I}(y)\) consists of two distinct open intervals and
\begin{equation}\label{ced1}
\Gamma^{-1} \epsilon^3 l^{-16} e^{-|\log |x_2^*(y) - \tilde{z}_2(y)||^8} \leq |\tilde{I}(y)| \leq 2 \epsilon l^{-8} e^{-|\log |x_2^*(y) - \tilde{z}_2(y)||^8}.
\end{equation}

Clearly, \(\tilde{z}_i(y) \in \tilde{I}(y)\) for \(i = 1, 2\). From \eqref{iii-3*}, we have
$
\left| \partial_x F(\tilde{z}_i(y), y) \right| \geq \epsilon |x_2^*(y) - \tilde{z}_2(y)|.
$

On the other hand, from \eqref{2jiedaos} and \eqref{ced1}, we obtain for any \(x \in \tilde{I}(y)\) that
\[
\begin{aligned}
&\left| \partial_x F(\tilde{z}_2(y), y) - \partial_x F(x, y) \right| \leq C l^8 |x_2^*(y) - \tilde{z}_2(y)| \\
&\leq C l^8 |\tilde{I}(y)| \leq C l^8 (2 \epsilon l^{-8} e^{-|\log |x_2^*(y) - \tilde{z}_2(y)||^8}) \leq 2 C \epsilon e^{-|\log |x_2^*(y) - \tilde{z}_2(y)||^8}.
\end{aligned}
\]

Thus we have
\[
\begin{aligned}
&\frac{\left| \partial_x F(\tilde{z}_2(y), y) - \partial_x F(x, y) \right|}{\left| \partial_x F(\tilde{z}_i(y), y) \right|} \leq \frac{2 C \epsilon e^{-|\log |x_2^*(y) - \tilde{z}_2(y)||^8}}{\epsilon |x_2^*(y) - \tilde{z}_2(y)|} \\
&\leq e^{-\frac{1}{2} |\log |x_2^*(y) - \tilde{z}_2(y)||^8} = \Gamma^{\frac{1}{2}} \epsilon^{-\frac{3}{2}} l^8 \left( \Gamma^{-\frac{1}{2}} \epsilon^{\frac{3}{2}} l^{-8} e^{-\frac{1}{2} |\log |x_2^*(y) - \tilde{z}_2(y)||^8} \right) \\
&\leq \Gamma^{\frac{1}{2}} \epsilon^{-\frac{3}{2}} l^8 |\tilde{I}(y)|^{\frac{1}{2}} \quad \text{(by \eqref{ced1})} \leq l^9 |\tilde{I}(y)|^{\frac{1}{2}}.
\end{aligned}
\]

Hence for \(x \in \tilde{I}(y)\), we obtain
\[
\partial_x F(x, y) \sim_{0, l^9 |\tilde{I}(y)|^{\frac{1}{2}}} \partial_x F(\tilde{z}_2(y), y),
\]
which gives \eqref{iii-3}.

Similarly, from \eqref{iii-2}, we have
\[
\partial_y F(x, y) > b_1 > \Gamma^{-1} \gg \epsilon > \epsilon |x_2^*(y) - \tilde{z}_2(y)|.
\]
Using the same argument as above, for \(x \in \tilde{I}(y)\), we obtain
\[
\partial_y F(x, y) \sim_{0, l^9 |\tilde{I}(y)|^{\frac{1}{2}}} \partial_y F(\tilde{z}_2(y), y),
\]
which gives \eqref{iii-3'}.
This completes the proof of \(\rm iii-e\).

Then we finish all the proof.
\hfill\qed\end{proof}

\subsection{The proof of Lemma \ref{derivative-t-gn}}

\textbf{The proof of a:}~Notice that $\textbf{I}_{i}\rightarrow \textbf{I}_{i+1}$ and $\textbf{II}_{i}\rightarrow \textbf{I}_{i+1}$ has little impact on the lower bound and upper bound. And the worst case occur in the case $\textbf{I}_{i}\rightarrow \textbf{II}_{i+1}.$ Then \eqref{lm17-main} directly follows from \eqref{partxj*} as desired.

\textbf{The proof of b and c:}
For each $i$, we denote ${n}^{-}_i$  the smallest inductive step $n$ such that step $n$ belongs to Type $\textbf{II}_n^{\hat{k}_i},$ $({n}^-_{i}<){n}^{+}_i(< {n}^-_{i+1})$ by the smallest  inductive step $n$ such that step $n$ belongs to Type $\textbf{I}_n.$ Clearly, step $l$ belongs to Type $\textbf{I}_l$ for any $n^+_{i}\leq l< n^{-}_{i+1}$ and for any $i\in \Z_+.$ And step $m$ belongs to Type $\textbf{II}_m$ for any $n^-_{i}\leq m<n^{+}_{i}.$

By Theorem \ref{theorem12}, all the results hold true for all $l-th$ step,~$n^-_0=n^+_0=1\leq l < n^-_1.$

 We begin from the $(n^+_{2}-1)$-th step. (since the steps between the $(n^-_1)$-th~and the~$(n^+_{2}-2)$-th have the same type as $(n^+_{2}-1),$ only little impact exists on the lower bound and upper bound.)  Let $d_{n^+_2-1}=\|c'_{n^+_2-1,2}-c_{n^+_2-1,1}\|.$ By  Theorem \ref{theorem12} we know that $$\begin{array}{ll}&d_{n^+_2-1}\sim_{0,\Upsilon} \|c_{n^+_2-1,1}+\hat{k}_i\alpha-c_{n^+_2-1,2}\|\\&\sim_{0,\Upsilon} \|c_{n^+_2-1,1}-c'_{n^+_2-1,2}\|\leq \|c_{n^+_2-1,1}-\tilde{c}_{n^+_2-1,2}\|+\|c'_{n^+_2-1,2}-\tilde{c}_{n^+_2-1,2}\|\leq \|c_{n^+_2-1,1}-\tilde{c}_{n^+_2-1,2}\|+\lambda^{-c\hat{k}_1},\end{array}$$
  where $\Upsilon:=\lambda^{-cr_{n^+_2-2}}.$ The Diophantine condition ensures that $\hat{k}_{2}\geq c|{d}_{n^+_2-1}|^{-C}.$ Therefore the above estimate implies that
\begin{equation}\label{(1)-1}\hat{k}_{2}\geq cd_{n^+_2-1}^{-C}\geq c(\|c_{n^+_2-1,1}-\tilde{c}_{n^+_2-1,2}\|+\lambda^{-c\hat{k}_1})^{-C}.\end{equation}

On the other hand, Theorem \ref{theorem12} and \eqref{iii-4} imply that
\begin{equation}\label{(1)-2}\|c_{n^+_2-1,1}-\tilde{c}_{n^+_2-1,2}\|\leq C\lambda^{-c\hat{k}^c_{1}}.
\end{equation}

Combining \eqref{(1)-1} and \eqref{(1)-2},  we obtain (c) as desired.

Next, we consider $n^+_2-th$ step. Note Theorem \ref{theorem12} implies that
$\|\tilde{c}_{l,1}-c_{l,1} \| \sim_{0,\lambda^{-ck_1^c}}\|\tilde{c}_{n^+_2-1,1}-c_{n^+_2-1,1} \|,$ $n^+_2-1\leq l\leq n^-_3-1.$ Then \eqref{iii-3*} yields

\begin{equation}\label{(1)-3}\vert \partial_x g_{n^+_2}(c_{n^+_2,1}(t),t) \vert\geq c(\lambda^{-|\log \hat{k}_1|^C})\|\tilde{c}_{n^+_2,1}-c_{n^+_2,1}\|
\geq c(\lambda^{-|\log \hat{k}_1|^C})\|{c}'_{n^+_2-1,2}-c_{n^+_2-1,1}\|.\end{equation}

Then the first inequality of \eqref{(1)-1}, (c) and \eqref{(1)-3} yield that $\vert \partial_x g_{n^+_2}(c_{n^+_2,1}(t),t) \vert\geq c\hat{k}^{-C}_{2},$ which is the lower bound of $|a_{n,i}|$ in (b). The upper bound of $|a_{n,i}|$ follows from \eqref{iii-2'}. The estimates on $b_{n,i}$ directly follows from \eqref{iii-2'} and the estimate on $|a_{n,i}|.$ \eqref{gnjinsss} directly follows from Theorem \ref{theorem12} since $g_n$ satisfies $\lambda^{-(\log \mathcal{N}_n)^C}$ non-resonant condition.


\subsection{Some useful inequalities}

In this subsection, we introduce some elementary integral inequalities which have been frequently used in the proof of Lemma \ref{lemma 6.1}.

\begin{definition}\label{def-H} Given $\zeta>0,$ compact intervals $I_j\subset \R,~j=1,2,\cdots,M$ and $f(x)\in C^0(\bigcup\limits_{j=1}^M I_j)$ with $Range(f)\subset [-100,100]$ if there exists some function $p(x)\in C^1(\bigcup\limits_{j=1}^M I_j)$ satisfying $\min\limits_{x\in \bigcup\limits_{j=1}^M I_j}\vert p'(x)\vert\geq \zeta$ and $\vert f(x)\vert \geq \vert p(x)\vert,\ ~\forall x\in \bigcup\limits_{j=1}^M I_j,$ then we say $f\in \mathcal{H}(M,\zeta).$
\end{definition}


\begin{lemma}\label{usefullemma2}Given compact intervals $I_j\subset [-1,1],~j=1,2,\cdots,N$ and $0<\gamma_1,\gamma_2\leq 10^{-6}$, if $f\in \mathcal{H}(N,\gamma_1)$, then we have $$\int_{\bigcup\limits_{j=1}^N I_j}\frac{1}{\sqrt{f^2+\gamma_2}}\leq CN\gamma_1^{-1}\vert \log \gamma_2\vert.$$\end{lemma}
\begin{proof} We only prove the case $N=1$ and the case $N>1$ can be similarly proved. By our assumption we have $|p(x)|\geq \gamma_1x+c$ with some suitable constant $|c|\leq 100.$ Clearly, $\frac{1}{\sqrt{x^2+\gamma_2}}$ is monotonically decreasing as $x\rightarrow +\infty.$
Note for $10^{-6}>B>0,|x|<200$, we have $$\begin{array}{ll}&\vert\log (\sqrt{x^2+B}+x) \vert\leq \max\{\vert\log (\sqrt{40000+B}-200) \vert, \vert\log (\sqrt{40000+B}+200) \vert \}\\&\leq \max\{\vert\log \frac{B}{800} \vert, 200 \}\leq 10^{-10} |\log B|. \end{array}$$

One also notes that for $a, b>0$, it holds that $\int \frac{1}{\sqrt{(ax+c)^2+b}}=\frac{ \log \left( \sqrt{(ax+c)^2+b}+ax+c \right)}{a}+C.$
Then by a direct calculation, we have $$\begin{array}{ll}\int_{I}\frac{1}{\sqrt{f^2+\gamma_2}}&\leq \int_{I}\frac{1}{\sqrt{p^2+\gamma_2}}\leq  \int_{I}\frac{1}{\sqrt{(\gamma_1x+c)^2+\gamma_2}}\leq C\gamma_1^{-1}\vert \log \gamma_2\vert .\hfill\qed\end{array}$$
\end{proof}

\begin{lemma}\label{useful111} For $a,b\in \R$ satisfying $0<|a|,b\ll 1,$ it holds  that
$$\int_{\R}\frac{1}{|x^2-a|+b} dx\leq C\frac{\vert\log b\vert}{\sqrt{|a|+b}}.$$
\end{lemma}
\begin{proof}

\begin{enumerate}

\item For the case $a\leq 0,$  we have $$\int_{\R}\frac{1}{|x^2-a|+b} dx=\int_{\R}\frac{1}{x^2+(-a+b)} dx=\frac{\pi}{\sqrt{b+|a|}}<\frac{\pi\cdot |\log b|}{\sqrt{b+|a|}}.$$

\item For the case $a\leq 0,$  we have $$\int_{\R}\frac{1}{|x^2-a|+b} dx\leq \int_{(-\sqrt{a},\sqrt{a})}\frac{1}{|x^2-(a+b)|} dx+\int_{\R-(-\sqrt{a},\sqrt{a})}\frac{1}{|x^2+(-a+b)|} dx:=INT_1+INT_2.$$

On one hand, by a direct calculation and the fact $\sqrt{b}<\sqrt{a}+\sqrt{a+b}\ll 1$, we have \begin{equation}\label{intt1}INT_1=\frac{\vert\log \vert\frac{\sqrt{a}+\sqrt{a+b}}{\sqrt{a}-\sqrt{a+b}}\vert \vert}{\sqrt{a+b}}=\frac{2\vert\log \vert\sqrt{a}+\sqrt{a+b}\vert \vert+\vert\log b\vert}{\sqrt{a+b}}\leq \frac{2\vert\log \vert\sqrt{b}\vert \vert+\vert\log b\vert}{\sqrt{a+b}}\leq 3\frac{\vert\log b\vert}{\sqrt{a+b}}.\end{equation}

On the other hand, for $INT_2$ we have to consider the following three cases:

\begin{enumerate}

\item If $0<a<b,$ then $$INT_2=2\frac{\vert \frac{\pi}{2}-\arctan (\frac{\sqrt{a}}{\sqrt{b-a}})\vert}{\sqrt{b-a}}.$$ Then by  $$\frac{\vert \frac{\pi}{2}-\arctan (\sqrt{\frac{1}{x-1}})\vert}{\sqrt{1-\frac{2}{x+1}}}\leq \frac{2}{\pi},\  x>1$$ and taking $x=\frac{b}{a}$, we immediately obtain $$INT_2\leq \frac{\frac{4}{\pi}}{\sqrt{a+b}}< \frac{100\vert \log b\vert}{\sqrt{a+b}}.$$

\item If $a=b,$ then $INT_2=\frac{2}{\sqrt{a}}=\frac{2\sqrt{2}}{\sqrt{a+b}}.$

\item If $a>b,$ then $$INT_2=\frac{\vert \log \vert\frac{\sqrt{a}-\sqrt{a-b}}{\sqrt{a}+\sqrt{a-b}}\vert\vert}{\sqrt{a-b}}.$$ We claim that \begin{equation}\label{int22}INT_2\leq \min\{\frac{3\vert \log b\vert}{\sqrt{a-b}},\frac{2}{\sqrt{a}-\sqrt{a-b}}\}.\end{equation} In fact, on one hand, the fact $1\geq \sqrt{a}+\sqrt{a-b}\geq \sqrt{a}>\sqrt{b}$ yields that $$\vert \log \vert\frac{\sqrt{a}-\sqrt{a-b}}{\sqrt{a}+\sqrt{a-b}}\vert\vert\leq \vert 2\log \vert \sqrt{a}+\sqrt{a-b} \vert \vert+ \vert \log b\vert \leq 3\vert\log b\vert,$$ therefore $INT_2\leq \frac{3\vert \log b\vert}{\sqrt{a-b}}.$ On the other hand, the fact $\log (1+\frac{2\sqrt{a-b}}{\sqrt{a}-\sqrt{a-b}})\leq \frac{2\sqrt{a-b}}{\sqrt{a}-\sqrt{a-b}}$ yields that $$INT_2\leq \frac{\vert \log (1+\frac{2\sqrt{a-b}}{\sqrt{a}-\sqrt{a-b}})\vert}{\sqrt{a-b}}\leq \frac{\frac{2\sqrt{a-b}}{\sqrt{a}-\sqrt{a-b}}}{\sqrt{a-b}}=\frac{2}{\sqrt{a}-\sqrt{a-b}}.$$

By \eqref{int22}, we have $$INT_2\leq \frac{3\vert \log b\vert}{\sqrt{a-b}}\leq \frac{3\sqrt{2}\vert \log b\vert}{\sqrt{a}}= \frac{6\vert \log b\vert}{\sqrt{a+a}}< \frac{6\vert \log b\vert}{\sqrt{a+b}}$$ for $a>2b$ and $$INT_2\leq \frac{2}{\sqrt{a}-\sqrt{a-b}}=\frac{2(\sqrt{a}+\sqrt{a+b})}{b}\leq \frac{2(1+\sqrt{2})\sqrt{a}}{\frac{a}{2}}\leq \frac{8\sqrt{2}}{\sqrt{a}}\leq \frac{16}{\sqrt{a+b}}<\frac{16\vert \log b\vert}{\sqrt{a+b}}$$ for $b<a\leq 2b$. Therefore \begin{equation}\label{intt2}INT_2\leq \frac{16\vert \log b\vert}{\sqrt{a+b}}.\end{equation}
\end{enumerate}
\end{enumerate}
Finally, \eqref{intt1} and \eqref{intt2} complete the proof.
\hfill\qed\end{proof}

\begin{lemma}\label{useful221221} For $a,b,c\in \R,$ it holds that \begin{equation}\label{usef233}\frac{|a-b|}{2}\min\{|a-c|,|b-c|\}\leq |a-c||b-c|\end{equation}and
\begin{equation}\label{usef233'}|c-\frac{a+b}{2}|\leq \min\{|a-c|,|b-c|\}+\frac{|b-a|}{2}.\end{equation}
\end{lemma}

\begin{proof} The case $a=b$ is trivial. For $a\neq b,$ without loss of generality, we can assume that
$a<b$ and $c\geq \frac{a+b}{2}.$
One notes that
$$|c-a||b-c|>|\frac{a+b}{2}-a||b-c|=\frac{b-a}{2}|b-c|=\frac{|b-a|}{2}\min\{|b-c|,|a-c|\},$$ which yields \eqref{usef233}.

One also notes
$$|c-\frac{a+b}{2}|\leq |c-a|+|a-\frac{a+b}{2}|=|c-a|+\frac{b-a}{2}=\min\{|a-c|,|b-c|\}+\frac{|b-a|}{2},$$ which obtains \eqref{usef233'}.
\hfill\qed\end{proof}

\begin{lemma}\label{AZH}Given $1\gg \delta_1,\gamma >0,$ $\bar{\delta}_2(x)\in C^2((-\gamma,\gamma))$ satisfying
$|\bar{\delta}_2(0)|=\delta_2,$
$\delta_1\geq \max\limits_{x\in I}|\bar{\delta}_2(x)|,$
$\gamma\geq (\delta_1)^{\frac{1}{40000}}$
 and $ \vert\bar{\delta}_2'(x)\vert\leq \delta_1^{1-\frac{1}{10000}},$  it holds that
$$\left\vert\int_{0}^{\gamma}\frac{x dx}{\sqrt{x^4+2\delta^4_1 x^2+\bar{\delta}_2^8(x)}}-\frac{x dx}{\sqrt{x^4+2\delta^4_1 x^2+\delta_2^8}}\right\vert\leq C\delta_1^{\frac{1}{4}}.$$\end{lemma}

\begin{proof} Denote $a=\delta_1^4$. Then from $\vert \bar{\delta}_2(x)-\delta_2\vert\leq \delta_1^{1-\frac{1}{10000}}|x|$ and $\delta_1\geq \max\limits_{x\in I}|\bar{\delta}_2(x)|\geq \delta_2$, we have that $$|\bar{\delta}^8_2(x)-\delta_2^8|\leq 8|\tilde{\delta}_2|^7|\tilde{\delta}_2'||x|\leq 8\delta_1^{8-\frac{1}{10000}}|x|\leq 8a^{2-\frac{1}{40000}}|x|.$$ Note \begin{equation}\label{a14}\gamma>\delta_1=a^{\frac{1}{160000}}.\end{equation} It holds that
\begin{equation}\label{a141}
\begin{array}{ll}&|\int_{0}^{\gamma}\frac{x dx}{\sqrt{x^4+2a x^2+\bar{\delta}^8_2(x)}}-\frac{x dx}{\sqrt{x^4+2a x^2+\delta_2^8}}|
\leq |\int_{0}^{\gamma}\frac{x dx}{\sqrt{x^4+2a x^2}}-\frac{x dx}{\sqrt{x^4+2a x^2+|\bar{\delta}_2^8-\delta_2^8|}}|\\&
\leq |\int_{0}^{a^{\frac{3}{4}}\gamma}\frac{x dx}{\sqrt{x^4+2a x^2}}-\frac{x dx}{\sqrt{x^4+2a x^2+(8a^{2-\frac{1}{40000}})x}}|+|\int_{a^{\frac{3}{4}}\gamma}^{\gamma}\frac{x dx}{\sqrt{x^4+2a x^2}}-\frac{x dx}{\sqrt{x^4+2a x^2+8(a^{2-\frac{1}{40000}})x}}|\\& := \bar{P}_1+\bar{P}_2
\end{array}
\end{equation}

On one hand, by a direct calculation we have \begin{equation}\label{a412}\begin{array}{ll}&|\bar{P}_1|\leq \int_{0}^{a^{\frac{3}{4}}\gamma}|\frac{x dx}{\sqrt{x^4+2a x^2}}|+\int_{0}^{a^{\frac{3}{4}}\gamma}|\frac{x dx}{\sqrt{x^4+2a x^2+(8a^{2-\frac{1}{40000}})x}}|\leq \int_{0}^{a^{\frac{3}{4}}\gamma}|\frac{x dx}{\sqrt{2a x^2}}|+\int_{0}^{a^{\frac{3}{4}}\gamma}|\frac{x dx}{\sqrt{(8a^{2-\frac{1}{40000}})x}}|\\&\leq C(a^{\frac{1}{4}}\gamma)+(a^{-1+\frac{1}{80000}})\int_{0}^{a^{\frac{3}{4}}\gamma}x^{\frac{1}{2}}dx\leq C(a^{\frac{1}{4}}\gamma+a^{\frac{1}{8}+\frac{1}{80000}}\gamma^{\frac{3}{2}})\ (note~0<a,\gamma\ll 1)\leq C a^{\frac{1}{16}}\leq C \delta_1^{\frac{1}{4}}.\end{array}\end{equation}

On the other hand, if $x\geq a^{\frac{3}{4}}\gamma,$ then by \eqref{a14} we have $$\begin{array}{ll}&0<x^4+ax^2+a^{2-\frac{1}{40000}}x=(x^4+ax^2)(1+\frac{a^{2-\frac{1}{40000}}}{x^3+ax})\leq (x^4+ax^2)(1+\frac{a^{2-\frac{1}{40000}}}{(a^{\frac{3}{4}}\gamma)^3+a^{\frac{7}{4}}\gamma })\\&\leq (x^4+ax^2)(1+\frac{a^{2-\frac{1}{40000}}}{(a^{\frac{3}{4}+\frac{1}{160000}})^3+a^{\frac{7}{4}+\frac{1}{160000}} })\leq (1+a^{\frac{1}{8}})(x^4+ax^2).\end{array}$$

Therefore on ${[a^{\frac{3}{4}}\gamma,\gamma]}$, it holds that $\sqrt{x^4+ax^2+a^{2-\frac{1}{40000}}x}\sim_{0,a^{\frac{1}{8}}}\sqrt{x^4+ax^2}.$

Hence
\begin{equation}\label{a4122}|\bar{P}_2|\leq a^{\frac{1}{8}}\cdot \int_{0}^{1}\frac{x dx}{\sqrt{x^4+2a x^2}}\leq a^{\frac{1}{8}}\cdot C |\log a|\leq C a^{\frac{1}{16}}\leq C\delta_1^{\frac{1}{4}}.\end{equation}

By \eqref{a141}, \eqref{a412} and \eqref{a4122}, we complete the proof.
\hfill\qed\end{proof}

\

\begin{lemma}\label{usefullemma1}Given a compact interval $I\subset \R/\Z,$ $f(x),\bar{\delta}_2(x)\in C^2(I)$  and \begin{equation}\label{dammm}M\geq \gamma^{-10^5}\gg 1\gg \delta_1>\max\limits_{x\in I}|\bar{\delta}_2(x)|,\end{equation} we assume that $\vert\frac{d\bar{\delta}_2}{dx}\vert\leq (\vert \log \delta_1\vert^{\left(\log\vert \log \delta_1\vert\right)^C})\delta_1,$  \begin{equation}\label{f'''}f''(x)\sim_{0,\gamma} 2M\end{equation} and $$M|I|^2\geq \gamma^{10^{-3}}.$$ Suppose $x^*$ is the center of $I$ satisfying $f'(x^*)=0$ and \begin{equation}\label{etaeta111}\frac{(\vert \log \delta_1\vert^{\left(\log\vert \log \delta_1\vert\right)^C})\delta^2_1}{|f(x^*)|}:=\eta\leq \gamma^{10}.\end{equation}
Let $J_{\gamma}:=\{x\in I \vert |f|\geq \gamma^{\frac{1}{4}}\sqrt{\frac{f(x^*)}{M}}\}.$ Then the following holds true.

\begin{enumerate}
\item[i:] If $\vert f(x^*)\vert> \gamma^{\frac{1}{4}} M^{-1}$, then we have \begin{equation}\label{jiandangujia}\int_{I}\left\vert\frac{f}{\sqrt{f^4+2\delta^4_1 f^2+\bar{\delta}^8_2(x)}}\right\vert dx< C\gamma^{-100}\cdot \vert \log \delta_1 \vert\end{equation}
and
    \begin{equation}\label{jiandangujia1}\int_{J_{\gamma}}\frac{dx}{|f|}< C\gamma^{-100}\log M.\end{equation}

\item[ii:] If $\vert f(x^*)\vert\leq \gamma^{\frac{1}{4}} M^{-1}$, then the following holds true.
\begin{enumerate}
\item[$ii_a:$] If $\{x\in I\vert f(x)=0\}\neq \emptyset,$ then
 \begin{equation}\label{useeq2}\left\vert\int_{I}\frac{f dx}{\sqrt{f^4+2\delta^4_1 f^2+\bar{\delta}^8_2(x)}}\right\vert< C\gamma^{\frac{1}{10}}(M\cdot |f(x^*)|)^{-\frac{1}{2}};\end{equation}
\begin{equation}\label{useeq3}\left\vert\int_{J_{\gamma}}\frac{dx}{f}\right\vert< C\gamma^{\frac{1}{10}}(M\cdot |f(x^*)|)^{-\frac{1}{2}};\end{equation}
\begin{equation}\label{useeq4}\int_{J_{\gamma}}\frac{dx}{|f|}< C(\big|\log \gamma|+|\log (M\cdot |f(x^*)|\big|)(M\cdot |f(x^*)|)^{-\frac{1}{2}}.\end{equation}

  \item[$ii_b:$] If  $\{x\in I\vert f(x)=0\}= \emptyset,$ then  \begin{equation}\label{useeq2*}\left\vert\int_{I}\frac{fdx}{\sqrt{f^4+2\delta^4_1 f^2+\bar{\delta}^8_2(x)}}\right\vert\sim_{0, \gamma^{\frac{1}{10}}} \pi \cdot(M\cdot |f(x^*)|)^{-\frac{1}{2}}.\end{equation}

  \end{enumerate}
  \end{enumerate}\end{lemma}

 \begin{proof} For some fixed $w\in I,$ let $P(f,w):=\frac{f}{\sqrt{Q(f,w)}}$ and $Q(f,w):=f^4+2\delta^4_1 f^2+\tilde{\delta}^8_2(w).$

  \textbf{The proof of $ii_a$:}~In this case, the assumption implies \begin{equation}\label{assumpt00}0>f(x^*)\geq -\gamma^{\frac{1}{4}} M^{-1}.\end{equation} Note \eqref{f'''} and \eqref{dammm} imply that \begin{equation}\label{liangbbb}(1-\gamma)M(x-x^*)^2+f(x^*)<f(x)\leq (1+\gamma)M(x-x^*)^2+f(x^*).\end{equation} The fact $M|I|^2\geq \gamma^{10^{-3}}(\gg \gamma^{\frac{1}{4}} M^{-1})$ implies \begin{equation}\label{yibianbbb}M(\frac{|I|}{2})^2>\frac{1}{4}\gamma^{10^{-3}}\gg \gamma^{10^4}\geq \gamma^{\frac{1}{4}} M^{-1}>-f(x^*).\end{equation}
  Hence $Range(f(x))\subset (f(x^*),a)$ for some $a\gg -f(x^*)>0.$

  Hence there exist two distinct zeros $x_1< x^*< x_2$ on $I$ (i.e. $f(x_1)=f(x_2)=0$).

  Now we define $$x^*-x_1=l_1,~x_2-x^*=l_2 ~f'(x_i)=d_i,~i=1,2.$$ By \eqref{liangbbb} we have \begin{equation}\label{di}l_1\sim_{0,\gamma^{\frac{1}{2}}} l_2;~d_i\sim_{0, \gamma^{\frac{1}{2}}} 2Ml_i\end{equation}
  and\begin{equation}\label{fx*}|f(x^*)|\sim_{0, \gamma^{\frac{1}{2}}} M(l_i)^2 \sim_{0, \gamma^{\frac{1}{2}}} M(\frac{l_1+l_2}{2})^2,~i=1,2.\end{equation}

  By \eqref{yibianbbb} and \eqref{fx*} we have
  \begin{equation}\label{qujianbianjie}|I|\geq M^{-\frac{1}{2}} \gamma^{\frac{1}{2000}}\geq \gamma^{-\frac{1}{3}} M^{-\frac{1}{2}}\gamma^{\frac{1}{2}}\geq \gamma^{-\frac{1}{3}}\sqrt{\vert f(x^*)\vert}\geq \gamma^{-\frac{1}{3}}M^{\frac{1}{2}}\frac{(l_1+l_2)}{2}\geq \frac{1}{2}\gamma^{-\frac{1}{4}}M^{\frac{1}{2}} l_i,~i=1,2.\end{equation}
 Let $I_i$ be the interval with $x_i$ as the center satisfying $|I_i|=\gamma^{\frac{1}{4}} l_i,~i=1,2.$  Then {\rm\ on\ } ${I_i}$ it holds that \begin{equation}\label{fgfgfg}|f(x)-d_i(x-x_i)|\leq M(x-x_i)^2,\qquad~f(x)\sim_{0,\gamma^{\frac{1}{2}}} d_i(x-x_i)\end{equation}and
$$~f(x)\sim_{0,\gamma^{\frac{1}{2}}} d_i(x-x_i)\ .
$$

  Then a direct calculation yields that (note $(f\cdot d_i(x-x_i))>0,~x\in I_i$)

  \begin{equation}\label{zhgj1}\begin{array}{ll}&\vert\sqrt{Q(f,x_i)}\sqrt{Q(d_i(x-x_i),x_i)}(f\sqrt{Q(d_i(x-x_i),x_i)}+d_i(x-x_i)\sqrt{Q(f,x_i)})\vert \\&\geq \vert fQ(d_i(x-x_i),x_i)\sqrt{Q(f,x_i)}\vert+\vert (d_i(x-x_i))Q(f,x_i)\sqrt{Q(d_i(x-x_i),x_i)}\vert \\&\geq c\left(|d_i(x-x_i)|\cdot |d_i(x-x_i)|^4\cdot |d_i(x-x_i)|^2+|d_i(x-x_i)|\cdot\tilde{\delta}_2^8(x_i)\cdot|d_i(x-x_i)|^2\right)\\&=c\left(\vert d_i\vert^7|x-x_i|^7+\tilde{\delta}_2^{8}(x_i)|d_i|^3\vert x-x_i\vert^3\right);\end{array}\end{equation}

   \begin{equation}\label{zhgj2}\begin{array}{ll}&\vert f^2Q(d_i(x-x_i),x_i)-d_i^2(x-x_i)^2Q(f,x_i)\vert
   \\&=\vert f^2Q(f,x_i)-d_i^2(x-x_i)^2Q(f,x_i)-[f^2Q(f,x_i)-f^2Q(d_i(x-x_i),x_i)]\vert
   \\&= \vert \left[f^2d_i^2(x-x_i)^2+\delta_2^8\right](d_i(x-x_i)+f)(d_i(x-x_i)-f) \vert
   \\& \leq \left\vert f^2d_i^2(x-x_i)^2\cdot (d_i(x-x_i)-f)(d_i(x-x_i)+f)\right\vert+\left\vert\tilde{\delta}_2^8(x_i)(d_i(x-x_i)-f)(d_i(x-x_i)+f) \right\vert
   \\&\leq C\left\vert \vert d_i (x-x_i)\vert^5\cdot M\cdot \vert x-x_i\vert^2+\tilde{\delta}_2^8(x_i) |d_i(x-x_i)|\cdot M\cdot \vert x-x_i\vert^2\right\vert(~\text{by}~\eqref{fgfgfg})
   \\&= C\left\vert\vert d_i \vert^5\cdot\vert (x-x_i)\vert^7\cdot M+\tilde{\delta}_2^8(x_i)\cdot\vert d_i\vert\cdot\vert (x-x_i)|^3\cdot M\right\vert
   \\&\leq C\cdot \frac{M}{|d_i|^2}\cdot \left(\vert d_i\vert^7|x-x_i|^7+\tilde{\delta}_2^{8}(x_i)|d_i|^3\vert x-x_i\vert^3\right).
   \end{array}\end{equation}

Combining \eqref{zhgj1} and \eqref{zhgj2} we immediately have
  \begin{equation}\label{cfgj}\begin{array}{ll}\vert P(f,x_i)-P(d_i(x-x_i),x_i)\vert &= \vert\frac{f^2Q(d_i(x-x_i),x_i)-d_i^2(x-x_i)^2Q(f,x_i)}{\sqrt{Q(f,x_i)}\sqrt{Q(d_i(x-x_i),x_i)}(f\sqrt{Q(d_i(x-x_i),x_i)}+(d_i(x-x_i))\sqrt{Q(f,x_i)})} \vert\leq C\left\vert\frac{M}{d_i^2}\right\vert.\end{array}\end{equation}
  On the other hand,
  since $d_i(x-x_i)$ is odd on $I_i,$ we have \begin{equation}\label{Pdi}\int_{I_i} P(d_i(x-x_i),x_i) dx=0,~i=1,2.\end{equation}

  Note that \eqref{di}, \eqref{assumpt00} and \eqref{fx*} imply
  \begin{equation}\label{mli}\gamma^{\frac{1}{4}}\frac{Ml_i}{d_i^2}\sim_{0,\gamma^{\frac{1}{2}}}\gamma^{\frac{1}{4}}\frac{1}{Ml_i}>c\gamma^{\frac{1}{4}}\frac{1}{\sqrt{|f(x^*)|M}}>c\gamma^{\frac{1}{8}}>\gamma^{\frac{5}{2}}.\end{equation}

  \eqref{etaeta111} and \eqref{fx*} lead to
  \begin{equation}\label{f*>0} M(l_i)^2\sim_{0,\gamma^{\frac{1}{2}}}|f(x^*)|>\delta_1^{\frac{5}{2}}\gamma^{-10}>0\end{equation} and
  \begin{equation}\label{gamma>eta} \gamma>\eta^{\frac{1}{10}}>\frac{\delta_2^{\frac{1}{10}}}{|f(x^*)|^{\frac{1}{10}}}\geq \frac{\delta_1^{\frac{1}{10}}}{M^{-\frac{1}{10}}\gamma^{
  \frac{1}{40}}}>\delta_1^{\frac{1}{10}}(~by~M>1>\gamma).\end{equation}
  Note that \eqref{f*>0} guarantees that $d_i\neq 0.$
  Then for $i=1,2,$ \begin{equation}\label{eqfinal1*}\begin{array}{ll}&\left\vert\int_{I_i} P(f,x) dx \right\vert\leq \left\vert\int_{I_i} P(f,x)-P(f,x_i) dx \right\vert +\left\vert\int_{I_i} P(f,x_i) dx \right\vert\\& \leq \left\vert\int_{I_i} P(f,x)-P(f,x_i) dx \right\vert +\left\vert\int_{I_i} P(f,x_i)-P(d_i(x-x_i),x_i) dx \right\vert+\left\vert\int_{I_i} P(d_i(x-x_i),x_i) dx \right\vert
  \\&\leq C\delta_1^{\frac{1}{4}}+C|I_i| \left\vert\frac{M}{d_i^2}\right\vert(~\text{by}~\eqref{cfgj}, \eqref{Pdi}~\text{and~Lemma}~\ref{AZH})
  \end{array}
  \end{equation}

  Then \eqref{eqfinal1*} implies
  \begin{equation}\label{eqfinal1}\begin{array}{ll}&\left\vert\int_{I_i} P(f,x) dx \right\vert
  \leq \gamma^{\frac{5}{2}}+C\gamma^{\frac{1}{4}} l_i M d_i^{-2}(~by~\eqref{gamma>eta})
  \leq 2C\gamma^{\frac{1}{4}} l_i M d_i^{-2}(~by~\eqref{mli})\\
  \\&<C\gamma^{\frac{1}{4}} M^{-1}(l_1+l_2)^{-1}(~by~\eqref{di}~and~\eqref{fx*})
  =C \gamma^{\frac{1}{4}} (M(l_1+l_2))^{-1}
  \leq C\gamma^{\frac{1}{10}}(M\cdot |f(x^*)|)^{-\frac{1}{2}}(~\text{by}~\eqref{fx*}).\end{array}\end{equation}

  Now we consider $I-\left(I_1\bigcup I_2\right).$ Note that for $x\in I-\left(I_1\bigcup I_2\right),$
  \begin{equation}\label{I1I1I22} |M(x-x_1)(x-x_2)|\geq \gamma^{\frac{1}{4}}M l_1^2\geq c\gamma^{\frac{1}{4}}|f(x^*)|
  \end{equation} and
  \begin{equation}\label{I1I1I22*} \min\{|x-x_1|,|x-x_2|\}\geq \gamma^{\frac{1}{4}}\min\{l_1,l_2\}.
  \end{equation}
  And one also notes \eqref{di} and \eqref{I1I1I22*} imply
 $$|l_1-l_2|\leq \gamma^{\frac{1}{2}}\min\{l_1,l_2\}\leq \gamma^{\frac{1}{4}}\min\{|x-x_1|,|x-x_2|\}
$$
and
\begin{equation}\label{**112*}\begin{array}{ll}&\left\vert \frac{(x-\frac{x_2+x_1}{2})^2-(x-x^*)^2}{(x-x_1)(x-x_2)} \right\vert=\left\vert \frac{(x-\frac{x_2+x_1}{2})^2-(x-x^*)^2}{(x-x_1)(x-x_2)} \right\vert=\left\vert (x^*-\frac{x_1+x_2}{2})\frac{2x-\frac{x_1+x_2}{2}-x^*}{(x-x_1)(x-x_2)}\right\vert
\\&= \left\vert (x^*-\frac{x_1+x_2}{2})\frac{2x-\frac{x_1+x_2}{2}-\frac{x_1+x_2}{2}+\frac{x_1+x_2}{2}-x^*}{(x-x_1)(x-x_2)}\right\vert
\leq \left\vert (x^*-\frac{x_1+x_2}{2})\frac{2|x-\frac{x_1+x_2}{2}|+|\frac{x_1+x_2}{2}-x^*|}{(x-x_1)(x-x_2)}\right\vert
\\&\leq \frac{1}{2}(|l_2-l_1|) \left\vert\frac{2|x-\frac{x_1+x_2}{2}|+\frac{1}{2}(|l_2-l_1|)}{(x-x_1)(x-x_2)}\right\vert
\\&\leq \frac{1}{2}(|l_2-l_1|) \left\vert\frac{2\min\{|x-x_1|,|x-x_2|\}+(l_1+l_2)+\frac{1}{2}(|l_2-l_1|)}{\frac{l_1+l_2}{2}\min\{|x-x_1|,|x-x_2|\}}\right\vert(~\text{by~Lemma}~\ref{useful221221})
\\&\leq 2\frac{|l_1-l_2|}{l_1+l_2}+\frac{|l_2-l_1|}{\min\{|x-x_1|,|x-x_2|\}}+\frac{|l_2-l_1|^2}{2(l_1+l_2)\min\{|x-x_1|,|x-x_2|\}}
\\&\leq 2\frac{|l_1-l_2|}{2\min\{l_1,l_2\}}+\frac{|l_2-l_1|}{\min\{|x-x_1|,|x-x_2|\}}+\frac{|l_2-l_1|\cdot |l_2-l_1|}{4\min\{l_1,l_2\}\min\{|x-x_1|,|x-x_2|\}}
\\&\leq \gamma^{\frac{1}{2}}+\gamma^{\frac{1}{4}}+\gamma^{\frac{1}{2}}\gamma^{\frac{1}{4}}
\leq 3\gamma^{\frac{1}{4}}.\end{array}\end{equation}

  By the definition and \eqref{etaeta111}, we have

  \begin{equation}\label{611}\begin{array}{ll}&\left\vert\frac{f-M(x-x_1)(x-x_2)}{M(x-x_1)(x-x_2)}\right\vert
  =\left\vert\frac{f+M(\frac{x_2-x_1}{2})^2-M(x-\frac{x_1+x_2}{2})^2}{M(x-x_1)(x-x_2)}\right\vert
  \\&\leq \left\vert\frac{f+M(\frac{x_2-x_1}{2})^2-M(x-x^*)^2}{M(x-x_1)(x-x_2)}\right\vert+\left\vert \frac{M(x-\frac{x_2+x_1}{2})^2-M(x-x^*)^2}{M(x-x_1)(x-x_2)} \right\vert
  \\&= \left\vert\frac{f(x^*)+\int_{x^*}^{x}f''(t)(x-t)dt+M(\frac{x_2-x_1}{2})^2-M(x-x^*)^2}{M(x-x_1)(x-x_2)}\right\vert+\gamma^{\frac{1}{4}}(~\text{by~Taylor~and}~\eqref{**112*})
  \\&\leq \left\vert\frac{f(x^*)-(-M(\frac{x_2-x_1}{2})^2)+\int_{x^*}^{x}(f''(t)-2M)(x-t)dt}{M(x-x_1)(x-x_2)}\right\vert+\gamma^{\frac{1}{4}}
  \\&\leq \left\vert\frac{C\gamma^{\frac{1}{2}}|M(\frac{x_2-x_1}{2})^2|+C\gamma^{\frac{1}{2}}M(x-x^*)^2}{M(x-x_1)(x-x_2)}\right\vert(~\text{by}~f''(x)\sim_{0,\gamma} 2M~and~\eqref{fx*})+\gamma^{\frac{1}{4}}
  \\&\leq C\gamma^{\frac{1}{2}}\left\vert\frac{M(\frac{x_2-x_1}{2})^2-M(x-x^*)^2}{M(x-x_1)(x-x_2)}\right\vert+C\gamma^{\frac{1}{2}}\left\vert\frac{2M(\frac{x_2-x_1}{2})^2}{M(x-x_1)(x-x_2)}\right\vert+\gamma^{\frac{1}{4}}
  \\&\leq C\gamma^{\frac{1}{2}}+C\gamma^{\frac{1}{2}}\left\vert\frac{2|f(x^*)|}{M(x-x_1)(x-x_2)}\right\vert+\gamma^{\frac{1}{4}}
  \\&\leq C\gamma^{\frac{1}{2}}+C\gamma^{\frac{1}{2}}\left\vert\frac{2|f(x^*)|}{c\gamma^{\frac{1}{4}}|f(x^*)|}\right\vert(~\text{by}~\eqref{I1I1I22})+\gamma^{\frac{1}{4}}
  \leq C\gamma^{\frac{1}{4}}.\end{array}\end{equation}

Note that, by \eqref{I1I1I22} and \eqref{611},
$$|f(x)|>(1-\gamma^{\frac{1}{4}})M|(x-x_1)(x-x_2)|>c\gamma^{\frac{1}{4}}|f(x^*)|.$$
Hence by \eqref{etaeta111}, for any $w\in I$ and $x\in {I-\left(I_1\bigcup I_2\right)}$ it holds that
$$\left\vert \frac{1}{\sqrt{1+2\frac{\delta_1^4}{f^2}+\frac{\tilde{\delta}^8_2(w)}{f^4}}}-1\right\vert\leq 2\frac{\delta_1^4}{f^2}+\frac{\tilde{\delta}^8_2(w)}{f^4}\leq 2\frac{\delta_1^4}{f^2}+\frac{{\delta}^8_1}{f^4}\leq 3\frac{\delta_1^4}{f^2}\leq C\gamma^{-\frac{1}{4}} \frac{\delta_1^4}{f^2(x^*)}< \gamma^{10}.$$
Thus
\begin{equation}\label{603}P(f,w)=\frac{f}{\sqrt{f^4+2\delta_1^4 f^2 +\tilde{\delta}_2^8(w)}}\sim_{0,\gamma^{10}} \frac{1}{f}\quad {\rm\ on\ } {I-\left(I_1\bigcup I_2\right)}.\end{equation}

Hence by \eqref{603} and \eqref{611}, it holds that \begin{equation}\label{eq31231213}~P(f,x)\sim_{0, \gamma^{\frac{1}{4}}} \frac{1}{f}\sim_{0, \gamma^{\frac{1}{4}}} \frac{1}{M(x-x_1)(x-x_2)}\quad {\rm\ on\ } {I-\left(I_1\bigcup I_2\right)}.\end{equation} Therefore

  \begin{equation}\label{dierbf}\begin{array}{ll}&\vert\int_{I-\left(I_1\bigcup I_2\right)}P(f,x) dx \vert \leq \vert\int_{I-\left(I_1\bigcup I_2\right)} \frac{1}{M(x-x_1)(x-x_2)} dx \vert+ C\gamma^{\frac{1}{4}}\int_{I-\left(I_1\bigcup I_2\right)} \frac{1}{M|(x-x_1)(x-x_2)|} dx .\end{array}\end{equation}

   Note that \eqref{qujianbianjie} implies $$\left\vert\log \left\vert \frac{x-x_2}{x-x_1}\right\vert\right\vert\big\vert_{\partial I}\leq C \gamma^{\frac{1}{4}}\cdot M^{-\frac{1}{2}}\leq C\gamma^{\frac{1}{4}}.$$  Therefore from $$\int \frac{1}{M(x-x_1)(x-x_2)}=\frac{1}{M(l_1+l_2)}\log \left\vert \frac{x-x_2}{x-x_1}\right\vert+C,$$ we have
 $$\begin{array}{ll}&\left\vert\int_{I-\left(I_1\bigcup I_2\right)}\frac{1}{M(x-x_1)(x-x_2)} dx \right\vert \leq \left\vert \frac{\log\left\vert \frac{(l_1+l_2-\gamma^{\frac{1}{4}} l_1)(l_1+l_2-\gamma^{\frac{1}{4}} l_2)}{(l_1+l_2+\gamma^{\frac{1}{4}} l_1)(l_1+l_2+\gamma^{\frac{1}{4}} l_2)}\right\vert}{M(l_1+l_2)}\right\vert+C\left\vert\frac{\gamma^{\frac{1}{4}}}{{M(l_1+l_2)}}\right\vert\leq C\vert \frac{\gamma^{\frac{1}{8}}}{M(l_1+l_2)}\vert\end{array}$$ and $$C\gamma^{\frac{1}{4}}\int_{I-\left(I_1\bigcup I_2\right)}\frac{1}{|M(x-x_1)(x-x_2)|} dx\leq C \gamma^{\frac{1}{4}}\vert \frac{ \vert \log \gamma\vert}{M(l_1+l_2)}\vert<C\gamma^{\frac{1}{8}}(M(l_1+l_2))^{-1}.$$ Therefore by \eqref{dierbf} and \eqref{fx*} we have
  \begin{equation}\label{eqfinal2}\left\vert\int_{I-\left(I_1\bigcup I_2\right)}P(f,x) dx \right\vert \leq C \gamma^{\frac{1}{8}}(M(l_1+l_2))^{-1}\leq C\gamma^{\frac{1}{10}}(M\cdot |f(x^*)|)^{-\frac{1}{2}}.\end{equation}
  Then by the help of \eqref{eq31231213} we immediately obtain \boxed{\eqref{useeq3}}. For \eqref{useeq4}, by \eqref{fx*} we have
  $$\begin{array}{ll}&\left\vert\int_{I-\left(I_1\bigcup I_2\right)}\frac{1}{M|(x-x_1)(x-x_2)|} dx \right\vert \\&\leq \left\vert \frac{\log\left\vert \frac{\gamma^{\frac{1}{4}} l_1}{(l_1+l_2+\gamma^{\frac{1}{4}} l_1)}\right\vert+\log\left\vert \frac{\gamma^{\frac{1}{4}} l_1}{(l_1+l_2-\gamma^{\frac{1}{4}} l_1)}\right\vert+\log\left\vert \frac{\gamma^{\frac{1}{4}} l_2}{(l_1+l_2+\gamma^{\frac{1}{4}} l_2)}\right\vert+\log\left\vert \frac{\gamma^{\frac{1}{4}} l_2}{(l_1+l_2-\gamma^{\frac{1}{4}} l_2)}\right\vert}{M(l_1+l_2)}\right\vert+C\left\vert\frac{\gamma^{\frac{1}{4}}}{{M(l_1+l_2)}}\right\vert
  \\&\leq C\vert \frac{|\log \gamma|}{M(l_1+l_2)}\vert<C\vert \frac{|\log \gamma|+|\log (M\cdot |f(x^*)|)|}{M(l_1+l_2)}\vert\leq  C(\big|\log \gamma|+|\log (M\cdot |f(x^*)|\big|)(M\cdot |f(x^*)|)^{-\frac{1}{2}},\end{array}$$ which yields \boxed{\eqref{useeq4}}.

  By \eqref{fx*}, \eqref{eqfinal1} and \eqref{eqfinal2} we have $$\left\vert\int_{I} P(f,x) dx\right\vert\leq \left\vert\int_{I_1\bigcup I_2} P(f,x) dx\right\vert+\left\vert\int_{I-(I_1\bigcup I_2)} P(f,x) dx\right\vert< C\gamma^{\frac{1}{10}}(M(l_1+l_2))^{-1}\leq C\gamma^{\frac{1}{10}}(M\cdot |f(x^*)|)^{-\frac{1}{2}},$$ which yields \boxed{\eqref{useeq2}}.

  \textbf{The proof of $ii_b$:}
  Note in this case we have $\min\limits_{x\in I}|f|=f(x^*)>0.$ Then by \eqref{etaeta111} for any $w,x\in I$ it holds that
  \begin{equation}\label{leflefl}\left\vert \frac{1}{\sqrt{1+2\frac{\delta_1^4}{f^2}+\frac{\tilde{\delta}^8_2(w)}{f^4}}}-1\right\vert\leq 2\frac{\delta_1^4}{f^2}+\frac{\tilde{\delta}^8_2(w)}{f^4}\leq 2\frac{\delta_1^4}{f^2}+\frac{{\delta}^8_1}{f^4}\leq 3\frac{\delta_1^4}{f^2}\leq 3\frac{\delta_1^4}{\min\limits_{x\in I}f^2}\leq 3\frac{\delta_1^4}{f^2(x^*)}<3\gamma^{10}.\end{equation}
  On the other hand, it holds that

  $$\begin{array}{ll}&\left\vert\frac{f}{M(x-x^{*})^2+f(x^*)}-1\right\vert=\left\vert\frac{f(x^*)+\int_{x^*}^{x}f''(t)(x-t)dt-(M(x-x^{*})^2+f(x^*))}{M(x-x^{*})^2+f(x^*)}\right\vert(~\text{by~Taylor})
  \\&=\left\vert\frac{\int_{x^*}^{x}(f''(t)-2M)(x-t)dt}{M(x-x^{*})^2+f(x^*)}\right\vert
  \leq \left\vert\frac{C\gamma^{\frac{1}{2}}M(x-x^*)^2}{M(x-x^{*})^2+f(x^*)}\right\vert=C\gamma^{\frac{1}{2}}(~\text{by}~M,f(x^*)>0).\end{array}$$

  Therefore \begin{equation}\label{ibb}f(x)\sim_{0, \gamma^{\frac{1}{2}}}M(x-x^{*})^2+f(x^*)~\text{on}~I.\end{equation}
   Note \eqref{leflefl} and \eqref{ibb} imply
$$P(f,x)\sim_{0,\gamma^{\frac{1}{4}}} \frac{1}{M(x-x^*)^2+f(x^*)}.$$

Hence $$\left\vert\frac{\int_{I} P(f,x)}{\int_{I}\frac{1}{M(x-x^*)^2+f(x^*)}}-1\right\vert\leq \gamma^{\frac{1}{8}}.$$

 Moreover \eqref{qujianbianjie} implies that $\left\vert 1-\frac{\left\vert\arctan(\frac{{M^{\frac{1}{2}}(x-x^*)}}{\sqrt{f(x^*)}})\right\vert}{\frac{\pi}{2}}\right\vert
 \big\vert_{\partial I} \leq \gamma^{\frac{1}{4}}M^{-\frac{1}{2}}\leq \gamma^{\frac{1}{8}}.$ Finally, by the fact $$\int \frac{1}{M(x-x^*)^2+f(x^*)}=\frac{\arctan(\frac{M^{\frac{1}{2}}(x-x^*)}{\sqrt{f(x^*)}})}{\sqrt{(M f(x^*))}}+C,$$ we immediately obtain \boxed{\eqref{useeq2*}}.

 \textbf{The proof of $i$:}
  It remains to prove \eqref{jiandangujia} and \eqref{jiandangujia1}.  Note that by the assumption $|f(x^*)|>\gamma^{\frac{1}{4}} M^{-1},$ we have $l_i\geq c\gamma^{\frac{1}{8}} M^{-1},~i=1,2 .$
  And we need to consider the following two cases:

  \begin{enumerate}
  \item[(1):] If $f(x^*)>0,$ then since $f''(x)\sim_{0,\gamma^{\frac{1}{2}}} 2M>0\ {\rm\ on\ } {I},$ we immediately have $f(x)\geq M(x-x^*)^2+f(x^*).$ Hence we get $$P(|f|,x)\leq \frac{1}{|f|}\leq \frac{1}{M(x-x^*)^2+f(x^*)}.$$
      By the same calculation as in the proof of \eqref{useeq2*}, the fact $M\geq \gamma^{-10000}$ and \eqref{603}, we obtain \eqref{jiandangujia}.
  \item[(2):] If $f(x^*)<0,$ then it holds from Lemma \ref{useful221221} and
   $\vert f(x^*)\vert\leq 100 \cdot M(\frac{|x_1-x_2|}{2})^2$
  that \begin{equation}\label{rf2}|f(x)|\geq \vert\frac{1}{4}M(x-x_1)(x-x_2)\vert\geq \frac{1}{8}M|x_1-x_2|\min\{|x-x_1|,|x-x_2|\}\geq \frac{5}{8} \gamma^{\frac{1}{8}}\min\{|x-x_1|,|x-x_2|\}.\end{equation}   Hence $f\in \mathcal{H}(1,\gamma^{\frac{1}{8}}).$ Combining this with the fact $P(|f|)\leq \frac{1}{\sqrt{f^2+\delta_1^4}},$ Lemma \ref{usefullemma2} yields \eqref{jiandangujia}.
  \end{enumerate}

  For \eqref{jiandangujia1}, we only need to consider the case $(2)$ above.
Note $$|f|\geq\gamma^{\frac{1}{4}}\sqrt{\frac{|f(x^*)|}{M}}\geq \gamma^{\frac{3}{8}}M^{-1},$$ which together with \eqref{rf2}, implies
  $$\frac{1}{|f|}=\frac{\sqrt{2}}{\sqrt{f^2+f^2}}\leq \frac{C}{\sqrt{(\gamma^{\frac{1}{8}}\min\{|x-x_1|,|x-x_2|\})^2+\gamma^{\frac{3}{4}}M^{-2}}}.$$

  Then again by Lemma \ref{usefullemma2} we obtain \eqref{jiandangujia1} as desired.
 \hfill\qed\end{proof}
  \begin{lemma}\label{lipuseful} Let $d\in \R_+$ and $M\gg 1\gg \delta_1>0$. Consider two closed intervals $I_1\subset I_2\subset \R$ centered at $x_0$ satisfying $\delta_1^{10^{-9}}\leq |I_1|^{\frac{1}{200000}}\leq |I_2|^{10^{-5}}\leq \min\{M^{-1}, d, d^{-1}\}$ and $M|I_2|\leq 10^{-7}d$.
   Assume $f(x),\bar{\delta}_2(x)\in C^2(I_2)$ satisfying $\delta_1>\max\limits_{x\in I_2}|\bar{\delta}_2(x)|>0,$  $\vert\bar{\delta}_2'(x)\vert\leq \delta_1^{1-10^{-4}},$
   and $|f(x)-d(x-x_0)|\leq M(x-x_0)^2.$ Then we have
    $$\left\vert\int_{I_2-I_1}\frac{f dx}{\sqrt{f^4+2\delta^4_1 f^2+\bar{\delta}^8_2(x)}}\right\vert,\quad
    \left\vert \int_{I_2-I_1} \frac{1}{f} dx \right\vert\leq C\frac{M(|I_2|-|I_1|)}{d^2}$$
       and
   $$\int_{I_2-I_1}\frac{|f| dx}{\sqrt{f^4+2\delta^4_1 f^2+\bar{\delta}^8_2(x)}}\leq Cd^{-1}|\log \delta_1|.$$
  \end{lemma}

\begin{proof}
By \eqref{eqfinal1*}, one has $\left\vert\int_{I_2-I_1}\frac{f dx}{\sqrt{f^4+2\delta^4_1 f^2+\bar{\delta}^8_2(x)}}\right\vert\leq C\frac{M(|I_2|-|I_1|)}{d^2}+C\delta_1^4.$
By the assumption, one notes that
$$\frac{M(|I_2|-|I_1|)}{d^2}\geq \frac{1}{2}|I_2|^3\geq \frac{1}{2}|I_1|^{\frac{3}{2}}\geq \frac{1}{2}\delta_1^{10}.$$
Then
$$\left\vert\int_{I_2-I_1}\frac{f dx}{\sqrt{f^4+2\delta^4_1 f^2+\bar{\delta}^8_2(x)}}\right\vert\leq C\frac{M(|I_2|-|I_1|)}{d^2}.$$

Furthermore, the assumption $M|I_2|\leq 10^{-7}d$ and  $|f(x)-d(x-x_0)|\leq M(x-x_0)^2$ imply that $$|f|\geq d(x-x_0)-M(x-x_0)^2\geq [d-M(x-x_0)]\geq [d-M|I_2|](x-x_0)\geq \frac{1}{2} d(x-x_0).$$ Then
$$\begin{array}{ll}\left\vert\frac{1}{f}-\frac{1}{d(x-x_0)}\right\vert\leq \left\vert\frac{f-d(x-x_0)}{f d(x-x_0)}\right\vert\leq \frac{M(x-x_0)^2}{\frac{1}{2}d^2(x-x_0)^2}\leq 2\frac{M}{d^2},\end{array}$$ which implies the second inequality.
For the last one, notes that
$$\frac{|f| dx}{\sqrt{f^4+2\delta^4_1 f^2+\bar{\delta}^8_2(x)}}\leq \frac{1}{\sqrt{f^2+2\delta_1^4}}.$$ Then Lemma \ref{usefullemma2} directly implies what we desire.
  \hfill\qed\end{proof}

\begin{lemma}\label{lm104} Given an open interval $I\subset \R$ centered by $x_0,$ $0<\epsilon\ll 1$ and $f\in C^2(I),$ assume that $f\sim_{2,\epsilon} a+b(x-x_0)^2$ on $I$ and $|a|\leq \epsilon^{10^5},$ then there exists
$A,B\in \R$ such that
$$|f|\geq \frac{1}{10}|A+B(x-x^*)^2|$$
with $x^*\in I,$ $A\sim_{0,\epsilon} a$ and $B\sim_{0,\epsilon} b.$\end{lemma}

\begin{proof} We separately consider the following two cases.

For the case {\bf $ab\leq 0$}, $f$ has two distinct zeros denoted by $z_1,z_2$ and one extreme point $x_0.$ Without loss of generality, we assume that
$z_1\leq x_x \leq z_2;~a<0;~b>0.$ A direct calculation yields that $$\frac{z_2-z_1}{2}\sim_{0,\epsilon} z_2-x_0 \sim_{0,\epsilon} x_0-z_1;~\frac{b(z_1-z_2)^2}{4}\sim_{0,\epsilon} -a;~f'(z_1)\sim_{0,\epsilon}-f'(z_2)\sim_{0,\epsilon} -b(z_2-z_1).$$
Note that
\begin{itemize}
\item[(i)] if $x\leq z_1+\frac{1}{2000}\frac{z_2-z_1}{2},$ then it holds that $f'(x)\sim_{0,\epsilon} -b(z_2-z_1)+2b(x-z_1)\leq -\frac{1}{10}b(z_2-z_1).$\vskip 0.2cm
\item[(ii)] if $x\geq z_2-\frac{1}{2000}\frac{z_2-z_1}{2},$ then it holds that $f'(x)\sim_{0,\epsilon}b(z_2-z_1)+2b(x-z_2)\geq \frac{1}{10}b(z_2-z_1).$\vskip 0.2cm
\item[(iii)] if $z_1+\frac{1}{2000}\frac{z_2-z_1}{2}<x< z_2-\frac{1}{2000}\frac{z_2-z_1}{2},$ then it holds that $|b(x-z_1)(x-z_2)|,|f(x)|\geq \frac{1}{200000}|a|.$ Note
$$\begin{array}{ll}&\left\vert f(x)-b(x-z_1)(x-z_2)\right\vert \leq \left\vert a+b(x-x_0)^2-\frac{b(z_2-z_1)^2}{4}-b(x-\frac{z_1+z_2}{2})^2\right\vert+\epsilon|f(x)|
\\&\leq \left\vert b(x_0-\frac{z_1+z_2}{2})(2x-x_0-\frac{z_1+z_2}{2})\right\vert+\left\vert \epsilon a\right\vert+\epsilon|f(x)|
\leq 1000\epsilon b\vert z_1-z_2\vert^2+|\epsilon a|+\epsilon|f(x)|
\leq 10000\epsilon |a|+\epsilon|f(x)|.
\end{array}$$
Therefore
$$\begin{array}{ll}
&f(x)-\frac{1}{10}b(x-z_1)(x-z_2)
=\frac{9}{10} f(x)+ (\frac{1}{10}(f(x)-b(x-z_1)(x-z_2)))\\
\\&\leq \frac{9}{10} f(x)+\frac{1}{10}(10000\epsilon |a|+\epsilon|f(x)|)
\leq \frac{9}{10} f(x)+10^{10}\epsilon |f(x)|
\leq \frac{4}{5} f(x)<0.
\end{array}$$
\end{itemize}

Combining (i),(ii) and (iii), we obtain
$$|f(x)|\geq \frac{1}{10}\left\vert b(x-z_1)(x-z_2)\right\vert=\frac{1}{10}\left\vert-\frac{b(z_1-z_2)^2}{4}+b(x-\frac{z_1+z_2}{2})^2 \right\vert.$$

Then taking $A=-\frac{b(z_1-z_2)^2}{4},$ $B=b$ and $x^*=\frac{z_1+z_2}{2}$ completes the proof.
\vskip 0.2cm
For the case {\bf $ab\geq 0,$}\  without loss of generality we assume that $a,b\geq 0.$ Then Directly taking $A=a,$ $B=b$ and $x^*=x_0$ completes the proof.
\hfill\qed\end{proof}

\noindent\bf{\footnotesize Acknowledgements}\quad\rm
{\footnotesize The second author was partially supported by NSF of China (Grants 12371185) and the Fundamental Research Funds for the Central Universities (the start-up fund), Peking University. The third author is supported by the NSF of China (Grants 12271245).}\\[4mm]

\end{document}